\documentclass[11pt,psamsfonts]{amsart}

\usepackage{amsmath}
\usepackage{amssymb}
\usepackage{amsfonts}
\usepackage{verbatim}
\usepackage{overpic}
\usepackage{enumitem}   
\usepackage{graphicx} 
\usepackage{multicol}
\usepackage{xcolor}
\usepackage{amsrefs}
\usepackage{pdfpages}
\usepackage[a4paper,margin=2cm]{geometry}
\usepackage[hidelinks]{hyperref}
\usepackage{colortbl}
\definecolor{mygray}{gray}{0.85}

\graphicspath{{Figures/}}

\newtheorem{theorem}{Theorem}

\newtheorem{proposition}{Proposition} 
\newtheorem{corollary}{Corollary} 
\newtheorem{remark}{Remark} 
 
\newtheorem{definition}{Definition}

\begin{document}
	
\title[Evolutionary Stable Strategies and Cubic Vector Fields]{Evolutionary Stable Strategies and Cubic Vector Fields}

\author[Jefferson Bastos, Claudio Buzzi and Paulo Santana]
{Jefferson Bastos$^1$, Claudio Buzzi$^1$ and Paulo Santana$^1$}

\address{$^1$ IBILCE--UNESP, CEP 15054--000, S. J. Rio Preto, S\~ao Paulo, Brazil}
\email{jefferson.bastos@unesp.br; claudio.buzzi@unesp.br; paulo.santana@unesp.br}



\begin{abstract}
	The introduction of concepts of Game Theory and Ordinary Differential Equations into Biology gave birth to the field of Evolutionary Stable Strategies, with applications in Biology, Genetics, Politics, Economics and others. In special, the model composed by two players having two pure strategies each results in a planar cubic vector field with an invariant octothorpe. Therefore, in this paper we study such class of vector fields, suggesting the notion of genericity and providing the global phase portraits of the generic systems with a singularity at the central region of the octothorpe.
\end{abstract}

\maketitle

\section{Introduction}\label{Sec1}

The introduction of concepts of Game Theory into Biology, by Maynard Smith and Price \cite{SmiPri}, gave birth to the field of Evolutionary Stable Strategies. The introduction of Ordinary Differential Equations in such concepts, by Taylor and Jonker \cite{TayJon}, proved to be a very fruitful tool with applications in a myriad of research areas, as Parental Investing, Biology, Genetics, Evolution, Ecology, Politics and Economics. See \cites{Hines,SchSig,Bomze,AccMarOvi,CanBer,XiaoYu} and the references therein. In this paper we present a brief explanation of the model, followed by an example, and then we study the case where the model is composed by two players having two pure strategies each, which is given by a planar cubic vector field with an invariant octhotorpe.

The model work as follows. Suppose that we have $n$ players, denoted by $\Gamma_1,\dots \Gamma_n$, and that each player has an unique strategy $X_i$. Let $a_{ij}\in\mathbb{R}$ denote the payoff of the pure strategy $X_i$ against the pure strategy $X_j$. Also, let $A=(a_{ij})$ be the payoff matrix. In what follows, $\left<\cdot,\cdot\right>$ denotes the standard inner product on $\mathbb{R}^n$. Given a probabilistic vector
	\[x\in S^n:=\{(x_1,\dots,x_n)\in\mathbb{R}^n\colon x_i\geqslant0,\; x_1+\dots+x_n=1\},\]
we define
\begin{equation}\label{16}
	\left<e_i,Ax\right>=\sum_{j=1}^{n}a_{ij}x_j,
\end{equation}
as the payoff of the pure strategy $X_i$ against the mix strategy $x_1X_1+\dots+ x_nX_n$. Furthermore, the average payoff of such mix strategy is defined by,
\begin{equation}\label{17}
	\left<x,Ax\right>=\sum_{i,j=1}^{n}a_{ij}x_ix_j.
\end{equation}
Therefore, the dynamics between the strategies is defined by the \emph{replicator equation},
\begin{equation}\label{39}
	\dot x_i=x_i\left(\sum_{j=1}^{n}a_{ij}x_j-\sum_{j=1}^{n}a_{ij}x_ix_j\right), \quad i\in\{1,\dots,n\}.
\end{equation}
In simple words, the rate of change of the population $x_i$ depends on the difference between the payoffs \eqref{16} and \eqref{17} (the bigger the difference, the bigger the superiority of strategy $X_i$), and on the size of the population itself, as a percentage of the total population $x_1+\dots+x_n$.

An important fact about the replicator equation \eqref{39}, is its equivalence with the compactification of the Volterra-Lotka equation
	\[\dot y_i=y_i\left(b_{i0}+\sum_{j=1}^{n}b_{ij}y_j\right), \quad i\in\{1,\dots,n-1\},\]
defined on $\mathbb{R}_+^{n-1}=\{(y_1,\dots,y_{n-1})\in\mathbb{R}^{n-1}\colon y_i\geqslant0\}$, for some $b_{ij}\in\mathbb{R}$. In special, system \eqref{39} has limit cycle (i.e. an isolated periodic orbit) if, and only if, $n\geqslant4$. For more details, see the work of Hofbauer \cite{Hof}. For a classification of all the possible phase portraits of system \eqref{39} with $n=3$, see the works of Zeeman \cite{Zeeman1980} and Bomze \cite{Bomze}. 

We now look at the case where each player $\Gamma_i$ have two strategies. To simplify the approach, we consider two players $\Gamma_1$ and $\Gamma_2$ with two strategies each, denoted by  $\{X_1,X_2\}$ and $\{Y_1,Y_2\}$. In this case, let $a_{ij}^*\in\mathbb{R}$ be the payoff of the pure strategy $X_i$ against the pure strategy $Y_j$ and $b_{ij}^*\in\mathbb{R}$ be the payoff of the pure strategy $Y_i$ against the pure strategy $X_j$. Let also
	\[A^*=\left(\begin{array}{cc} a_{11}^* & a_{12}^* \\ a_{21}^* & a_{22}^* \end{array}\right), \quad B^*=\left(\begin{array}{cc} b_{11}^* & b_{12}^* \\ b_{21}^* & b_{22}^* \end{array}\right),\]
be the payoff matrices. Similarly to the previous case, given a planar probabilistic vector
	\[y\in S^2=\{(y_1,y_2)\in\mathbb{R}^2\colon y_1+y_2=1\},\]
we define,
	\[\left<e_i,A^*y\right>=a_{i1}^*y_1+a_{i2}^*y_2,\]
as the payoff of the pure strategy $X_i$ of player $\Gamma_1$, against the mix strategy $y_1Y_1+ y_2Y_2$ of player $\Gamma_2$. Similarly, given a probabilistic vector $x\in S^2$, the payoff of the pure strategy $Y_i$ against the mix strategy $x_1X_1+x_2X_2$ is given by,
	\[\left<e_i,B^*x\right>=b_{i1}^*x_1+b_{i2}^*x_2.\]
Furthermore, the average payoff of $x_1X_1+x_2X_2$ against $y_1Y_1+y_2Y_2$ is given by,
	\[\left<x,A^*y\right>=a_{11}^*x_1y_1+a_{12}^*x_1y_2+a_{21}^*x_2y_1+a_{22}^*x_2y_2,\]
while the average payoff of $y_1Y_1+y_2Y_2$ against $x_1X_1+x_2X_2$ is given by,
	\[\left<y,B^*x\right>=b_{11}^*x_1y_1+b_{12}^*x_1y_2+b_{21}^*x_2y_1+b_{22}^*x_2y_2.\]
Therefore, similarly to the previous case, the dynamics between players $\Gamma_1$ and $\Gamma_2$ is given by,
\begin{equation}\label{2}
	\begin{array}{l}
		\dot x_1=x_1(a_{11}^*y_1+a_{12}^*y_2-(a_{11}^*x_1y_1+a_{12}^*x_1y_2+a_{21}^*x_2y_1+a_{22}^*x_2y_2)), \vspace{0.2cm} \\
		\dot x_2=x_2(a_{21}^*y_1+a_{22}^*y_2-(a_{11}^*x_1y_1+a_{12}^*x_1y_2+a_{21}^*x_2y_1+a_{22}^*x_2y_2)), \vspace{0.2cm} \\
		\dot y_1=y_1(b_{11}^*x_1+b_{12}^*x_2-(b_{11}^*y_1x_1+b_{12}^*y_1x_2+b_{21}^*y_2x_1+b_{22}^*y_2x_2)), \vspace{0.2cm} \\
		\dot y_2=y_2(b_{21}^*x_1+b_{22}^*x_2-(b_{11}^*y_1x_1+b_{12}^*y_1x_2+b_{21}^*y_2x_1+b_{22}^*y_2x_2)). 
	\end{array}
\end{equation} 
However, since $x_1+x_2=y_1+y_2=1$, it follows that to describe the dynamic between players $\Gamma_1$ and $\Gamma_2$, it is necessary only two variables, namely $x=x_1$ and $y=y_1$. Therefore, if we define $x_2=1-x$ and $y_2=1-y$, then we can simplify system \eqref{2} and obtain the planar polynomial system,
\begin{equation}\label{3}
	\begin{array}{l}
		\dot x=x(x-1)(a_{22}^*-a_{12}^*+(a_{12}^*+a_{21}^*-a_{11}^*-a_{22}^*)y), \vspace{0.2cm} \\
		\dot y=y(y-1)(b_{22}^*-b_{12}^*+(b_{12}^*+b_{21}^*-b_{11}^*-b_{22}^*)x). 
	\end{array}
\end{equation}
Therefore, to understand the dynamics between players $\Gamma_1$ and $\Gamma_2$, it is enough to understand the phase portrait of system \eqref{3} within the square given by,
	\[\left\{(x,y)\in\mathbb{R}^2\colon 0\leqslant x\leqslant 1,\; 0\leqslant y\leqslant 1\right\}.\]
In this paper we consider \eqref{3} as a system defined on the plane $\mathbb{R}^2$ and we study its compactification on the Poincar\'e Sphere. For a study focused on the above square, see Schuster et al \cite{SchSigHof}.
	
\section{Example}\label{Sec2}	
	
As an example, we present a brief explanation of a model due to Accinelli et al \cite{AccMarOvi}. Consider a democratic society and let $G$ denote its elected government. The government can be re-elected or not after an electoral period. When elected, such government appoint public officials, denoted by $O$, who may or may not be renewed by a new government. Such officials are in charge of carrying out the legal and administrative management of the government and serve directly to the citizens, when they require to carry out this type of formalities. When required by a citizen, an official may by honest and just fulfill his/her duty, or be corrupt and fulfills his/her duty as long the citizen pays a bribe. Similarly, a member of the government $G$ may also be honest or corrupt. We say that a government's member is corrupt when he/she collude with a corrupt official, sharing the bribe. In return, the member of the government does not punish the official. On the other hand, a government's members is honest when, after an offer of collude, he/she punishes the corrupt official. Therefore, following Accinelli et al \cite{AccMarOvi}, we let $\Gamma_1=O$ and $\Gamma_2=G$, i.e. the officials are the first player and the government the second. Moreover, we denote by $n_c$ and $n_{nc}$ the percentages of corrupt and non-corrupt officials, respectively (i.e. $x_1=n_c$ and $x_2=n_{nc})$. Similarly, we let $(g_c,g_{nc})$ denote the corrupt and non-corrupt government's members. Moreover, the payoffs coefficients are given by the following table.
\begin{table}[h]
	\caption{The payoffs table.}
	\begin{tabular}{c cc c} 
		\hline
		$O|G$ & $G_{c}$ && $G_{nc}$ \\
		\hline
		$O_{c}$ & $W+M_c-M_g$, $M_g-W+V_{G_c}-KP$ & $\quad$ & $W+M_c-M$, $M-W-e+V_{G_{nc}}$ \\
		$O_{nc}$ & $W-M_g'$, $M_g'-W+V_{G_c}-KP$ & $\quad$ &$W$, $-W+V_{G_{nc}}$ \\
		\hline
	\end{tabular}
\end{table}

Where:
\begin{itemize}
	\item $W$ is the wage of the officials, paid by the government;
	\item $M$ is the fine imposed by an honest government to a corrupt official;
	\item $M_c$ is the bribe that a corrupt official takes from a citizen;
	\item $M_g$ is the fraction of the bribe that a corrupt officials shares with a corrupt government's member;
	\item $M_g'$ is the amount that an honest official pays to a corrupt government's members to keep his position;
	\item $e$ is the cost associated with the capture of a corrupt official;
	\item $V_{G_c}$ and $V_{G_{nc}}$ is the value that a corrupt and non-corrupt government assign to be re-elected for the next period. It may also depends on the level of corrupt-intolerance of the citizens.
	\item $KP$ is the total amount of money that the government offers to the citizens to buy their votes, where $P$ is the unitary valued payed to each person and $K$ is the number of citizens to which the government pays.
\end{itemize}
Therefore, we have the following values for $a_{ij}^*$ and $b_{ij}^*$:
\[\begin{array}{ll}
		a_{11}^*=W+M_c-M_g, & \quad b_{11}^*=M_g-W+V_{G_c}, \vspace{0.2cm} \\
		a_{12}^*=W+M_c-M, & \quad b_{12}^*=M_g'-W+V_{G_c}-KP, \vspace{0.2cm} \\
		a_{21}^*=W-M_g', & \quad b_{21}^*=M-W-e+V-{G_{nc}}, \vspace{0.2cm} \\
		a_{22}^*=W, & \quad b_{22}^*=-W+V_{G_{nc}}.
\end{array}\]
Hence, the ordinary differential system is given by:
\begin{equation}\label{18}
	\dot {n_c}=n_c(n_c-1)(a_{00}+a_{01}g_c), \quad \dot {g_c}=g_c(g_c-1)(b_{00}+b_{10}n_c),
\end{equation}
where the coefficients $a_{ij}$ and $b_{ij}$ are given by,
	\[a_{00}= M-M_c, \quad a_{01}=M_g-M_g'-M, \quad b_{00}=KP+V_{G_{nc}}-V_{G_c}-M_g', \quad b_{10}=M+M_g'-M_g-e.\]
Let $(\overline{n_c},\overline{g_c})=(-\frac{b_{00}}{b_{10}},-\frac{a_{00}}{a_{01}})$ be the unique singularity of \eqref{18} that may not lay in any of the four invariant straight lines. Since $n_c$ and $g_c$ denotes the proportions of the corrupted officials and government's members, then such singularity is well defined in the model if, and only if,
	\[(\overline{n_c},\overline{g_c})\in\left\{(n_c,g_c)\in\mathbb{R}^2\colon 0\leqslant n_c\leqslant 1,\; 0\leqslant g_c\leqslant 1\right\}.\]
Therefore, except for degenerate cases, we can assume $0<\overline{n_c}<1$ and $0<\overline{g_c}<1$. Under the principle that \emph{corruption corrupts} (see Proposition~$1$ of Accinelli et al \cite{AccMarOvi}), we can assume that a corrupt (non-corrupt) official encourages a corrupt (non-corrupt) government and vice-versa. Therefore, we can assume,
	\[b_{11}^*>b_{21}^*, \quad b_{22}^*>b_{12}^*, \quad a_{11}^*>a_{21}^*, \quad a_{22}^*>a_{12}^*.\]
Hence, it follows from \eqref{3} that,
	\[b_{00}>0, \quad b_{00}+b_{10}<0, \quad a_{00}>0, \quad a_{00}+a_{01}<0.\]
Thus, it can be seen that the phase portrait of system \eqref{18} is given by Figure~\ref{Fig8}(a). Observe that the points $(0,0)$ (zero corruption) and $(1,1)$ (fully corrupted) are the only stable singularities and that almost all initial conditions evolve to one of such singularities, i.e. for almost all initial conditions, the democratic society either annihilate the corruption or corrupts itself completely.
\begin{figure}[h]
	\begin{center}
		\begin{minipage}{5cm}
			\begin{center}
				\begin{overpic}[height=3cm]{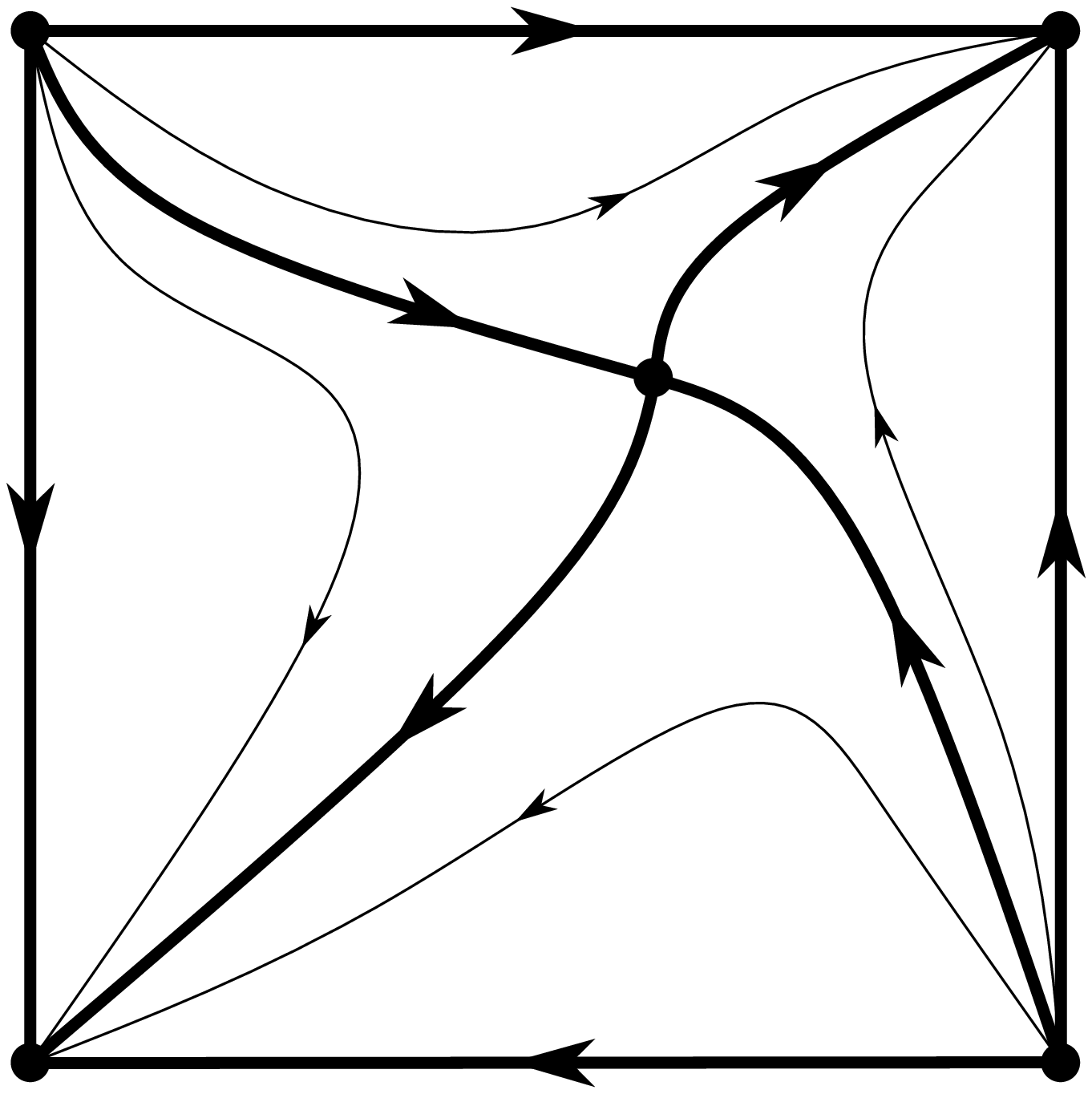} 
				\end{overpic}
				
				$(a)$
			\end{center}
		\end{minipage}
		\begin{minipage}{5cm}
			\begin{center}
				\begin{overpic}[height=3cm]{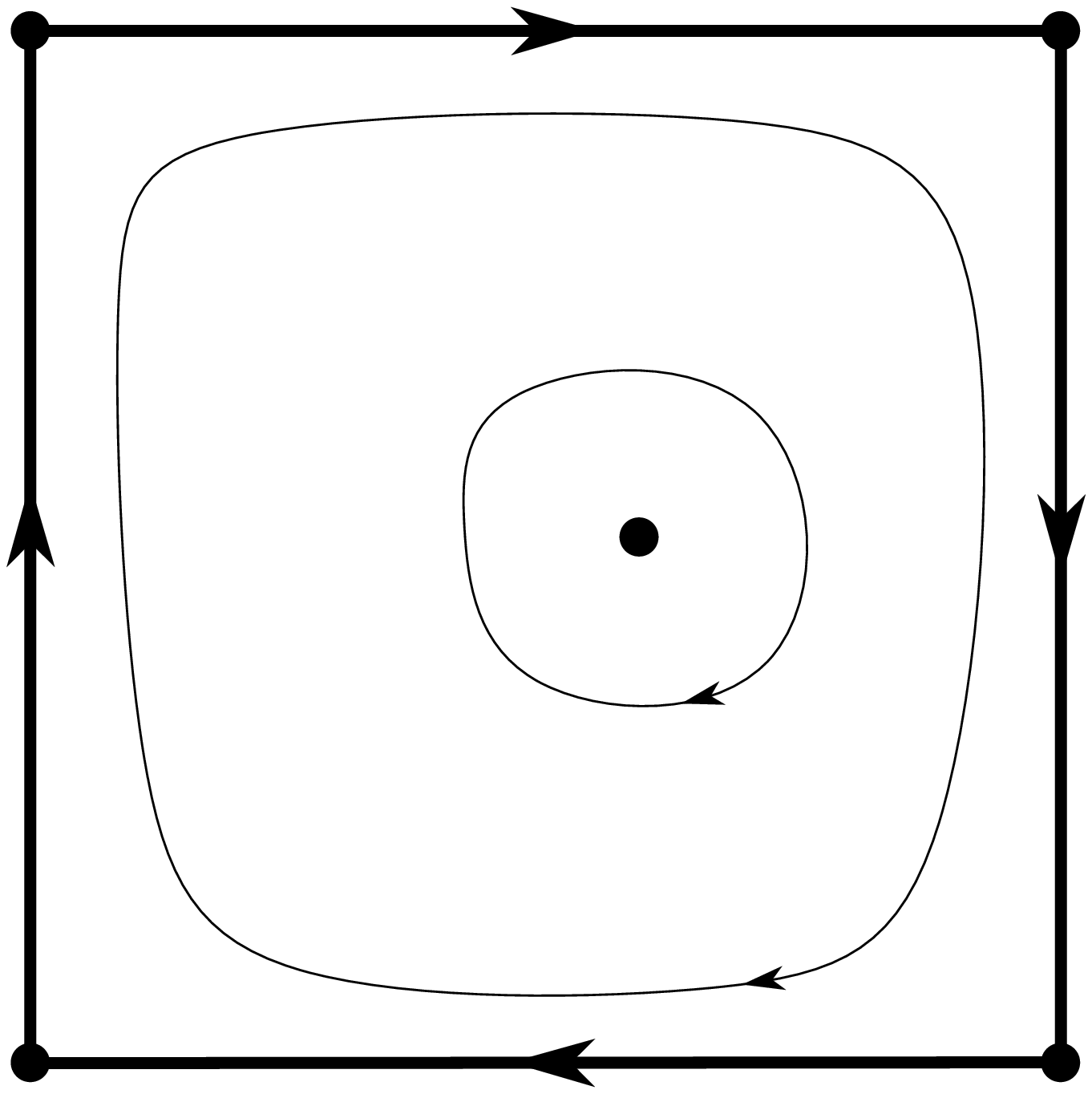} 
				\end{overpic}
				
				$(b)$
			\end{center}
		\end{minipage}
	\end{center}
	\caption{Phase portraits of the model due to Accinelli et al \cite{AccMarOvi}.}\label{Fig8}
\end{figure}

Another principle explored by the authors of the model is the one in which the intolerance of the citizens depends on the level of corruption of the system. In simple words, the bigger the corruption, the bigger the intolerance to corruption. More precisely, if corruption increases, then $V_{G_{nc}}$ increases and $V_{G_c}$ decreases. On the other hand, if corruption decreases, then $V_{G_{nc}}$ decreases and $V_{G_c}$ increases. Under this principle, we have (following Accinelli et al \cite{AccMarOvi}) that \emph{a low index of intolerance causes an increase in government corruption. Such increase causes more officials to prefer to be corrupt, which in turn increases the corruption of the system. Facing such increase in corruption, the intolerance rises, forcing the level of corruption of the government to decrease, which causes the officials to prefer to be non-corrupt, which in turn decreases the corruption of the system. Thus, decreasing the index of intolerance.} See Figure~\ref{Fig8}(b).

\section{Statement of the Main Results}\label{Sec3}

The chasing for more realistic models demand the payoff coefficients to depend on the weight given to strategies $X_i$ and $Y_j$, rather than being constant, i.e. $a_{ij}^*=a_{ij}^*(x,y)$ and $b_{ij}^*=b_{ij}^*(x,y)$. Hence, if we assume that each payoff $a_{ij}^*$, $b_{ij}^*$ is a polynomial of degree at most $n$, then system \eqref{3} becomes
\begin{equation}\label{4}
	\dot x=x(x-1)\mathcal{P}(x,y), \quad \dot y=y(y-1)\mathcal{Q}(x,y), 
\end{equation}
with $\mathcal{P}$ and $\mathcal{Q}$ polynomials of degree at most $n+1$. In this paper we assume that $\mathcal{P}$ and $\mathcal{Q}$ are polynomials of degree one. Thus, system \eqref{4} becomes,
\begin{equation}\label{1}
	\dot x=x(x-1)(a_{00}+a_{10}x+a_{01}y), \quad \dot y=y(y-1)(b_{00}+b_{10}x+b_{01}y). 
\end{equation}
Observe that the hypothesis of every payoff $a_{ij}^*$ and $b_{ij}^*$ being constant is equivalent to $a_{10}=b_{01}=0$. In this paper we suggest the notion of \emph{genericity}, prove that the phase portrait of system \eqref{1} depends mainly on the relative position between its singularities and provide the generic phase portraits with a singularity inside the square,
	\[V=\left\{(x,y)\in\mathbb{R}^2\colon 0\leqslant x\leqslant1,\; 0\leqslant y\leqslant1\right\},\]
which is the region of the plane in which the model is defined.	Following Sotomayor \cite{Soto1974}, given a family $\mathfrak{X}$ of smooth vector fields, we say that $X\in\mathfrak{X}$ is \emph{generic} if $X$ is an element of a collection $\Sigma_0\subset\mathfrak{X}$ such that:
\begin{enumerate}[label=(\alph*)]
	\item $\Sigma_0$ is large with respect to $\mathfrak{X}$;
	\item its elements are amenable to simple description.
\end{enumerate}
More precisely, if $\mathfrak{X}$ is endowed with some interesting topology, then condition $(a)$ can be replaced by:
\begin{enumerate}
	\item[$(a_1)$] $\Sigma_0$ is open and dense in $\mathfrak{X}$.
\end{enumerate}
Usually, $\mathfrak{X}$ is endowed with the $C^r$ or $C^\infty$-topology (or the coefficients topology, if $\mathfrak{X}$ is a class of polynomial vector fields). Although it precedes the work of Sotomayor, one of the biggest example of generic family is due to Peixoto \cite{Pei1959,Pei1961}. Following his work \cite{Pei1959}, we say that two vector fields $X$, $Y\in\mathfrak{X}$ are \emph{topologically equivalent} if there exists an homeomorphism $h$ which sends orbits of $X$ to orbits of $Y$, preserving or reversing the orientation of all orbits. Furthermore, if there exists a neighborhood $N\subset\mathfrak{X}$ of $X$ such that $X$ is topologically equivalent to every $Y\in N$, then we say that $X$ is \emph{structurally stable}.

\begin{theorem}[Peixoto \cite{Pei1961}]
	Let $M$ be a compact differentiable manifold of dimension two, endowed with a metric, and $\mathfrak{X}$ be the class of $C^1$-vector fields over $M$, endowed with the $C^1$-topology. Let also $\Sigma_0\subset\mathfrak{X}$ be the class of the structurally stable vector fields. Given $X\in\mathfrak{X}$, then $X\in\Sigma_0$ if, and only if, the following conditions are satisfied.
	\begin{enumerate}[label=(\alph*)]
		\item There is only a finite number of singularities, all hyperbolic;
		\item There is only a finite number of closed orbits, all hyperbolic;
		\item There is no connections between saddles;
		\item The $\alpha$ and $\omega$-limits of every orbit is either a singularity or a closed orbit.
	\end{enumerate}
	Moreover, $\Sigma_0$ is open and dense in $\mathfrak{X}$.
\end{theorem}

Given $X\in\Sigma_0$, we say that $X$ is of \emph{codimension zero}. Let $\mathfrak{X}_1=\mathfrak{X}\backslash\Sigma_0$. Following Sotomayor \cite{Soto1974}, we say that $X\in\mathfrak{X}_1$ is of \emph{codimension one} if $X$ is an element of a collection $\Sigma_1\subset\mathfrak{X}_1$ such that:
\begin{enumerate}[label=(\alph*)]
	\item $\Sigma_1$ is a submanifold of codimension one of $\mathfrak{X}$;
	\item $\Sigma_1$ is open and dense in $\mathfrak{X}_1$;
	\item If $Y\in\Sigma_1$, then $Y$ is structurally stable with respect to $\mathfrak{X}_1$.
\end{enumerate}
In simple words, $X$ is of codimension one if it is generic with respect to $\mathfrak{X}_1$. Similarly, one may define vector fields of codimension two, three, etc. For the characterization of the codimension one vector fields over a compact differentiable manifold $M$ of dimension two, see Sotomayor \cite{Soto1974}. Given a family of vector fields $\mathfrak{X}$, there are multiple ways to systematically study its phase portraits. A common one is to study $\mathfrak{X}$ under some extra hypothesis, i.e. assuming upfront some degree of degeneracy. For example, if $\mathfrak{X}$ is the class of the quadratic vector fields, then one may study $\mathfrak{X}$ by means of the number of its finite singularities \cites{Rey0,Rey1,ReyBook}; by assuming some suitable algebraic hypothesis \cites{ArtLli2,LliVal,SchVul}, by intersecting it with some renowned other class of vector fields \cites{ArtLli1,CaoJia,GouLliRob,LliPerPes,LiaZha}, or by studying its phase portraits \cite{PerPes2010,PerPes2012,LiLli,LliSil,CaiLli,LliOli}. But, with the notion of genericity one may also work on $\mathfrak{X}$ from the lower degree of degeneracy to the higher ones. See for example the classification of phase portraits of the quadratic planar vector fields of codimension zero \cite{ArtKooLli}, codimension one \cite{ArtRezLli} and the ongoing classification of those of codimension two \cites{ArtOliRez,ArtMotRez}. In this paper, we focus on the first approach, giving a systematically study of the phase portraits of the system of differential equations given by
	\[\dot x=x(x-1)(a_{00}+a_{10}x+a_{01}y), \quad \dot y=y(y-1)(b_{00}+b_{10}x+b_{01}y),\]
with some generic restrictions following the foundations given by Peixoto's theorem. Before we state our main results, we need first to define a \emph{generic polycycle}. Let $\Gamma\subset\mathbb{R}^2$ be a polycycle  (also known as a graph) composed by $n$ hyperbolic saddles $p_1,\dots p_n$, whose eigenvalues are given by $\mu_i<0<\nu_i$, $i\in\{1,\dots,n\}$. The \emph{hyperbolicity ratio} of $p_i$ is given by $r_i=\frac{|\mu_i|}{\nu_i}$. Cherkas \cite{Cherkas} proved that if $r(\Gamma):=r_1\dots r_n>1$, (resp. $r(\Gamma)<1$), then polycycle $\Gamma$ is stable (resp. unstable). Therefore, we say that $\Gamma$ is generic if $r(\Gamma)\neq1$. From now on, let $\mathfrak{X}$ denote the family of vector fields $X=(P,Q)$ given by
	\[P(x,y)=x(x-1)(a_{00}+a_{10}x+a_{01}y), \quad Q(x,y)=y(y-1)(b_{00}+b_{10}x+b_{01}y),\]
with $a_{ij}$, $b_{ij}\in\mathbb{R}$. Let $\varphi\colon\mathfrak{X}\to\mathbb{R}^6$ be the linear isomorphism given by,
	\[\varphi(X)=(a_{00},a_{10},a_{01},b_{00},b_{10},b_{01}).\]
The \emph{coefficients topology} of $\mathfrak{X}$ is the topology given by the metric,
	\[\rho(X,Y)=||\varphi(X)-\varphi(Y)||,\]
where $||\cdot||$ denotes the standard norm of $\mathbb{R}^6$. Given $X\in\mathfrak{X}$, let $p(X)$ denote its compactification in the Poincar\'e sphere (see Subsection~\ref{SubSec4.1} for more details about the Poincar\'e compactification). Let $\gamma_1,\gamma_2,\gamma_3,\gamma_4\subset\mathbb{S}^2$ be the respective invariant curves of $p(X)$ given by the compactification of the four invariant straight lines,
	\[x=0, \quad x=1, \quad y=0, \quad y=1.\]
Let also $\Lambda=\gamma_1\cup\gamma_2\cup\gamma_3\cup\gamma_4\cup\mathbb{S}^1$, where $\mathbb{S}^1$ denotes the infinity of the Poincar\'e sphere. Following the foundations given by Peixoto's Theorem, we suggest the following definition.

\begin{definition}\label{Def1} 
	A vector field $X\in\mathfrak{X}$ is generic if its compactification $p(X)$ satisfies the following conditions.
\begin{enumerate}[label=(\alph*)]
	\item There is only a finite number of singularities, all hyperbolic;
	\item There is only a finite number of closed orbits, all hyperbolic;
	\item There is no connections between saddles, except if such connection is contained in $\Lambda$;
	\item The $\alpha$ and $\omega$-limits of every orbit is either a singularity, a closed orbit or a generic polycycle contained in $\Lambda$.
\end{enumerate}
\end{definition}
Let $\Sigma_0$ denote the family of the generic vector fields in $\mathfrak{X}$. In our first main result we give a necessary condition for a vector field $X\in\mathfrak{X}$ to be generic. Given $X\in\Sigma_0$, such result will enable us to work with an equivalent version of $X$.

\begin{theorem}\label{Main1}
	Let $X\in\mathfrak{X}$ be given by,
		\[P(x,y)=x(x-1)(a_{00}+a_{10}x+a_{01}y), \quad Q(x,y)=y(y-1)(b_{00}+b_{10}x+b_{01}y),\]
	with $a_{ij}$, $b_{ij}\in\mathbb{R}$. Let also,
		\[A=\left(\begin{array}{cc} a_{10} & a_{01} \\ b_{10} & b_{01} \end{array}\right).\]
	If $X\in\Sigma_0$, then $a_{10}b_{01}\det A\neq 0$.
\end{theorem} 

Given $X\in\Sigma_0$, we can associate the system of differential equations,
\begin{equation}\label{42}
	\dot x=x(x-1)(a_{00}+a_{10}x+a_{01}y), \quad \dot y=y(y-1)(b_{00}+b_{10}x+b_{01}y).
\end{equation}
It follows from Theorem~\ref{Main1} that we can assume $\det A\neq0$. Therefore, let,
	\[A_1=\left(\begin{array}{cc} -a_{00} & a_{01} \\ -b_{00} & b_{01} \end{array}\right), \quad A_2=\left(\begin{array}{cc} a_{10} & -a_{00} \\ b_{10} & -b_{00} \end{array}\right).\]
It follows from Cramer's rule that the unique solution of the linear system,
	\[\left(\begin{array}{cc} a_{10} & a_{01} \\ b_{10} & b_{01} \end{array}\right)\left(\begin{array}{c} x \\ y \end{array}\right) = \left(\begin{array}{c} -a_{00} \\ -b_{00} \end{array}\right),\]
is given by,
\begin{equation}\label{5}
	p=\left(\frac{\det A_1}{\det A},\frac{\det A_2}{\det A}\right).
\end{equation}
Writing $p=(p_1,p_2)$ and applying the change of coordinates $u=x-p_1$, $v=y-p_2$, it follows that system \eqref{42} is equivalent to
	\[\dot u=(u+p_1)(u+p_1-1)(a_{10}u+a_{01}v), \quad \dot v=(v+p_2)(v+p_2-1)(b_{10}u+b_{01}v).\]
Therefore, if $X\in\Sigma_0$, then we can study its equivalent system of differential equations,
\begin{equation}\label{6}
	\dot x=(x+\alpha)(x+\alpha-1)(a_{10}x+a_{01}y), \quad \dot y=(y+\beta)(y+\beta-1)(b_{10}x+b_{01}y),
\end{equation}
with $\alpha$, $\beta$, $a_{ij}$, $b_{ij}\in\mathbb{R}$. Observe that system \eqref{6} can be divided in nine cases, given by the position of the origin in relation to the octothorpe defined by the four invariant lines, 
\begin{equation}\label{7}
	x=-\alpha, \quad x=1-\alpha, \quad y=-\beta, \quad y=1-\beta.
\end{equation}
More precisely, those nine cases are given by,
\begin{enumerate}[label=\arabic*)]\label{GenericCases}
	\item $0\leqslant\alpha\leqslant1$ and $0\leqslant\beta\leqslant1$, i.e. the origin lies in the \textbf{center} of the octothorpe; 
	\item $1\leqslant\alpha$ and $0\leqslant\beta\leqslant1$, i.e. the origin lies in the \textbf{middle-right} side of the octothorpe; 
	\item $1\leqslant\alpha$ and $1\leqslant\beta$, i.e. the origin lies in the \textbf{top-right} corner of the octothorpe; 	
	\item $0\leqslant\alpha\leqslant1$ and $1\leqslant\beta$, i.e. the origin lies in the \textbf{top-middle} side of the octothorpe;	
	\item $\alpha\leqslant0$ and $\beta\geqslant1$, i.e. the origin lies in the \textbf{top-left} corner of the octothorpe;		
	\item $\alpha\leqslant0$ and $0\leqslant\beta\leqslant1$, i.e. the origin lies in the \textbf{middle-left} side of the octothorpe; 	
	\item $\alpha\leqslant0$ and $\beta\leqslant0$, i.e. the origin lies in the \textbf{bottom-left} corner of the octothorpe; 		
	\item $0\leqslant\alpha\leqslant1$ and $\beta\leqslant0$, i.e. the origin lies in the \textbf{bottom-middle} side of the octothorpe;		
	\item $1\leqslant\alpha$ and $\beta\leqslant0$, i.e. the origin lies in the \textbf{bottom-right} corner of the octothorpe.
\end{enumerate}

Hence, our second main result states that it is enough to study only three of those nine positions, given by whether the origin is at the center, the side or the corner of the octothorpe. It also states a set of sufficient conditions for each of these three positions. Given two sets $M$, $N\subset\mathfrak{X}$, we say that $M$ and $N$ are \emph{topologically equivalent} if for every $X\in M$ there exists $Y\in N$ such that $X$ and $Y$ are topologically equivalent.

\begin{theorem}\label{Main2}
	Regarding the positions of the origin in relation to the octothorpe given by the four invariant straight lines, the following statements holds.
		\begin{enumerate}[label=(\alph*)]
			\item Positions $3$, $5$, $7$ and $9$ are topologically equivalents;
			\item Positions $2$, $4$, $6$ and $8$ are topologically equivalents.
		\end{enumerate}
		Hence, to classify all vector fields in $\Sigma_0$, it is enough to study positions $1$, $2$ and $3$. Furthermore, within positions $1$, $2$ and $3$, it is enough to study the families given by Tables~\ref{Table8}, \ref{Table9} and \ref{Table10}, respectively.
	\begin{table}[h]
		\caption{Position $1$.}\label{Table8}
		\begin{tabular}{c c c c c c c}
			\hline
			\rowcolor{mygray}
			Family $1$ & $0<\alpha<1$ & $\frac{1}{2}\leqslant\beta<1$ & $a_{10}>0$ & $a_{01}\geqslant0$ & $b_{10}\geqslant0$ & $b_{01}>0$ \vspace{0.1cm} \\
			Family $2$ & $0<\alpha<1$ & $\frac{1}{2}\leqslant\beta<1$ & $a_{10}>0$ & $a_{01}\geqslant0$ & $b_{10}\geqslant0$ & $b_{01}<0$ \vspace{0.1cm} \\
			\rowcolor{mygray}
			Family $3$ & $0<\alpha<1$ & $\frac{1}{2}\leqslant\beta<1$ & $a_{10}>0$ & $a_{01}\geqslant0$ & $b_{10}<0$ & $b_{01}>0$ \vspace{0.1cm} \\
			Family $4$ & $0<\alpha<1$ & $\frac{1}{2}\leqslant\beta<1$ & $a_{10}>0$ & $a_{01}\geqslant0$ & $b_{10}<0$ & $b_{01}<0$ \\
			\hline
		\end{tabular}
	\end{table}
	\begin{table}[h]
		\caption{Position $2$.}\label{Table9}
		\begin{tabular}{c c c c c c c}
			\hline
			\rowcolor{mygray}
			Family $1$ & $\alpha>1$ & $0<\beta<1$ & $a_{10}>0$ & $a_{01}\geqslant0$ & $b_{10}\geqslant0$ & $b_{01}>0$ \\
			Family $2$ & $\alpha>1$ & $0<\beta<1$ & $a_{10}>0$ & $a_{01}\geqslant0$ & $b_{10}\geqslant0$ & $b_{01}<0$ \\
			\rowcolor{mygray}
			Family $3$ & $\alpha>1$ & $0<\beta<1$ & $a_{10}>0$ & $a_{01}\geqslant0$ & $b_{10}<0$ & $b_{01}>0$ \\
			Family $4$ & $\alpha>1$ & $0<\beta<1$ & $a_{10}>0$ & $a_{01}\geqslant0$ & $b_{10}<0$ & $b_{01}<0$ \\
			\hline
		\end{tabular}
	\end{table}
	\begin{table}[h]
		\caption{Position $3$.}\label{Table10}
		\begin{tabular}{c c c c c c c}
			\hline
			\rowcolor{mygray}
			Family $1$ & $\alpha>1$ & $\beta>1$ & $a_{10}>0$ & $a_{01}\geqslant0$ & $b_{10}\geqslant0$ & $b_{01}>0$ \\
			Family $2$ & $\alpha>1$ & $\beta>1$ & $a_{10}>0$ & $a_{01}\geqslant0$ & $b_{10}\geqslant0$ & $b_{01}<0$ \\
			\rowcolor{mygray}
			Family $3$ & $\alpha>1$ & $\beta>1$ & $a_{10}>0$ & $a_{01}\geqslant0$ & $b_{10}<0$ & $b_{01}>0$ \\
			Family $4$ & $\alpha>1$ & $\beta>1$ & $a_{10}>0$ & $a_{01}\geqslant0$ & $b_{10}<0$ & $b_{01}<0$ \\
			\rowcolor{mygray}
			Family $5$ & $\alpha>1$ & $\beta>1$ & $a_{10}<0$ & $a_{01}\geqslant0$ & $b_{10}\geqslant0$ & $b_{01}<0$ \\
			Family $6$ & $\alpha>1$ & $\beta>1$ & $a_{10}<0$ & $a_{01}\geqslant0$ & $b_{10}\leqslant0$ & $b_{01}>0$ \\
			\hline
		\end{tabular}
	\end{table}
\end{theorem}

From the three sufficient positions given by Theorem~\ref{Main2}, it is clear that position $1$ is the richest one for the model of Evolutionary Stable Strategies, since the origin stays in the region of the plane in which the model is defined. Therefore, in our third main result, we classify all the possible phase portraits of $X\in\Sigma_0$ under position $1$.

\begin{theorem}\label{Main3}
	Let $X\in\Sigma_0$ be given by
		\[P(x,y)=x(x-1)(a_{00}+a_{10}x+a_{01}y), \quad Q(x,y)=y(y-1)(b_{00}+b_{10}x+b_{01}y),\]
	with $a_{ij}$, $b_{ij}\in\mathbb{R}$. Let also $(p_1,p_2)\in\mathbb{R}^2$ be the unique solution of the linear system,
		\[\left(\begin{array}{cc} a_{10} & a_{01} \\ b_{10} & b_{01} \end{array}\right)\left(\begin{array}{c} x \\ y \end{array}\right)=\left(\begin{array}{c} -a_{00} \\ -b_{00} \end{array}\right).\]
	If $(p_1,p_2)\in[0,1]\times[0,1]$, then the phase portrait of $X$ in the Poincar\'e disk is topologically equivalent to one of the twenty five phase portraits given by Figure~\ref{GenericFinal}. Moreover, all phase portraits are realizable.
\begin{figure}[h]
	\begin{center}
		\begin{minipage}{3.1cm}
			\begin{center}
				\begin{overpic}[height=3cm]{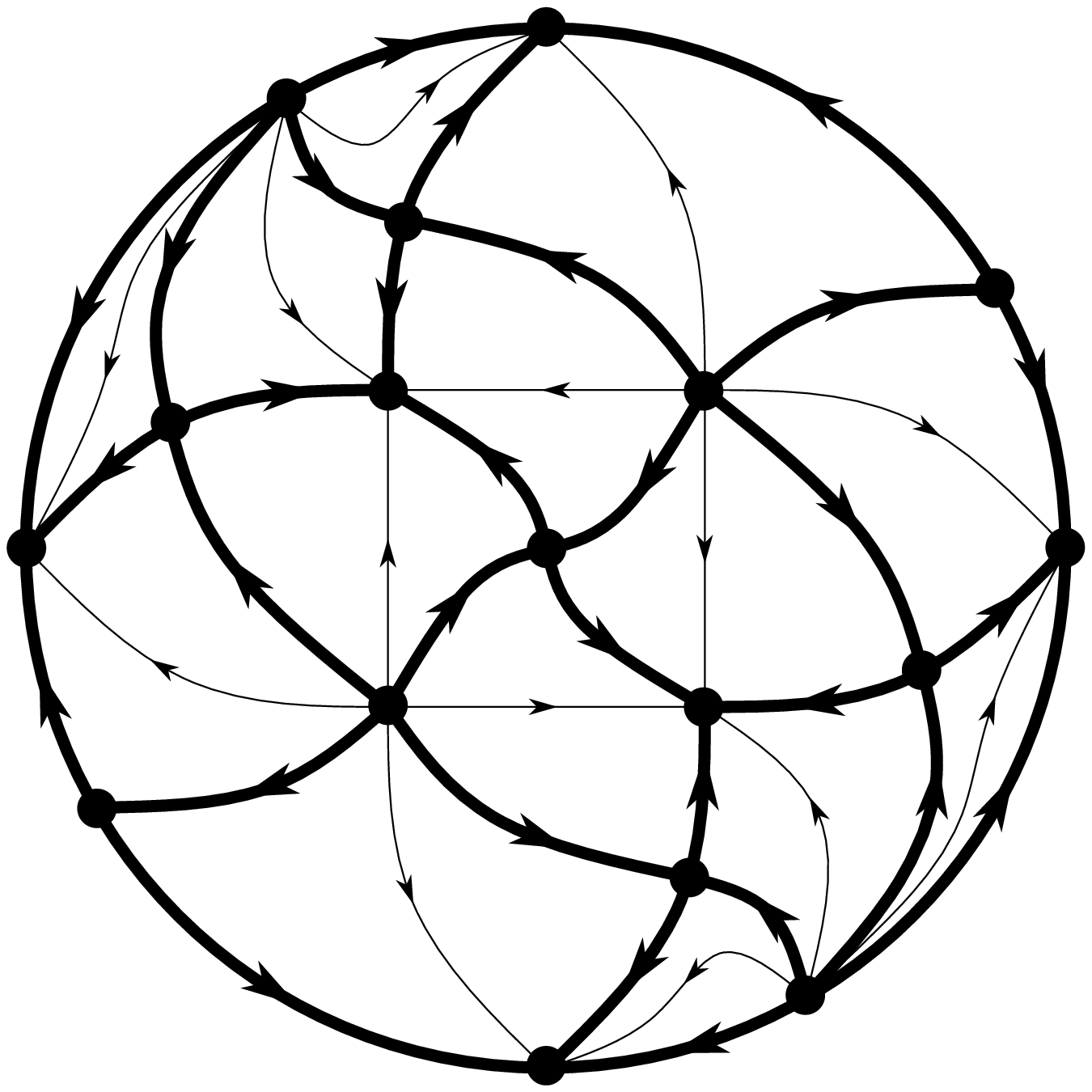} 
				\end{overpic}
				
				Case~$1.1$.
			\end{center}
		\end{minipage}
		\begin{minipage}{3.1cm}
			\begin{center}
				\begin{overpic}[height=3cm]{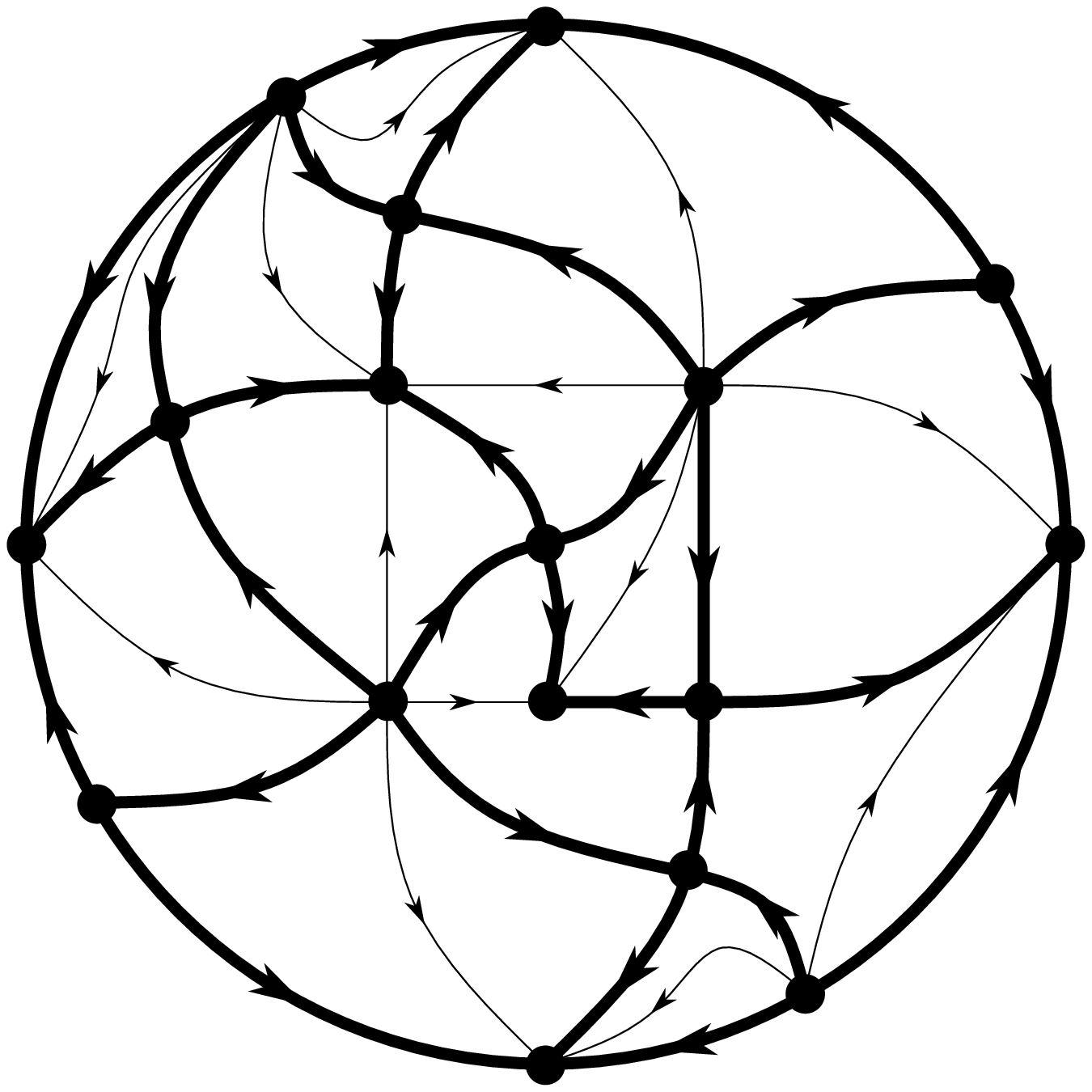} 
				\end{overpic}
				
				Case~$1.2$.
			\end{center}
		\end{minipage}
		\begin{minipage}{3.1cm}
			\begin{center}
				\begin{overpic}[height=3cm]{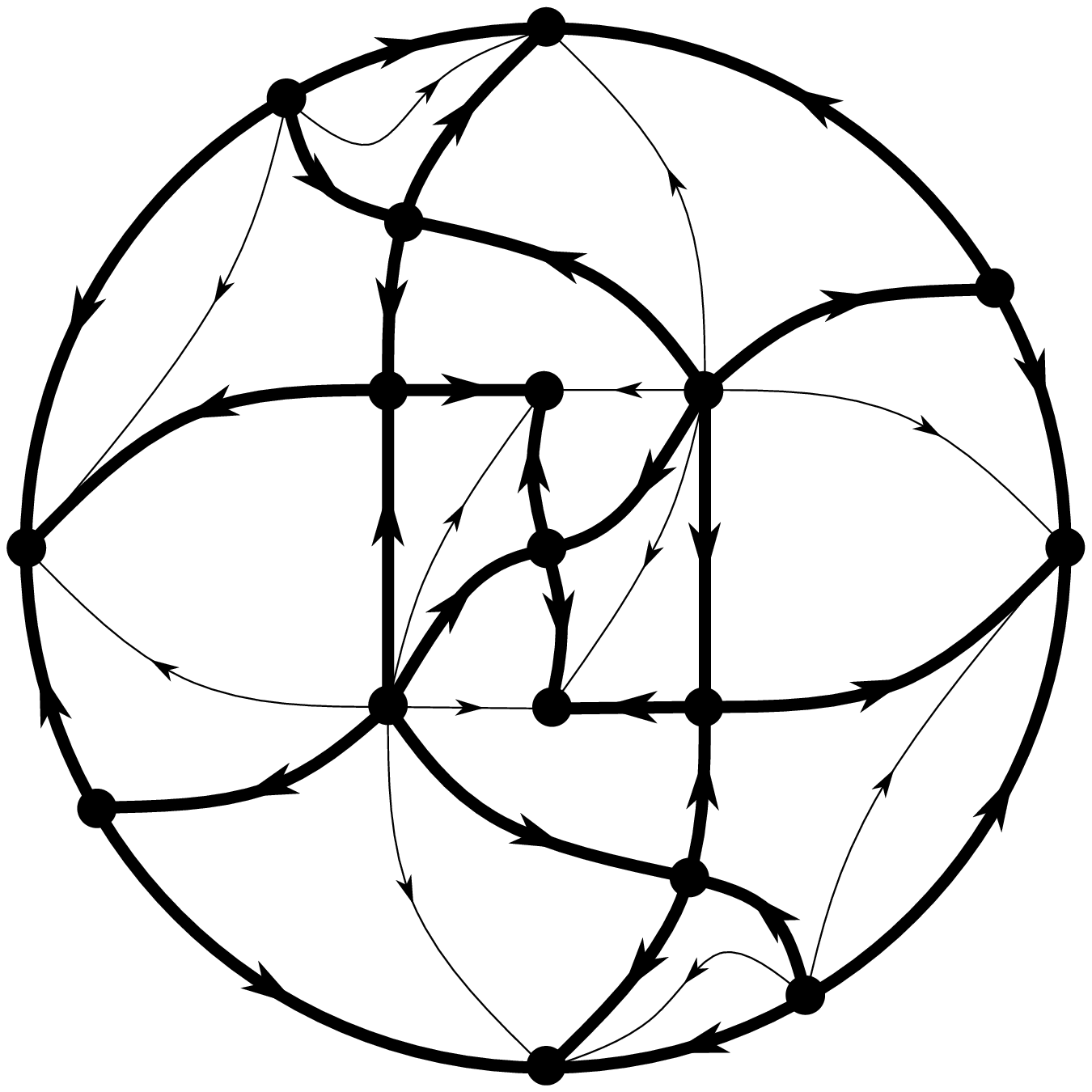} 
				\end{overpic}
				
				Case~$1.4a$.
			\end{center}
		\end{minipage}
		\begin{minipage}{3.1cm}
			\begin{center}
				\begin{overpic}[height=3cm]{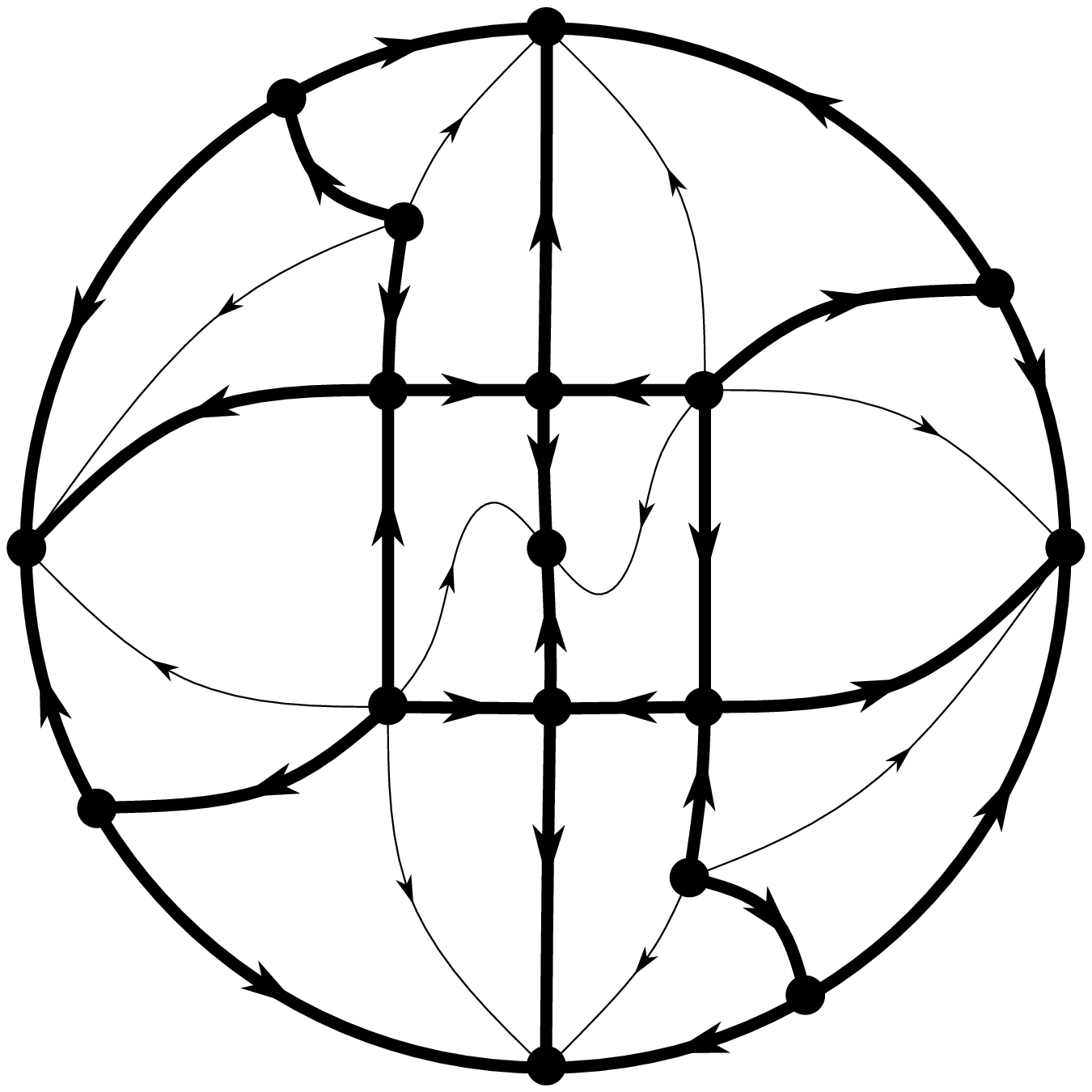} 
				\end{overpic}
				
				Case~$1.4b$.
			\end{center}
		\end{minipage}
		\begin{minipage}{3.1cm}
			\begin{center}
				\begin{overpic}[height=3cm]{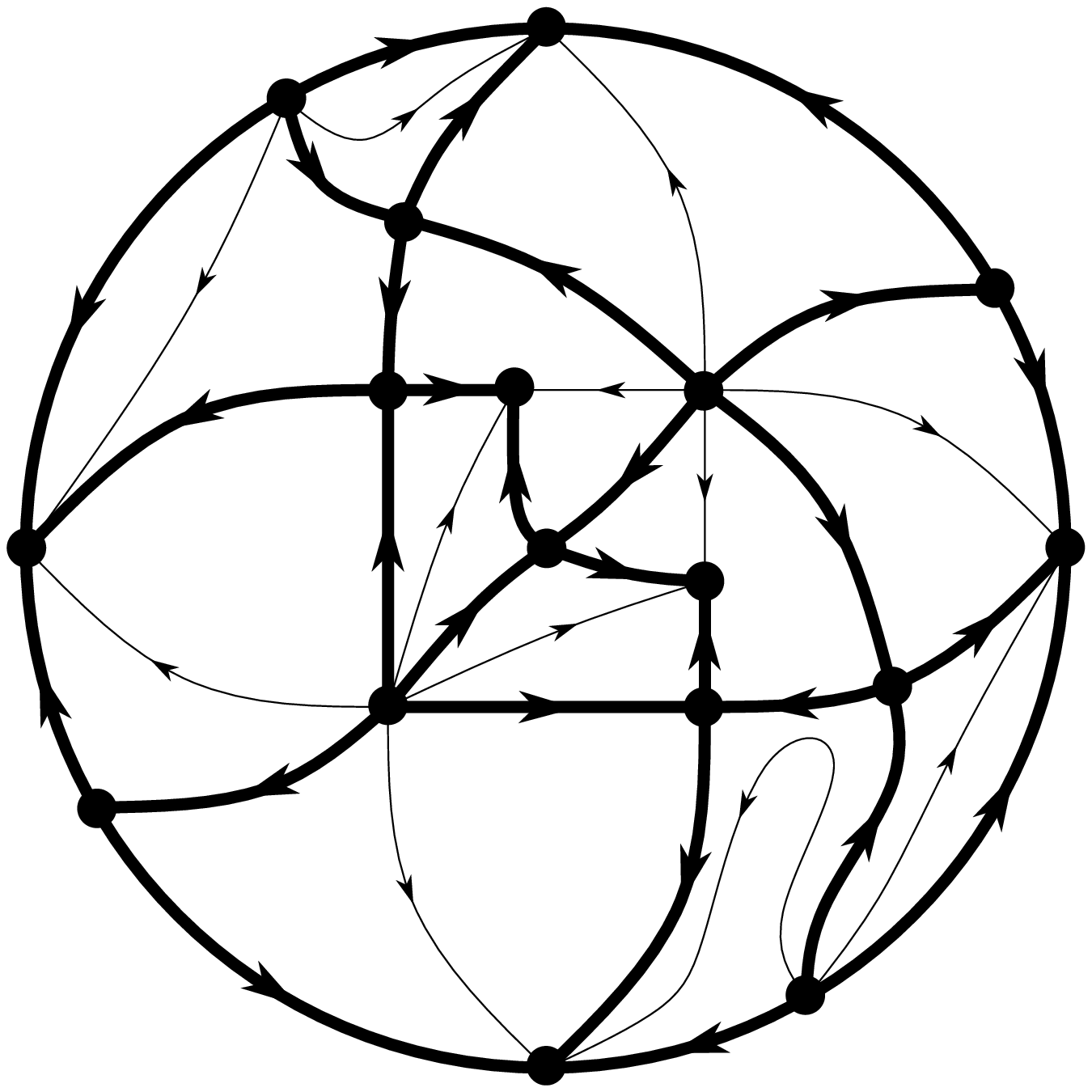} 
				\end{overpic}
				
				Case~$1.6a$.
			\end{center}
		\end{minipage}
	\end{center}
$\;$
	\begin{center}
		\begin{minipage}{3.1cm}
			\begin{center}
				\begin{overpic}[height=3cm]{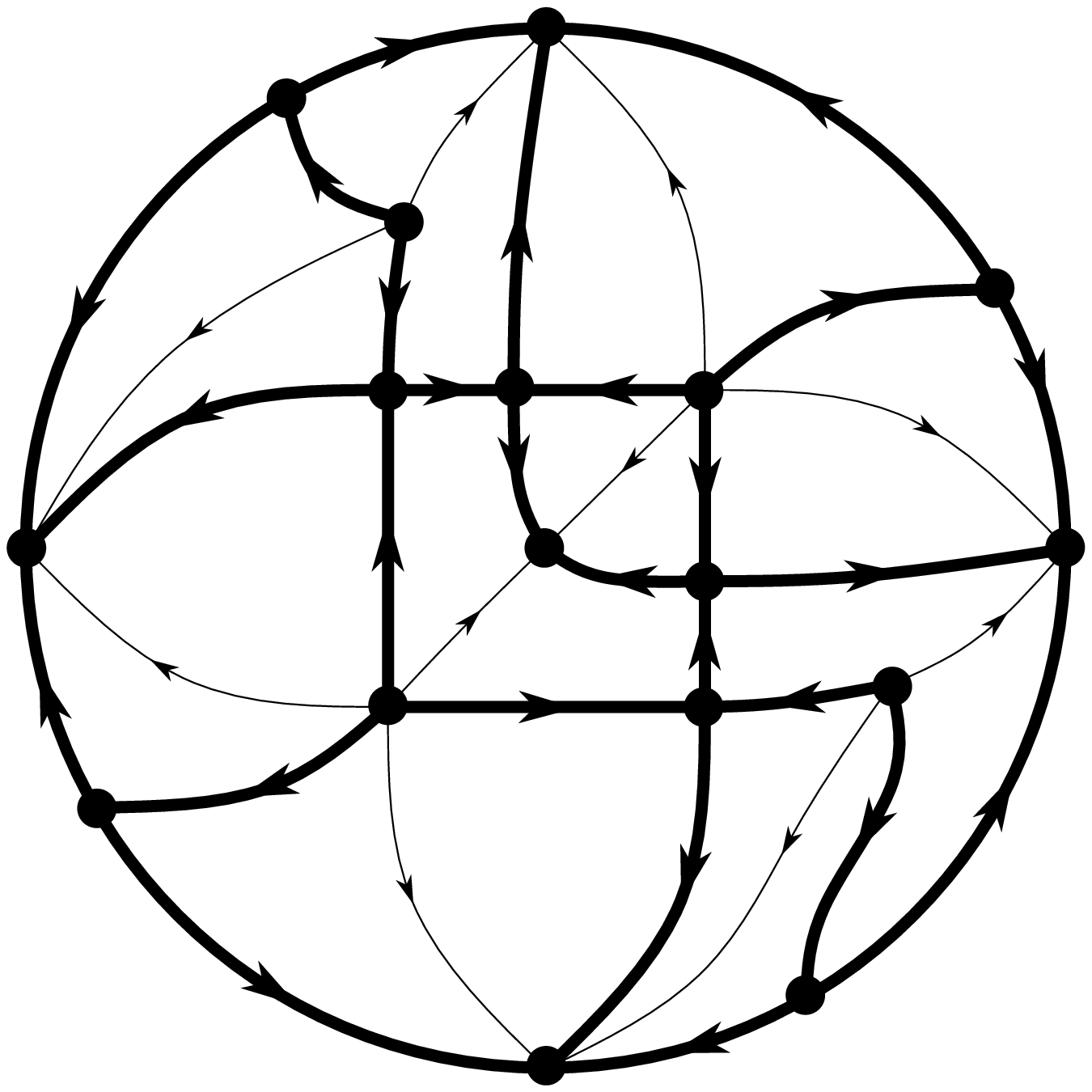} 
				\end{overpic}
				
				Case~$1.6b$.
			\end{center}
		\end{minipage}
		\begin{minipage}{3.1cm}
			\begin{center}
				\begin{overpic}[height=3cm]{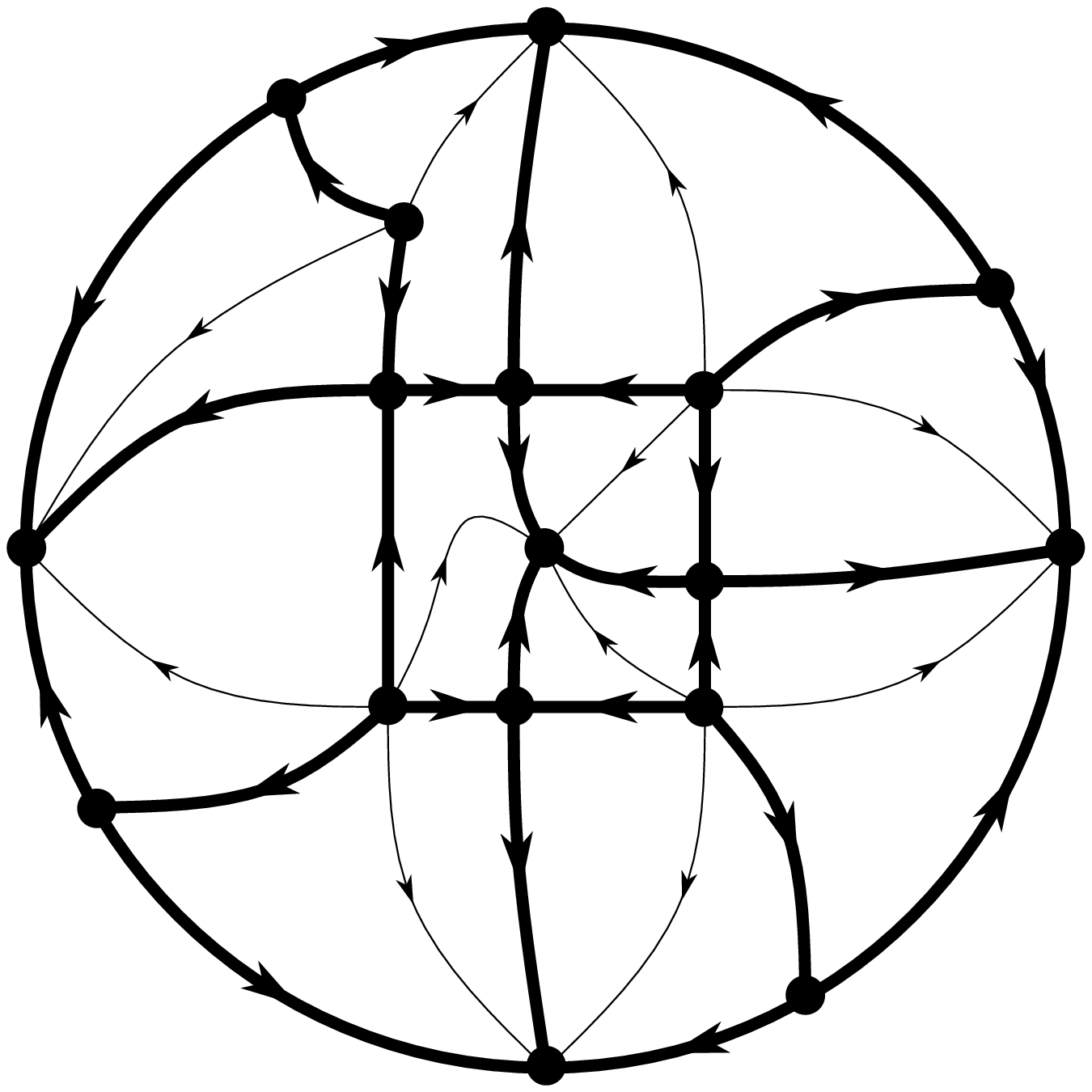} 
				\end{overpic}
				
				Case~$1.7$.
			\end{center}
		\end{minipage}
		\begin{minipage}{3.1cm}
			\begin{center}
				\begin{overpic}[height=3cm]{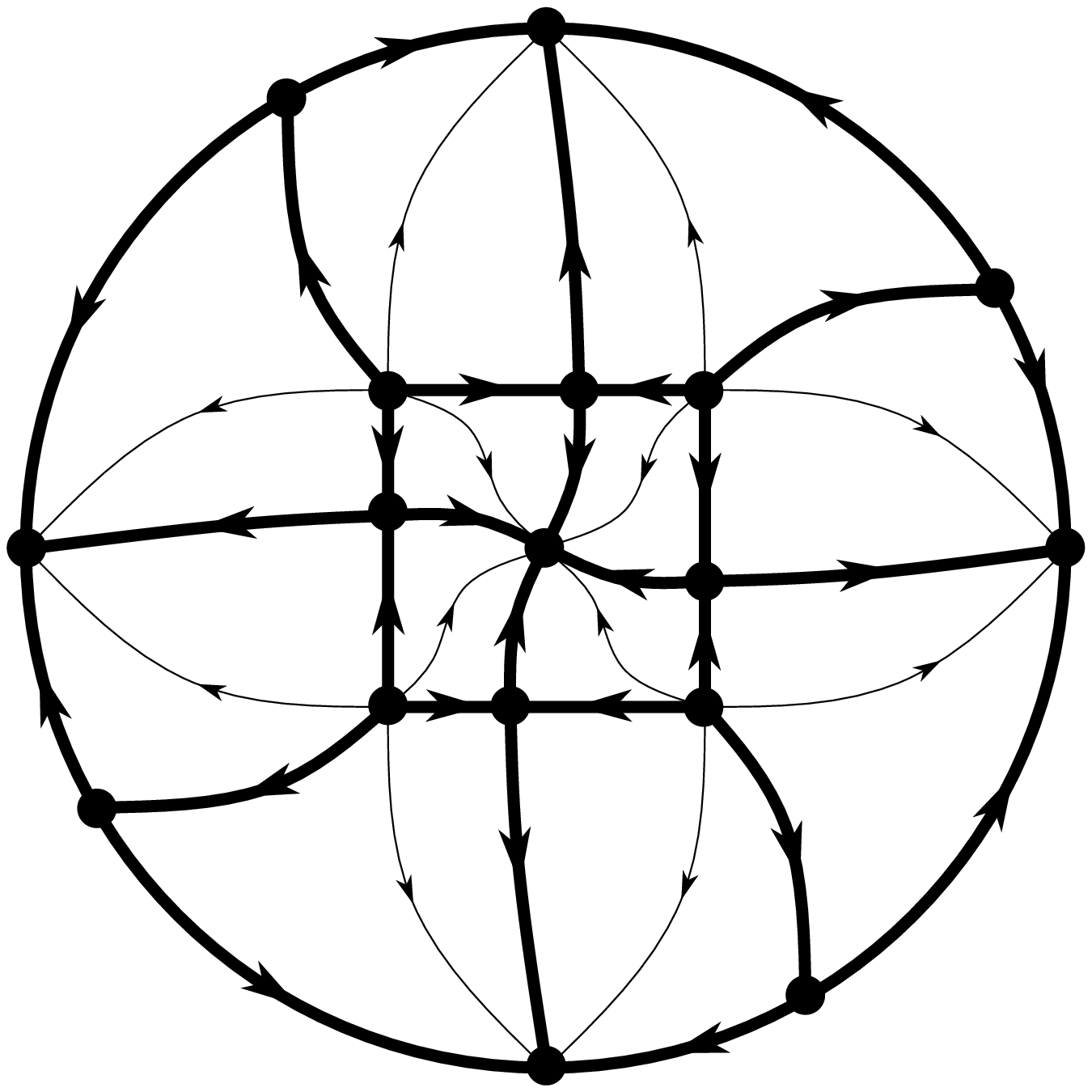} 
				\end{overpic}
				
				Case~$1.14$.
			\end{center}
		\end{minipage}
		\begin{minipage}{3.1cm}
			\begin{center}
				\begin{overpic}[height=3cm]{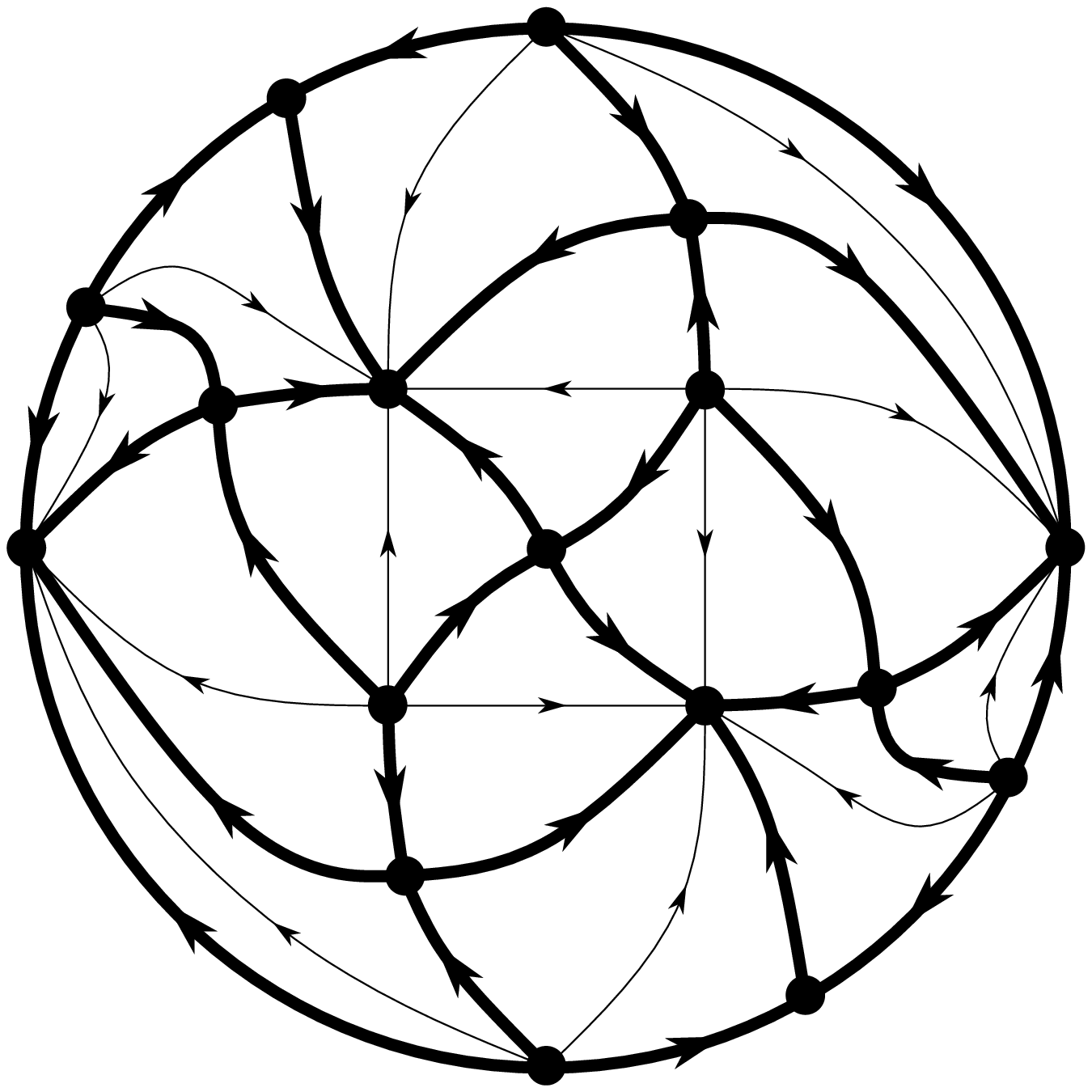} 
				\end{overpic}
				
				Case~$2.1a1$.
			\end{center}
		\end{minipage}
		\begin{minipage}{3.1cm}
			\begin{center}
				\begin{overpic}[height=3cm]{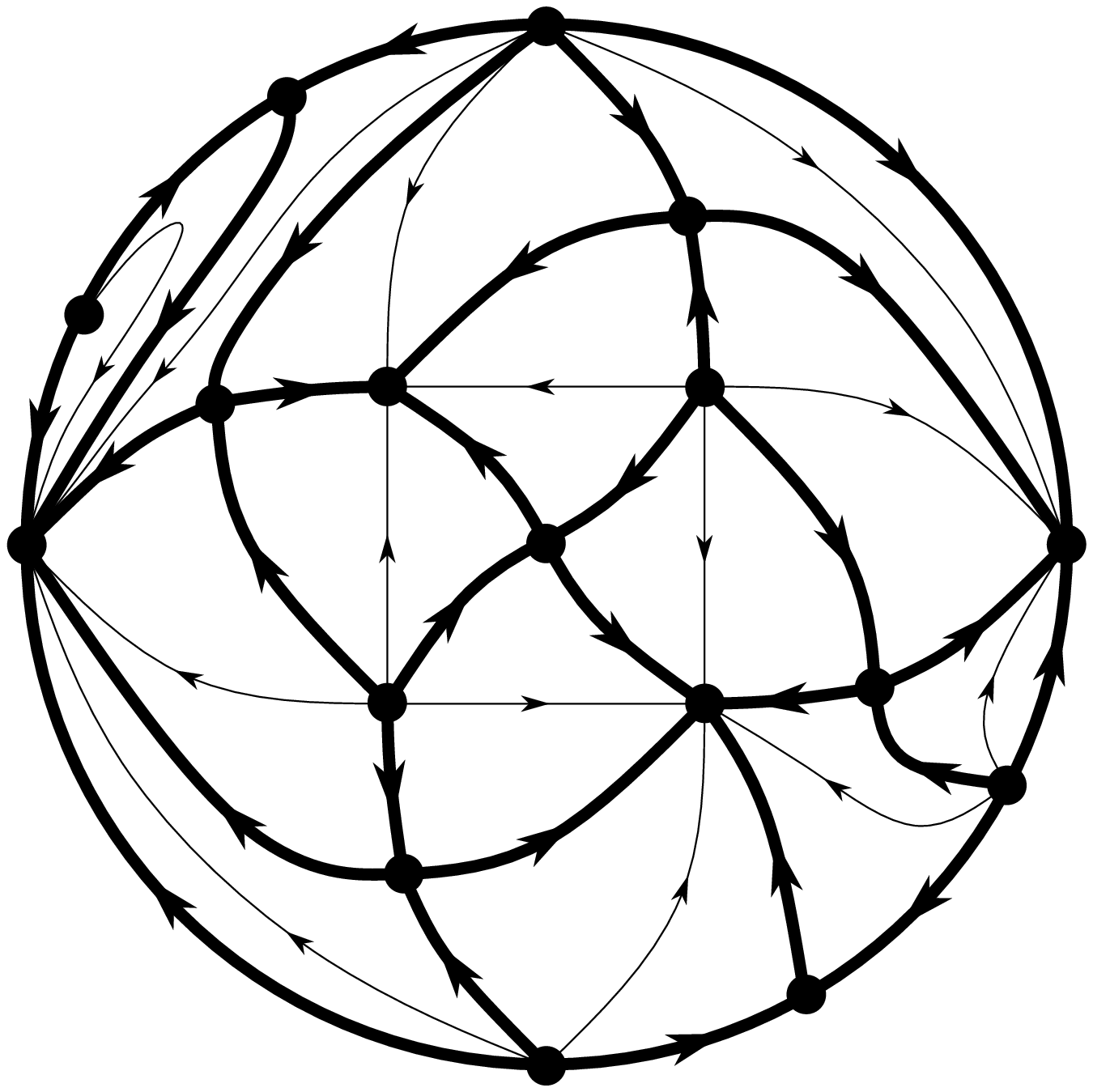} 
				\end{overpic}
				
				Case~$2.1a3$.
			\end{center}
		\end{minipage}
	\end{center}
$\;$
	\begin{center}
		\begin{minipage}{3.1cm}
			\begin{center}
				\begin{overpic}[height=3cm]{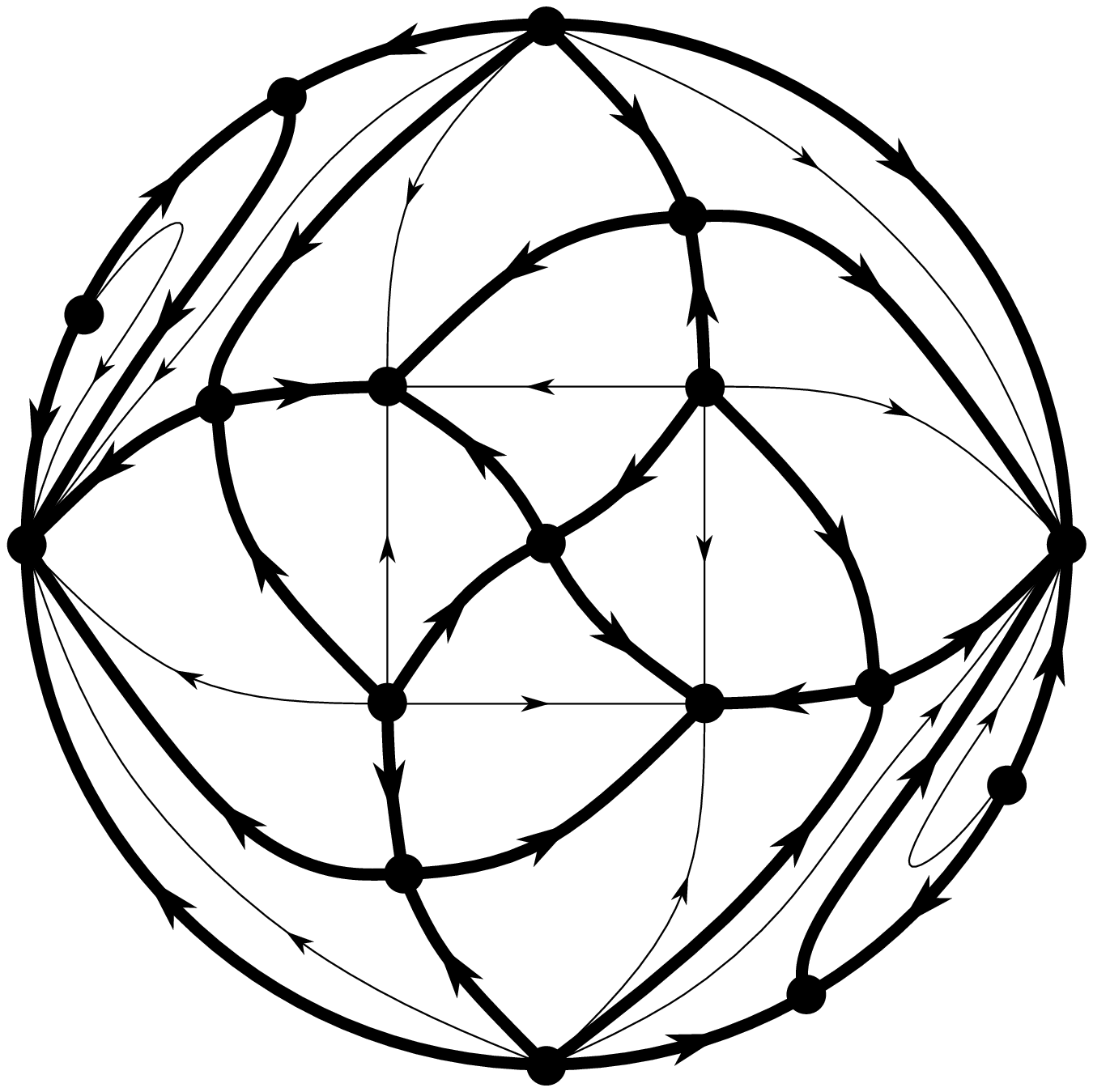} 
				\end{overpic}
				
				Case~$2.1a4$.
			\end{center}
		\end{minipage}	
		\begin{minipage}{3.1cm}
			\begin{center}
				\begin{overpic}[height=3cm]{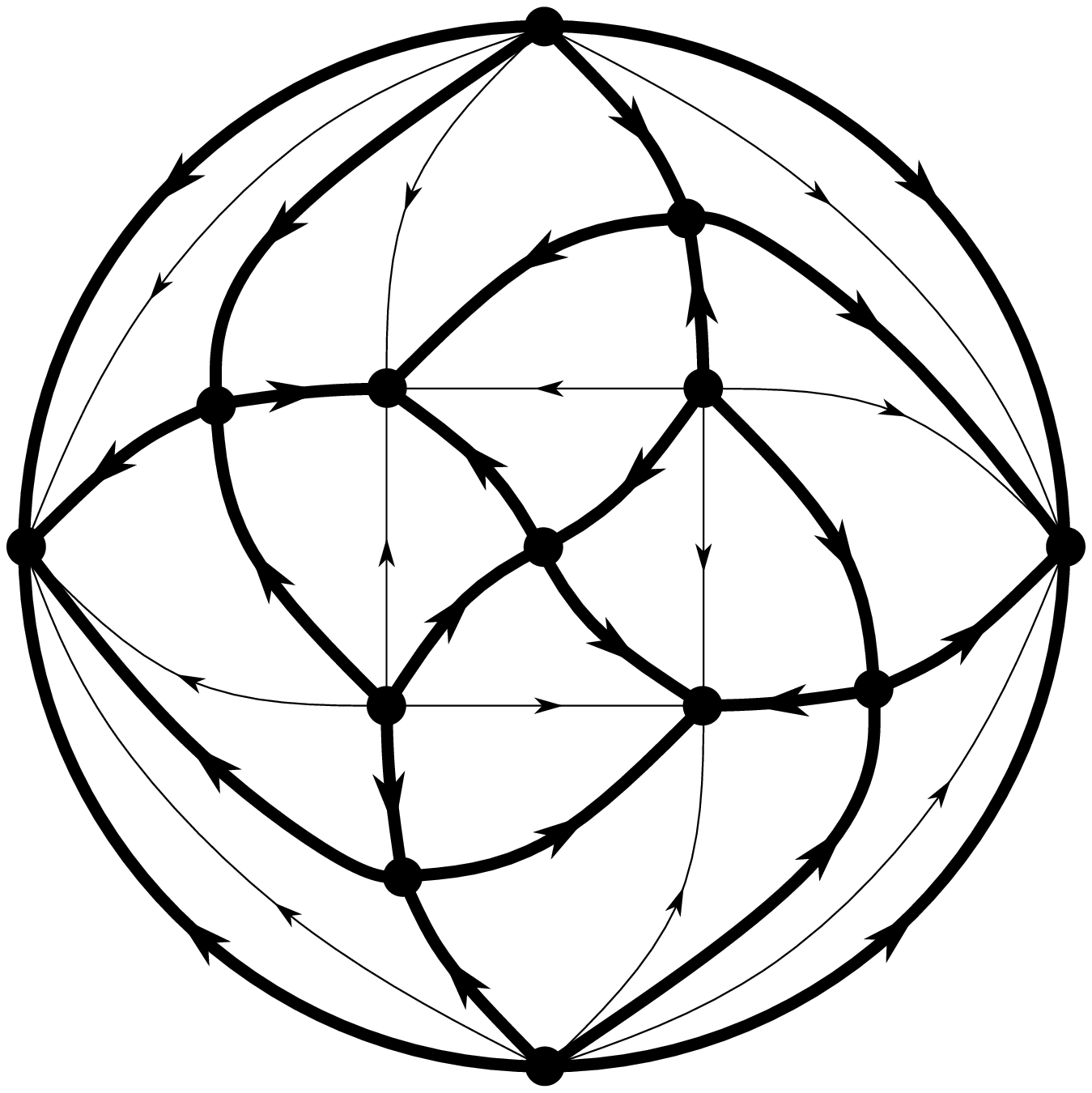} 
				\end{overpic}
				
				Case~$2.1c$.
			\end{center}
		\end{minipage}	
		\begin{minipage}{3.1cm}
			\begin{center}
				\begin{overpic}[height=3cm]{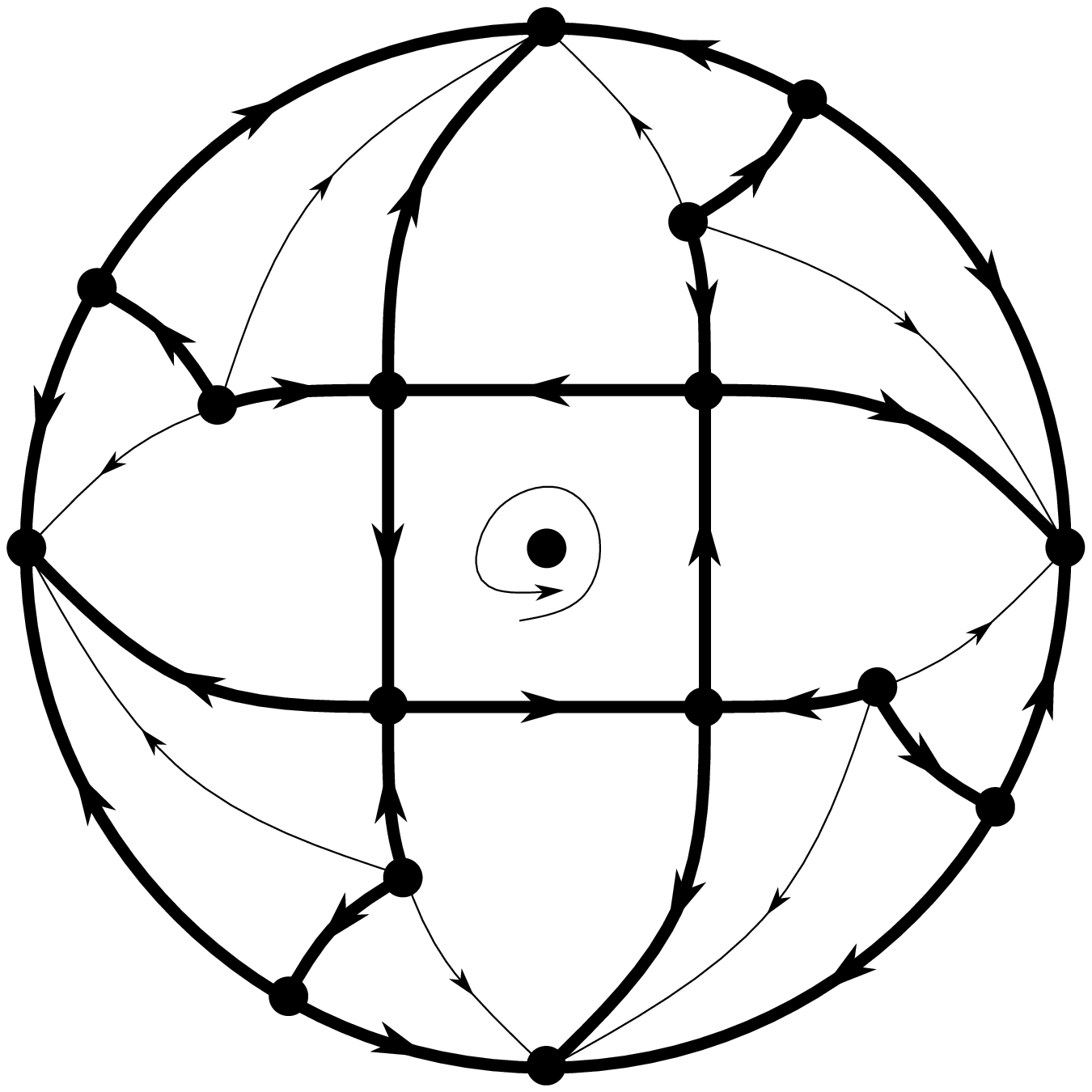} 
				\end{overpic}
				
				Case~$3.1$.
			\end{center}
		\end{minipage}
		\begin{minipage}{3.1cm}
			\begin{center}
				\begin{overpic}[height=3cm]{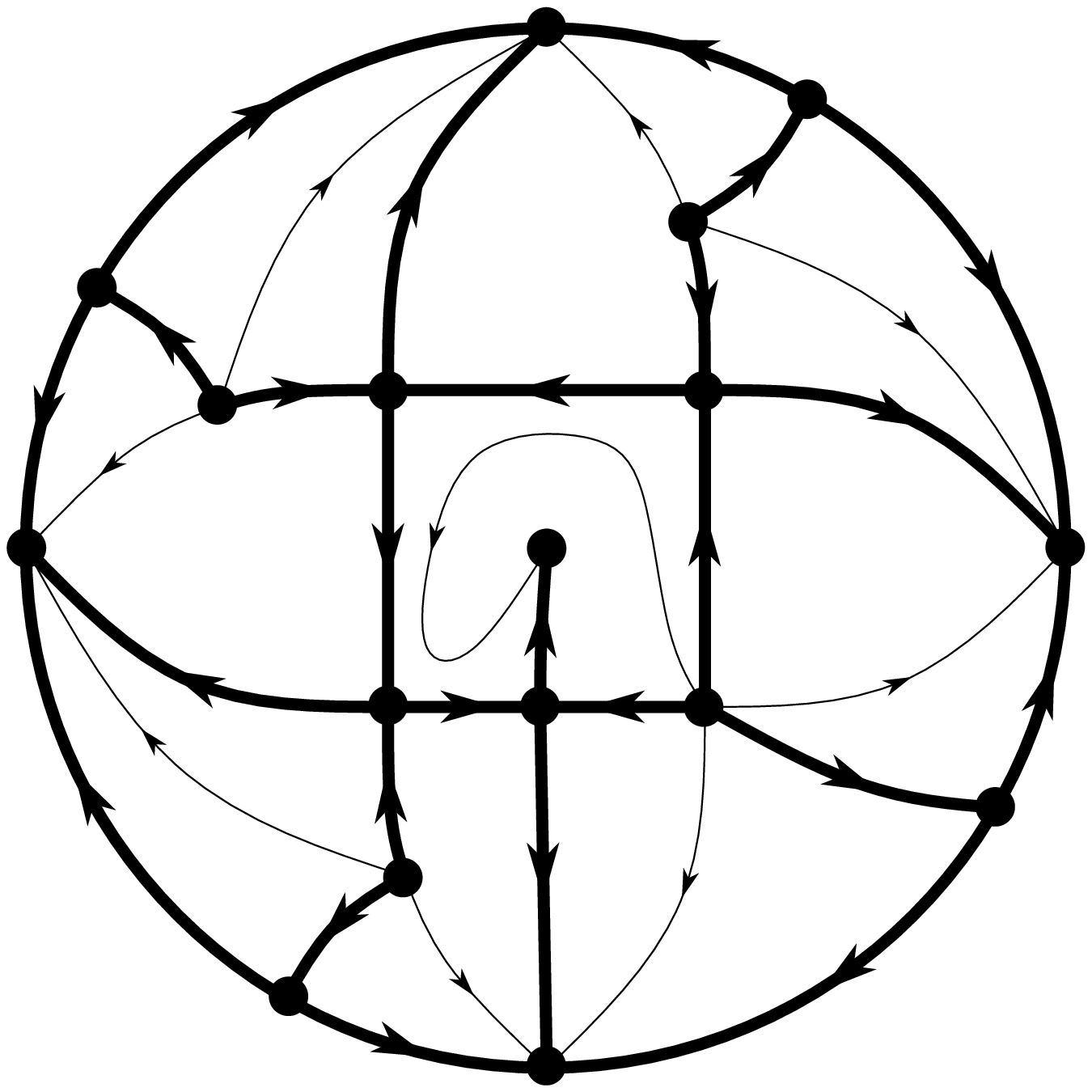} 
				\end{overpic}
				
				Case~$3.2$.
			\end{center}
		\end{minipage}
		\begin{minipage}{3.1cm}
			\begin{center}
				\begin{overpic}[height=3cm]{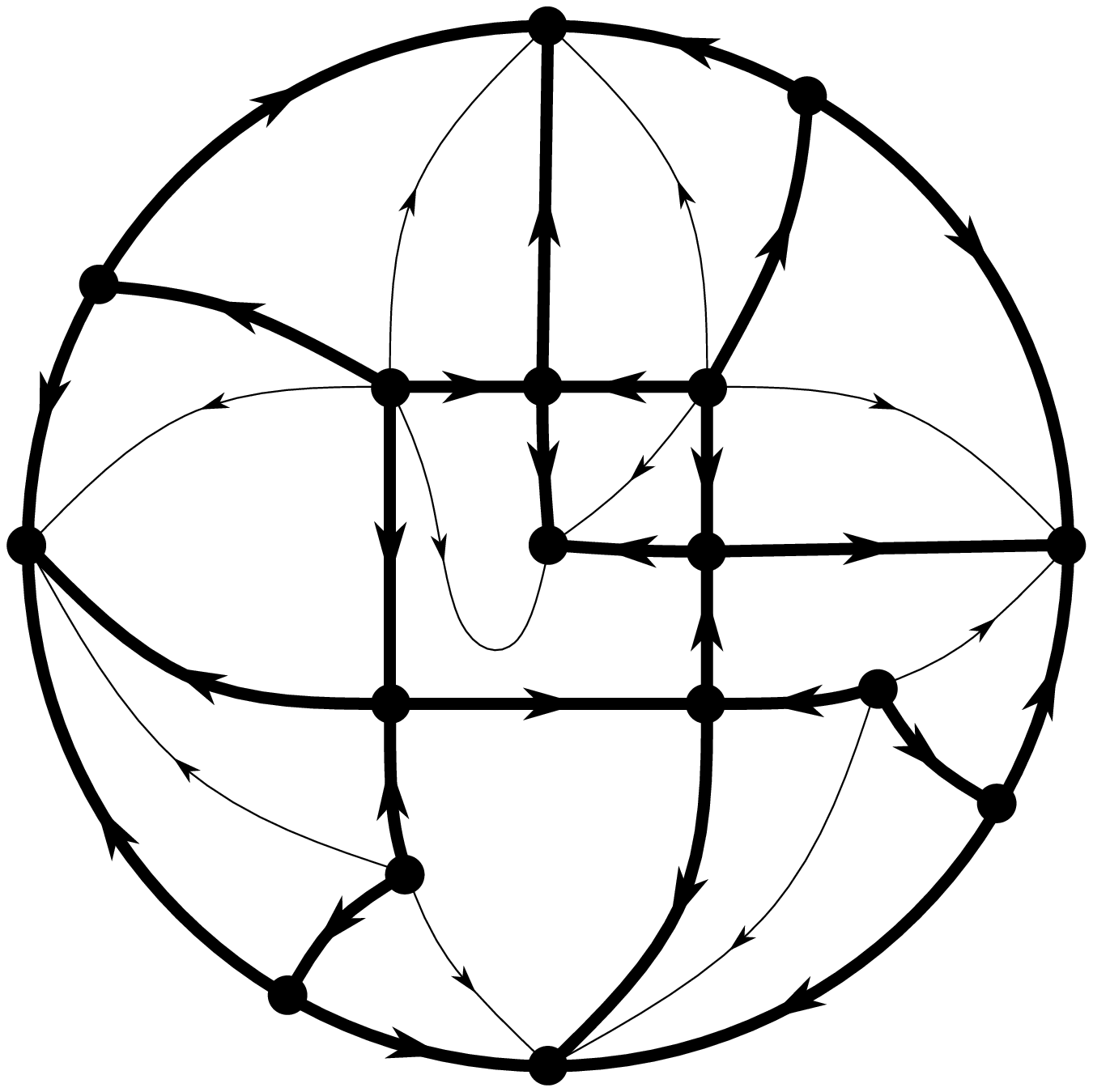} 
				\end{overpic}
				
				Case~$3.6$.
			\end{center}
		\end{minipage}
	\end{center}
$\;$
	\begin{center}
		\begin{minipage}{3.1cm}
			\begin{center}
				\begin{overpic}[height=3cm]{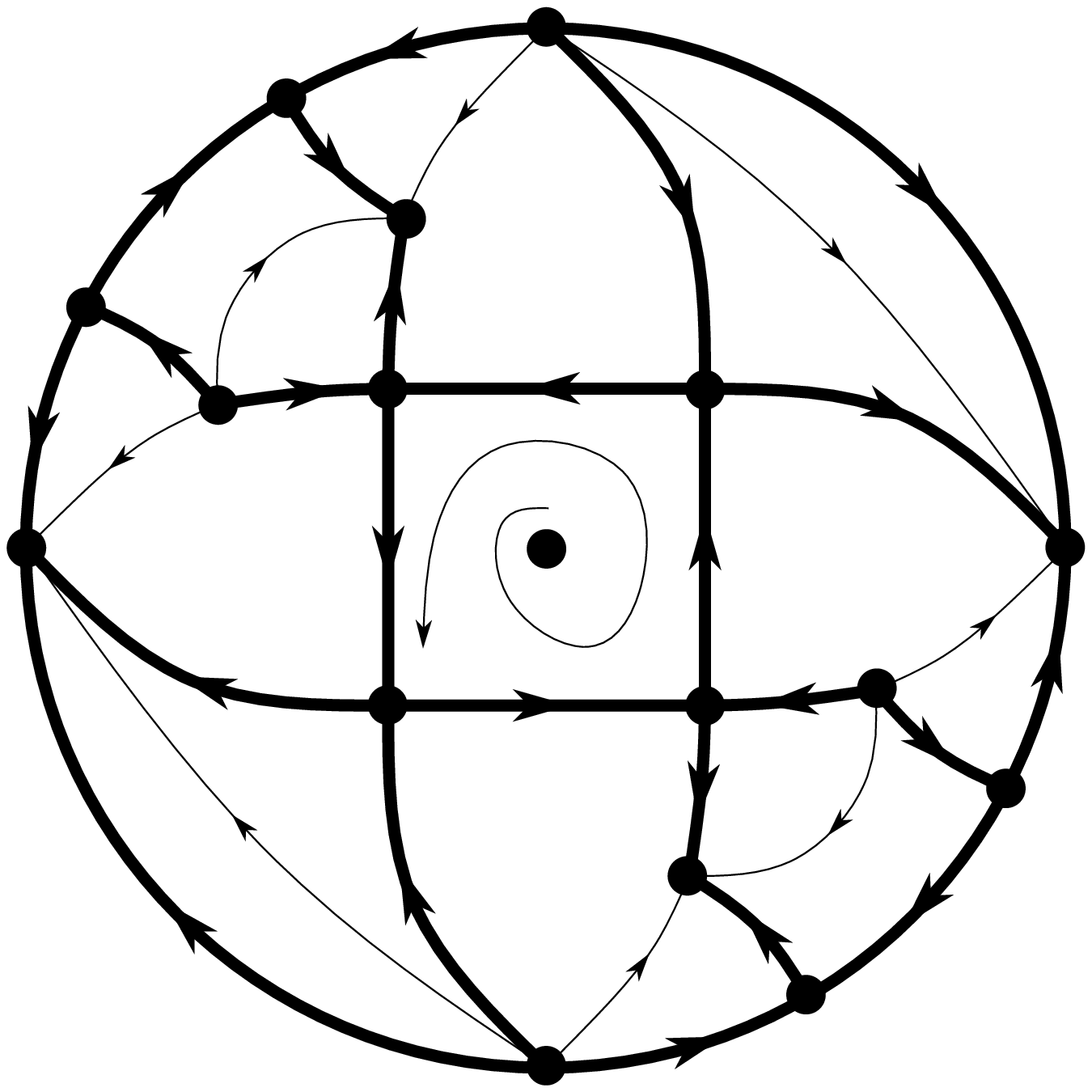} 
				\end{overpic}
				
				Case~$4.1a.i$.
			\end{center}
		\end{minipage}
		\begin{minipage}{3.1cm}
			\begin{center}
				\begin{overpic}[height=3cm]{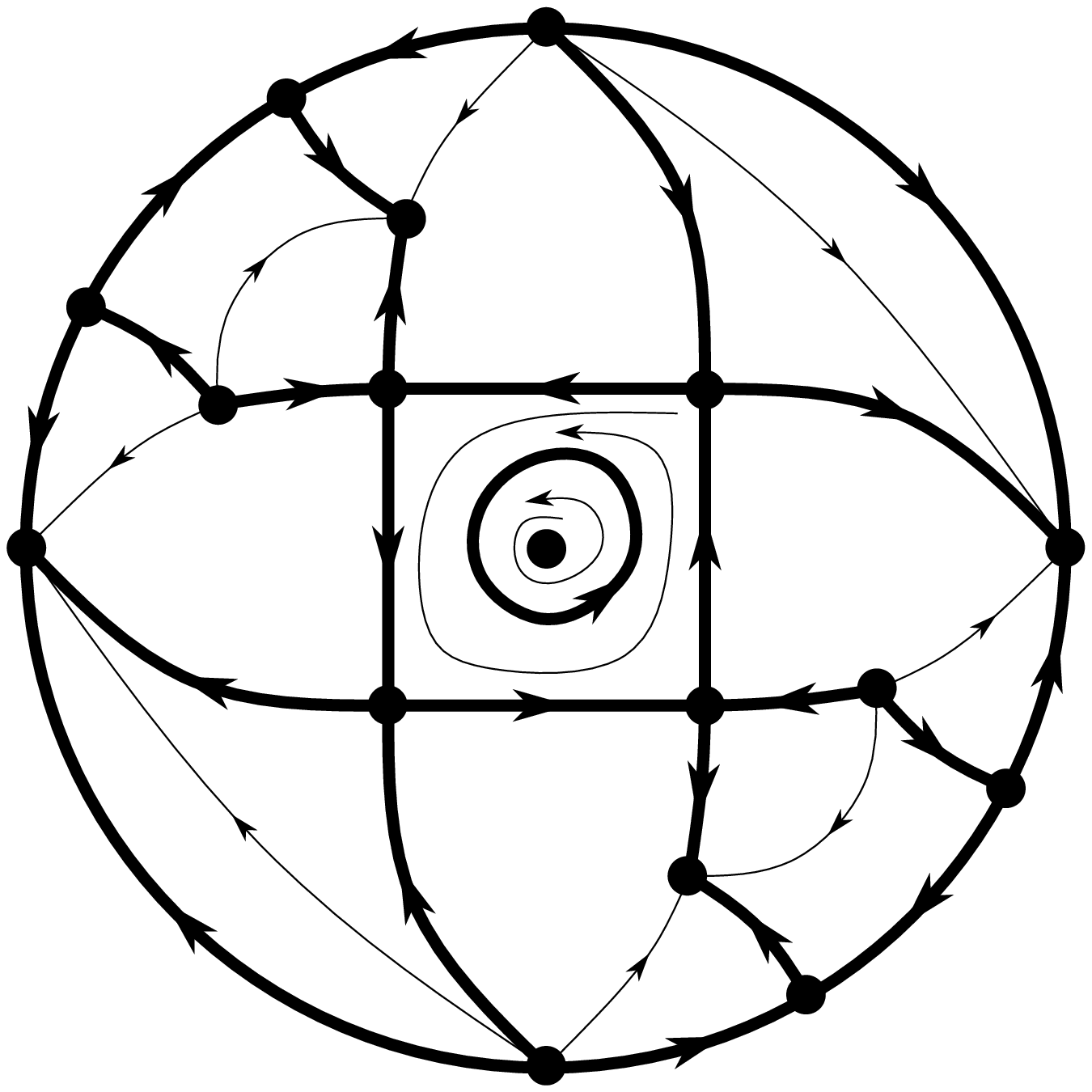} 
				\end{overpic}
				
				Case~$4.1a.ii$.
			\end{center}
		\end{minipage}
		\begin{minipage}{3.1cm}
			\begin{center}
				\begin{overpic}[height=3cm]{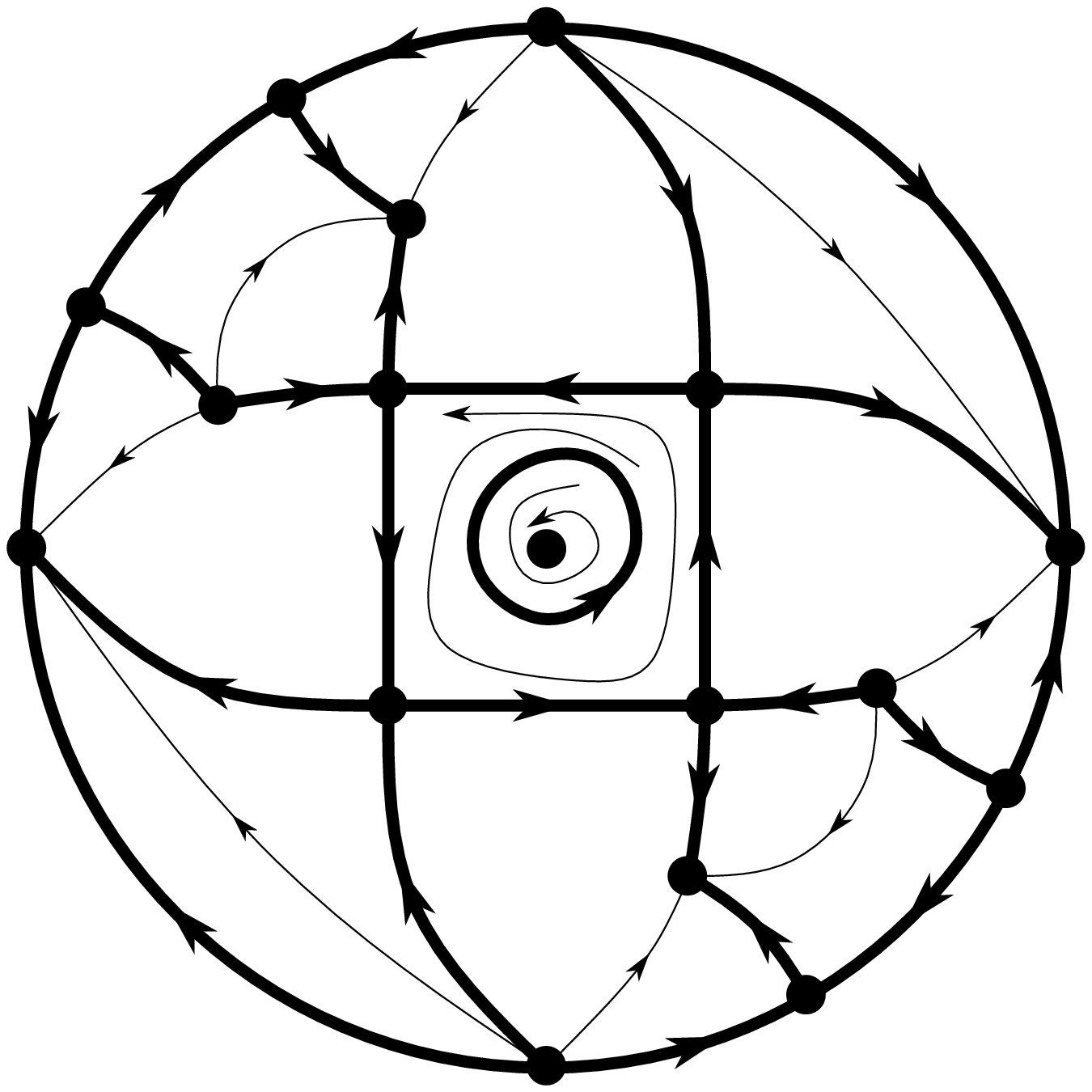} 
				\end{overpic}
				
				Case~$4.1b.i$.
			\end{center}
		\end{minipage}	
		\begin{minipage}{3.1cm}
			\begin{center}
				\begin{overpic}[height=3cm]{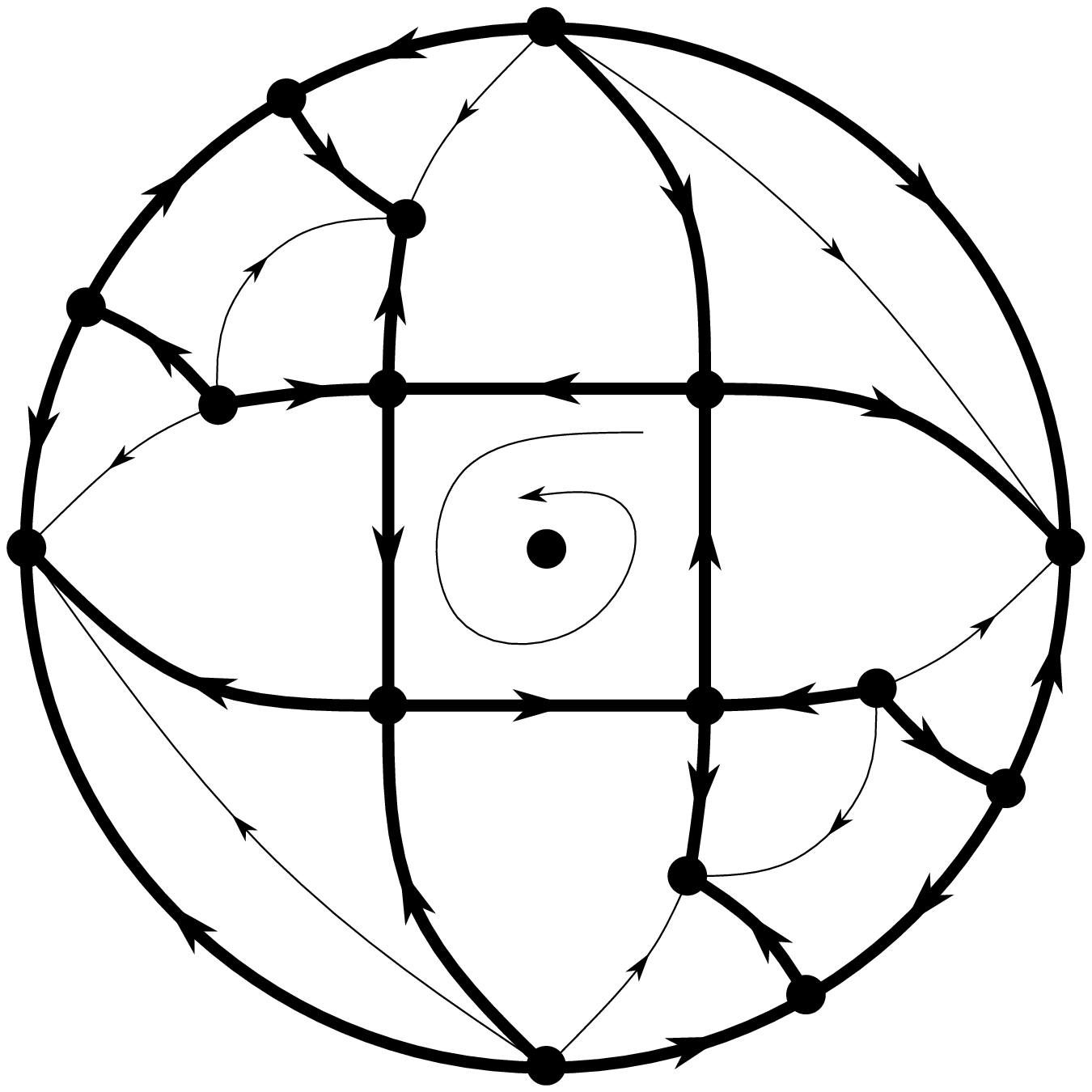} 
				\end{overpic}
				
				Case~$4.1b.ii$.
			\end{center}
		\end{minipage}
		\begin{minipage}{3.1cm}
			\begin{center}
				\begin{overpic}[height=3cm]{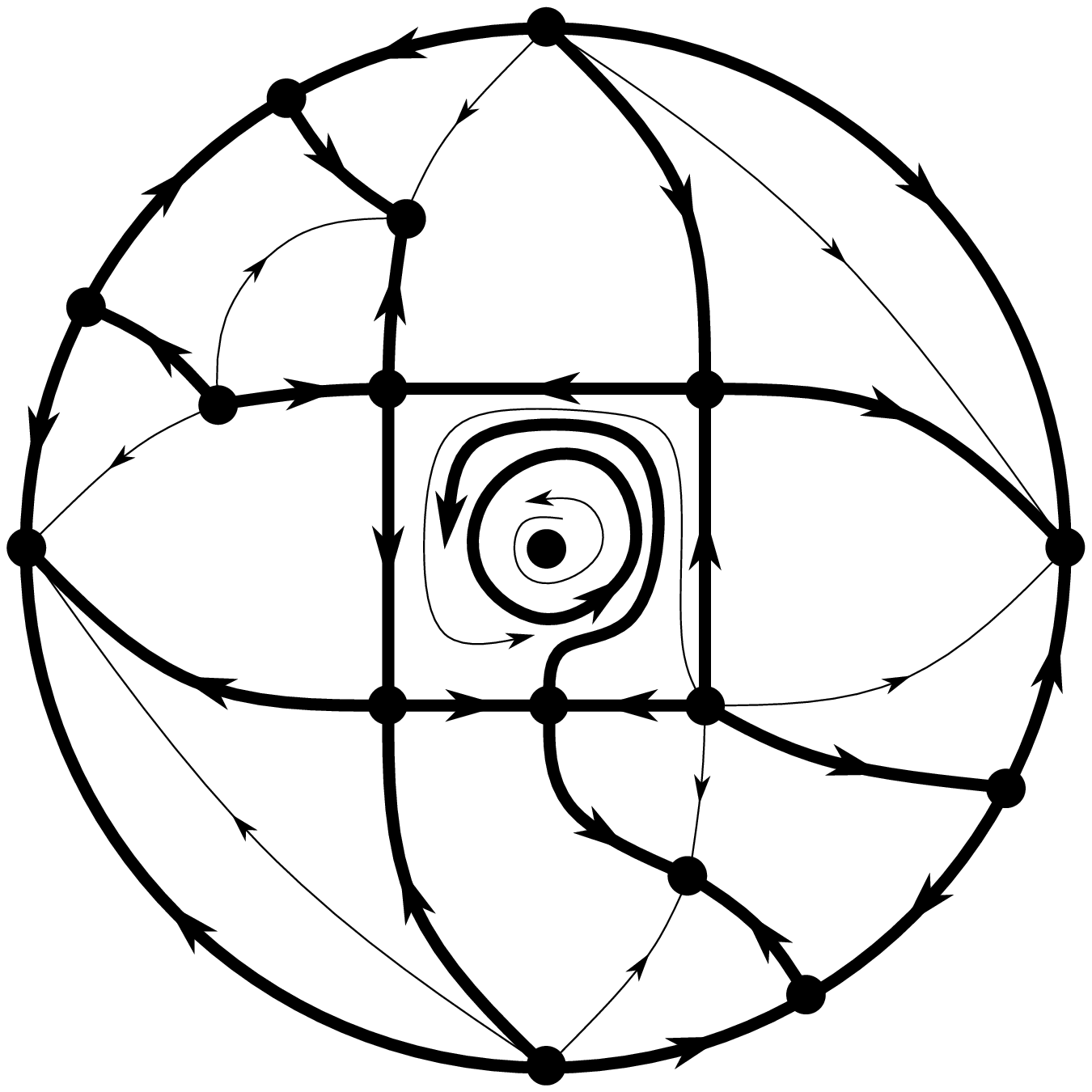} 
				\end{overpic}
				
				Case~$4.2a$.
			\end{center}
		\end{minipage}
	\end{center}
$\;$
	\begin{center}
		\begin{minipage}{3.1cm}
			\begin{center}
				\begin{overpic}[height=3cm]{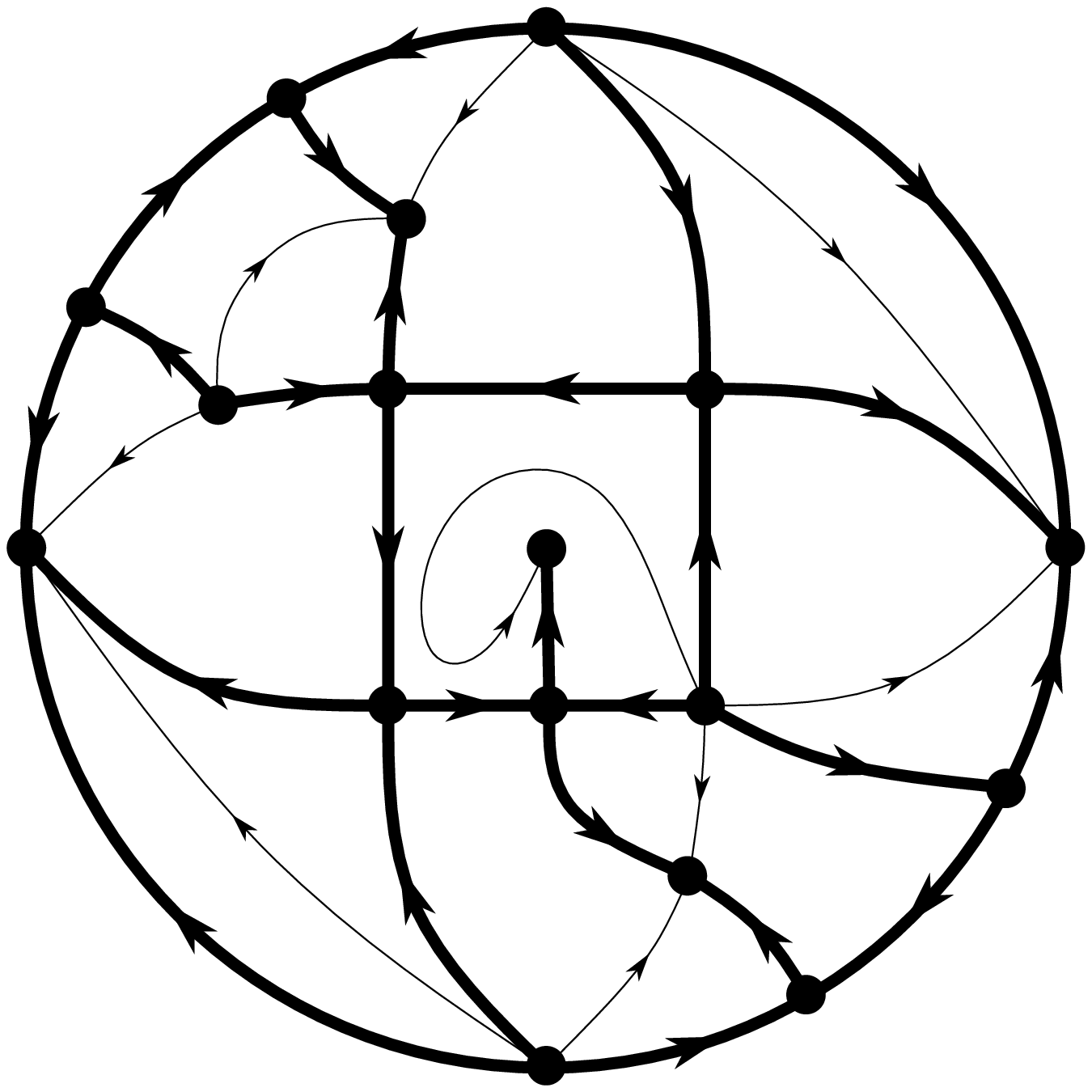} 
				\end{overpic}
				
				Case~$4.2b$.
			\end{center}
		\end{minipage}
		\begin{minipage}{3.1cm}
			\begin{center}
				\begin{overpic}[height=3cm]{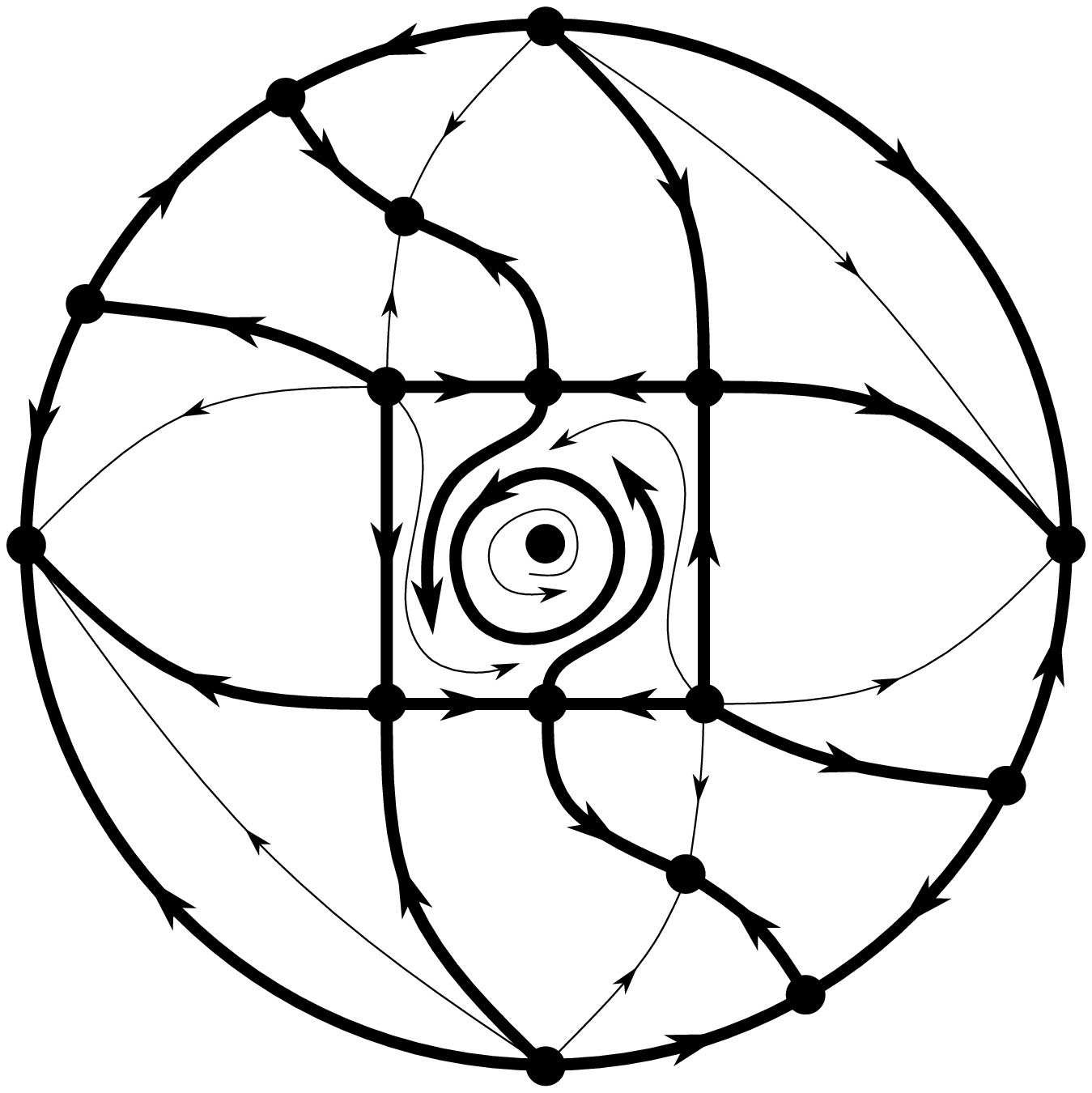} 
				\end{overpic}
				
				Case~$4.4a.i$.
			\end{center}
		\end{minipage}
		\begin{minipage}{3.1cm}
			\begin{center}
				\begin{overpic}[height=3cm]{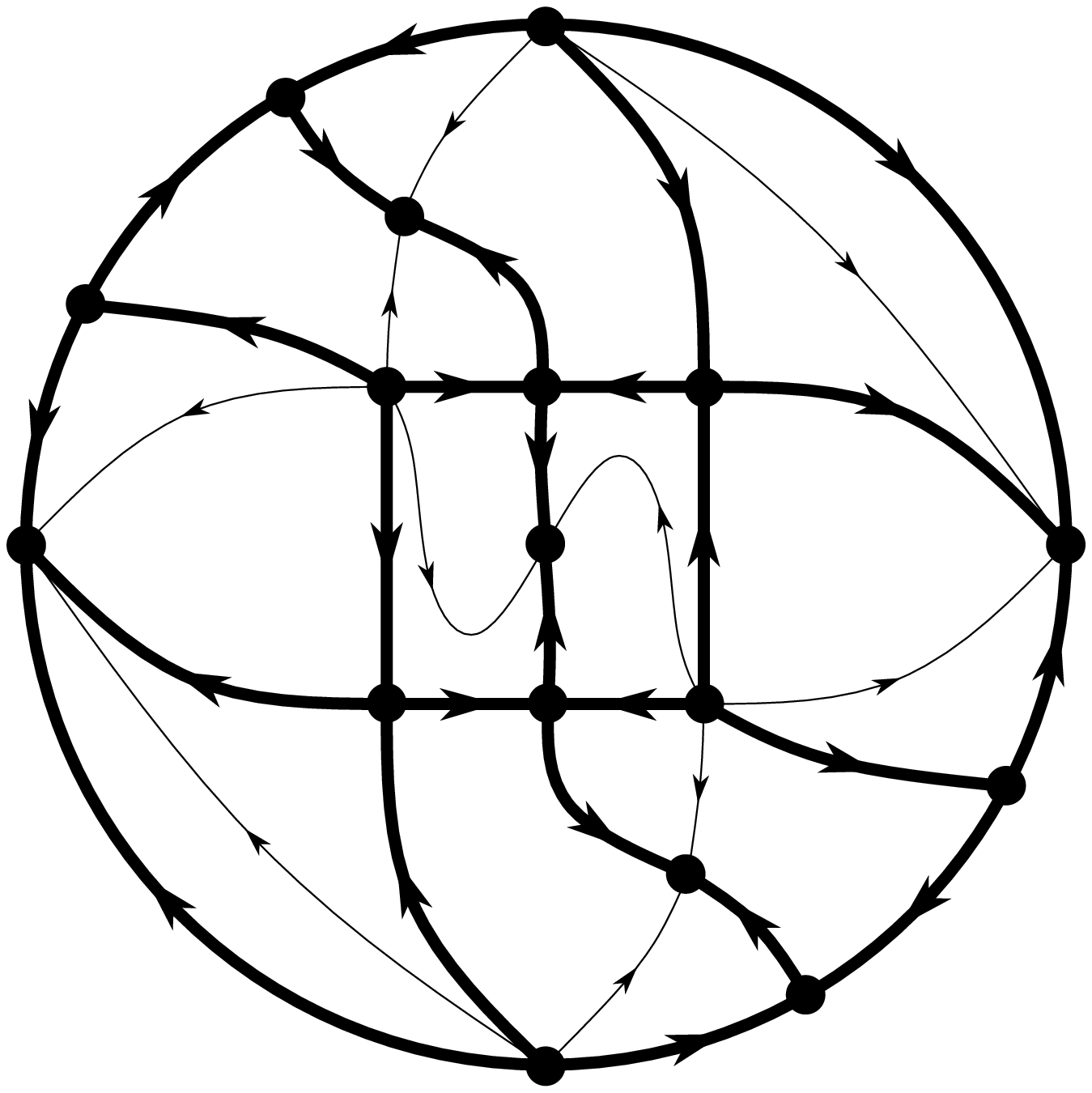} 
				\end{overpic}
				
				Case~$4.4a.ii$.
			\end{center}
		\end{minipage}
		\begin{minipage}{3.1cm}
			\begin{center}
				\begin{overpic}[height=3cm]{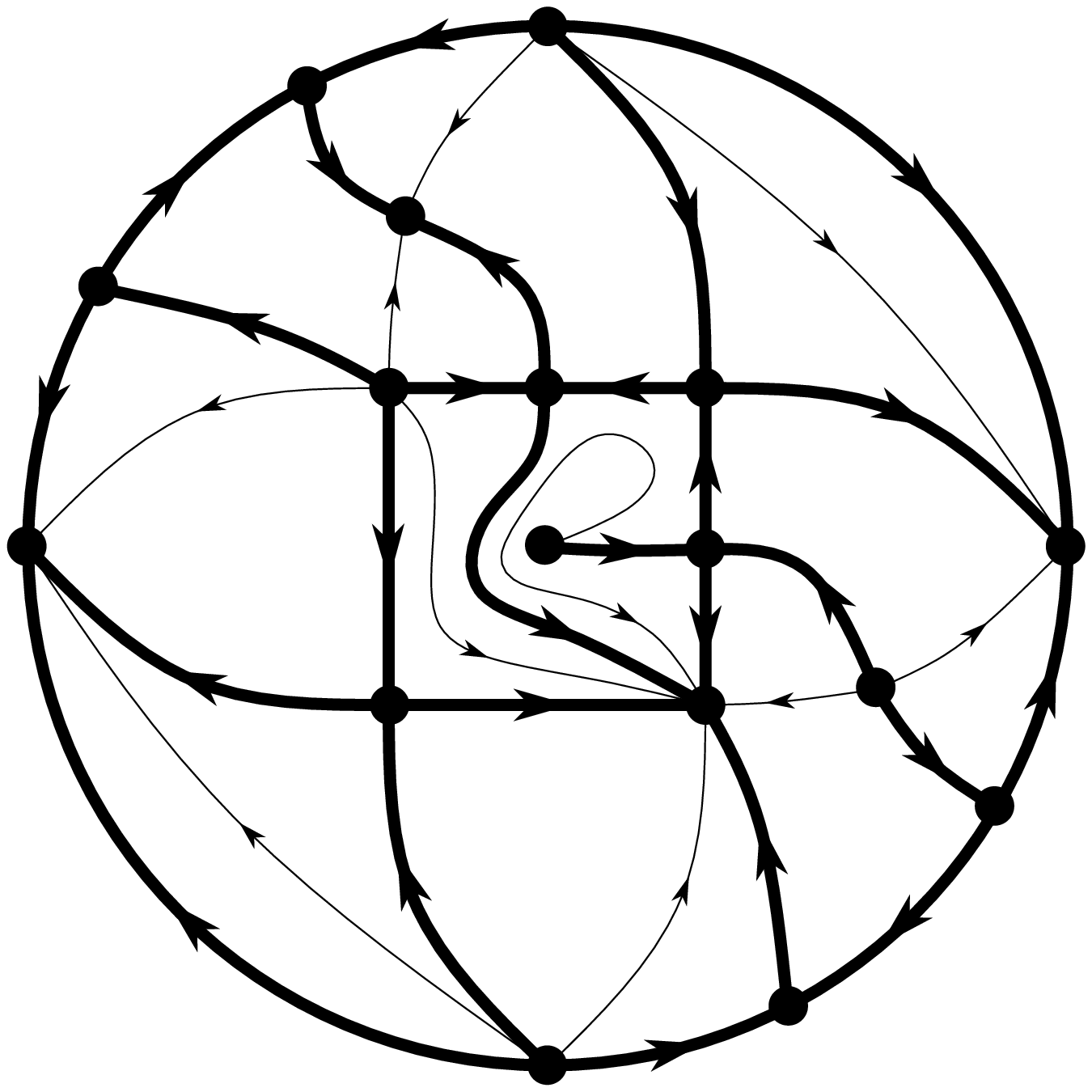} 
				\end{overpic}
				
				Case~$4.6a.i.1$.
			\end{center}
		\end{minipage}
		\begin{minipage}{3.1cm}
			\begin{center}
				\begin{overpic}[height=3cm]{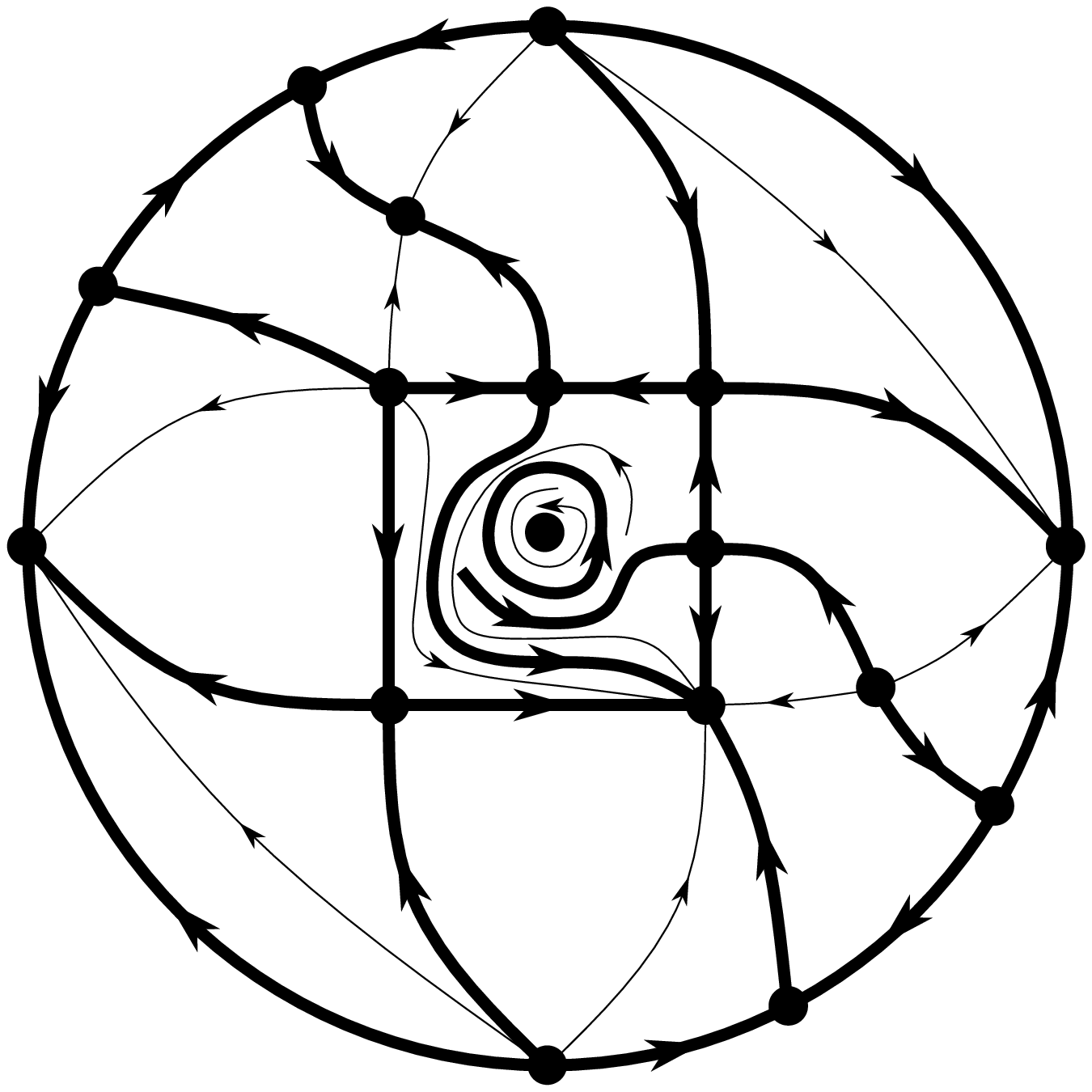} 
				\end{overpic}
				
				Case~$4.6a.ii.1$.
			\end{center}
		\end{minipage}
	\end{center}
	\caption{Topologically distinct phase portraits under the hypothesis of Theorem~\ref{Main3}.}\label{GenericFinal}
\end{figure}
\end{theorem}

\begin{corollary}\label{Main4}
	Under the hypothesis of Theorem~\ref{Main3}, the phase portrait in the square
		\[\{(x,y)\in\mathbb{R}^2\colon 0\leqslant x\leqslant1, \; 0\leqslant y\leqslant 1\},\]
	is topologically equivalent to one of the twenty phase portraits given by Figure~\ref{FinalSquares}. Moreover, all phase portraits are realizable.
\begin{figure}[h]
	\begin{center}
		\begin{minipage}{3.1cm}
			\begin{center}
				\begin{overpic}[height=3cm]{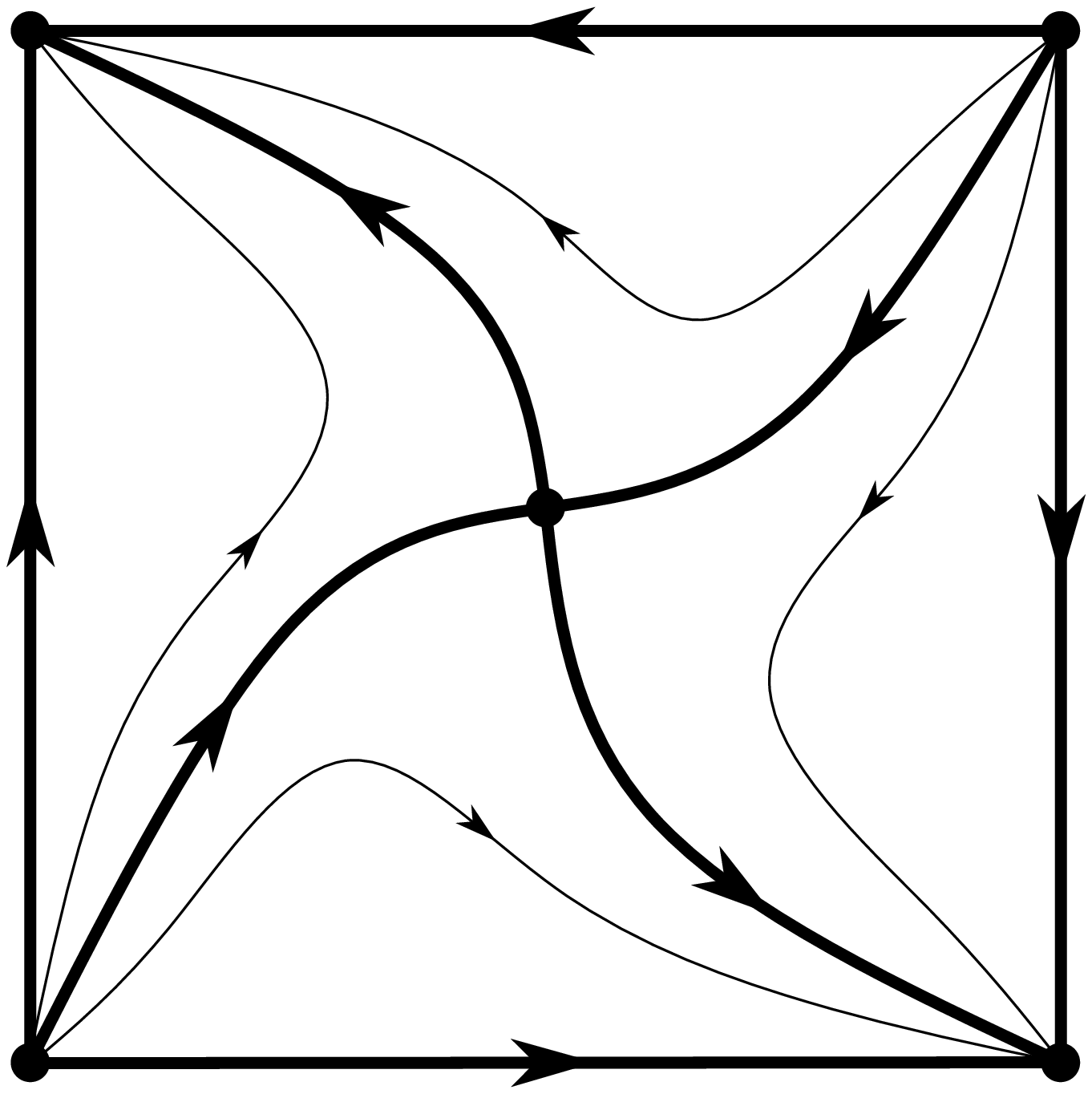} 
				\end{overpic}
				
				Case~$1.1$.
			\end{center}
		\end{minipage}
		\begin{minipage}{3.1cm}
			\begin{center}
				\begin{overpic}[height=3cm]{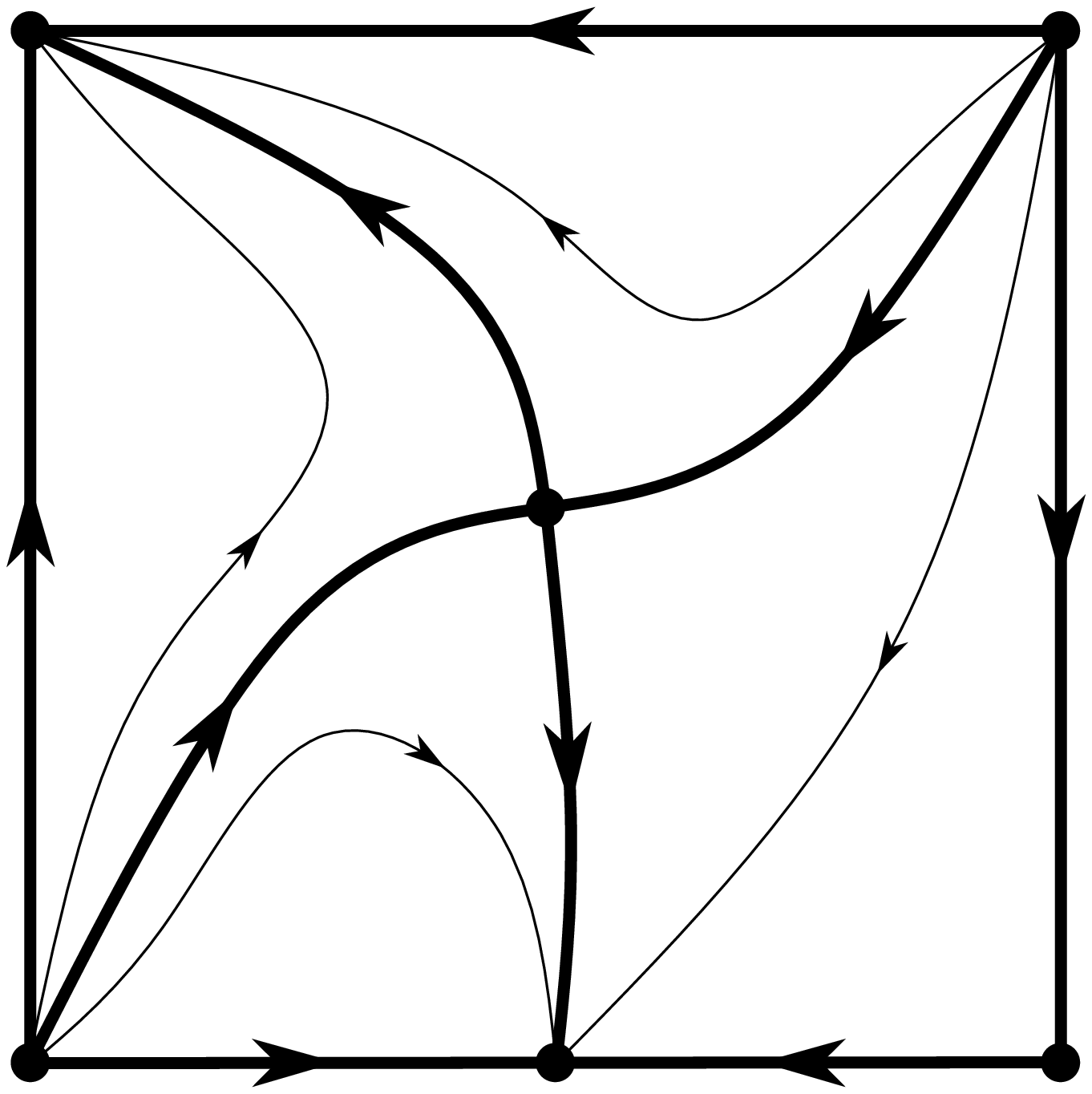} 
				\end{overpic}
				
				Case~$1.2$.
			\end{center}
		\end{minipage}
		\begin{minipage}{3.1cm}
			\begin{center}
				\begin{overpic}[height=3cm]{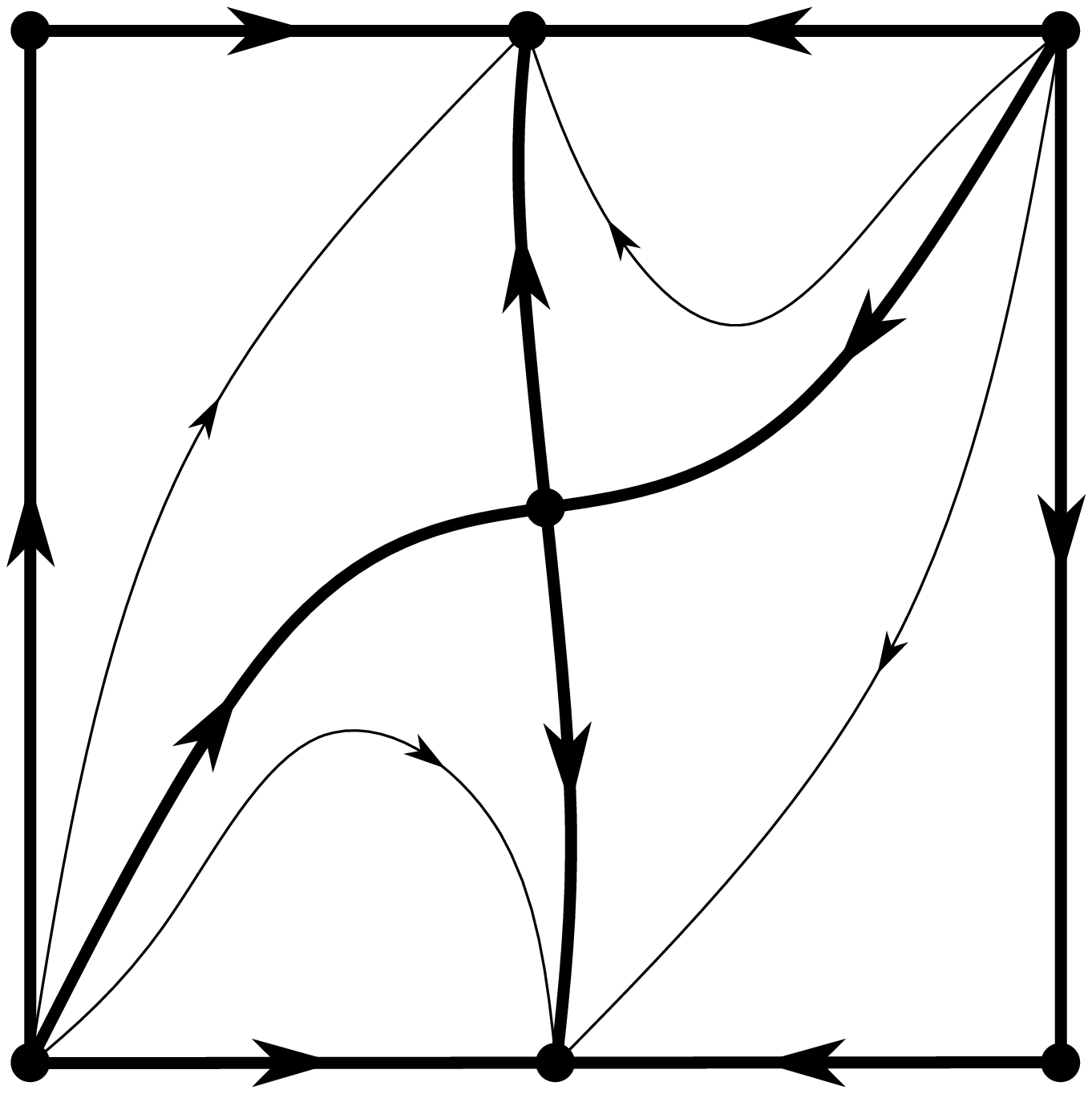} 
				\end{overpic}
				
				Case~$1.4a$.
			\end{center}
		\end{minipage}
		\begin{minipage}{3.1cm}
			\begin{center}
				\begin{overpic}[height=3cm]{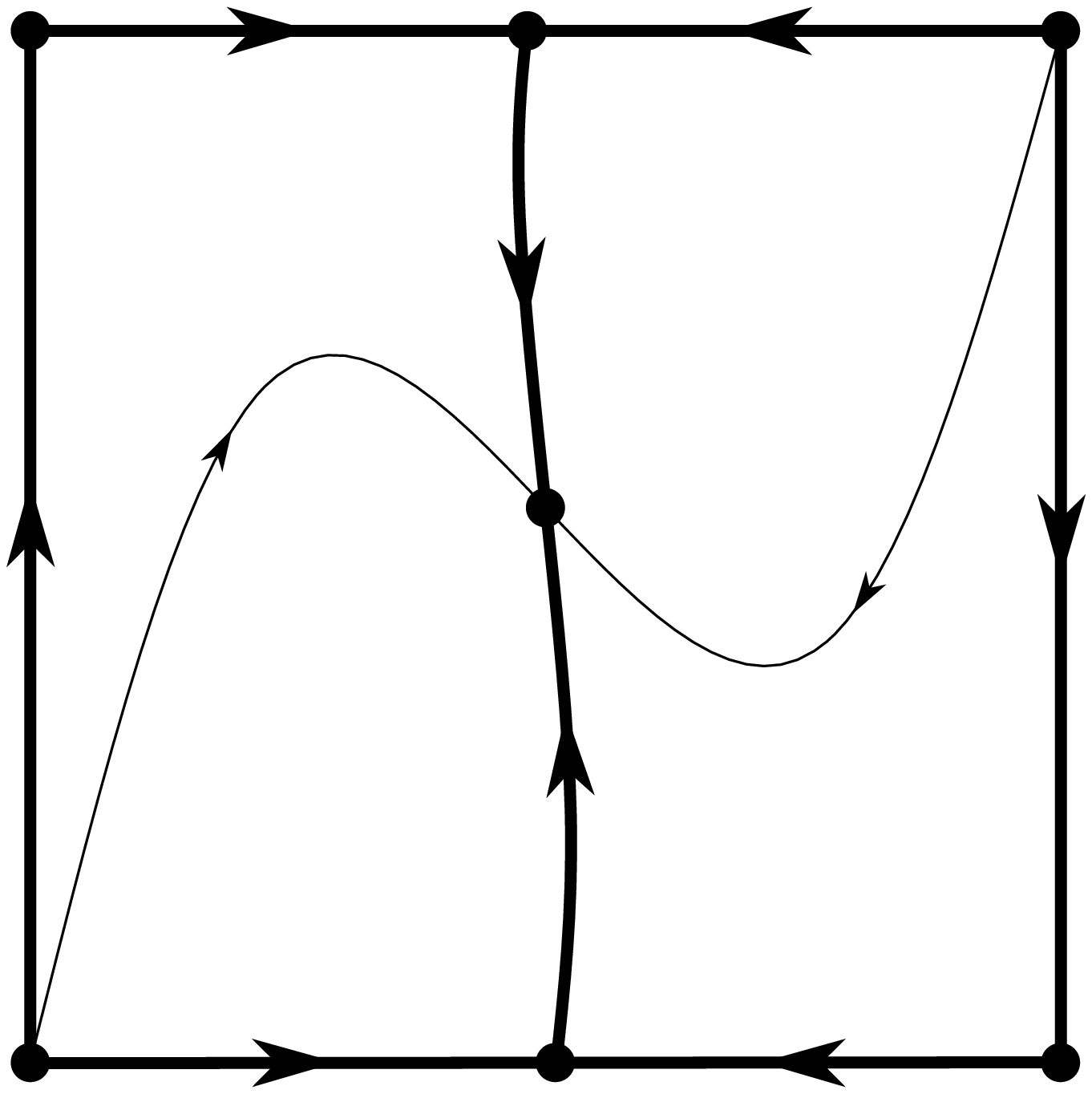} 
				\end{overpic}
				
				Case~$1.4b$.
			\end{center}
		\end{minipage}
		\begin{minipage}{3.1cm}
			\begin{center}
				\begin{overpic}[height=3cm]{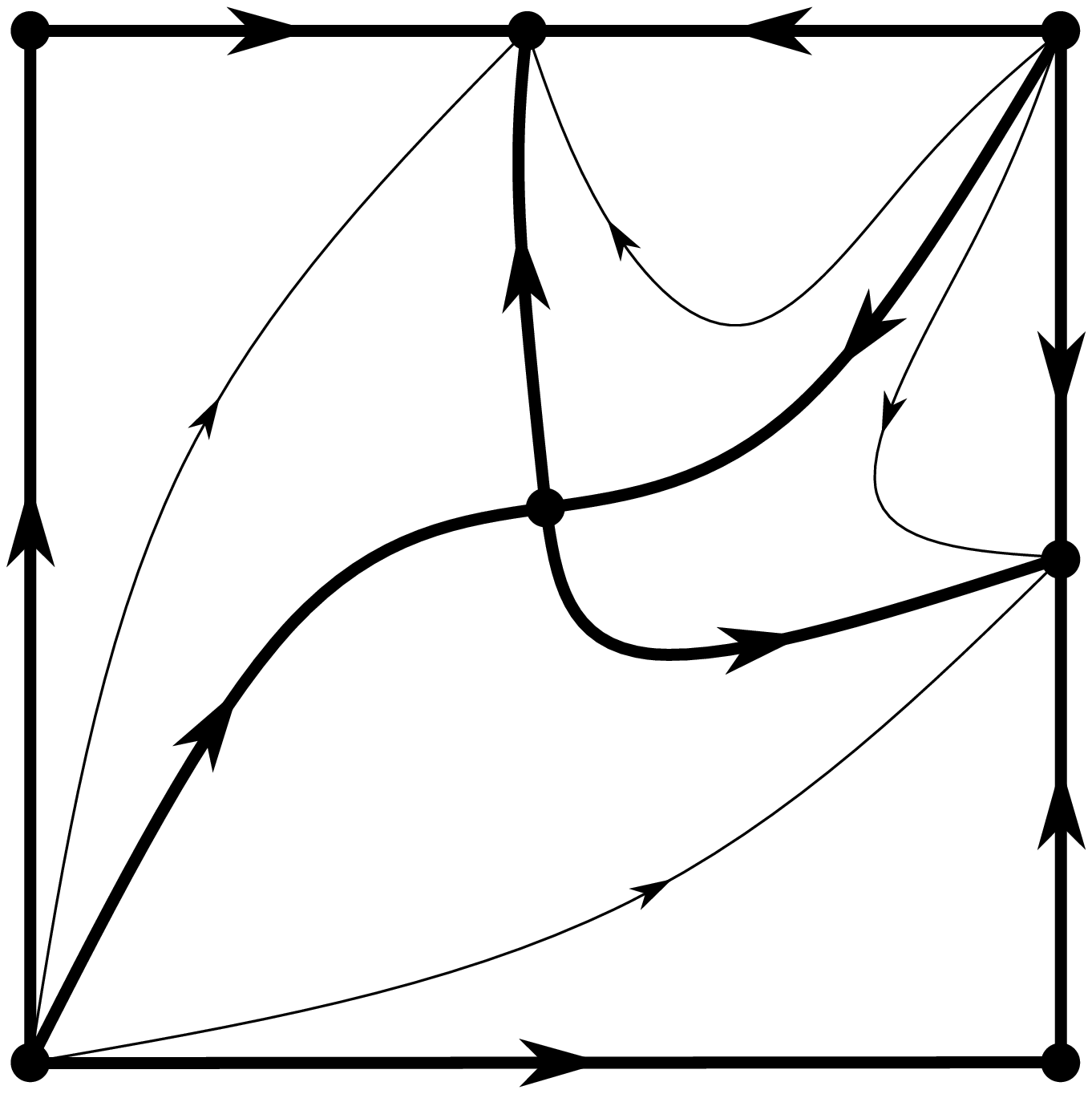} 
				\end{overpic}
				
				Case~$1.6a$.
			\end{center}
		\end{minipage}
	\end{center}
$\;$
	\begin{center}
		\begin{minipage}{3.1cm}
			\begin{center}
				\begin{overpic}[height=3cm]{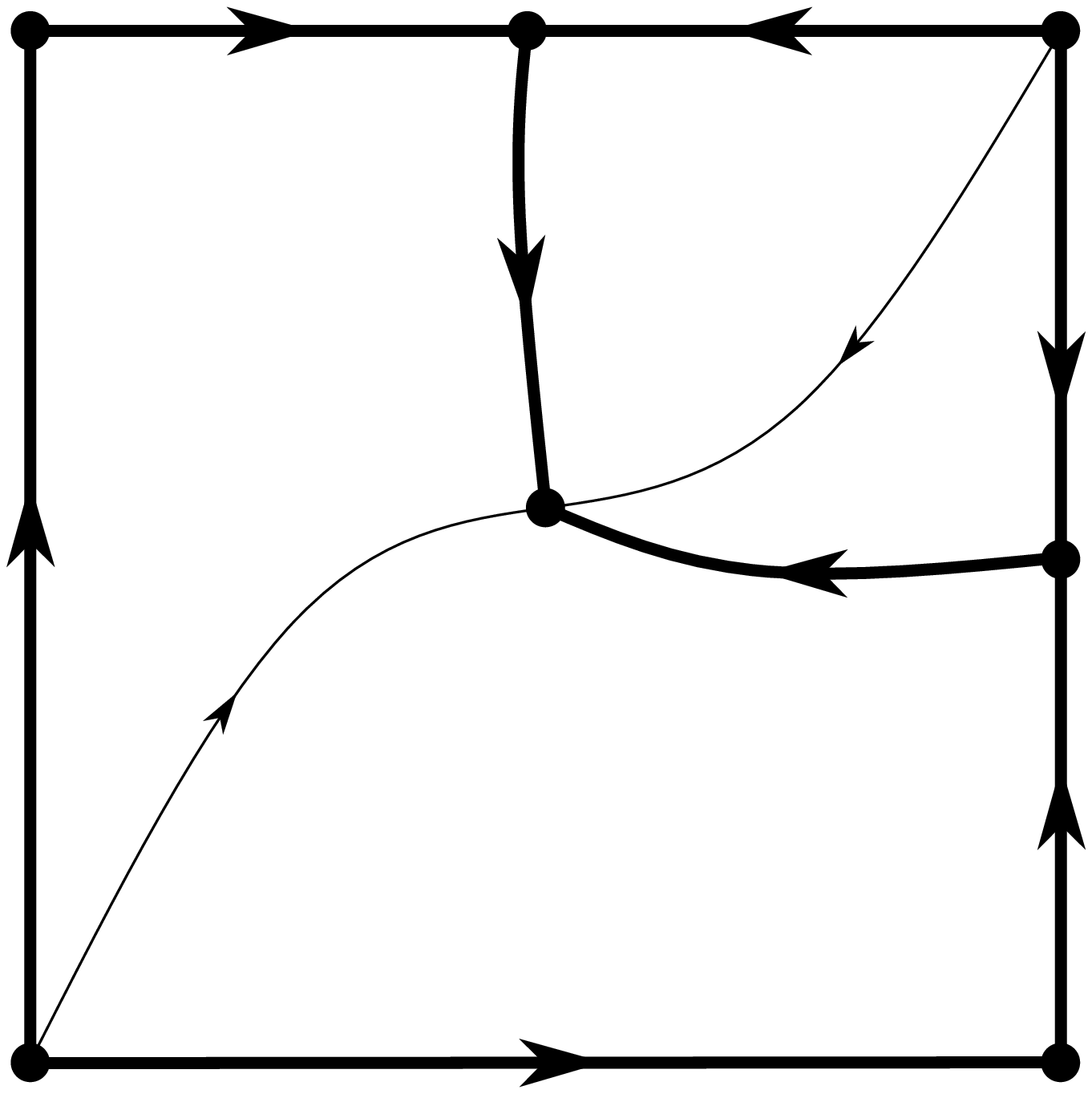} 
				\end{overpic}
				
				Case~$1.6b$.
			\end{center}
		\end{minipage}
		\begin{minipage}{3.1cm}
			\begin{center}
				\begin{overpic}[height=3cm]{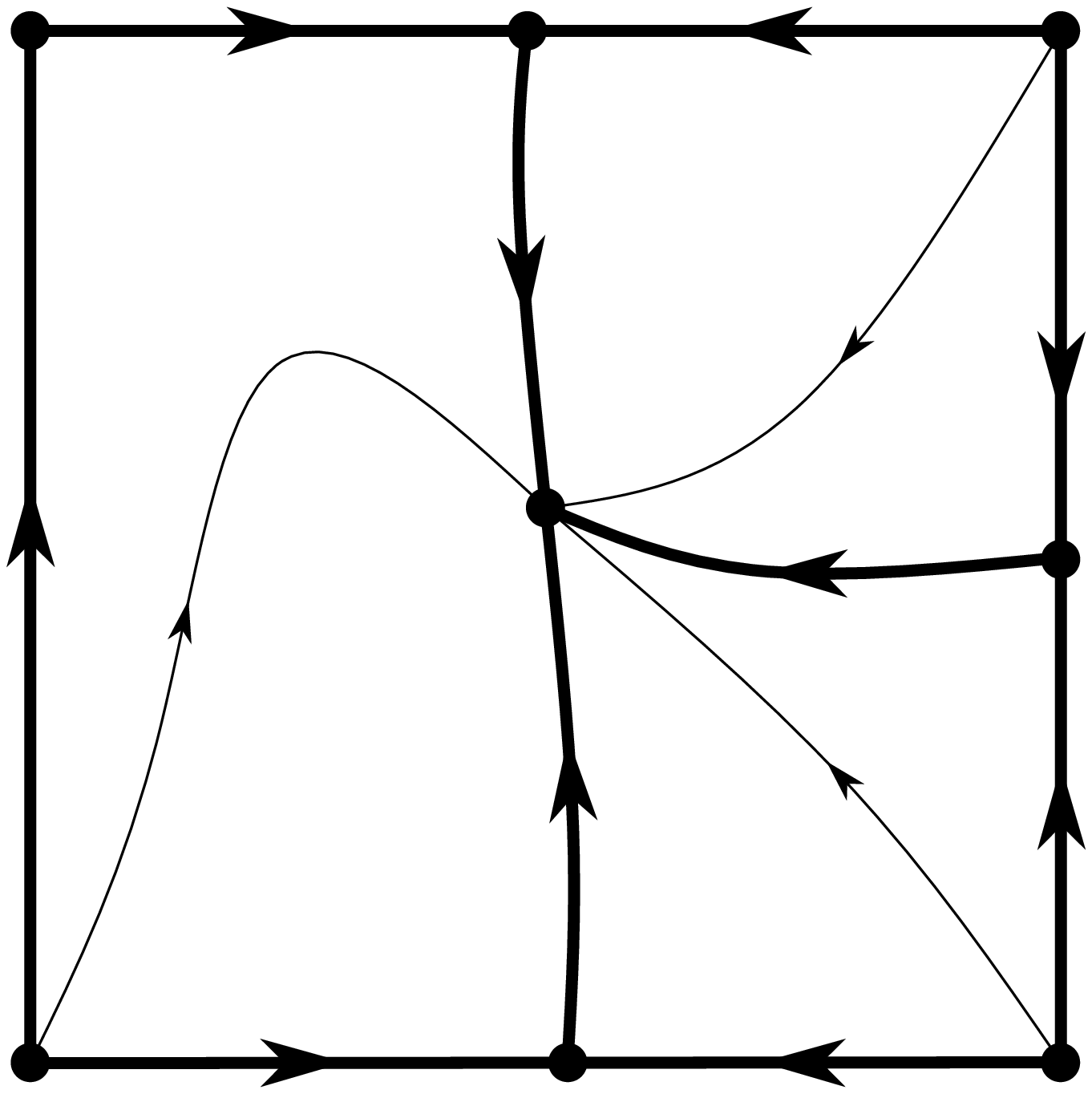} 
				\end{overpic}
				
				Case~$1.7$.
			\end{center}
		\end{minipage}
		\begin{minipage}{3.1cm}
			\begin{center}
				\begin{overpic}[height=3cm]{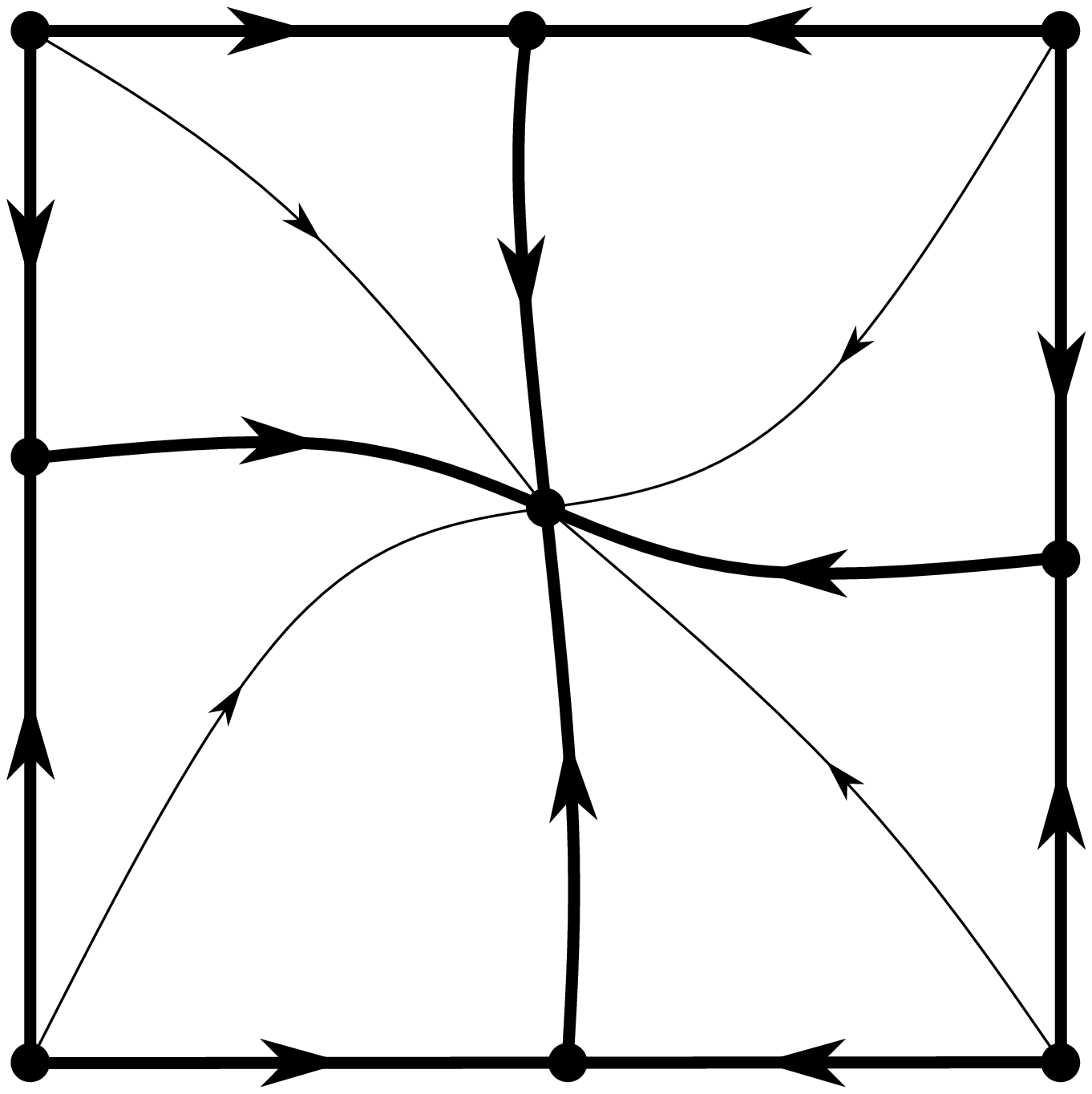} 
				\end{overpic}
				
				Case~$1.14$.
			\end{center}
		\end{minipage}
		\begin{minipage}{3.1cm}
			\begin{center}
				\begin{overpic}[height=3cm]{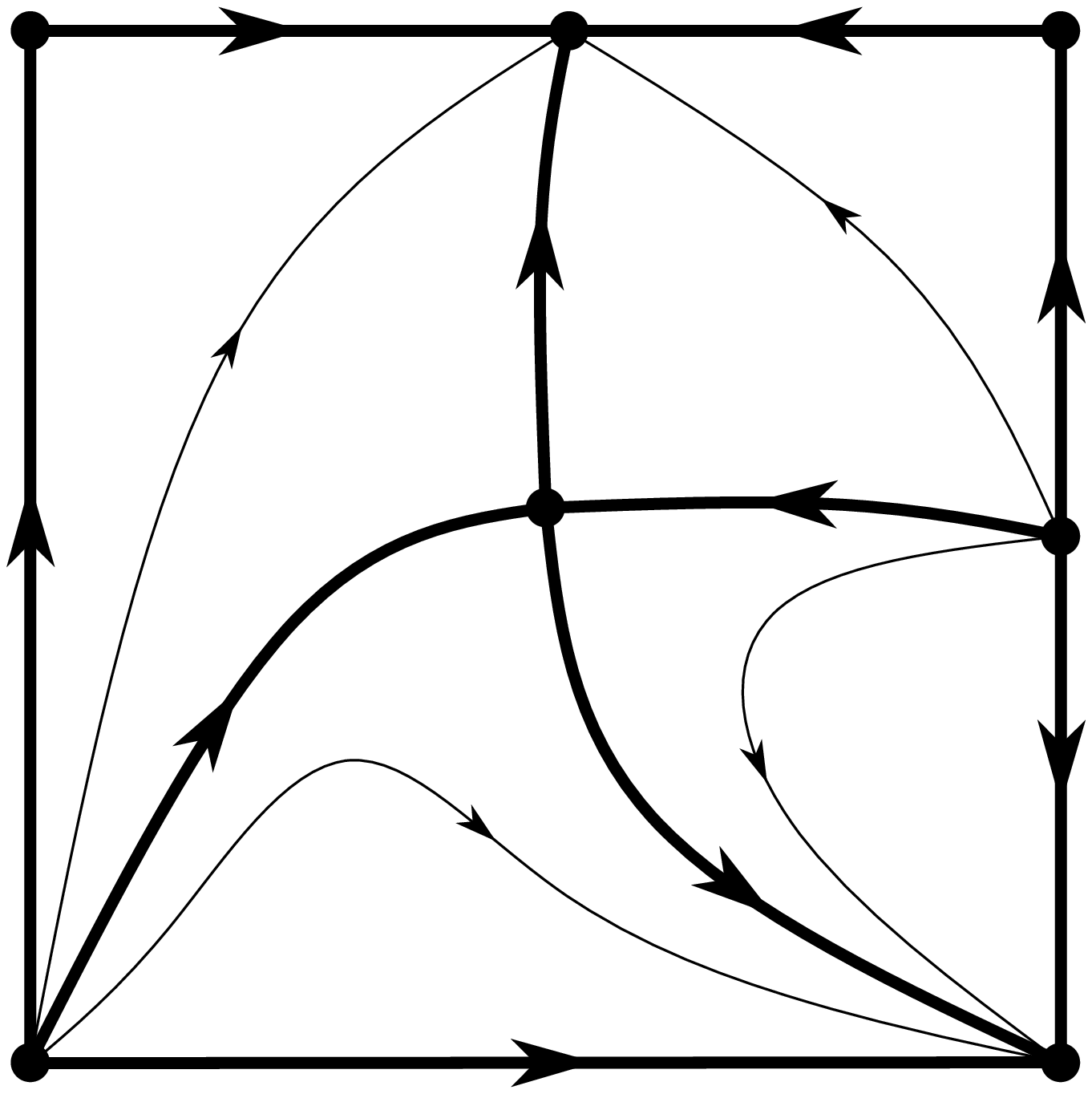} 
				\end{overpic}
				
				Case~$2.6a1$.
			\end{center}
		\end{minipage}
		\begin{minipage}{3.1cm}
			\begin{center}
				\begin{overpic}[height=3cm]{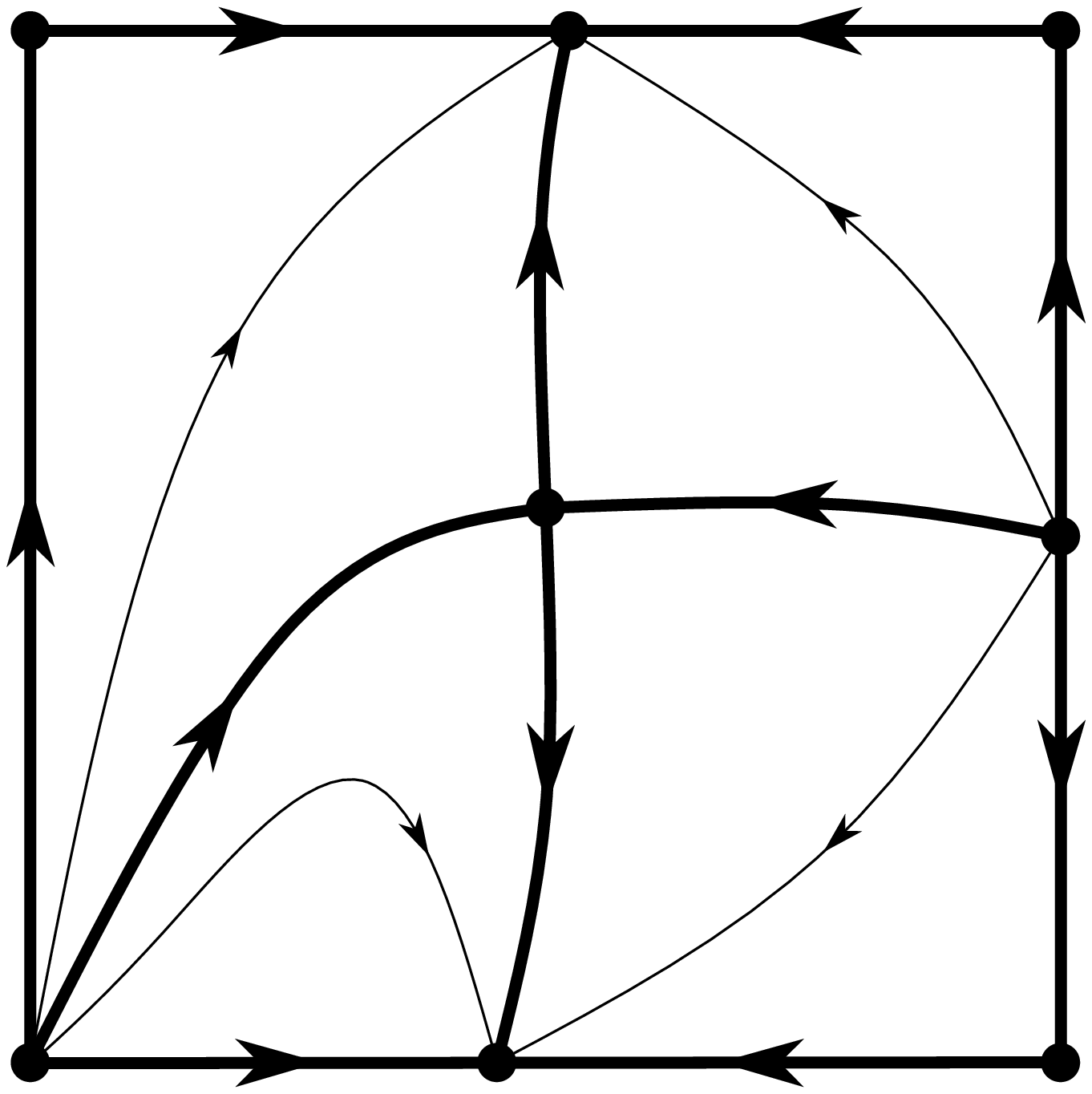} 
				\end{overpic}
				
				Case~$2.7a$.
			\end{center}
		\end{minipage}
	\end{center}
$\;$
	\begin{center}	
		\begin{minipage}{3.1cm}
			\begin{center}
				\begin{overpic}[height=3cm]{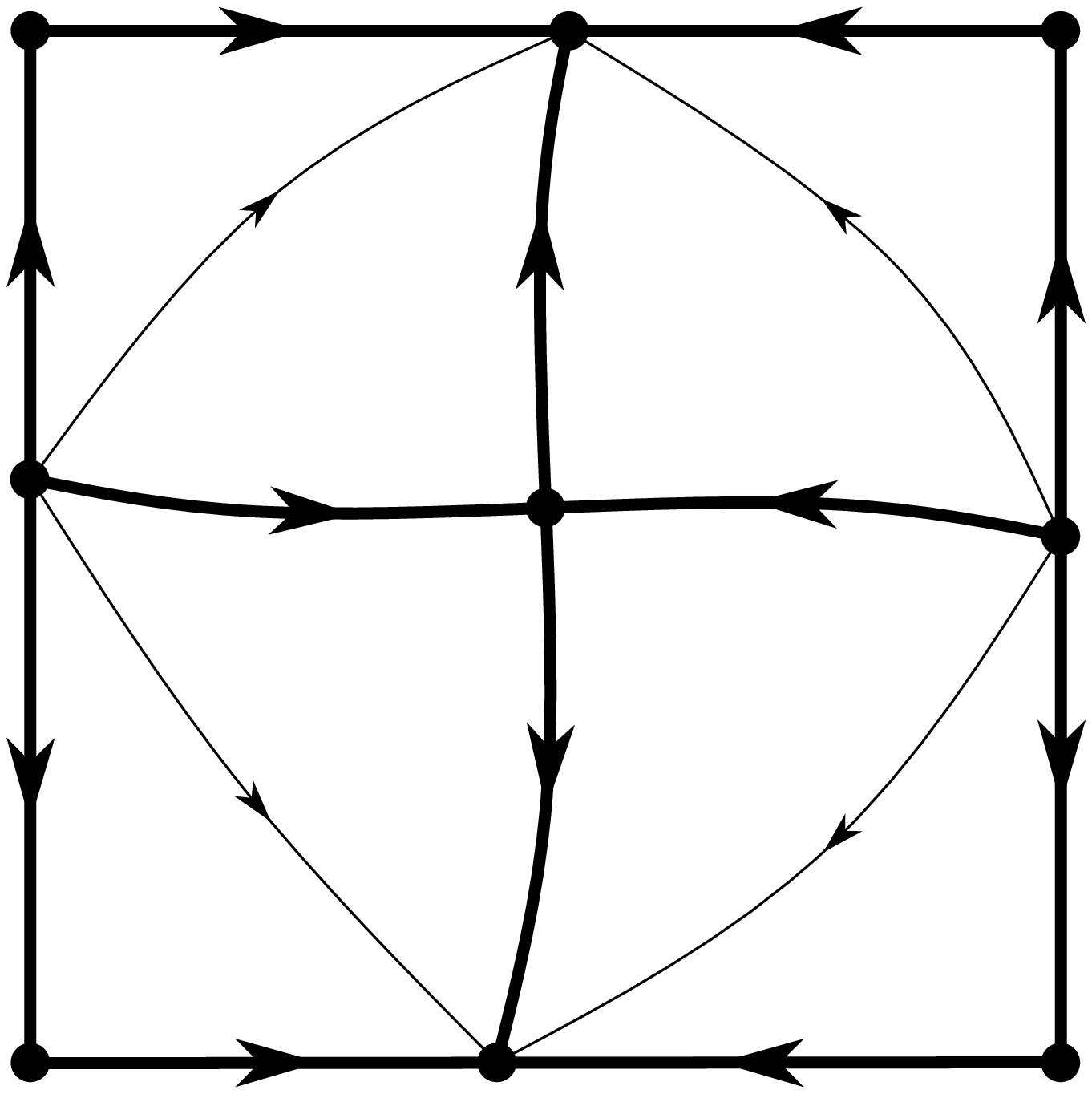} 
				\end{overpic}
				
				Case~$2.15a$.
			\end{center}
		\end{minipage}
			\begin{minipage}{3.1cm}
			\begin{center}
				\begin{overpic}[height=3cm]{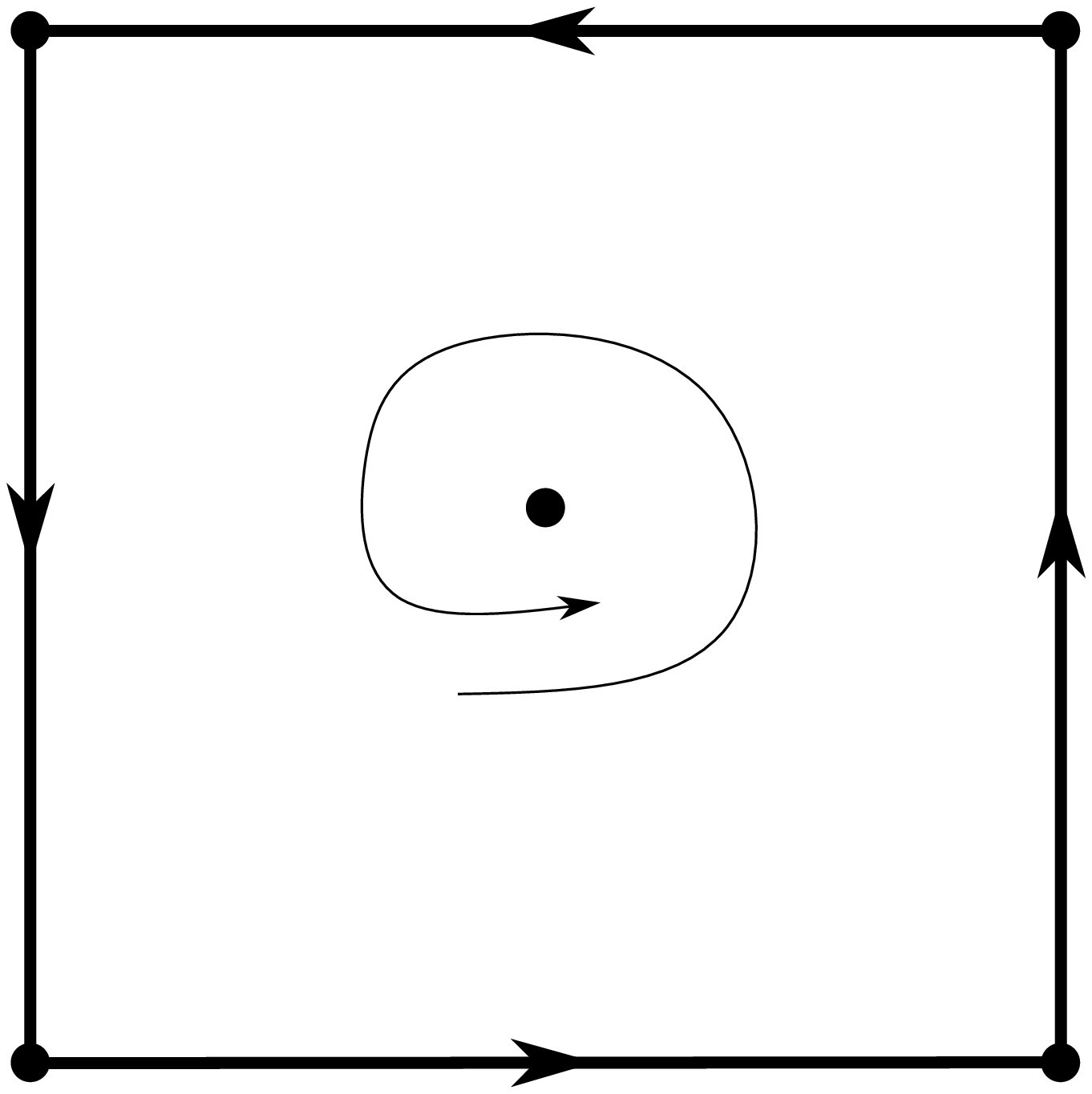} 
				\end{overpic}
				
				Case~$3.1$.
			\end{center}
		\end{minipage}
		\begin{minipage}{3.1cm}
			\begin{center}
				\begin{overpic}[height=3cm]{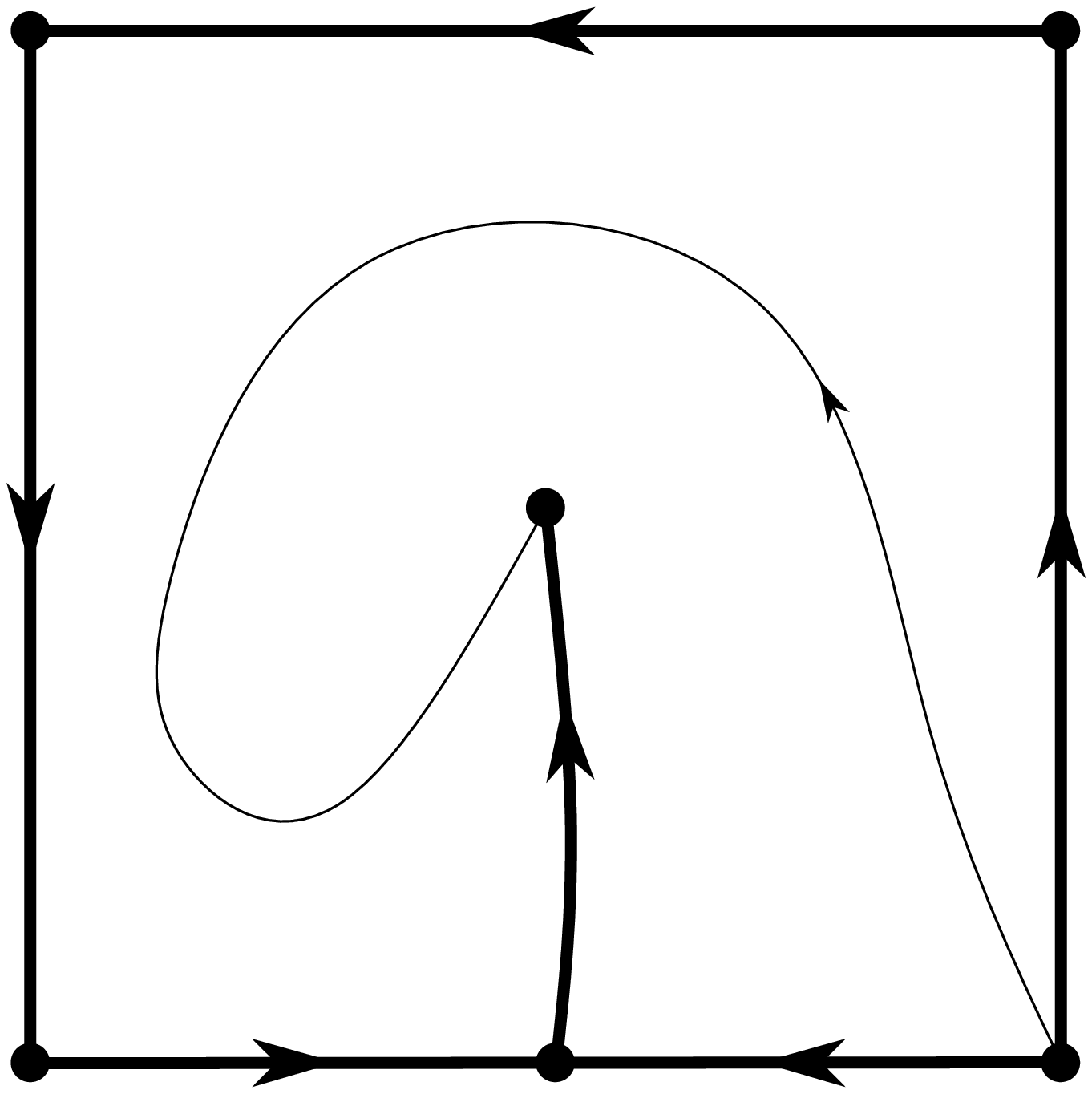} 
				\end{overpic}
				
				Case~$3.2$.
			\end{center}
		\end{minipage}
		\begin{minipage}{3.1cm}
			\begin{center}
				\begin{overpic}[height=3cm]{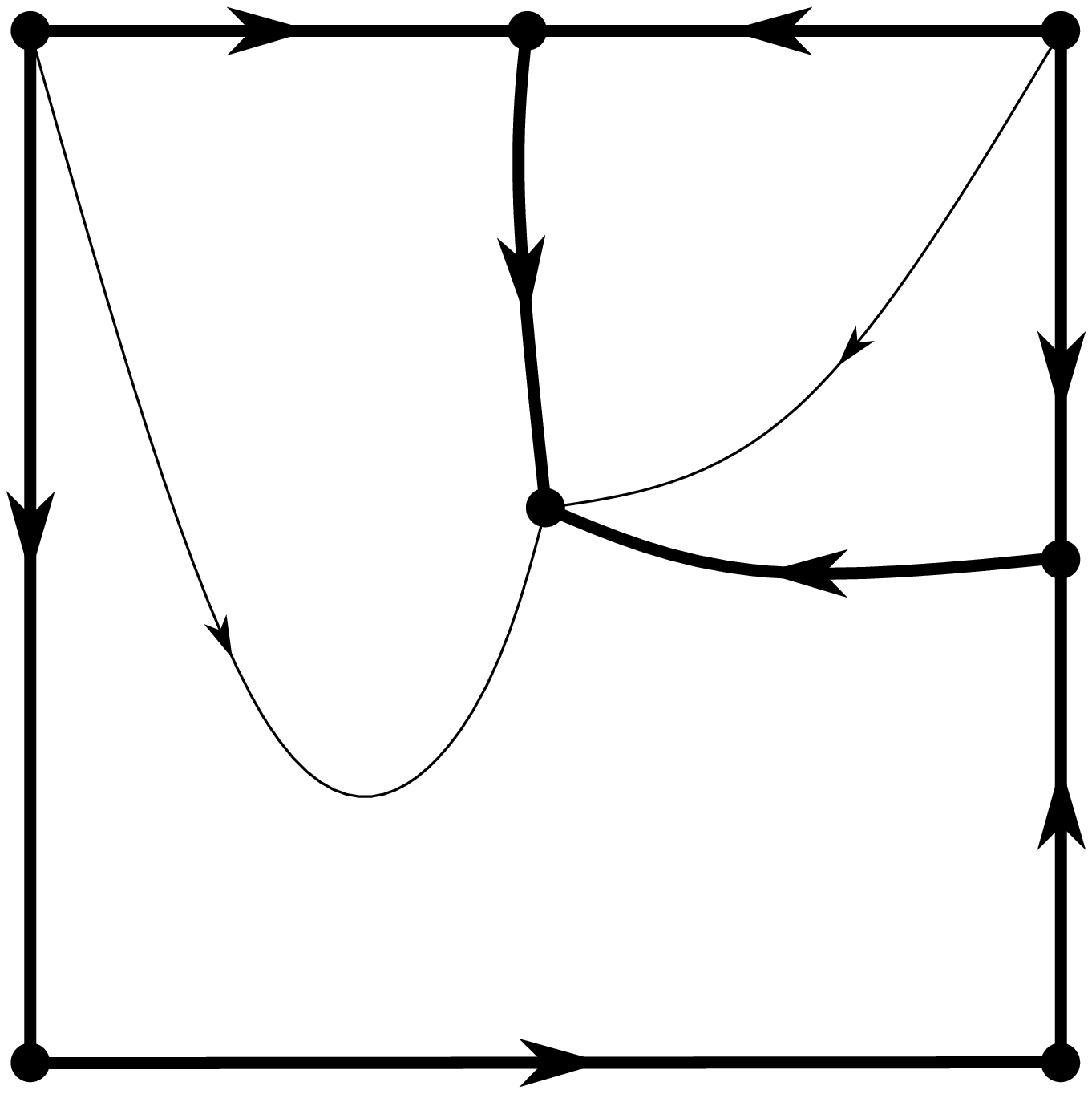} 
				\end{overpic}
				
				Case~$3.6$.
			\end{center}
		\end{minipage}
		\begin{minipage}{3.1cm}
			\begin{center}
				\begin{overpic}[height=3cm]{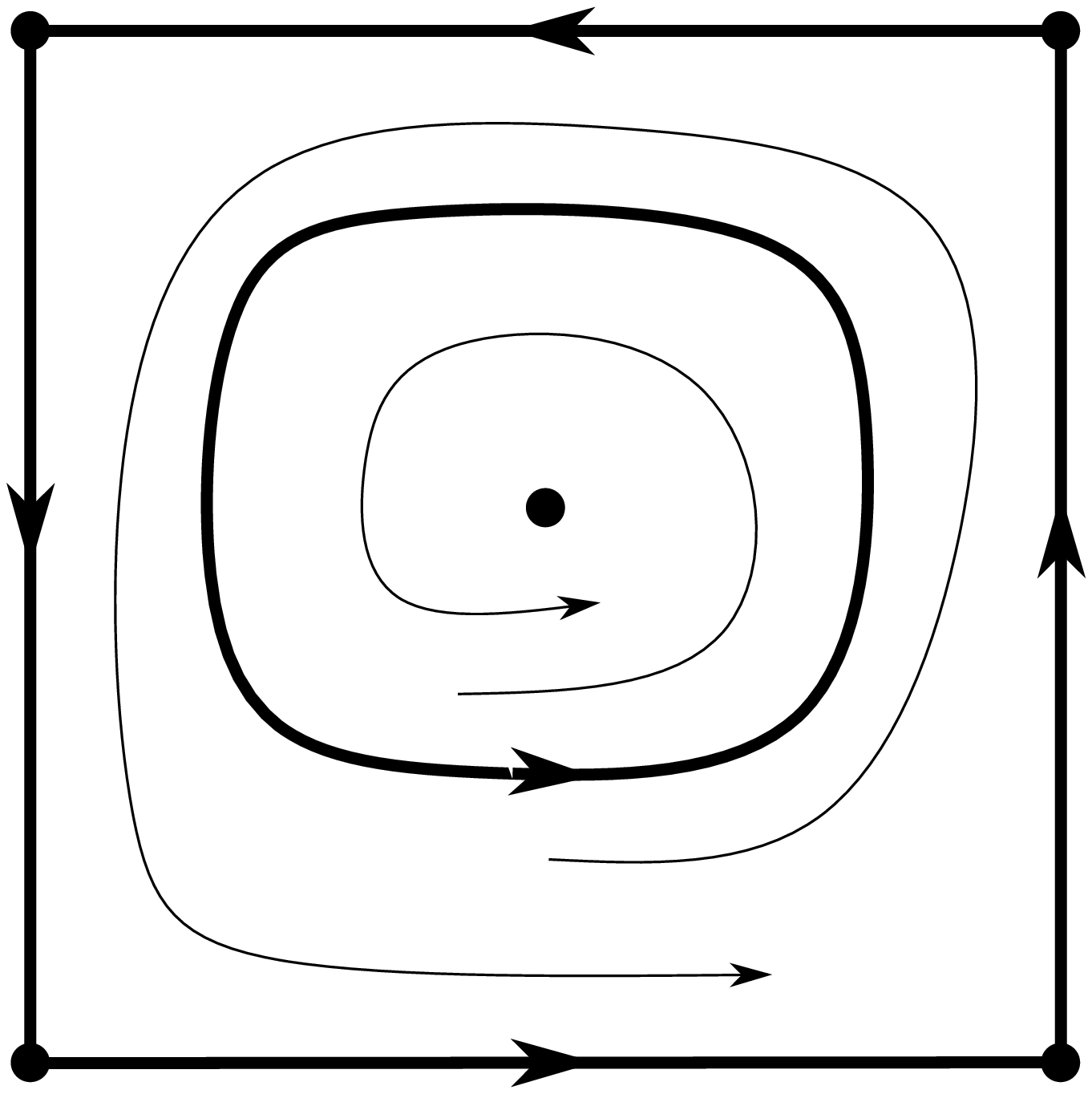} 
				\end{overpic}
				
				Case~$4.1b.i$.
			\end{center}
		\end{minipage}
	\end{center}
$\;$
	\begin{center}
		\begin{minipage}{3.1cm}
			\begin{center}
				\begin{overpic}[height=3cm]{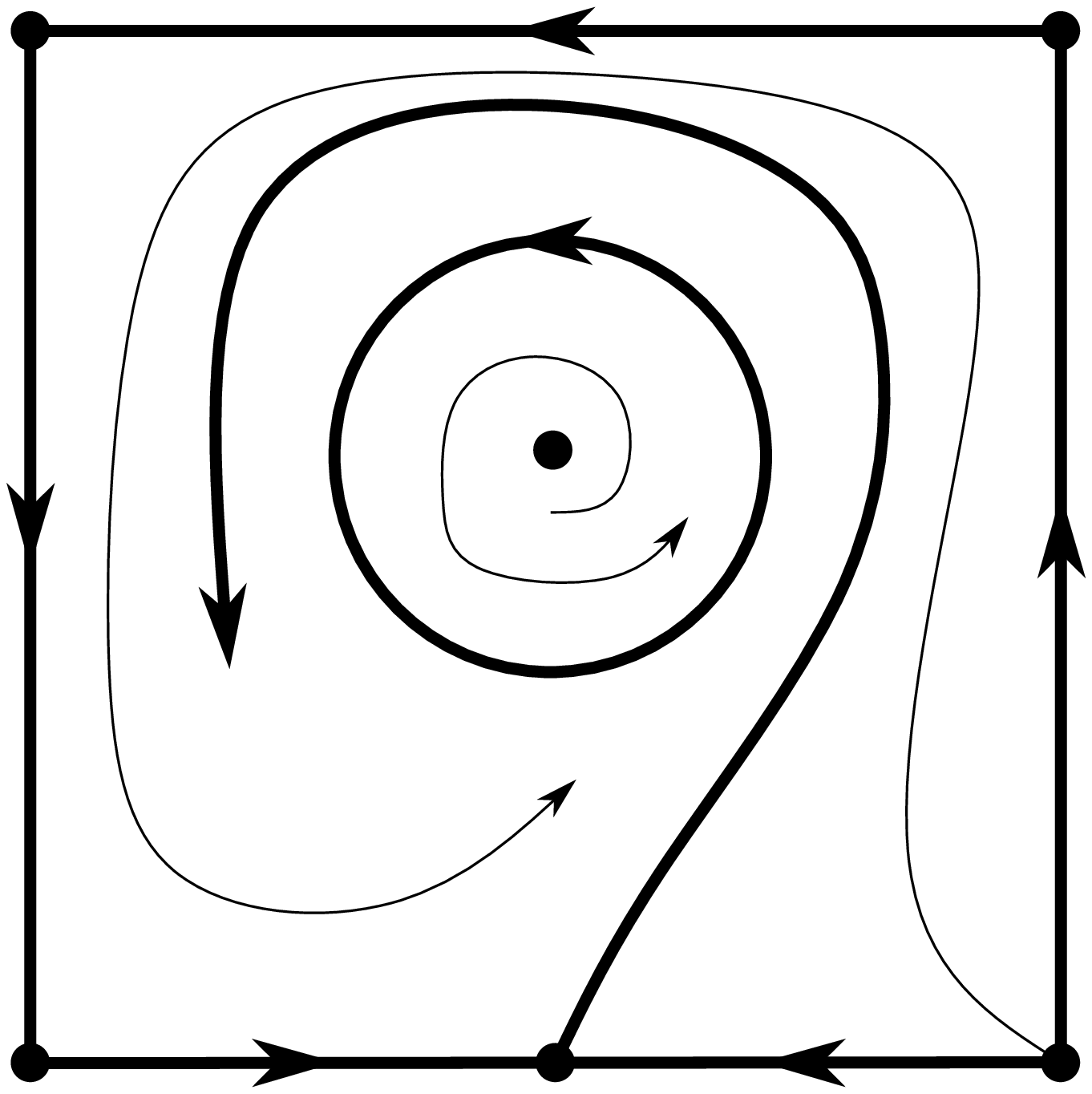} 
				\end{overpic}
				
				Case~$4.2a$.
			\end{center}
		\end{minipage}
		\begin{minipage}{3.1cm}
			\begin{center}
				\begin{overpic}[height=3cm]{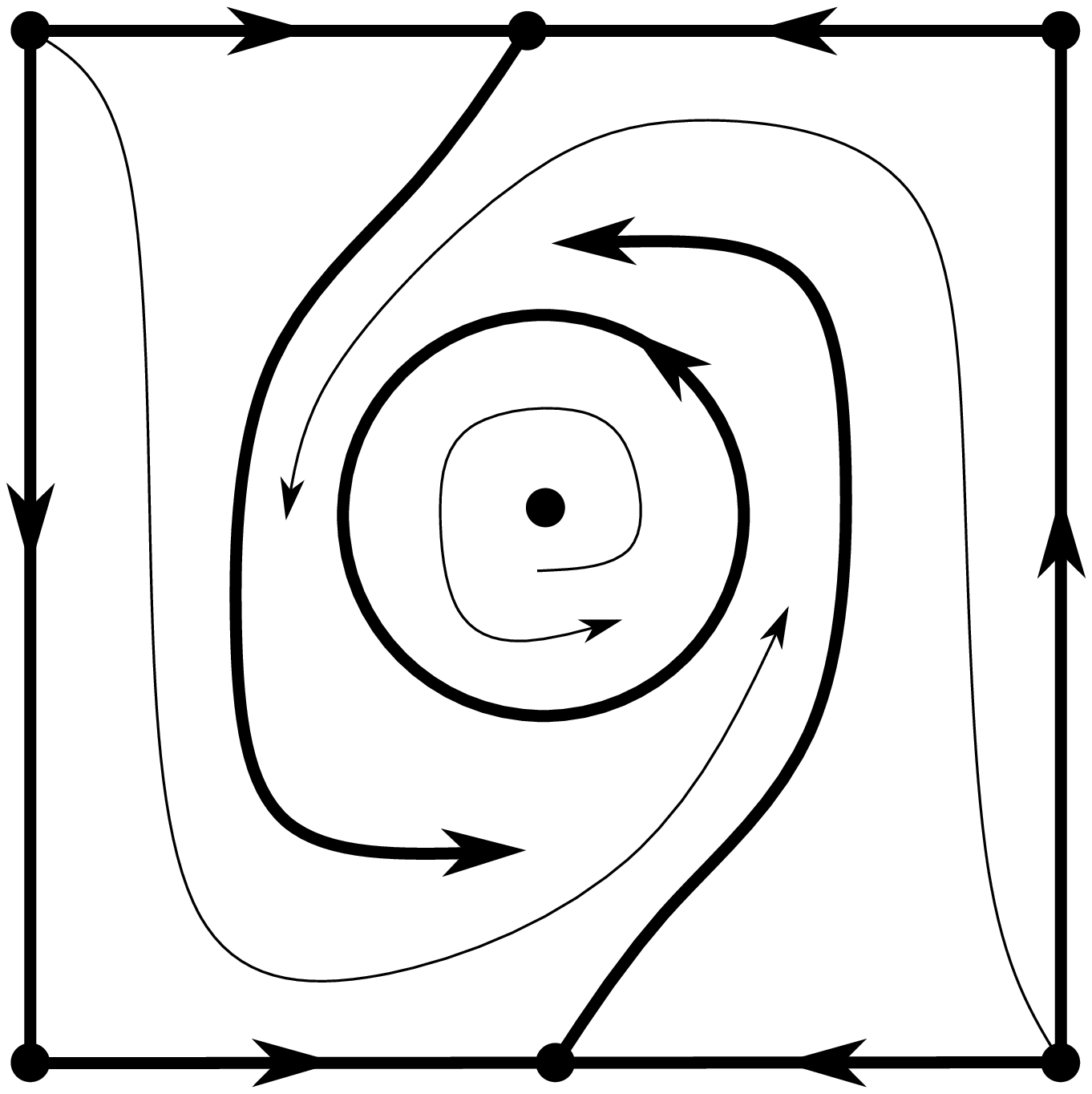} 
				\end{overpic}
				
				Case~$4.4a.i$.
			\end{center}
		\end{minipage}	
		\begin{minipage}{3.1cm}
			\begin{center}
				\begin{overpic}[height=3cm]{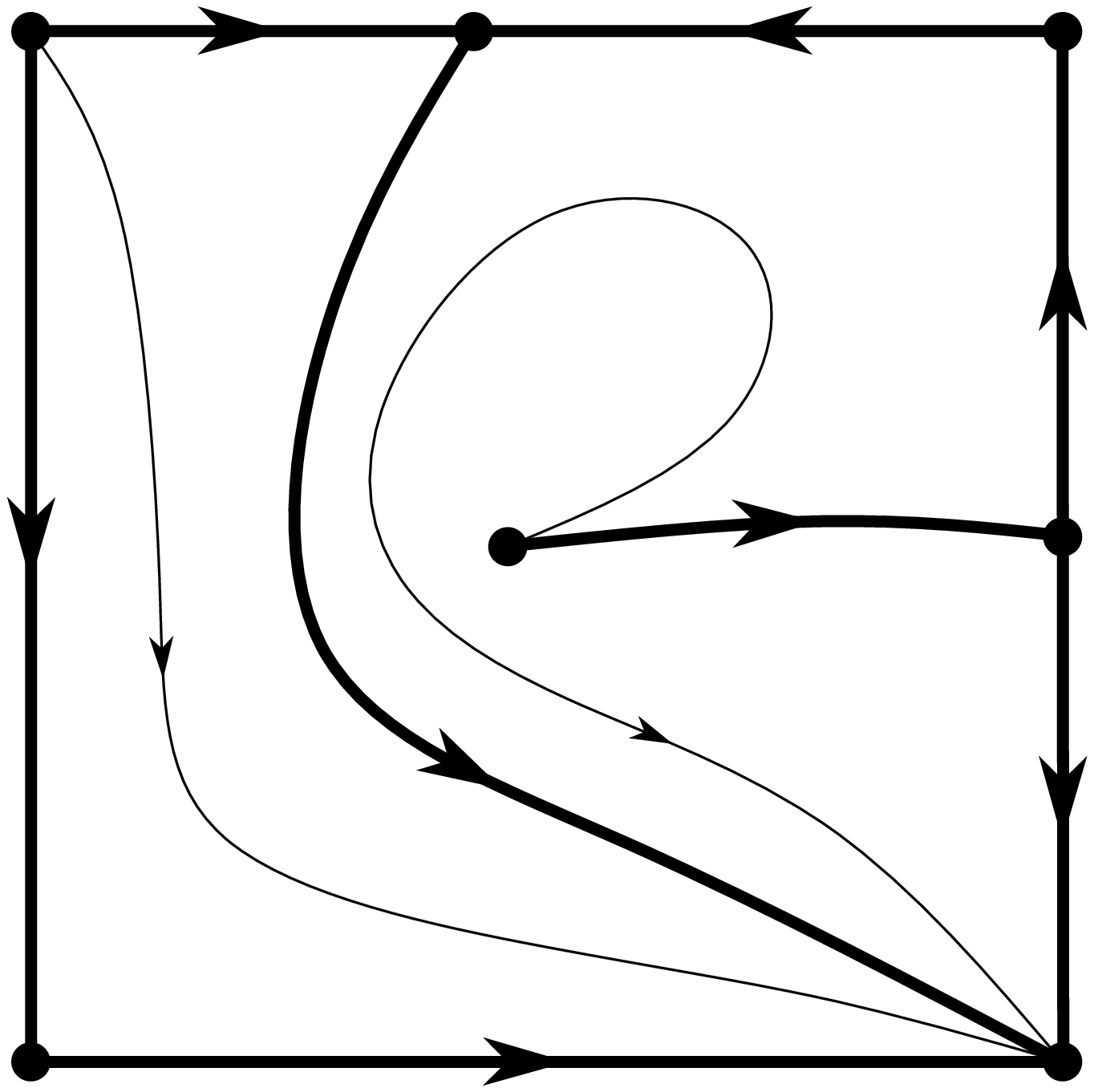} 
				\end{overpic}
				
				Case~$4.6a.i.1$.
			\end{center}
		\end{minipage}
		\begin{minipage}{3.1cm}
			\begin{center}
				\begin{overpic}[height=3cm]{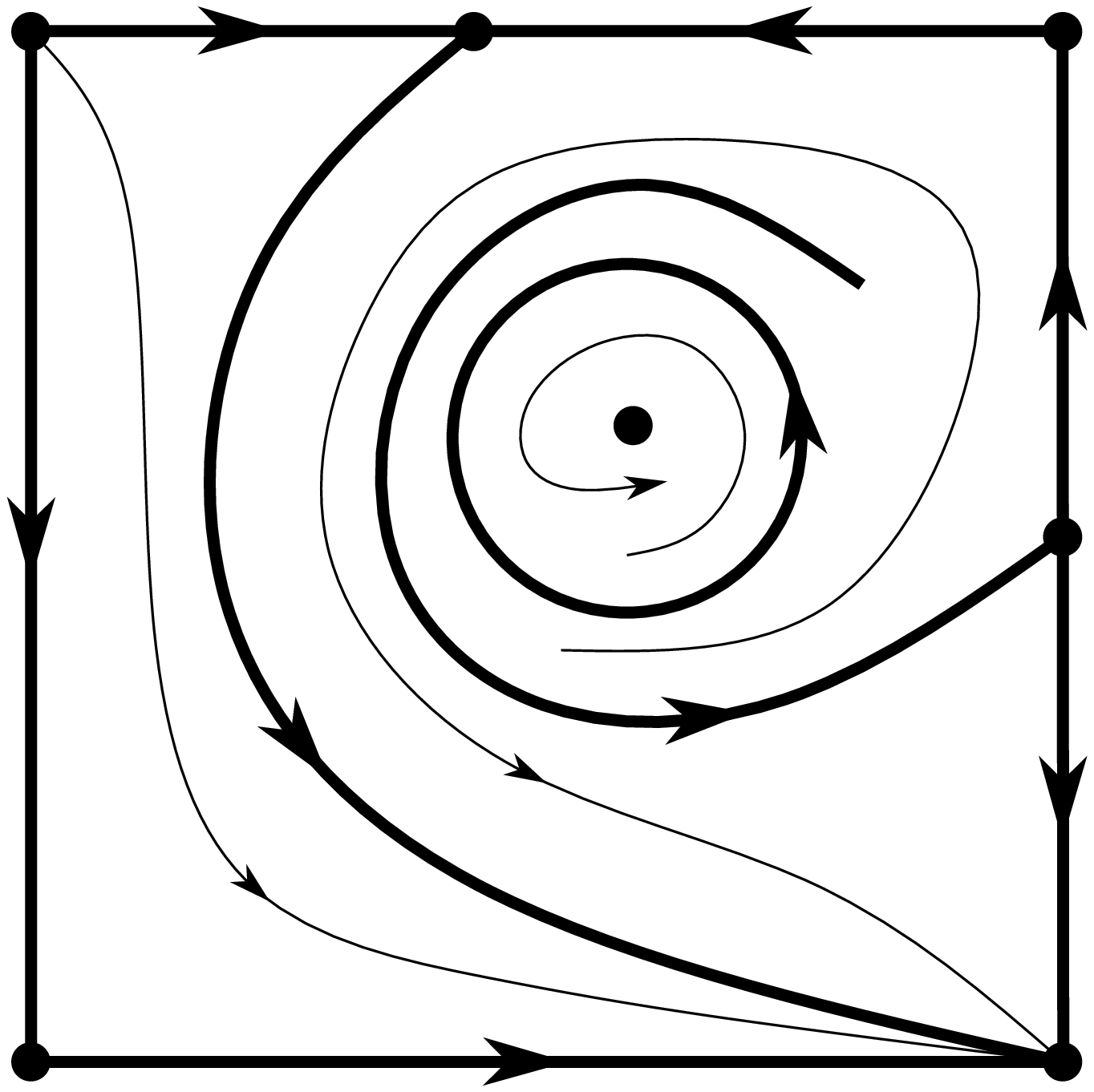} 
				\end{overpic}
				
				Case~$4.6a.ii.1$.
			\end{center}
		\end{minipage}
		\begin{minipage}{3.1cm}
			\begin{center}
				\begin{overpic}[height=3cm]{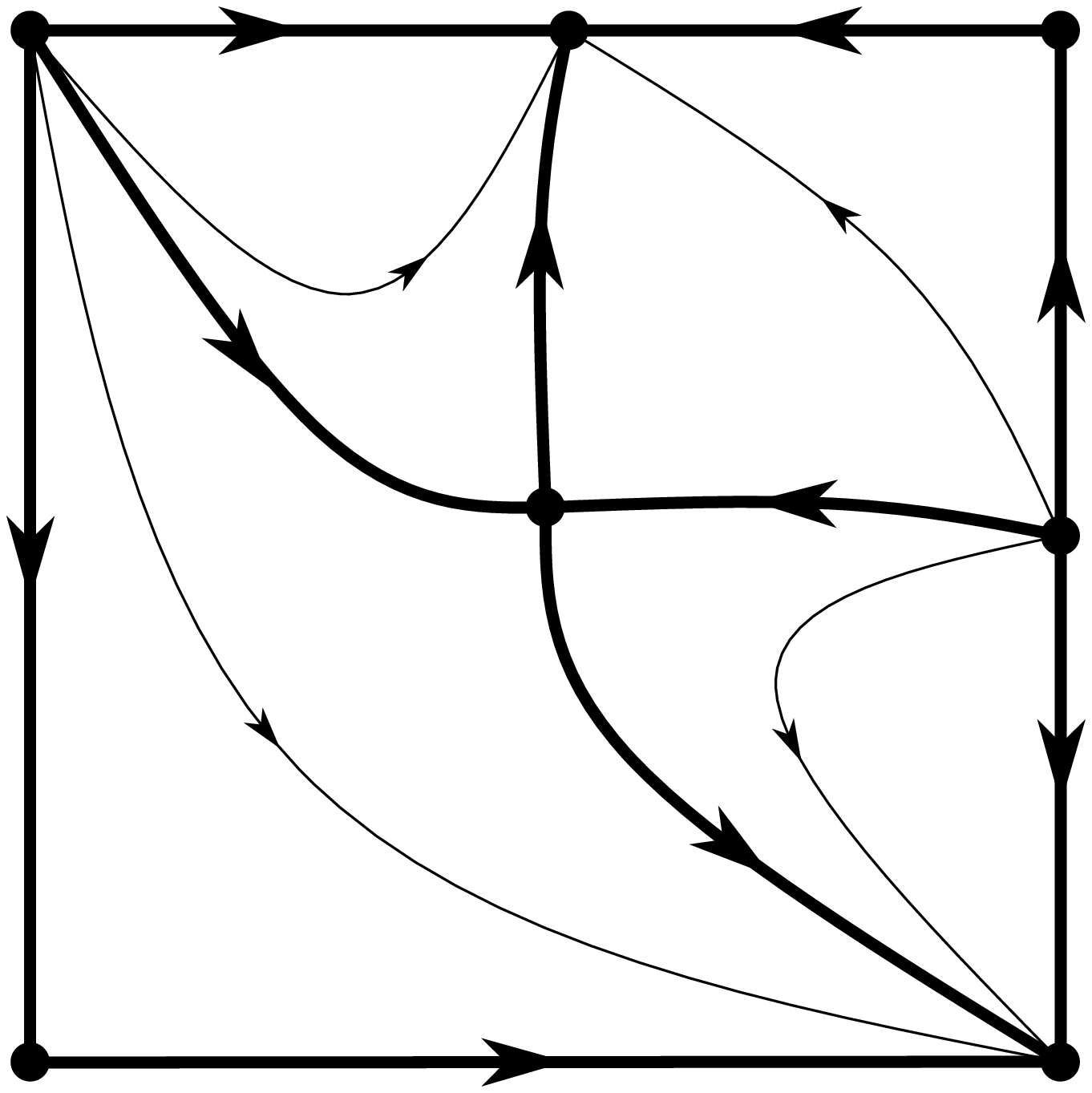} 
				\end{overpic}
				
				Case~$4.6b.i.1$.
			\end{center}
		\end{minipage}
	\end{center}
	\caption{Topologically distinct phase portraits under the hypothesis of Corollary~\ref{Main4}.}\label{FinalSquares}
\end{figure}
\end{corollary}

The paper is organized as follows. At Section~\ref{Sec4} we present some preliminaries about the Poincar\'e compactification, the notion of \emph{separatrix} and some known results about cubic systems with invariant straight lines. Theorem~\ref{Main1} is proved in Section~\ref{Sec5}. Theorem~\ref{Main2} is proved in Section~\ref{Sec6}. At Sections~\ref{Sec7} and \ref{Sec8} we study the singularities of system \eqref{6}. The phase portraits are studied in Section~\ref{Sec9}. Theorem~\ref{Main3} and Corollary~\ref{Main4} are proved in Section~\ref{Sec10}. Most of the phase portraits, and its characterizations, lies in the Appendix. 

\section{Preliminaries}\label{Sec4}

\subsection{The Poincar\'e Compactification}\label{SubSec4.1}

Let $X=(P,Q)$ be a planar \emph{polynomial} vector field of degree $n\in\mathbb{N}$ (observe that $n=3$ in our case). The \emph{Poincar\'e compactified vector field} $p(X)$ is an analytic vector field on $\mathbb{S}^2$ constructed as follows (for more details see \cite{Vel} or Chapter $5$ of \cite{DumLliArt2006}). First we identify $\mathbb{R}^2$ with the plane $(x_1,x_2,1)$ in $\mathbb{R}^3$ and define the \emph{Poincar\'e sphere} as $\mathbb{S}^2=\{y=(y_1,y_2,y_3)\in\mathbb{R}^3:y_1^2+y_2^2+y_3^2=1\}$. We define the \emph{northern hemisphere}, the \emph{southern hemisphere} and the \emph{equator} respectively by $H_+=\{y\in\mathbb{S}^2:y_3>0\}$, $H_-=\{y\in\mathbb{S}^2:y_3<0\}$ and $\mathbb{S}^1=\{y\in\mathbb{S}^2:y_3=0\}$. Consider now the projections $f_\pm:\mathbb{R}^2\rightarrow H_\pm$, given by $f_\pm(x_1,x_2)=\pm \Delta(x_1,x_2)(x_1,x_2,1)$, where $\Delta(x_1,x_2)=(x_1^2+x_2^2+1)^{-\frac{1}{2}}$. These two maps define two copies of $X$, one copy $X^+$ in $H_+$ and one copy $X^-$ in $H_-$. Consider the vector field $X'=X^+\cup X^-$ defined in $\mathbb{S}^2\backslash\mathbb{S}^1$. Note that the \emph{infinity} of $\mathbb{R}^2$ is identified with the equator $\mathbb{S}^1$. The Poincar\'e compactified vector field $p(X)$ is the analytic extension of $X'$ from $\mathbb{S}^2\backslash\mathbb{S}^1$ to $\mathbb{S}^2$ given by $y_3^{n-1}X'$. The \emph{Poincar\'e disk} $\mathbb{D}$ is the projection of the closed northern hemisphere to $y_3=0$ under $(y_1,y_2,y_3)\mapsto(y_1,y_2)$ (the vector field given by this projection will also be denoted by $p(X)$). Note that to know the behavior $p(X)$ near $\mathbb{S}^1$ is the same than to know the behavior of $X$ near the infinity. We define the local charts of $\mathbb{S}^2$ by $U_i=\{y\in\mathbb{S}^2:y_i>0\}$ and $V_i=\{y\in\mathbb{S}^2:y_i<0\}$ for $i\in\{1,2,3\}$. In these charts we define $\phi_i:U_i\rightarrow\mathbb{R}^2$ and $\psi_i:V_i\rightarrow\mathbb{R}^2$ by $\phi_i(y_1,y_2,y_3)=-\psi_i(y_1,y_2,y_3) = \left(\frac{y_m}{y_i},\frac{y_n}{y_i}\right)$, where $m\neq i$, $n\neq i$ and $m<n$. Denoting by $(u,v)$ the image of $\phi_i$ and $\psi_i$ in every chart (therefore $(u,v)$ will play different roles in each chart), one can see the following expression for $p(X)$ at chart $U_1$ is given by,
	\[\dot u = v^n \left[Q\left(\frac{1}{v},\frac{u}{v}\right)-uP\left(\frac{1}{v},\frac{u}{v}\right)\right], \quad \dot v = -v^{n+1}P\left(\frac{1}{v},\frac{u}{v}\right),\]
and at chart $U_2$ it is given by,
	\[\dot u = v^n \left[P\left(\frac{u}{v},\frac{1}{v}\right)-uQ\left(\frac{u}{v},\frac{1}{v}\right)\right], \quad \dot v = -v^{n+1}Q\left(\frac{u}{v},\frac{1}{v}\right).\]
The expressions of $p(X)$ in $V_1$ and $V_2$ is the same as that for $U_1$ and $U_2$, except by a multiplicative factor of $(-1)^{n-1}$. In these coordinates for $i\in\{1,2\}$, $v=0$ represents the points of $\mathbb{S}^1$. Thus, the infinity of $\mathbb{R}^2$. Note that $\mathbb{S}^1$ is invariant under the flow of $p(X)$. 

\subsection{The Markus-Neumann-Peixoto Theorem}\label{SubSec4.2}

Let $X$ be a \emph{polynomial} vector field, $p(X)$ its compacti\-fication defined on $\mathbb{D}$ and $\phi$ the flow defined by $p(X)$. The separatrices of $p(X)$ are:
\begin{enumerate}[label=\arabic*.]
	\item All the orbits contained in $\mathbb{S}^1$, i.e. at infinity;
	\item All the singular points;
	\item All the separatrices of the hyperbolic sectors of the finite and infinite singular points;
	\item All the limit cycles of $X$.
\end{enumerate}
Denote by $\mathcal{S}$ the set of all separatrices. Each connected component of $\mathbb{D}\backslash\mathcal{S}$ is called a \emph{canonical region} of the flow $(\mathbb{D},\phi)$. The \emph{separatrix configuration} $\mathcal{S}_c$ of a flow $(\mathbb{D},\phi)$, is the union of all the separatrices $\mathcal{S}$ of the flow, together with one orbit belonging to each canonical region. The separatrix configuration $\mathcal{S}_c$ of the flow $(\mathbb{D},\phi)$ is topologically equivalent to the separatrix configuration $\mathcal{S}_c^*$ of the flow $(\mathbb{D},\phi^*)$ if there exists an homeomorphism from $\mathbb{D}$ to $\mathbb{D}$ which transforms orbits of $\mathcal{S}_c$ into orbits of $\mathcal{S}_c^*$, orbits of $\mathcal{S}$ into orbits of $\mathcal{S}^*$ and preserves or reverses the orientation of all these orbits.

\begin{theorem}[Markus-Neumann-Peixoto]
	Let $p(X)$ and $p(Y)$ be two Poin\-car\'e compactifications in the Poincar\'e disk $\mathbb{D}$ of the two polynomial vector fields $X$ and $Y$, with finitely many singularities. Then the phase portraits of $p(X)$ and $p(Y)$ are topologically equivalent if and only if their separatrix configurations are topologically equivalent.
\end{theorem}

\begin{proof} See \cites{Mar1954,Neu1975,Pei1971,CorrectionMNP} and Section~$1.9$ of \cite{DumLliArt2006}. \end{proof}

\subsection{Some results about cubic systems with invariant lines}\label{SubSec4.3}

In this subsection we present some results due to Kooij \cite{Kooij} about cubic systems with four invariant straight lines. 

\begin{theorem}[Theorem~$2.5.2$ of \cite{Kooij}]\label{Theo4}
	If a cubic system with exactly two pairs of parallel invariant real lines, where coinciding lines are counted double, has a weak focus, then, it has no limit cycles.
\end{theorem}

We recall that a singularity $p$ of a planar analytic vector field is a \emph{degenerated monodromic singularity} if the eigenvalues of the linearized system at $p$ are given by $\lambda_i=\pm\omega i$, with $\omega>0$ and $i^2=-1$. If the origin is a degenerated monodromic singularity, then the canonical form of the system is given by,
\begin{equation}\label{40}
	\dot x= -y+F_1(x,y), \quad \dot y = x+F_2(x,y),
\end{equation}
with the linear parts of $F_1$ and $F_2$ being zero. In this case, it is well known that there is a function $V(x,y)$ defined in a neighborhood of the origin such that its rate of change along the orbits of \eqref{40} is given by the formal series,
\[\dot V= V_1(x^2+y^2)^2+V_2(x^2+y^2)^3+\dots.\]
The constants $V_i$ are the focal values. The stability of the origin is determined by the first non-vanishing focal value. In special, the origin is a center if, and only if, $V_i=0$ for all $i\in\mathbb{N}$. For more details, see Section~$4.2$ of \cite{DumLliArt2006} and the references therein.

\begin{theorem}[Corollary~$2.5.3$ of \cite{Kooij}]\label{Theo5}
	If the first focal value of a degenerated monodromic singularity of a cubic system with exactly two pairs of parallel invariant real lines, where coinciding lines are counted double, vanishes, then, it is a center.
\end{theorem}

\begin{theorem}[Theorem~$2.5.4$ of \cite{Kooij}]\label{Theo6}
	A cubic system with exactly two pairs of parallel invariant real lines has at most one limit cycle and if it exists, it is hyperbolic.
\end{theorem}

Observe that Theorems~\ref{Theo4}-\ref{Theo6} trivialize two big problems about the classification of phase portraits. Namely, the center-focus problem and the number of limit cycles. Such results will be of great help to deal with item $(b)$ of the definition of a generic vector field $X\in\mathfrak{X}$. For similar results about the number of limit cycles in planar systems of degree $n$ with $n$ or $n+1$ invariant straight lines, see the works of Llibre and Rodrigues \cite{LliRod} and Kooij \cite{Kooij2}, respectively. For a cubic system with a cusp, see the work of Xian and  Kooij \cite{XianKooij}. See also the works of Cozma and Suba \cites{CozSub1,CozSub2} for related center-focus problem of cubic systems.

\section{Proof of Theorem~\ref{Main1}}\label{Sec5}

\noindent \textit{Proof of Theorem~\ref{Main1}.} Let $X\in\mathfrak{X}$ be given by,
	\[P(x,y)=x(x-1)(a_{00}+a_{10}x+a_{01}y), \quad Q(x,y)=y(y-1)(b_{00}+b_{10}x+b_{01}y),\]
and
	\[A=\left(\begin{array}{cc} a_{10} & a_{01} \\ b_{10} & b_{01} \end{array}\right).\]
We want to prove that if $X\in\Sigma_0$ (see Definition~\ref{Def1}), then $a_{10}b_{01}\det A\neq0$. To prove this, we look to the singularities at infinity. It follows from Subsection~\ref{SubSec4.1} that the compactified vector field $p(X)$ at chart $U_1$ is given by,
\begin{equation}\label{43}
	\begin{array}{rl}
		P_1(u,v) &= -a_{10}u+(b_{10}-a_{01})u^2+(a_{10}-a_{00}-b_{10})uv \vspace{0.2cm} \\
		&\quad+b_{01}u^3+(a_{01}+b_{00}-b_{01})u^2v+(a_{00}-b_{00})uv^2, \vspace{0.2cm} \\
		Q_1(u,v) &= -a_{10}v-a_{01}uv+(a_{10}-a_{00})v^2+a_{01}uv^2+a_{00}v^3.
	\end{array}	
\end{equation}
Let $X_1$ denote the vector field given by \eqref{43}. At chart $U_2$, $p(X)$ is given by,
\begin{equation}\label{47}
	\begin{array}{rl}
		P_2(u,v) &= -b_{01}u+(a_{01}-b_{10})u^2+(b_{01}-b_{00}-a_{01})uv \vspace{0.2cm} \\
		&\quad+a_{10}u^3+(a_{00}-a_{10}+b_{10})u^2v+(b_{00}-a_{00})uv^2, \vspace{0.2cm} \\
		Q_2(u,v) &= -b_{01}v-b_{10}uv+(b_{01}-b_{00})v^2+b_{10}uv^2+b_{00}v^3.
	\end{array}	
\end{equation}
Let $X_2$ denote the vector field given by \eqref{47}. Observe that the origin of both $X_1$ and $X_2$ are singularities with their Jacobian matrices given by,
	\[DX_1(0,0)=\left(\begin{array}{cc} -a_{10} & 0 \vspace{0.2cm} \\ 0 & -a_{10} \end{array}\right), \quad DX_2(0,0)=\left(\begin{array}{cc} -b_{01} & 0 \vspace{0.2cm} \\ 0 & -b_{01} \end{array}\right).\]
Therefore, if $a_{10}b_{01}=0$, then $p(X)$ has a non-hyperbolic singularity. Thus, $X\not\in\Sigma_0$. To finish the proof, we assume $a_{10}b_{01}\neq0$, $\det A=0$ and conclude that $X\not\in\Sigma_0$. Let,
	\[\mathcal{P}(x,y)=a_{00}+a_{10}x+a_{01}y, \quad \mathcal{Q}(x,y)=b_{00}+b_{10}x+b_{01}y,\]
and $r_1$, $r_2$ be the straight lines given by $r_1=\mathcal{P}^{-1}(0)$ and $r_2=\mathcal{Q}^{-1}(0)$. It follows from $\det A=0$ that either $r_1=r_2$, or $r_1$ and $r_2$ are distinct and parallel. If $r_1=r_2$, then such straight line is a set of singularities. Therefore, $X$ has infinitely many singularities. Thus, $X\not\in\Sigma_0$. Suppose that $r_1$ and $r_2$ are distinct and parallel. We claim that at least one singularity at infinity is non-hyperbolic. Observe that the singularities of $X_1$ with $v=0$ are given by $u=0$ or by the real solutions of,
\begin{equation}\label{44}
	b_{01}u^2+(b_{10}-a_{01})u-a_{10}=0.
\end{equation}
Since $b_{01}\neq0$, it is clear that the solutions of \eqref{44} are given by,
\begin{equation}\label{45}
	u_0^\pm=\frac{(a_{01}-b_{10})\pm\sqrt{\Delta}}{2b_{01}}, \quad \Delta=(b_{10}-a_{01})^2+4a_{10}b_{01}.
\end{equation}
It follows from $\det A=0$ that $a_{10}b_{01}=a_{01}b_{10}$. Replacing this at $\Delta$ we get,
	\[\Delta=(b_{10}+a_{01})^2\geqslant0.\]
Hence, $u_0^\pm$ are real. It also follows from $\det A=0$ that,
\begin{equation}\label{46}
	a_{10}=\frac{a_{01}b_{10}}{b_{01}}.
\end{equation}
If we calculate the determinant of the Jacobian matrices of $X_1$ at $(u_0^\pm,0)$ and then replace \eqref{46}, we get,
	\[\det DX_1\bigl(u_0^\pm,0\bigr)=-\frac{a_{01}^2(a_{01}+b_{10})}{2b_{01}^2}\bigl((a_{01}+b_{10})\pm|a_{01}+b_{10}|\bigr).\]
Therefore, $(u_0^+,0)$ or $(u_0^-,0)$ is a non-hyperbolic singularity of $X_1$. Hence, $p(X)$ has a non-hyperbolic singularity. Thus, $X\not\in\Sigma_0$. {\hfill$\square$}

\section{Proof of Theorem~\ref{Main2}}\label{Sec6}

\noindent \textit{Proof of Theorem~\ref{Main2}.} Given a vector field $X\in\Sigma_0$, in this proof we work on its equivalent form,
\begin{equation}\label{41}
	\dot x=(x+\alpha)(x+\alpha-1)(a_{10}x+a_{01}y), \quad \dot y=(y+\beta)(y+\beta-1)(b_{10}x+b_{01}y).
\end{equation}
Consider the maps $\varphi_i\colon\mathbb{R}^8\to\mathbb{R}^8$, $i\in\{1,2,3,4\}$, given by,
\[\begin{array}{rl} 
	\varphi_1(x,y;\alpha,\beta,a_{10},a_{01},b_{10},b_{01}) &=(y,x;\beta,\alpha,b_{01},b_{10},a_{01},a_{10}); \vspace{0.2cm} \\
	\varphi_2(x,y;\alpha,\beta,a_{10},a_{01},b_{10},b_{01}) &= (-y,-x;1-\beta,1-\alpha,b_{01},b_{10},a_{01},a_{10}); \vspace{0.2cm} \\
	\varphi_3(x,y;\alpha,\beta,a_{10},a_{01},b_{10},b_{01}) &= (-x,y;1-\alpha,\beta,a_{10},-a_{01},-b_{10},b_{01}); \vspace{0.2cm} \\
	\varphi_4(x,y;\alpha,\beta,a_{10},a_{01},b_{10},b_{01}) &= (x,-y;\alpha,1-\beta,a_{10},-a_{01},-b_{10},b_{01}).
\end{array}\]
If we handle the parameters $\alpha$, $\beta$, $a_{ij}$ and $b_{ij}$ as variables, then we can write system \eqref{41} as the eight-dimensional system given by,
\begin{equation}\label{8}
	\begin{array}{l}
		\dot x = (x+\alpha)(x+\alpha-1)(a_{10}x+a_{01}y), \quad \dot y = (y+\beta)(y+\beta-1)(b_{10}x+b_{01}y), \vspace{0.2cm} \\
		\dot \alpha=0, \quad \dot \beta=0, \quad \dot a_{10}=0, \quad \dot a_{01}=0, \quad \dot b_{10}=0, \quad \dot b_{01}=0.
	\end{array}
\end{equation}
Therefore, let $Y=(P,Q;R_1,R_2,R_3,R_4,R_5,R_6)$ be the vector field given by system \eqref{8}, i.e.
\[\begin{array}{rl}
	P(x,y;\alpha,\beta,a_{10},a_{01},b_{10},b_{01}) &=(x+\alpha)(x+\alpha-1)(a_{10}x+a_{01}y), \vspace{0.2cm} \\ Q(x,y;\alpha,\beta,a_{10},a_{01},b_{10},b_{01}) &=(y+\beta)(y+\beta-1)(b_{10}x+b_{01}y), \vspace{0.2cm} \\
	R_j(x,y;\alpha,\beta,a_{10},a_{01},b_{10},b_{01}) &=0,
\end{array}\]
$j\in\{1,\dots,6\}$. We claim that $Y=(\varphi_i)_*Y$. More precisely, we claim that
\begin{equation}\label{9}
	Y(z)=D\varphi_i(\varphi_i^{-1}(z))\cdot Y(\varphi_i^{-1}(z)),
\end{equation}
for $i\in\{1,2,3,4\}$ and $z\in\mathbb{R}^8$. To prove this claim, we observe that,
\begin{equation}\label{49}
\begin{array}{ll}
	Y=(\varphi_1)_*Y \Leftrightarrow \left\{\begin{array}{rl} P &= Q\circ\varphi_1, \vspace{0.2cm} \\ Q & =P\circ\varphi_1, \end{array}\right. &
	Y=(\varphi_2)_*Y \Leftrightarrow \left\{\begin{array}{rl} P &= -Q\circ\varphi_2, \vspace{0.2cm} \\ Q &= -P\circ\varphi_2, \end{array}\right. \vspace{0.2cm} \\
	Y=(\varphi_3)_*Y \Leftrightarrow \left\{\begin{array}{rl} P &= -P\circ\varphi_3, \vspace{0.2cm} \\ Q &= Q\circ\varphi_3, \end{array}\right. &
	Y=(\varphi_4)_*Y \Leftrightarrow \left\{\begin{array}{rl} P &= P\circ\varphi_4, \vspace{0.2cm} \\ Q &= -Q\circ\varphi_4. \end{array}\right. 
\end{array}
\end{equation}
Simple calculations show that the right-hand sides of \eqref{49} holds. Thus, \eqref{9} also holds. Let $X\in\Sigma_0$. If the origin is under position $5$, it is clear that $(\varphi_3)_*X\in\Sigma_0$ has the origin under position $3$. Similarly, $\varphi_2$ and $\varphi_4$ sends positions $7$ and $9$ to position $3$. Therefore, we have the topological equivalence between positions $3$, $5$, $7$ and $9$. The equivalence of positions $2$, $4$, $6$ and $8$ follows similarly. We now look at the second claim of the theorem, about Tables~\ref{Table8}, \ref{Table9} and \ref{Table10}. We first focus at position $1$ and then obtain Table~\ref{Table8}. Consider the four parameters,
\begin{equation}\label{34}
	(a_{10},a_{01},b_{10},b_{01})\in\mathbb{R}^4.
\end{equation}
Assume $a_{10}a_{01}b_{10}b_{01}\neq0$. Hence, we have sixteen possible cases given by the signals of such parameters. To simplify the notation, we denote such cases by,
\begin{multicols}{4}
	\begin{enumerate}[label=\arabic*)]
		\item $(0,0,0,0)$, 
		\item $(0,0,0,1)$, 
		\item $(0,0,1,0)$, 
		\item $(0,1,0,0)$, 
	\end{enumerate}
	\columnbreak
	\begin{enumerate}[label=\arabic*)]
		\setcounter{enumi}{4}
		\item $(1,0,0,0)$, 
		\item $(0,0,1,1)$, 
		\item $(0,1,0,1)$, 
		\item $(1,0,0,1)$,
	\end{enumerate}
	\columnbreak
	\begin{enumerate}[label=\arabic*)]
		\setcounter{enumi}{8}
		\item $(1,1,1,1)$, 
		\item $(1,1,1,0)$, 
		\item $(1,1,0,1)$, 
		\item $(1,0,1,1)$, 
	\end{enumerate}
	\columnbreak
	\begin{enumerate}[label=\arabic*)]
		\setcounter{enumi}{12}
		\item $(0,1,1,1)$, 
		\item $(1,1,0,0)$, 
		\item $(1,0,1,0)$, 
		\item $(0,1,1,0)$, 
	\end{enumerate}
\end{multicols}
\noindent where $1$ denotes that its respective parameter, according to \eqref{34}, is positive and $0$ denote that it is negative. For example, case $(3)$ denote,
	\[a_{10}<0, \quad a_{01}<0, \quad b_{10}>0, \quad b_{01}<0,\]
while case $(15)$ denotes,
	\[a_{10}>0, \quad a_{01}<0, \quad b_{10}>0, \quad b_{01}<0.\]
With the change of time $\tau= -t$, one can conclude that cases $k$ and $k+8$ are equivalent, for $k\in\{1,\dots,8\}$. Thus, we have the eight classes given by,
\begin{multicols}{2}
	\begin{enumerate}[label=\arabic*)]
		\item $\{(0,0,0,0),(1,1,1,1)\}$, 
		\item $\{(0,0,0,1),(1,1,1,0)\}$,
		\item $\{(0,0,1,0),(1,1,0,1)\}$, 
		\item $\{(0,1,0,0),(1,0,1,1)\}$, 
	\end{enumerate}
	\columnbreak
	\begin{enumerate}[label=\arabic*)]
		\setcounter{enumi}{4}
		\item $\{(1,0,0,0),(0,1,1,1)\}$, 
		\item $\{(0,0,1,1),(1,1,0,0)\}$, 
		\item $\{(0,1,0,1),(1,0,1,0)\}$, 
		\item $\{(1,0,0,1),(0,1,1,0)\}$.
	\end{enumerate}
\end{multicols}
It is clear that positions $1$ is invariant by $\varphi_4$. Hence, we can apply $\varphi_4$ and then reduce the above eight classes to the following four.
\begin{enumerate}[label=\arabic*)]
	\item $\{(0,0,0,0),(1,1,1,1),(1,0,0,1),(0,1,1,0)\}$, 
	\item $\{(0,0,1,0),(1,1,0,1),(0,1,0,0),(1,0,1,1)\}$,
	\item $\{(0,0,0,1),(1,1,1,0),(1,0,0,0),(0,1,1,1)\}$,
	\item $\{(0,0,1,1),(1,1,0,0),(0,1,0,1),(1,0,1,0)\}$.
\end{enumerate}
Therefore, if $a_{10}a_{01}b_{10}b_{01}\neq0$, then it is enough to completely describe one case of each of the above four classes. In special, observe that each class has a member given by $(1,1,i,j)$, with $i$, $j\in\{0,1\}$. Hence, if $a_{10}a_{01}b_{10}b_{01}\neq0$, then we can assume $a_{10}>0$ and $a_{01}>0$. Thus, if we also consider the cases where $a_{10}a_{01}b_{10}b_{01}=0$, then we can assume $a_{10}\geqslant0$, $a_{01}\geqslant0$ and $b_{10}$, $b_{01}\in\mathbb{R}$. Observe that the Jacobian matrix of $X$ at the origin is given by,
	\[DX(0,0)=\left(\begin{array}{cc} a_{10}(\alpha-1)\alpha & a_{01}(\alpha-1)\alpha \\ b_{10}(\beta-1)\beta & b_{01}(\beta-1)\beta \end{array}\right).\]
Since every singularity of $X\in\Sigma_0$ is hyperbolic, it follows that $\alpha$, $\beta\not\in\{0,1\}$. In special, under position $1$, we conclude $0<\alpha<1$ and $0<\beta<1$. Consider the diffeomorphism $\varphi_5=\varphi_4\circ\varphi_3$ given by,
	\[\varphi_5(x,y;\alpha,\beta,a_{10},a_{01},b_{10},b_{01})=(-x,-y;1-\alpha,1-\beta,a_{10},a_{01},b_{10},b_{01}).\]
Since position $1$ is invariant by $\varphi_5$, we can assume $\beta\geqslant\frac{1}{2}$. Thus, we obtain Table~\ref{Table8}. Tables~\ref{Table9} and \ref{Table10} follows similarly. One need only to observe that positions $2$ and $3$ are invariant by $\varphi_4$ and $\varphi_1$, respectively. {\hfill$\square$}

\section{The finite singularities}\label{Sec7}

Let $X\in\Sigma_0$. Our goal in this section is to study the dynamics of the equivalent system of differential equations,
\begin{equation}\label{19}
	\dot x=(x+\alpha)(x+\alpha-1)(a_{10}x+a_{01}y), \quad \dot y=(y+\beta)(y+\beta-1)(b_{10}x+b_{01}y),
\end{equation}
at its finite singularities. 
\begin{proposition}\label{Theo7}
About system \eqref{19}, the following statements hold.
\begin{enumerate}[label=\arabic*)]
	\item If $(\alpha-1)\alpha(\beta-1)\beta\det A<0$, then the origin is a hyperbolic saddle; 
	\item If $(\alpha-1)\alpha(\beta-1)\beta\det A>0$, 
	\begin{enumerate}[label=(\alph*)]
		\item and $a_{10}(\alpha-1)\alpha+b_{01}(\beta-1)\beta>0$, then the origin is a hyperbolic unstable node/focus; 
		\item and $a_{10}(\alpha-1)\alpha+b_{01}(\beta-1)\beta<0$, then the origin is a hyperbolic stable node/focus; 
		\item $a_{10}(\alpha-1)\alpha+b_{01}(\beta-1)\beta=0$, 
		\begin{enumerate}[label=(\roman*)]
			\item and $b_{10}b_{01}(a_{01}+b_{10})>0$, then the origin is a weak unstable focus;
			\item and $b_{10}b_{01}(a_{01}+b_{10})<0$, then the origin is a weak stable node;
			\item and $b_{01}(a_{01}+b_{10})=0$, then the origin is a center.
		\end{enumerate}
	\end{enumerate}
\end{enumerate}
\end{proposition}

\begin{proof} Statements $(1)$, $(2)(a)$ and $(2)(b)$ follow direct from the fact that the Jacobian matrix of system \eqref{19} at the origin is given by,
	\[DX(0,0)=\left(\begin{array}{cc} a_{10}(\alpha-1)\alpha & a_{01}(\alpha-1)\alpha \\ b_{10}(\beta-1)\beta & b_{01}(\beta-1)\beta \end{array}\right).\]
Hence, its determinant and its trace are given by,
	\[\det DX(0,0) =(\alpha-1)\alpha(\beta-1)\beta\det A, \quad Tr\bigl(DX(0,0)\bigr)=a_{10}(\alpha-1)\alpha+b_{01}(\beta-1)\beta.\]
If the origin is a degenerated monodromic singularity, then using the formula given by Kertész and Kooij \cite{KerKoo}, one can calculate the first focal value of system \eqref{19} at the origin and then see that it is given by,
	\[V_1=\frac{b_{01}(a_{01}+b_{10})}{8b_{10}}.\]
This and Theorem~\ref{Theo5} prove the other statements. \end{proof}

\begin{proposition}\label{Theo8}
	If $X\in\Sigma_0$, then the local phase portrait of system \eqref{19} at each one of its finite singularities, except the origin, is fully described, provided they are hyperbolic, by:
	\begin{enumerate}[label=(\alph*)]
		\item The position of such singularities;
		\item The position of the origin in relation to the octothorpe; 
		\item The signals of $a_{10}$, $b_{01}$ and $\det A$.
	\end{enumerate}
\end{proposition}	

\begin{proof} Besides the origin, system \eqref{19} can have at most eight other singularities, which can be divided in two groups. The first group of singularities are those given by the intersection of the four invariant lines of system \eqref{19}. Namely,
\begin{equation}\label{20}
	p_1=(-\alpha,-\beta), \quad p_2=(1-\alpha,-\beta), \quad p_3=(1-\alpha,1-\beta) \quad p_4=(-\alpha,1-\beta).
\end{equation}
We refer to them as the $p$-singularities. To understand the eigenvalue of the Jacobian matrices of system \eqref{19} at the $p$-singularities, we need first to define the linear functions,
\begin{equation}\label{21}
	f_1(x,y)=a_{10}x+a_{01}y, \quad f_2(x,y)=b_{10}x+b_{01}y.
\end{equation}
With the definitions above, it can be seen that the Jacobian matrices of system \eqref{19} at the $p$-singularities are given by,
\begin{equation}\label{22}
\begin{array}{l}
	\displaystyle DX(p_1)=\left(\begin{array}{cc} -f_1(p_1) & 0 \\ 0 & -f_2(p_1) \end{array}\right), \quad DX(p_2)=\left(\begin{array}{cc} f_1(p_2) & 0 \\ 0 & -f_2(p_2) \end{array}\right), \vspace{0.2cm} \\
	\displaystyle DX(p_3)=\left(\begin{array}{cc} f_1(p_3) & 0 \\ 0 & f_2(p_3) \end{array}\right), \quad DX(p_4)=\left(\begin{array}{cc} -f_1(p_4) & 0 \\ 0 & f_2(p_4) \end{array}\right).
\end{array}
\end{equation}
The second group of singularities are those given (when well defined) by, 
\begin{equation}\label{23}
	q_1=\left(\frac{a_{01}}{a_{10}}\beta,-\beta\right), \quad q_2=\left(1-\alpha,\frac{b_{10}}{b_{01}}(\alpha-1)\right), \quad q_3=\left(\frac{a_{01}}{a_{10}}(\beta-1),1-\beta\right), \quad q_4=\left(-\alpha,\frac{b_{10}}{b_{01}}\alpha\right).
\end{equation}
See Figure~\ref{Fig1}. We refer to them as the $q$-singularities.
\begin{figure}[h]
	\begin{center}
		\begin{overpic}[height=7cm]{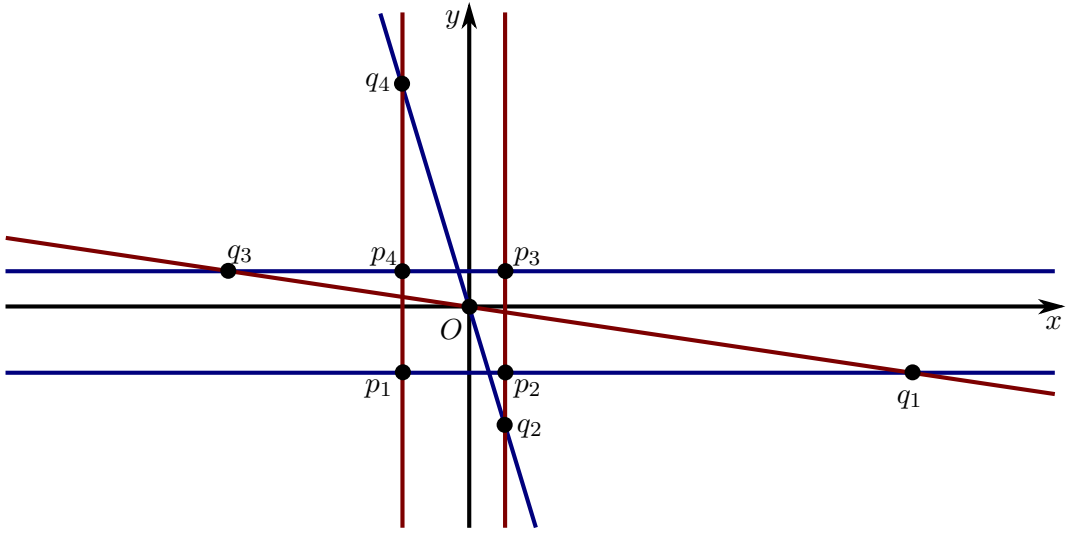} 
			\put(34,13){$p_1$}
			\put(48,13){$p_2$}
			\put(48,25.5){$p_3$}
			\put(34.5,25.5){$p_4$}
			\put(84,12){$q_1$}
			\put(48.25,9.25){$q_2$}
			\put(21,25.75){$q_3$}
			\put(34,42){$q_4$}
			\put(41,18){$O$}
			\put(98,19){$x$}
			\put(41.5,48){$y$}
		\end{overpic}
	\end{center}
	\caption{An example of the disposition of the $p$ and $q$-singularities. The sets $\{\dot x=0\}$ and $\{\dot y=0\}$ are illustrated in red and blue, respectively.}\label{Fig1}
\end{figure}

The Jacobian matrices of system \eqref{19} at the $q$-singularities are given by,
\begin{equation}\label{24}
\begin{array}{l}
	\displaystyle DX(q_1)=\frac{1}{a_{10}}\left(\begin{array}{cc} f_1(p_1)f_1(p_2) & \star \\ 0 & \beta \det A \end{array}\right), \quad DX(q_2)=\frac{1}{b_{01}}\left(\begin{array}{cc} (1-\alpha)\det A & 0 \\ \star & f_2(p_2)f_2(p_3) \end{array}\right), \vspace{0.2cm} \\
	\displaystyle DX(q_3)=\frac{1}{a_{10}}\left(\begin{array}{cc} f_1(p_3)f_1(p_4) & \star \\ 0 & (1-\beta)\det A \end{array}\right), \quad DX(q_4)=\frac{1}{b_{01}}\left(\begin{array}{cc} \alpha \det A & 0 \\ \star & f_2(p_1)f_2(p_4) \end{array}\right).
\end{array}
\end{equation}
Knowing that $p_1$, $p_2$ and $q_1$ lies on the horizontal straight line given by $y=-\beta$, we establish some nomenclatures to study their position. If $q_1$ is on the left-hand side of $p_1$, then we denote $q_1<p_1$; if $q_1$ is on the right-hand side of $p_1$, then we denote $q_1>p_1$. The relative positions between $q_1$ and $p_2$ will be denoted in the similar way. Observe that we have $p_1<p_2$. It follows from \eqref{20} and \eqref{23} that:
\begin{enumerate}[label=\arabic*)]
	\item $q_1<p_1$ if, and only if, $a_{10}f_1(p_1)>0$; 
	\item $q_1<p_2$ if, and only if, $a_{10}f_1(p_2)>0$.
\end{enumerate}
In special, $q_1=p_1$ (resp. $q_1=p_2$) if and only if, $f_1(p_1)=0$ (resp. $f_1(p_2)=0$). In the same way (this time denoting $q_2<p_2$ if $q_2$ lies bellow $p_2$), one can easily see that,
\begin{enumerate}[label=\arabic*)]
	\setcounter{enumi}{2}
	\item $q_2<p_2$ if, and only if, $b_{01}f_2(p_2)>0$; 
	\item $q_2<p_3$ if, and only if, $b_{01}f_2(p_3)>0$. 
\end{enumerate}
Finally, we observe that,
\begin{enumerate}[label=\arabic*)]
	\setcounter{enumi}{4}
	\item $q_3<p_3$ if, and only if, $a_{10}f_1(p_3)>0$; 
	\item $q_3<p_4$ if, and only if, $a_{10}f_1(p_4)>0$;
	\item $q_4<p_4$ if, and only if, $b_{01}f_2(p_4)>0$; 
	\item $q_4<p_1$ if, and only if, $b_{01}f_2(p_1)>0$.
\end{enumerate}
Therefore, if we fix the signals of $a_{10}$ and $b_{01}$, then the signals of $f_i(p_j)$ follows from the relative positions between the $p$ and $q$-singularities. Hence, it follows from \eqref{22} and from the \emph{Hartman-Grobman Theorem} (see Section~$2.8$ of \cite{Perko}), that we can describe the local phase portraits at the $p$-singularities only by means of their position in relation to the $q$-singularities. Furthermore, observe that the signals of $\alpha$, $\beta$, $1-\alpha$ and $1-\beta$ are given by the position of the origin in relation to the invariant octothorpe. Therefore, the position of the origin in relation to the ocotothorpe, the signal of $\det A$ and \eqref{24} provide the local phase portrait at the $q$-singularities. \end{proof}

\section{The singularities at infinity}\label{Sec8}

Let $X\in\Sigma_0$. Our goal in this section is to study the dynamics of the equivalent system of differential equations,
\begin{equation}\label{48}
	\dot x=(x+\alpha)(x+\alpha-1)(a_{10}x+a_{01}y), \quad \dot y=(y+\beta)(y+\beta-1)(b_{10}x+b_{01}y),
\end{equation}
at the infinity. Without loss of generality, let $X$ denote the vector field given by \eqref{48}. It follows from Subsection~\ref{SubSec4.1} that the compactified vector field $p(X)$ at chart $U_1$ is given by,
\begin{equation}\label{25}
	\begin{array}{rl}
		\dot u &= -a_{10}u+(b_{10}-a_{01})u^2+[a_{10}(1-2\alpha)+b_{10}(2\beta-1)]uv+b_{10}\beta(\beta-1)v^2+b_{01}u^3+ \vspace{0.2cm} \\
		&\quad+[a_{01}(1-2\alpha)+b_{01}(2\beta-1)]u^2v+[a_{10}\alpha(1-\alpha)+b_{01}\beta(\beta-1)]uv^2+a_{01}\alpha(1-\alpha)u^2v^2, \vspace{0.2cm} \\
		\dot v &= -a_{10}v-a_{01}uv+a_{10}(1-2\alpha)v^2+a_{01}(1-2\alpha)uv^2+a_{10}\alpha(1-\alpha)v^3+a_{01}\alpha(1-\alpha)uv^3,
	\end{array}	
\end{equation}
while at chart $U_2$ it is given by,
\begin{equation}\label{26}
	\begin{array}{rl}
		\dot u &= -b_{01}u+(a_{01}-b_{10})u^2+[b_{01}(1-2\beta)+a_{01}(2\alpha-1)]uv+a_{01}\alpha(\alpha-1)v^2+a_{10}u^3+ \vspace{0.2cm} \\
		&\quad+[a_{10}(2\alpha-1)+b_{10}(1-2\beta)]u^2v+[a_{10}\alpha(\alpha-1)+b_{01}\beta(1-\beta)]uv^2+b_{10}\beta(1-\beta)u^2v^2, \vspace{0.2cm} \\
		\dot v &=-b_{01}v-b_{10}uv+b_{01}(1-2\beta)v^2+b_{10}(1-2\beta)uv^2+b_{01}\beta(1-\beta)v^3+b_{10}\beta(1-\beta)uv^3.
	\end{array}	
\end{equation}
Let $X_1$ and $X_2$ denote the vector fields given by systems \eqref{25} and \eqref{26}, respectively. It follows from \eqref{25} that the singularities of $X_1$, with $v=0$, are given by $u=0$ and by the real zeros of,
\begin{equation}\label{27}
	g(u):=b_{01}u^2+(b_{10}-a_{01})u-a_{10}.
\end{equation}
Since $b_{01}\neq0$, these solutions are given by,
\begin{equation}\label{28}
	u_0^\pm=\frac{(a_{01}-b_{10})\pm\sqrt{\Delta}}{2b_{01}}, \quad \Delta=(b_{10}-a_{01})^2+4a_{10}b_{01}.
\end{equation}
Furthermore, it can be seen that,
\begin{equation}\label{29}
\begin{array}{l}
	\displaystyle DX_1(0,0)=\left(\begin{array}{cc} -a_{10} & 0 \vspace{0.2cm} \\ 0 & -a_{10} \end{array}\right), \quad DX_2(0,0)=\left(\begin{array}{cc} -b_{01} & 0 \vspace{0.2cm} \\ 0 & -b_{01} \end{array}\right), \vspace{0.2cm} \\
	\displaystyle DX_1(u_0^\pm,0)=\left(\begin{array}{cc} 2a_{10}+(a_{01}-b_{10})u_0^\pm & \star \vspace{0.2cm} \\ 0 & -(a_{01}u_0^\pm+a_{10}) \end{array}\right).
\end{array}
\end{equation}
In the next two propositions we prove that $(u_0^+,0)$ and $(u_0^-,0)$ are non-hyperbolic if, and only if, it collapses with some other singularity.

\begin{proposition}\label{Theo9}
	Let $u_0\in\{u_0^+,u_0^-\}$. If $u_0\in\mathbb{R}$, then the following statements are equivalent.
	\begin{enumerate}[label=(\alph*)]
		\item $2a_{10}+(a_{01}-b_{10})u_0=0$; 
		\item $u_0^+=u_0^-$ or $u_0=0$.
	\end{enumerate}
\end{proposition}

\begin{proof} Assume statement $(a)$. If $a_{01}-b_{10}=0$, then $a_{10}=0$. Hence, it follows from \eqref{28} that $u_0=0$. If $a_{01}-b_{10}\neq0$, then $u_0=2\frac{a_{10}}{b_{10}-a_{01}}$. Therefore, $2\frac{a_{10}}{b_{10}-a_{01}}$ is a solution of $g(u)=0$. Thus,
	\[g\left(2\frac{a_{10}}{b_{10}-a_{01}}\right)=0.\]
Hence,
	\[\frac{a_{10}}{(b_{10}-a_{01})^2}\Delta=0.\]
Which in turn implies $a_{10}=0$ or $\Delta=0$. In the former we have $u_0=0$. In the later, $u_0^+=u_0^-$. Conversely, suppose statement $(b)$. If $u_0=0$, then zero is a solution of \eqref{27}. Thus, $a_{10}=0$. Hence, we have statement $(a)$. If $u_0^+=u_0^-$, then $\Delta=0$. Hence, $u_0=\frac{a_{01}-b_{10}}{2b_{01}}$. Therefore,

	\[2a_{10}+(a_{01}-b_{10})u_0=\frac{1}{2b_{01}}\Delta=0.\]
Which in turn implies statement $(a)$. \end{proof}

\begin{proposition}\label{Theo15}
	Let $u_0\in\{u_0^+,u_0^-\}$. If $u_0\in\mathbb{R}$, then the following statements are equivalent.
	\begin{enumerate}[label=(\alph*)]
		\item $a_{01}u_0+a_{10}=0$; 
		\item $u_0=0$.
	\end{enumerate}
\end{proposition}

\begin{proof} Suppose statement $(a)$. If $a_{01}=0$, then $a_{10}=0$. Therefore, $u_0=0$ is a solution of \eqref{27}. Thus, we have statement $(b)$. If $a_{01}\neq0$, then $u_0=-\frac{a_{10}}{a_{01}}$. Hence,
	\[g\left(-\frac{a_{10}}{a_{01}}\right)=\frac{a_{10}}{a_{01}^2}\det A=0.\]
Since $\det A\neq0$, it follows that $a_{10}=0$. Hence, we have statement $(b)$. Conversely, suppose $u_0=0$. It follows from \eqref{27} that $a_{10}=0$. Therefore, we have statement $(a)$. \end{proof}

\section{The generic phase portraits}\label{Sec9} 

Let $X\in\Sigma_0$ and consider its equivalent system of differential equations,
	\[\dot x=(x+\alpha)(x+\alpha-1)(a_{10}x+a_{01}y), \quad \dot y=(y+\beta)(y+\beta-1)(b_{10}x+b_{01}y).\]
From now on we assume that the origin lies in the center of the octothorpe. Hence, we assume $\alpha$, $\beta\in(0,1)$. From Theorems \ref{Main1} and \ref{Main2} we assume $a_{10}>0$ and taking $\tau=a_{10}t$, we study system
\begin{equation}\label{50}
	\dot x=(x+\alpha)(x+\alpha-1)(x+a_{01}y), \quad \dot y=(y+\beta)(y+\beta-1)(b_{10}x+b_{01}y).
\end{equation}
In the next four subsections we approach system \eqref{50} under each of the four set of hypothesis given by Table~\ref{Table8}. Namely, we assume 
	\[0<\alpha<1, \quad \frac{1}{2}\leqslant\beta<1, \quad a_{01}\geqslant0,\]
and one of the for conditions conditions bellow.
\begin{enumerate}[label=\arabic*)]
	\item $b_{10}\geqslant0$, $b_{01}>0$;
	\item $b_{10}\geqslant0$, $b_{01}<0$;
	\item $b_{10}<0$, $b_{01}>0$;
	\item $b_{10}<0$, $b_{01}<0$.		
\end{enumerate}
Following the nomenclature defined at the proof of Proposition~\ref{Theo8}, we remember that $q_1>p_1$ denotes that $q_1$ is on the right-hand side of $p_1$. Furthermore, $q_2<p_3$ denotes that $q_2$ lies bellow $p_2$. Similarly, $q_3<p_3$ denotes $q_3$ on the left-hand side of $p_3$, while $q_4>p_1$ denote $q_4$ above $p_1$. 
	
\subsection{Case \boldmath{$1$}: \boldmath{$b_{10}\geqslant0$} and \boldmath{$b_{01}>0$}}

In this case we have
	\[q_1>p_1, \quad  q_2<p_3, \quad q_3<p_3, \quad q_4>p_1.\]
There exits sixteen possible relative positions between singularities \eqref{20} and \eqref{23}. Fourteen of which are realizable.
\begin{multicols}{2}
	\begin{enumerate}[label=\arabic*)]
		\item $q_1>p_2$, $q_2<p_2$, $q_3<p_4$ and $q_4>p_4$; 
		\item $q_1<p_2$, $q_2<p_2$, $q_3<p_4$ and $q_4>p_4$; 
		\item $q_1>p_2$, $q_2<p_2$, $q_3>p_4$ and $q_4>p_4$; 
		\item $q_1<p_2$, $q_2<p_2$, $q_3>p_4$ and $q_4>p_4$; 
		\item $q_1>p_2$, $q_2>p_2$, $q_3<p_4$ and $q_4>p_4$; 
		\item $q_1>p_2$, $q_2>p_2$, $q_3>p_4$ and $q_4>p_4$; 
		\item $q_1<p_2$, $q_2>p_2$, $q_3>p_4$ and $q_4>p_4$; 	
	\end{enumerate}
	\columnbreak
	\begin{enumerate}[label=\arabic*)]
		\setcounter{enumi}{7}
		\item $q_1>p_2$, $q_2<p_2$, $q_3<p_4$ and $q_4<p_4$; 
		\item $q_1<p_2$, $q_2<p_2$, $q_3<p_4$ and $q_4<p_4$;
		\item $q_1<p_2$, $q_2<p_2$, $q_3>p_4$ and $q_4<p_4$; 
		\item $q_1>p_2$, $q_2>p_2$, $q_3<p_4$ and $q_4<p_4$;
		\item $q_1<p_2$, $q_2>p_2$, $q_3<p_4$ and $q_4<p_4$;
		\item $q_1>p_2$, $q_2>p_2$, $q_3>p_4$ and $q_4<p_4$;
		\item $q_1<p_2$, $q_2>p_2$, $q_3>p_4$ and $q_4<p_4$.
	\end{enumerate}
\end{multicols}

\begin{theorem}\label{Theo10}
	Let $X=(P,Q)$ be the planar vector field given by,
		\[P(x,y)=(x+\alpha)(x+\alpha-1)(x+a_{01}y), \quad Q(x,y)=(y+\beta)(y+\beta-1)(b_{10}x+b_{01}y).\]
	If $X\in\Sigma_0$ and $\frac{1}{2}\leqslant\beta<1$, $0<\alpha<1$, $a_{01}\geqslant0$, $b_{10}\geqslant0$, $b_{01}>0$,	then its phase portrait in the Poincar\'e disk is topologically equivalent to one of the phase portraits given by Figure~\ref{Case1.1}.
\begin{figure}[h]
	\begin{center}
		\begin{minipage}{3.1cm}
			\begin{center}
				\begin{overpic}[height=3cm]{Case1.1.1.eps} 
				\end{overpic}
				
				Case~$1.1$.
			\end{center}
		\end{minipage}
		\begin{minipage}{3.1cm}
			\begin{center}
				\begin{overpic}[height=3cm]{Case1.1.2.eps} 
				\end{overpic}
				
				Case~$1.2$.
			\end{center}
		\end{minipage}
		\begin{minipage}{3.1cm}
			\begin{center}
				\begin{overpic}[height=3cm]{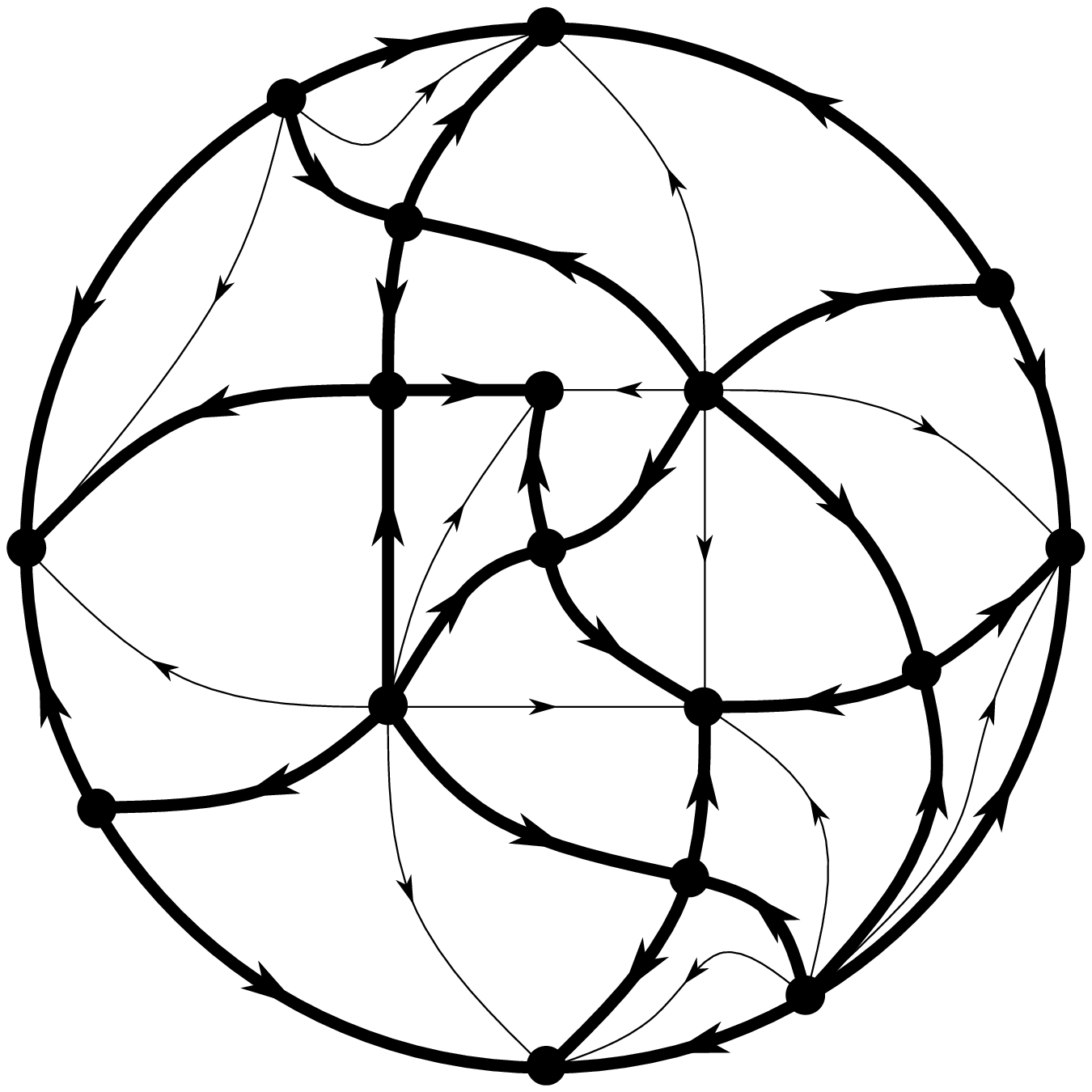} 
				\end{overpic}
				
				Case~$1.3$.
			\end{center}
		\end{minipage}	
		\begin{minipage}{3.1cm}
			\begin{center}
				\begin{overpic}[height=3cm]{Case1.1.4a.eps} 
				\end{overpic}
				
				Case~$1.4a$.
			\end{center}
		\end{minipage}
		\begin{minipage}{3.1cm}
			\begin{center}
				\begin{overpic}[height=3cm]{Case1.1.4b.eps} 
				\end{overpic}
				
				Case~$1.4b$.
			\end{center}
		\end{minipage}
	\end{center}
$\;$
	\begin{center} 
		\begin{minipage}{3.1cm}
			\begin{center}
				\begin{overpic}[height=3cm]{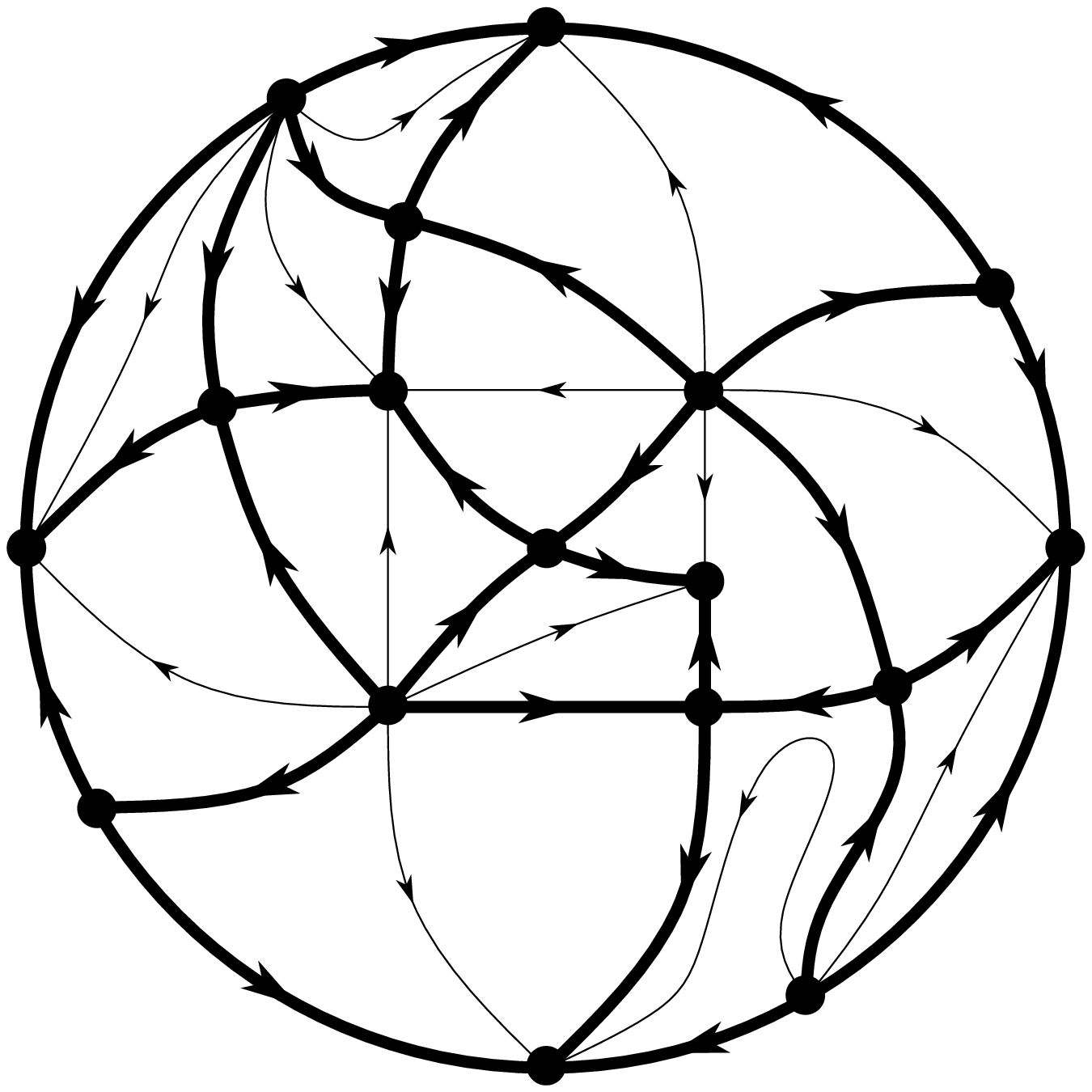} 
				\end{overpic}
				
				Case~$1.5$.
			\end{center}
		\end{minipage}
		\begin{minipage}{3.1cm}
			\begin{center}
				\begin{overpic}[height=3cm]{Case1.1.6a.eps} 
				\end{overpic}
				
				Case~$1.6a$.
			\end{center}
		\end{minipage}
		\begin{minipage}{3.1cm}
			\begin{center}
				\begin{overpic}[height=3cm]{Case1.1.6b.eps} 
				\end{overpic}
				
				Case~$1.6b$.
			\end{center}
		\end{minipage}
		\begin{minipage}{3.1cm}
			\begin{center}
				\begin{overpic}[height=3cm]{Case1.1.7.eps} 
				\end{overpic}
				
				Case~$1.7$.
			\end{center}
		\end{minipage}
		\begin{minipage}{3.1cm}
			\begin{center}
				\begin{overpic}[height=3cm]{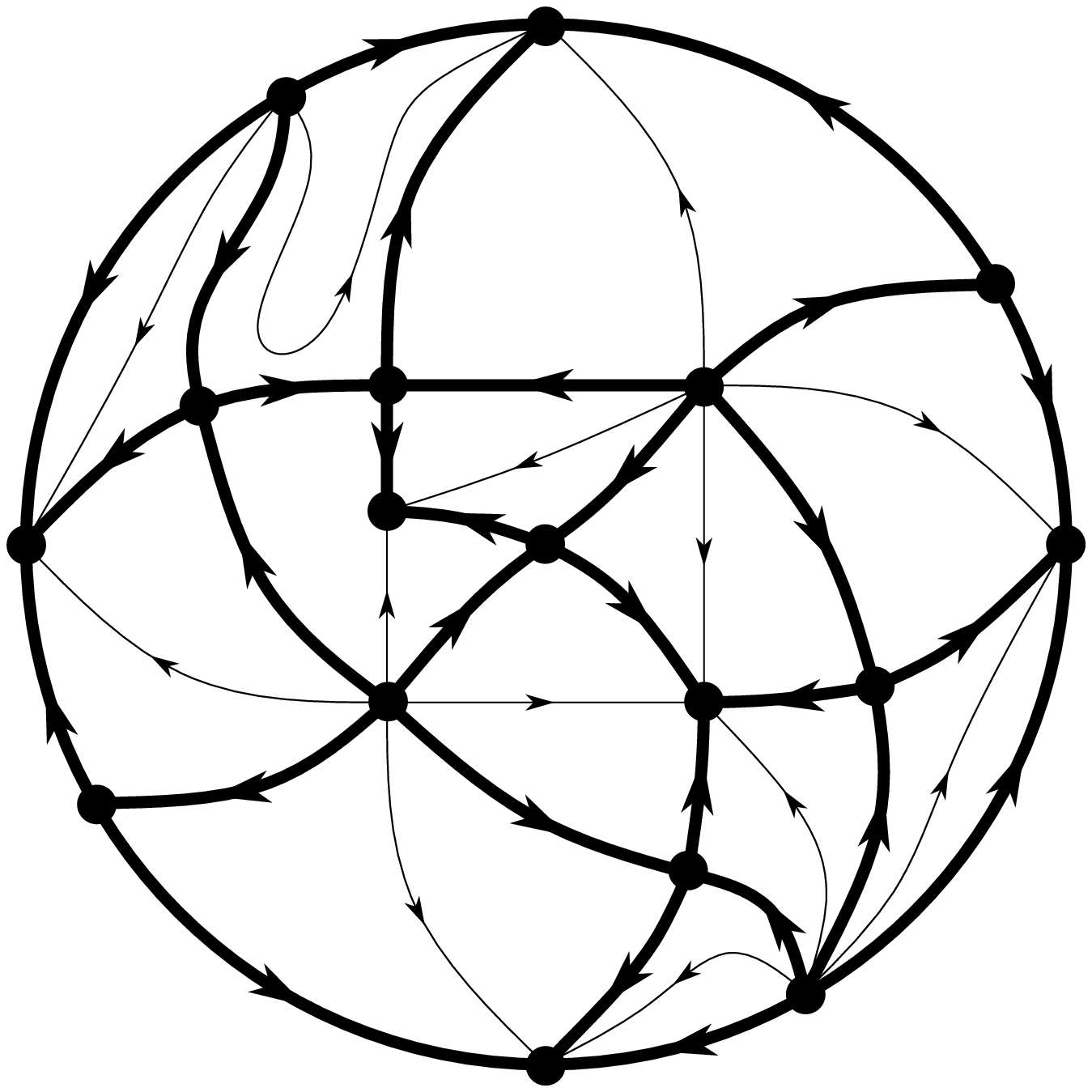} 
				\end{overpic}
				
				Case~$1.8$.
			\end{center}
		\end{minipage}
	\end{center}
$\;$
	\begin{center}
		\begin{minipage}{3.1cm}
			\begin{center}
				\begin{overpic}[height=3cm]{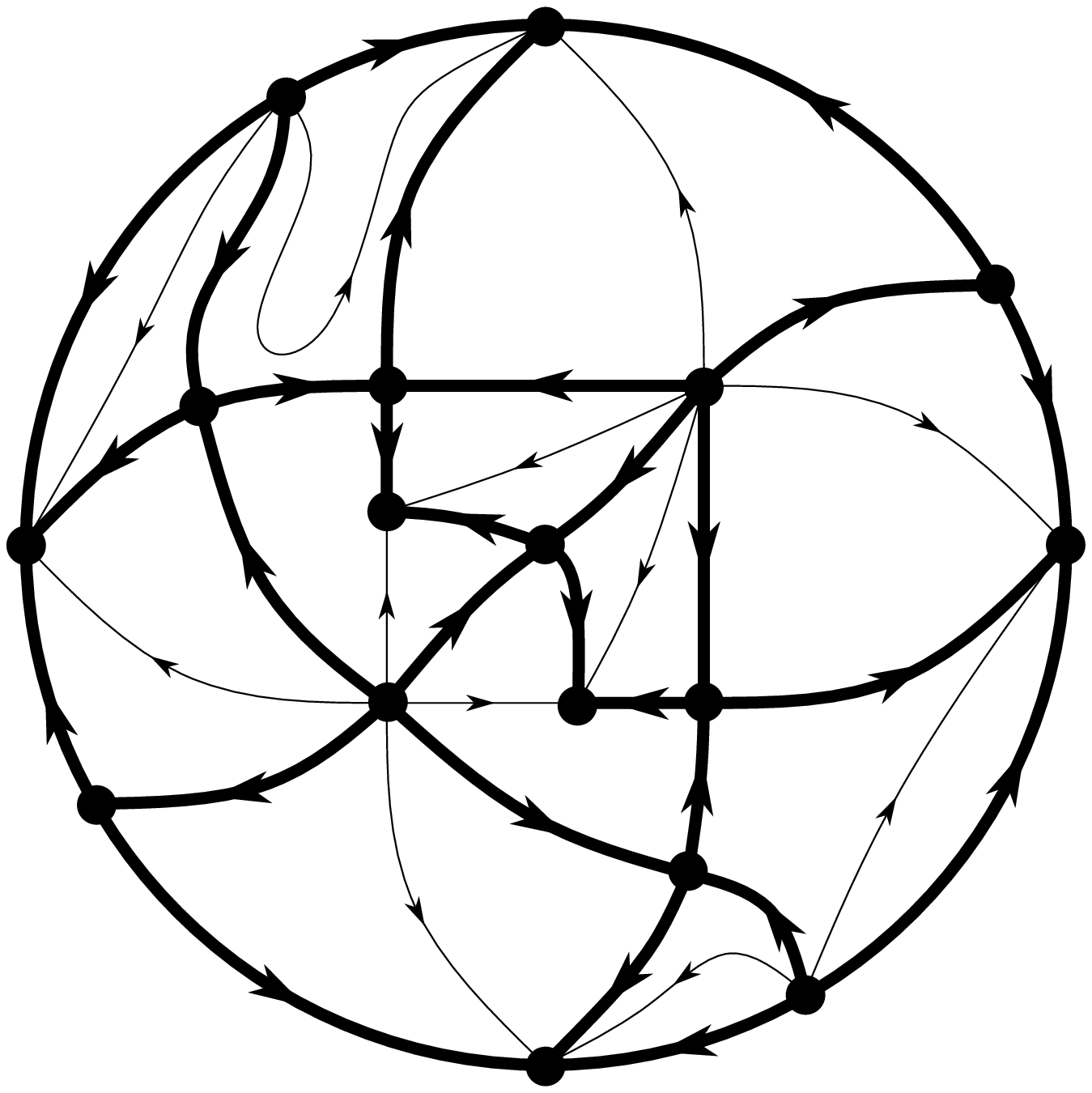} 
				\end{overpic}
				
				Case~$1.9a$.
			\end{center}
		\end{minipage}
		\begin{minipage}{3.1cm}
			\begin{center}
				\begin{overpic}[height=3cm]{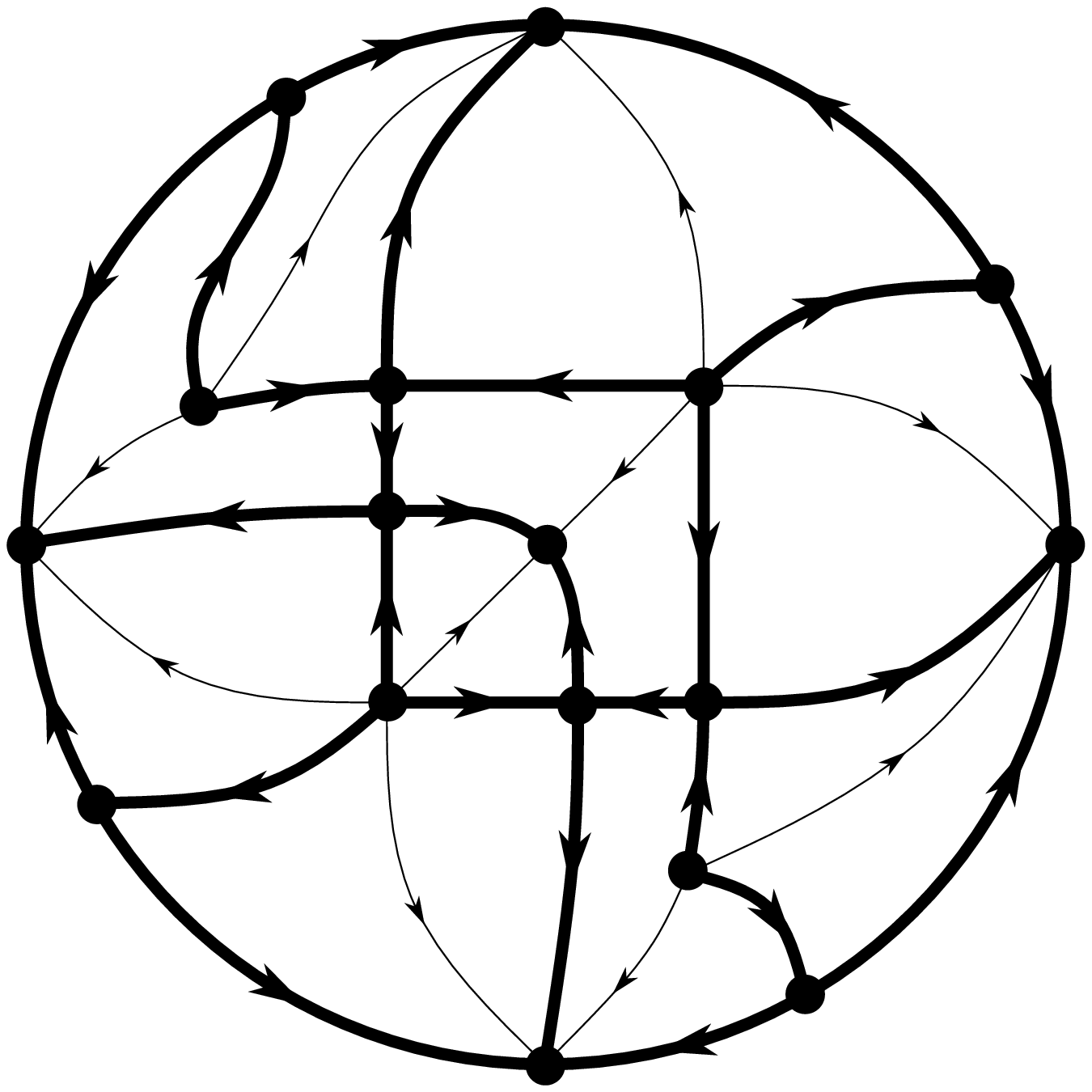} 
				\end{overpic}
				
				Case~$1.9b$.
			\end{center}
		\end{minipage}
		\begin{minipage}{3.1cm}
			\begin{center}
				\begin{overpic}[height=3cm]{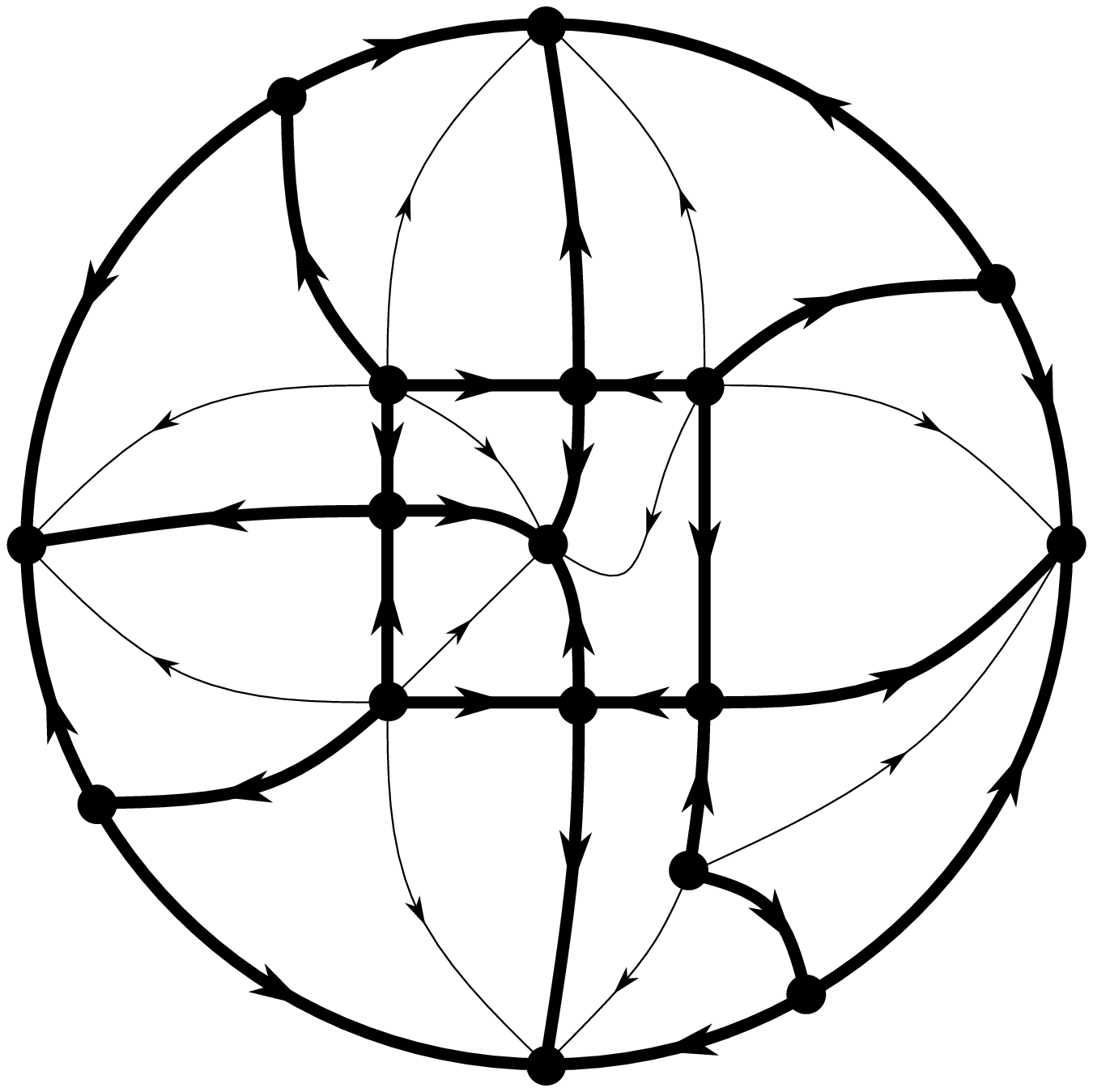} 
				\end{overpic}
				
				Case~$1.10$.
			\end{center}
		\end{minipage}
		\begin{minipage}{3.1cm}
			\begin{center}
				\begin{overpic}[height=3cm]{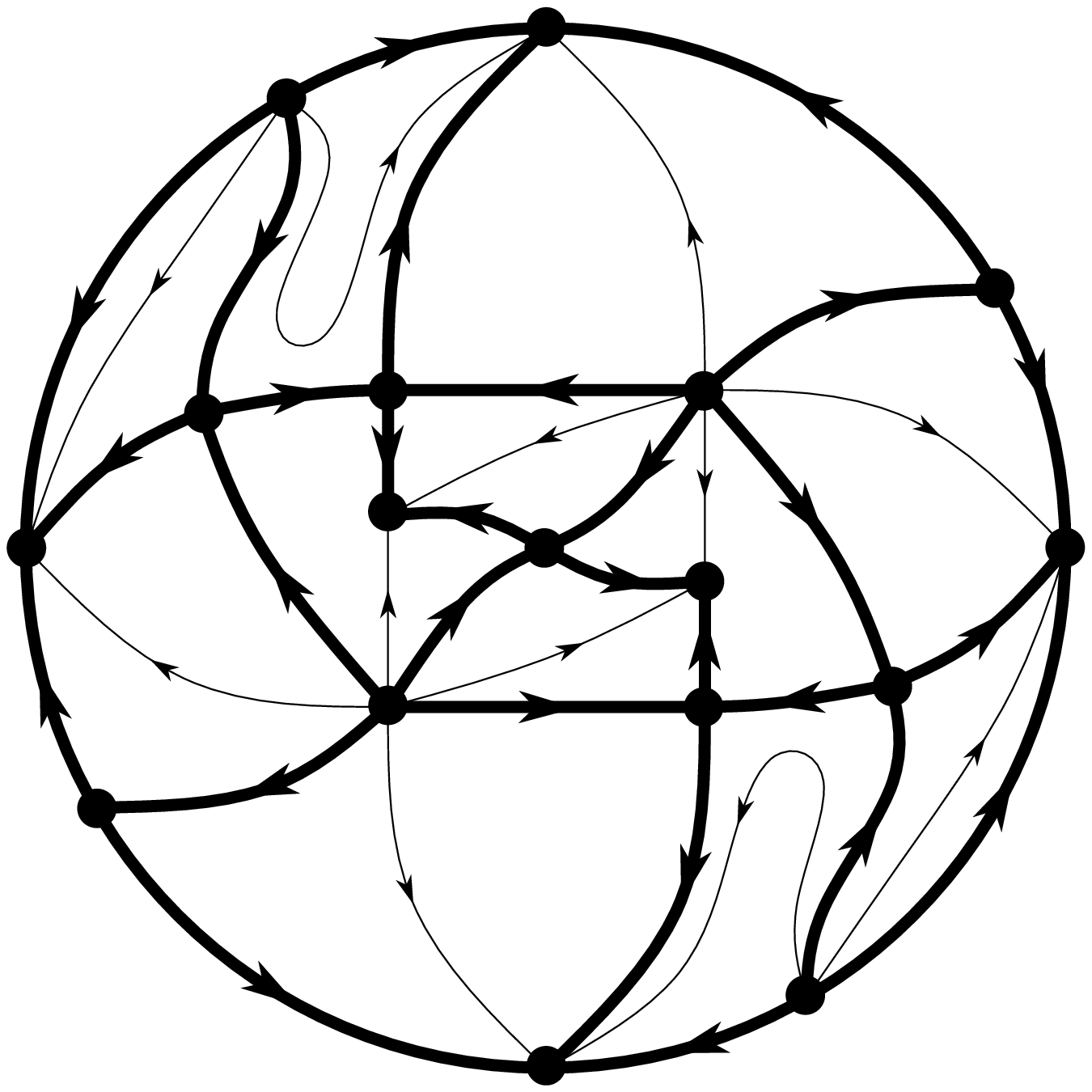} 
				\end{overpic}
				
				Case~$1.11a$.
			\end{center}
		\end{minipage}
		\begin{minipage}{3.1cm}
			\begin{center}
				\begin{overpic}[height=3cm]{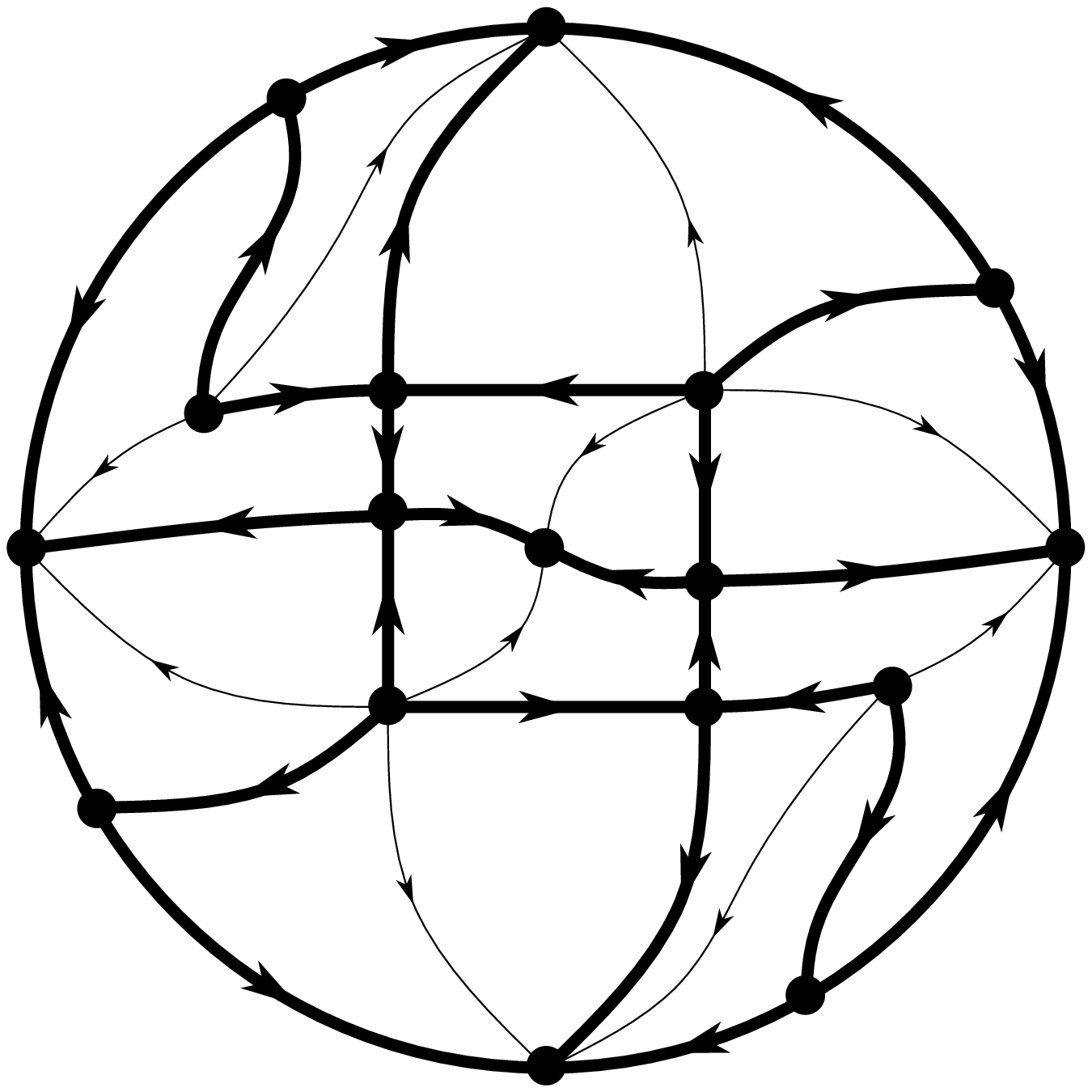} 
				\end{overpic}
				
				Case~$1.11b$.
			\end{center}
		\end{minipage}
	\end{center}
$\;$
	\begin{center}
		\begin{minipage}{3.1cm}
			\begin{center}
				\begin{overpic}[height=3cm]{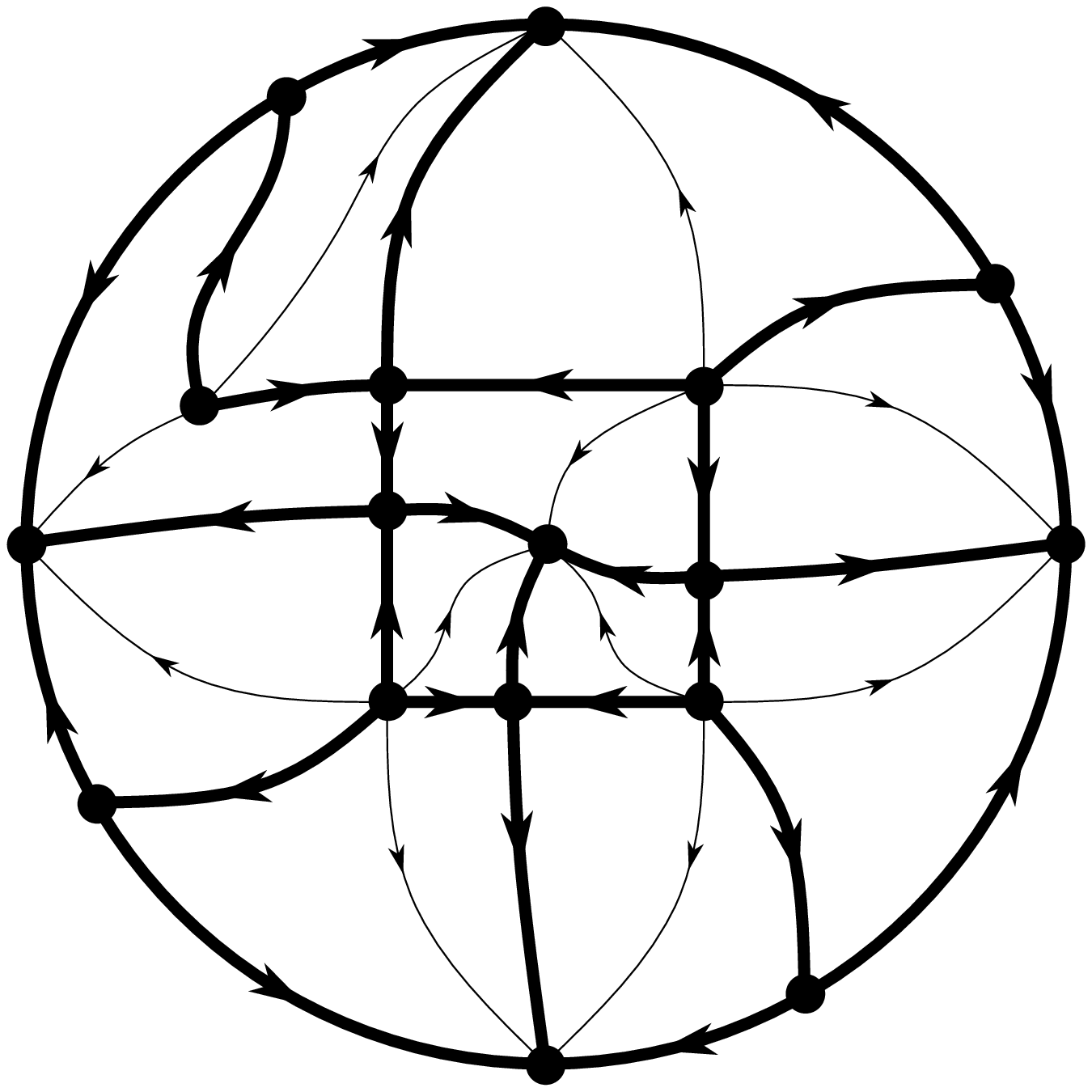} 
				\end{overpic}
				
				Case~$1.12$.
			\end{center}
		\end{minipage}
		\begin{minipage}{3.1cm}
			\begin{center}
				\begin{overpic}[height=3cm]{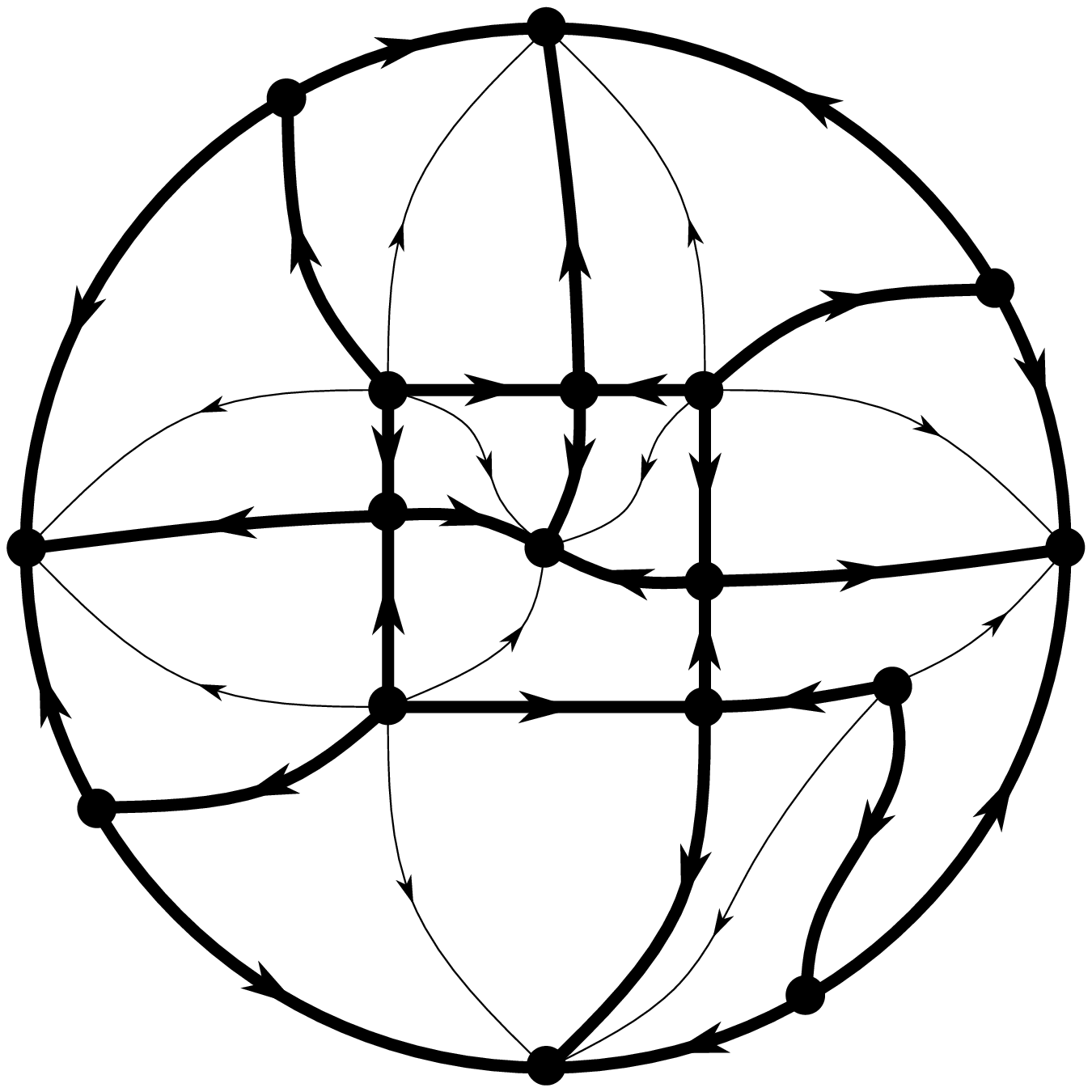} 
				\end{overpic}
				
				Case~$1.13$.
			\end{center}
		\end{minipage}
		\begin{minipage}{3.1cm}
			\begin{center}
				\begin{overpic}[height=3cm]{Case1.1.14.eps} 
				\end{overpic}
				
				Case~$1.14$.
			\end{center}
		\end{minipage}
	\end{center}
	\caption{Phase portraits of the fourteen realizable positions under the hypothesis of Theorem~\ref{Theo10} (Family 1).}\label{Case1.1}
\end{figure}
\end{theorem}

\begin{proof} First, we study the signal of $\det A$. Since $b_{01}>0$ and $1-\beta>0$, we observe that, 
	\[q_4>p_4 \Rightarrow \frac{b_{10}}{b_{01}}\alpha>1-\beta \Rightarrow b_{10}>\frac{1-\beta}{\alpha}b_{01}.\]
Remembering that $f_1(x,y)=a_{10}x+a_{01}y$ (see Section~\ref{Sec7}), we observe that,
	\[\det A=b_{01}-b_{10}a_{01}<b_{01}-\frac{1-\beta}{\alpha}b_{01}a_{01}=\frac{b_{01}}{\alpha}\bigl(\alpha+(\beta-1)a_{01}\bigr)=-\frac{b_{01}}{\alpha}f_1(p_4).\]
We also remember that at Proposition~\ref{Theo8}, it was proved that $q_3<p_4$ if, and only if, $a_{10}f_1(p_4)>0$. Therefore, we conclude that if $q_4>p_4$ and $q_3<p_4$, then $\det A<0$. Similarly, we obtain the following statements.
\begin{multicols}{2}
\begin{enumerate}[label=(\alph*)]
	\item if $q_3<p_4$ and $q_4>p_4$, then $\det A<0$;
	\item if $q_3>p_4$ and $q_4<p_4$, then $\det A>0$; 
\end{enumerate}
	\columnbreak
\begin{enumerate}[label=(\alph*)]
	\setcounter{enumi}{2}
	\item if $q_1>p_2$ and $q_2<p_2$, then $\det A<0$; 
	\item if $q_1<p_2$ and $q_2>p_2$, then $\det A>0$.
\end{enumerate}
\end{multicols}
Thus, we have the determinant of all cases except $4$, $6$, $9$ and $11$. Each of these four cases must be separated in two more cases, given by $\det A<0$ and $\det A>0$. Furthermore, since $0<\alpha<1$, $0<\beta<1$ and $b_{01}>0$, it follows that $(\alpha-1)\alpha+b_{01}(\beta-1)\beta<0$. Therefore, it follows from Proposition~\ref{Theo7} that the origin is either a hyperbolic saddle or a hyperbolic stable node/focus. Hence, we can divide the first generic family in the eighteen cases given by Table~\ref{Table2}.
\begin{table}[h]
	\caption{Table of the realizable positions under hypothesis of Theorem~\ref{Theo10} (Family 1).}\label{Table2}
	\begin{tabular}{c c c c c c}
		\hline
		Cases & $q_1$ & $q_2$ & $q_3$ & $q_4$ & $\det A$ \\
		\hline
		\rowcolor{mygray}
		$1.1$ & $q_1>p_2$ & $q_2<p_2$ & $q_3<p_4$ & $q_4>p_4$ & $\det A<0$ \\
		$1.2$ & $q_1<p_2$ & $q_2<p_2$ & $q_3<p_4$ & $q_4>p_4$ & $\det A<0$ \\
		\rowcolor{mygray}
		$1.3$ & $q_1>p_2$ & $q_2<p_2$ & $q_3>p_4$ & $q_4>p_4$ & $\det A<0$ \\
		$1.4a$ & $q_1<p_2$ & $q_2<p_2$ & $q_3>p_4$ & $q_4>p_4$ & $\det A<0$ \\
		\rowcolor{mygray}
		$1.4b$ & $q_1<p_2$ & $q_2<p_2$ & $q_3>p_4$ & $q_4>p_4$ & $\det A>0$ \\
		$1.5$ & $q_1>p_2$ & $q_2>p_2$ & $q_3<p_4$ & $q_4>p_4$ & $\det A<0$ \\
		\rowcolor{mygray}
		$1.6a$ & $q_1>p_2$ & $q_2>p_2$ & $q_3>p_4$ & $q_4>p_4$ & $\det A<0$ \\
		$1.6b$ & $q_1>p_2$ & $q_2>p_2$ & $q_3>p_4$ & $q_4>p_4$ & $\det A>0$ \\
		\rowcolor{mygray}
		$1.7$ & $q_1<p_2$ & $q_2>p_2$ & $q_3>p_4$ & $q_4>p_4$ & $\det A>0$ \\
		$1.8$ & $q_1>p_2$ & $q_2<p_2$ & $q_3<p_4$ & $q_4<p_4$ & $\det A<0$ \\
		\rowcolor{mygray}
		$1.9a$ & $q_1<p_2$ & $q_2<p_2$ & $q_3<p_4$ & $q_4<p_4$ & $\det A<0$ \\
		$1.9b$ & $q_1<p_2$ & $q_2<p_2$ & $q_3<p_4$ & $q_4<p_4$ & $\det A>0$ \\
		\rowcolor{mygray}
		$1.10$ & $q_1<p_2$ & $q_2<p_2$ & $q_3>p_4$ & $q_4<p_4$ & $\det A>0$ \\
		$1.11a$ & $q_1>p_2$ & $q_2>p_2$ & $q_3<p_4$ & $q_4<p_4$ & $\det A<0$ \\
		\rowcolor{mygray}
		$1.11b$ & $q_1>p_2$ & $q_2>p_2$ & $q_3<p_4$ & $q_4<p_4$ & $\det A>0$ \\
		$1.12$ & $q_1<p_2$ & $q_2>p_2$ & $q_3<p_4$ & $q_4<p_4$ & $\det A>0$ \\
		\rowcolor{mygray}
		$1.13$ & $q_1>p_2$ & $q_2>p_2$ & $q_3>p_4$ & $q_4<p_4$ & $\det A>0$ \\
		$1.14$ & $q_1<p_2$ & $q_2>p_2$ & $q_3>p_4$ & $q_4<p_4$ & $\det A>0$ \\
		\hline
	\end{tabular}
\end{table}
We claim that each case is a class of topologically equivalent vector fields. To prove this, we use Sections~\ref{Sec7} and \ref{Sec8} to obtain the local phase portrait at each singularity. Then, we prove that there is only one way for each separatrix to evolve. As an example, we work on the phase portrait of case $1$. In this case, observe that the disposal between the $p$ and $q$-singularities is given by Figure~\ref{Fig1}. We remember that $u_0^+$ and $u_0^-$ are the zeros of
	\[g(u)=b_{01}u^2+(b_{10}-a_{01})^2-1.\]
Since $b_{01}>0$, its is clear that $\Delta=(b_{10}-a_{01})^2+4b_{01}>0$. Hence, both $u_0^+$ and $u_0^-$ are real. Furthermore, observe that $u_0^-u_0^+=-\frac{1}{b_{01}}<0$. Thus, $u_0^-<0<u_0^+$. Therefore, we can draw the relative position of all the singularities in the Poincar\'e disk. See Figure~\ref{Fig2}(a). Since $b_{01}>0$, it follows from \eqref{29} that all the four poles at infinity are stable nodes, obtaining Figure~\ref{Fig2}(b). Knowing that all singularities are hyperbolic, we can draw all the arrows on all the four invariant lines $\bigl($Figure~\ref{Fig2}(c)$\bigr)$. 
\begin{figure}[h]
	\begin{center}
		\begin{minipage}{5.2cm}
			\begin{center}
				\begin{overpic}[height=4cm]{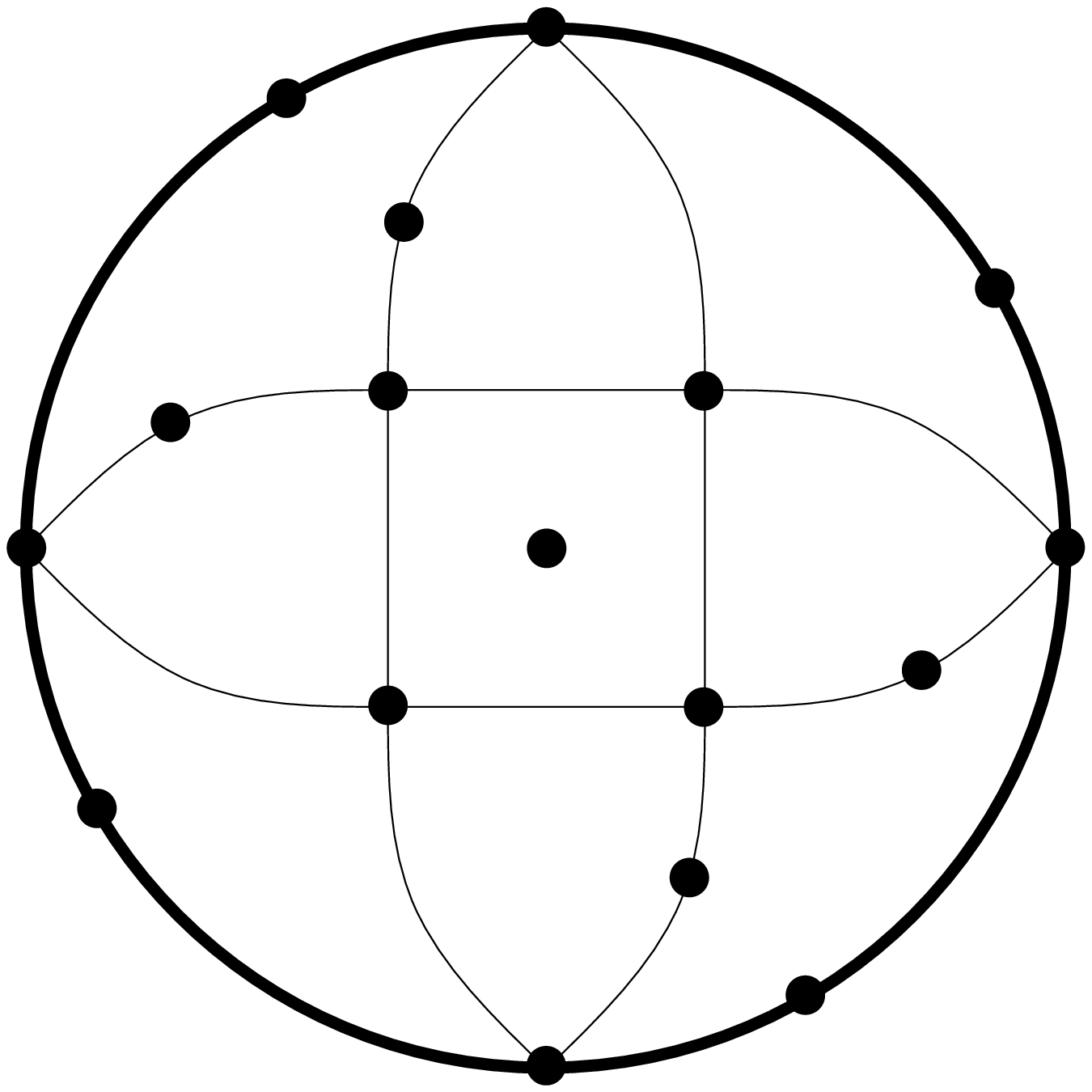} 
					\put(26,29){$p_1$}
					\put(67,29){$p_2$}
					\put(67,67){$p_3$}
					\put(26,67){$p_4$}
					\put(86,33){$q_1$}
					\put(66,16){$q_2$}
					\put(13,66){$q_3$}
					\put(40,79){$q_4$}
					\put(94,74){$u_0^+$}
					\put(77,4){$u_0^-$}
					\put(7,94){$-u_0^-$}
					\put(-13,23){$-u_0^+$}
				\end{overpic}
				
				$(a)$
			\end{center}
		\end{minipage}
		\begin{minipage}{5.2cm}
			\begin{center}
				\begin{overpic}[height=4cm]{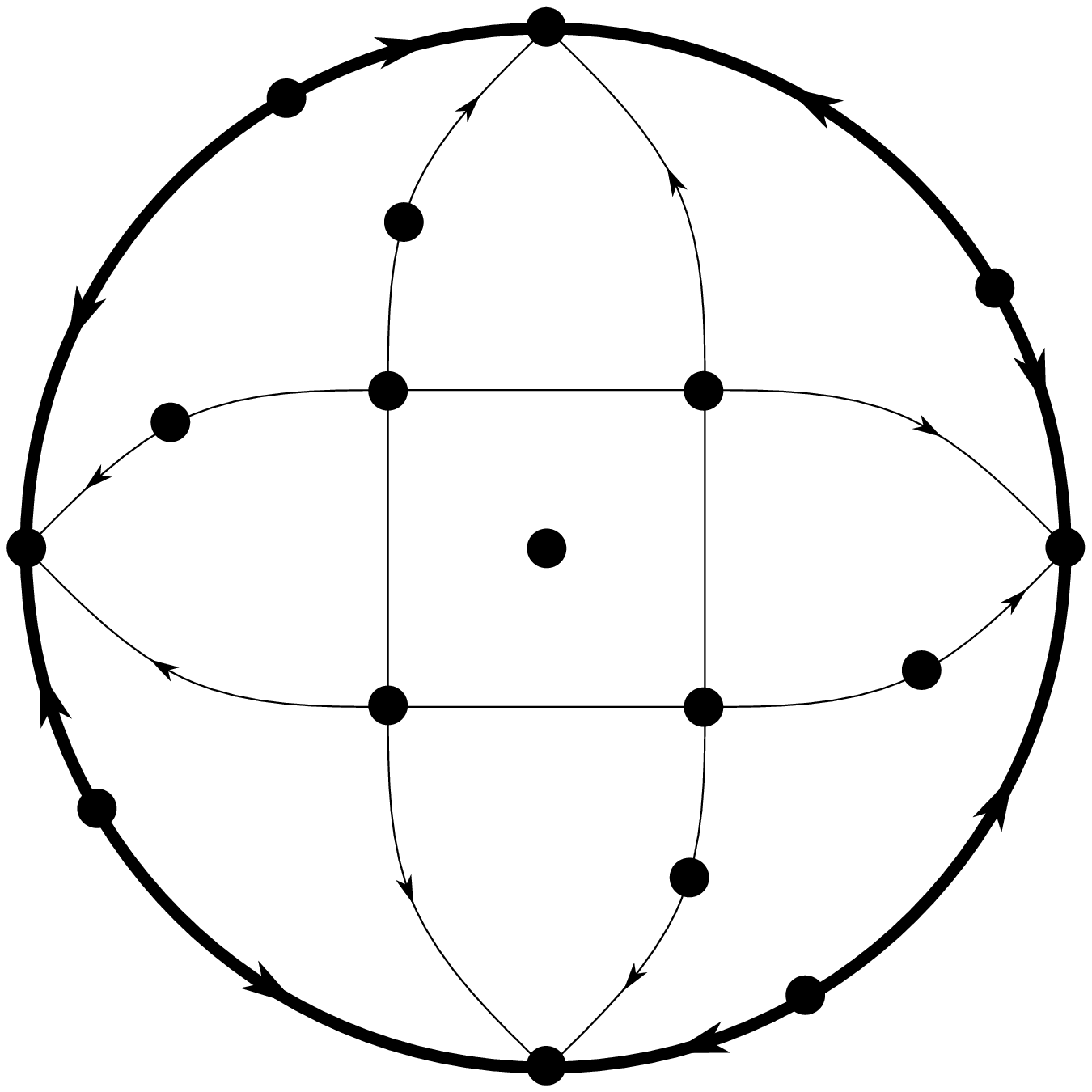} 
					\put(26,29){$p_1$}
					\put(67,29){$p_2$}
					\put(67,67){$p_3$}
					\put(26,67){$p_4$}
					\put(86,33){$q_1$}
					\put(66,16){$q_2$}
					\put(13,66){$q_3$}
					\put(40,79){$q_4$}
					\put(94,74){$u_0^+$}
					\put(77,4){$u_0^-$}
					\put(7,94){$-u_0^-$}
					\put(-13,23){$-u_0^+$}
				\end{overpic}
				
				$(b)$
			\end{center}
		\end{minipage}
		\begin{minipage}{5.2cm}
			\begin{center}
				\begin{overpic}[height=4cm]{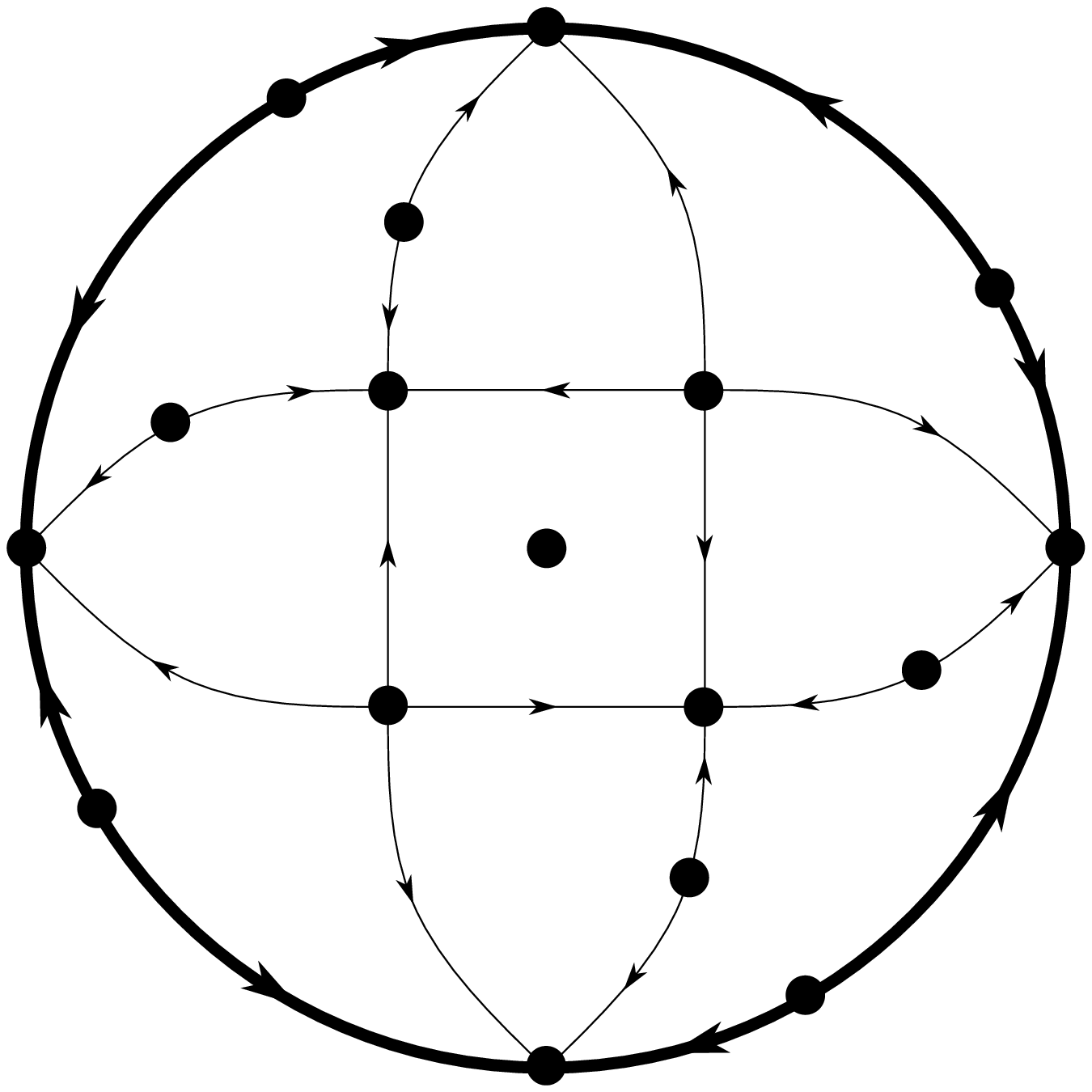} 
					\put(26,29){$p_1$}
					\put(67,29){$p_2$}
					\put(67,67){$p_3$}
					\put(26,67){$p_4$}
					\put(86,33){$q_1$}
					\put(66,16){$q_2$}
					\put(13,66){$q_3$}
					\put(40,79){$q_4$}
					\put(94,74){$u_0^+$}
					\put(77,4){$u_0^-$}
					\put(7,94){$-u_0^-$}
					\put(-13,23){$-u_0^+$}
				\end{overpic}
				
				$(c)$
			\end{center}
		\end{minipage}
	\end{center}
	\caption{Drawing the phase portrait of Case~$1.1$.}\label{Fig2}
\end{figure}
We also know that $\det A<0$. Thus, it follows from Proposition~\ref{Theo7} that the origin is a hyperbolic saddle. Consider the straight lines given by the zeros of
	\[f_1(x,y)=a_{10}x+a_{01}y, \quad f_2(x,y)=b_{10}x+b_{01}y.\]
In special, observe that these straight lines are contained in the solutions of $\dot x=0$ and $\dot y=0$, respectively. Therefore, from the disposal of the $p$ and $q$-singularities and from the local phase portrait at the $p$-singularities, we obtain Figure~\ref{Fig4}(a) and then Figure~\ref{Fig4}(b).
\begin{figure}[h]
	\begin{center}
		\begin{minipage}{7cm}
			\begin{center}
				\begin{overpic}[height=4cm]{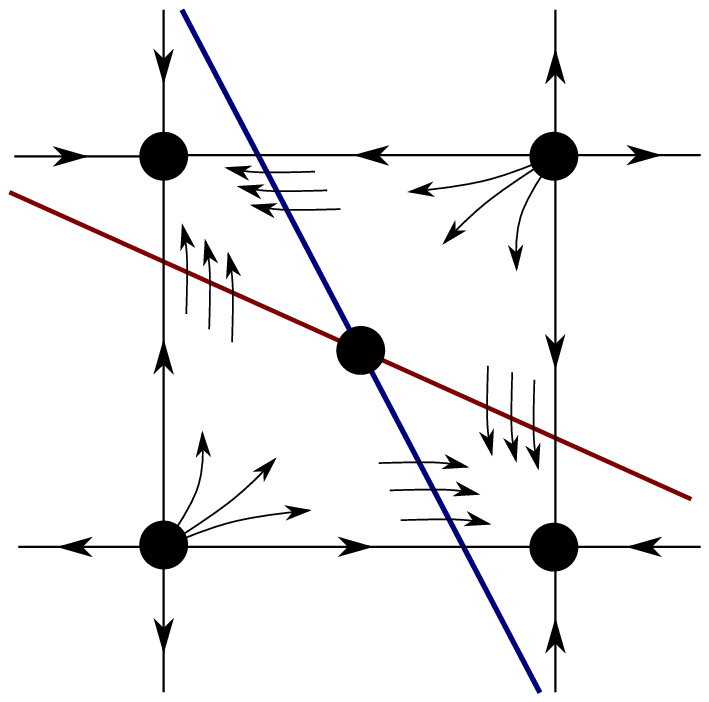} 
					\put(14,14){$p_1$}
					\put(82,15){$p_2$}
					\put(82,83){$p_3$}
					\put(13,83){$p_4$}
					\put(81,39){$f_1(x,y)=0$}
					\put(29,96){$f_2(x,y)=0$}
				\end{overpic}
				
				$(a)$
			\end{center}
		\end{minipage}
		\begin{minipage}{7cm}
			\begin{center}
				\begin{overpic}[height=4cm]{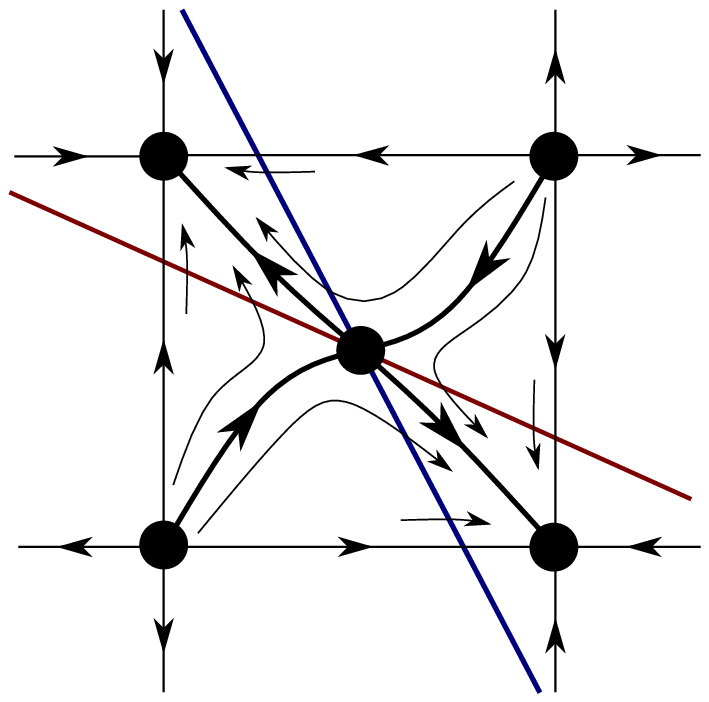} 
					\put(14,14){$p_1$}
					\put(82,15){$p_2$}
					\put(82,83){$p_3$}
					\put(13,83){$p_4$}
					\put(81,39){$f_1(x,y)=0$}
					\put(29,96){$f_2(x,y)=0$}
				\end{overpic}
				
				$(b)$
			\end{center}
		\end{minipage}
	\end{center}
	\caption{The phase portrait at the origin. In red we have $\dot x=0$ and in blue $\dot y=0$.}\label{Fig4}
\end{figure}
We can now complete the local phase portrait at the $q$-singularities. We look at $q_1$ for instance. We need to know its behavior at the $y$-axis. Since $\beta\det A<0$, it follows from \eqref{24} that $q_1$ has a negative eigenvalue relative to the $y$-axis. Hence, it follows from Figure~\ref{Fig2}(c) that $q_1$ is a hyperbolic saddle. Similarly, one can see that all the $q$-singularities are hyperbolic saddles. We now look at the phase portrait at infinity. It follows from $b_{01}>0$ that $g(u)<0$ if, and only if, $u_0^-<u<u_0^+$. Observe that within position $1$ we have $a_{01}\neq0$ (otherwise we would have $q_1<p_2$ and $q_3>p_4$). Therefore, it follows from
	\[g\left(-\frac{1}{a_{01}}\right)=\frac{1}{a_{01}^2}\det A<0,\]
that $u_0^-<-\frac{1}{a_{01}}<u_0^+$. Hence,
	\[a_{01}u_0^- -1<0, \quad a_{01}u_0^+ +1>0.\]
Since $a_{10}=1$, we conclude,
	\[a_{01}u_0^- -a_{10}<0, \quad a_{01}u_0^+ +a_{10}>0.\]
Therefore, it follows from \eqref{29} that $(u_0^-,0)$ is unstable in the $v$-axis, while $(u_0^+,0)$ is stable in the $v$-axis. Hence, it follows from Figure~\ref{Fig2}(c) that $(u_0^-,0)$ is an unstable node and $(u_0^+,0)$ is a saddle. See Figure~\ref{Fig5}(a). 
\begin{figure}[h]
	\begin{center}
		\begin{minipage}{7cm}
			\begin{center}
				\begin{overpic}[height=4cm]{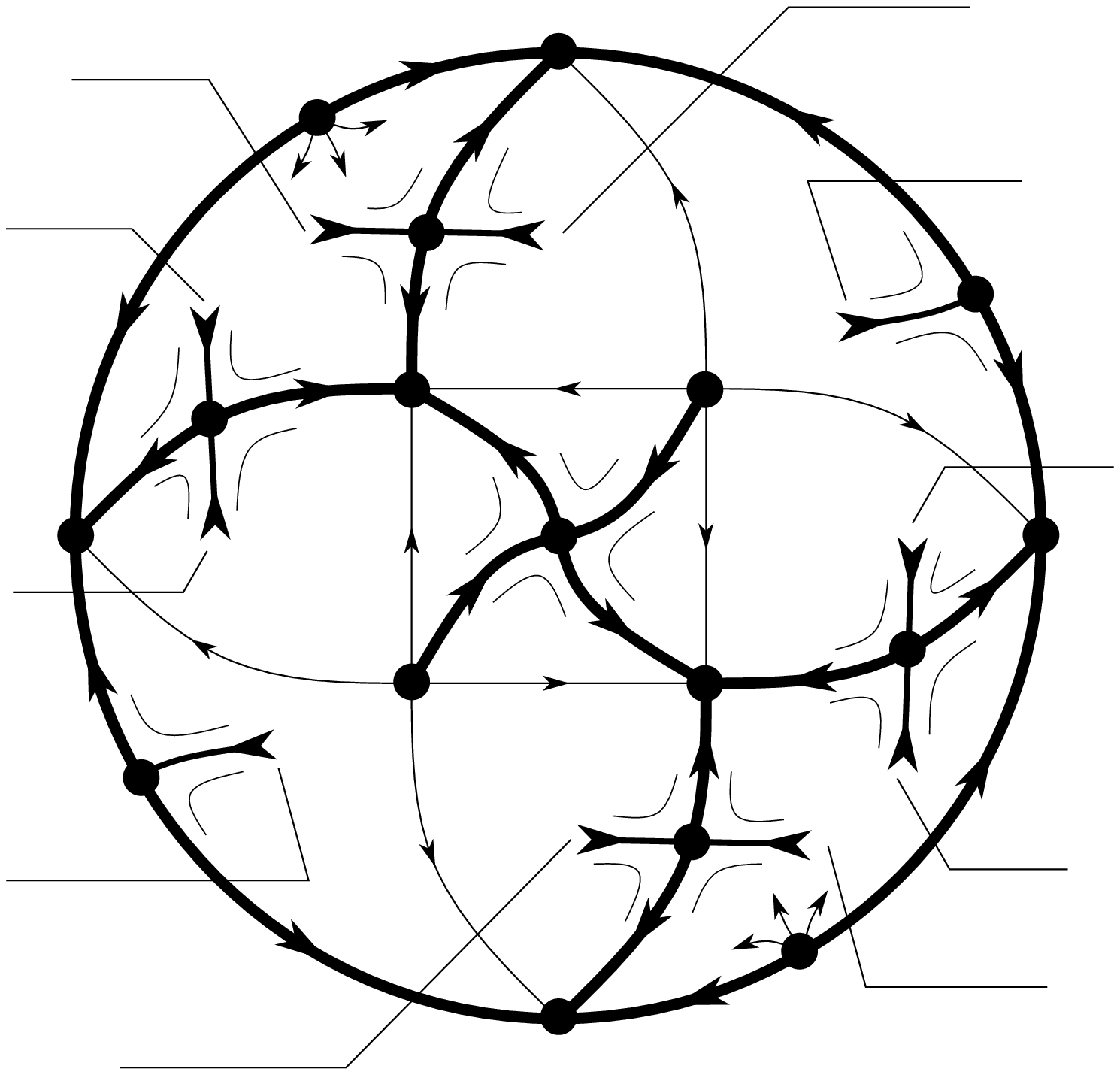} 
					\put(38,29){$p_1$}
					\put(54,64){$p_3$}
					\put(66,1){$u_0^-$}
					\put(18,92){$-u_0^-$}
					\put(95,5){$1$}
					\put(97,16){$2$}
					\put(101,52){$3$}
					\put(92.5,78){$4$}
					\put(88,93){$5$}
					\put(0,86){$6$}
					\put(-6,73){$7$}
					\put(-5,40){$8$}
					\put(-6,15){$9$}
					\put(1,-2.5){$10$}
				\end{overpic}
				
				$(a)$
			\end{center}
		\end{minipage}
		\begin{minipage}{7cm}
			\begin{center}
				\begin{overpic}[height=4cm]{Case1.1.1.eps} 
				\end{overpic}
				
				$(b)$
			\end{center}
		\end{minipage}
	\end{center}
	\caption{Finishing the drawing of the phase portrait of Case~$1.1$.}\label{Fig5}
\end{figure}
So far we have proved that, within case $1$, each vector field has its local phase portrait given by Figure~\ref{Fig5}(a). We now prove that each of those separatrices has only one option to evolve and then conclude that case $1$ is actually a class of topologically equivalent vector fields. Observe that separatrices $1$ and $2$ must born at $u_0^-$. Separatrices $3$, $4$ and $5$ must born at $p_3$. Separatrices $6$ and $7$ must born at $-u_0^-$. Finally, separatrices $8$, $9$ and $10$ has no other option than to born at $p_1$. This completes the phase portrait of Case~$1.1$. See Figure~\ref{Fig5}(b). The other eleven cases are similar. We only point out the reasoning about limit cycles. With the same technique, it can be seen that the phase portrait of case $4b$ is given by Figure~\ref{Fig7}.
\begin{figure}[h]
		\begin{minipage}{4.5cm}
			\begin{center}
				\begin{overpic}[height=4cm]{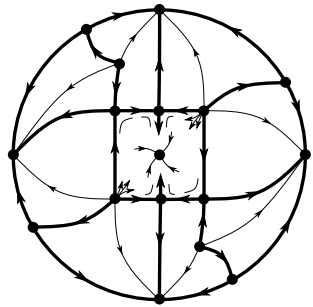} 
				\end{overpic}
			\end{center}
		\end{minipage}
	\caption{Uncompleted phase portrait of Case~$1.4b$.}\label{Fig7}
\end{figure}	
We need only to determine if there is a limit cycle around the origin. Observe that if such limit cycle does exit, then it follows from Theorem~\ref{Theo6} that it is unique and hyperbolic. Hence, it must be unstable because the origin is stable. But since it is unique, it follows that the flow outside such limit cycle and within the center of the octothorpe has no $\omega$-limit. But this is a contradiction with the Poincaré-Bendixson Theory (see Section~$3.7$ of \cite{Perko}). Therefore, the limit cycle does not exist. \end{proof}

\begin{remark}
	About the phase portraits given by Figure~\ref{Case1.1}, the following statements hold.
	\begin{enumerate}[label=(\alph*)]
		\item Case $1.1$ is topologically equivalent to none of the other cases;
		\item Cases $1.2$, $1.3$, $1.5$ and $1.8$ are topologically equivalent;
		\item Cases $1.4a$ and $1.11a$ are topologically equivalent;
		\item Cases $1.4b$ and $1.11b$ are topologically equivalent;
		\item Cases $1.6a$ and $1.9a$ are topologically equivalent;
		\item Cases $1.6b$ and $1.9b$ are topologically equivalent;
		\item Cases $1.7$, $1.10$, $1.12$ and $1.13$ are topologically equivalent;
		\item Case $1.14$ is topologically equivalent to none of the other cases.
	\end{enumerate}
	Therefore, from eighteen phase portraits, only eight remains.
\end{remark}

\begin{corollary}\label{Coro3}
	If $X\in\Sigma_0$ satisfies the hypothesis of Theorem~\ref{Theo10}, then its phase portrait in the Poincar\'e disk is topologically equivalent to one of the eight phase portraits given by Figure~\ref{Case1.1Final}.
	\begin{figure}[h]
		\begin{center}
			\begin{minipage}{3.1cm}
				\begin{center}
					\begin{overpic}[height=3cm]{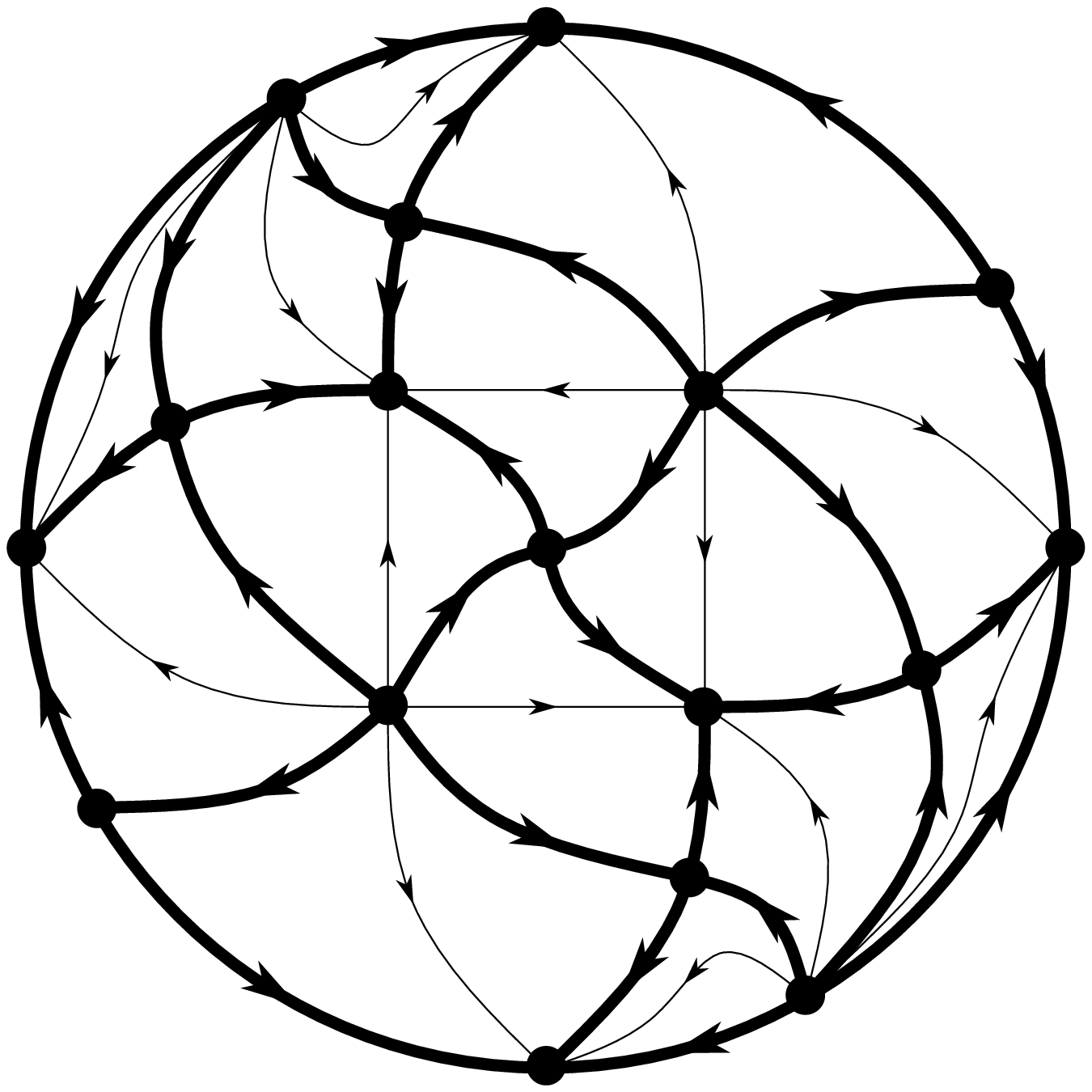} 
					\end{overpic}
					
					Case~$1.1$.
				\end{center}
			\end{minipage}
			\begin{minipage}{3.1cm}
				\begin{center}
					\begin{overpic}[height=3cm]{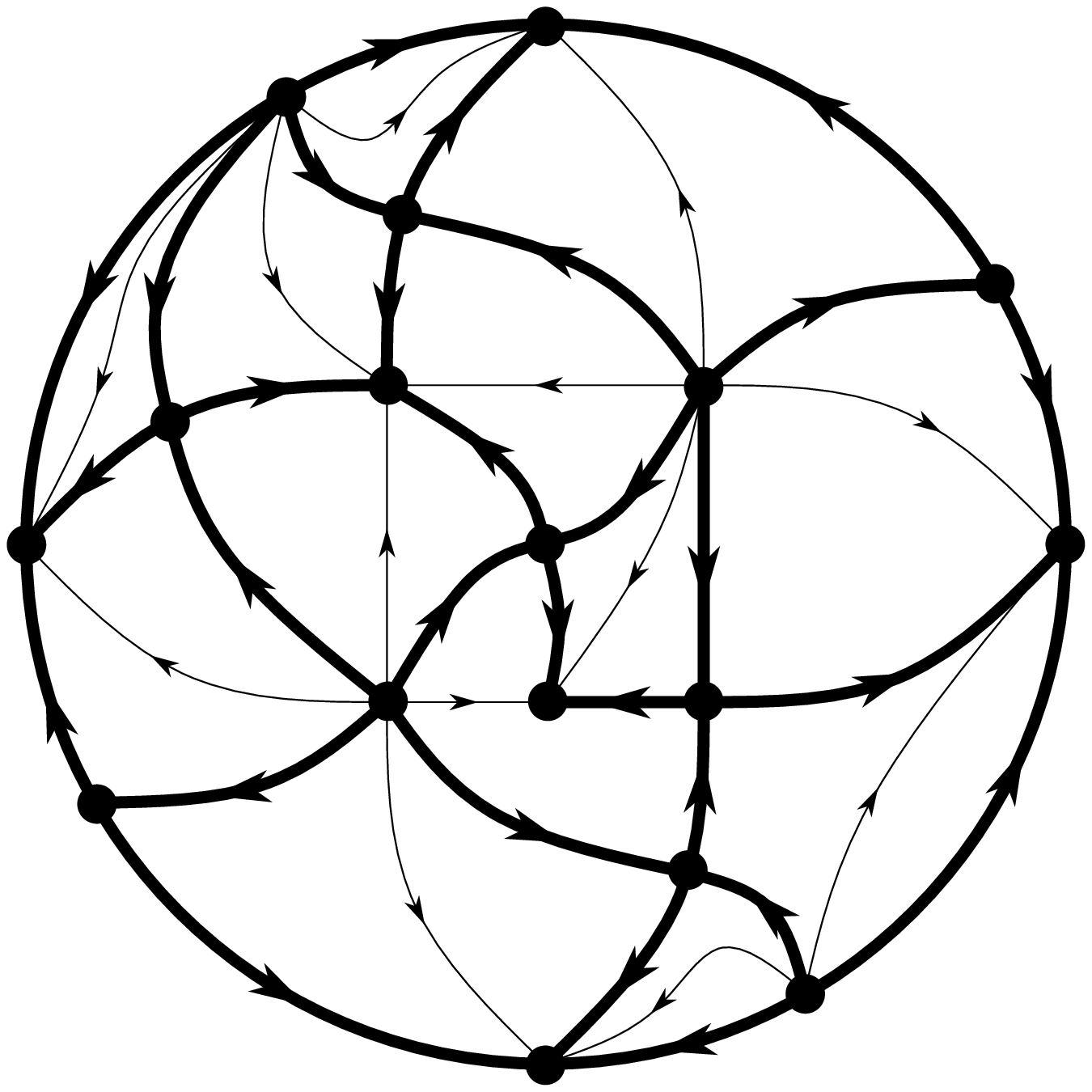} 
					\end{overpic}
					
					Case~$1.2$.
				\end{center}
			\end{minipage}
			\begin{minipage}{3.1cm}
				\begin{center}
					\begin{overpic}[height=3cm]{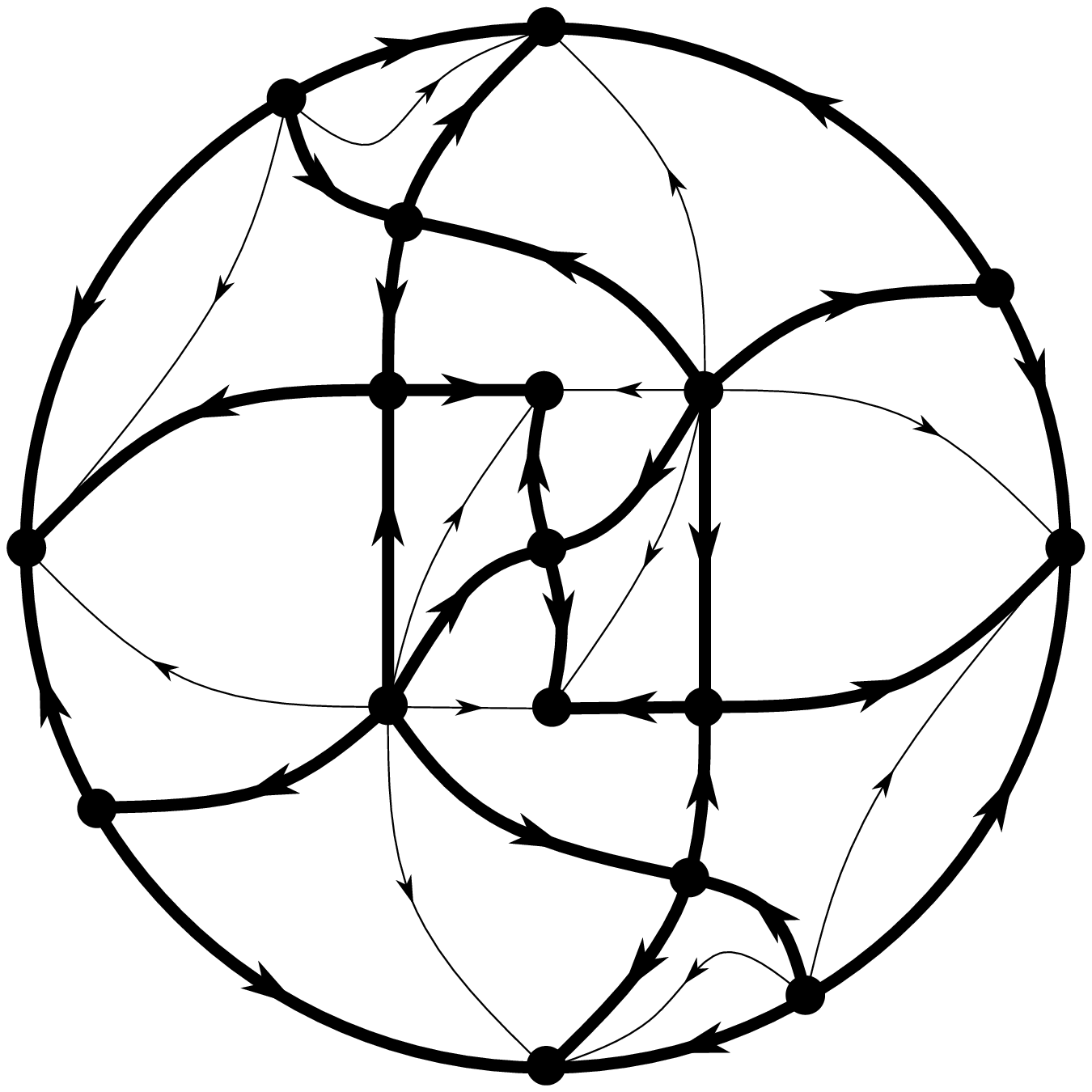} 
					\end{overpic}
					
					Case~$1.4a$.
				\end{center}
			\end{minipage}
			\begin{minipage}{3.1cm}
				\begin{center}
					\begin{overpic}[height=3cm]{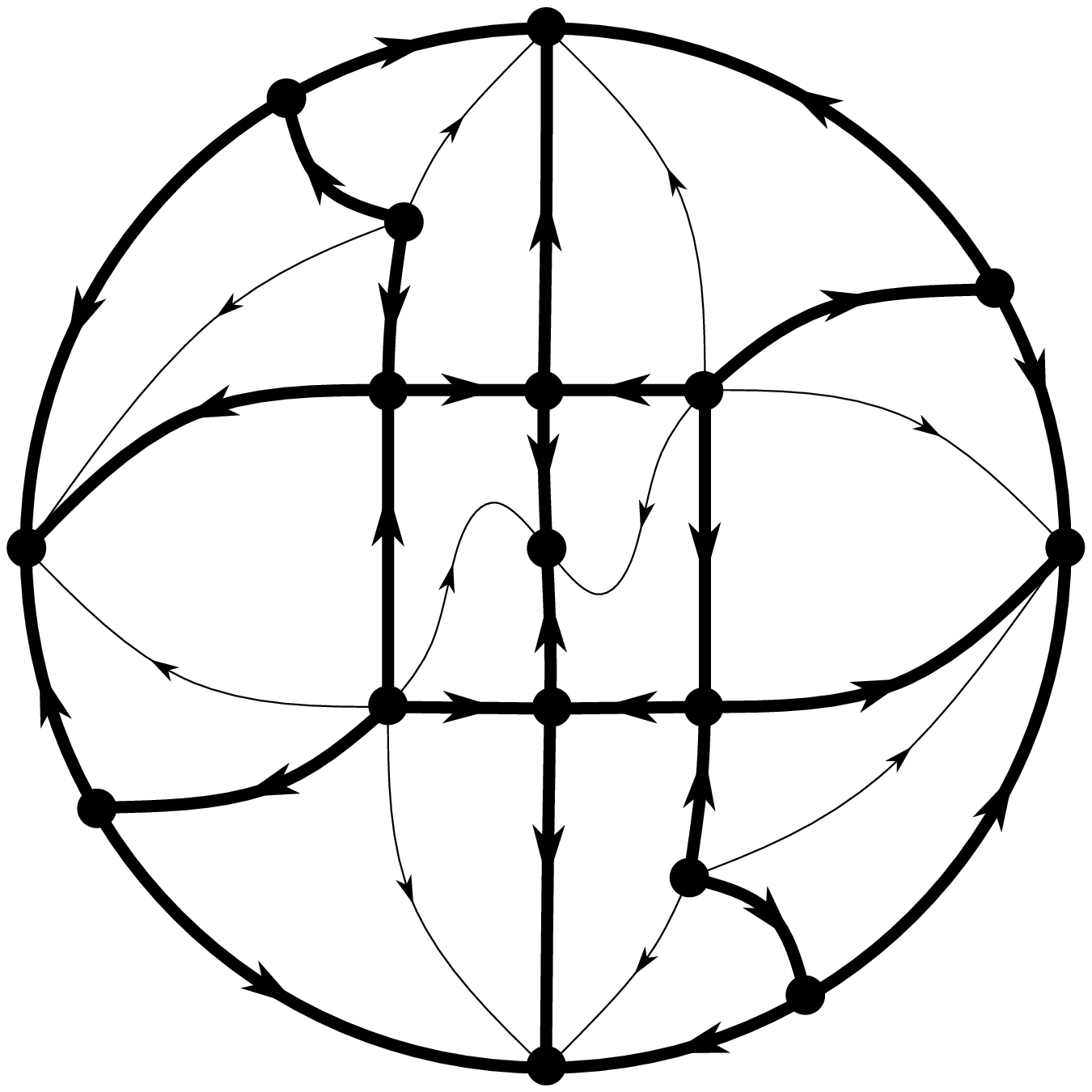} 
					\end{overpic}
					
					Case~$1.4b$.
				\end{center}
			\end{minipage}
		\end{center}
	$\;$
		\begin{center} 
			\begin{minipage}{3.1cm}
				\begin{center}
					\begin{overpic}[height=3cm]{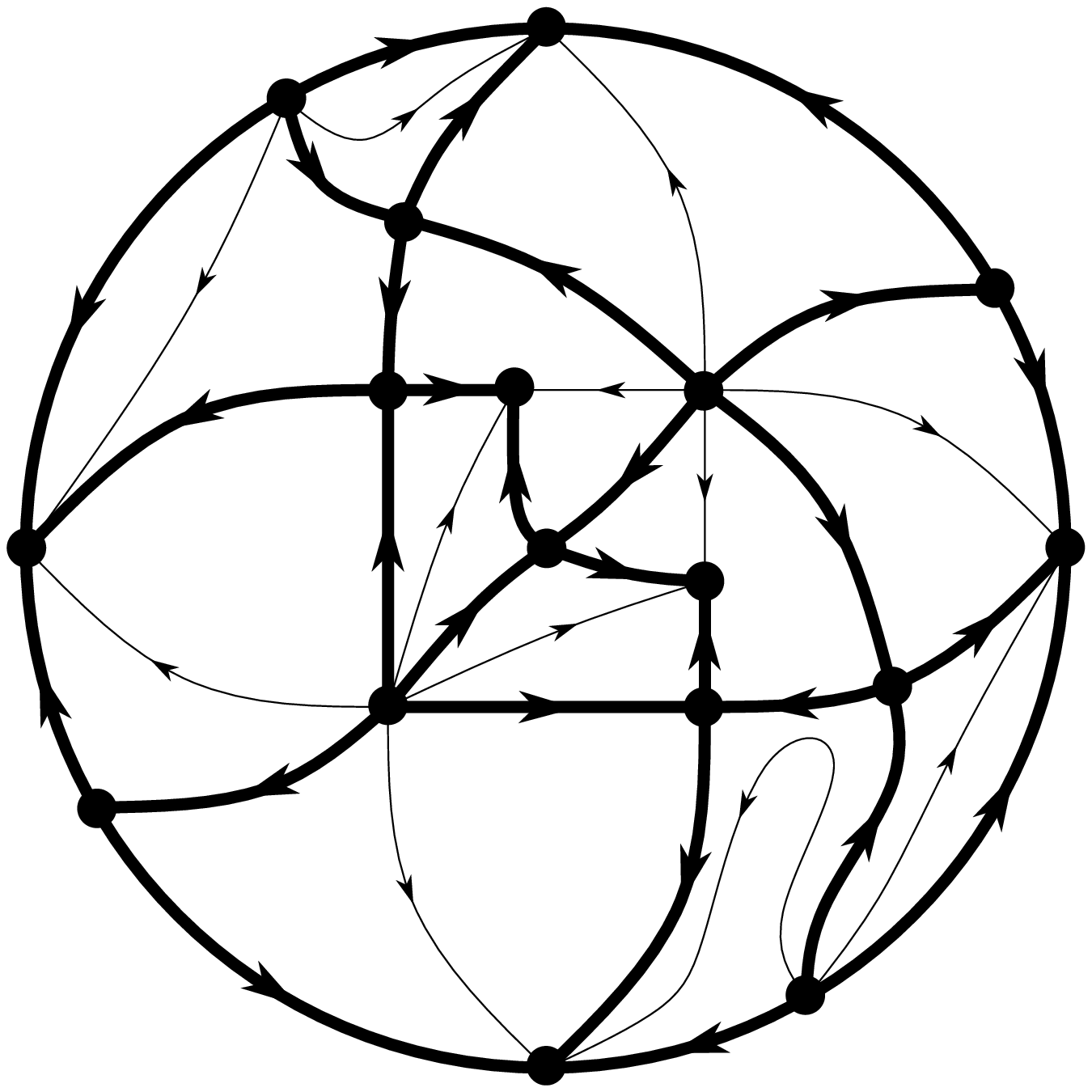} 
					\end{overpic}
					
					Case~$1.6a$.
				\end{center}
			\end{minipage}
			\begin{minipage}{3.1cm}
				\begin{center}
					\begin{overpic}[height=3cm]{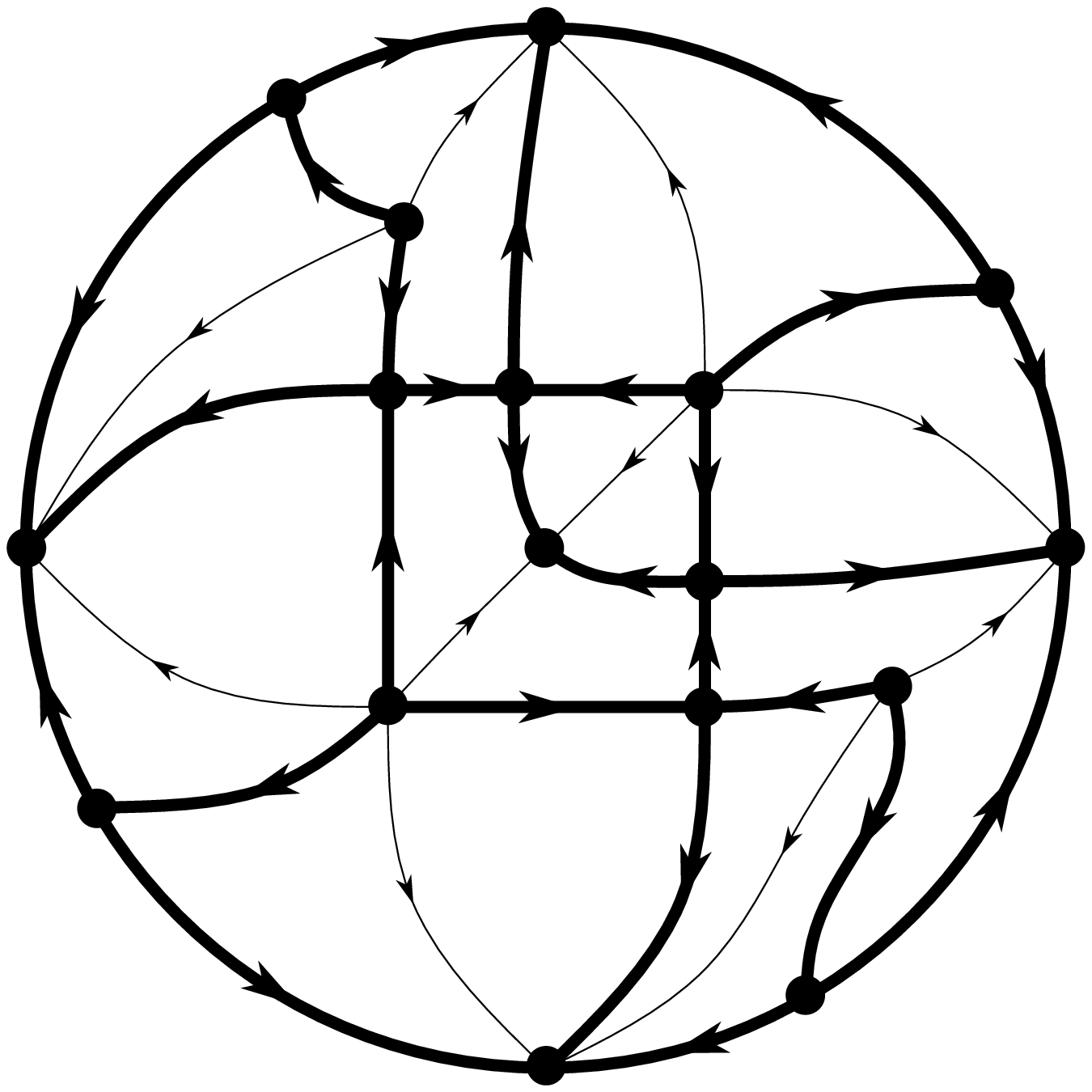} 
					\end{overpic}
					
					Case~$1.6b$.
				\end{center}
			\end{minipage}
			\begin{minipage}{3.1cm}
				\begin{center}
					\begin{overpic}[height=3cm]{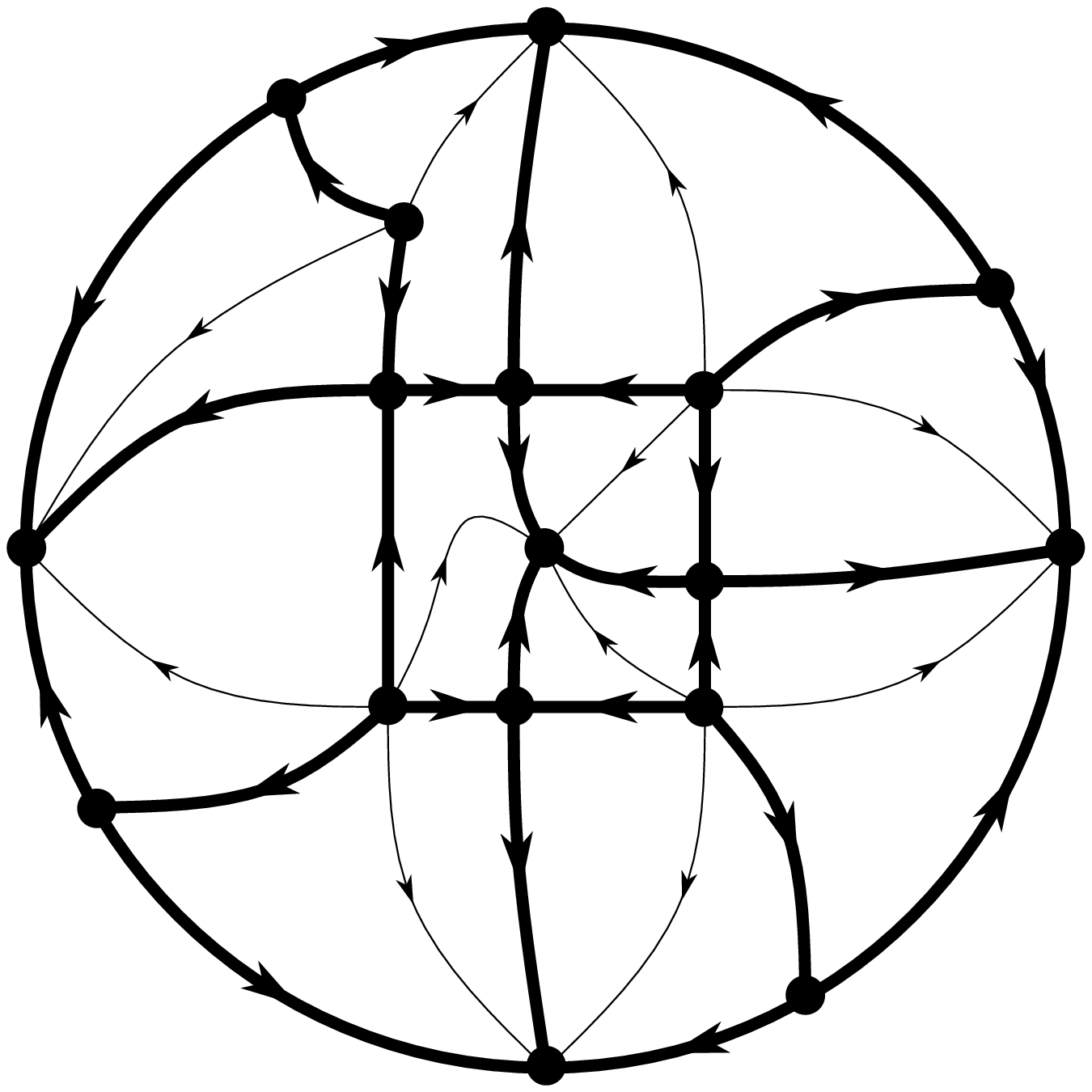} 
					\end{overpic}
					
					Case~$1.7$.
				\end{center}
			\end{minipage}
			\begin{minipage}{3.1cm}
				\begin{center}
					\begin{overpic}[height=3cm]{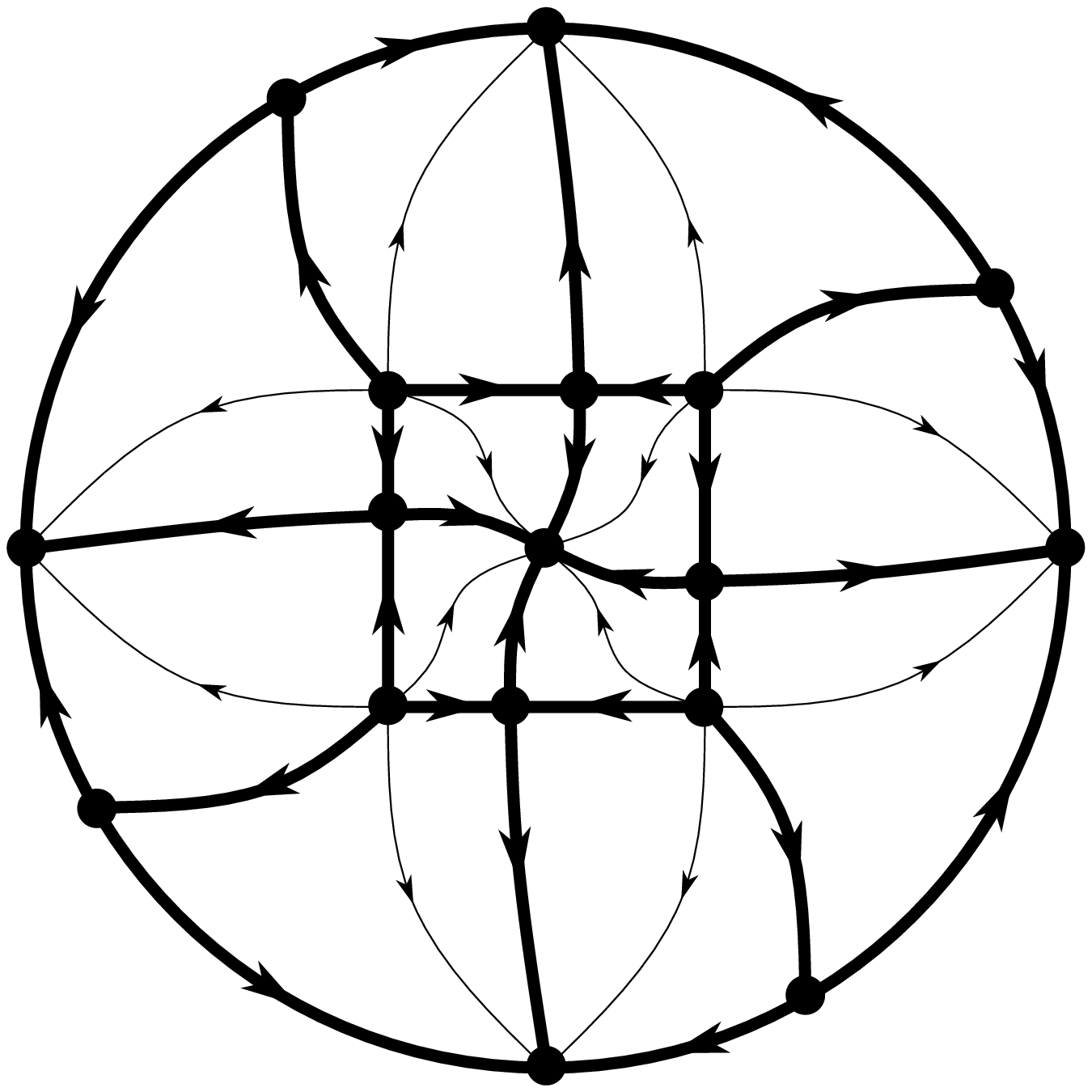} 
					\end{overpic}
					
					Case~$1.14$.
				\end{center}
			\end{minipage}
		\end{center}
		\caption{Topologically distinct phase portraits of Family~$1$.}\label{Case1.1Final}
	\end{figure}
\end{corollary}

\subsection{Case \boldmath{$2$}: \boldmath{$b_{10}\geqslant0$} and \boldmath{$b_{01}<0$}}

\begin{theorem}\label{Theo11}
	Let $X=(P,Q)$ be the planar vector field given by,
		\[P(x,y)=(x+\alpha)(x+\alpha-1)(x+a_{01}y), \quad Q(x,y)=(y+\beta)(y+\beta-1)(b_{10}x+b_{01}y).\]
	If $X\in\Sigma_0$ and $\frac{1}{2}\leqslant\beta<1$, $0<\alpha<1$, $a_{01}\geqslant0$, $b_{10}\geqslant0$, $b_{01}<0$, then its phase portrait in the Poincar\'e disk is topologically equivalent to one of the phase portraits given by Figures~\ref{Case1.2a}, \ref{Case1.2b} and \ref{Case1.2c}.
\end{theorem}

\begin{proof} Observe that $\det A=b_{01}-b_{10}a_{01}<0$. However, in this case the polynomial
	\[g(u)=b_{01}u^2+(b_{10}-a_{01})u-1,\]
given by \eqref{27}, may have complex roots. Thus, we need to study the signal of,
\begin{equation}\label{30}
	\Delta = (b_{10}-a_{01})^2+4b_{01}.
\end{equation}
We claim that for each disposal between the $p$ and $q$-singularities, both $\Delta>0$ and $\Delta<0$ are realizable. To prove this, we observe that $b_{01}\neq0$. Hence, we can consider the parameter,
\begin{equation}\label{35}
	r=-\frac{b_{10}}{b_{01}}\geqslant 0.
\end{equation}
Consider $f_2(x,y)=b_{10}x+b_{01}y$ (see \eqref{21}). Observe that $r$ is such that $f_2(x,y)=0$ if, and only if, $y=rx$. Therefore, if we fix $\alpha$, $\beta$, $a_{01}$ and allow $b_{01}$ and $b_{10}$ to move such that $r$ stays constant, then the relative position between the $p$ and $q$-singularities stays the same. More precisely, given a choice of parameters $(\alpha,\beta,a_{01},b_{10},b_{01})$, we can fix $(\alpha,\beta,a_{01},r)$, with $r\geqslant0$ given by \eqref{35}, taking
\begin{equation}\label{31}
	b_{10}=-rb_{01},
\end{equation}
and let $b_{01}<0$ free. The variation of $b_{01}$ along $(-\infty,0)$ may change the phase portrait (and in fact, it will), but it will not change the relative position between the $p$ and $q$-singularities. Replacing \eqref{31} at \eqref{30} we get,
	\[\Delta=r^2b_{01}^2+(2a_{01}r+4a)b_{01}+a_{01}^2.\]
Hence, we can define the polynomial,
	\[f(x)=r^2x^2+(2a_{01}r+4)x+a_{01}^2.\]
Suppose $r>0$. Observe that
	\[x^\pm=\frac{-(a_{01}r+2)\pm2\sqrt{a_{01}r+1}}{r^2},\]
are the solutions of $f(x)=0$. Since $a_{01}r\geqslant0$, it follows that $x^\pm\in\mathbb{R}$. Therefore, it follows from
	\[x^+x^-=\frac{1}{r^2}>0, \quad x^++x^-=-\frac{2a_{01}r+4}{r^2}<0,\]
that $x^\pm<0$. Since $x^+$ and $x^-$ do not depend on $b_{01}$, it follows that for each position we can choose $b_{01}<0$ such that $\Delta>0$ and $\Delta<0$ are realizable. Thus, each relative position must be divided in two cases. More precisely, for each relative position between the $p$ and the $q$-singularities, we can choose $b_{01}<x^-$ and $x^-<b_{01}<x^+$. Obtaining $\Delta>0$ and $\Delta<0$, respectively. Furthermore, observe that we must divide it again due to the signal of $b_{10}-a_{01}$, i.e. due the fact that $u_0^\pm$ may be both positive or negative. If $r=0$, then we have $q_2<p_3$ and $q_4>q_1$. Therefore, we can take $r>0$ small enough without changing the given relative position. Thus, we can divide the second generic family in the forty five cases given by Table~\ref{Table3} (which also lies in the appendix section). We now study the infinity. If $\Delta<0$, then $(u_0^+,0)$ and $(u_0^-,0)$ are not singularities. Hence, it follows from Section~\ref{Sec8} that the phase portrait at infinity is given by Figure~\ref{Case1.2Infinity}(a). 
\begin{figure}[h]
	\begin{center}
		\begin{minipage}{4.5cm}
			\begin{center}
				\begin{overpic}[height=3cm]{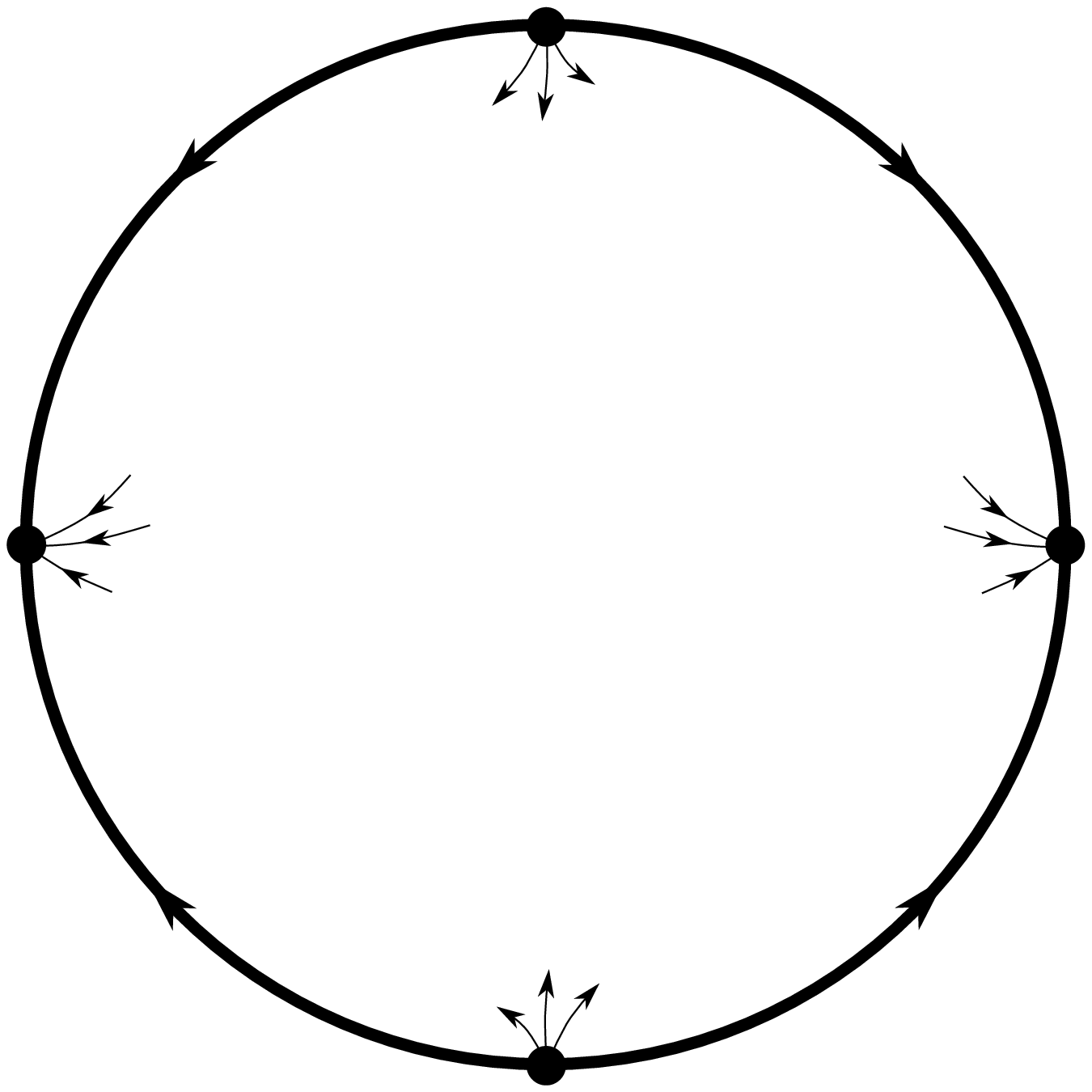} 
				\end{overpic}
				
				$(a)$
				
				$\Delta<0$.
			\end{center}
		\end{minipage}
		\begin{minipage}{4.5cm}
			\begin{center}
				\begin{overpic}[height=3cm]{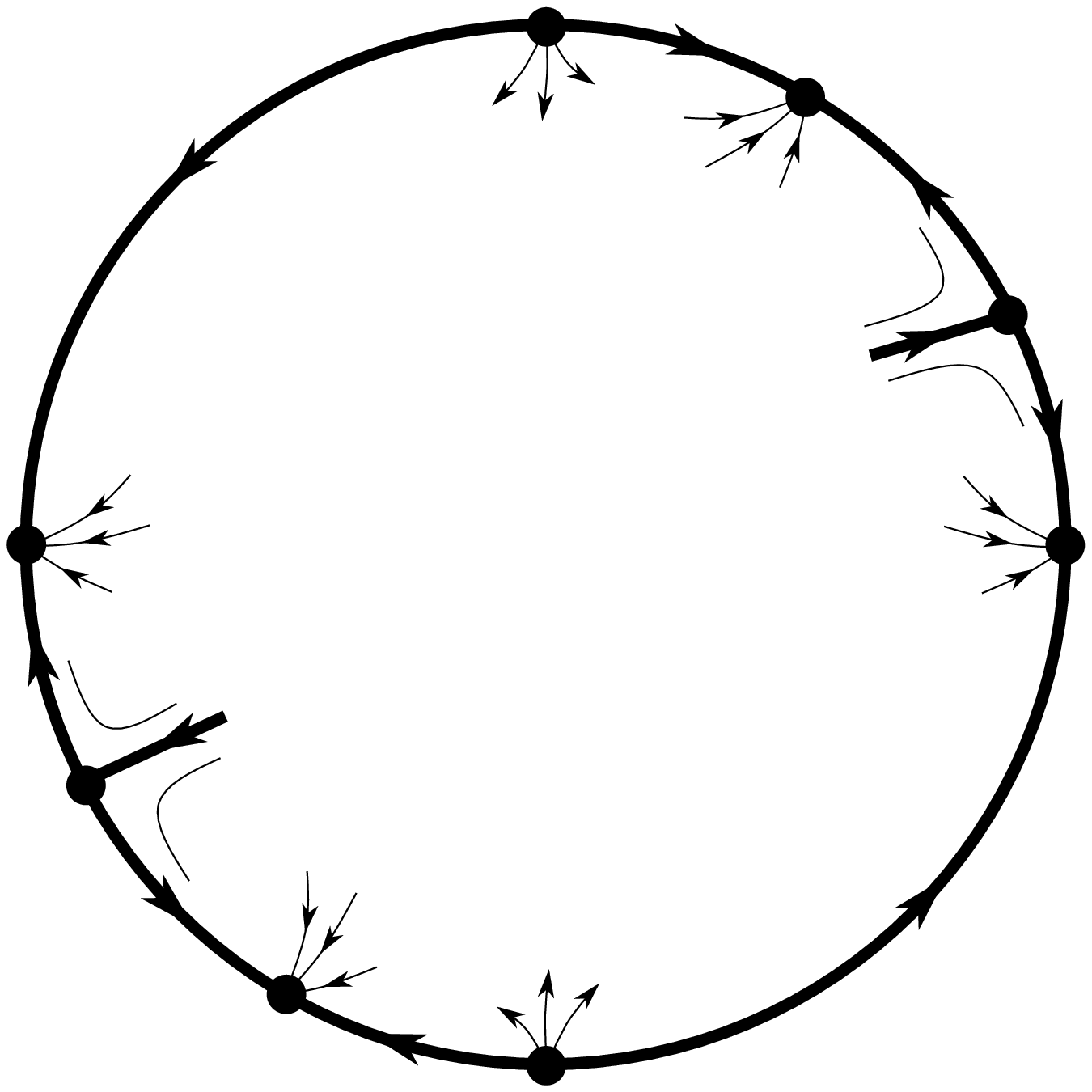} 
				\end{overpic}
				
				$(b)$
				
				$\Delta>0$ and $b_{10}>a_{01}$.
			\end{center}
		\end{minipage}
		\begin{minipage}{4.5cm}
			\begin{center}
				\begin{overpic}[height=3cm]{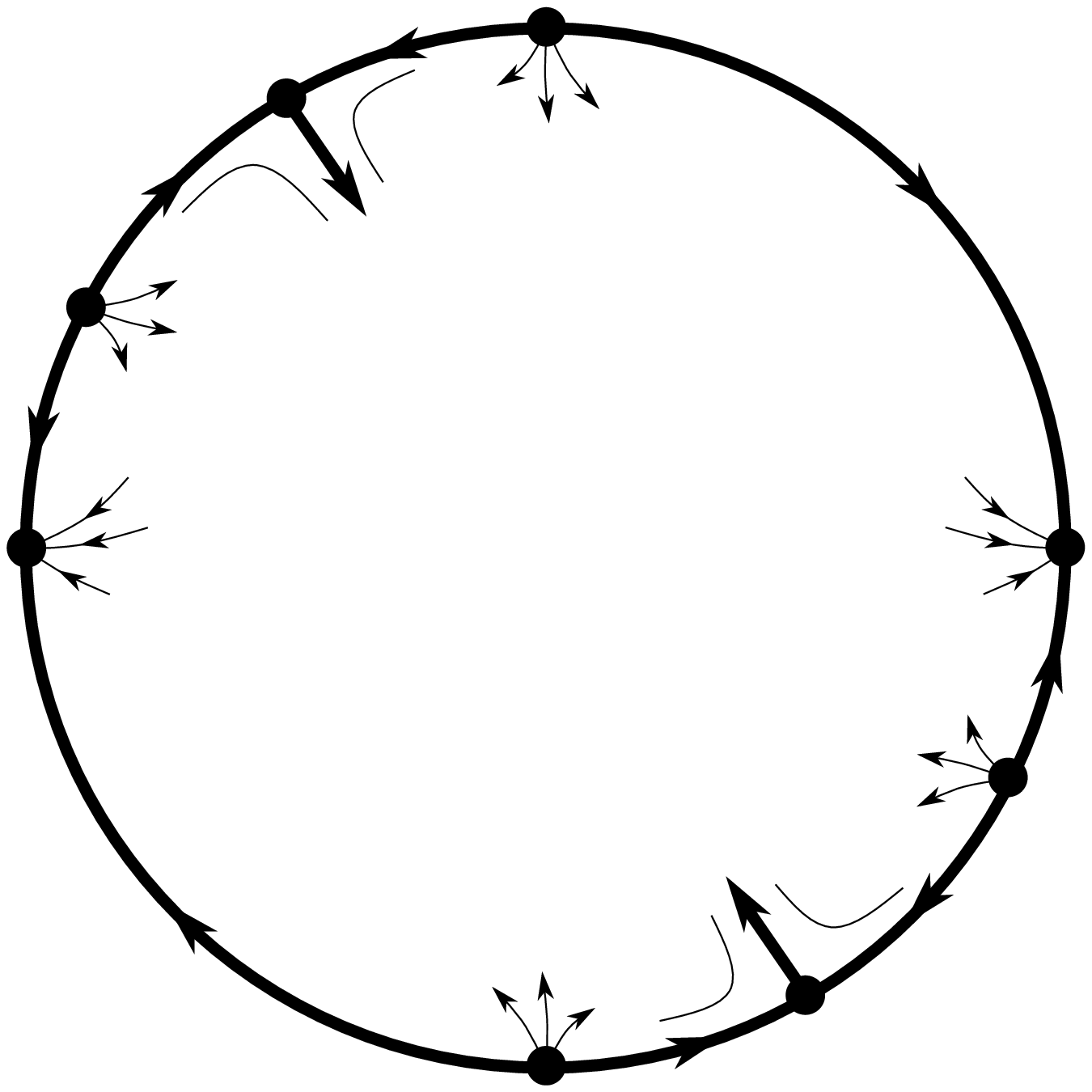} 
				\end{overpic}
				
				$(c)$
				
				$\Delta>0$ and $b_{10}<a_{01}$.
			\end{center}
		\end{minipage}
	\end{center}
	\caption{Local phase portrait of the second family at infinity.}\label{Case1.2Infinity}
\end{figure}
If $\Delta>0$, then $u_0^\pm\in\mathbb{R}$ and in this case we observe that,
	\[u_0^+u_0^-=-\frac{1}{b_{01}}>0.\]
Therefore, $u_0^+$ and $u_0^-$ have the same signal. Since $\Delta>0$, it follows from \eqref{30} and $b_{01}<0$ that $b_{10}\neq a_{01}$. Therefore,
	\[u_0^++u_0^-=-\frac{b_{10}-a_{01}}{b_{01}}=\frac{b_{10}-a_{01}}{|b_{01}|}\neq0.\]
In special, $u_0^\pm$ are both positive if $b_{10}>a_{01}$ and negative if $b_{10}<a_{01}$. In special, it follows from $b_{01}<0$ that $g(u)>0$ if, and only if, $u_0^+<u<u_0^-$. Suppose first $b_{10}>a_{01}$. It follows from \eqref{29} and from $a_{10}=1$ that to obtain the phase portrait at singularities $(u_0^\pm,0)$, we need to obtain the signal of,
\begin{equation}\label{51}
	a_{01}u_0^\pm+1.
\end{equation}
If $a_{01}=0$, then it is clear that \eqref{51} is positive. Suppose $a_{01}\neq0$. In this case we observe that,
	\[g\left(-\frac{1}{a_{01}}\right)=\frac{1}{a_{01}^2}\det A<0.\]
Hence, $-\frac{1}{a_{01}}<0<u^\pm$. Thus, \eqref{51} is positive. Therefore, we obtain the local phase portrait given by Figure~\ref{Case1.2Infinity}(b). Suppose $b_{10}<a_{01}$. In special, observe that $a_{01}\neq0$. Hence, from the signal of
\begin{equation}\label{13}
	g'\left(-\frac{1}{a_{01}}\right)=2\frac{|b_{01}|}{a_{01}}+(b_{10}-a_{01}),
\end{equation}
we can obtain the position between $u_0^\pm$ and $-\frac{1}{a_{01}}$. More precisely, 
\begin{enumerate}[label=(\alph*)]
	\item if $g'\left(-\frac{1}{a_{01}}\right)<0$, then $u_0^\pm<-\frac{1}{a_{01}}$. Thus, \eqref{51} is negative; 
	\item if $g'\left(-\frac{1}{a_{01}}\right)>0$, then $u_0^\pm>-\frac{1}{a_{01}}$. Thus, \eqref{51} is positive.
\end{enumerate}
Since $\Delta>0$, it follows that $(b_{10}-a_{01})^2>4|b_{01}|$. Therefore, it follows from $b_{10}-a_{01}<0$ that,
	\[0<\sqrt{|b_{01}|}<2\sqrt{|b_{01}|}<|b_{10}-a_{01}|=a_{01}-b_{10}<a_{01}.\]
In special, it follows that
\begin{equation}\label{10}
	b_{10}-a_{01}<-2\sqrt{|b_{01}|},
\end{equation}	
and that
\begin{equation}\label{11}
	\sqrt{|b_{01}|}<a_{01}.
\end{equation}
From \eqref{11} we have,
\begin{equation}\label{12}
	\frac{\sqrt{|b_{01}|}}{a_{01}}-1<0.
\end{equation}
Therefore, from \eqref{13}, \eqref{10} and \eqref{12} we have,
	\[g'\left(-\frac{1}{a_{01}}\right)=2\frac{|b_{01}|}{a_{01}}+(b_{10}-a_{01})<2\frac{|b_{01}|}{a_{01}}-2\sqrt{|b_{01}|}=2\sqrt{|b_{01}|}\left(\frac{\sqrt{|b_{01}|}}{a_{01}}-1\right)<0.\]
Hence, from \eqref{29} we can obtain the local phase portrait at the singularities given by $(u_0^\pm,0)$. See Figure~\ref{Case1.2Infinity}(c). The proof now follows as in Theorem~\ref{Theo10}. \end{proof}

\begin{remark}\label{Remark5}
From all the phase portraits given by Figures~\ref{Case1.2a}, \ref{Case1.2b} and \ref{Case1.2c}, if we look to the disposal of the $q$-singularities inside the square given by the $p$-singularities, it can be seen that:
	\begin{enumerate}[label=\arabic*)]
		\item Families $2.1a$ and $2.1b$ are topologically equivalent;
		\item Families $2.2a$, $2.3a$, $2.5b$ and $2.8b$ are topologically equivalent;
		\item Families $2.2b$, $2.3b$, $2.5a$ and $2.8a$ are topologically equivalent;
		\item Families $2.4a$ and $2.12b$ are topologically equivalent;
		\item Families $2.4b$ and $2.12a$ are topologically equivalent;
		\item Families $2.4c$ and $2.12c$ are topologically equivalent;
		\item Families $2.6a$, $2.9a$ and $2.10b$ are topologically equivalent;
		\item Families $2.6b$, $2.9b$ and $2.10a$ are topologically equivalent;
		\item Families $2.6c$, $2.9c$ and $2.10c$ are topologically equivalent;
		\item Families $2.7a$, $2.11a$, $2.13b$ and $2.14b$ are topologically equivalent;
		\item Families $2.7b$, $2.11b$, $2.13a$ and $2.14a$ are topologically equivalent;
		\item Families $2.7c$, $2.11c$, $2.13c$ and $2.14c$ are topologically equivalent;
		\item Families $2.15a$ and $2.15b$ are topologically equivalent.
	\end{enumerate}
Therefore, from forty five families, only sixteen remains (observe that families $2.1c$ and $.15c$ were not named above). If we look at the first family of each class above, then it can be seen that: 
	\begin{enumerate}[label=(\alph*)]
		\item Families $2.1a$, $2.2b$ and $2.4b$ are topologically equivalent;
		\item Families $2.1c$, $2.2c$, $2.4c$, $2.6c$, $2.7c$ and $2.15c$ are topologically equivalent;
		\item Families $2.2a$, $2.6a$, $2.6b$ and $2.7b$ are topologically equivalent;
		\item Families $2.4a$, $2.7a$ and $2.15a$ are topologically equivalent.
	\end{enumerate}
For example, given case $2.1c$, one need just to ``slide'' the $q$-singularities inside the square given by the $p$-singularities to obtain case $2.15c$. Therefore, from the sixteen classes above, only four remains. Finally, we observe that:
	\begin{enumerate}[label=(\alph*)]
		\item Cases $2.1a2$, $2.1a3$ and $2.2a1$ are topologically equivalent;
		\item Cases $2.1a4$, $2.2a2$ and $2.4a1$ are topologically equivalent.
	\end{enumerate}
Thus, from the eighty three phase portraits given by Figures~\ref{Case1.2a}, \ref{Case1.2b} and \ref{Case1.2c}, only four remains. Finally, we observe that this technique of sliding the $q$-singularities does not apply in the first generic family because it would break/create heteroclinic connections between hyperbolic saddles. Look for example at Cases~$1.1$ and $1.2$. More precisely, look at the heteroclinic connection between $p_2$ and $q_2$ at Case $1.2$.
\end{remark}

\begin{corollary}\label{Coro4}
	If $X\in\Sigma_0$ satisfies the hypothesis of Theorem~\ref{Theo11}, then its phase portrait in the Poincar\'e disk is topologically equivalent to one of the four phase portraits given by Figure~\ref{Case1.2Final}.
	\begin{figure}[h]
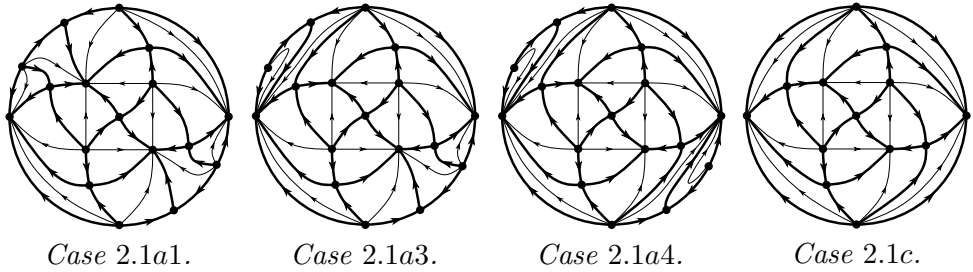

		\begin{center}
			\begin{minipage}{3.1cm}
				\begin{center}
					\begin{overpic}[height=3cm]{Case1.2.1a1.eps} 
					\end{overpic}
					
					Case~$2.1a1$.
				\end{center}
			\end{minipage}
			\begin{minipage}{3.1cm}
				\begin{center}
					\begin{overpic}[height=3cm]{Case1.2.1a3.eps} 
					\end{overpic}
					
					Case~$2.1a3$.
				\end{center}
			\end{minipage}
			\begin{minipage}{3.1cm}
				\begin{center}
					\begin{overpic}[height=3cm]{Case1.2.1a4.eps} 
					\end{overpic}
					
					Case~$2.1a4$.
				\end{center}
			\end{minipage}	
			\begin{minipage}{3.1cm}
				\begin{center}
					\begin{overpic}[height=3cm]{Case1.2.1c.eps} 
					\end{overpic}
					
					Case~$2.1c$.
				\end{center}
			\end{minipage}
		\end{center}
		\caption{Topologically distinct phase portraits of Family $2$.}\label{Case1.2Final}
	\end{figure}
\end{corollary}

\subsection{Case \boldmath{$3$}: \boldmath{$b_{10}\leqslant0$} and \boldmath{$b_{01}>0$}}

In this section, we highlight a special case in which we need to understand the dynamics of a polycycle. Consider the position given by,
\begin{equation}\label{14}
	q_1>p_2, \quad q_2>p_3, \quad q_3<p_4, \quad q_4<p_1.
\end{equation}
As in Theorem~\ref{Theo10}, it can be seen that under \eqref{14}, the phase portrait is given by Figure~\ref{Fig6}.
\begin{figure}[h]
	\begin{center}
		\begin{minipage}{4.5cm}
			\begin{center}
				\begin{overpic}[height=3cm]{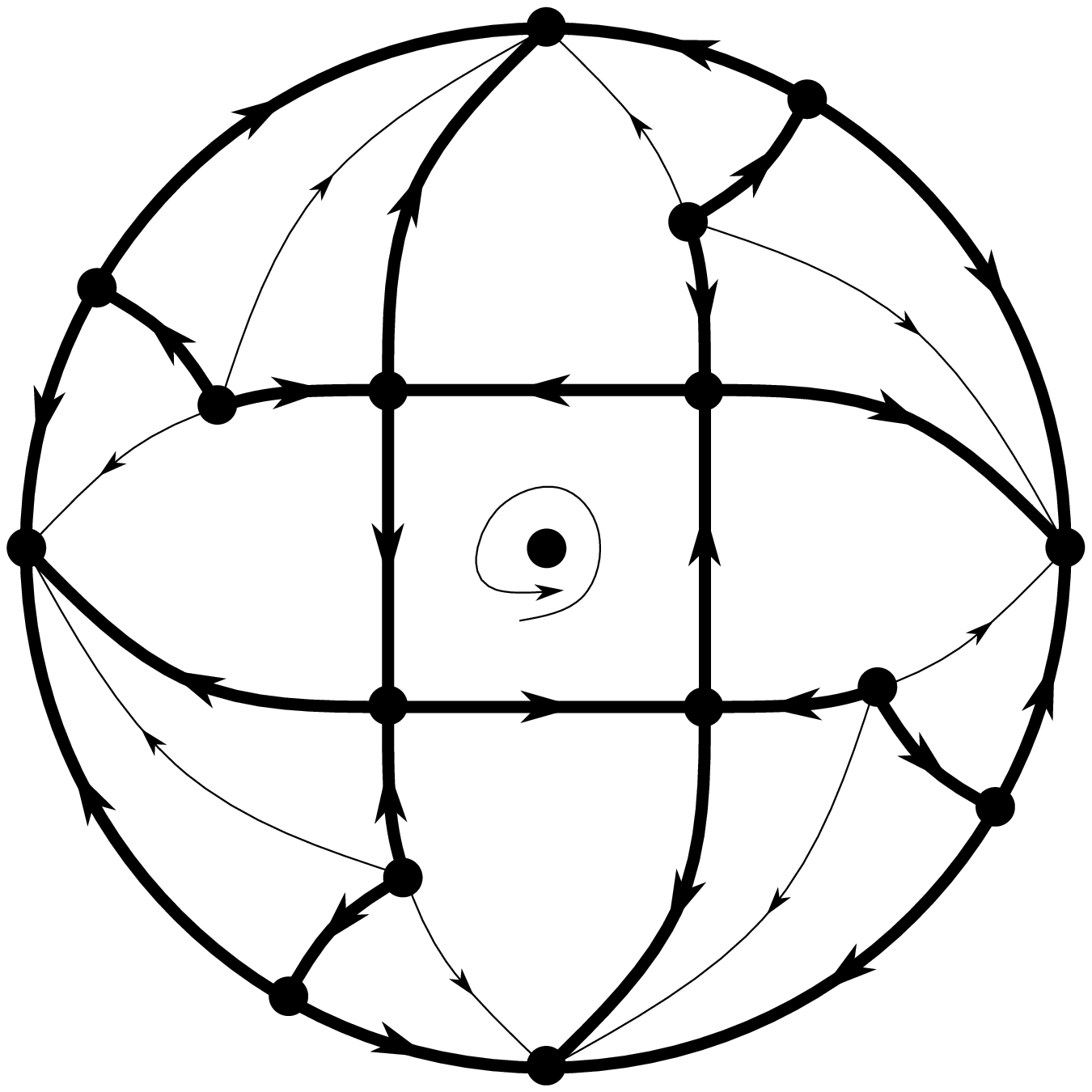} 
				\end{overpic}
				
				$(a)$
			\end{center}
		\end{minipage}
		\begin{minipage}{4.5cm}
			\begin{center}
				\begin{overpic}[height=3cm]{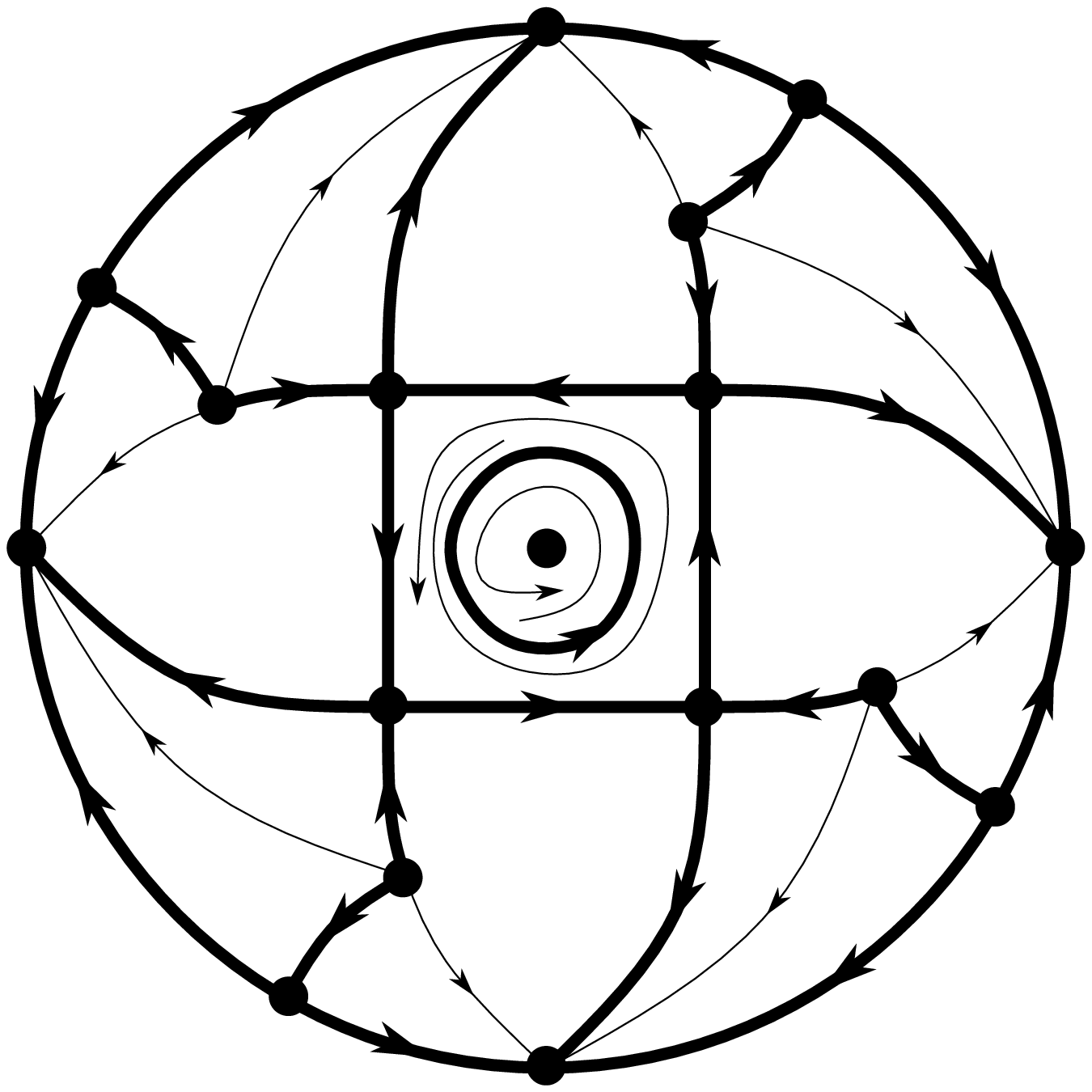} 
				\end{overpic}
				
				$(b)$
			\end{center}
		\end{minipage}
	\end{center}
	\caption{Phase portrait of Case~$3.1$, except by a limit cycle in the center of the octothorpe.}\label{Fig6}
\end{figure}
We know that these two cases are the only possibilities due to Theorem~\ref{Theo6}. Therefore, to proof which phase portraits of Figure~\ref{Fig6} is realizable (if not both), it is enough to study the stability of the polycycle $\Gamma$ given by singularities $p_1$, $p_2$, $p_3$ and $p_4$. Let $\mu_i<0<\nu_i$ be the eigenvalues of singularity $p_i$ and $r_i=\frac{|\mu_i|}{\nu_i}$
the \emph{hyperbolicity ratio} of $p_i$, $i\in\{1,2,3,4\}$. Cherkas \cite{Cherkas} proved that if $r(\Gamma)=r_1r_2r_3r_4>1$ (resp. $r(\Gamma)<1)$, then the polycycle $\Gamma$ is stable (resp. unstable). Therefore, if $r(\Gamma)\neq1$, then the limit cycle in the center of the octothorpe exists if, and only if, $r(\Gamma)>1$. Observe that $r(\Gamma)>1$ if, and only if,
\begin{equation}\label{32}
	a_{01}b_{01}(\beta-1)\beta-a_{10}b_{10}(\alpha-1)\alpha>0.
\end{equation}
It follows from \eqref{14} that $a_{01}b_{10}\neq0$. Thus, under the hypothesis of this subsection, observe that \eqref{32} is negative. Hence, $r(\Gamma)<1$. Therefore, polycycle $\Gamma$ is unstable. Thus, Figure~\ref{Fig6}(a) is the only realizable phase portrait. In the next Proposition we obtain a sufficient and necessary algebraic condition to deal with similar cases.

\begin{proposition}\label{Theo12}
	Let $X=(P,Q)$ be the planar vector field given by
		\[P(x,y)=(x+\alpha)(x+\alpha-1)(a_{10}x+a_{01}y), \quad Q(x,y)=(y+\beta)(y+\beta-1)(b_{10}x+b_{01}y),\]
	such that the origin is a focus/node and lies in the center of the octothorpe. Also, let the $p$ and $q$-singularities be arranged such that there is no $q$-singularity between two $p$-singularities and such that all the $p$-singularities are hyperbolic saddles. Then, $X$ has a limit cycle around the origin if, and only if,
	\begin{equation}\label{33}
		\bigl(a_{10}(\alpha-1)\alpha+b_{01}(\beta-1)\beta\bigr)\bigl(a_{01}b_{01}(\beta-1)\beta-a_{10}b_{10}(\alpha-1)\alpha\bigr)<0.
	\end{equation}
\end{proposition}

\begin{proof} It follows from the hypothesis that the $p$-singularities define a polycycle $\Gamma$ around the origin. Let,
	\[K=a_{01}b_{01}(\beta-1)\beta-a_{10}b_{10}(\alpha-1)\alpha.\]
Observe that $r(\Gamma)>1$ if, and only if, $K>0$. Also, let $T$ denote the trace of the Jacobian matrix of $X$ at the origin. It is clear that,
	\[T=a_{10}(\alpha-1)\alpha+b_{01}(\beta-1)\beta.\]
From the reasoning given above, it is clear that if \eqref{33} holds, then $X$ has a limit cycle around the origin. From this same reasoning, it is also clear that if $T K>0$, then $X$ has no limit cycles. Therefore, we need only to study the case $T K=0$. If $T=0$, then the origin is a weak focus or a center. In the former, we know from Theorem~\ref{Theo4} that $X$ cannot have limit cycles. In the later, it is clear that there is no limit cycles. Thus, suppose $T\neq0$ and $K=0$. Without loss of generality, suppose $T>0$. If $X$ has a limit cycle, then it follows from Theorem~\ref{Theo6} that it is hyperbolic. Therefore, any small enough perturbation of $X$ will also have a limit cycle. But, it is clear that we can choose an arbitrarily small perturbation of $X$ such that $T>0$ and $K>0$. Hence, we have $TK>0$. Thus, $X$ has no limit cycle. Contradicting the fact that the limit cycle is hyperbolic. \end{proof}

\begin{theorem}\label{Theo13}
	Let $X=(P,Q)$ be the planar vector field given by,
		\[P(x,y)=(x+\alpha)(x+\alpha-1)(x+a_{01}y), \quad Q(x,y)=(y+\beta)(y+\beta-1)(b_{10}x+b_{01}y).\]
	If $X\in\Sigma_0$ and $\frac{1}{2}\leqslant\beta<1$, $0<\alpha<1$, $a_{01}\geqslant0$, $b_{10}\leqslant0$, $b_{01}>0$,	then its phase portrait in the Poincar\'e disk is topologically equivalent to one of the phase portraits given by Figure~\ref{Case1.3}.
\end{theorem}

\begin{proof} This case is the simplest one and its proof follows similarly to the proof of Theorem~\ref{Theo10}. The only difference was already studied at Proposition~\ref{Theo12}. \end{proof}
	
\begin{remark}
	From all the phase portraits given by Figure~\ref{Case1.3}, it can be seen that:
	\begin{enumerate}[label=(\alph*)]
		\item Case $3.1$ is topologically equivalent to none of the other cases;
		\item Cases $3.2$, $3.3$, $3.5$ and $3.8$ are topologically equivalent;
		\item Cases $3.4$ and $3.12$ are topologically equivalent;
		\item Cases $3.6$, $3.9$, and $3.10$ are topologically equivalent;
		\item Cases $3.7$, $3.11$, $3.13$ and $3.14$ are topologically equivalent;
		\item Case $3.15$ is topologically equivalent to none of the other cases.
	\end{enumerate}
	Therefore, from fifteen phase portraits, only six remains.
\end{remark}

\begin{corollary}\label{Coro5}
	If $X\in\Sigma_0$ satisfies the hypothesis of Theorem~\ref{Theo13}, then its phase portrait in the Poincar\'e disk is topologically equivalent to one of the six phase portraits given by Figure~\ref{Case1.3Final}.	
	\begin{figure}[h]
		\begin{center}
			\begin{minipage}{3.1cm}
				\begin{center}
					\begin{overpic}[height=3cm]{Case1.3.1.eps} 
					\end{overpic}
					
					Case~$1.3.1$.
				\end{center}
			\end{minipage}
			\begin{minipage}{3.1cm}
				\begin{center}
					\begin{overpic}[height=3cm]{Case1.3.2.eps} 
					\end{overpic}
					
					Case~$1.3.2$.
				\end{center}
			\end{minipage}
			\begin{minipage}{3.1cm}
				\begin{center}
					\begin{overpic}[height=3cm]{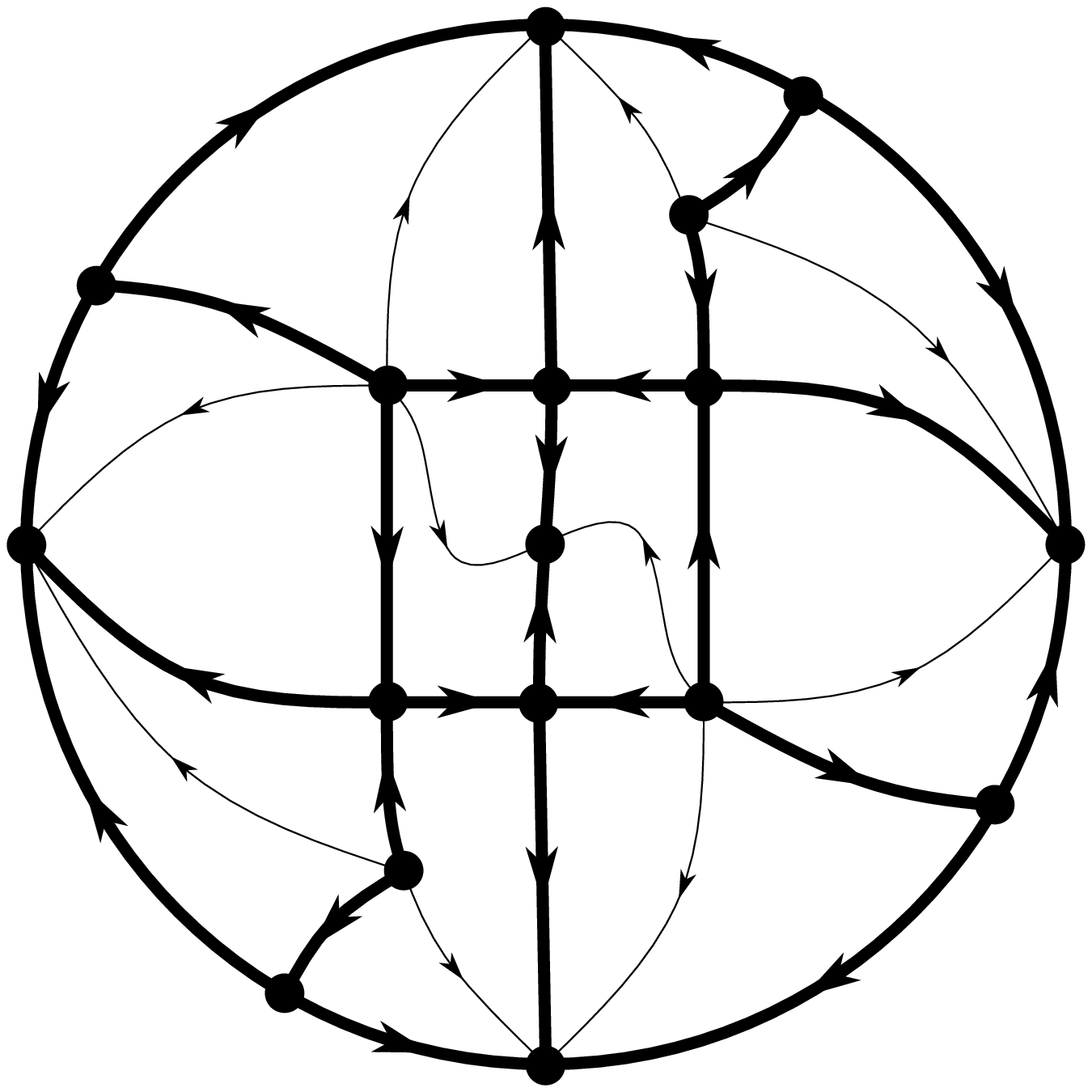} 
					\end{overpic}
					
					Case~$1.3.4$.
				\end{center}
			\end{minipage}
		\end{center}
		$\;$
		\begin{center}
			\begin{minipage}{3.1cm}
				\begin{center}
					\begin{overpic}[height=3cm]{Case1.3.6.eps} 
					\end{overpic}
					
					Case~$1.3.6$.
				\end{center}
			\end{minipage}
			\begin{minipage}{3.1cm}
				\begin{center}
					\begin{overpic}[height=3cm]{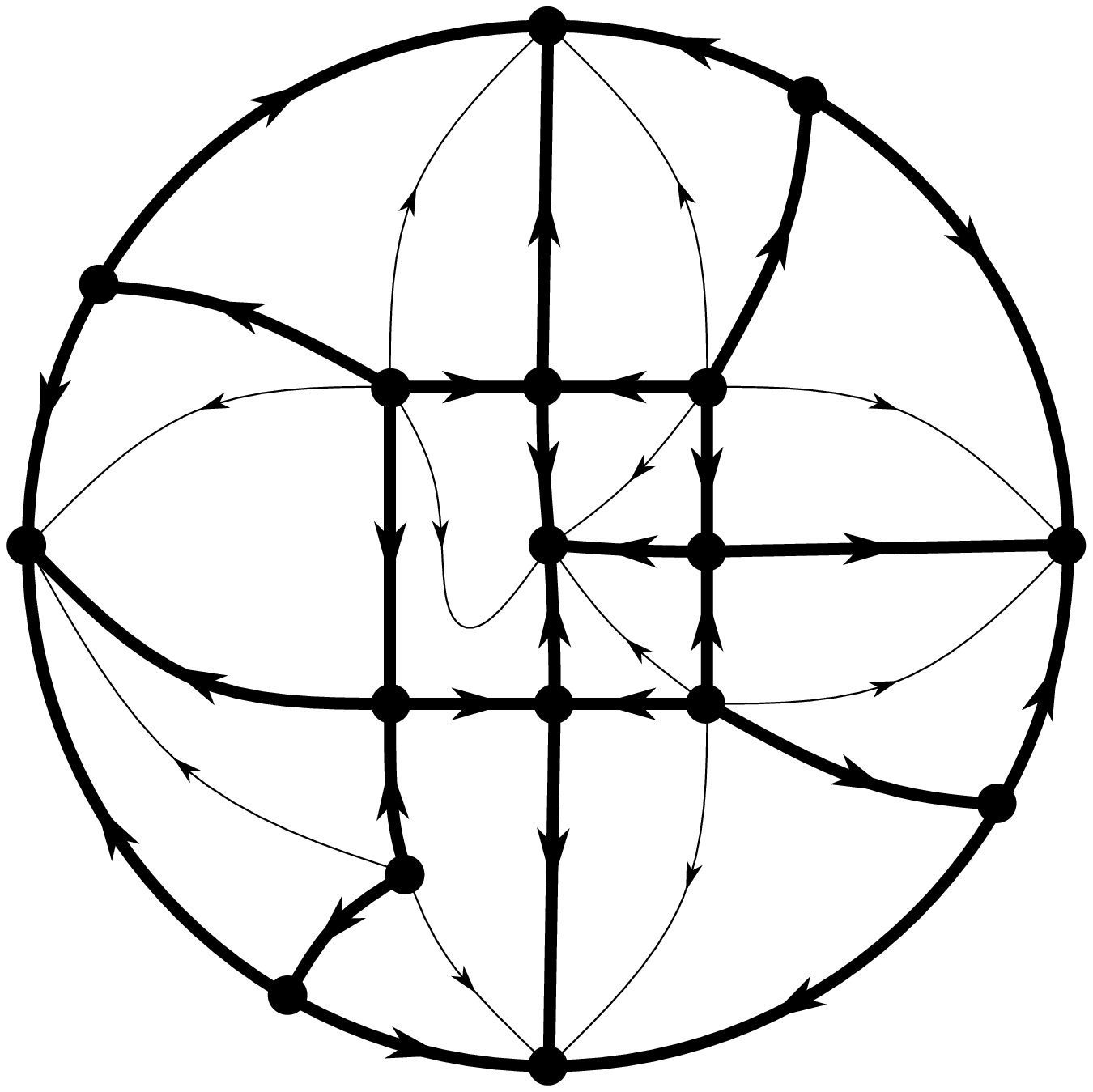} 
					\end{overpic}
					
					Case~$1.3.7$.
				\end{center}
			\end{minipage}
			\begin{minipage}{3.1cm}
				\begin{center}
					\begin{overpic}[height=3cm]{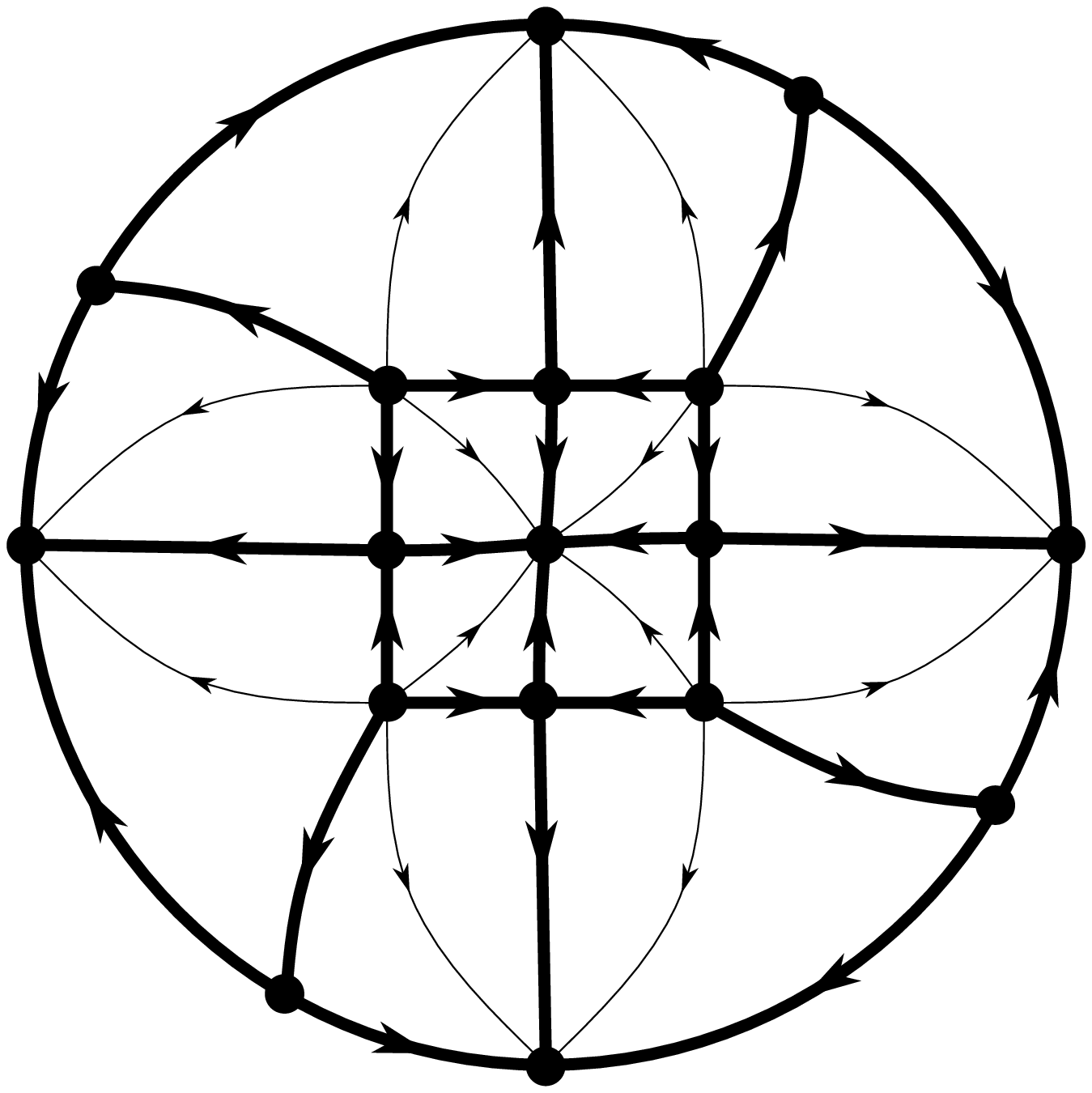} 
					\end{overpic}
					
					Case~$1.3.15$.
				\end{center}
			\end{minipage}
		\end{center}
		\caption{Topologically distinct phase portraits of Family~$3$.}\label{Case1.3Final}
	\end{figure}
\end{corollary}

\subsection{Case \boldmath{$4$}: \boldmath{$b_{10}\leqslant0$} and \boldmath{$b_{01}<0$}}

\begin{theorem}\label{Theo14}
	Let $X=(P,Q)$ be the planar vector field given by,
		\[P(x,y)=(x+\alpha)(x+\alpha-1)(x+a_{01}y), \quad Q(x,y)=(y+\beta)(y+\beta-1)(b_{10}x+b_{01}y).\]
	If $X\in\Sigma_0$ and $\frac{1}{2}\leqslant\beta<1$, $0<\alpha<1$, $a_{10}\geqslant0$, $b_{10}\leqslant0$, $b_{01}<0$,	then its phase portrait in the Poincar\'e disk is topologically equivalent to one of the phase portraits given by Figures~\ref{Case1.4a}, \ref{Case1.4b} and \ref{Case1.4c}.
\end{theorem}

\begin{proof} As in Theorem~\ref{Theo10}, observe that:
\begin{multicols}{2}
	\begin{enumerate}[label=(\alph*)]
		\item if $q_3<p_4$ and $q_4>p_4$, then $\det A>0$; 
		\item if $q_3>p_4$ and $q_4<p_4$, then $\det A<0$; 
	\end{enumerate}
	\columnbreak
	\begin{enumerate}[label=(\alph*)]
		\setcounter{enumi}{2}
		\item if $q_1>p_2$ and $q_2<p_2$, then $\det A>0$; 
		\item if $q_1<p_2$ and $q_2>p_2$, then $\det A<0$.
	\end{enumerate}
\end{multicols}
Furthermore, we also need to study the signal of,
\begin{equation}\label{36}
	\Delta=(b_{10}-a_{01})^2+4b_{01}.
\end{equation}
As in Theorem~\ref{Theo11}, we define $r=-\frac{b_{10}}{b_{01}}\leqslant0$. Hence, we can take
\begin{equation}\label{37}
	b_{10}=-rb_{01},
\end{equation} 
and let $b_{01}<0$ free, without changing the positions between the $p$ and $q$-singularities. Replacing \eqref{37} at \eqref{36} we get,
	\[\Delta=r^2b_{01}^2+(2a_{01}r+4)b_{01}+a_{01}^2.\]
Hence, we define,
	\[f(x)=r^2x^2+(2a_{01}r+4)x+a_{01}^2.\]
As in Theorem~\ref{Theo11}, observe that we can suppose $r\neq0$. Therefore, it is clear that,
	\[x^\pm=\frac{-(a_{01}r+2)\pm2\sqrt{a_{01}r+1}}{r^2},\]
are the solutions of $f(x)=0$. It follows from $b_{01}<0$ that, 
	\[a_{01}r+1>0 \Leftrightarrow \det A<0.\] 
Thus, if $\det A<0$, then we can choose $b_{01}<0$ such that $\Delta>0$ and $\Delta<0$. Hence, if $\det A<0$, then, for each disposal between the $p$ and $q$-singularities, we need to consider the cases $\Delta>0$ and $\Delta<0$. Suppose $\det A>0$. Then, $f(x)>0$ for all $x\in\mathbb{R}$. Thus, $\Delta>0$ for all $b_{01}\in(-\infty,0)$. It follows from $b_{10}<0$ and $b_{01}<0$ that $u_0^\pm<0$. Therefore, there is no need to divide such case as in Theorem~\ref{Theo11}. It follows from $\det A>0$ that the origin is a hyperbolic node/focus (see Proposition~\ref{Theo7}). Hence, we need to study the signal of,
	\[T=(\alpha-1)\alpha+b_{01}(\beta-1)\beta.\]
To do this, we remember that we have $b_{01}\in(-\infty,0)$ free. Then, it is clear that both $T>0$ and $T<0$ are realizable. Furthermore, we observe that we may have polycycles around the origin. Therefore, it follows from Proposition~\ref{Theo12} that we also need to study the signal of,
	\[K=a_{01}b_{01}(\beta-1)\beta-b_{10}(\alpha-1)\alpha.\]
Replacing \eqref{37} we get,
\begin{equation}\label{38}
	K=b_{01}[a_{01}(\beta-1)\beta+r(\alpha-1)\alpha.]
\end{equation}
Observe that, in this subsection, the polycycle exists if, and only if,
\begin{equation}\label{52}
	q_1>p_2, \quad q_2<p_2, \quad q_3<p_4, \quad q_4>p_4.
\end{equation}
Thus, we have,
\begin{equation}
	a_{01}>\frac{1-\alpha}{\beta}, \quad r<-\frac{\beta}{1-\alpha}, \quad a_{10}>\frac{\alpha}{1-\beta}, \quad r<-\frac{1-\beta}{\alpha}.
\end{equation}
Therefore, if we define the intervals
	\[I=\left(\max\left\{\frac{1-\alpha}{\beta},\frac{\alpha}{1-\beta}\right\},+\infty\right), \quad J=\left(-\infty,\;\min\left\{-\frac{\beta}{1-\alpha},-\frac{1-\beta}{\alpha}\right\}\right),\]
then, we may have $a_{01}\in I$, $r\in J$ free, without changing \eqref{52}. Hence, it follows from \eqref{38} that whether $T>0$ or $T<0$, both $K>0$ and $K<0$ are realizable by $a_{01}\to+\infty$ and $r\to-\infty$, respectively. We now study the infinity. If $\Delta<0$, then it is clear that we have Figure~\ref{Case1.4Infinity}(a).
\begin{figure}[h]
	\begin{center}
		\begin{minipage}{3.9cm}
			\begin{center}
				\begin{overpic}[height=3cm]{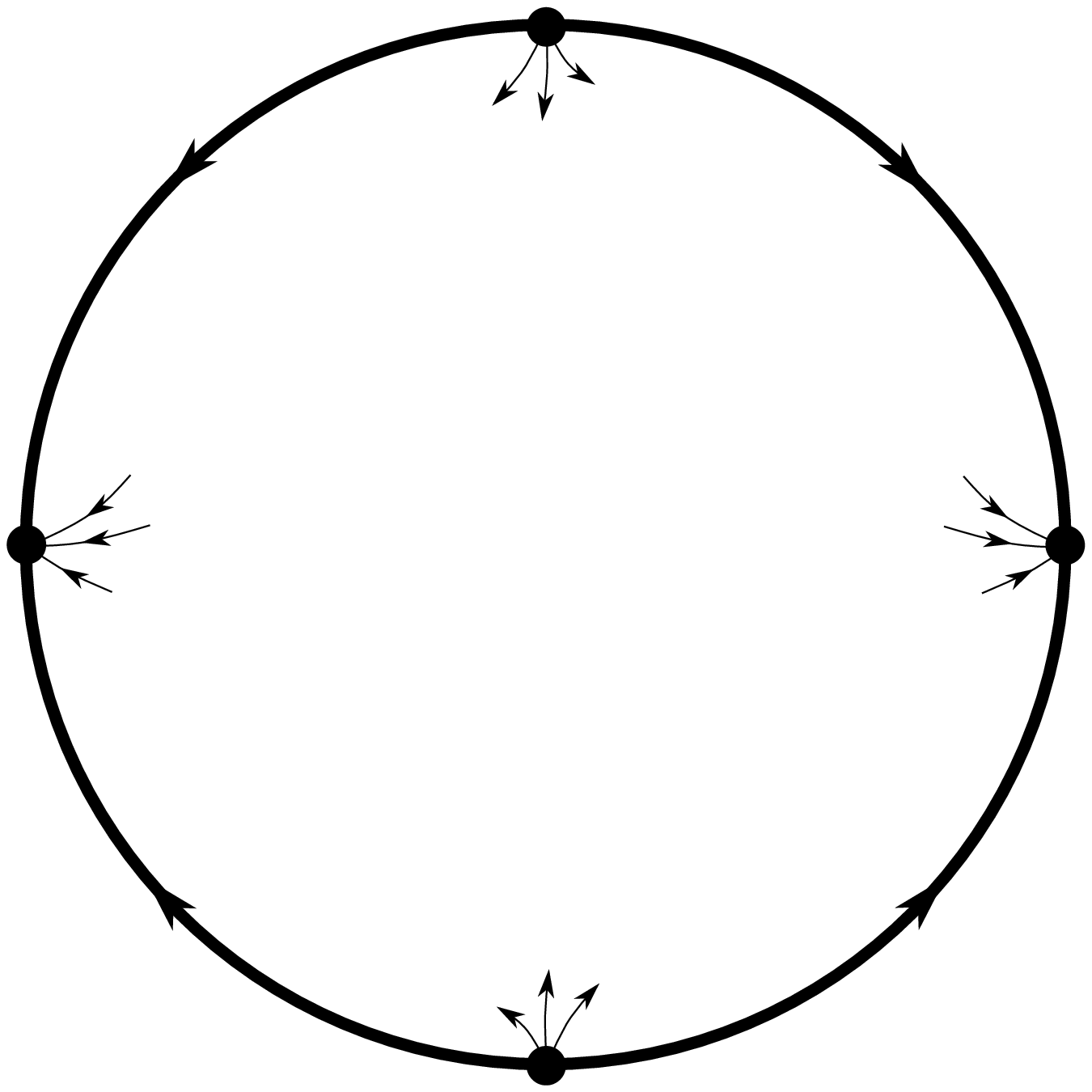} 
				\end{overpic}
			
				$(a)$
				
				$\Delta<0$.
			\end{center}
		\end{minipage}
		\begin{minipage}{3.9cm}
			\begin{center}
				\begin{overpic}[height=3cm]{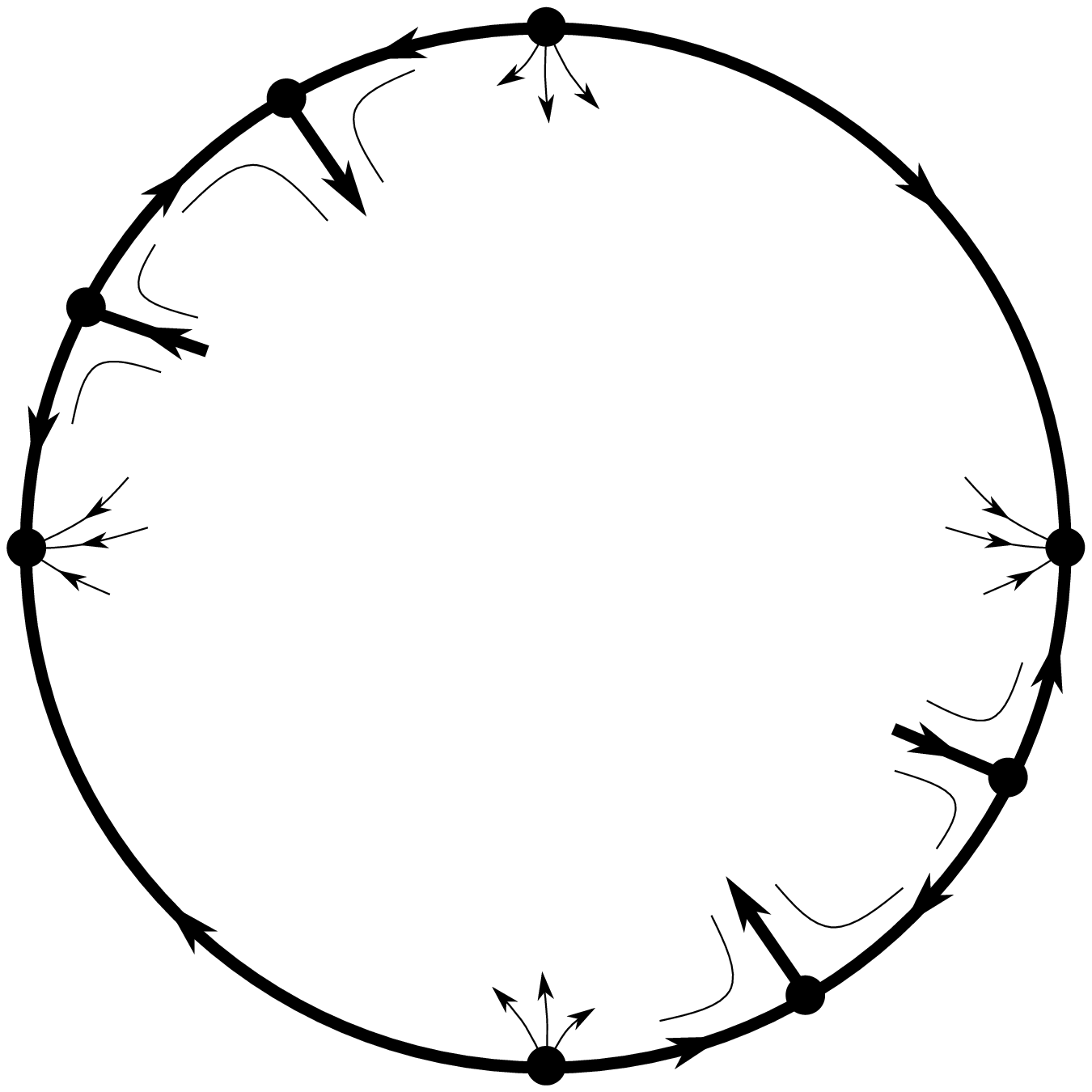} 
				\end{overpic}
			
				$(b)$
				
				$\det A>0$.
			\end{center}
		\end{minipage}
		\begin{minipage}{3.9cm}
			\begin{center}
				\begin{overpic}[height=3cm]{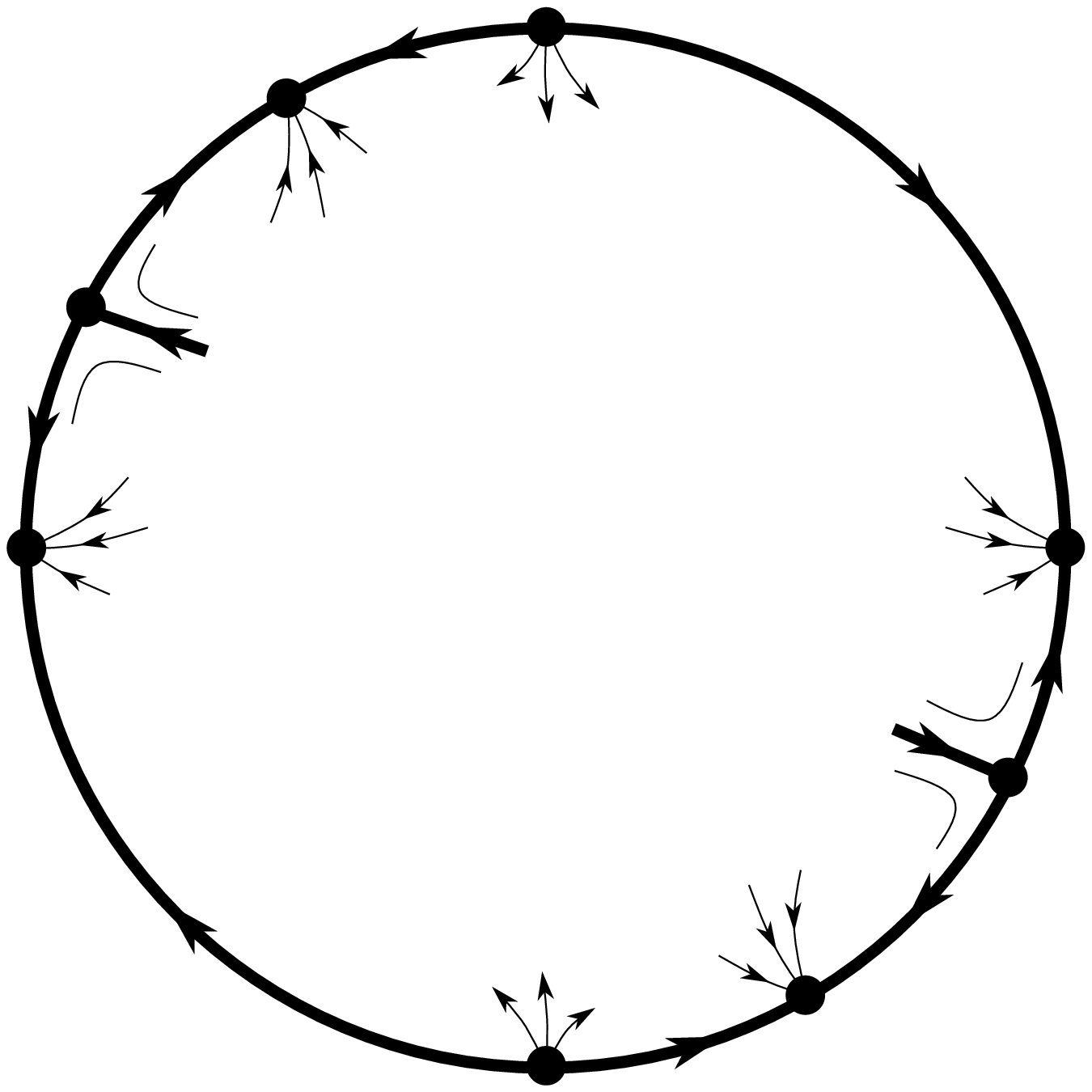} 
				\end{overpic}
				
				$(c)$
				
				$\det A<0$ and $g'>0$.
			\end{center}
		\end{minipage}
		\begin{minipage}{3.9cm}
			\begin{center}
				\begin{overpic}[height=3cm]{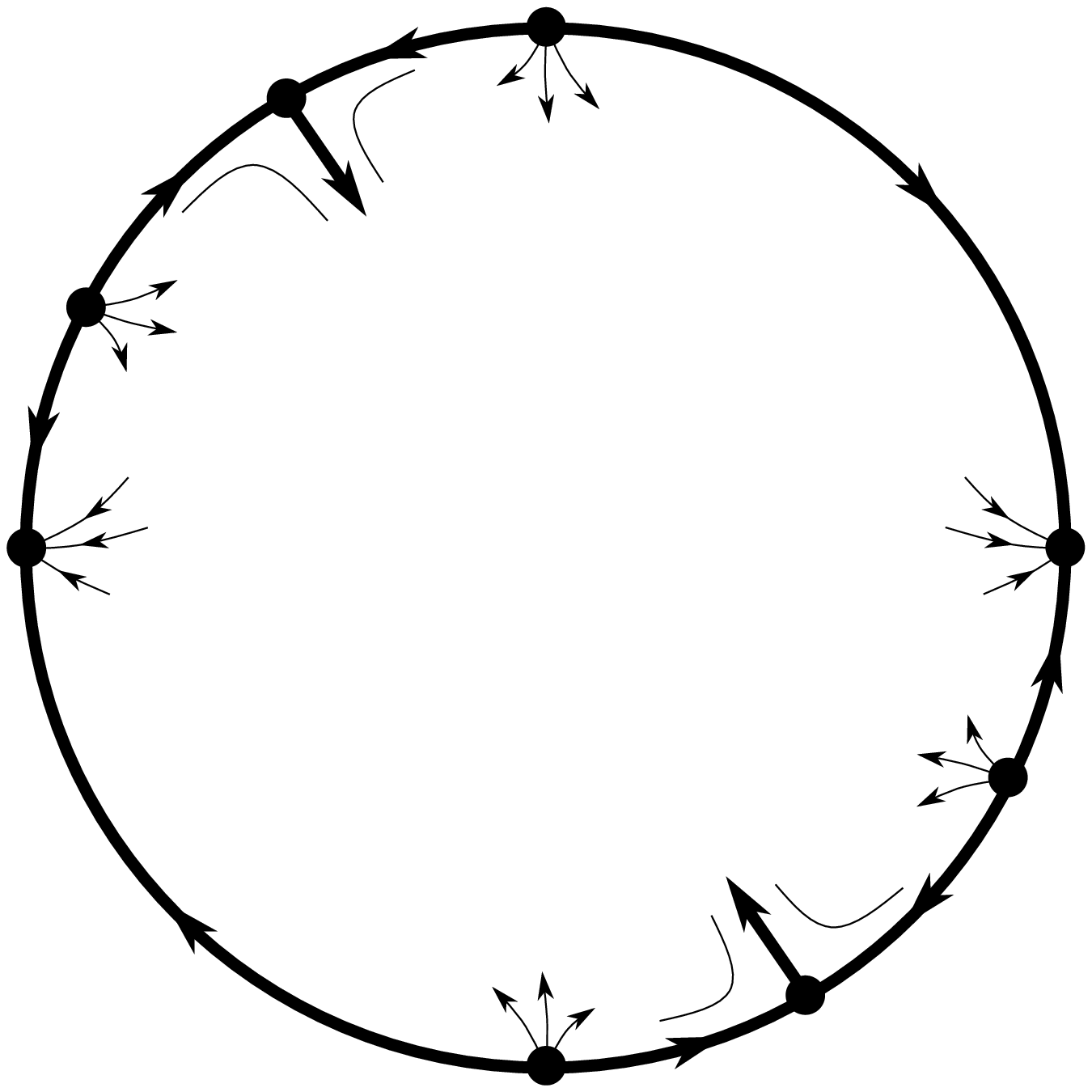} 
				\end{overpic}
				
				$(d)$
				
				$\det A<0$ and $g'<0$.
			\end{center}
		\end{minipage}
	\end{center}
	\caption{Local phase portrait of the fourth family at infinity.}\label{Case1.4Infinity}
\end{figure}
Suppose $\Delta>0$. We already saw that $u_0^\pm<0$. Observe that if $\det A>0$, then $a_{01}\neq0$. Thus, we can consider,
	\[g\left(-\frac{1}{a_{01}}\right)=\frac{1}{a_{01}^2}\det A>0.\]
Therefore, if $\det A>0$, then $u_0^+<-\frac{1}{a_{01}}<u_0^-$. Hence, we obtain Figure~\ref{Case1.4Infinity}(b). Suppose $\det A<0$. As in Theorem~\ref{Theo11}, if $a_{01}\neq0$, then we study the signal of $g'(-\frac{1}{a_{01}})$. We claim that for each disposal between the $p$ and $q$-singularities, both $g'(-\frac{1}{a_{01}})<0$ and $g'(-\frac{1}{a_{01}})>0$ are realizable. Suppose $a_{01}\neq0$. Observe that,
\begin{equation}\label{15}
	g'\left(-\frac{1}{a_{01}}\right)=2\frac{|b_{01}|}{a_{01}}+(b_{10}-a_{01}).
\end{equation}
Replacing \eqref{37} at \eqref{15} we get,
	\[g'\left(-\frac{1}{a_{01}}\right)=|b_{01}|\left(\frac{2}{a_{01}}+r\right)-a_{01}.\]
It follows from $\det A<0$ and $a_{01}>0$ that $\frac{1}{a_{01}}+r>0$. Thus, $\frac{2}{a_{01}}+r>0$. Hence, both signals of $g'(-\frac{1}{a_{01}})$ are realizable, provided $b_{01}\to-\infty$ and $b_{01}\to0$. Hence, we obtain Figures~\ref{Case1.4Infinity}(c) and (d). Finally, given a relative position between the $p$ and $q$-singularities such that $a_{01}=0$, observe that we can take $a_{01}>0$ small enough without changing such relative position. Hence, we obtain Table~\ref{Table6}. The proof now follows as in Theorem~\ref{Theo10}. \end{proof}

\begin{remark}\label{Remark6}
As in Remark~\ref{Remark5}, we observe that:
	\begin{enumerate}[label=\arabic*)]
		\item Families $4.2a$, $4.3a$, $4.5b$ and $4.8b$ are topologically equivalent;
		\item Families $4.2b$, $4.3b$, $4.5a$ and $4.8a$ are topologically equivalent;
		\item Families $4.4a.i$ and $4.11a.ii$ are topologically equivalent;
		\item Families $4.4a.ii$ and $4.11a.i$ are topologically equivalent;
		\item Families $4.4b.i$ and $4.4b.ii$ are topologically equivalent;
		\item Families $4.4b.ii$, $4.7b$, $4.10b$, $4.11b.i$, $4.12a$, $4.13a$, $4.14a$ and $4.14b$ are topologically equivalent;
		\item Families $4.4b.iii$, $4.6b.iii$, $4.7c$, $4.9b.iii$, $4.10c$, $4.11b.iii$, $4.12c$, $4.13c$ and $4.14c$ are topologically equivalent;
		\item Families $4.6a.i$, $4.6a.ii$, $4.9a.i$ and $4.9a.ii$ are topologically equivalent;
		\item Families $4.6b.i$, $4.6b.ii$, $4.7a$, $4.9b.i$, $4.9b.ii$, $4.10a$, $4.12b$ and $4.13b$ are topologically equivalent.
	\end{enumerate}
Therefore, from forty seven families, only thirteen remains (observe that cases $4.1a.i$, $4.1a.ii$, $4.1b.i$ and $4.1b.ii$ were not named above). Furthermore, we observe that:
	\begin{enumerate}[label=\arabic*)]
		\item Cases $4.4b.i.2$, $4.4b.i.3$ and $4.6b.i.1$ are topologically equivalent;
		\item Cases $4.4b.i.4$, $4.4b.ii$ and $4.6b.i.2$ are topologically equivalent.
	\end{enumerate}
Thus, from sixty five phase portraits, only fourteen remains. 
\end{remark}

\begin{corollary}\label{Coro6}
	If $X\in\Sigma_0$ satisfies the hypothesis of Theorem~\ref{Theo14}, then its phase portrait in the Poincar\'e disk is topologically equivalent to one of the fourteen phase portraits given by Figure~\ref{Case1.4Final}.	
	\begin{figure}[h]
		\begin{center}
			\begin{minipage}{3.1cm}
				\begin{center}
					\begin{overpic}[height=3cm]{Case1.4.1a1.eps} 
					\end{overpic}
					
					Case~$4.1a.i$.
				\end{center}
			\end{minipage}
			\begin{minipage}{3.1cm}
				\begin{center}
					\begin{overpic}[height=3cm]{Case1.4.1a2.eps} 
					\end{overpic}
					
					Case~$4.1a.ii$.
				\end{center}
			\end{minipage}
			\begin{minipage}{3.1cm}
				\begin{center}
					\begin{overpic}[height=3cm]{Case1.4.1b1.eps} 
					\end{overpic}
					
					Case~$4.1b.i$.
				\end{center}
			\end{minipage}	
			\begin{minipage}{3.1cm}
				\begin{center}
					\begin{overpic}[height=3cm]{Case1.4.1b2.eps} 
					\end{overpic}
					
					Case~$4.1b.ii$.
				\end{center}
			\end{minipage}
			\begin{minipage}{3.1cm}
				\begin{center}
					\begin{overpic}[height=3cm]{Case1.4.2a.eps} 
					\end{overpic}
					
					Case~$4.2a$.
				\end{center}
			\end{minipage}
		\end{center}
	$\;$
		\begin{center}
			\begin{minipage}{3.1cm}
				\begin{center}
					\begin{overpic}[height=3cm]{Case1.4.2b.eps} 
					\end{overpic}
					
					Case~$4.2b$.
				\end{center}
			\end{minipage}
			\begin{minipage}{3.1cm}
				\begin{center}
					\begin{overpic}[height=3cm]{Case1.4.4a1.eps} 
					\end{overpic}
					
					Case~$4.4a.i$.
				\end{center}
			\end{minipage}
			\begin{minipage}{3.1cm}
				\begin{center}
					\begin{overpic}[height=3cm]{Case1.4.4a2.eps} 
					\end{overpic}
					
					Case~$4.4a.ii$.
				\end{center}
			\end{minipage}
			\begin{minipage}{3.1cm}
				\begin{center}
					\begin{overpic}[height=3cm]{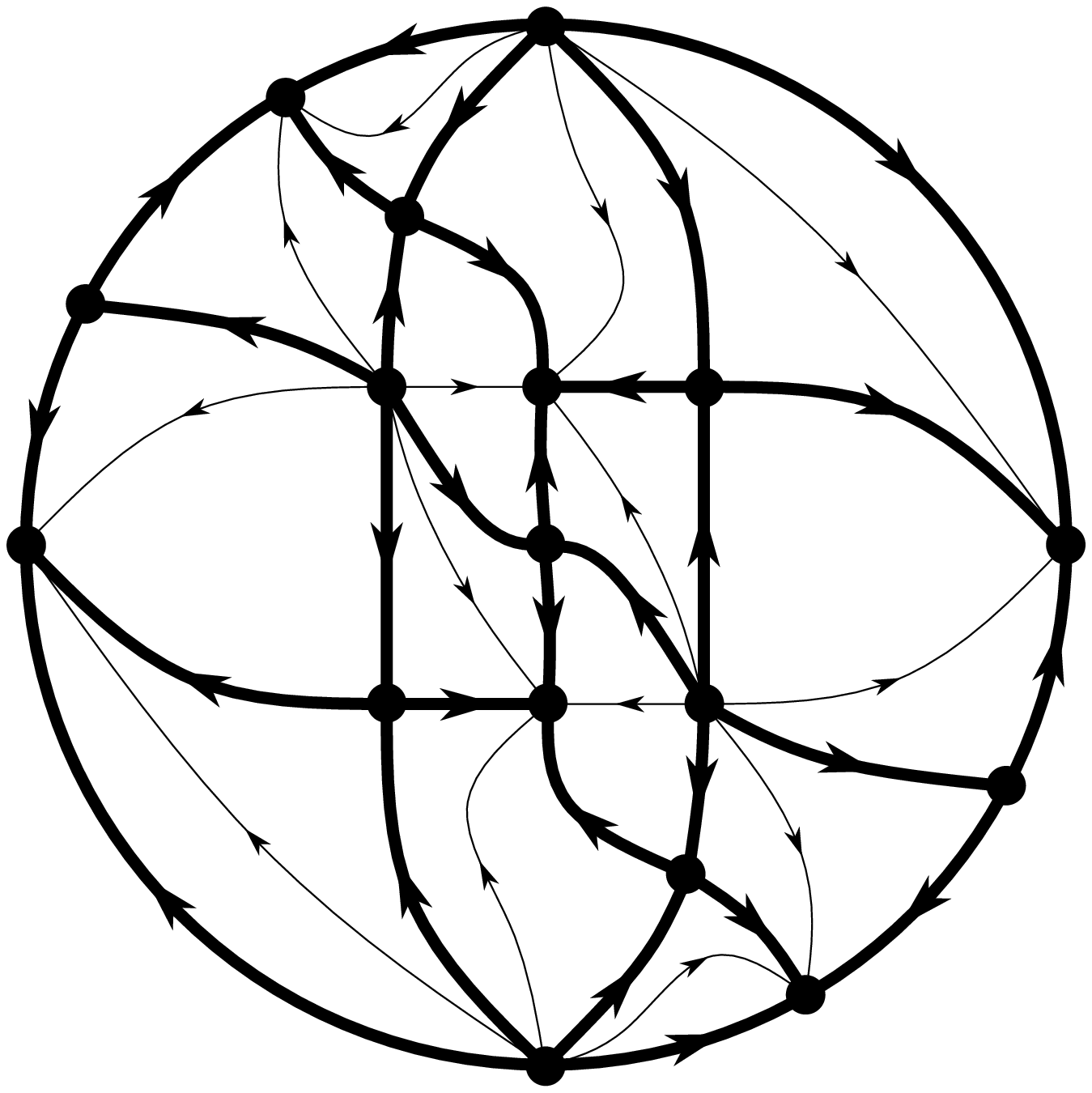} 
					\end{overpic}
					
					Case~$4.4b.i.1$.
				\end{center}
			\end{minipage}
			\begin{minipage}{3.1cm}
				\begin{center}
					\begin{overpic}[height=3cm]{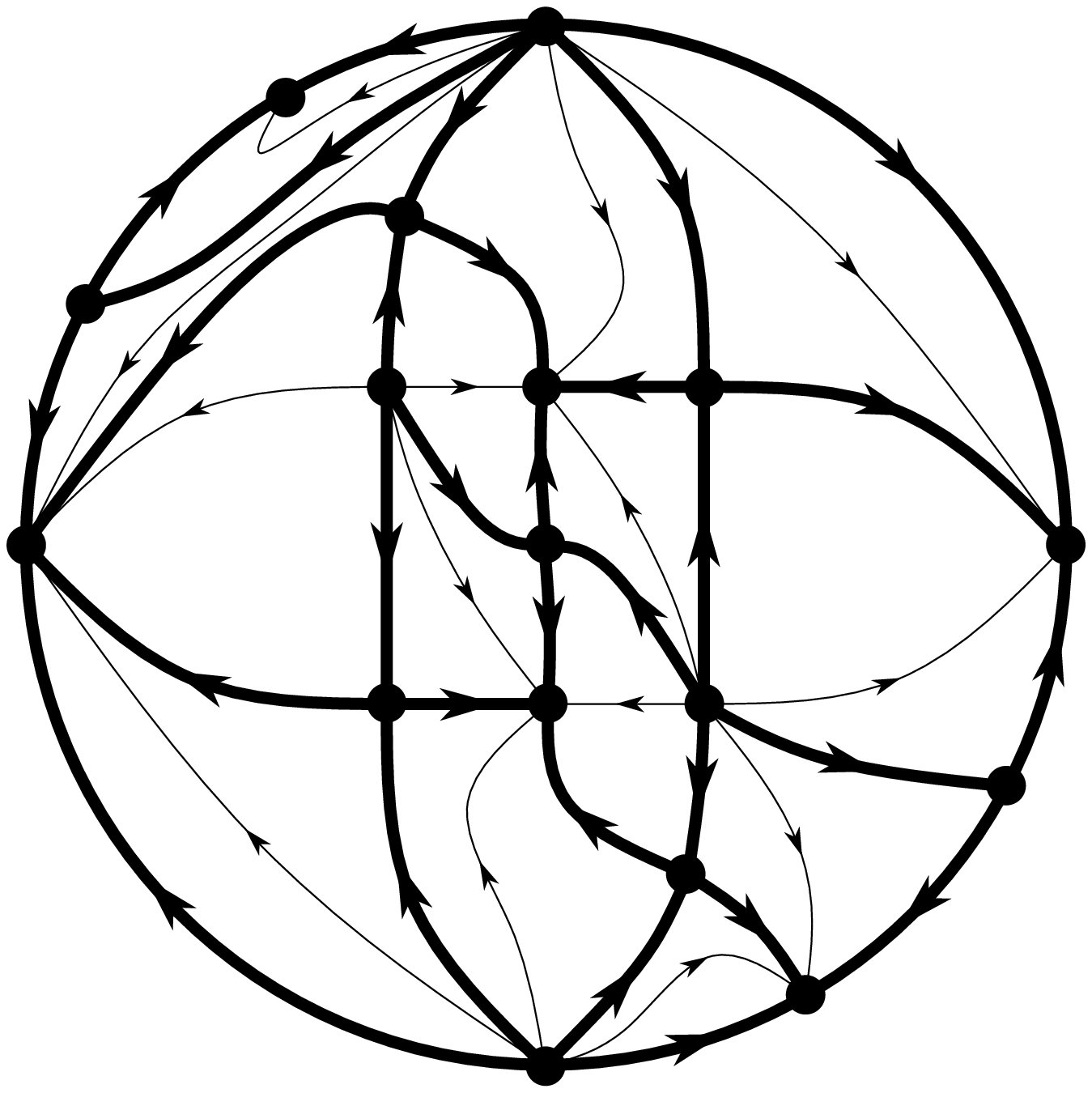} 
					\end{overpic}
					
					Case~$4.4b.i.3$.
				\end{center}
			\end{minipage}	
		\end{center}
		$\;$
		\begin{center}
			\begin{minipage}{3.1cm}
				\begin{center}
					\begin{overpic}[height=3cm]{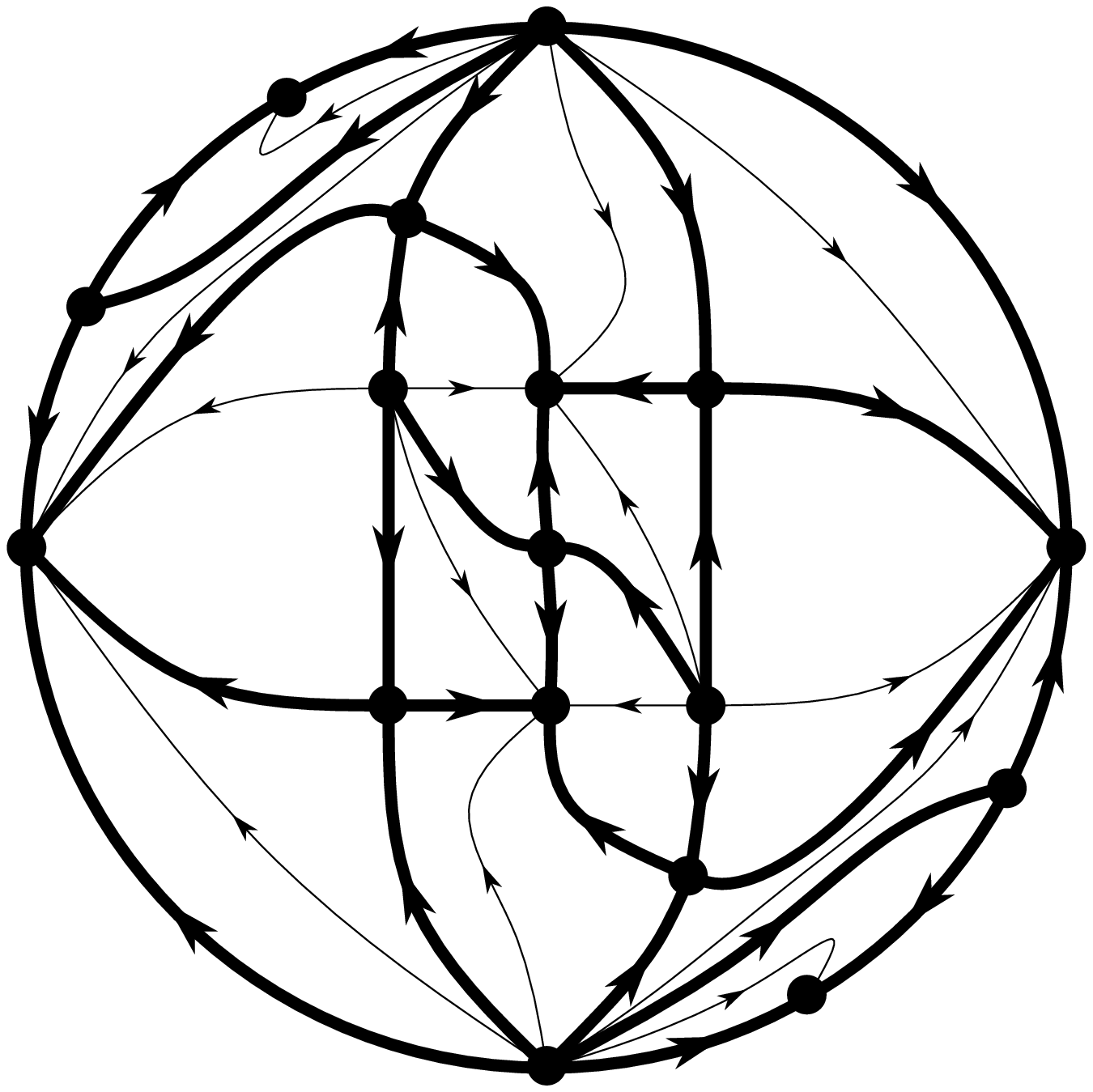} 
					\end{overpic}
					
					Case~$4.4b.i.4$.
				\end{center}
			\end{minipage}	
			\begin{minipage}{3.1cm}
				\begin{center}
					\begin{overpic}[height=3cm]{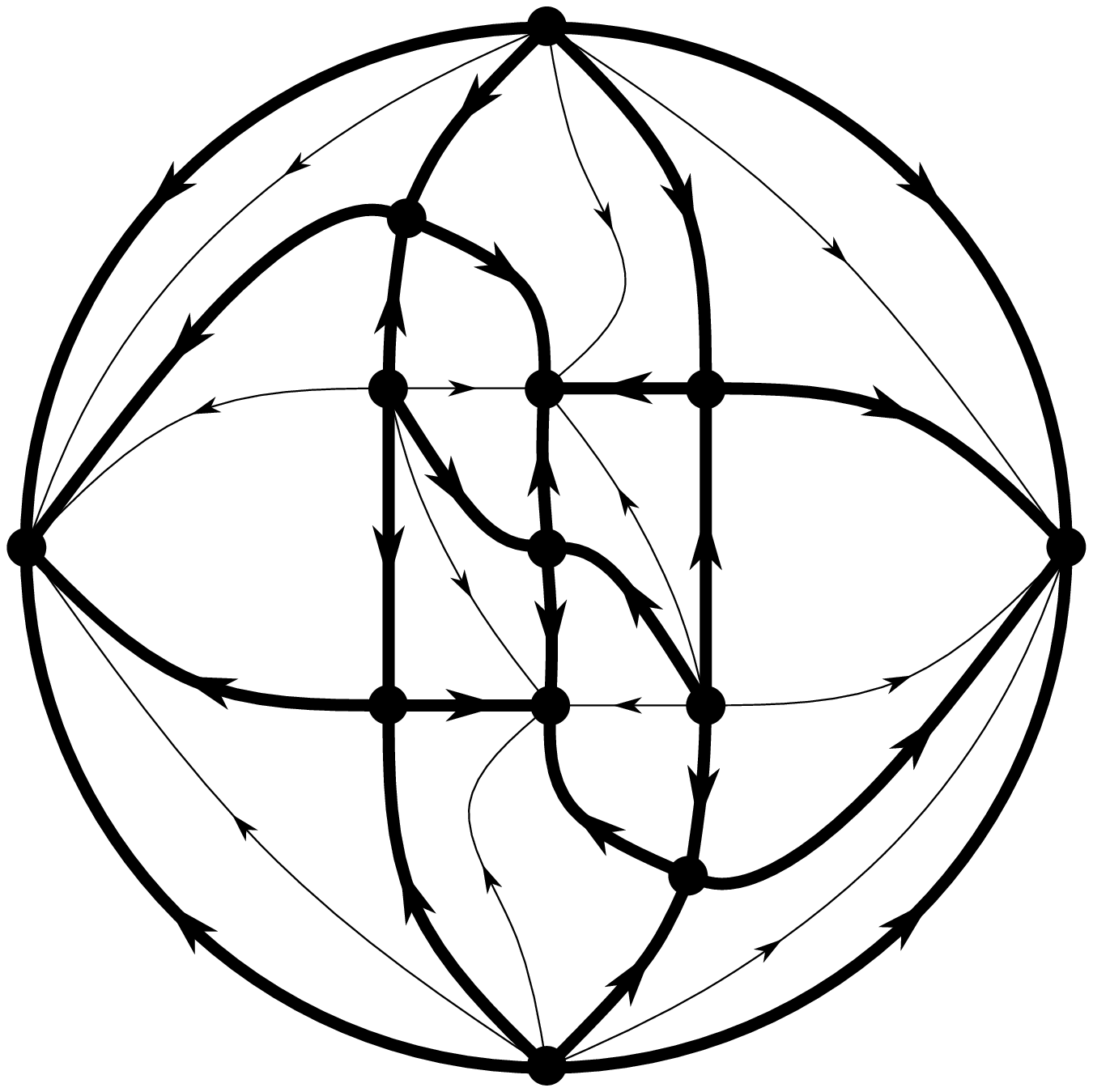} 
					\end{overpic}
					
					Case~$4.4b.iii$.
				\end{center}
			\end{minipage}	
			\begin{minipage}{3.1cm}
				\begin{center}
					\begin{overpic}[height=3cm]{Case1.4.6a1.1.eps} 
					\end{overpic}
					
					Case~$4.6a.i.1$.
				\end{center}
			\end{minipage}
			\begin{minipage}{3.1cm}
				\begin{center}
					\begin{overpic}[height=3cm]{Case1.4.6a2.1.eps} 
					\end{overpic}
					
					Case~$4.6a.ii.1$.
				\end{center}
			\end{minipage}
		\end{center}
		\caption{Topologically distinct phase portraits of Family~$4$}\label{Case1.4Final}
	\end{figure}
\end{corollary}

\section{Proof of Theorem~\ref{Main2} and Corollary~\ref{Main3}}\label{Sec10}

\noindent \textit{Proof of Theorem~\ref{Main2}.} It follows from Corollaries~\ref{Coro3}, \ref{Coro4}, \ref{Coro5} and \ref{Coro6} that there are thirty two possible phase portraits. However, observe that some phase portraits of distinct families are topologically equivalent. More precisely, observe that: 
\begin{enumerate}[label=\arabic*)]
	\item Cases $1.1b$ and $3.4$ are topologically equivalent;
	\item Cases $1.7$ and $3.7$ are topologically equivalent;
	\item Cases $1.14$ and $3.15$ are topologically equivalent;
	\item Cases $2.1a1$ and $4.4b.i.1$ are topologically equivalent;
	\item Cases $2.1a3$ and $4.4b.i.3$ are topologically equivalent;
	\item Cases $2.1a4$ and $4.4b.i.4$ are topologically equivalent;
	\item Cases $2.1c$ and $4.4b.iii$ are topologically equivalent.
\end{enumerate}
Hence, we have the twenty five phase portraits given by Figure~\ref{GenericFinal}. We now prove that each of these twenty five phase portraits is realizable. Except by Cases $2.1a1$, $2.1a3$ and $2.1a4$, it follows from Section~\ref{Sec9} that each phase portrait given by Figure~\ref{GenericFinal} is realizable. Cases $2.1a1$, $2.1a3$ and $2.1a4$ need a special analysis due to its multiple options for the separatrices. See Figure~\ref{Fig9}.
\begin{figure}[h]
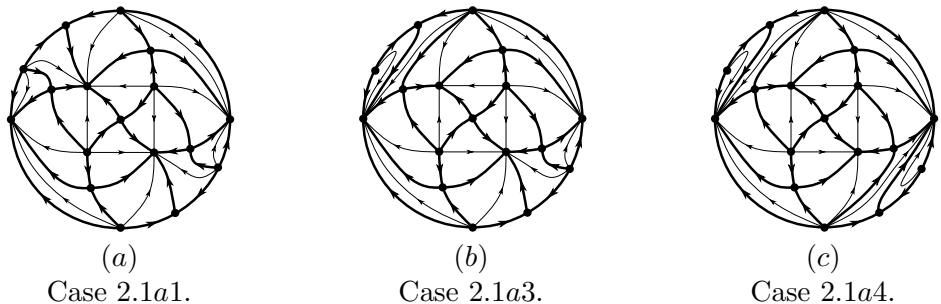

	\begin{center}
		\begin{minipage}{4.5cm}
			\begin{center}
				\begin{overpic}[height=3cm]{Case1.2.1a1.eps} 
				\end{overpic}
			
				$(a)$
				
				Case~$2.1a1$.
			\end{center}
		\end{minipage}
		\begin{minipage}{4.5cm}
			\begin{center}
				\begin{overpic}[height=3cm]{Case1.2.1a3.eps} 
				\end{overpic}
			
				$(b)$
				
				Case~$2.1a3$.
			\end{center}
		\end{minipage}
		\begin{minipage}{4.5cm}
			\begin{center}
				\begin{overpic}[height=3cm]{Case1.2.1a4.eps} 
				\end{overpic}
			
				$(c)$
				
				Case~$2.1a4$.
			\end{center}
		\end{minipage}	
	\end{center}
	\caption{Phase portraits with multiple options for the separatrices.}\label{Fig9}
\end{figure}
More precisely, each phase portrait given by Figure~\ref{Fig9} have the same disposal between the $p$ and $q$-singularities, which is precisely the disposal given by Case $2.1a$. Case $2.1a$ is the unique case at Figure~\ref{GenericFinal} with such property. Hence, it needs a special analysis. We now prove that each phase portraits given by Figure~\ref{Fig9} is realizable. Consider $X\in\Sigma_0$ with $\alpha=\frac{1}{2}$, $\beta=\frac{1}{2}$, $a_{01}=5$, $b_{10}=1$, and $b_{01}=-\frac{1}{2}$. It is clear that $\lambda=a_{10}$ is the only free parameter. We now study the one-parameter family $X_\lambda$, with $0\leqslant\lambda\leqslant 5$. We claim that there exists $\varepsilon>0$ such that if $0<\lambda<\varepsilon$, then the phase portrait of $X_\lambda$ is topologically equivalent to Figure~\ref{Fig9}(a). To prove this, first observe that if $0<\lambda<5$, then the uncompleted phase portrait of $X_\lambda$ is given by Figure~\ref{Fig10}(a).
\begin{figure}[h]
	\begin{center}
		\begin{minipage}{5cm}
			\begin{center}
				\begin{overpic}[height=3cm]{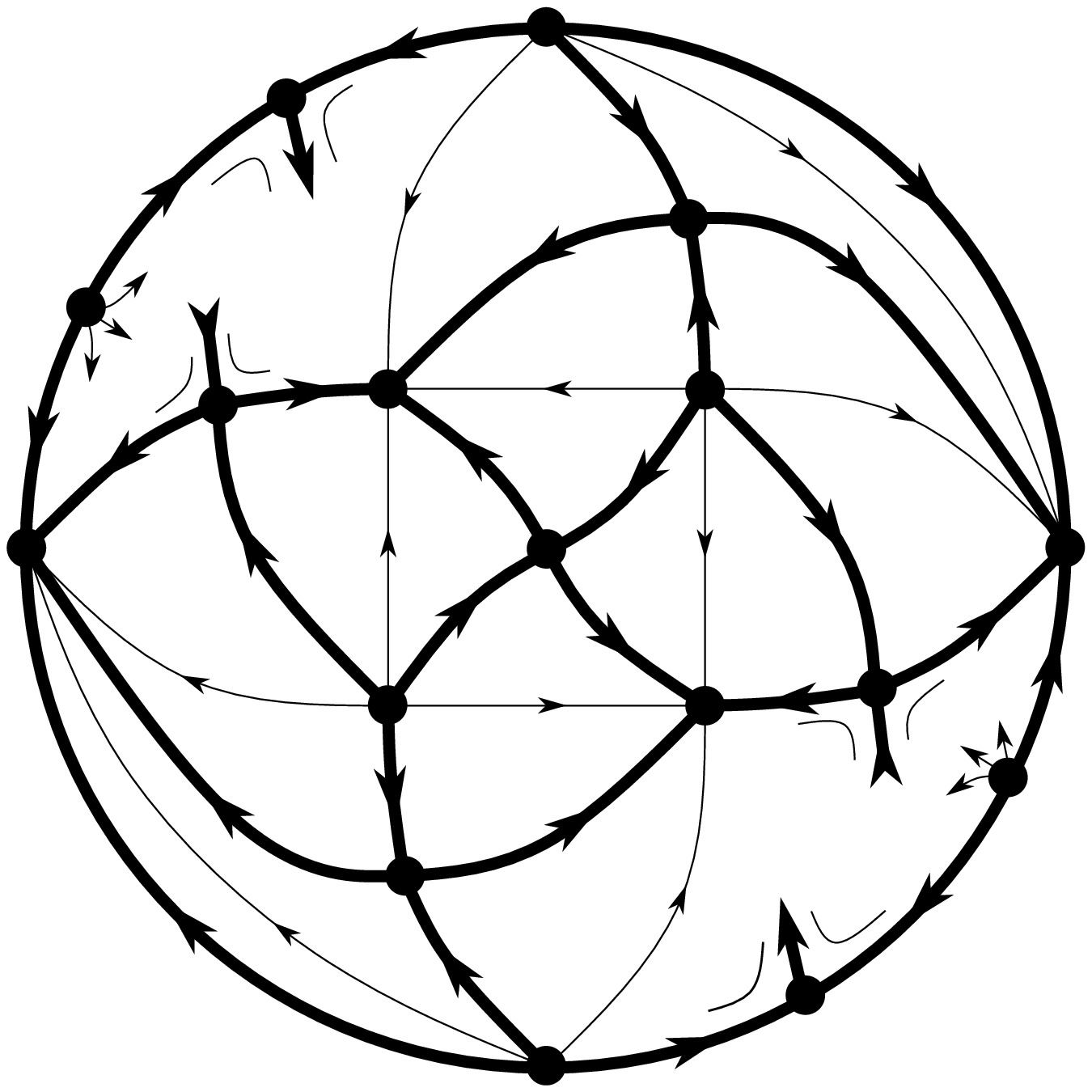} 
				\end{overpic}
				
				$(a)$
				
				$0<\lambda<5$.
			\end{center}
		\end{minipage}
		\begin{minipage}{5cm}
			\begin{center}
				\begin{overpic}[height=3cm]{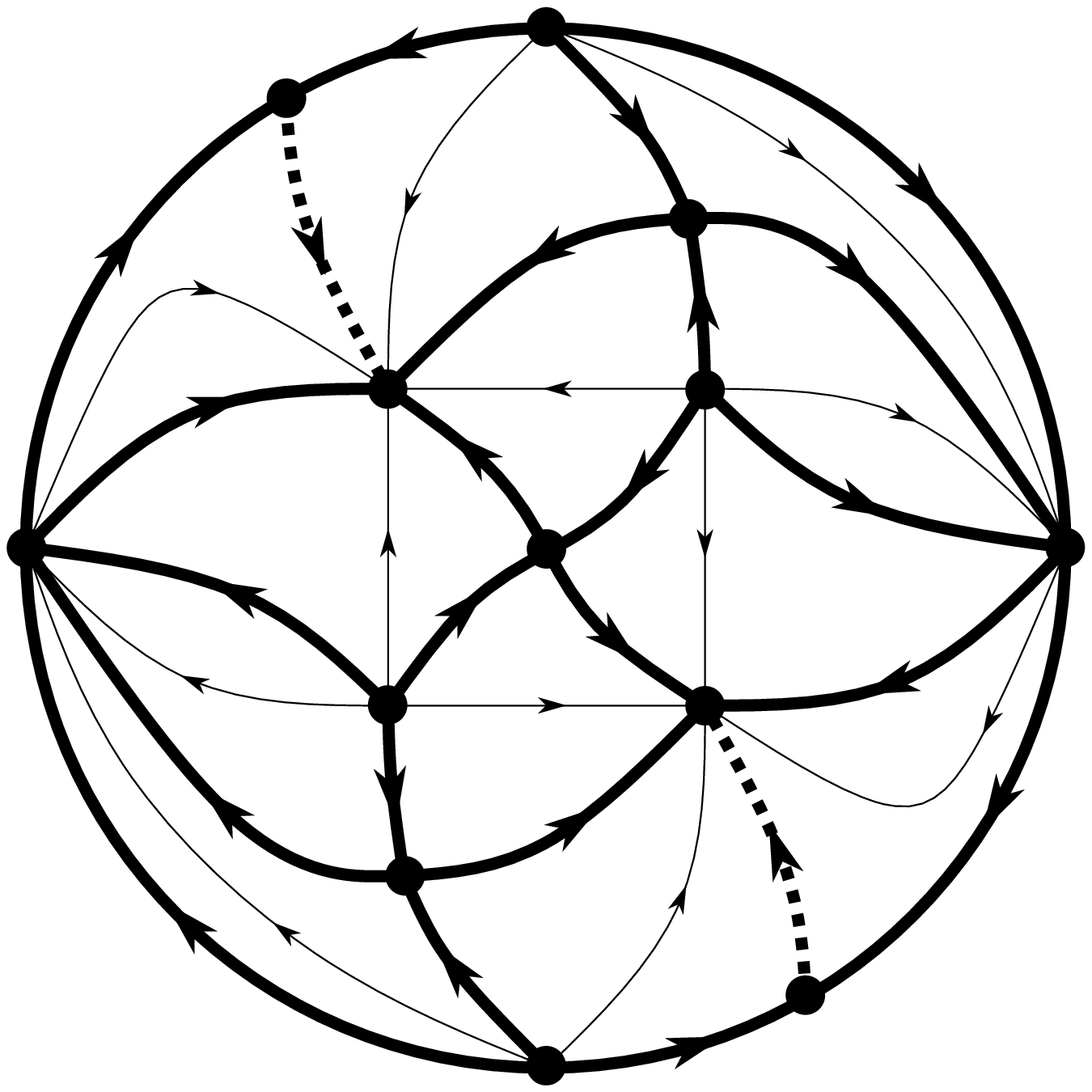} 
						\put(23,56){$p_4$}
						\put(67,40){$p_2$}
				\end{overpic}
				
				$(b)$
				
				$\lambda=0$.
			\end{center}
		\end{minipage}
		\begin{minipage}{5cm}
			\begin{center}
				\begin{overpic}[height=3cm]{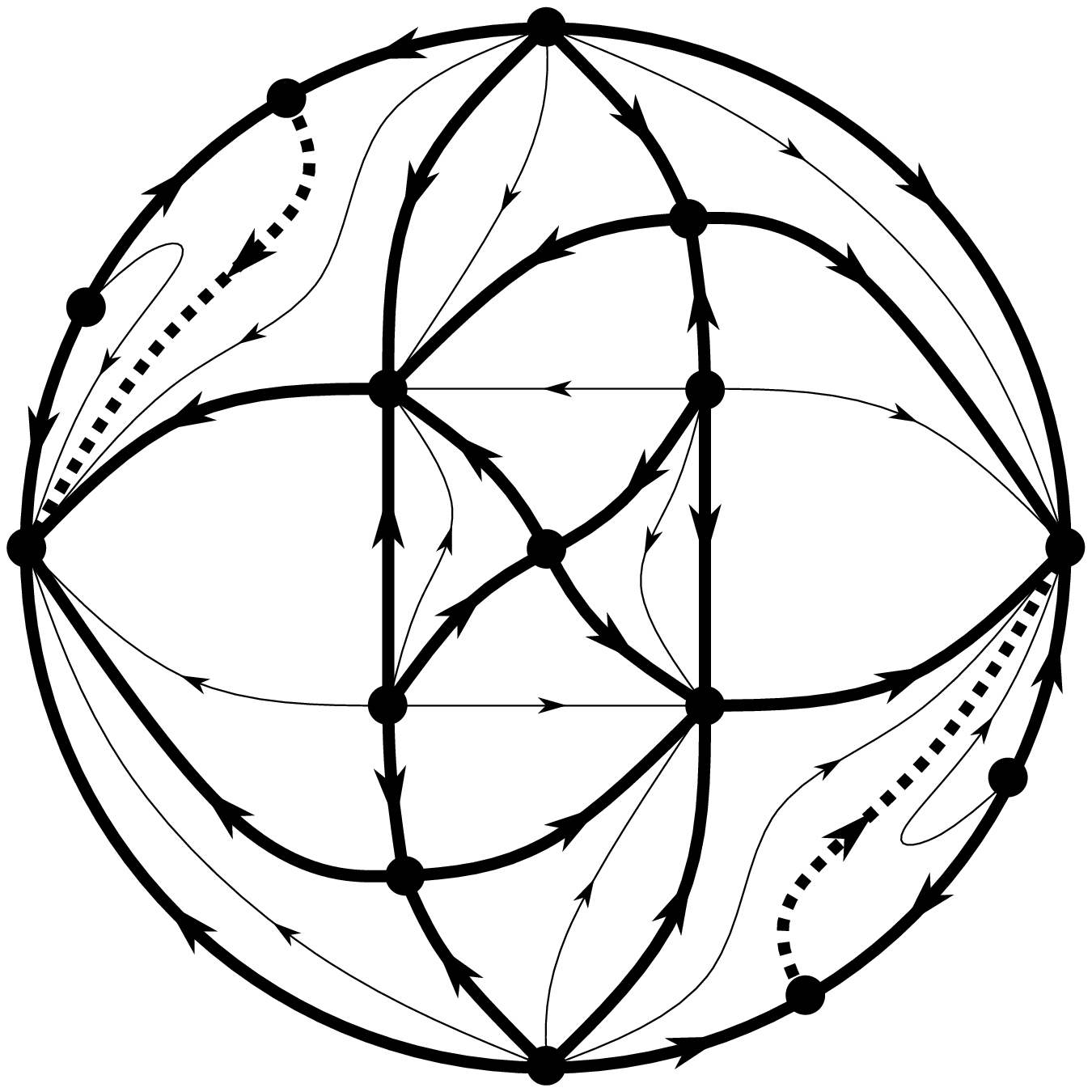} 
				\end{overpic}
				
				$(c)$
				
				$\lambda=5$.
			\end{center}
		\end{minipage}	
	\end{center}
	\caption{Phase portrait of $X_\lambda$.}\label{Fig10}
\end{figure}
Observe that if $\lambda=0$, then the phase portrait of $X_\lambda$ is given by Figure~\ref{Fig10}(b). In special, at Figure~\ref{Fig10}(b), consider the two separatrices (given by the dotted lines) between the hyperbolic saddles at infinity and the hyperbolic nodes $p_4$ and $p_2$. It is well known in the literature that a connection between a hyperbolic saddle and a hyperbolic node/focus is structurally stable. See Lemma $6$ of \cite{PeiPei}. Hence, for $\lambda>0$ small enough, it follows that such connections holds. Thus, we obtain the phase portrait given by Figure~\ref{Fig9}(a). Similarly, one can prove that if $5-\varepsilon<\lambda<5$, then the phase portrait of $X_\lambda$ is topologically equivalent to Figure~\ref{Fig9}(c). See Figure~\ref{Fig10}(c). Consider $X\in\Sigma_0$ with $\alpha=\frac{1}{2}$, $a_{01}=5$, $b_{10}=1$, and $b_{01}=-\frac{1}{2}$. Denote $\lambda=a_{10}$ and consider the two-parameter family $X_{\lambda,\beta}$. With a similar argumentation, observe that there exists $\varepsilon>0$ such that if $1-\varepsilon<\beta<1$ and $0<\lambda<\varepsilon$, then the phase portrait of $X_{\lambda,\beta}$ is topologically equivalent to Figure~\ref{Fig9}(b). Hence, all the phase portraits are realizable. {\hfill$\square$}

\noindent \textit{Proof of Corollary~\ref{Main3}.} Sixteen of the twenty phase portraits presented at Figure~\ref{FinalSquares} follows from Figure~\ref{GenericFinal}. The other four, namely, Cases $2.6a1$, $2.7a$, $2.15a$ and $4.6b.i.1$, follows from engineering back some of the techniques applied at Remarks~\ref{Remark5} and \ref{Remark6}, which changed the phase portrait at the square, without changing the phase portrait at the Poincar\'e Disk. {\hfill$\square$}

\section{Conclusion and Further Directions}

In this work we deal with the family of planar cubic vector fields with an invariant octothorpe, which models the Evolutionary Stable Strategies given by two players with two pure strategies each. Following the directions given by Peixoto's Theorem, we suggest the notion of genericity and provide the global phase portraits of such systems with an unique singularity in the central region of the octothorpe. We also provide an algebraic characterization for the bifurcation of limit cycles from a polycycle that has influence in the model.

In a future paper, the authors pushed on the phase portraits at the center of the octothorpe, aiming to classify all the realizable phase portraits on such square, including the non-generic ones, obtaining a complete qualitative characterization of the model. Regarding a more qualitative further direction, it is natural to wonder if the suggested family of generic vector fields is open and dense in the space considered in this paper. It is also natural to wonder if such family is equal to the set of the structurally stable vector fields. Such questions are already being considered by the authors. Finally, an approach to the family of codimension one vector fields is also considered.

\section*{Acknowledgments}

The authors are grateful to professor Armengol Gasull for sharing the idea of applying the tools of Dynamic Systems to study the models of Evolutionary Stable Strategies. 

\section*{Declarations}

\noindent\textbf{Ethical Approval:} Not applicable.

\noindent\textbf{Competing interests:} Not applicable.

\noindent\textbf{Authors' contributions:} All authors contributed equally to this work.

\noindent\textbf{Funding:} The authors are partially supported by CNPq, grant 304798/2019-3 and by S\~ao Paulo Research Foundation (FAPESP), grants 2019/10269-3 and 2021/01799-9.

\noindent\textbf{Availability of data and materials:} Not applicable.

\clearpage

\appendix

\section{Phase portraits under the hypothesis of Theorem~\ref{Theo11} (Family 2)}\label{Ap1}

To simplify Table~\ref{Table3}, we denote $\Delta=(b_{10}-a_{01})^2+4b_{01}$.

\begin{table}[h]
	\caption{Table of the realizable cases under hypothesis of Theorem~\ref{Theo11} (Family 2).}\label{Table3}
	\begin{tabular}{c c c c c c c}
		\hline
		Cases & $q_1$ & $q_2$ & $q_3$ & $q_4$ & $\Delta$ & $b_{10}-a_{01}$ \\
		\hline
		\rowcolor{mygray}
		$2.1a$ & $q_1>p_2$ & $q_2>p_3$ & $q_3<p_4$ & $q_4<p_1$ & $\Delta>0$ & $b_{10}-a_{01}<0$ \\
		$2.1b$ & $q_1>p_2$ & $q_2>p_3$ & $q_3<p_4$ & $q_4<p_1$ & $\Delta>0$ & $b_{10}-a_{01}>0$ \\
		\rowcolor{mygray}
		$2.1c$ & $q_1>p_2$ & $q_2>p_3$ & $q_3<p_4$ & $q_4<p_1$ & $\Delta<0$ & \\		
		$2.2a$ & $q_1<p_2$ & $q_2>p_3$ & $q_3<p_4$ & $q_4<p_1$ & $\Delta>0$ & $b_{10}-a_{01}<0$ \\
		\rowcolor{mygray}
		$2.2b$ & $q_1<p_2$ & $q_2>p_3$ & $q_3<p_4$ & $q_4<p_1$ & $\Delta>0$ & $b_{10}-a_{01}>0$ \\
		$2.2c$ & $q_1<p_2$ & $q_2>p_3$ & $q_3<p_4$ & $q_4<p_1$ & $\Delta<0$ & \\		
		\rowcolor{mygray}
		$2.3a$ & $q_1>p_2$ & $q_2>p_3$ & $q_3>p_4$ & $q_4<p_1$ & $\Delta>0$ & $b_{10}-a_{01}<0$ \\
		$2.3b$ & $q_1>p_2$ & $q_2>p_3$ & $q_3>p_4$ & $q_4<p_1$ & $\Delta>0$ & $b_{10}-a_{01}>0$ \\
		\rowcolor{mygray}
		$2.3c$ & $q_1>p_2$ & $q_2>p_3$ & $q_3>p_4$ & $q_4<p_1$ & $\Delta<0$ & \\		
		$2.4a$ & $q_1<p_2$ & $q_2>p_3$ & $q_3>p_4$ & $q_4<p_1$ & $\Delta>0$ & $b_{10}-a_{01}<0$ \\
		\rowcolor{mygray}
		$2.4b$ & $q_1<p_2$ & $q_2>p_3$ & $q_3>p_4$ & $q_4<p_1$ & $\Delta>0$ & $b_{10}-a_{01}>0$ \\
		$2.4c$ & $q_1<p_2$ & $q_2>p_3$ & $q_3>p_4$ & $q_4<p_1$ & $\Delta<0$ & \\		
		\rowcolor{mygray}
		$2.5a$ & $q_1>p_2$ & $q_2<p_3$ & $q_3<p_4$ & $q_4<p_1$ & $\Delta>0$ & $b_{10}-a_{01}<0$ \\
		$2.5b$ & $q_1>p_2$ & $q_2<p_3$ & $q_3<p_4$ & $q_4<p_1$ & $\Delta>0$ & $b_{10}-a_{01}>0$ \\
		\rowcolor{mygray}
		$2.5c$ & $q_1>p_2$ & $q_2<p_3$ & $q_3<p_4$ & $q_4<p_1$ & $\Delta<0$ & \\		
		$2.6a$ & $q_1>p_2$ & $q_2<p_3$ & $q_3>p_4$ & $q_4<p_1$ & $\Delta>0$ & $b_{10}-a_{01}<0$ \\
		\rowcolor{mygray}
		$2.6b$ & $q_1>p_2$ & $q_2<p_3$ & $q_3>p_4$ & $q_4<p_1$ & $\Delta>0$ & $b_{10}-a_{01}>0$ \\
		$2.6c$ & $q_1>p_2$ & $q_2<p_3$ & $q_3>p_4$ & $q_4<p_1$ & $\Delta<0$ & \\		
		\rowcolor{mygray}
		$2.7a$ & $q_1<p_2$ & $q_2<p_3$ &  $q_3>p_4$ & $q_4<p_1$ & $\Delta>0$ & $b_{10}-a_{01}<0$ \\
		$2.7b$ & $q_1<p_2$ & $q_2<p_3$ &  $q_3>p_4$ & $q_4<p_1$ & $\Delta>0$ & $b_{10}-a_{01}>0$ \\
		\rowcolor{mygray}
		$2.7c$ & $q_1<p_2$ & $q_2<p_3$ &  $q_3>p_4$ & $q_4<p_1$ & $\Delta<0$ & \\		
		$2.8a$ & $q_1>p_2$ & $q_2>p_3$ & $q_3<p_4$ & $q_4>p_1$ & $\Delta>0$ & $b_{10}-a_{01}<0$ \\
		\rowcolor{mygray}
		$2.8b$ & $q_1>p_2$ & $q_2>p_3$ & $q_3<p_4$ & $q_4>p_1$ & $\Delta>0$ & $b_{10}-a_{01}>0$ \\
		$2.8c$ & $q_1>p_2$ & $q_2>p_3$ & $q_3<p_4$ & $q_4>p_1$ & $\Delta<0$ & \\		
		\rowcolor{mygray}
		$2.9a$ & $q_1<p_2$ & $q_2>p_3$ & $q_3<p_4$ & $q_4>p_1$ & $\Delta>0$ & $b_{10}-a_{01}<0$ \\
		$2.9b$ & $q_1<p_2$ & $q_2>p_3$ & $q_3<p_4$ & $q_4>p_1$ & $\Delta>0$ & $b_{10}-a_{01}>0$ \\
		\rowcolor{mygray}
		$2.9c$ & $q_1<p_2$ & $q_2>p_3$ & $q_3<p_4$ & $q_4>p_1$ & $\Delta<0$ & \\		
		$2.10a$ & $q_1>p_2$ & $q_2>p_3$ & $q_3>p_4$ & $q_4>p_1$ & $\Delta>0$ & $b_{10}-a_{01}<0$ \\
		\rowcolor{mygray}
		$2.10b$ & $q_1>p_2$ & $q_2>p_3$ & $q_3>p_4$ & $q_4>p_1$ & $\Delta>0$ & $b_{10}-a_{01}>0$ \\
		$2.10c$ & $q_1>p_2$ & $q_2>p_3$ & $q_3>p_4$ & $q_4>p_1$ & $\Delta<0$ & \\
		\rowcolor{mygray}
		$2.11a$ & $q_1<p_2$ & $q_2>p_3$ & $q_3>p_4$ & $q_4>p_1$ & $\Delta>0$ & $b_{10}-a_{01}<0$ \\
		$2.11b$ & $q_1<p_2$ & $q_2>p_3$ & $q_3>p_4$ & $q_4>p_1$ & $\Delta>0$ & $b_{10}-a_{01}>0$ \\
		\rowcolor{mygray}
		$2.11c$ & $q_1<p_2$ & $q_2>p_3$ & $q_3>p_4$ & $q_4>p_1$ & $\Delta<0$ & \\		
		$2.12a$ & $q_1>p_2$ & $q_2<p_3$ & $q_3<p_4$ & $q_4>p_1$ & $\Delta>0$ & $b_{10}-a_{01}<0$ \\
		\rowcolor{mygray}
		$2.12b$ & $q_1>p_2$ & $q_2<p_3$ & $q_3<p_4$ & $q_4>p_1$ & $\Delta>0$ & $b_{10}-a_{01}>0$ \\
		$2.12c$ & $q_1>p_2$ & $q_2<p_3$ & $q_3<p_4$ & $q_4>p_1$ & $\Delta<0$ & \\		
		\rowcolor{mygray}
		$2.13a$ & $q_1<p_2$ & $q_2<p_3$ & $q_3<p_4$ & $q_4>p_1$ & $\Delta>0$ & $b_{10}-a_{01}<0$ \\
		$2.13b$ & $q_1<p_2$ & $q_2<p_3$ & $q_3<p_4$ & $q_4>p_1$ & $\Delta>0$ & $b_{10}-a_{01}>0$ \\
		\rowcolor{mygray}
		$2.13c$ & $q_1<p_2$ & $q_2<p_3$ & $q_3<p_4$ & $q_4>p_1$ & $\Delta<0$ & \\		
		$2.14a$ & $q_1>p_2$ & $q_2<p_3$ & $q_3>p_4$ & $q_4>p_1$ & $\Delta>0$ & $b_{10}-a_{01}<0$ \\
		\rowcolor{mygray}
		$2.14b$ & $q_1>p_2$ & $q_2<p_3$ & $q_3>p_4$ & $q_4>p_1$ & $\Delta>0$ & $b_{10}-a_{01}>0$ \\
		$2.14c$ & $q_1>p_2$ & $q_2<p_3$ & $q_3>p_4$ & $q_4>p_1$ & $\Delta<0$ & \\		
		\rowcolor{mygray}
		$2.15a$ & $q_1<p_2$ & $q_2<p_3$ & $q_3>p_4$ & $q_4>p_1$ & $\Delta>0$ & $b_{10}-a_{01}<0$ \\
		$2.15b$ & $q_1<p_2$ & $q_2<p_3$ & $q_3>p_4$ & $q_4>p_1$ & $\Delta>0$ & $b_{10}-a_{01}>0$ \\
		\rowcolor{mygray}
		$2.15c$ & $q_1<p_2$ & $q_2<p_3$ & $q_3>p_4$ & $q_4>p_1$ & $\Delta<0$ & \\		
		\hline
	\end{tabular}
\end{table}

\begin{figure}[h]
	\begin{center}
		\begin{minipage}{3.1cm}
			\begin{center}
				\begin{overpic}[height=3cm]{Case1.2.1a1.eps} 
				\end{overpic}
				
				Case~$2.1a1$.
			\end{center}
		\end{minipage}
		\begin{minipage}{3.1cm}
			\begin{center}
				\begin{overpic}[height=3cm]{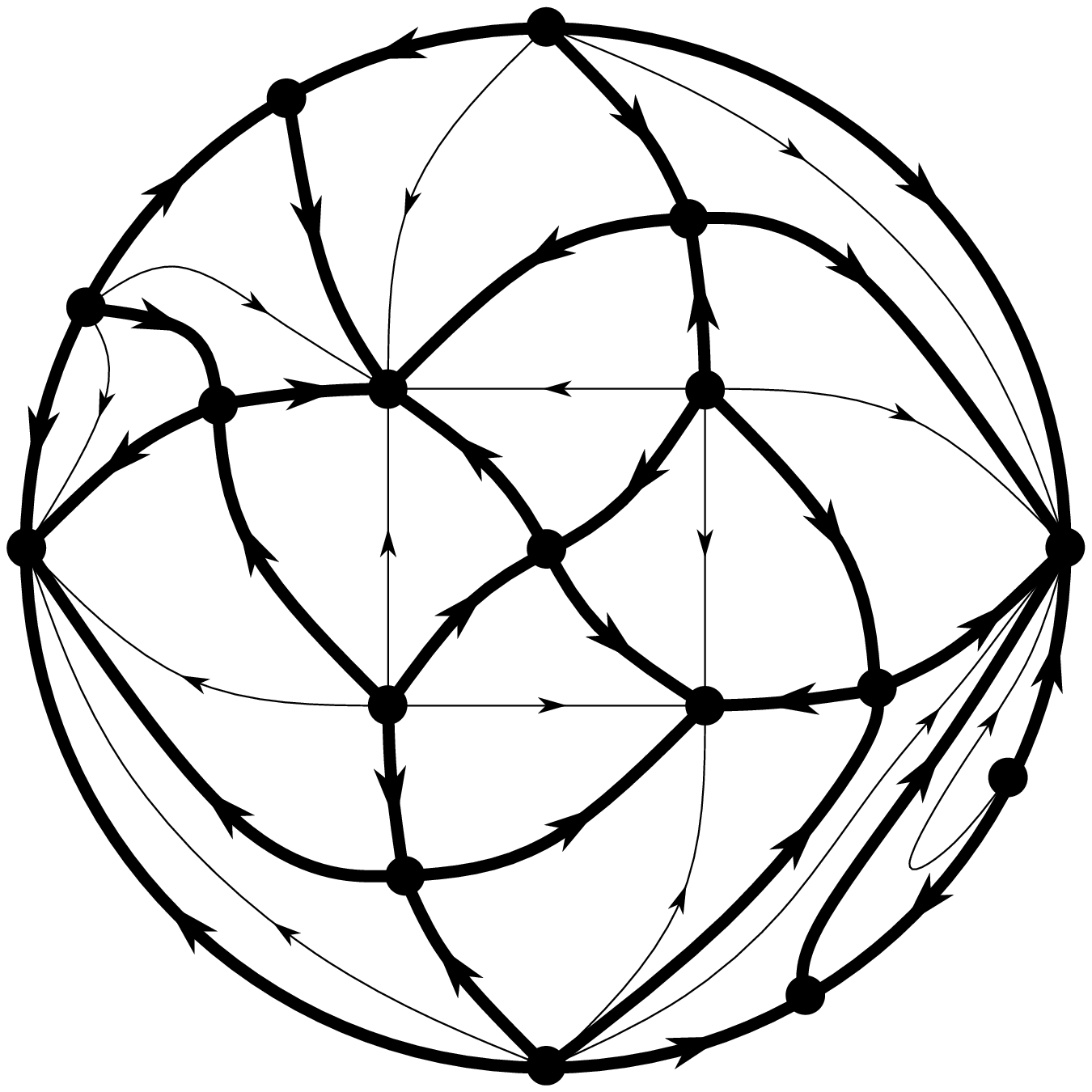} 
				\end{overpic}
				
				Case~$2.1a2$.
			\end{center}
		\end{minipage}
		\begin{minipage}{3.1cm}
			\begin{center}
				\begin{overpic}[height=3cm]{Case1.2.1a3.eps} 
				\end{overpic}
				
				Case~$2.1a3$.
			\end{center}
		\end{minipage}	
		\begin{minipage}{3.1cm}
			\begin{center}
				\begin{overpic}[height=3cm]{Case1.2.1a4.eps} 
				\end{overpic}
				
				Case~$2.1a4$.
			\end{center}
		\end{minipage}
		\begin{minipage}{3.1cm}
			\begin{center}
				\begin{overpic}[height=3cm]{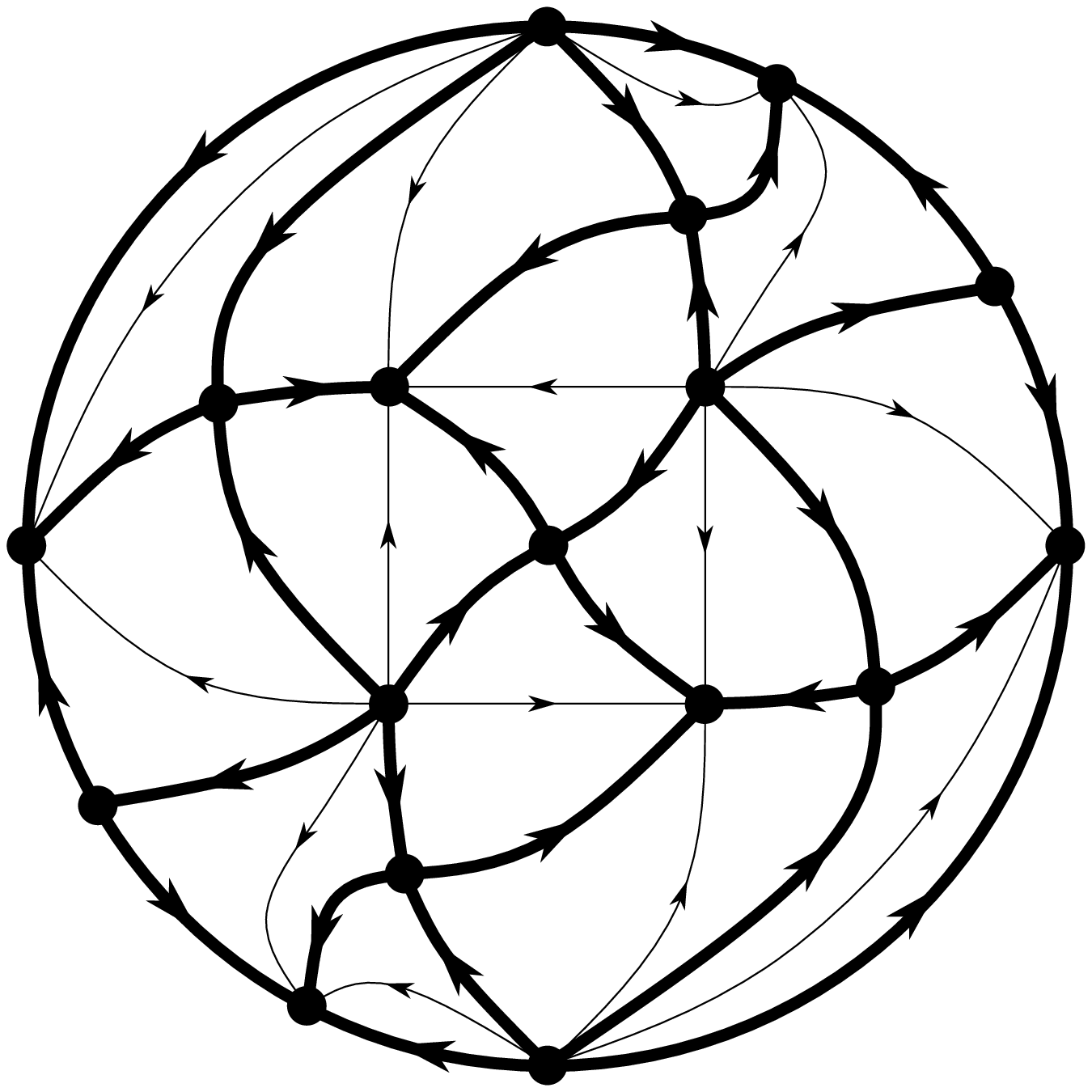} 
				\end{overpic}
				
				Case~$2.1b1$.
			\end{center}
		\end{minipage}
	\end{center}
	$\;$
	\begin{center}
		\begin{minipage}{3.1cm}
			\begin{center}
				\begin{overpic}[height=3cm]{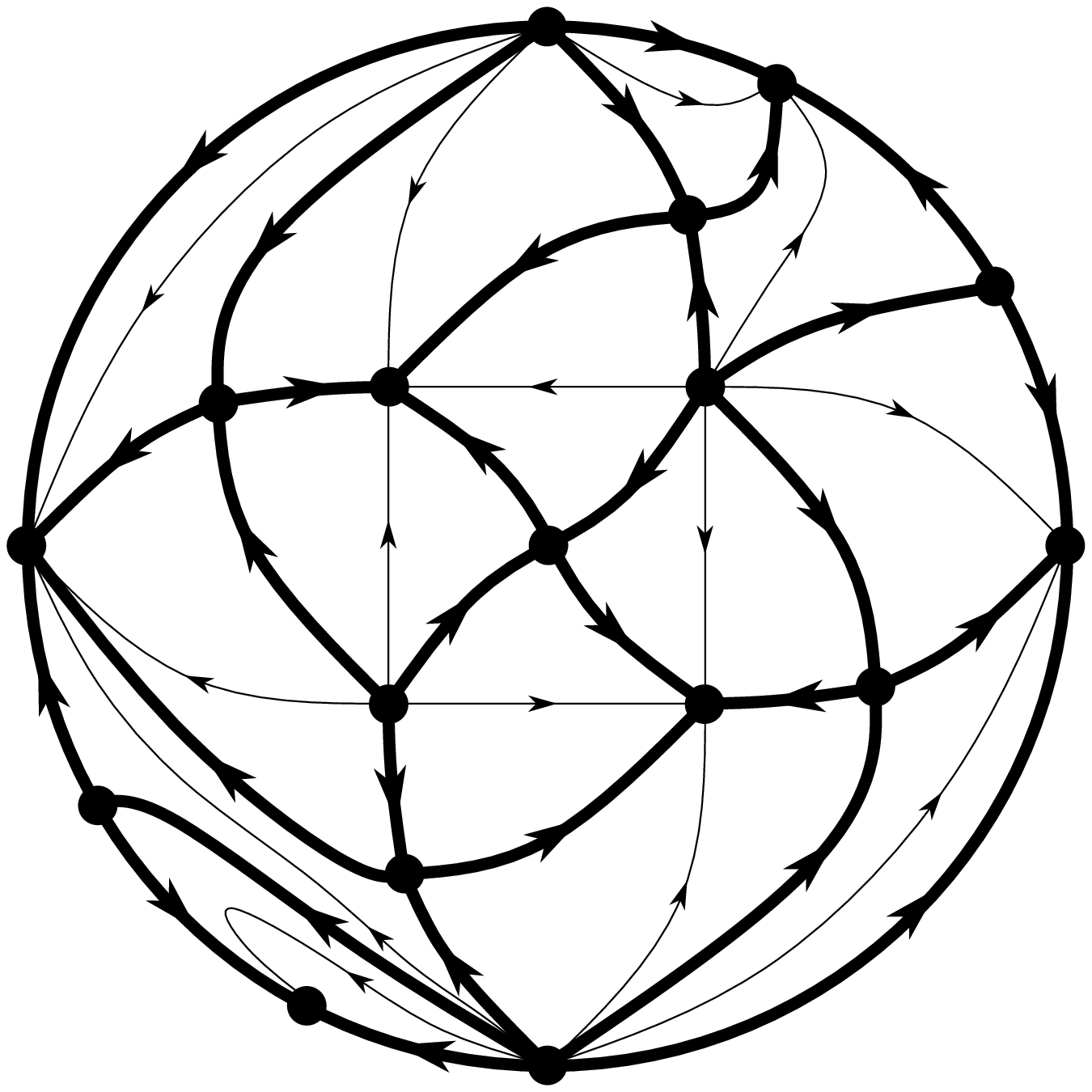} 
				\end{overpic}
				
				Case~$2.1b2$.
			\end{center}
		\end{minipage}
		\begin{minipage}{3.1cm}
			\begin{center}
				\begin{overpic}[height=3cm]{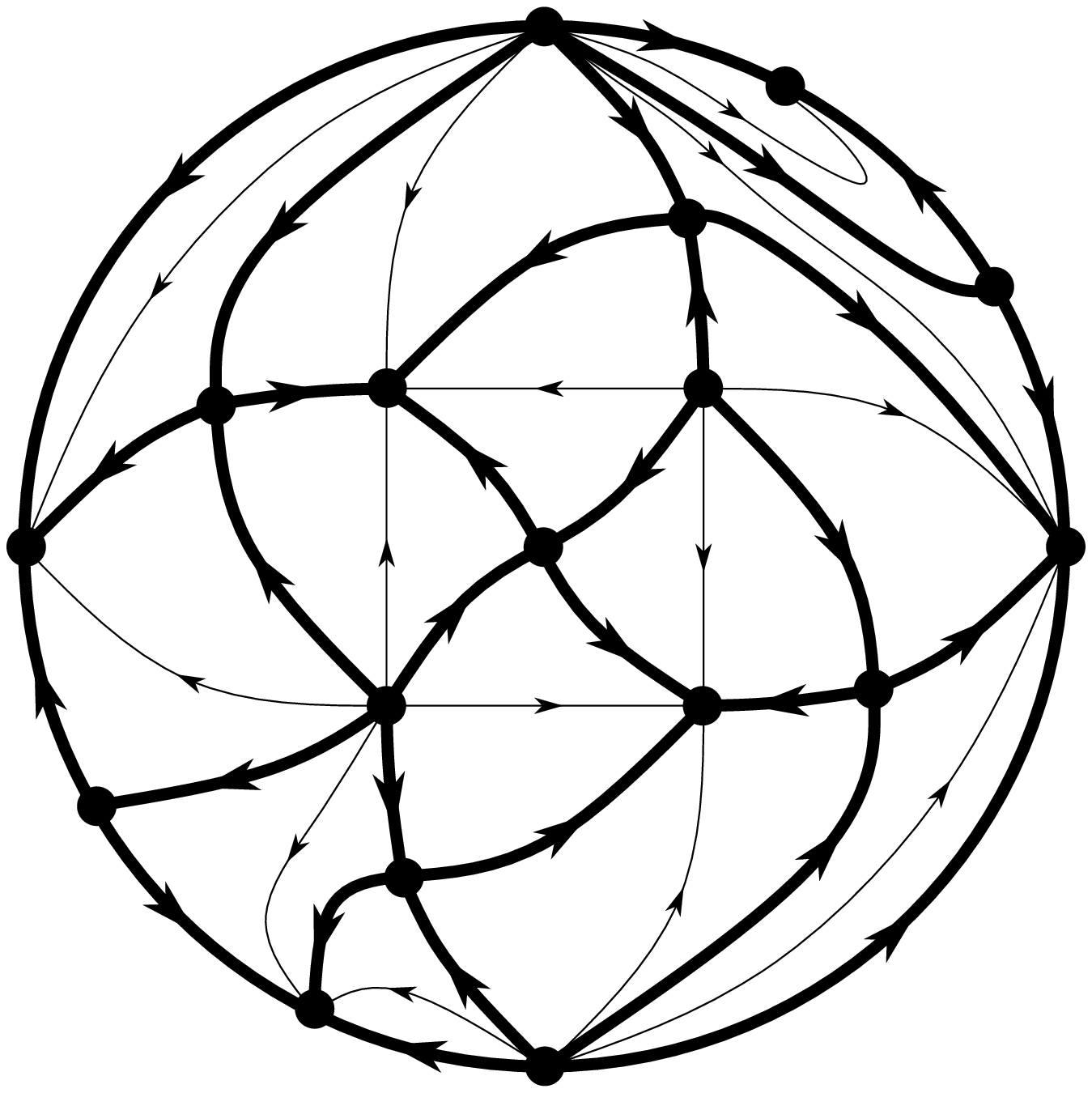} 
				\end{overpic}
				
				Case~$2.1b3$.
			\end{center}
		\end{minipage}
		\begin{minipage}{3.1cm}
			\begin{center}
				\begin{overpic}[height=3cm]{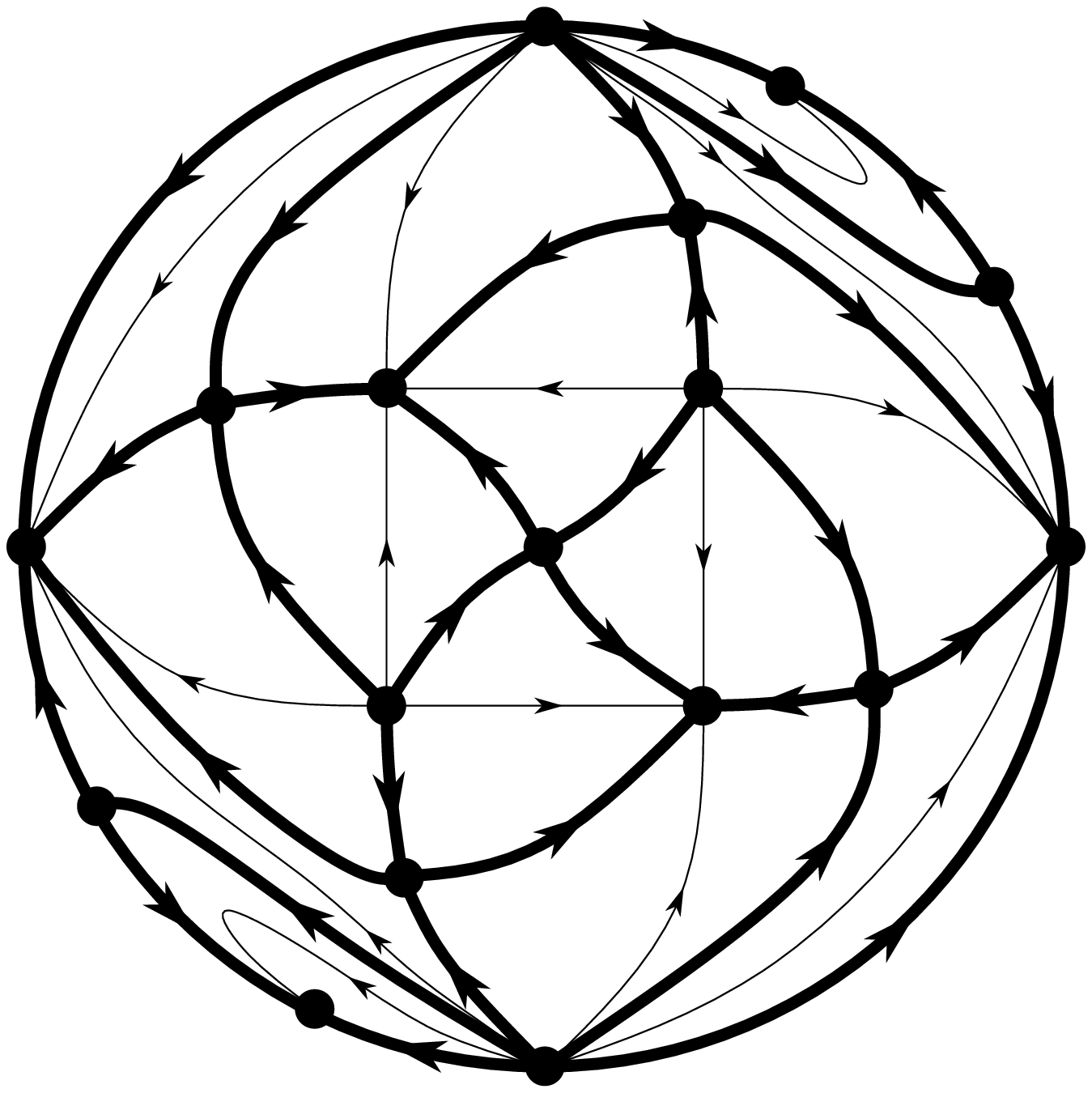} 
				\end{overpic}
				
				Case~$2.1b4$.
			\end{center}
		\end{minipage}
		\begin{minipage}{3.1cm}
			\begin{center}
				\begin{overpic}[height=3cm]{Case1.2.1c.eps} 
				\end{overpic}
				
				Case~$2.1c$.
			\end{center}
		\end{minipage}
		\begin{minipage}{3.1cm}
			\begin{center}
				\begin{overpic}[height=3cm]{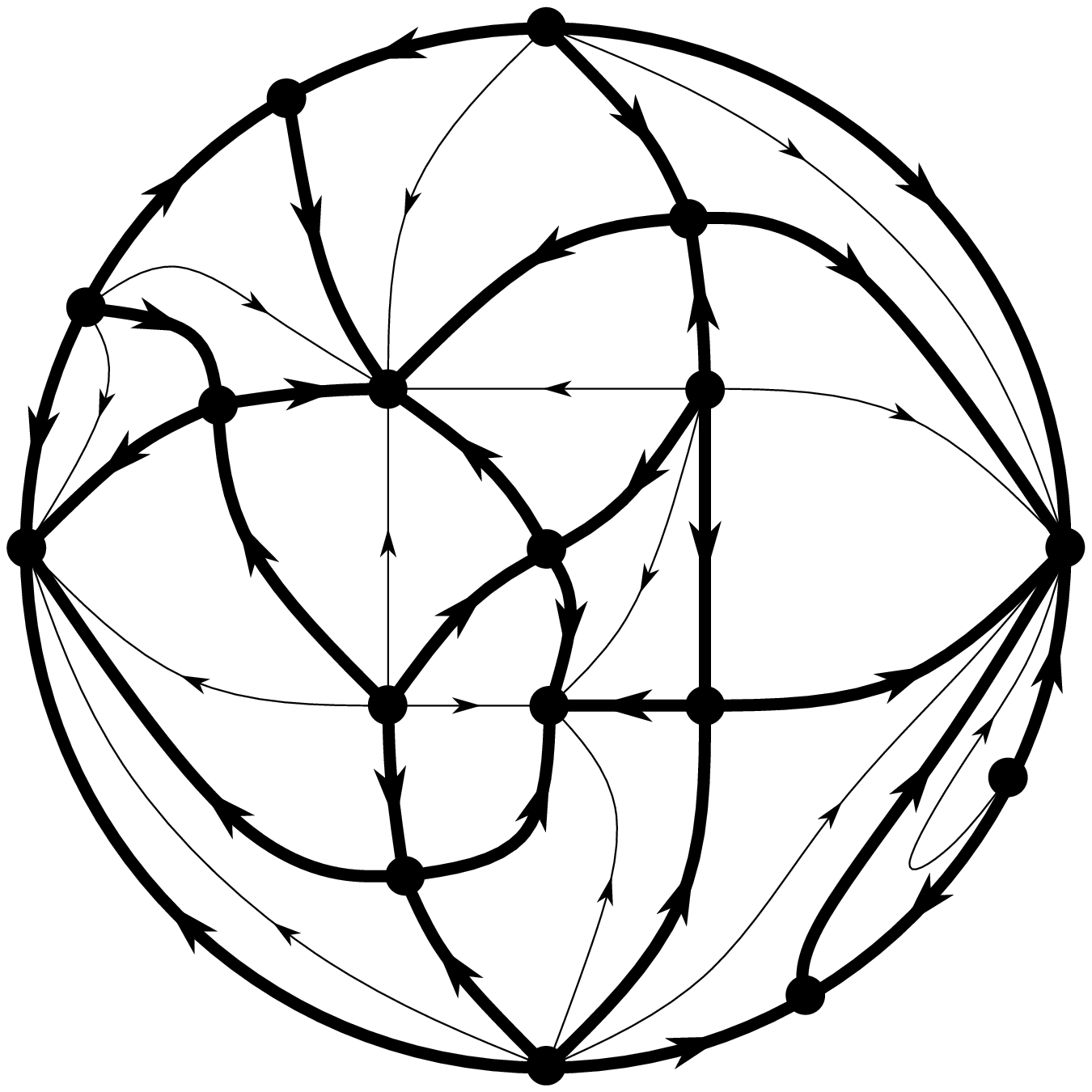} 
				\end{overpic}
				
				Case~$2.2a1$.
			\end{center}
		\end{minipage}
	\end{center}
	$\;$
	\begin{center}
		\begin{minipage}{3.1cm}
			\begin{center}
				\begin{overpic}[height=3cm]{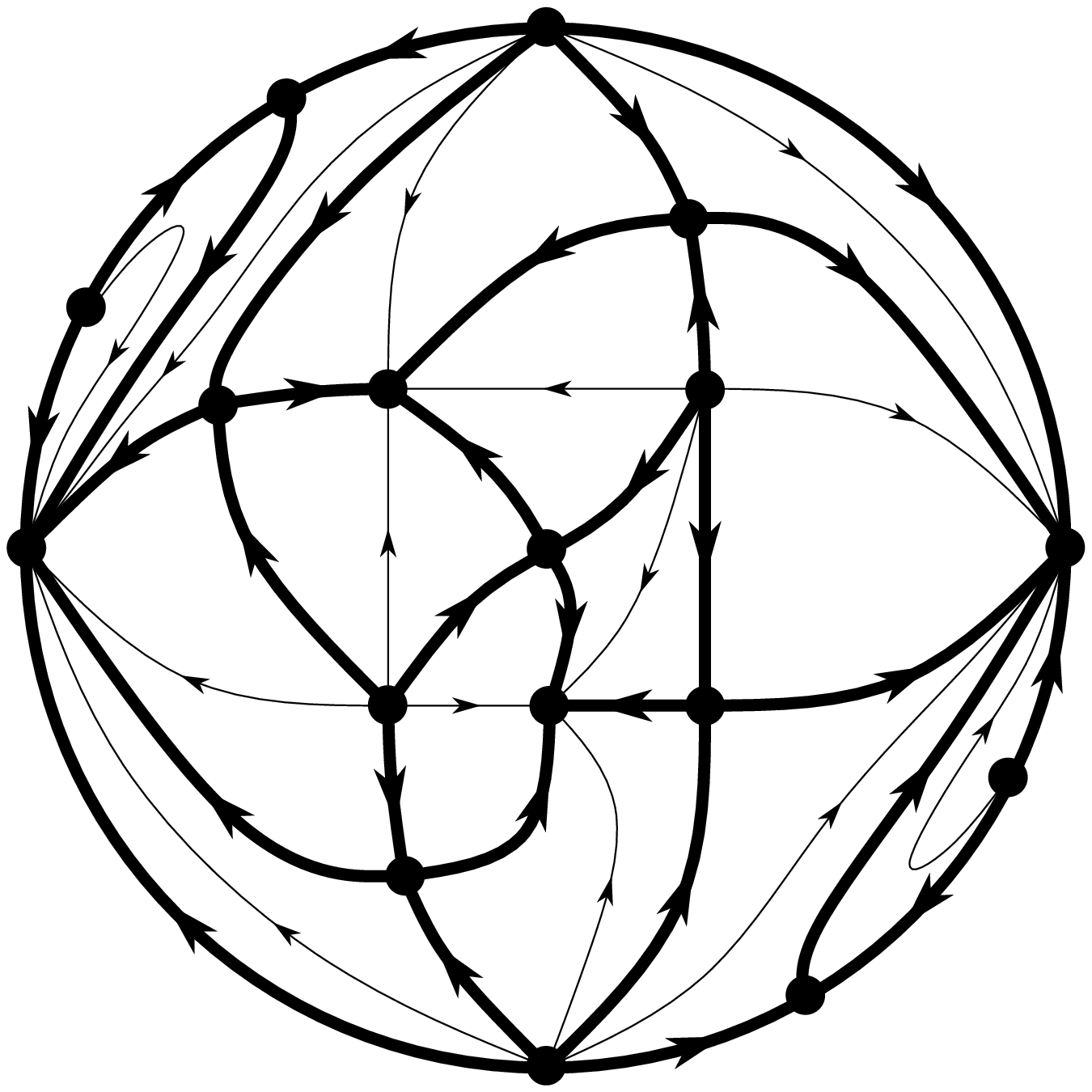} 
				\end{overpic}
				
				Case~$2.2a2$.
			\end{center}
		\end{minipage}
		\begin{minipage}{3.1cm}
			\begin{center}
				\begin{overpic}[height=3cm]{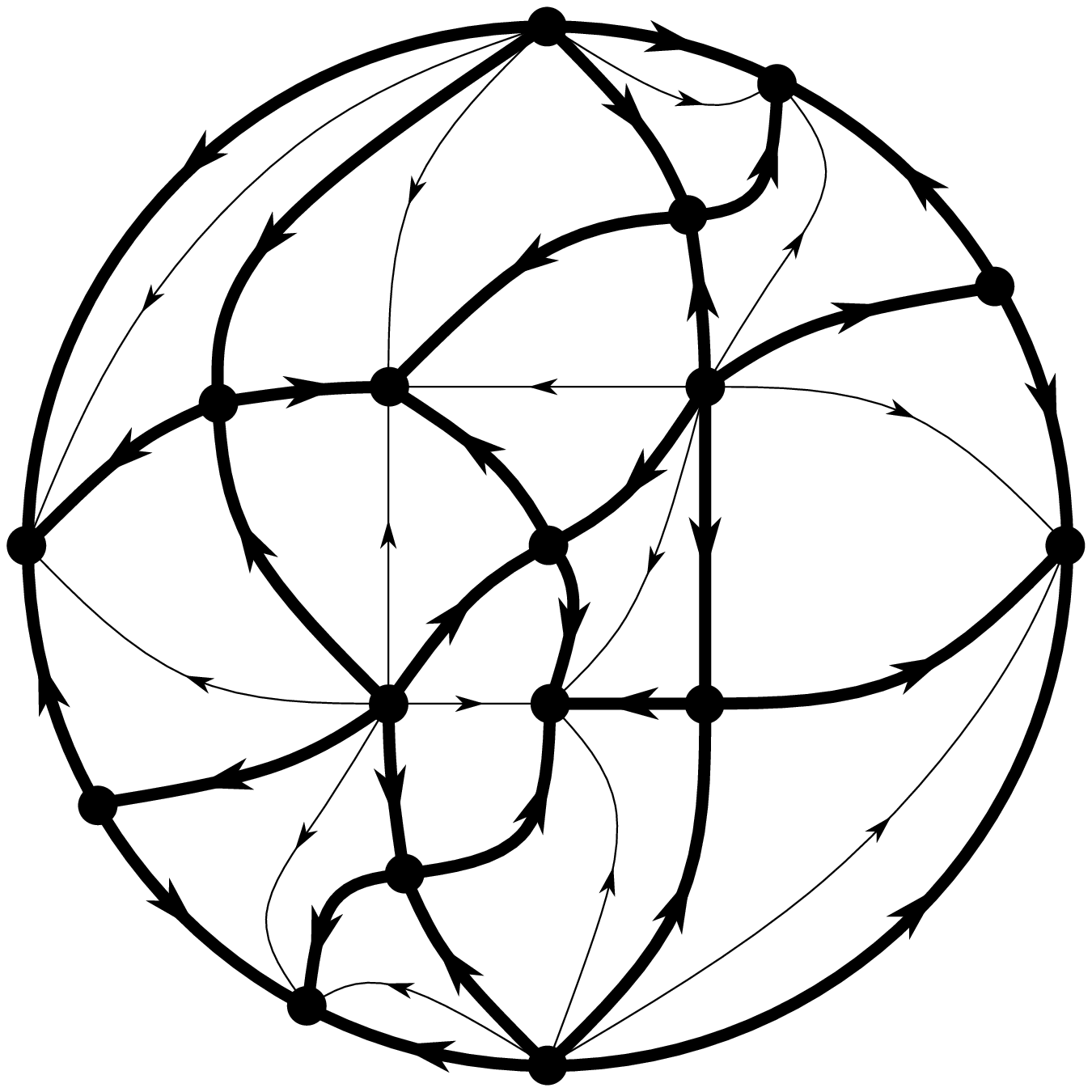} 
				\end{overpic}
				
				Case~$2.2b1$.
			\end{center}
		\end{minipage}	
		\begin{minipage}{3.1cm}
			\begin{center}
				\begin{overpic}[height=3cm]{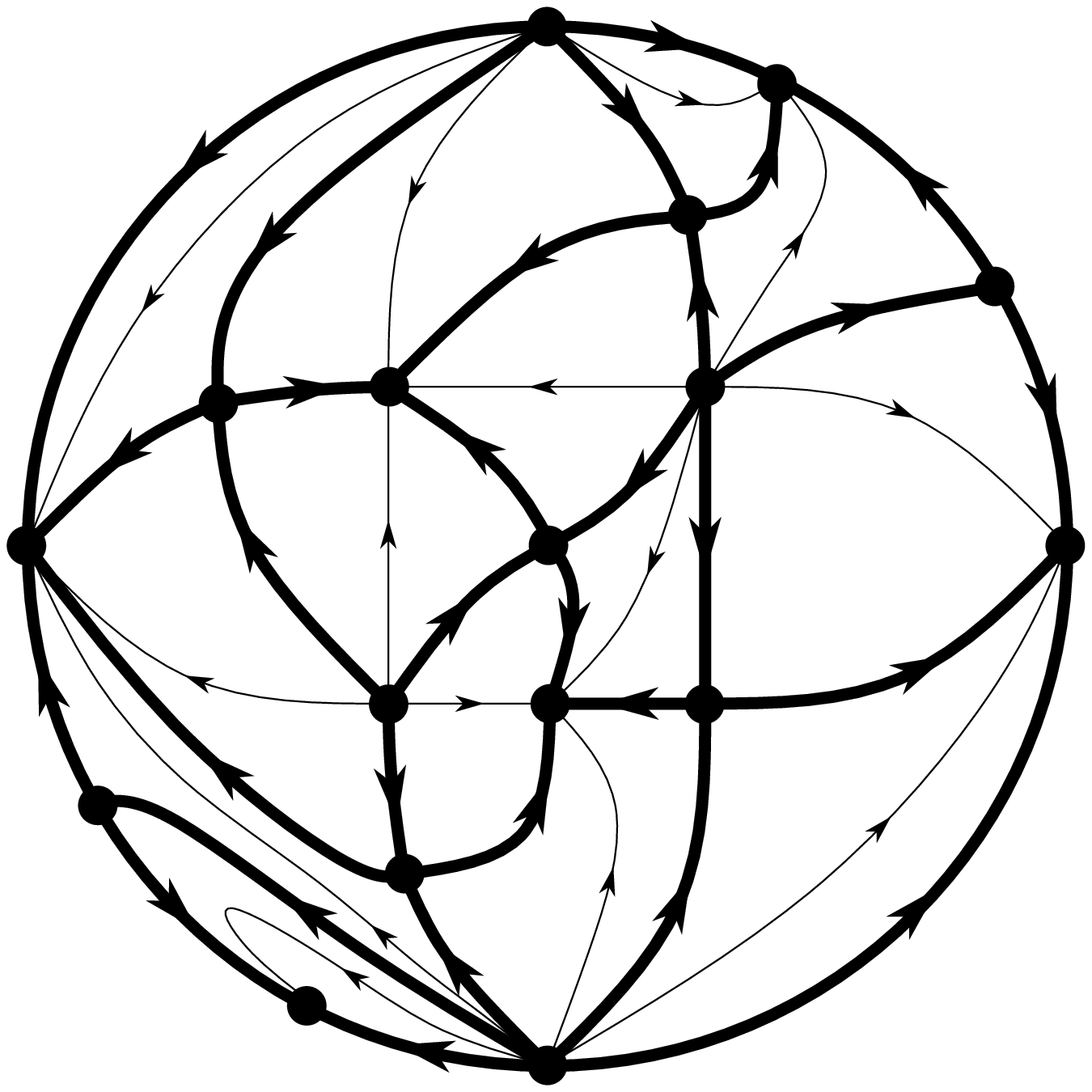} 
				\end{overpic}
				
				Case~$2.2b2$.
			\end{center}
		\end{minipage}
		\begin{minipage}{3.1cm}
			\begin{center}
				\begin{overpic}[height=3cm]{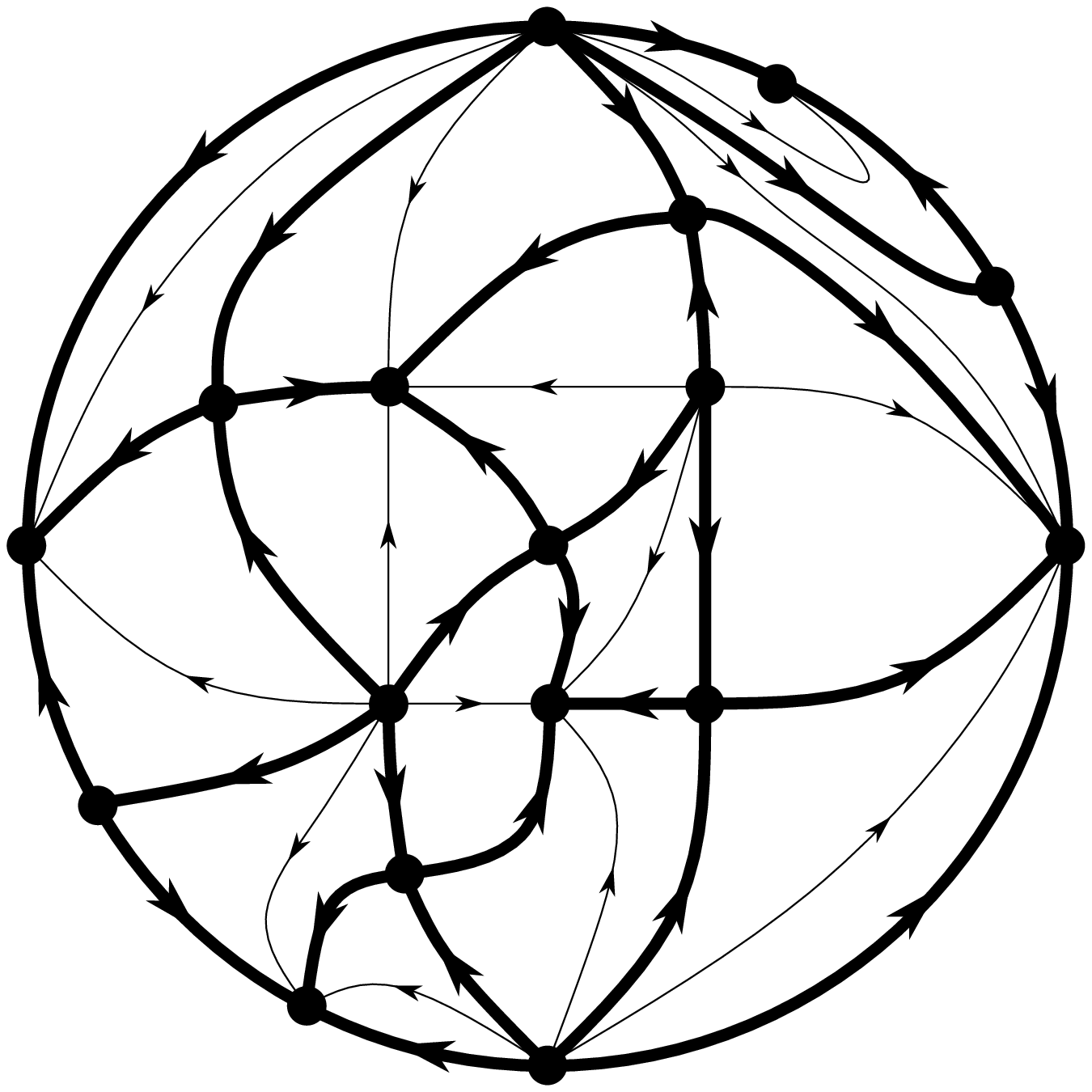} 
				\end{overpic}
				
				Case~$2.2b3$.
			\end{center}
		\end{minipage}
		\begin{minipage}{3.1cm}
			\begin{center}
				\begin{overpic}[height=3cm]{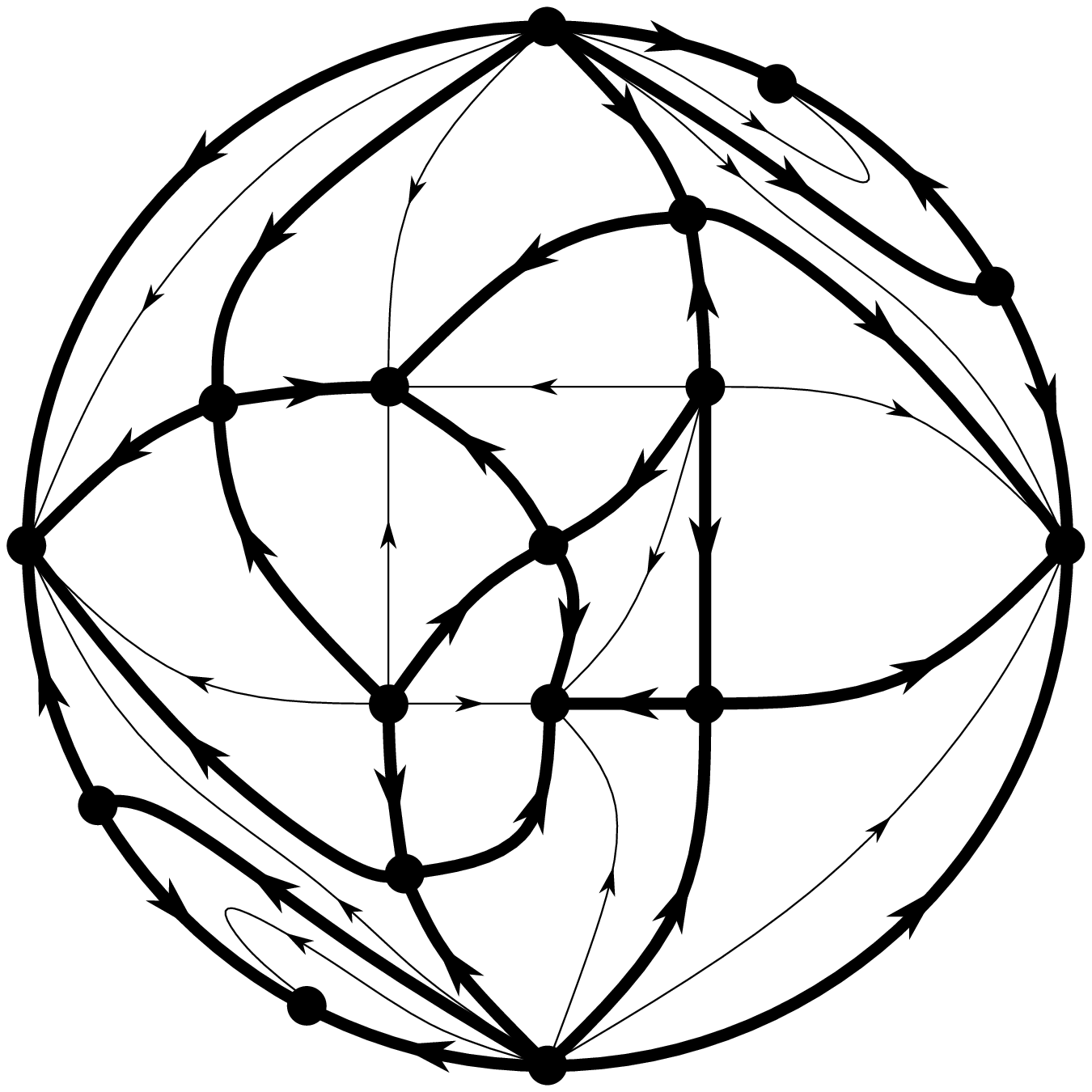} 
				\end{overpic}
				
				Case~$2.2b4$.
			\end{center}
		\end{minipage}
	\end{center}
	$\;$
	\begin{center}
		\begin{minipage}{3.1cm}
			\begin{center}
				\begin{overpic}[height=3cm]{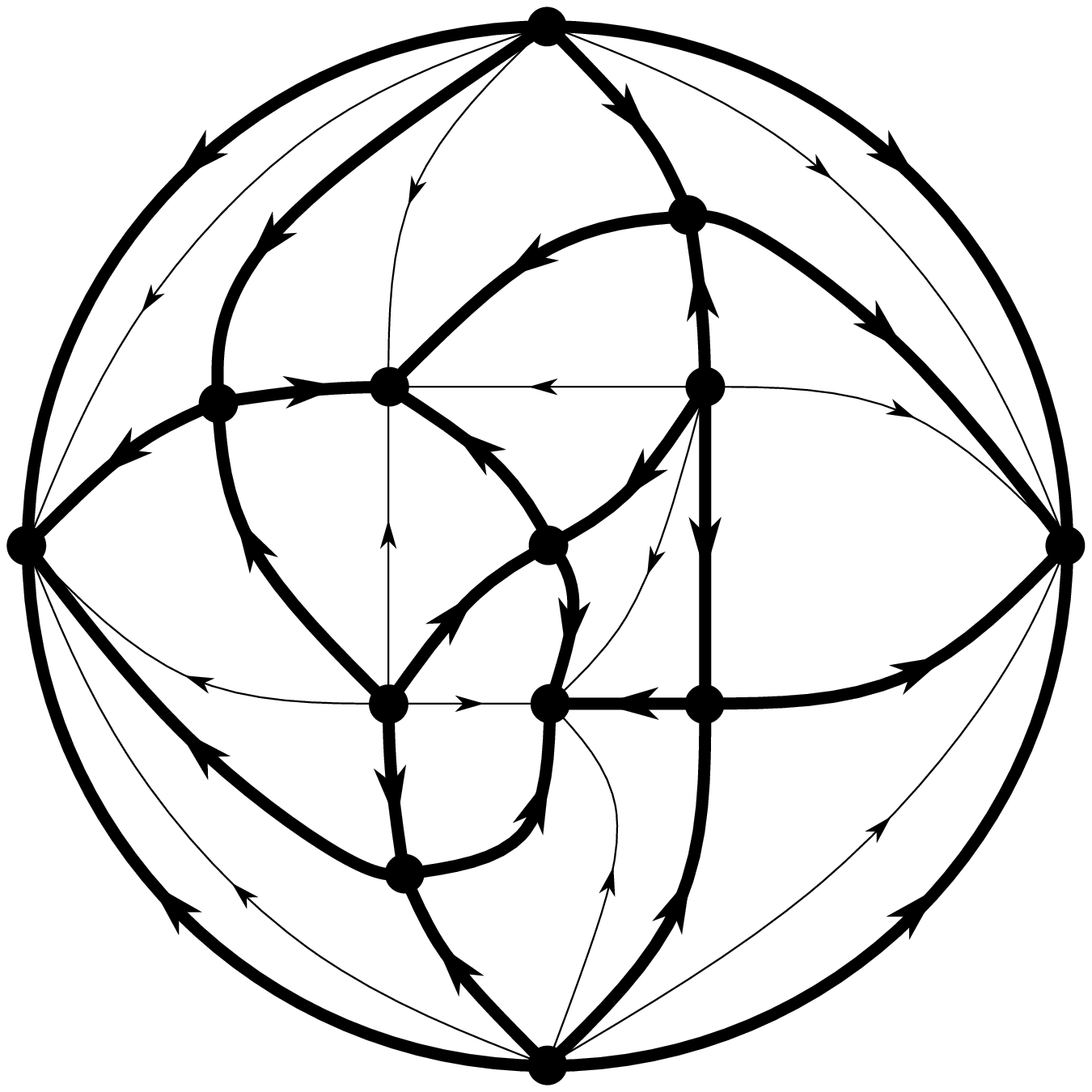} 
				\end{overpic}
				
				Case~$2.2c$.
			\end{center}
		\end{minipage}
		\begin{minipage}{3.1cm}
			\begin{center}
				\begin{overpic}[height=3cm]{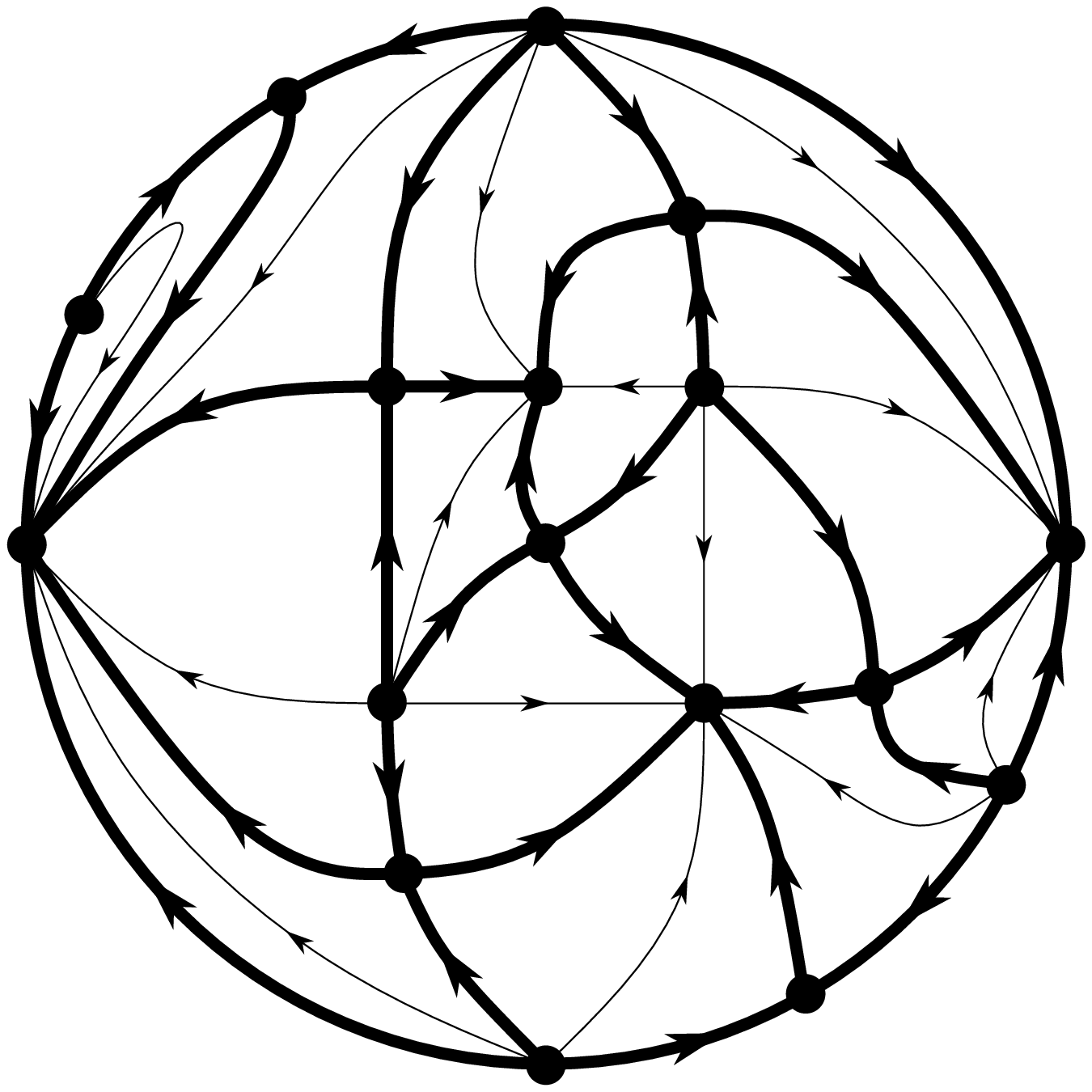} 
				\end{overpic}
				
				Case~$2.3a1$.
			\end{center}
		\end{minipage}
		\begin{minipage}{3.1cm}
			\begin{center}
				\begin{overpic}[height=3cm]{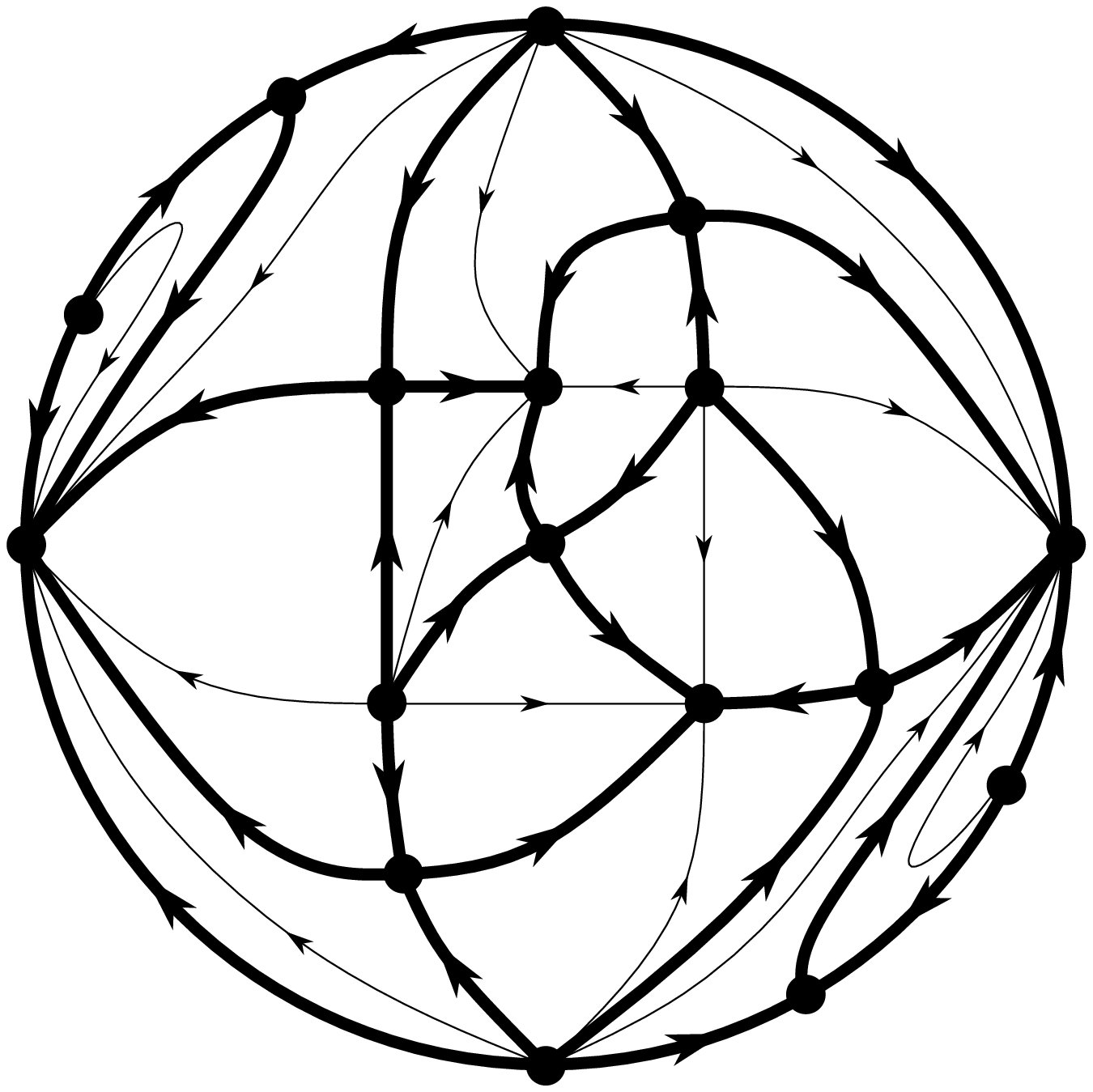} 
				\end{overpic}
				
				Case~$2.3a2$.
			\end{center}
		\end{minipage}
		\begin{minipage}{3.1cm}
			\begin{center}
				\begin{overpic}[height=3cm]{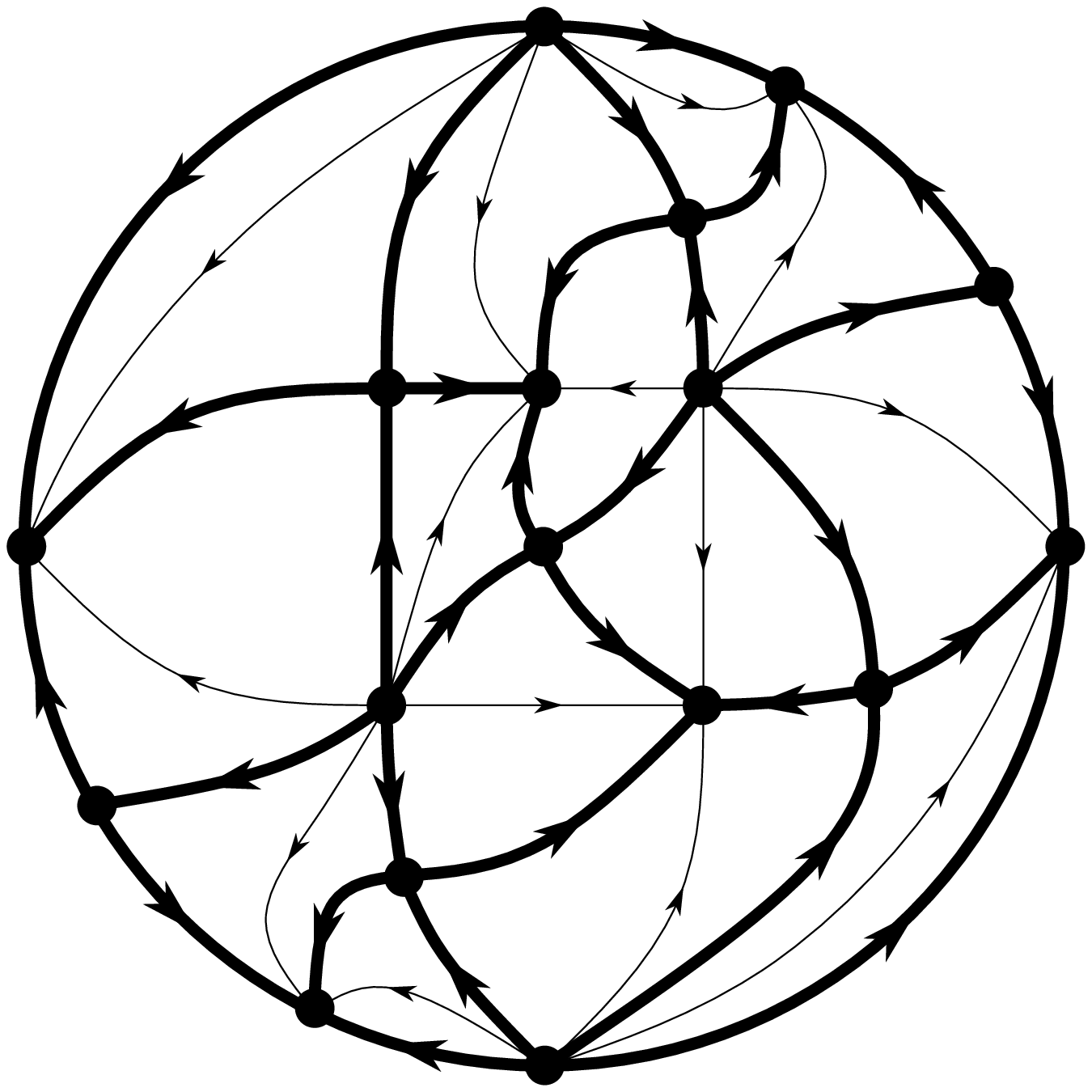} 
				\end{overpic}
				
				Case~$2.3b1$.
			\end{center}
		\end{minipage}	
		\begin{minipage}{3.1cm}
			\begin{center}
				\begin{overpic}[height=3cm]{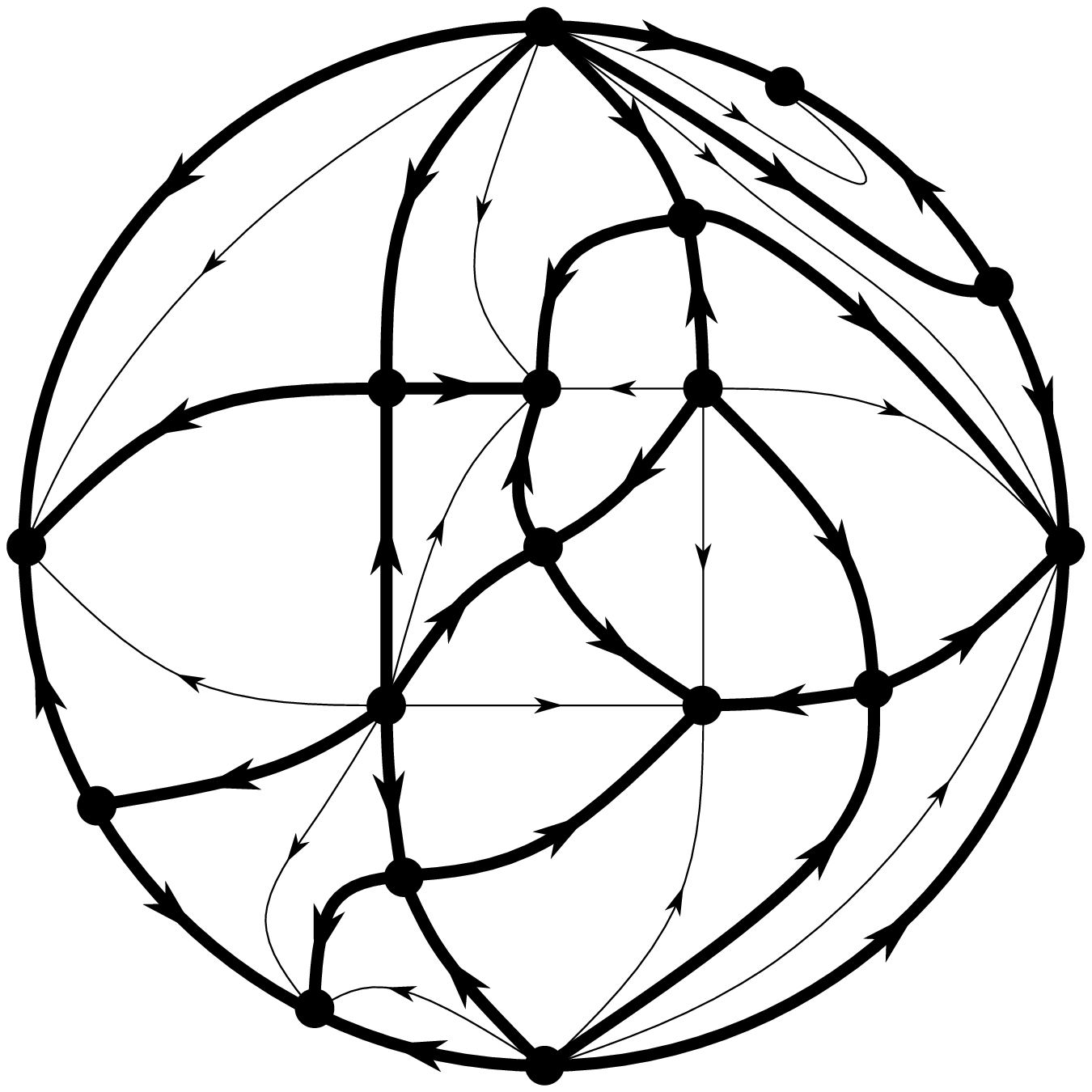} 
				\end{overpic}
				
				Case~$2.3b2$.
			\end{center}
		\end{minipage}
	\end{center}
	$\;$
	\begin{center}
		\begin{minipage}{3.1cm}
			\begin{center}
				\begin{overpic}[height=3cm]{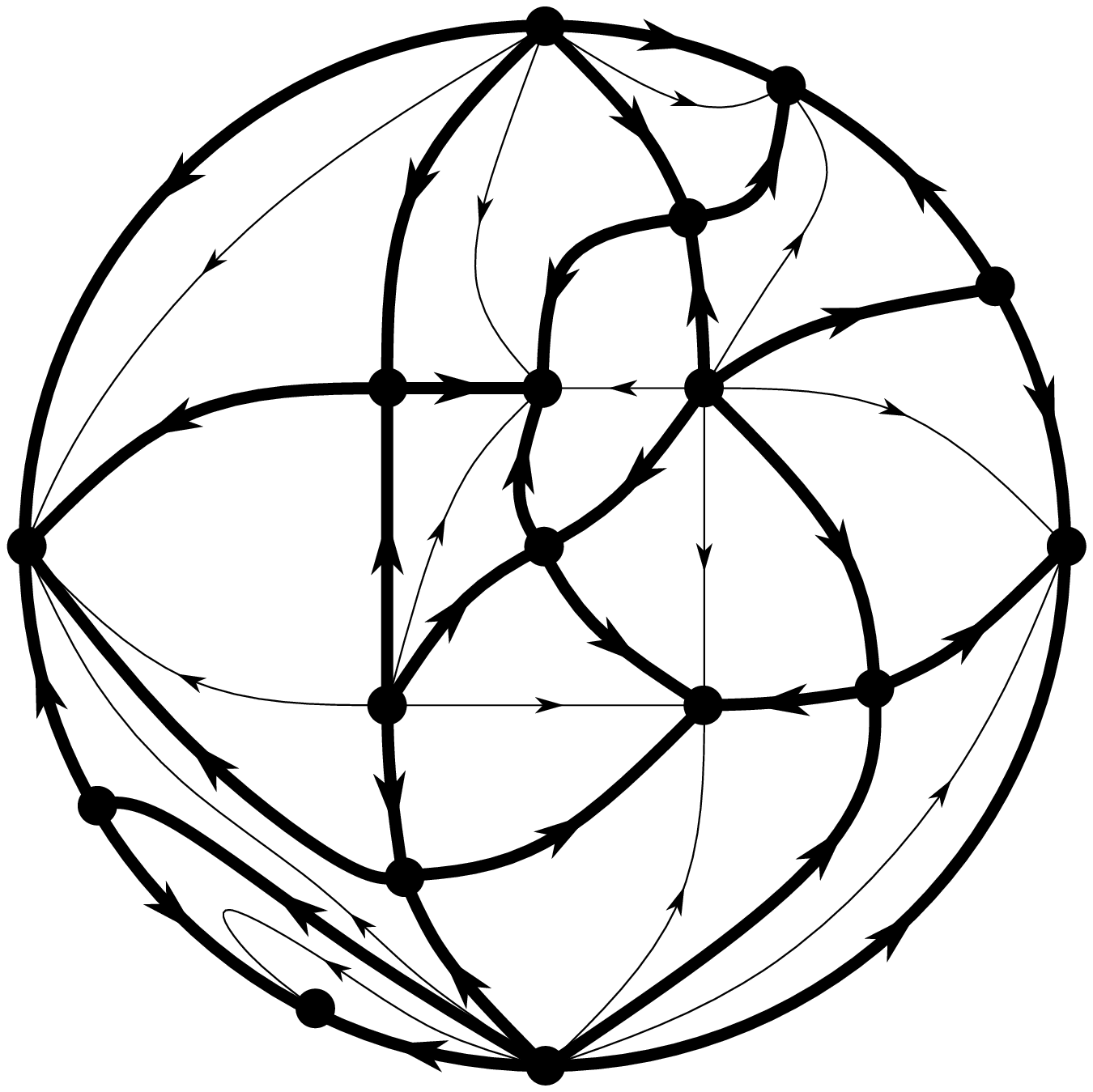} 
				\end{overpic}
				
				Case~$2.3b3$.
			\end{center}
		\end{minipage}
		\begin{minipage}{3.1cm}
			\begin{center}
				\begin{overpic}[height=3cm]{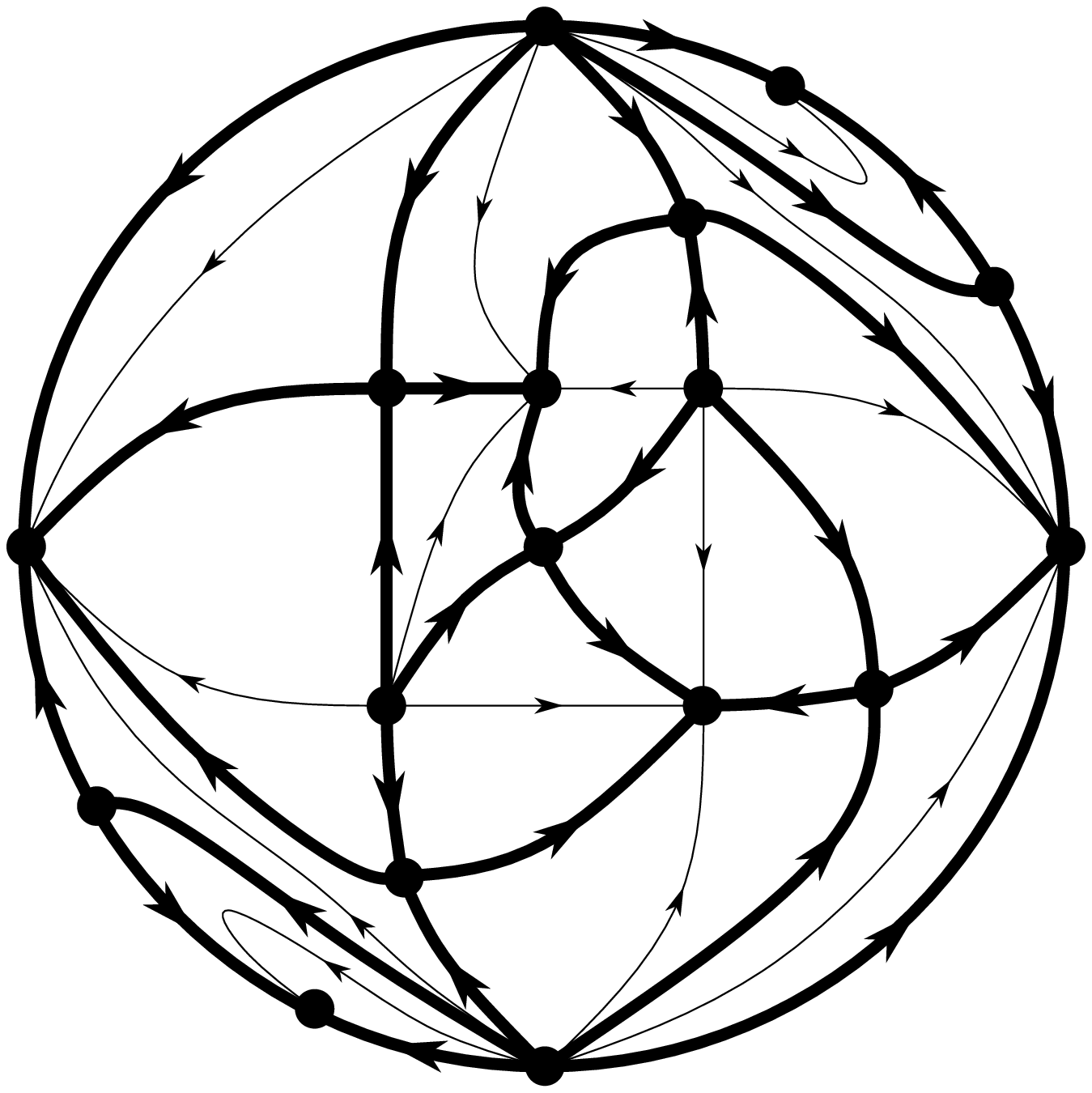} 
				\end{overpic}
				
				Case~$2.3b4$.
			\end{center}
		\end{minipage}
		\begin{minipage}{3.1cm}
			\begin{center}
				\begin{overpic}[height=3cm]{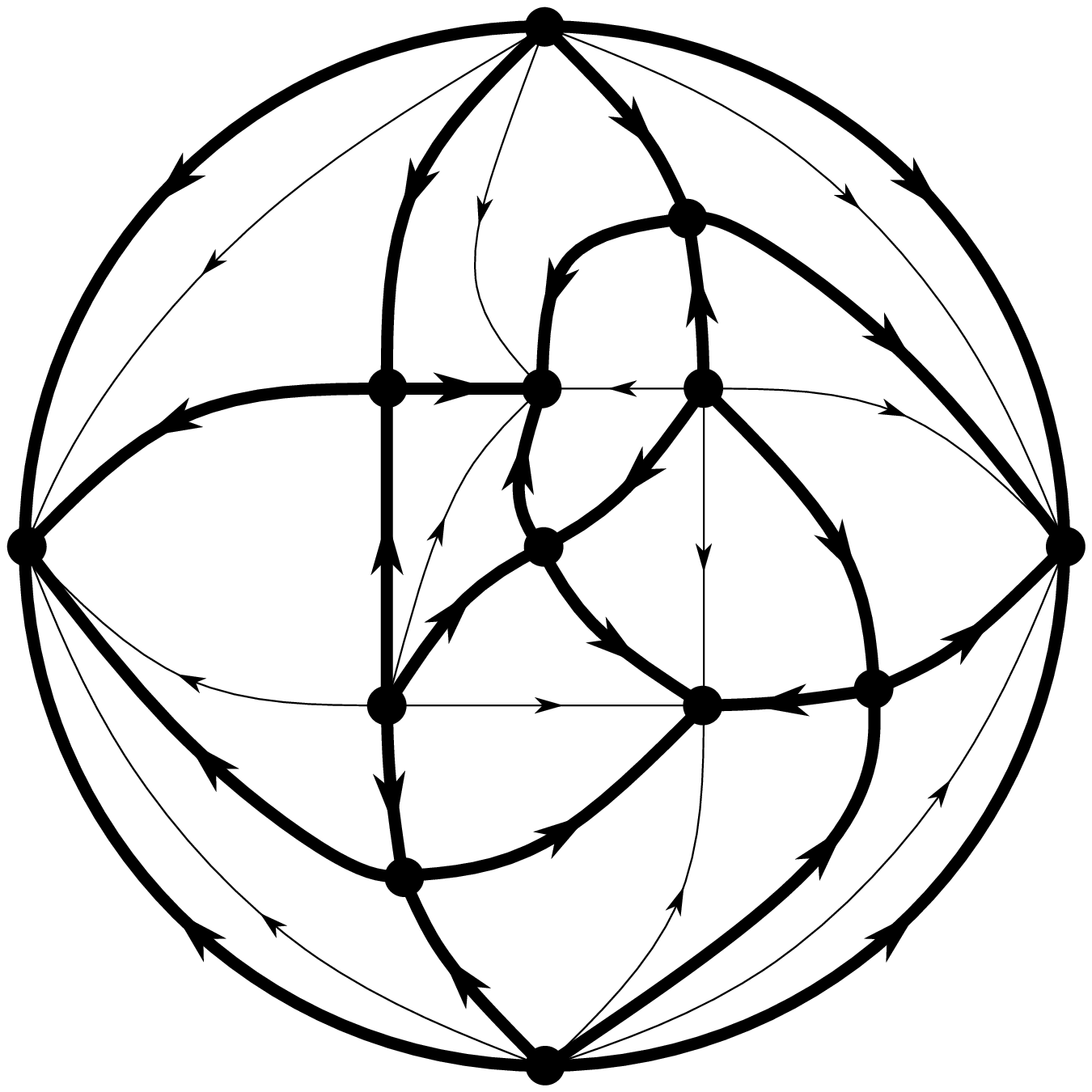} 
				\end{overpic}
				
				Case~$2.3c$.
			\end{center}
		\end{minipage}
		\begin{minipage}{3.1cm}
			\begin{center}
				\begin{overpic}[height=3cm]{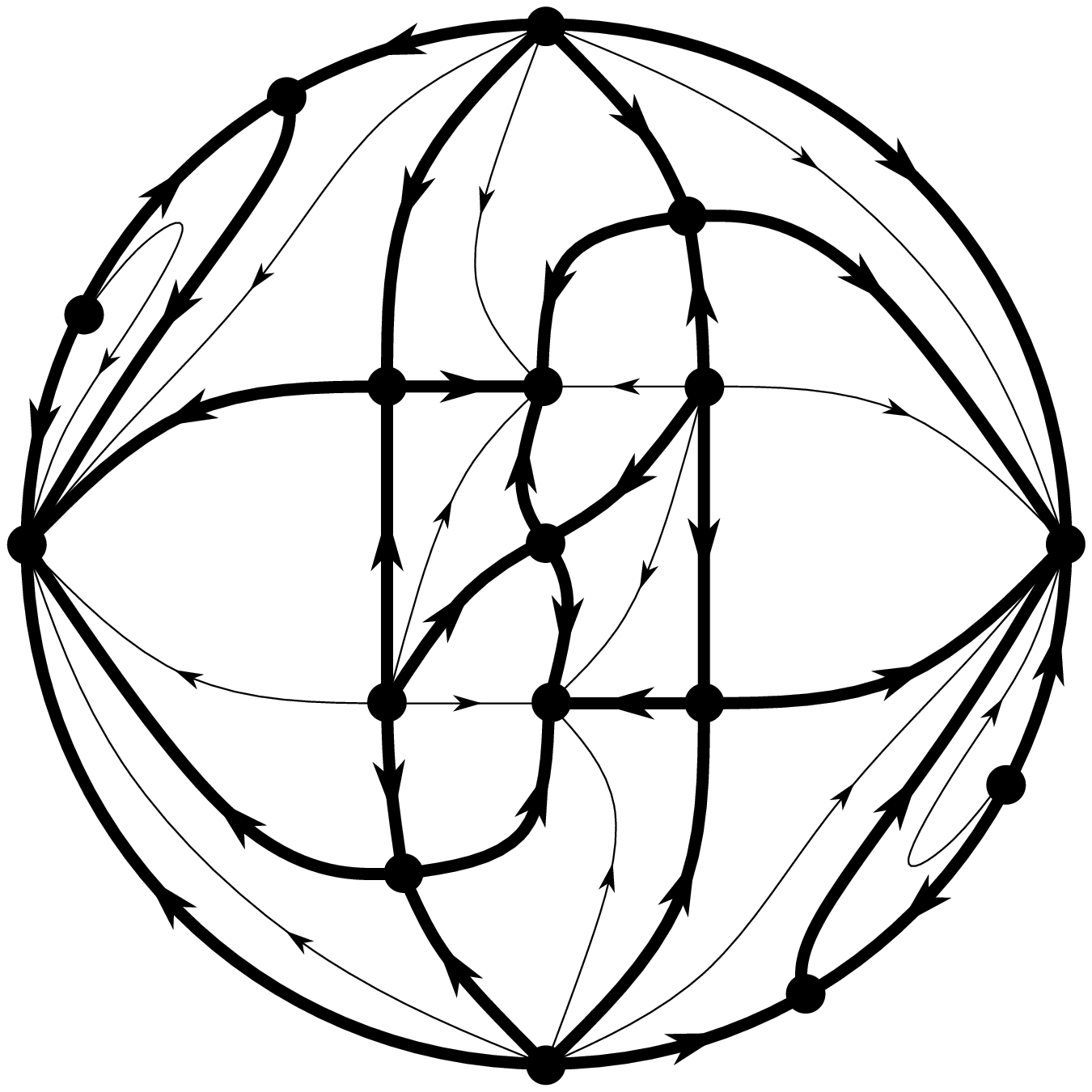} 
				\end{overpic}
				
				Case~$2.4a$.
			\end{center}
		\end{minipage}
		\begin{minipage}{3.1cm}
			\begin{center}
				\begin{overpic}[height=3cm]{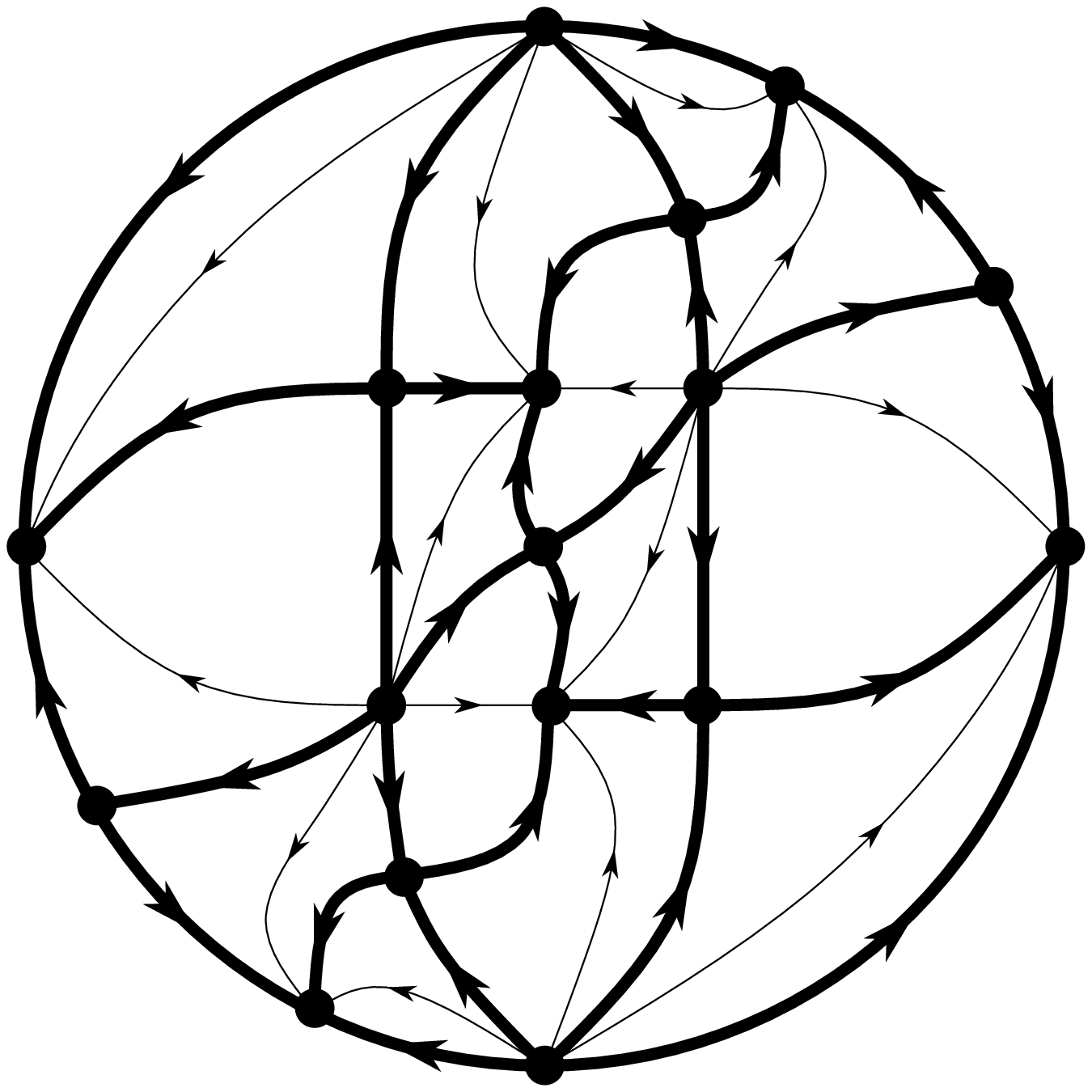} 
				\end{overpic}
				
				Case~$2.4b1$.
			\end{center}
		\end{minipage}
	\end{center}
	$\;$
	\begin{center}
		\begin{minipage}{3.1cm}
			\begin{center}
				\begin{overpic}[height=3cm]{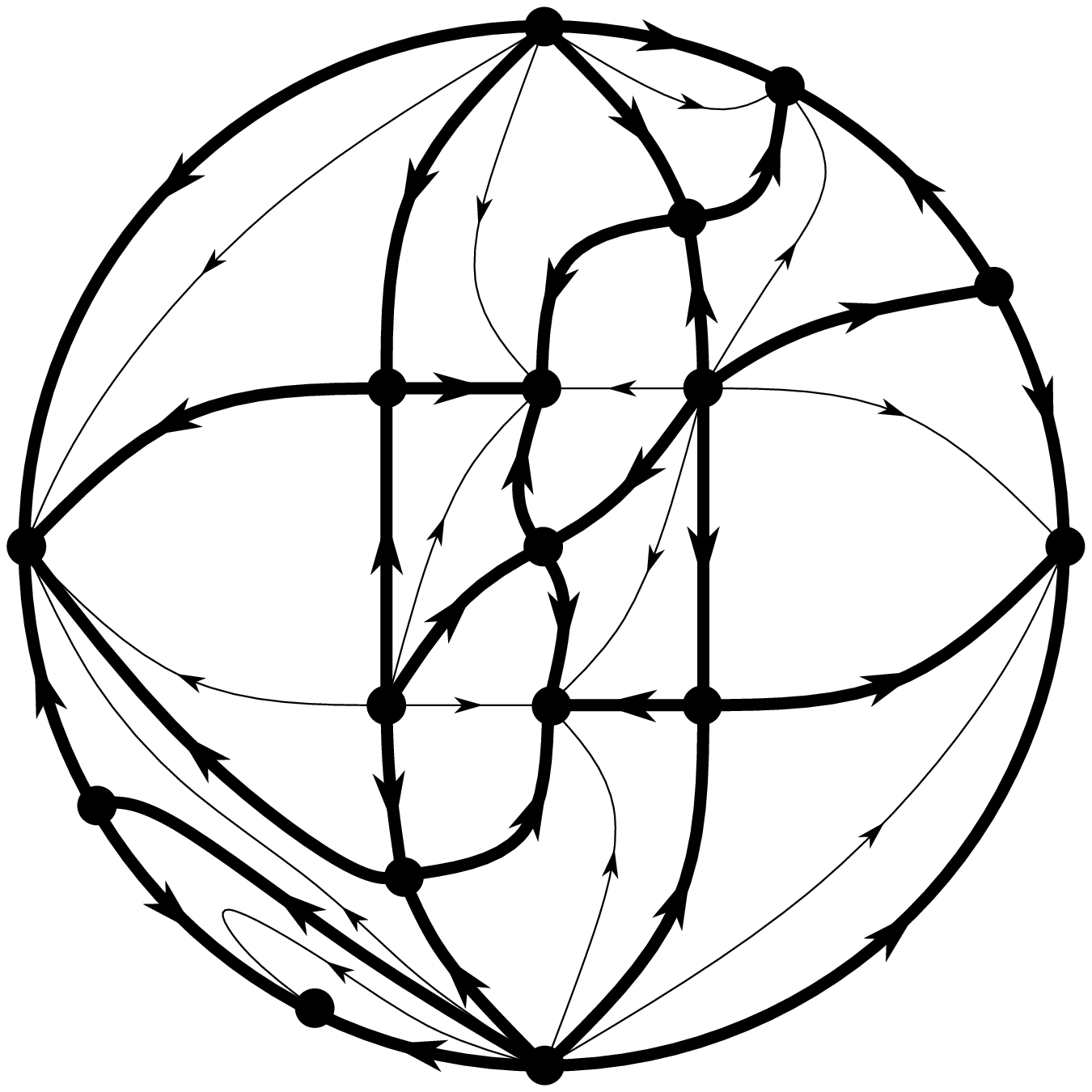} 
				\end{overpic}
				
				Case~$2.4b2$.
			\end{center}
		\end{minipage}	
		\begin{minipage}{3.1cm}
			\begin{center}
				\begin{overpic}[height=3cm]{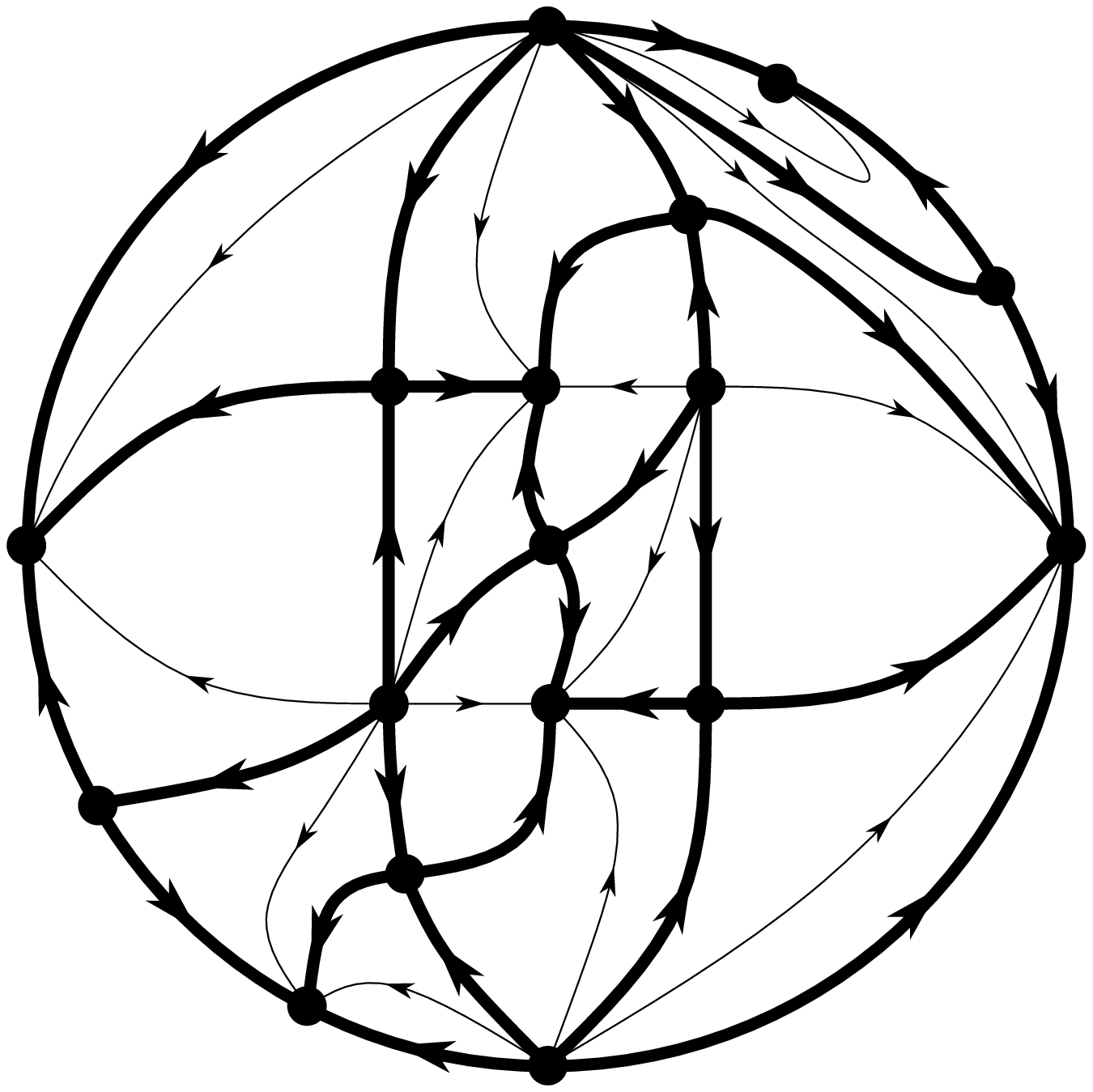} 
				\end{overpic}
				
				Case~$2.4b3$.
			\end{center}
		\end{minipage}
		\begin{minipage}{3.1cm}
			\begin{center}
				\begin{overpic}[height=3cm]{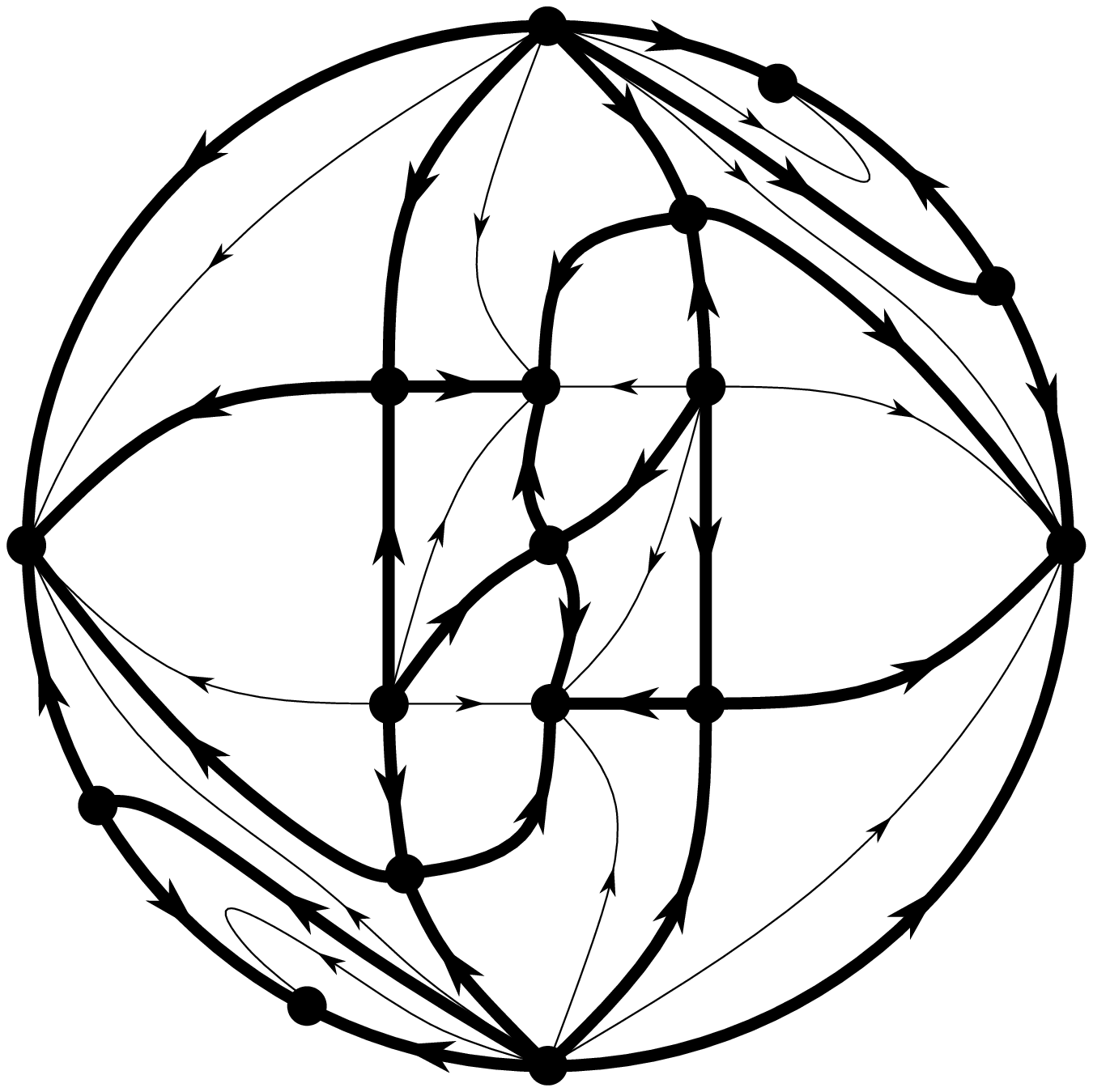} 
				\end{overpic}
				
				Case~$2.4b4$.
			\end{center}
		\end{minipage}
		\begin{minipage}{3.1cm}
			\begin{center}
				\begin{overpic}[height=3cm]{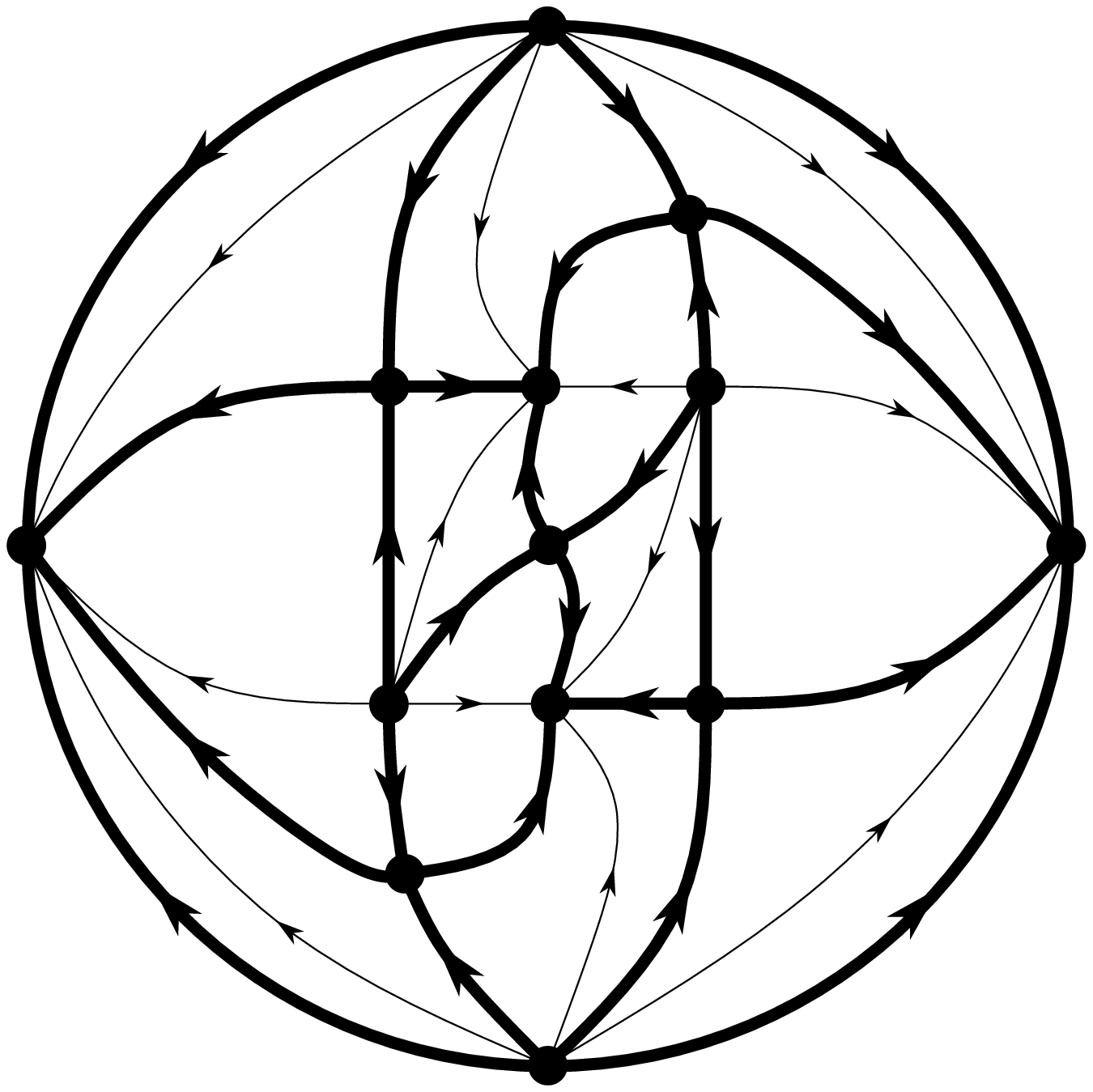} 
				\end{overpic}
				
				Case~$2.4c$.
			\end{center}
		\end{minipage}
	\end{center}
	\caption{Phase portraits from Cases~$2.1$ to $2.4$.}\label{Case1.2a}
\end{figure}

\begin{figure}[h]
	\begin{center}
		\begin{minipage}{3.1cm}
			\begin{center}
				\begin{overpic}[height=3cm]{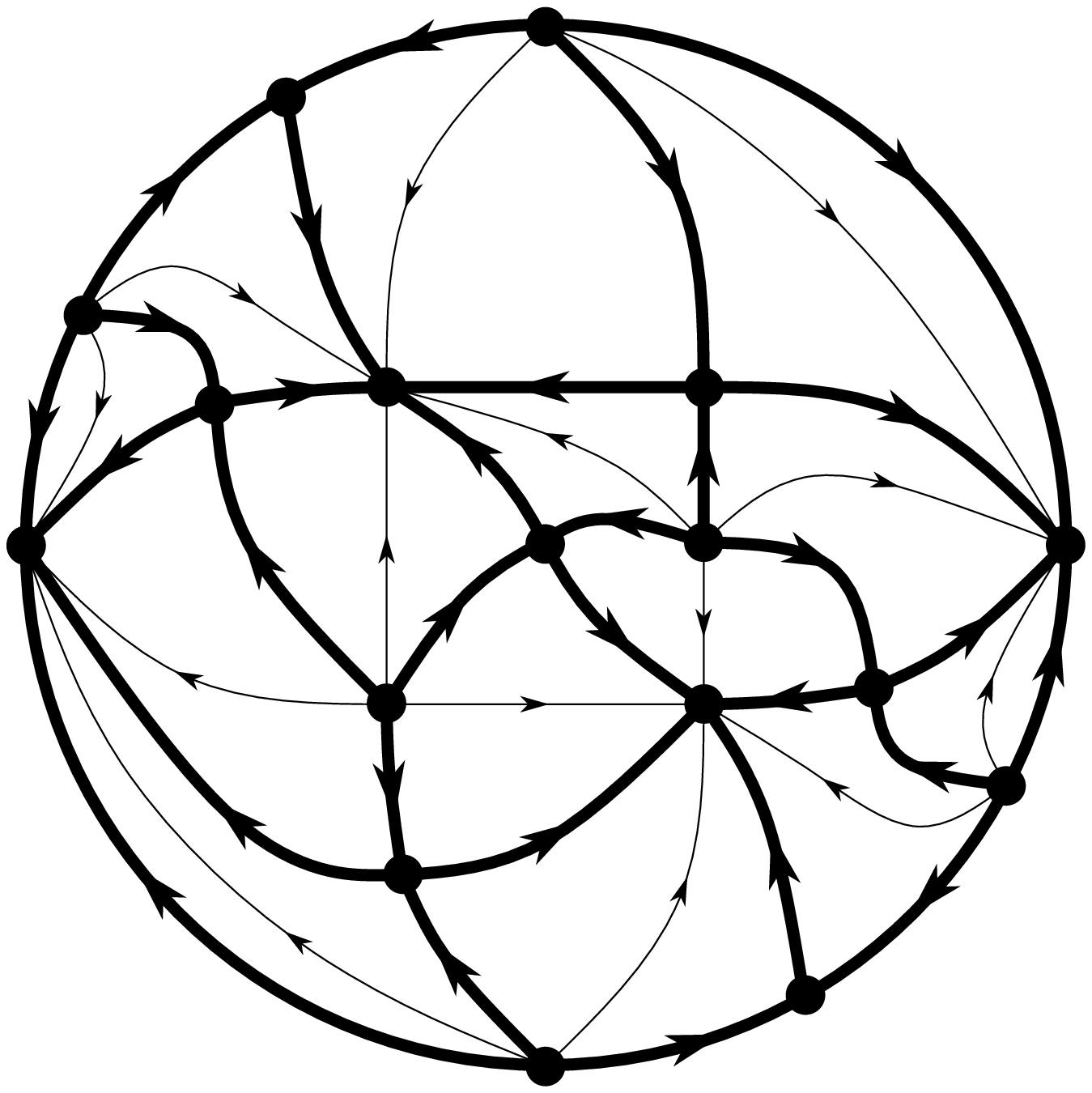} 
				\end{overpic}
				
				Case~$2.5a1$.
			\end{center}
		\end{minipage}
		\begin{minipage}{3.1cm}
			\begin{center}
				\begin{overpic}[height=3cm]{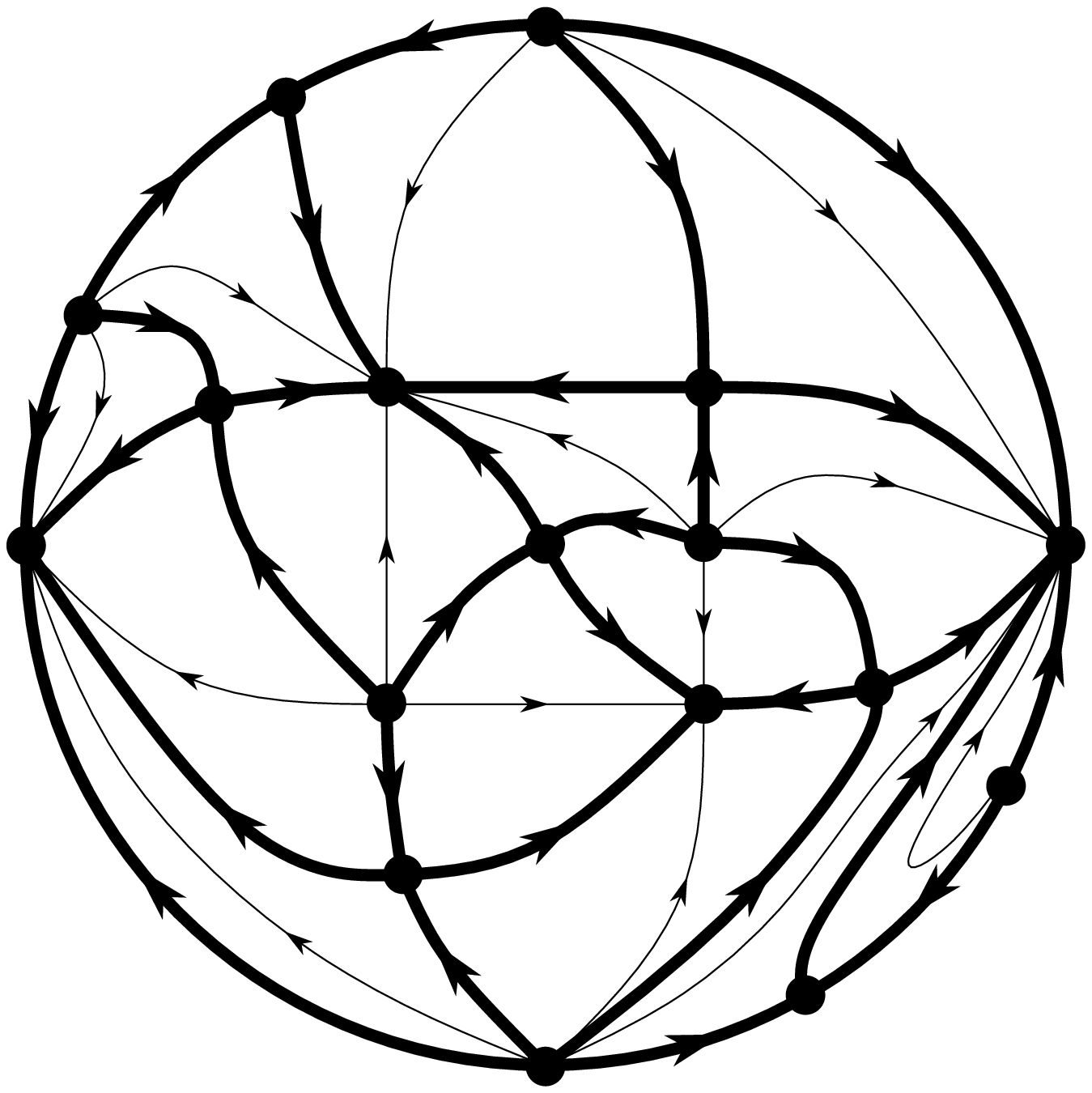} 
				\end{overpic}
				
				Case~$2.5a2$.
			\end{center}
		\end{minipage}
		\begin{minipage}{3.1cm}
			\begin{center}
				\begin{overpic}[height=3cm]{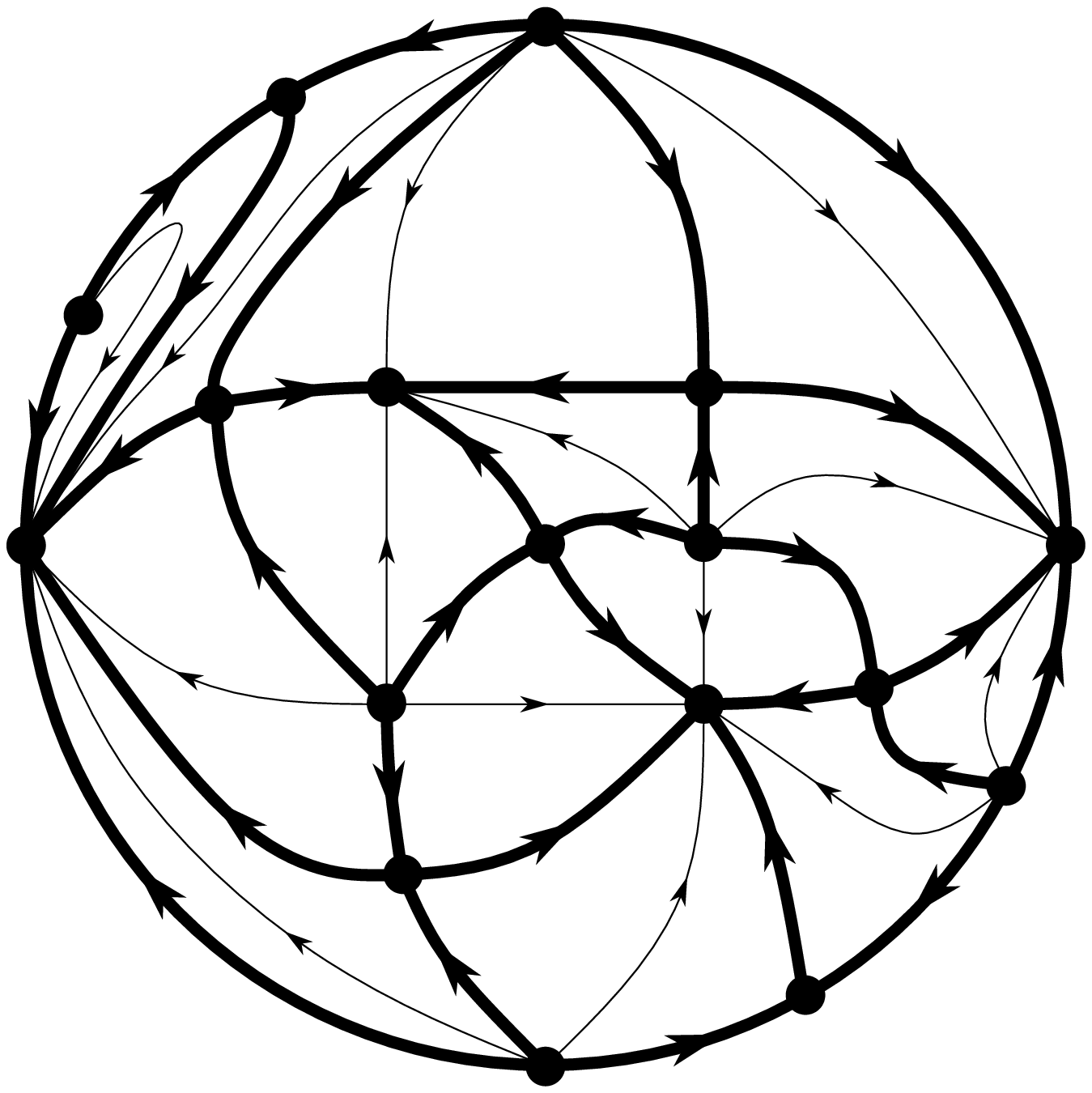} 
				\end{overpic}
				
				Case~$2.5a3$.
			\end{center}
		\end{minipage}	
		\begin{minipage}{3.1cm}
			\begin{center}
				\begin{overpic}[height=3cm]{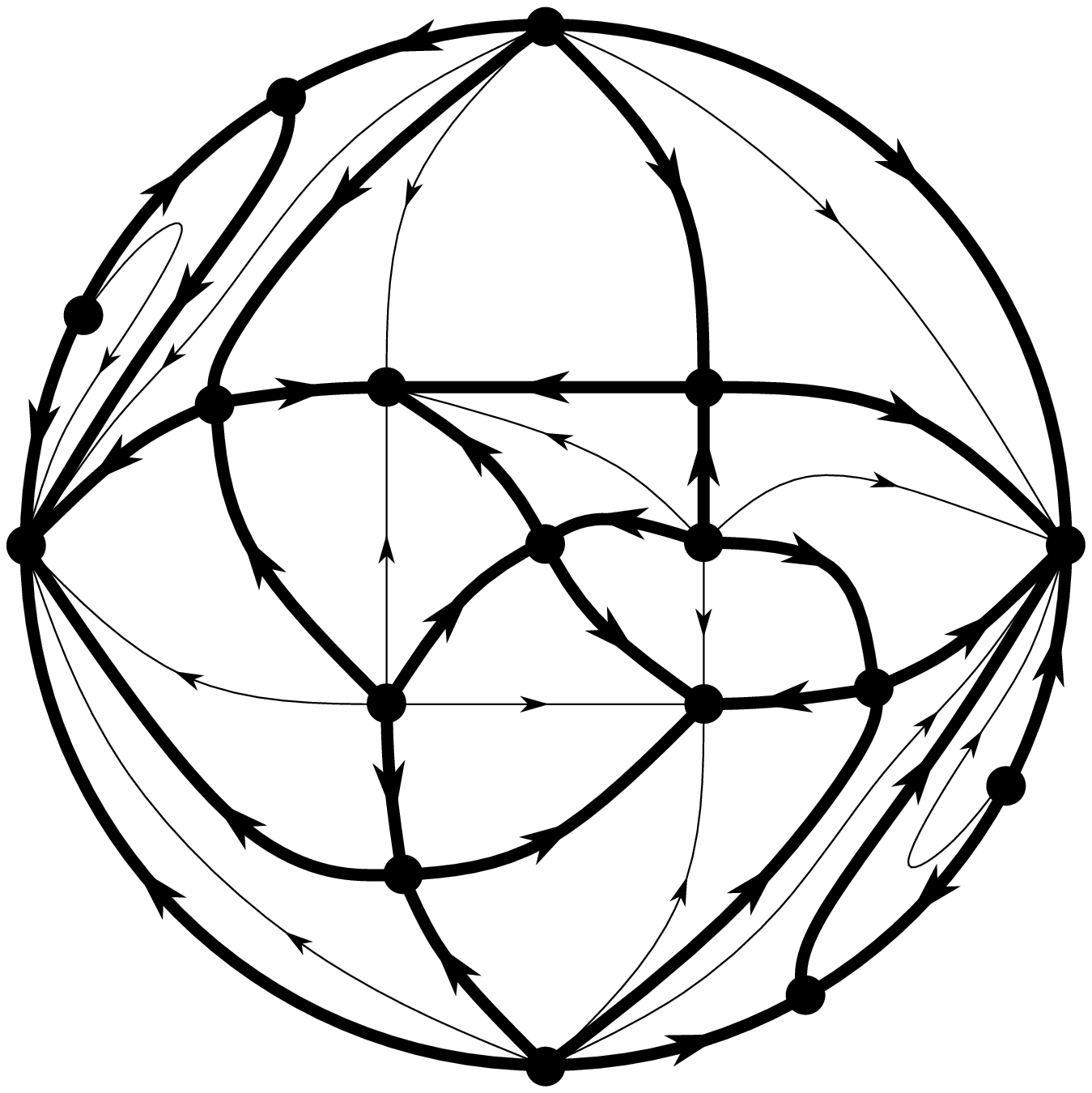} 
				\end{overpic}
				
				Case~$2.5a4$.
			\end{center}
		\end{minipage}
		\begin{minipage}{3.1cm}
			\begin{center}
				\begin{overpic}[height=3cm]{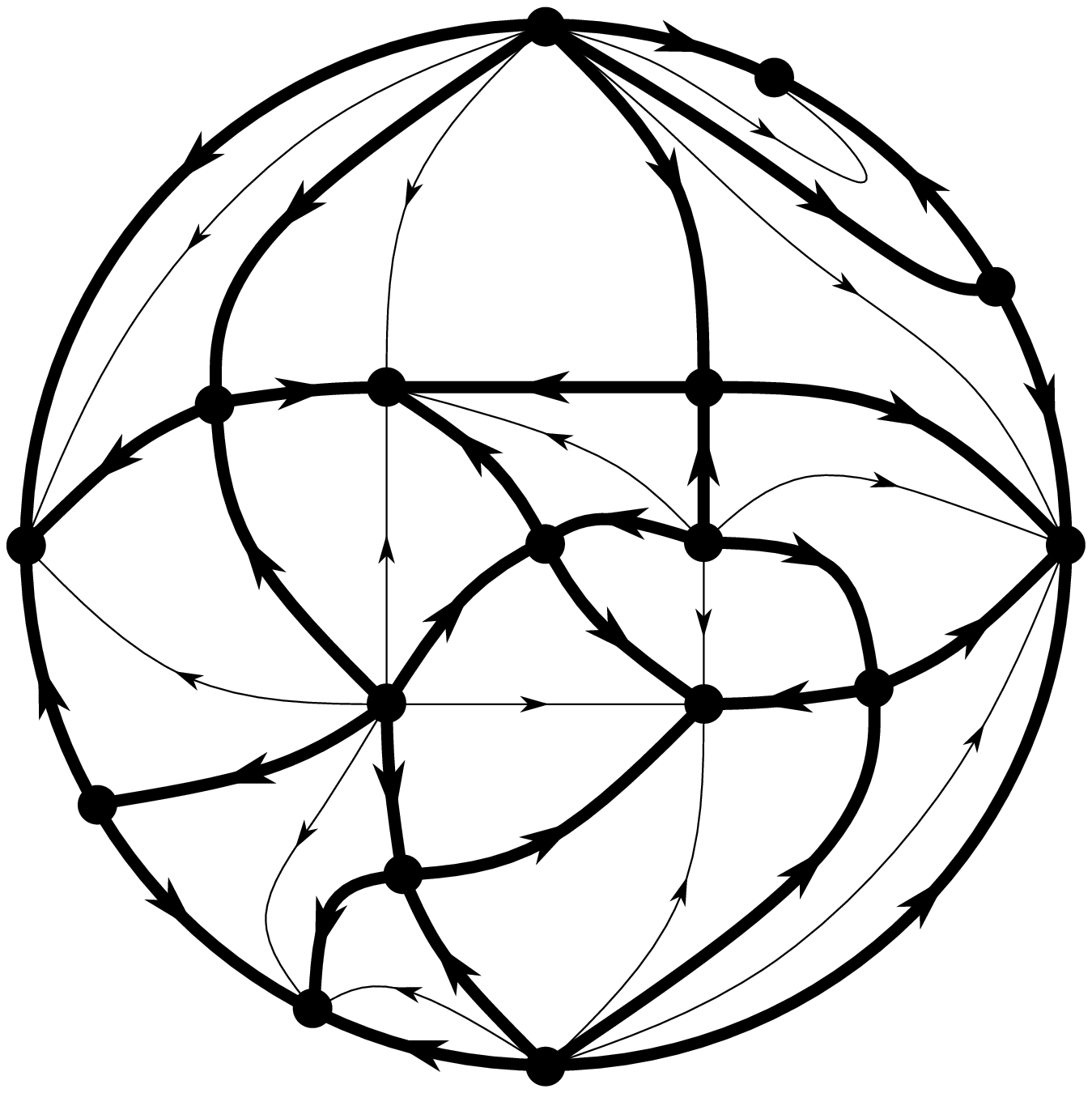} 
				\end{overpic}
				
				Case~$2.5b1$.
			\end{center}
		\end{minipage}
	\end{center}
	$\;$
	\begin{center}
		\begin{minipage}{3.1cm}
			\begin{center}
				\begin{overpic}[height=3cm]{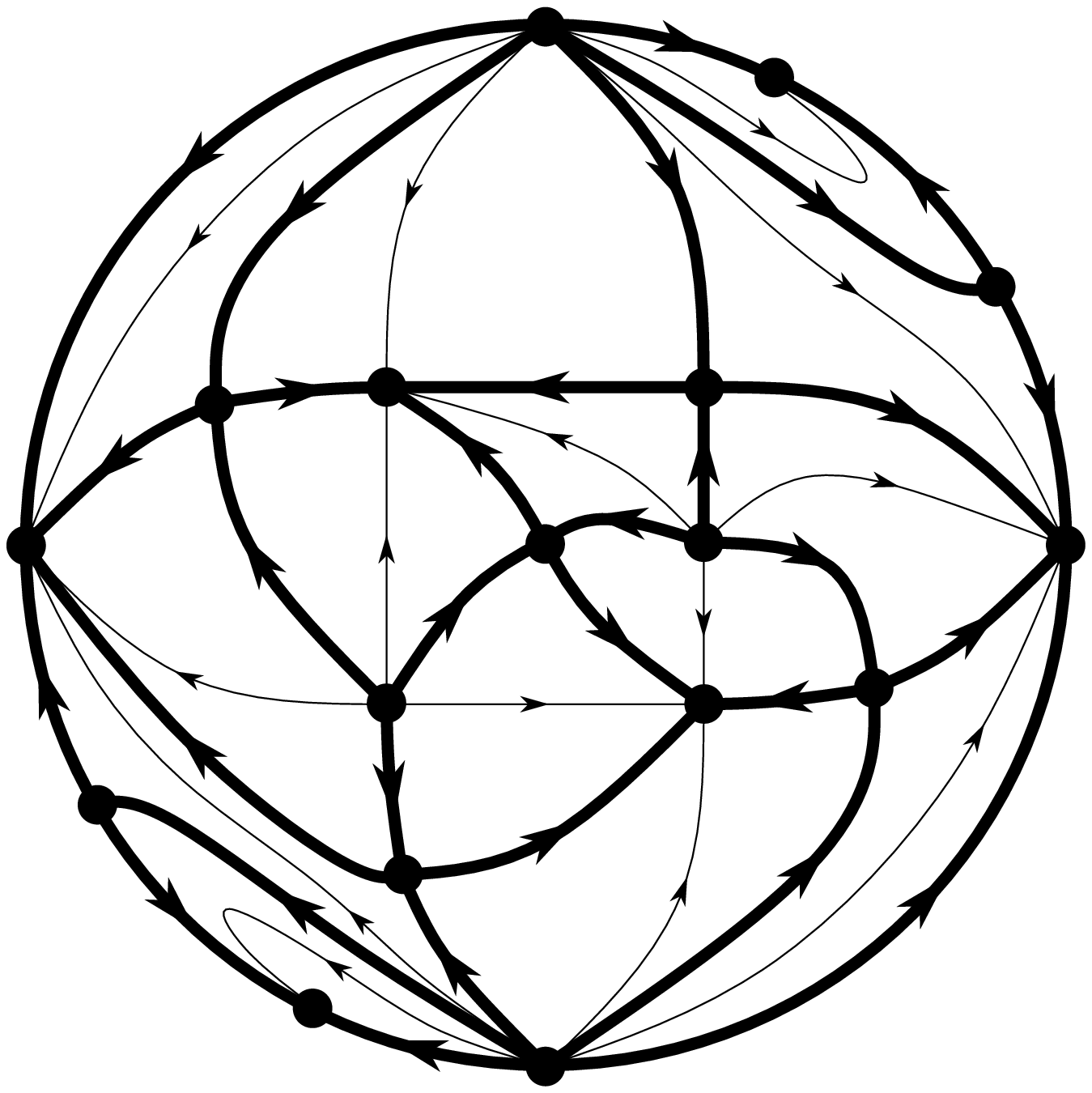} 
				\end{overpic}
				
				Case~$2.5b2$.
			\end{center}
		\end{minipage}
		\begin{minipage}{3.1cm}
			\begin{center}
				\begin{overpic}[height=3cm]{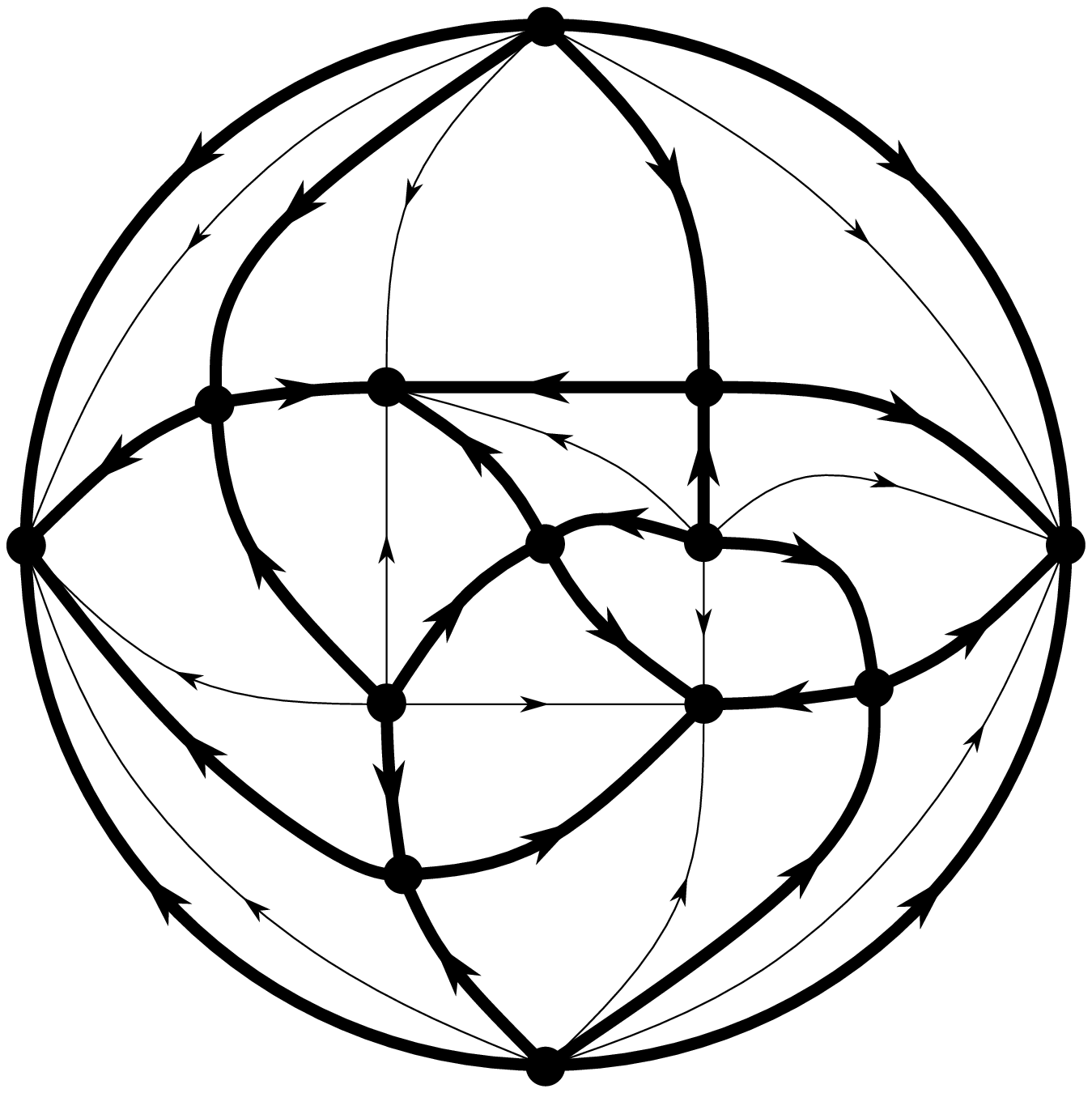} 
				\end{overpic}
				
				Case~$2.5c$.
			\end{center}
		\end{minipage}
		\begin{minipage}{3.1cm}
			\begin{center}
				\begin{overpic}[height=3cm]{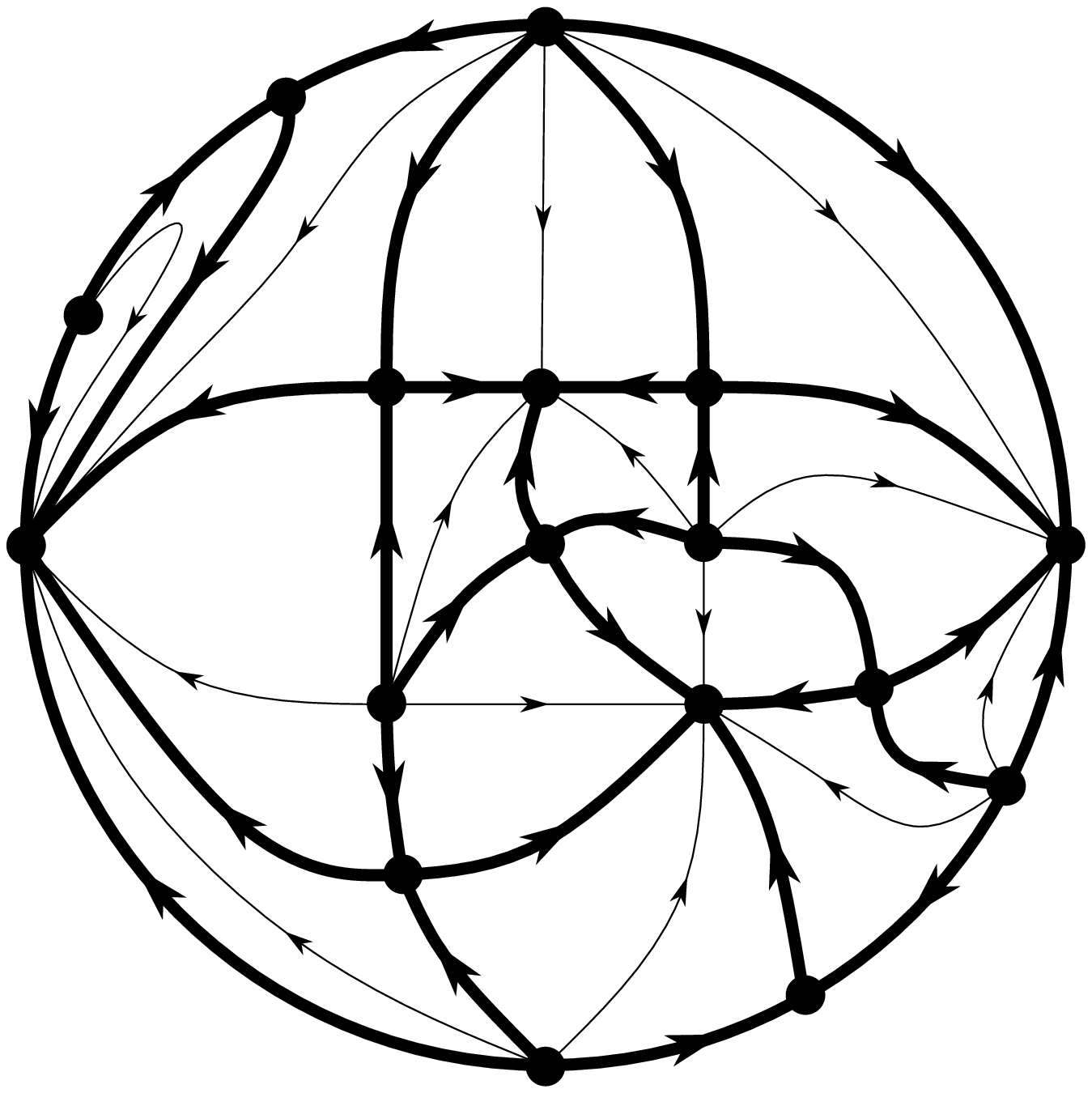} 
				\end{overpic}
				
				Case~$2.6a1$.
			\end{center}
		\end{minipage}
		\begin{minipage}{3.1cm}
			\begin{center}
				\begin{overpic}[height=3cm]{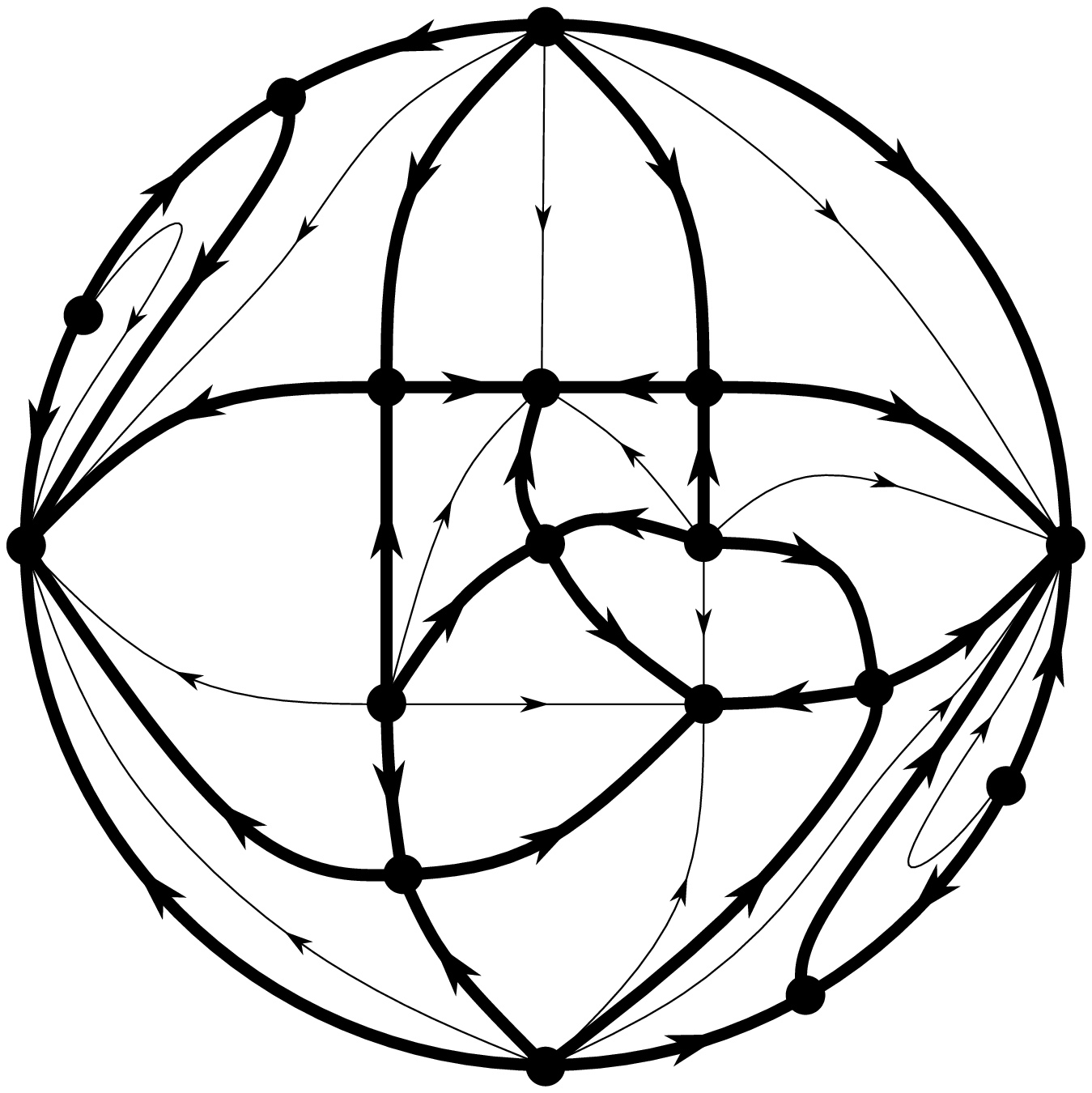} 
				\end{overpic}
				
				Case~$2.6a2$.
			\end{center}
		\end{minipage}
		\begin{minipage}{3.1cm}
			\begin{center}
				\begin{overpic}[height=3cm]{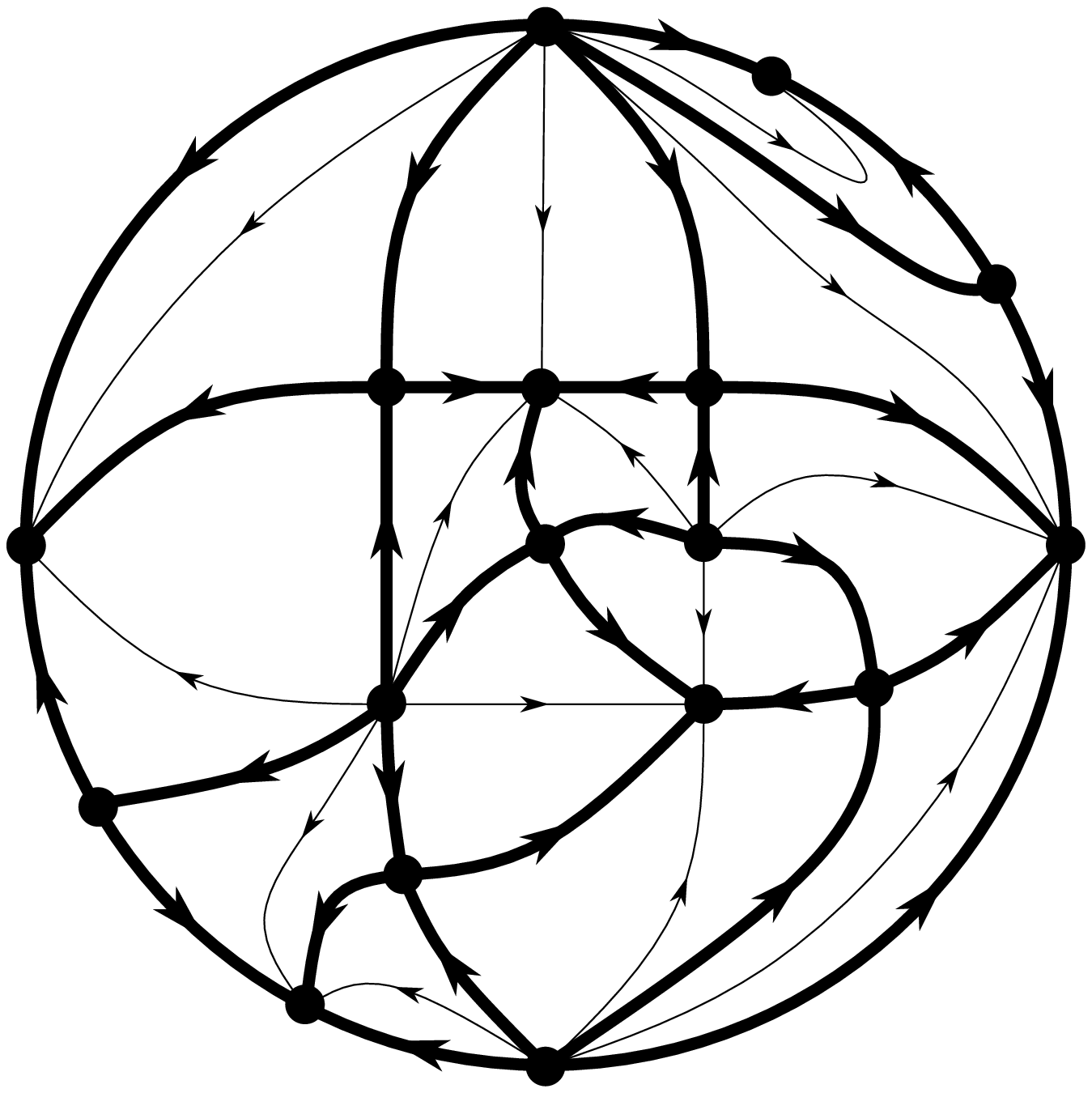} 
				\end{overpic}
				
				Case~$2.6b1$.
			\end{center}
		\end{minipage}	
	\end{center}
	$\;$
	\begin{center}
		\begin{minipage}{3.1cm}
			\begin{center}
				\begin{overpic}[height=3cm]{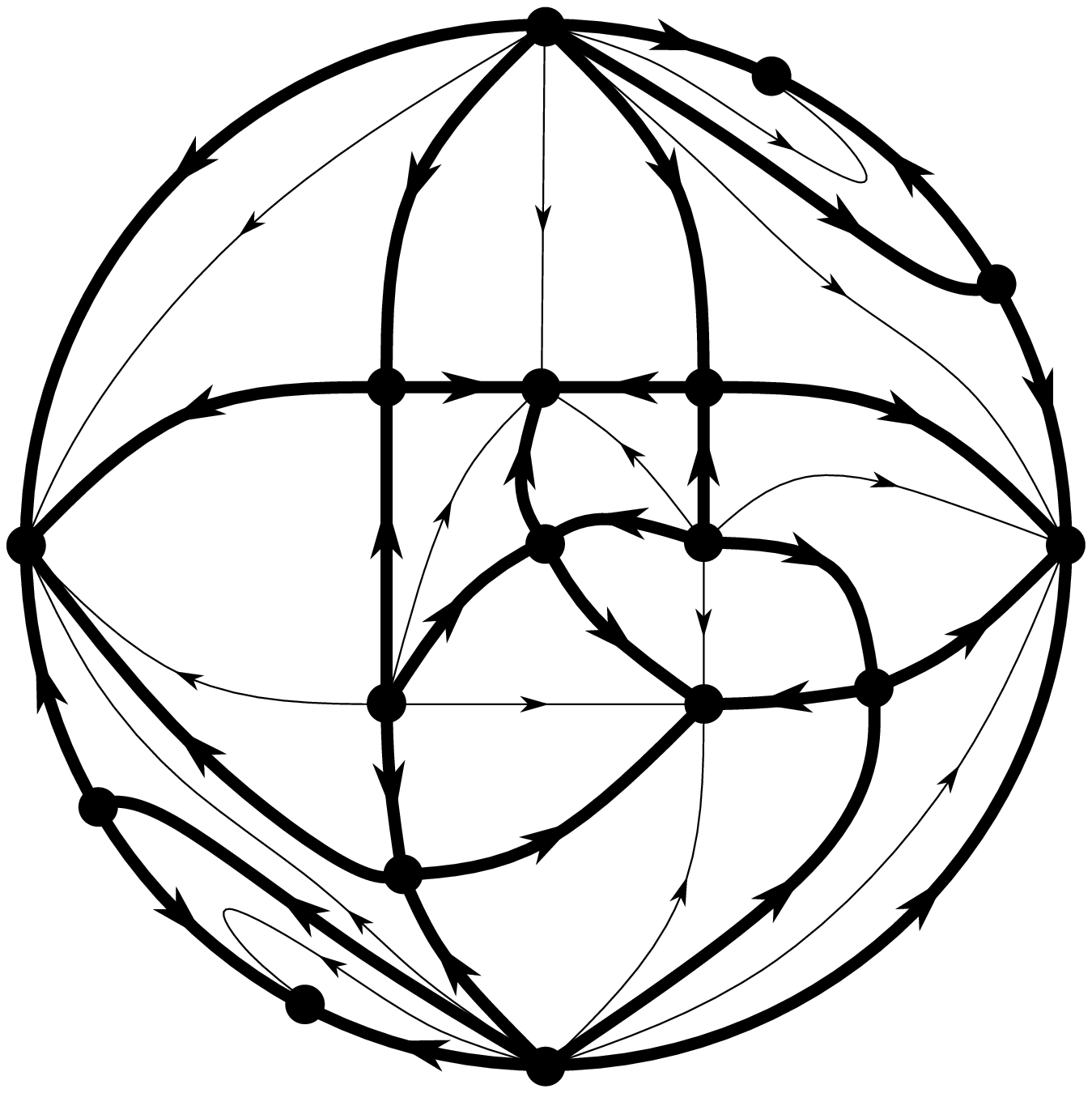} 
				\end{overpic}
				
				Case~$2.6b2$.
			\end{center}
		\end{minipage}
		\begin{minipage}{3.1cm}
			\begin{center}
				\begin{overpic}[height=3cm]{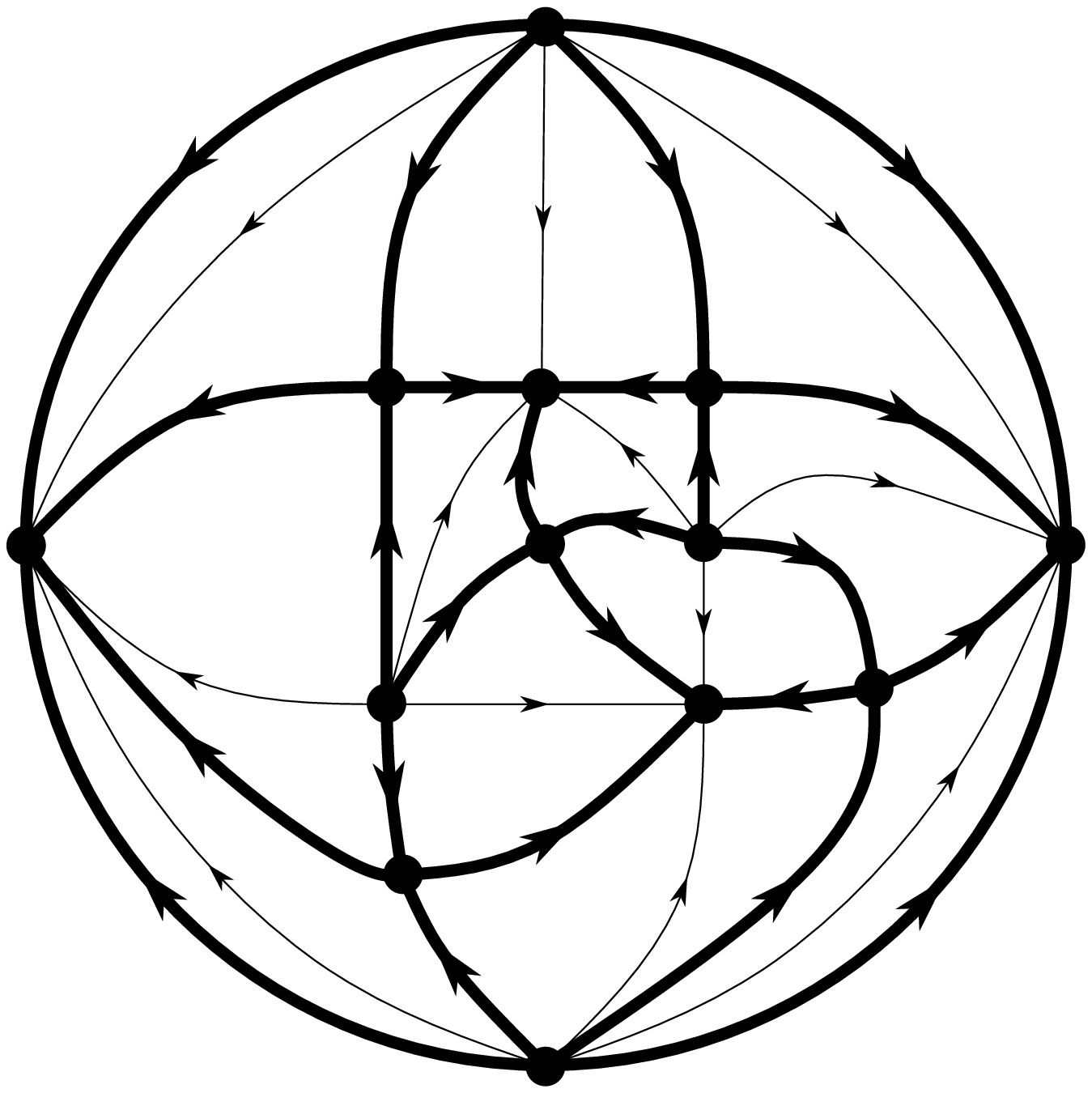} 
				\end{overpic}
				
				Case~$2.6c$.
			\end{center}
		\end{minipage}
		\begin{minipage}{3.1cm}
			\begin{center}
				\begin{overpic}[height=3cm]{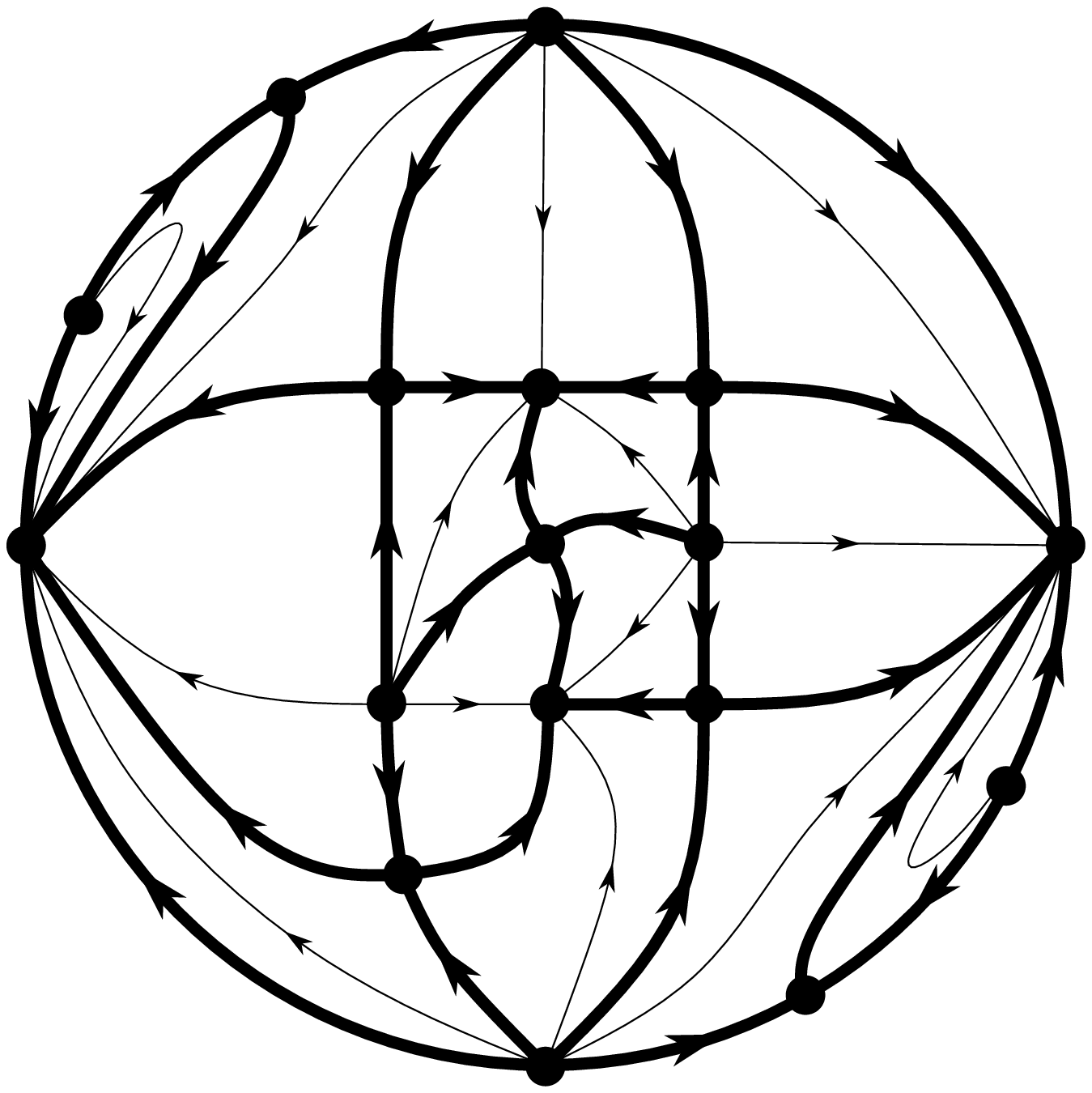} 
				\end{overpic}
				
				Case~$2.7a$.
			\end{center}
		\end{minipage}
		\begin{minipage}{3.1cm}
			\begin{center}
				\begin{overpic}[height=3cm]{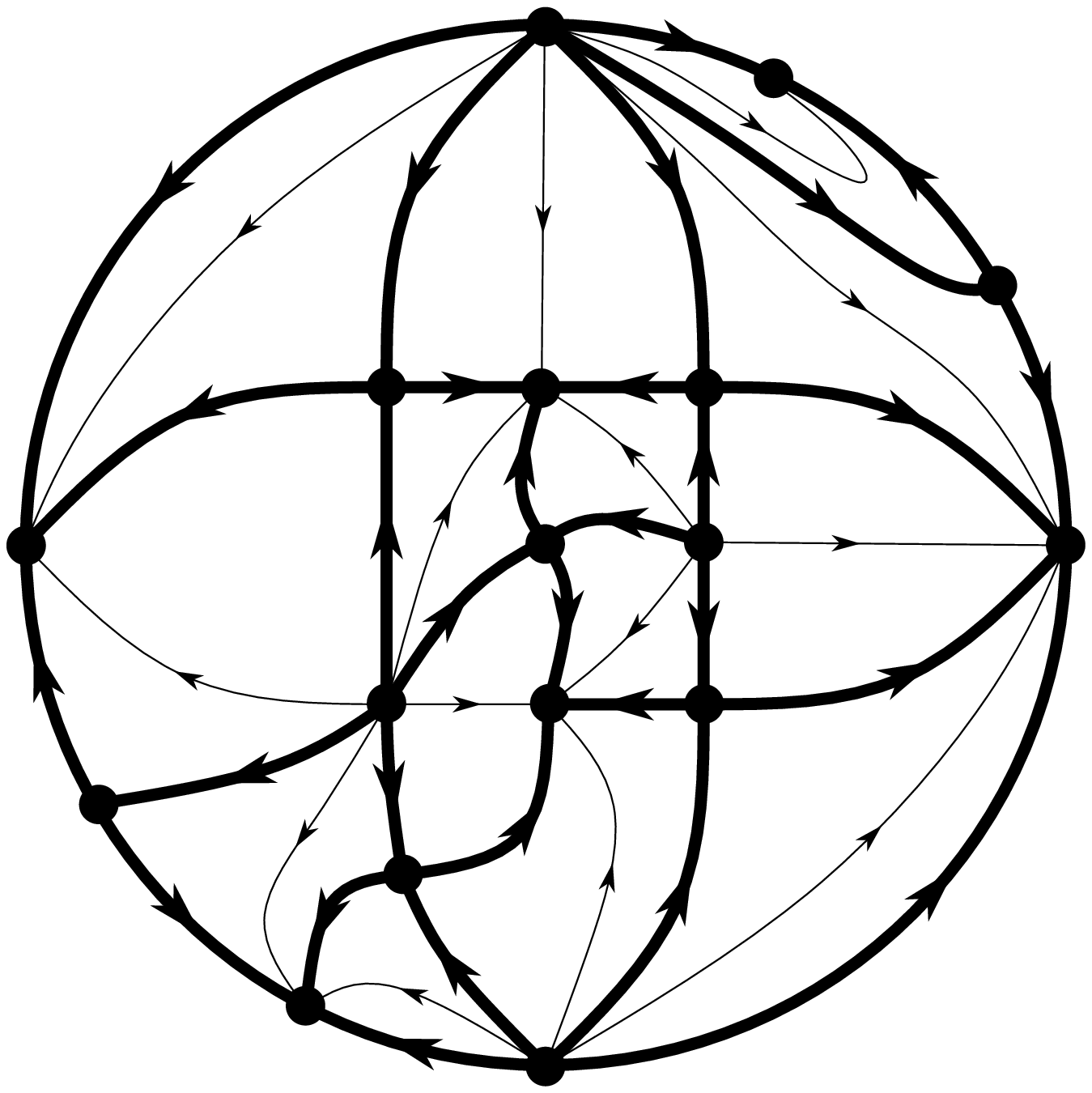} 
				\end{overpic}
				
				Case~$2.7b1$.
			\end{center}
		\end{minipage}
		\begin{minipage}{3.1cm}
			\begin{center}
				\begin{overpic}[height=3cm]{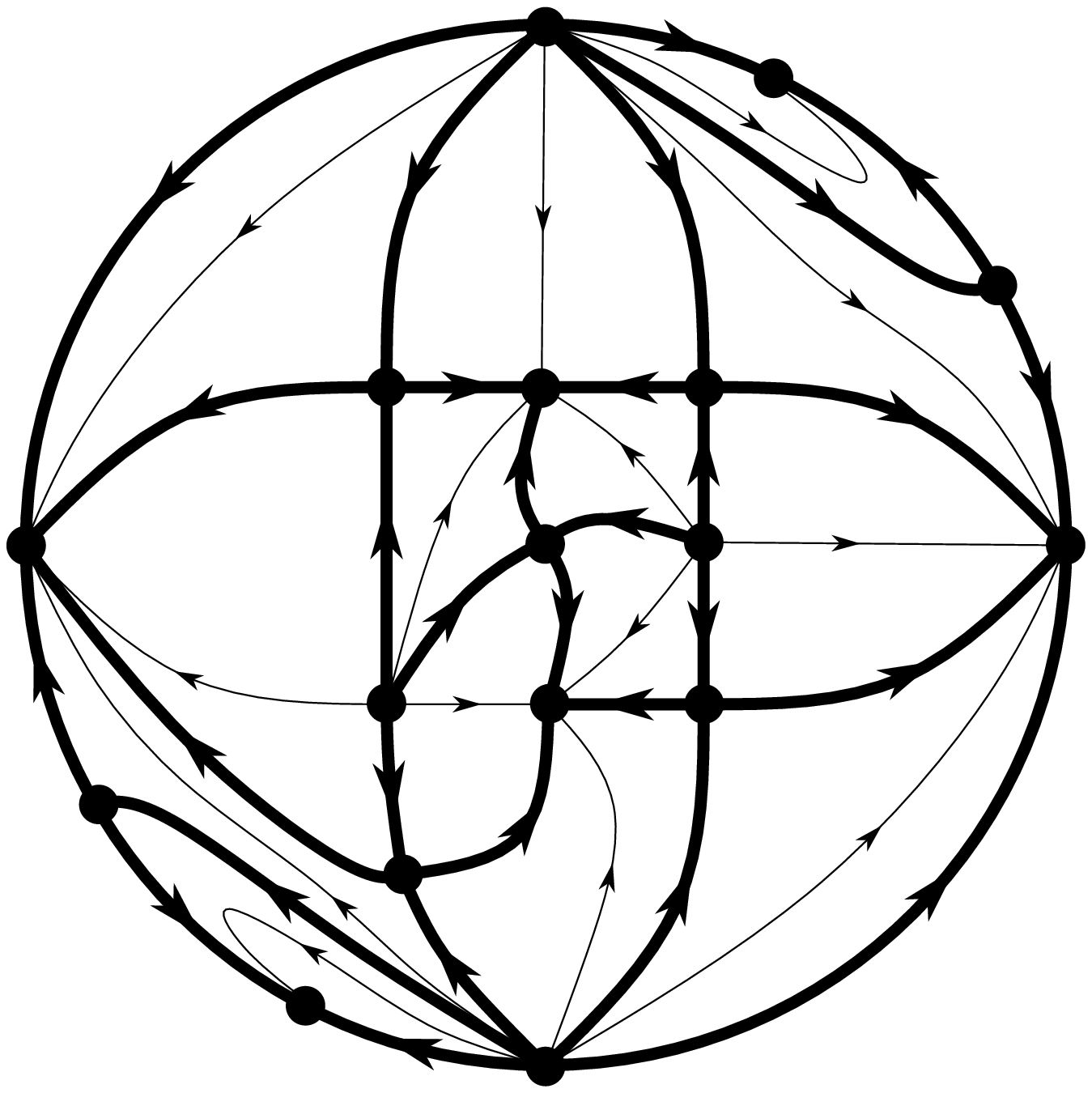} 
				\end{overpic}
				
				Case~$2.7b2$.
			\end{center}
		\end{minipage}
	\end{center}
	$\;$
	\begin{center}
		\begin{minipage}{3.1cm}
			\begin{center}
				\begin{overpic}[height=3cm]{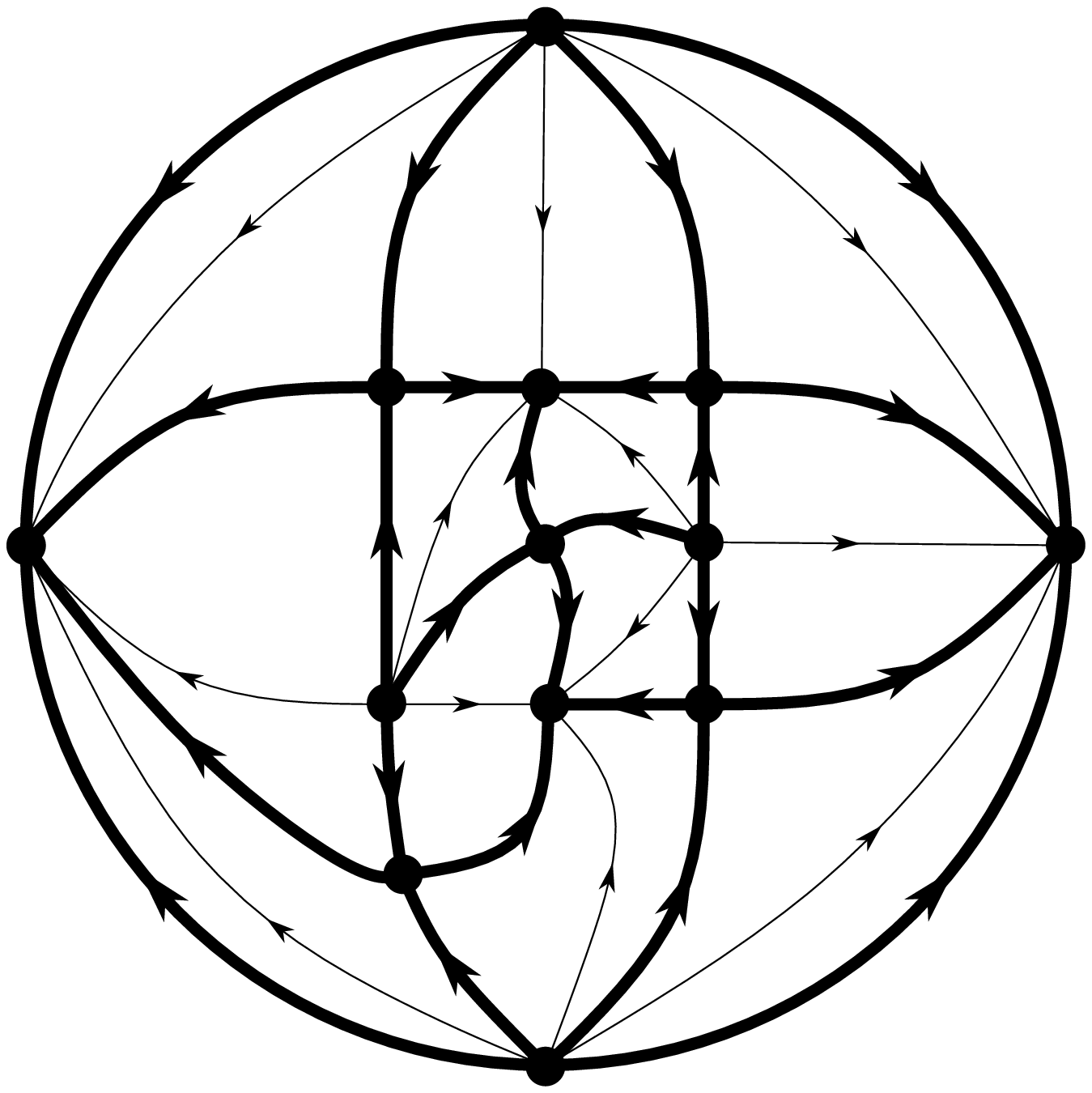} 
				\end{overpic}
				
				Case~$2.7c$.
			\end{center}
		\end{minipage}
		\begin{minipage}{3.1cm}
			\begin{center}
				\begin{overpic}[height=3cm]{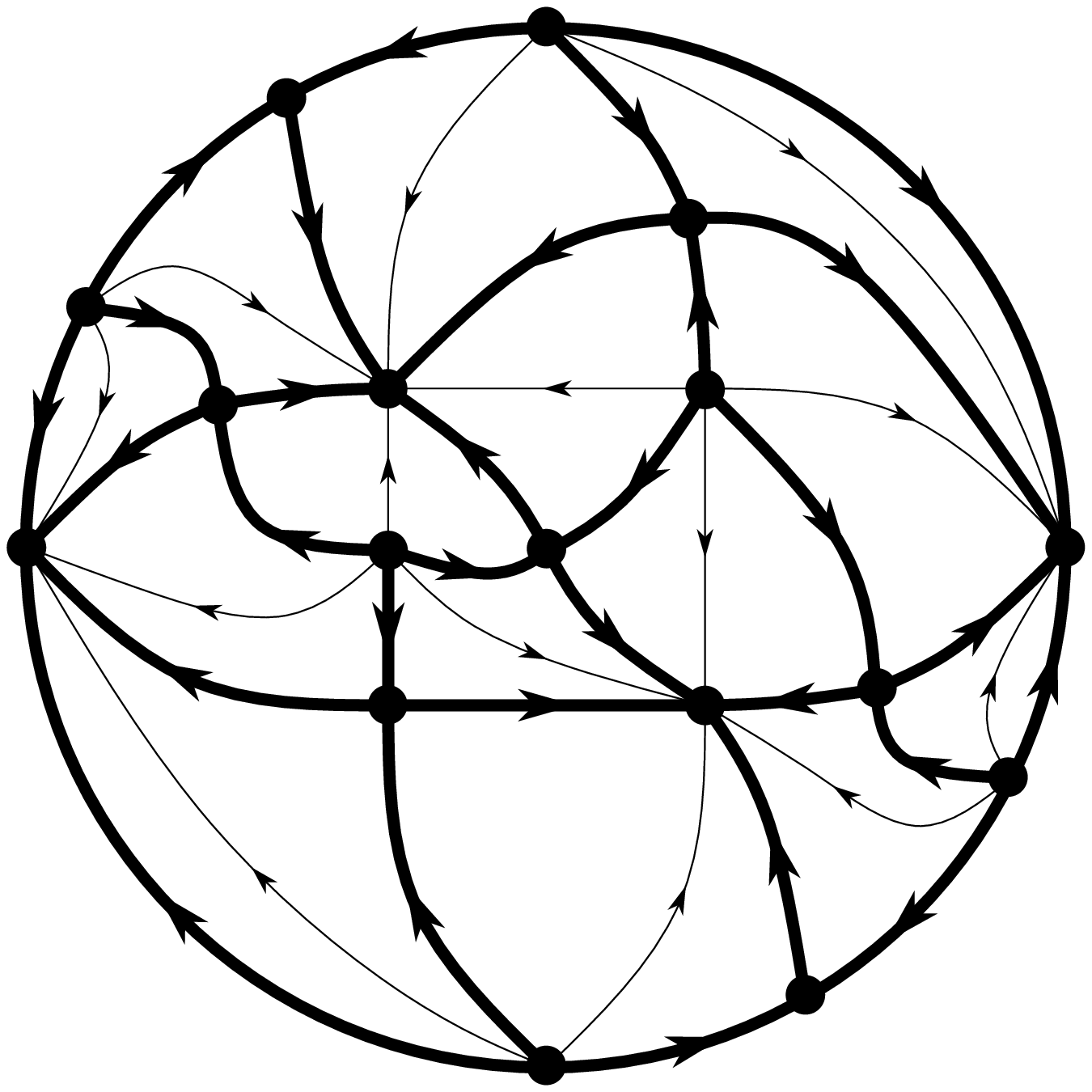} 
				\end{overpic}
				
				Case~$2.8a1$.
			\end{center}
		\end{minipage}	
		\begin{minipage}{3.1cm}
			\begin{center}
				\begin{overpic}[height=3cm]{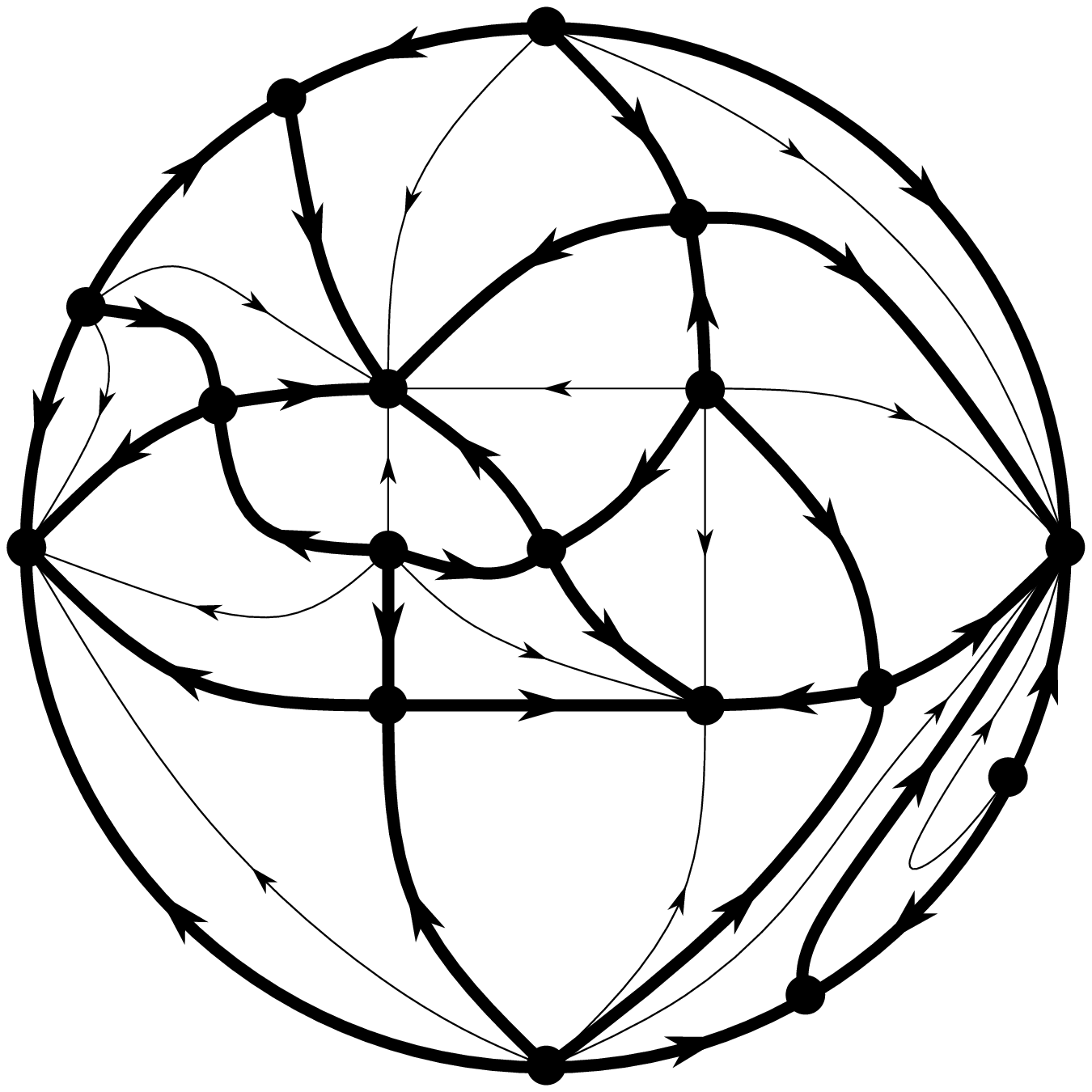} 
				\end{overpic}
				
				Case~$2.8a2$.
			\end{center}
		\end{minipage}
		\begin{minipage}{3.1cm}
			\begin{center}
				\begin{overpic}[height=3cm]{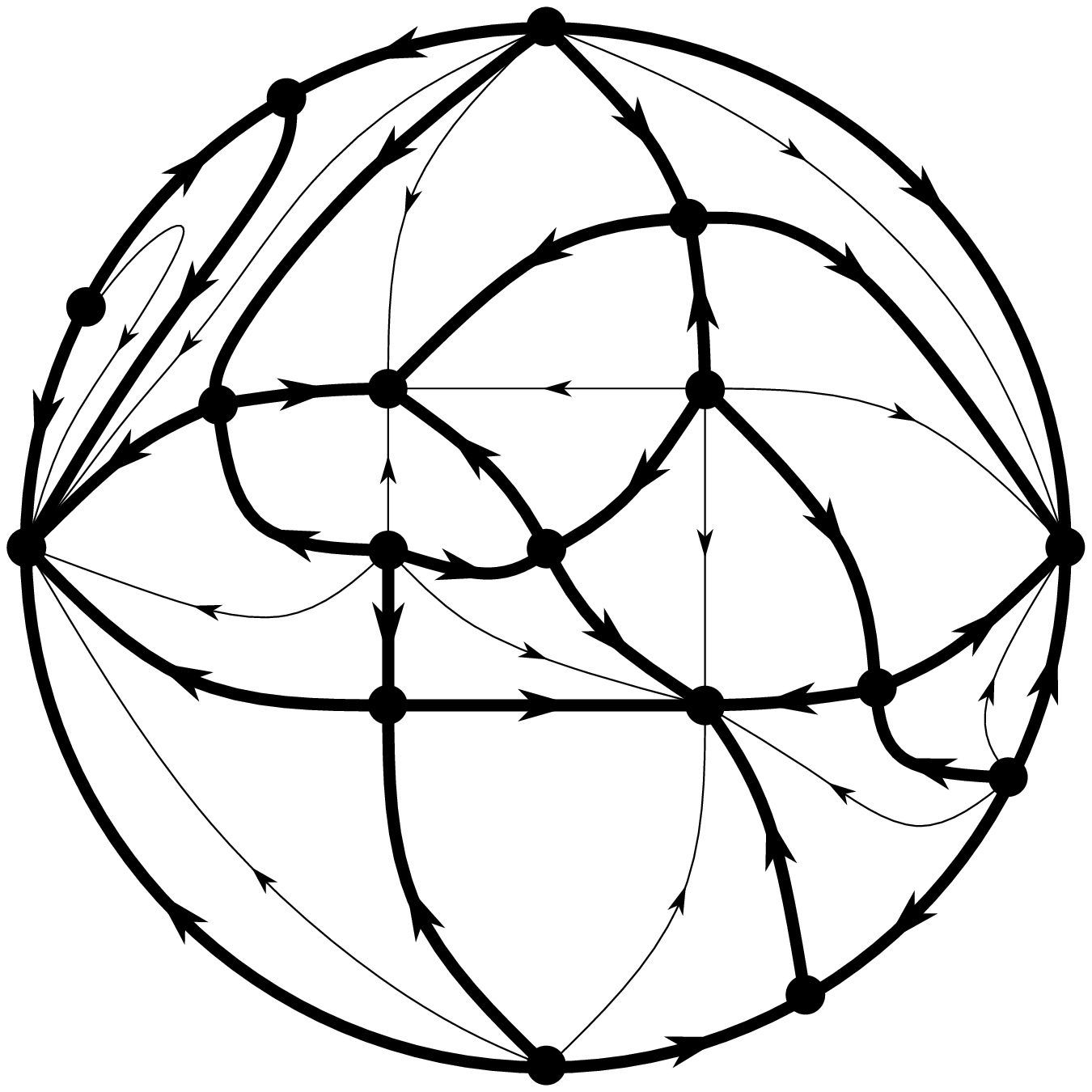} 
				\end{overpic}
				
				Case~$2.8a3$.
			\end{center}
		\end{minipage}
		\begin{minipage}{3.1cm}
			\begin{center}
				\begin{overpic}[height=3cm]{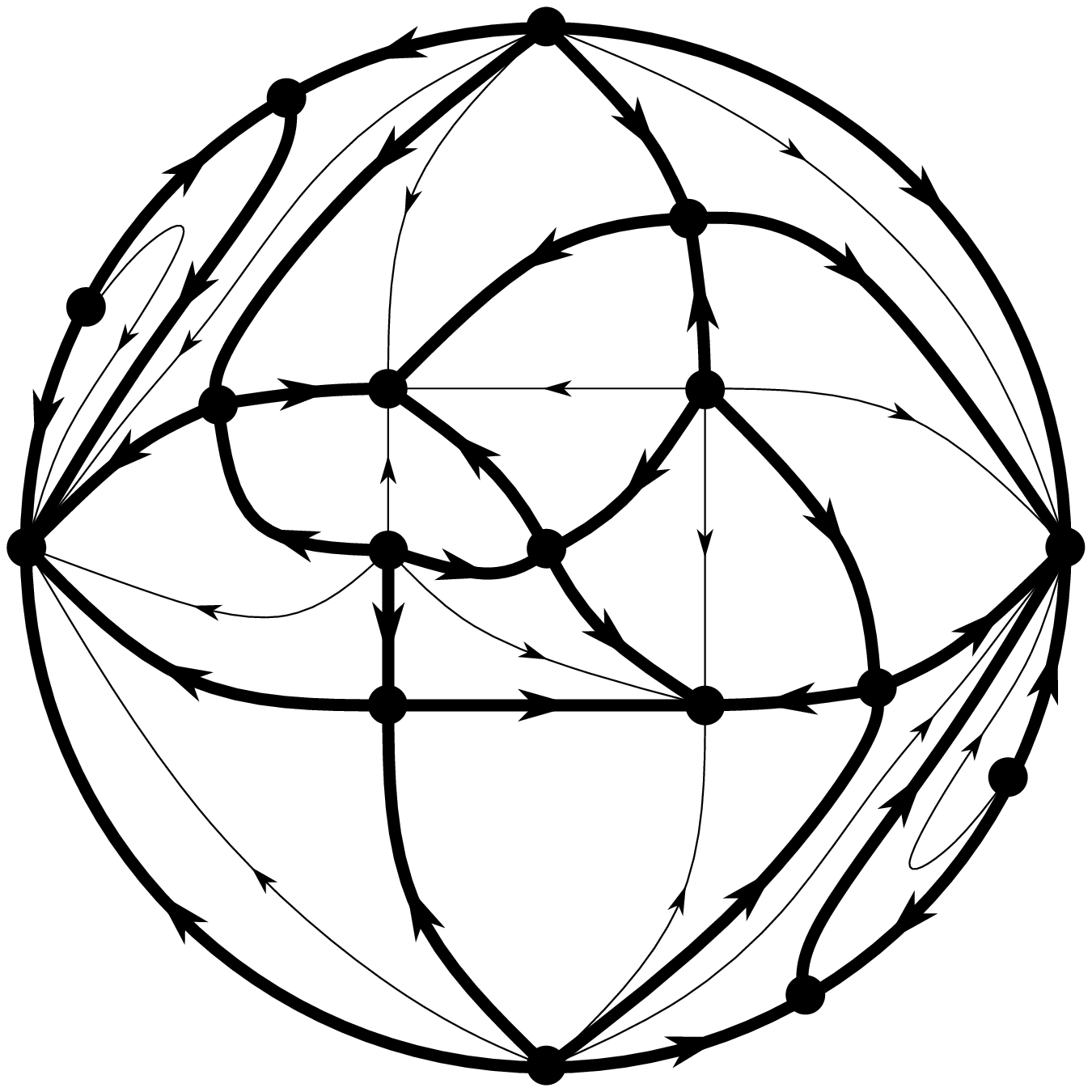} 
				\end{overpic}
				
				Case~$2.8a4$.
			\end{center}
		\end{minipage}
	\end{center}
	$\;$
	\begin{center}
		\begin{minipage}{3.1cm}
			\begin{center}
				\begin{overpic}[height=3cm]{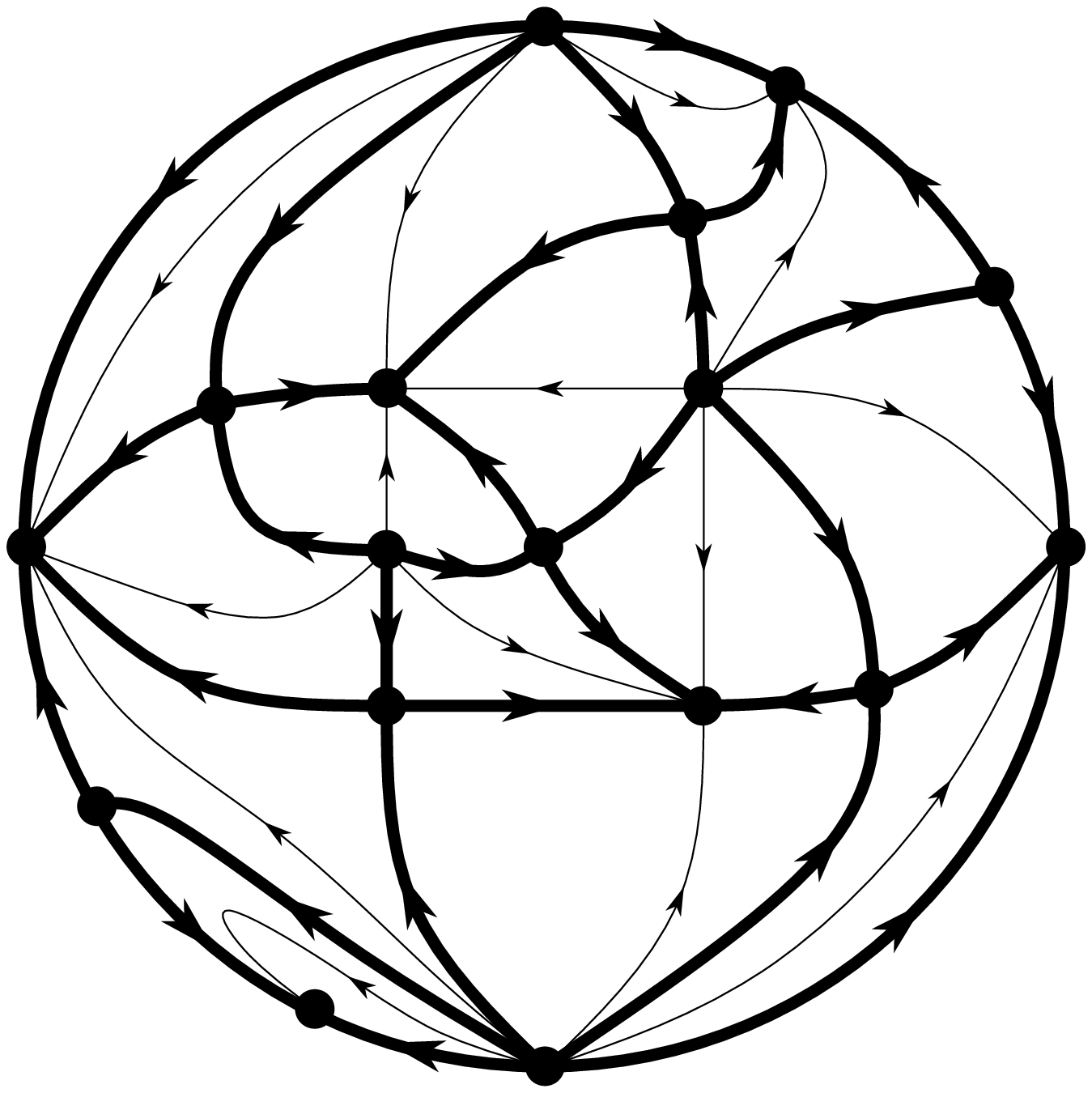} 
				\end{overpic}
				
				Case~$2.8b1$.
			\end{center}
		\end{minipage}
		\begin{minipage}{3.1cm}
			\begin{center}
				\begin{overpic}[height=3cm]{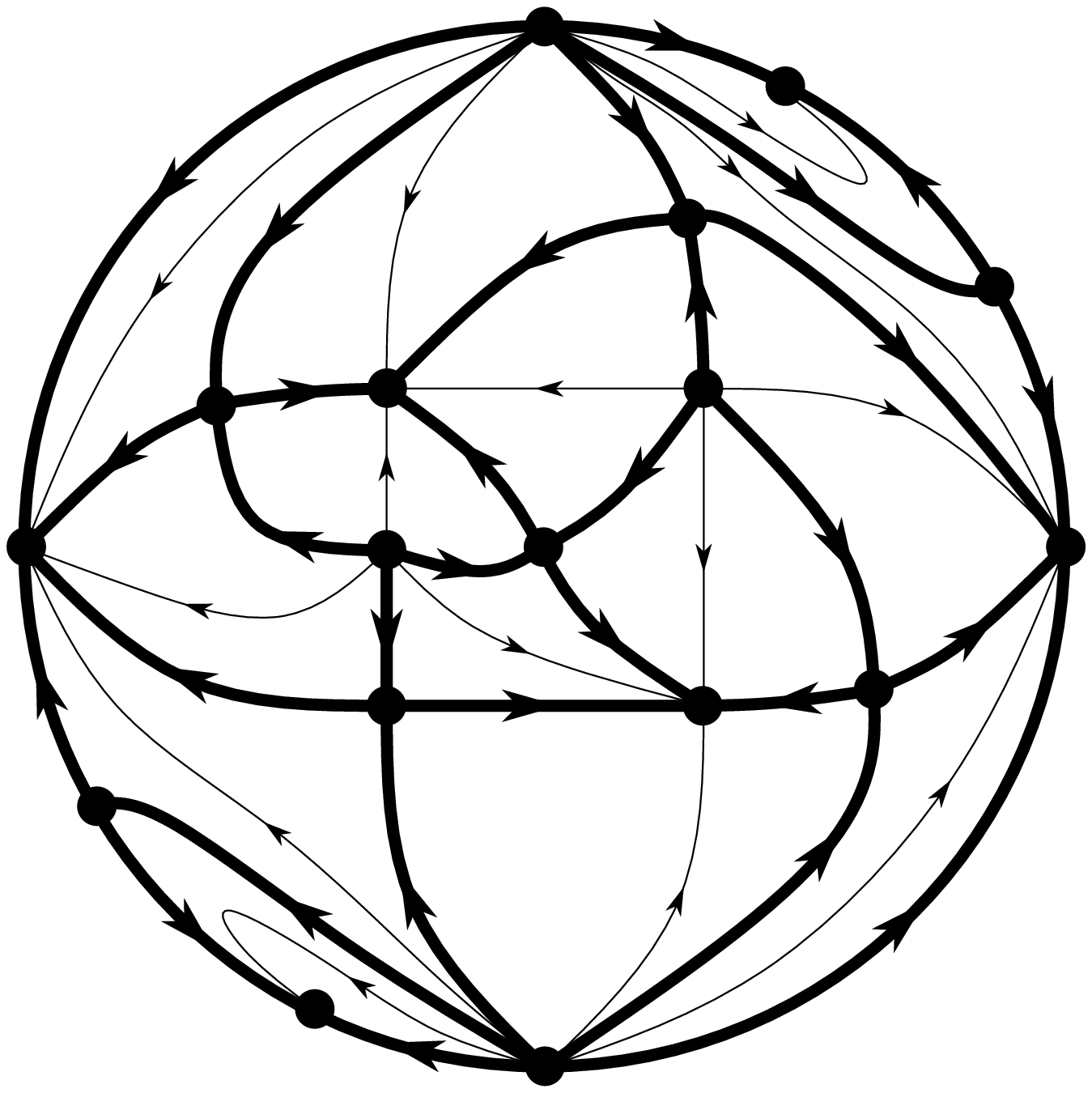} 
				\end{overpic}
				
				Case~$2.8b2$.
			\end{center}
		\end{minipage}
		\begin{minipage}{3.1cm}
			\begin{center}
				\begin{overpic}[height=3cm]{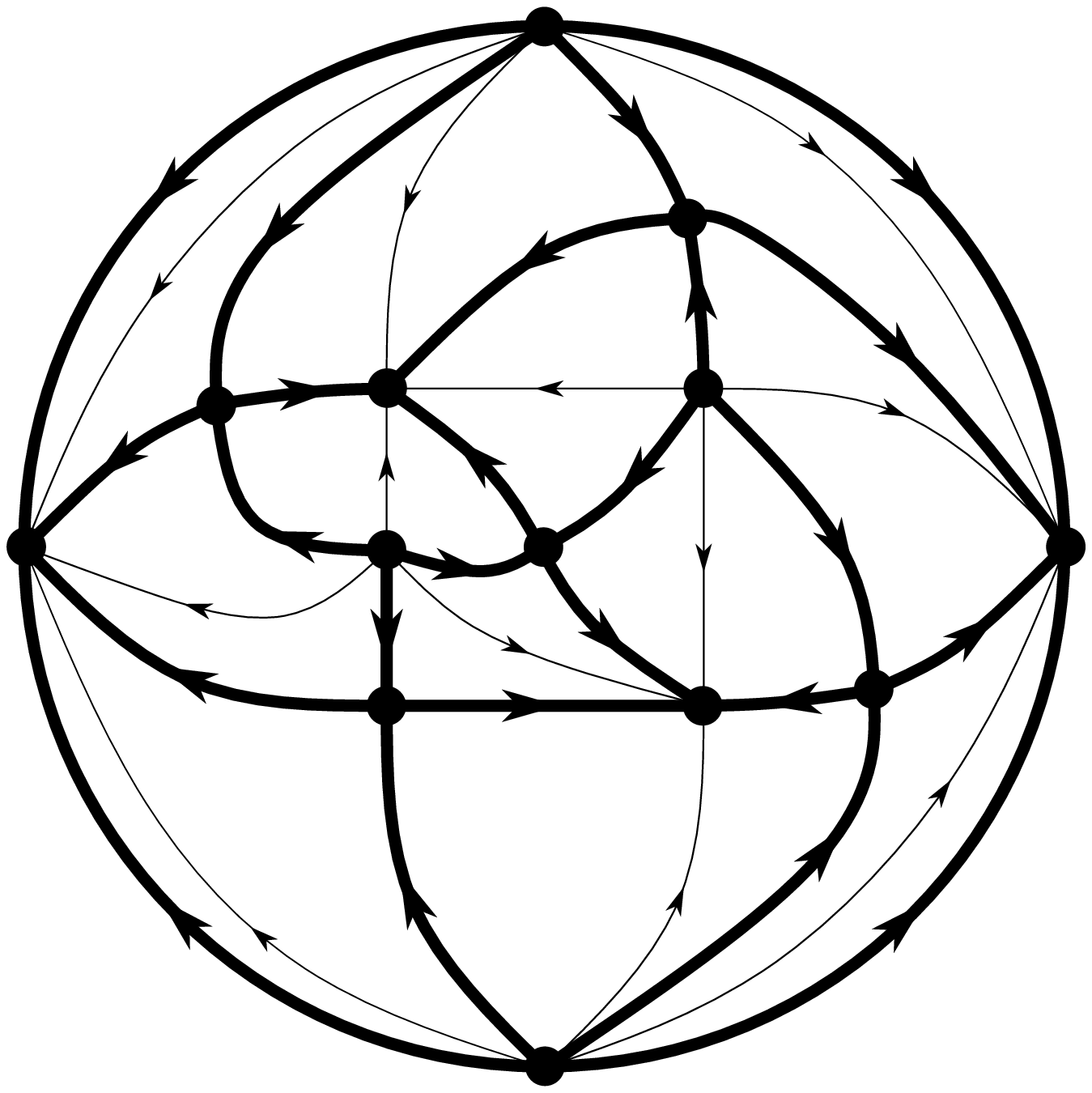} 
				\end{overpic}
				
				Case~$2.8c$.
			\end{center}
		\end{minipage}	
		\begin{minipage}{3.1cm}
			\begin{center}
				\begin{overpic}[height=3cm]{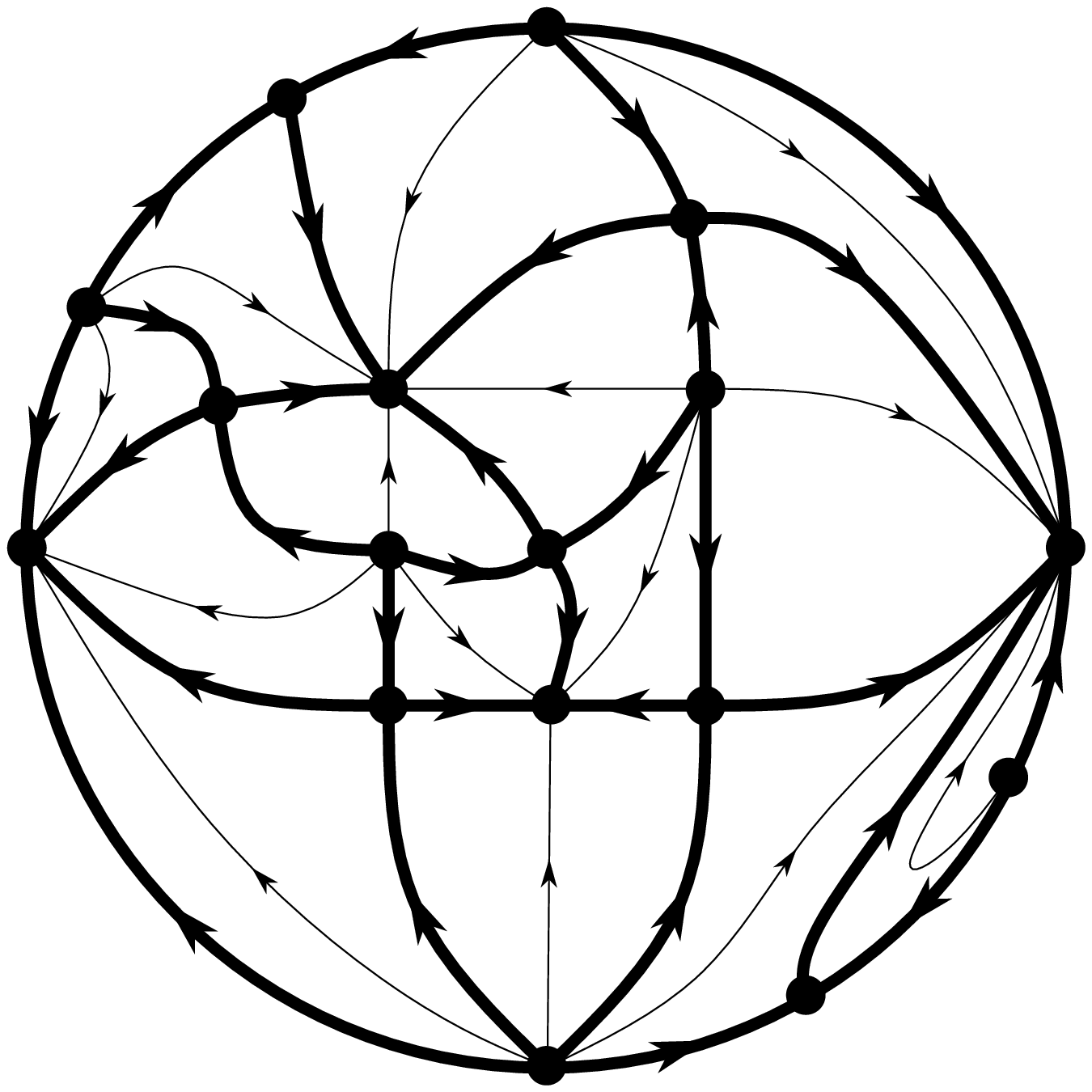} 
				\end{overpic}
				
				Case~$2.9a1$.
			\end{center}
		\end{minipage}
		\begin{minipage}{3.1cm}
			\begin{center}
				\begin{overpic}[height=3cm]{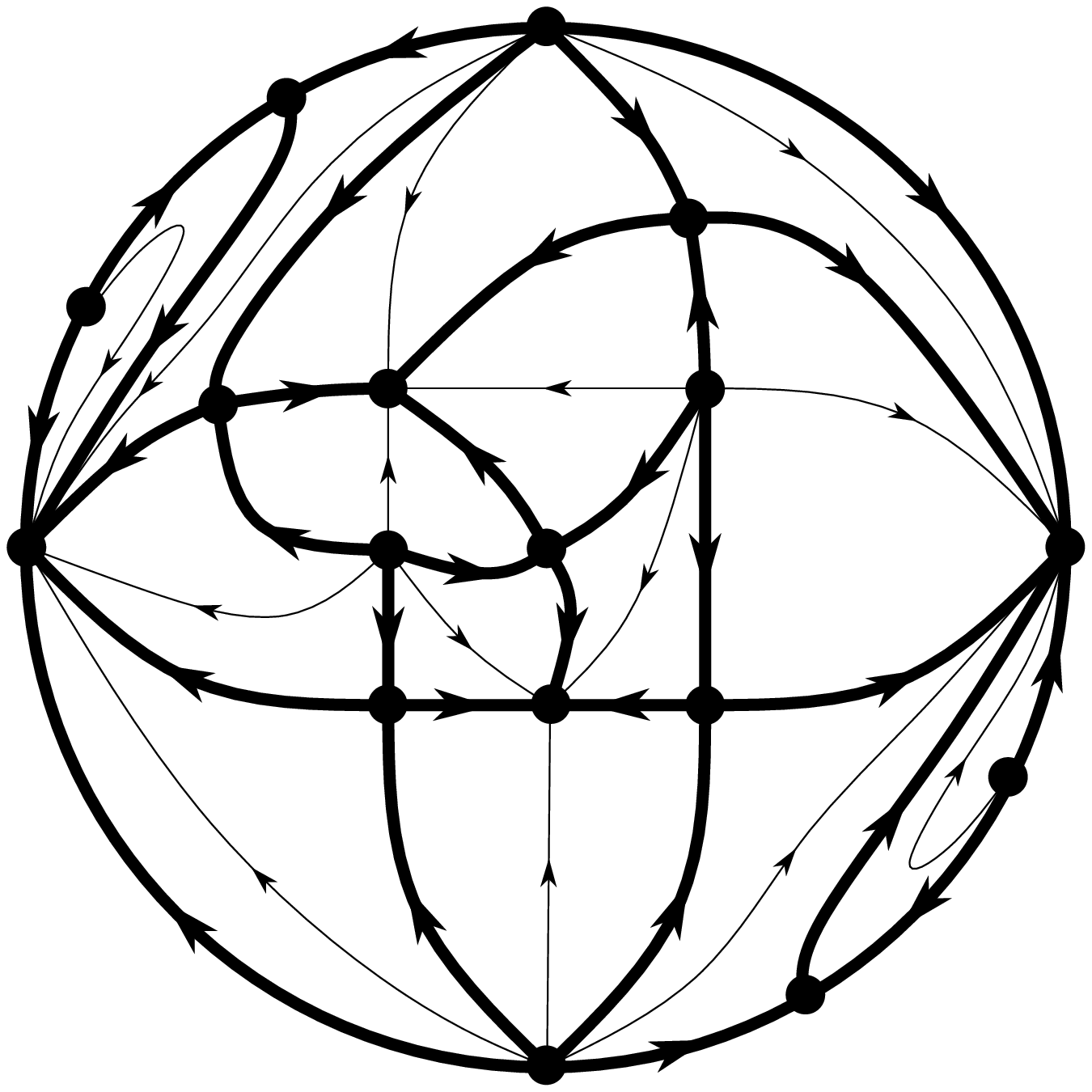} 
				\end{overpic}
				
				Case~$2.9a2$.
			\end{center}
		\end{minipage}
	\end{center}
	$\;$
	\begin{center}
		\begin{minipage}{3.1cm}
			\begin{center}
				\begin{overpic}[height=3cm]{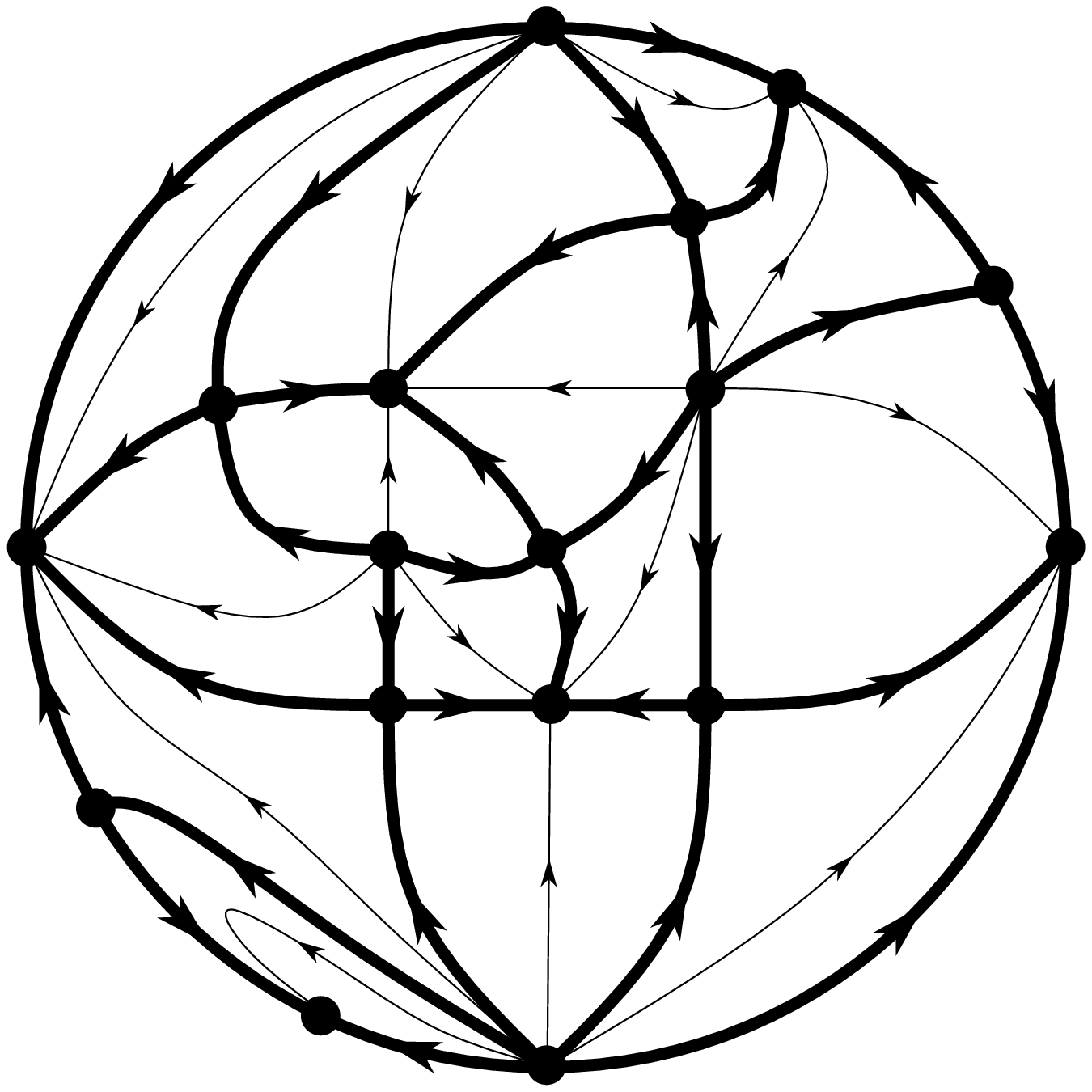} 
				\end{overpic}
				
				Case~$2.9b1$.
			\end{center}
		\end{minipage}	
		\begin{minipage}{3.1cm}
			\begin{center}
				\begin{overpic}[height=3cm]{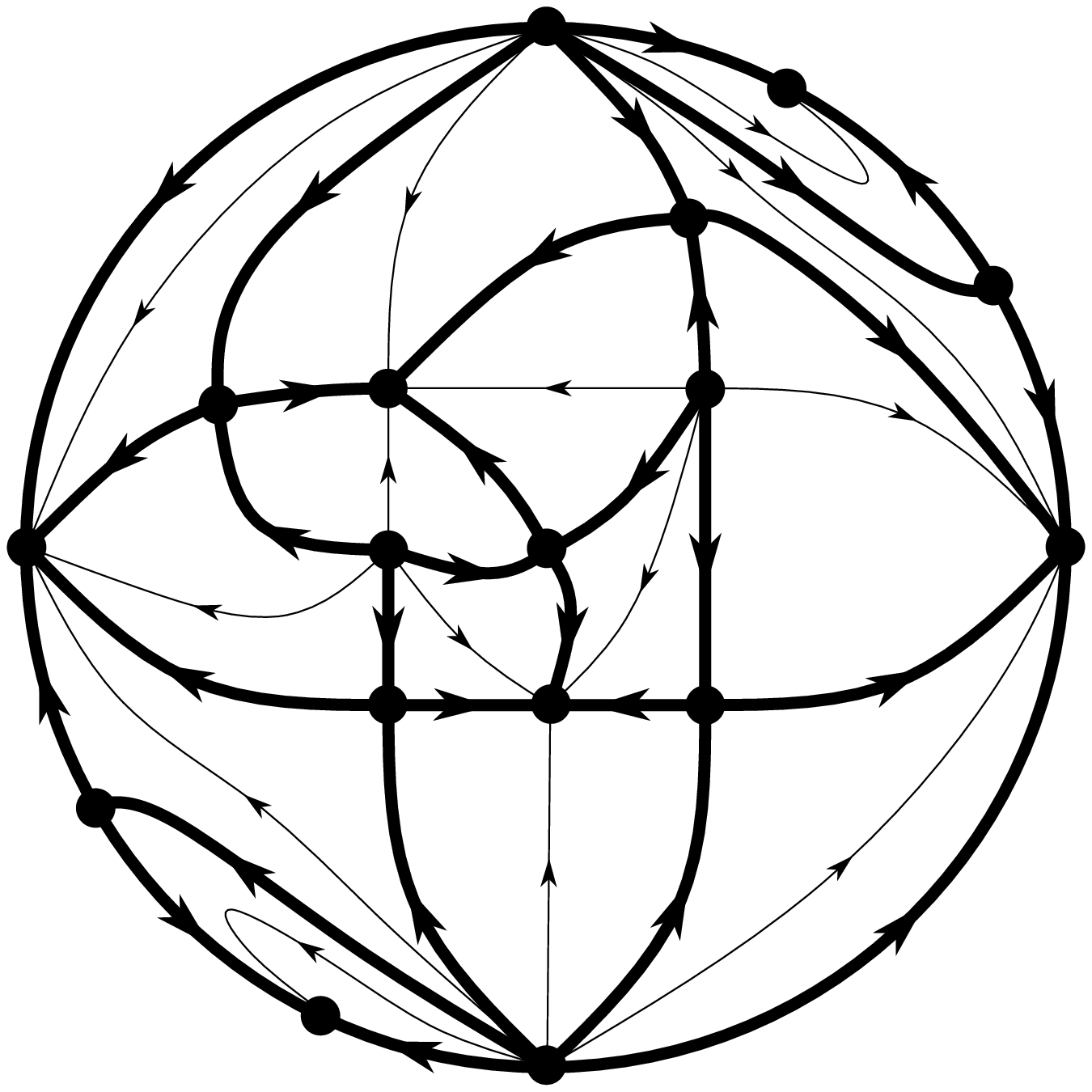} 
				\end{overpic}
				
				Case~$2.9b2$.
			\end{center}
		\end{minipage}
		\begin{minipage}{3.1cm}
			\begin{center}
				\begin{overpic}[height=3cm]{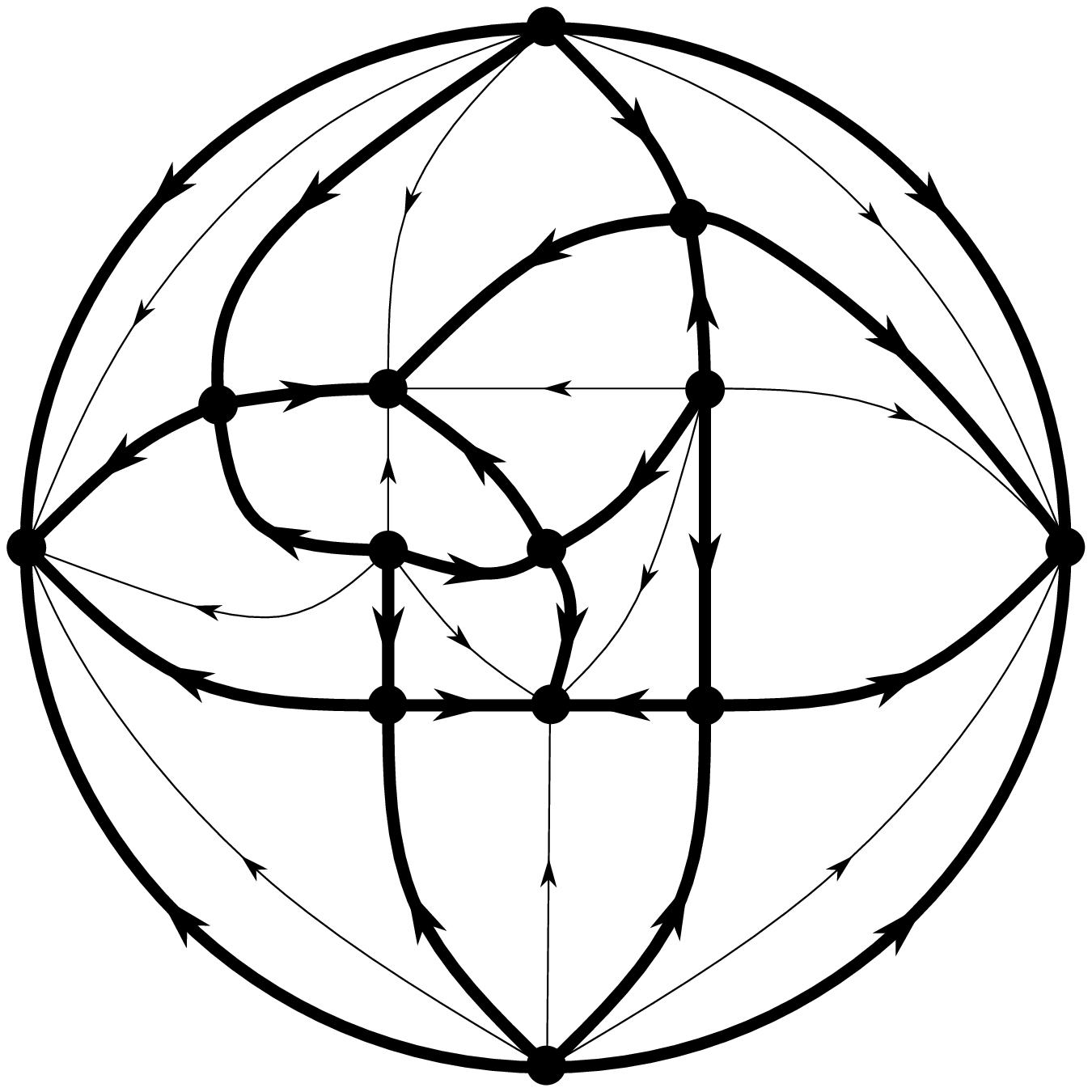} 
				\end{overpic}
				
				Case~$2.9c$.
			\end{center}
		\end{minipage}
	\end{center}
	\caption{Phase portraits from Cases~$2.4$ to $2.9$.}\label{Case1.2b}
\end{figure}

\begin{figure}[h]
	\begin{center}
		\begin{minipage}{3.1cm}
			\begin{center}
				\begin{overpic}[height=3cm]{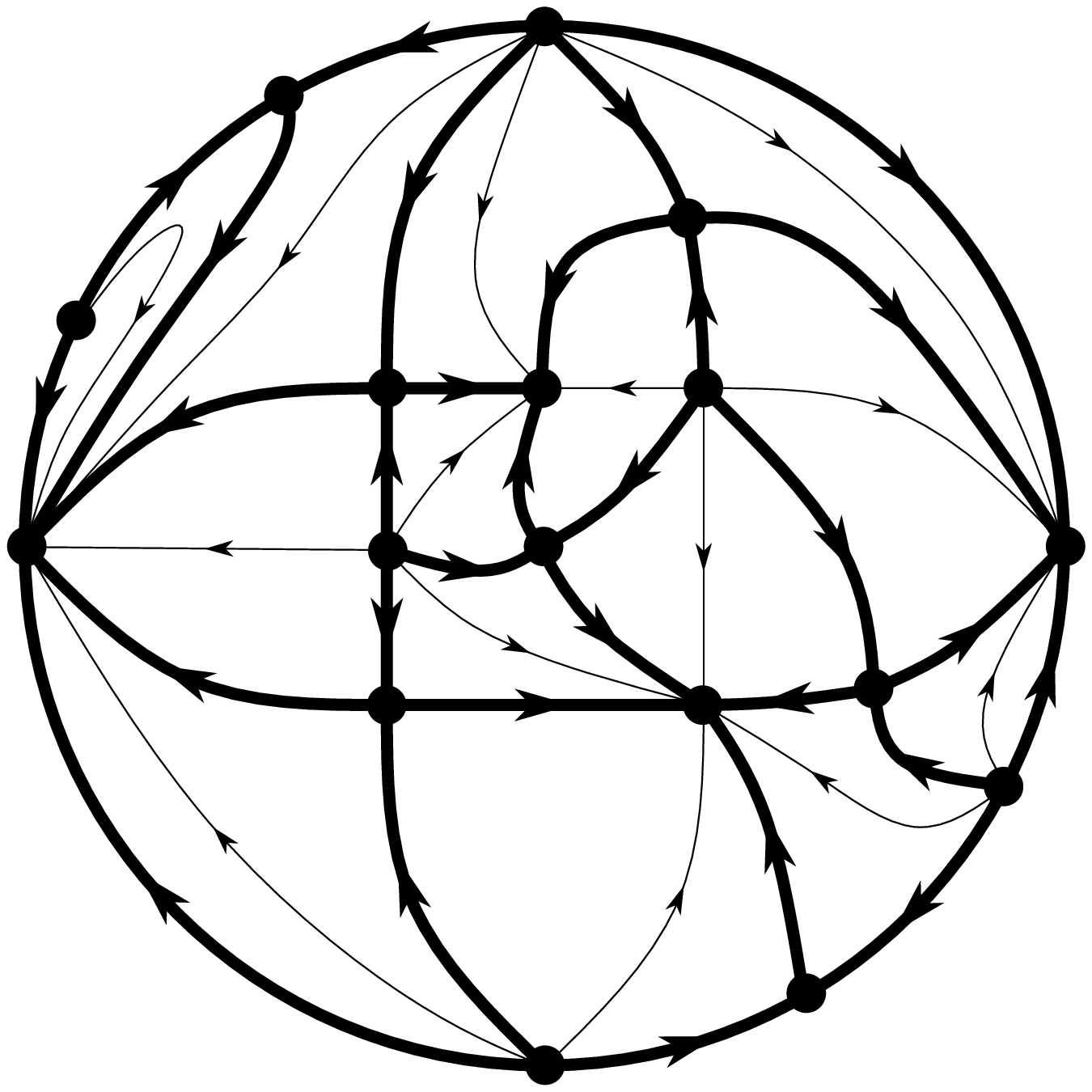} 
				\end{overpic}
				
				Case~$2.10a1$.
			\end{center}
		\end{minipage}
		\begin{minipage}{3.1cm}
			\begin{center}
				\begin{overpic}[height=3cm]{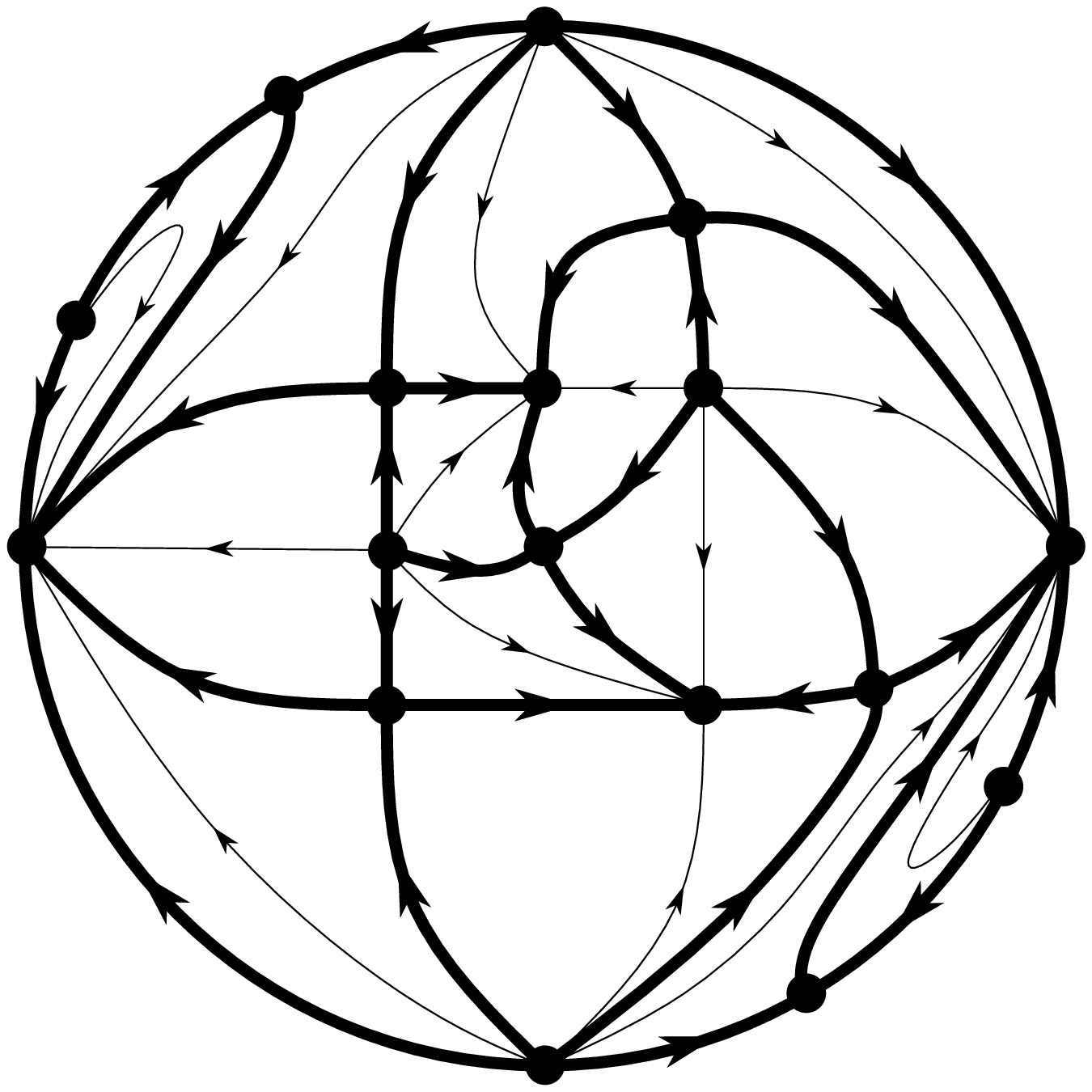} 
				\end{overpic}
				
				Case~$2.10a2$.
			\end{center}
		\end{minipage}
		\begin{minipage}{3.1cm}
			\begin{center}
				\begin{overpic}[height=3cm]{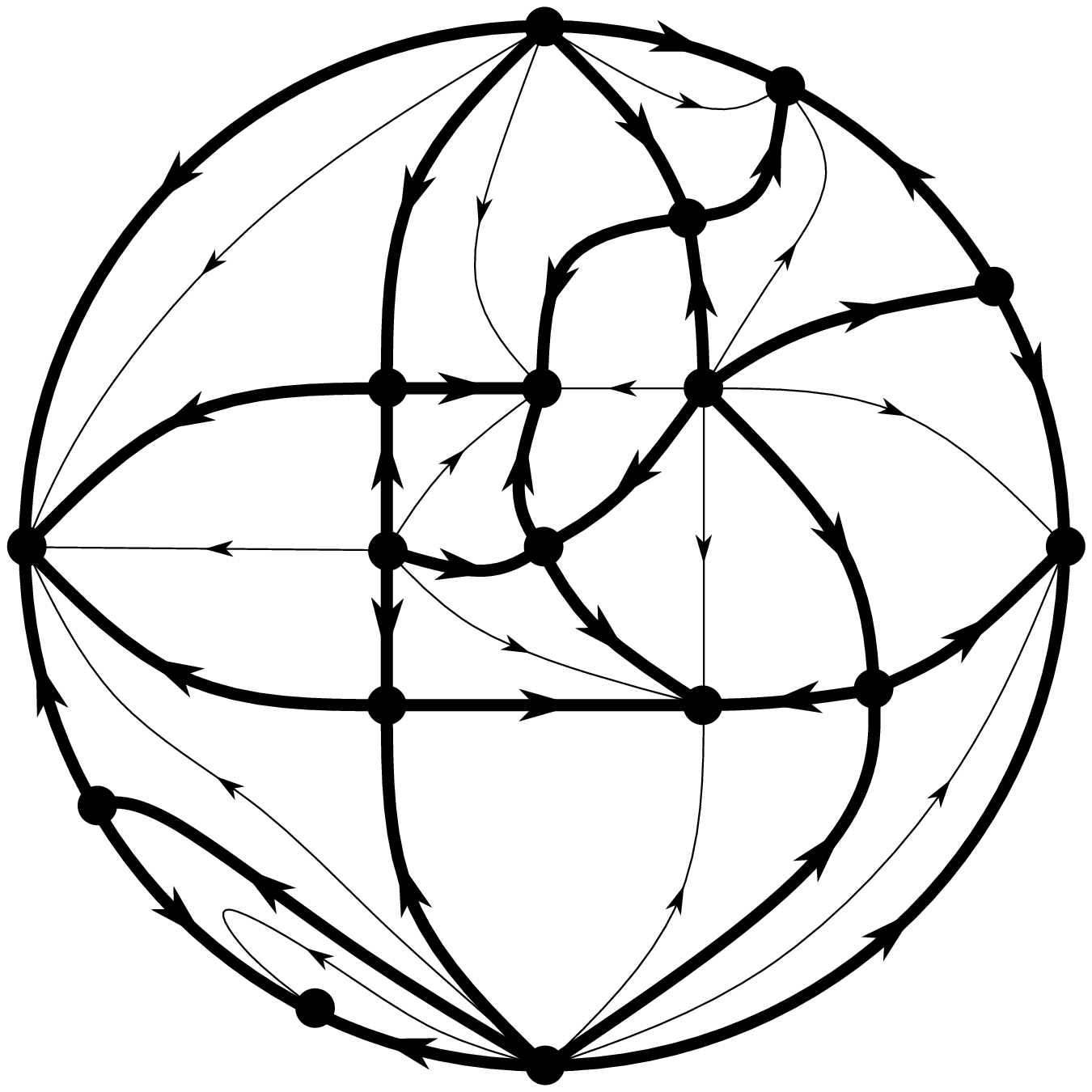} 
				\end{overpic}
				
				Case~$2.10b1$.
			\end{center}
		\end{minipage}
		\begin{minipage}{3.1cm}
			\begin{center}
				\begin{overpic}[height=3cm]{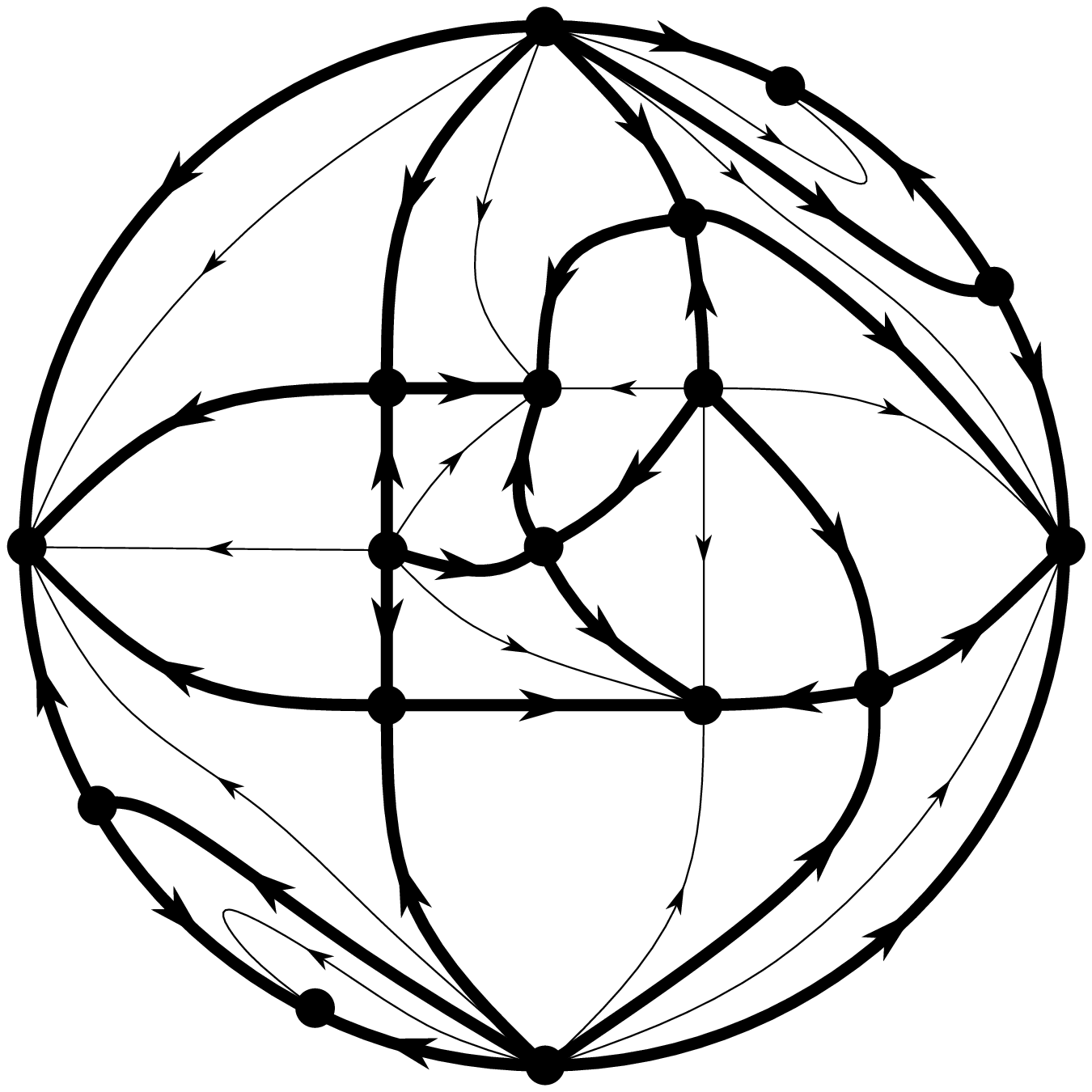} 
				\end{overpic}
				
				Case~$2.10b2$.
			\end{center}
		\end{minipage}
		\begin{minipage}{3.1cm}
			\begin{center}
				\begin{overpic}[height=3cm]{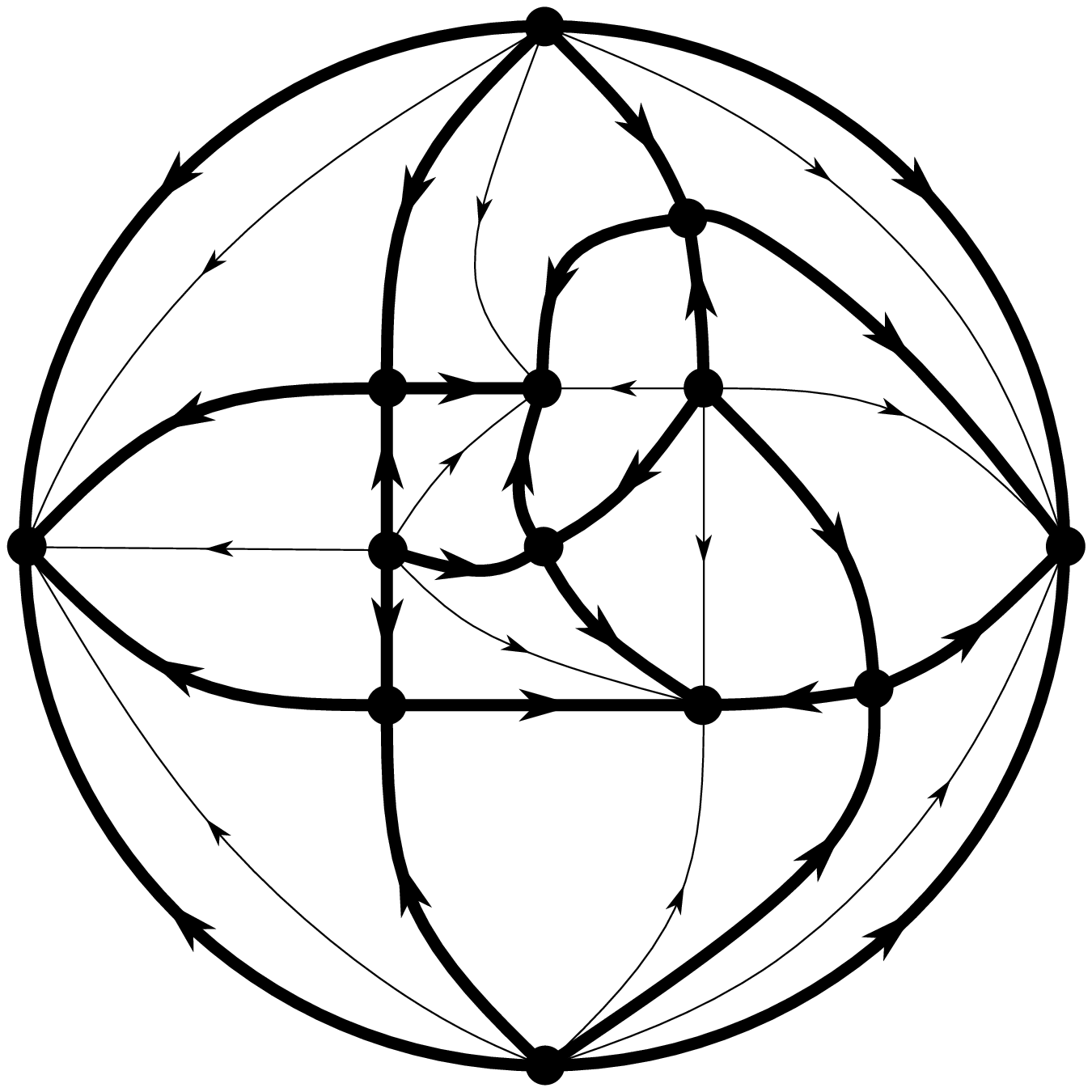} 
				\end{overpic}
				
				Case~$2.10c$.
			\end{center}
		\end{minipage}	
	\end{center}
	$\;$
	\begin{center}
		\begin{minipage}{3.1cm}
			\begin{center}
				\begin{overpic}[height=3cm]{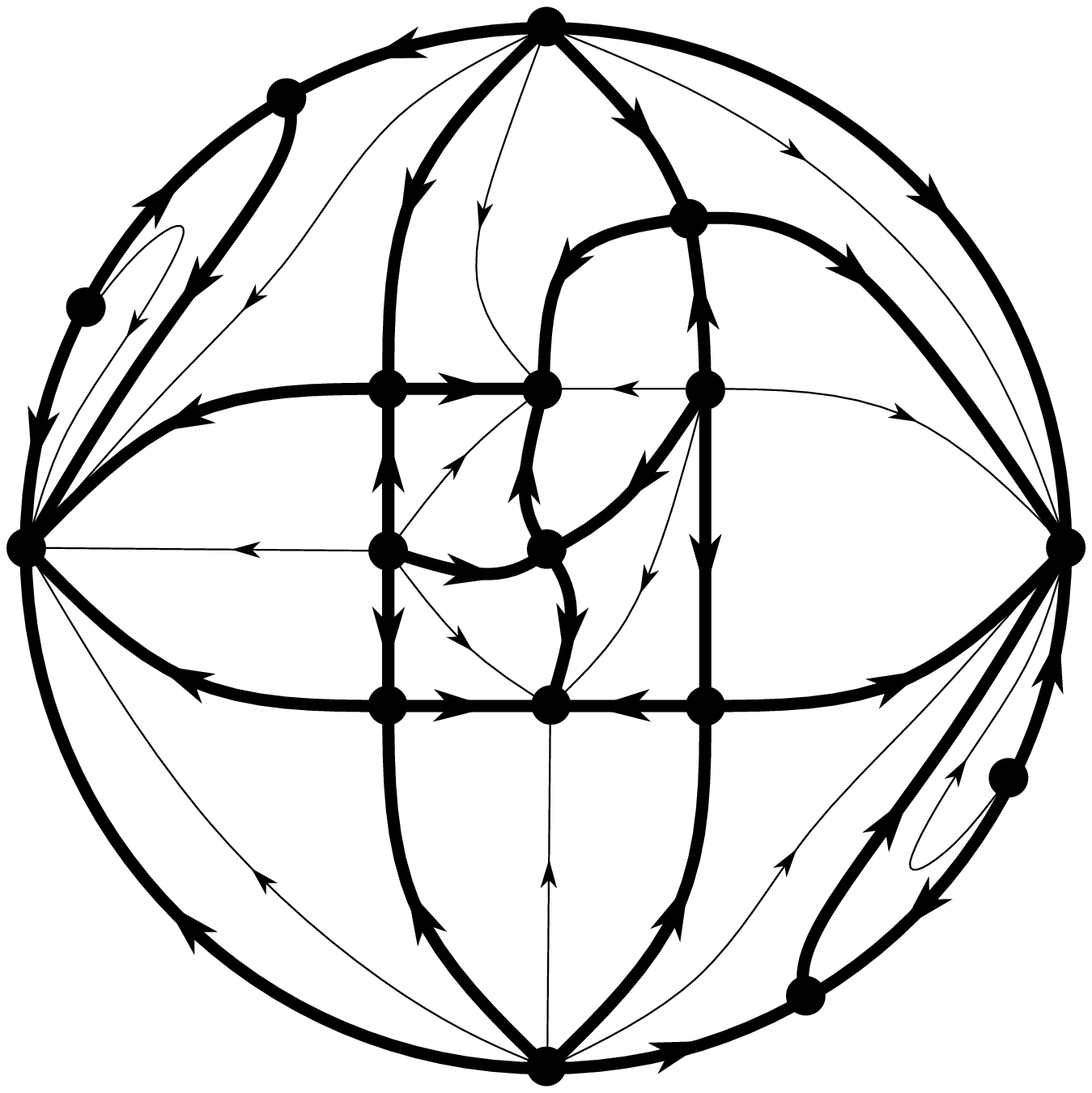} 
				\end{overpic}
				
				Case~$2.11a$.
			\end{center}
		\end{minipage}
		\begin{minipage}{3.1cm}
			\begin{center}
				\begin{overpic}[height=3cm]{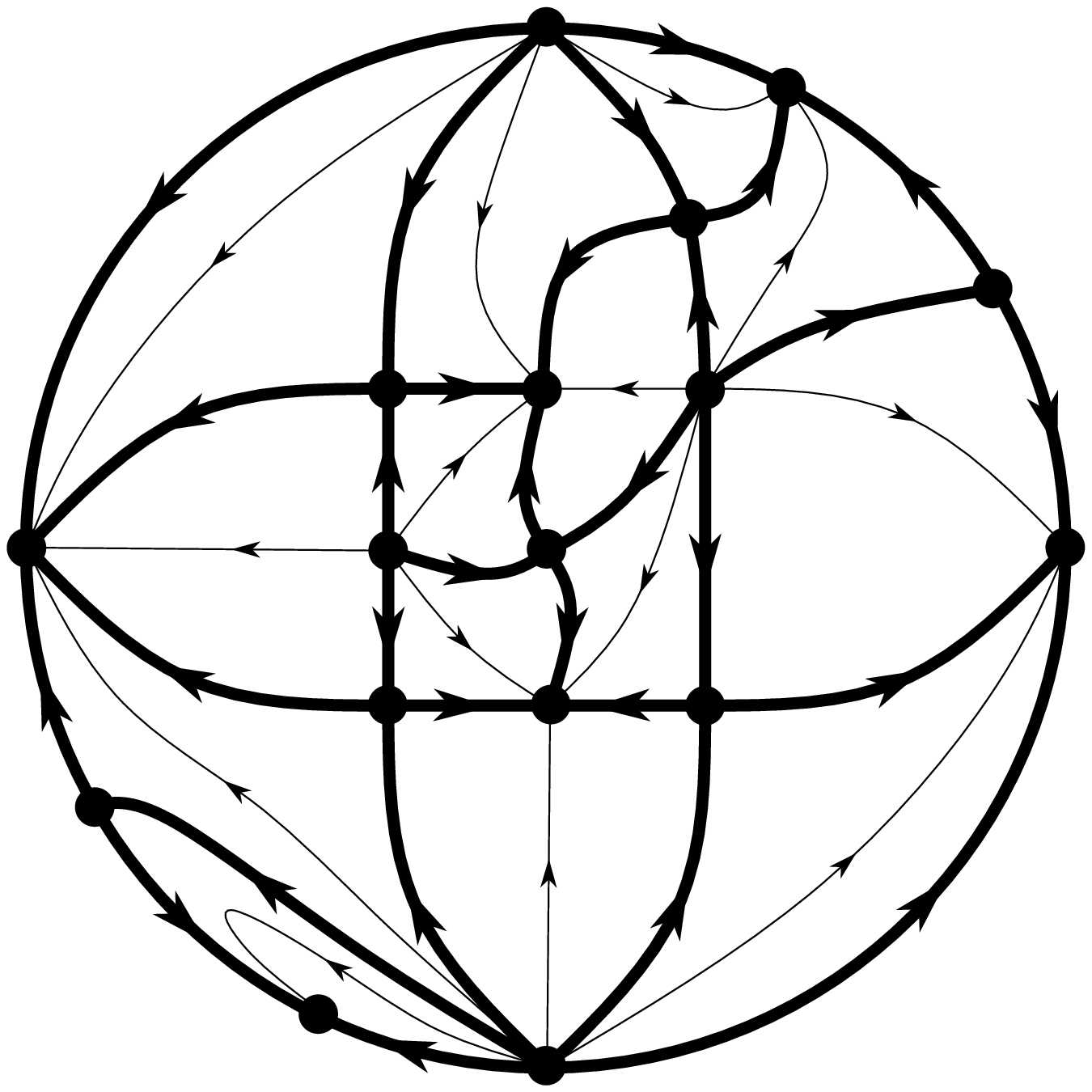} 
				\end{overpic}
				
				Case~$2.11b1$.
			\end{center}
		\end{minipage}
		\begin{minipage}{3.1cm}
			\begin{center}
				\begin{overpic}[height=3cm]{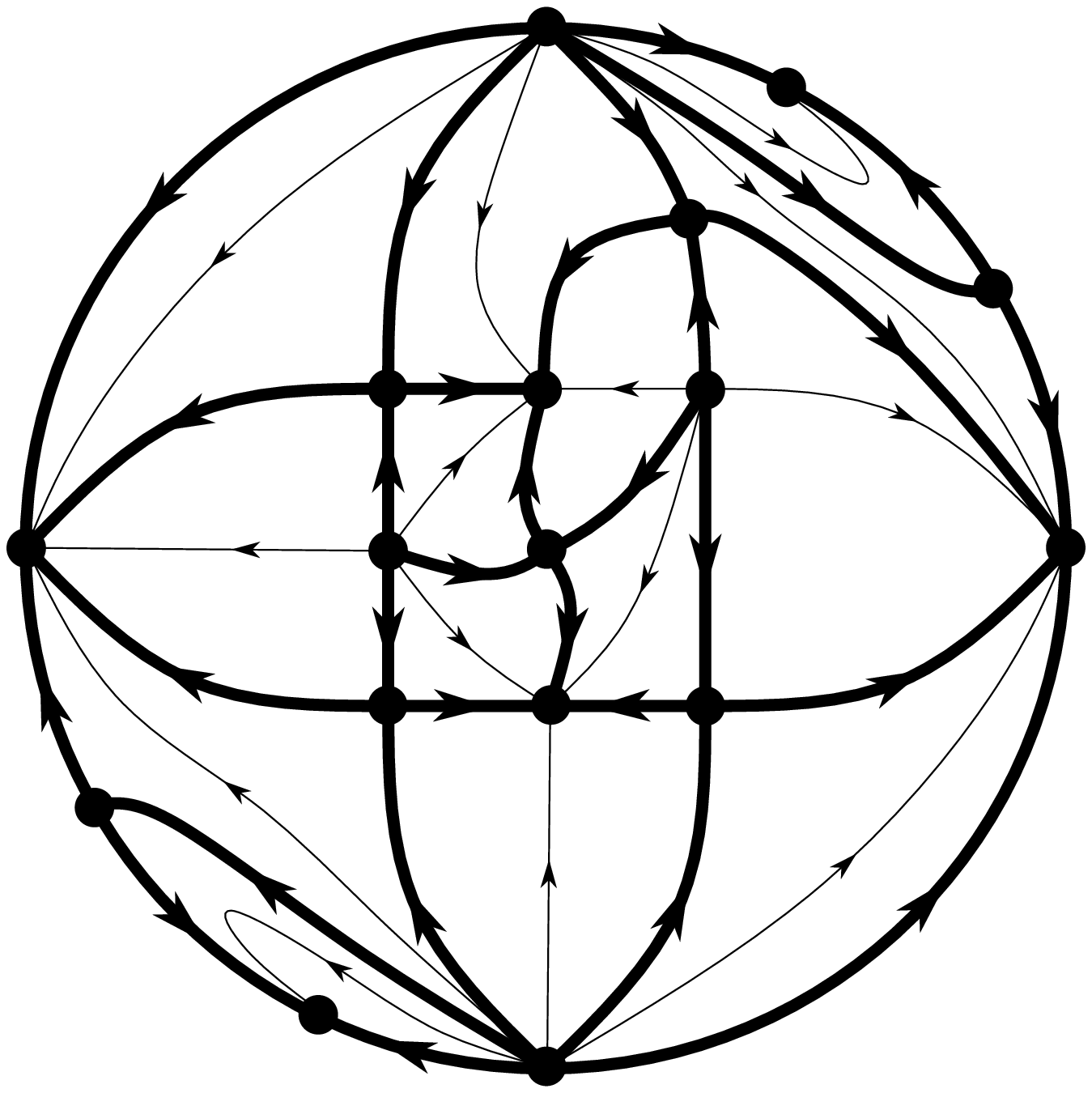} 
				\end{overpic}
				
				Case~$2.11b2$.
			\end{center}
		\end{minipage}
		\begin{minipage}{3.1cm}
			\begin{center}
				\begin{overpic}[height=3cm]{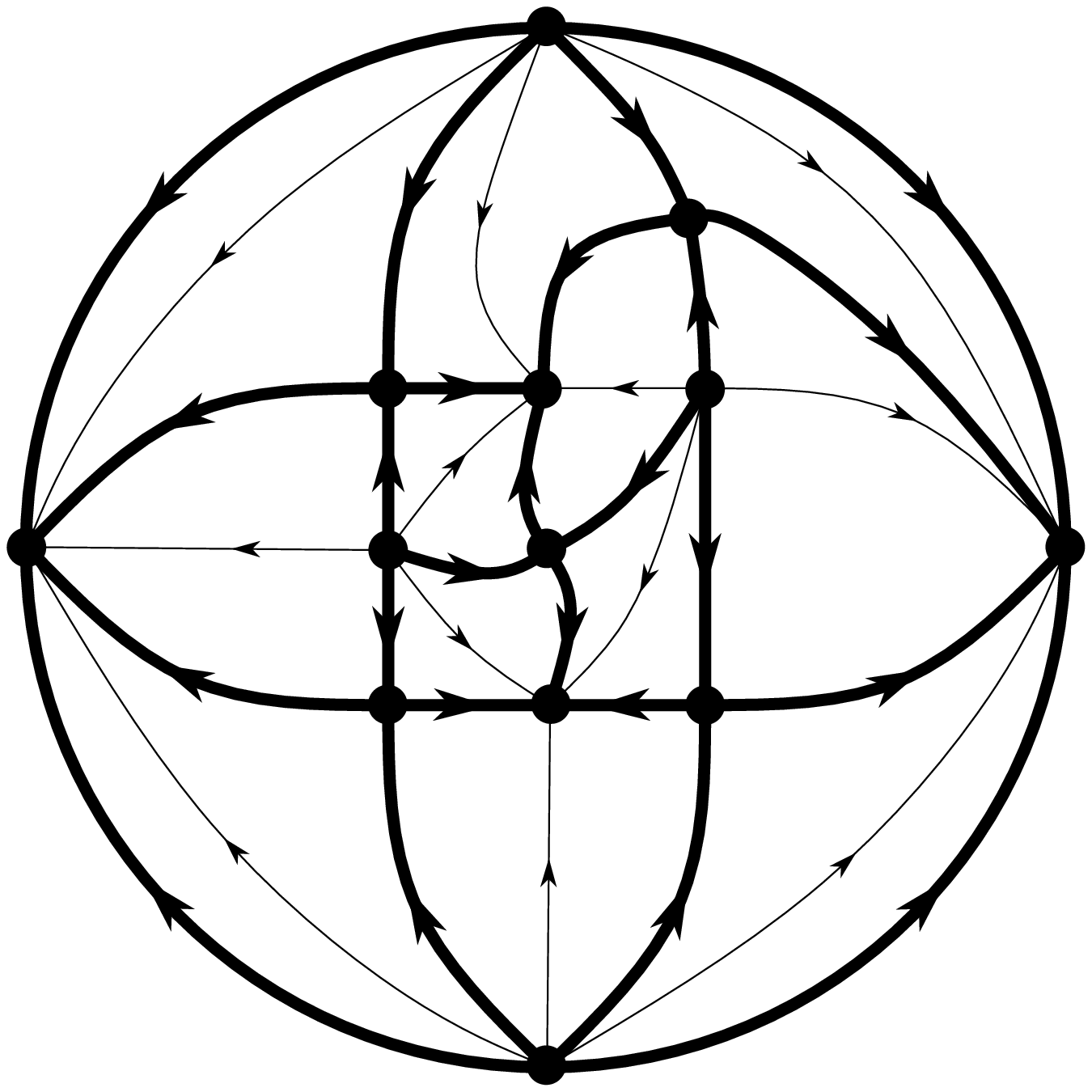} 
				\end{overpic}
				
				Case~$2.11c$.
			\end{center}
		\end{minipage}
		\begin{minipage}{3.1cm}
			\begin{center}
				\begin{overpic}[height=3cm]{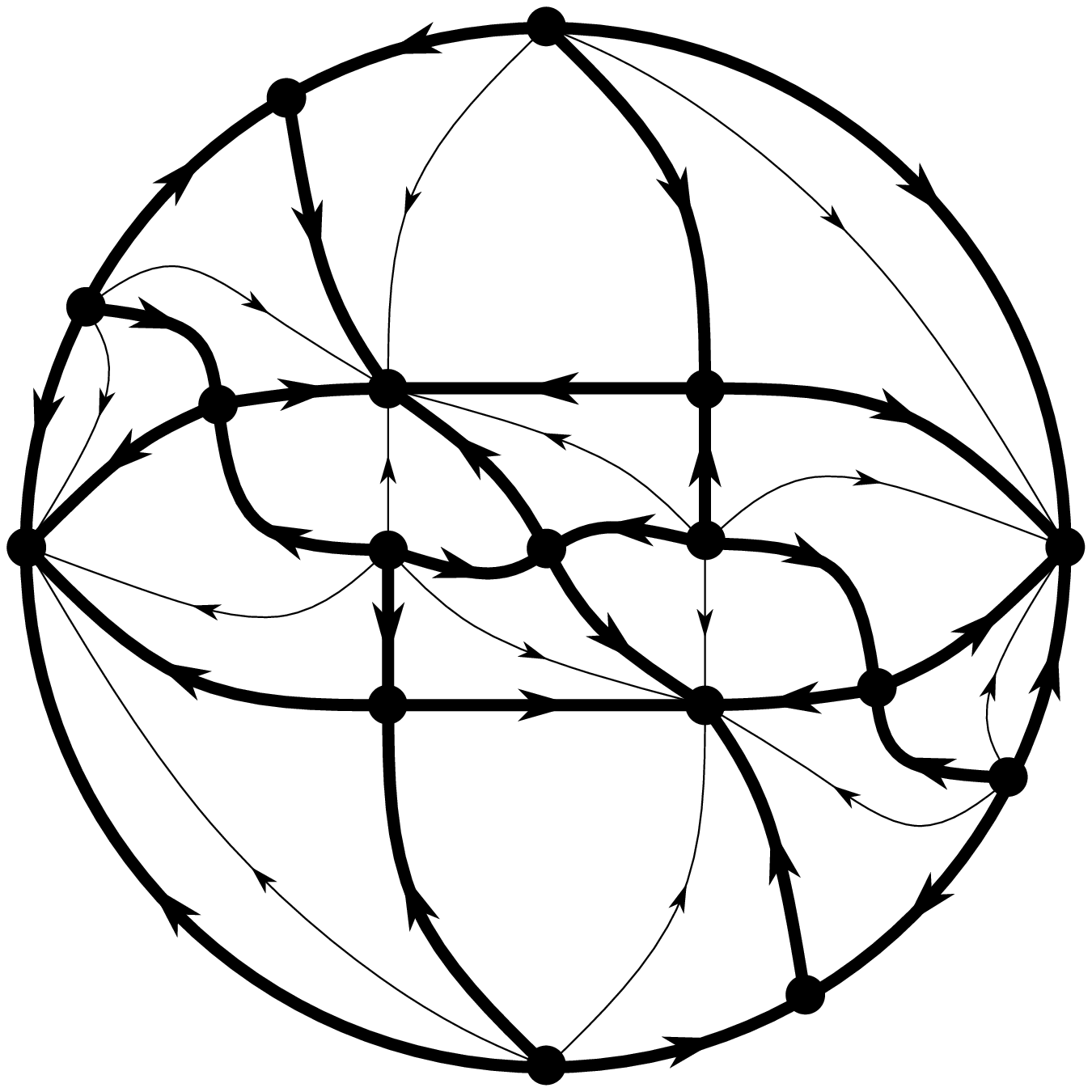} 
				\end{overpic}
				
				Case~$2.12a1$.
			\end{center}
		\end{minipage}
	\end{center}
	$\;$
	\begin{center}
		\begin{minipage}{3.1cm}
			\begin{center}
				\begin{overpic}[height=3cm]{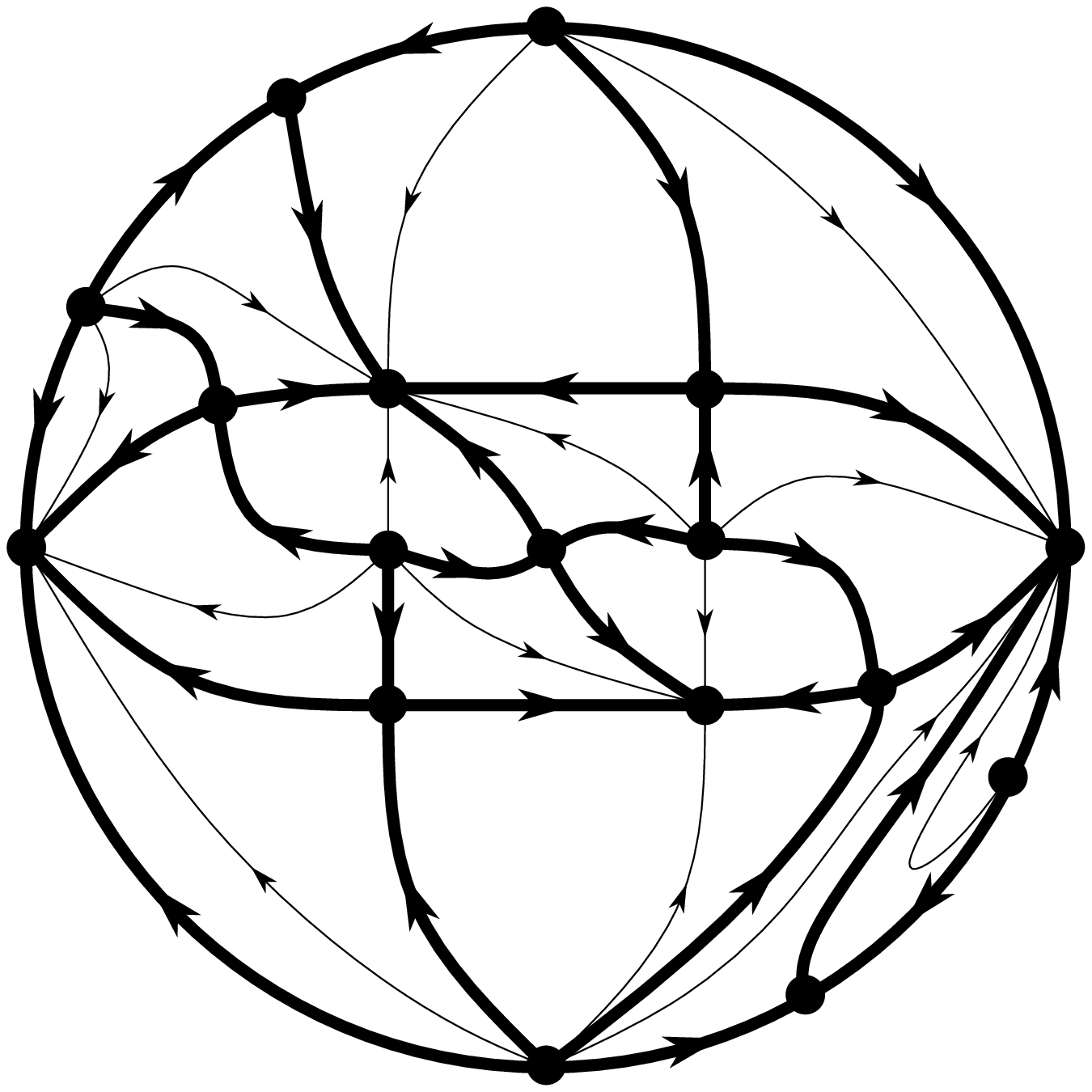} 
				\end{overpic}
				
				Case~$2.12a2$.
			\end{center}
		\end{minipage}	
		\begin{minipage}{3.1cm}
			\begin{center}
				\begin{overpic}[height=3cm]{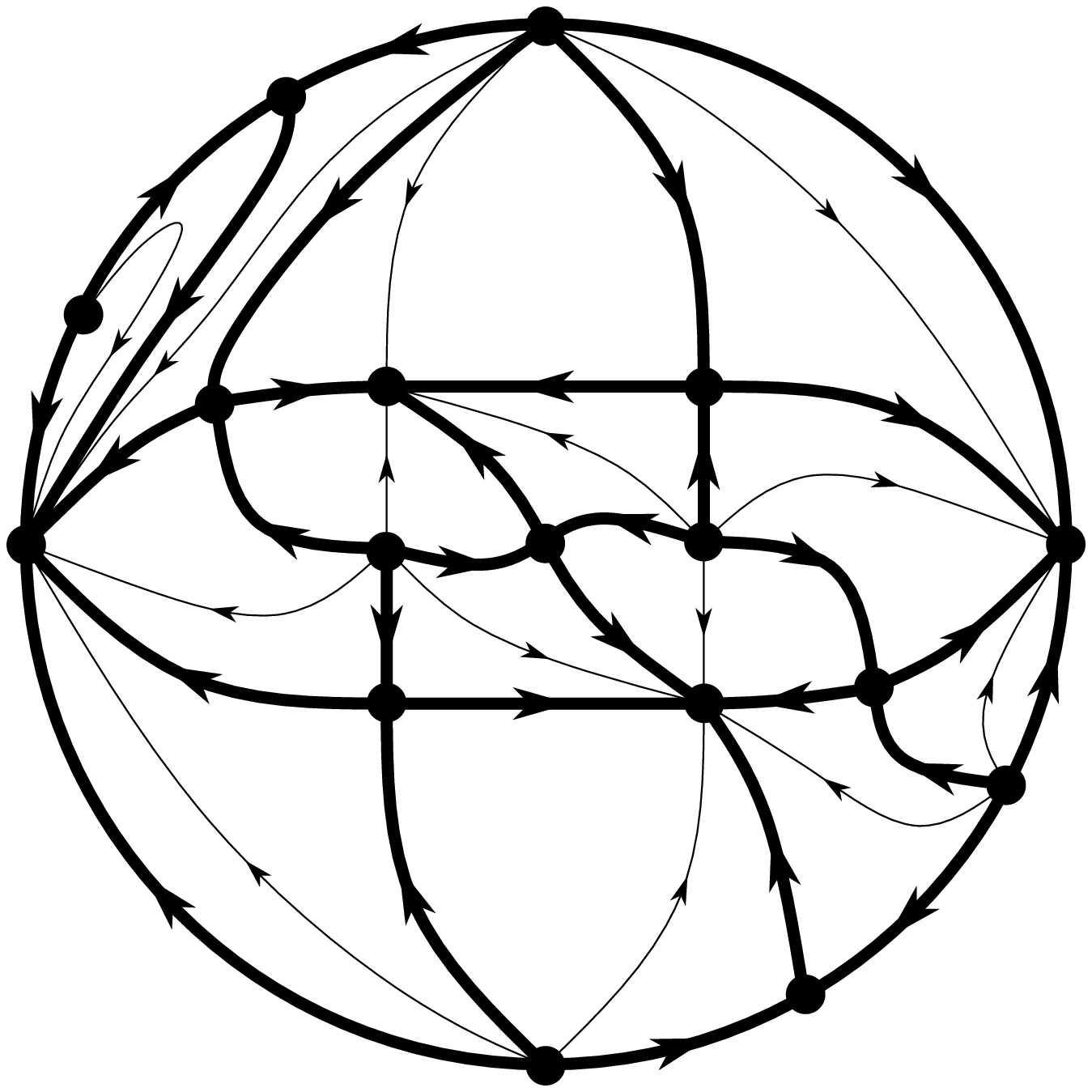} 
				\end{overpic}
				
				Case~$2.12a3$.
			\end{center}
		\end{minipage}
		\begin{minipage}{3.1cm}
			\begin{center}
				\begin{overpic}[height=3cm]{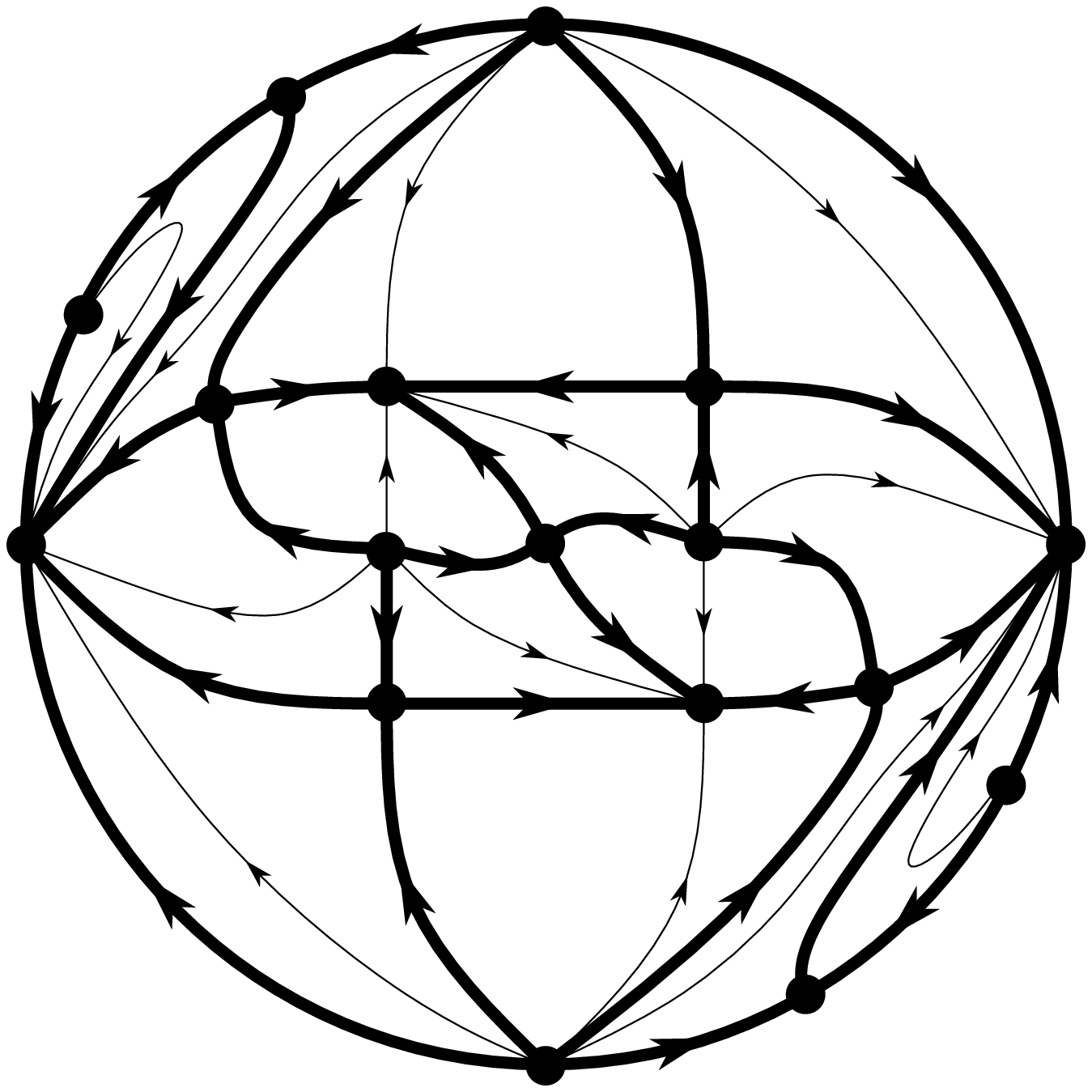} 
				\end{overpic}
				
				Case~$2.12a4$.
			\end{center}
		\end{minipage}
		\begin{minipage}{3.1cm}
			\begin{center}
				\begin{overpic}[height=3cm]{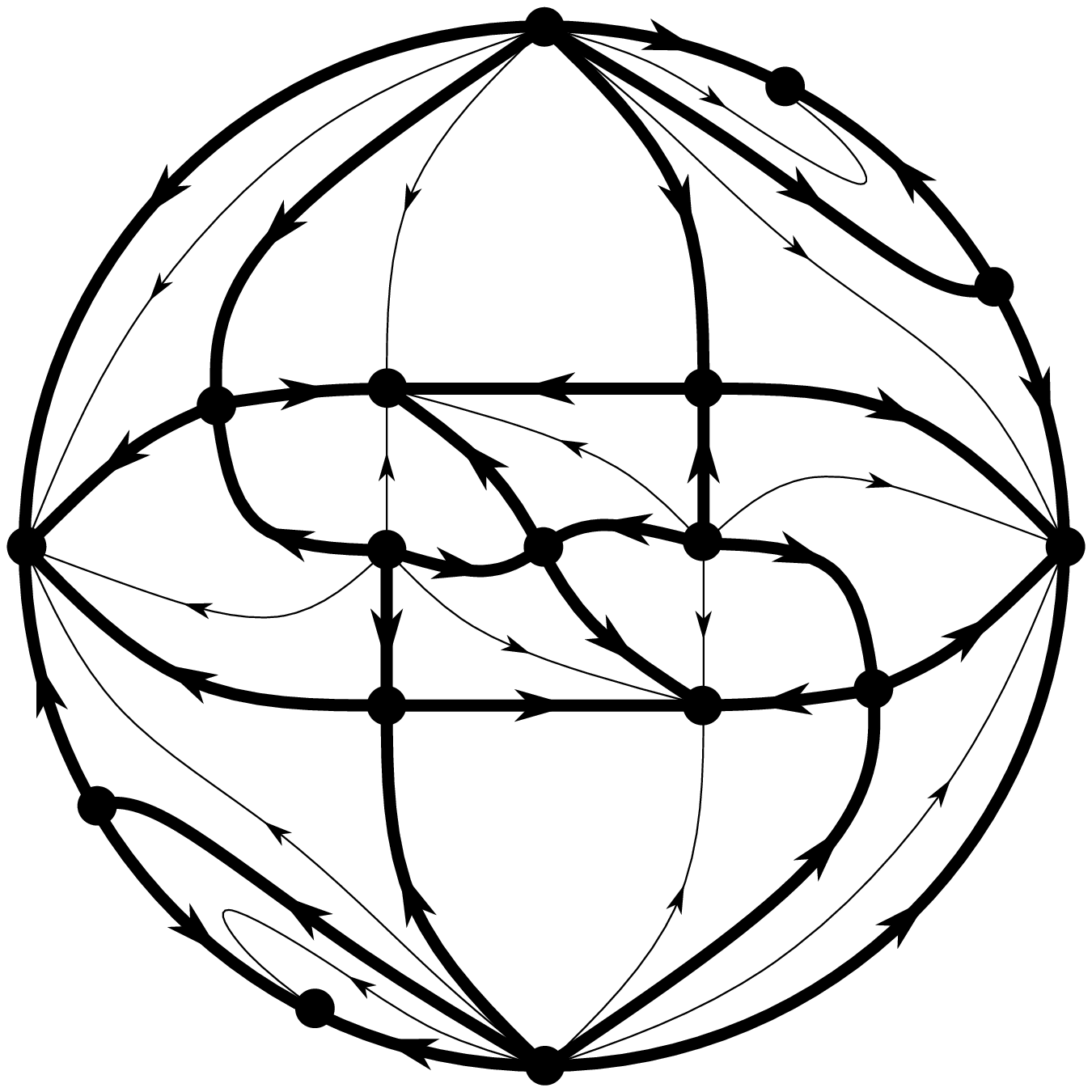} 
				\end{overpic}
				
				Case~$2.12b$.
			\end{center}
		\end{minipage}
		\begin{minipage}{3.1cm}
			\begin{center}
				\begin{overpic}[height=3cm]{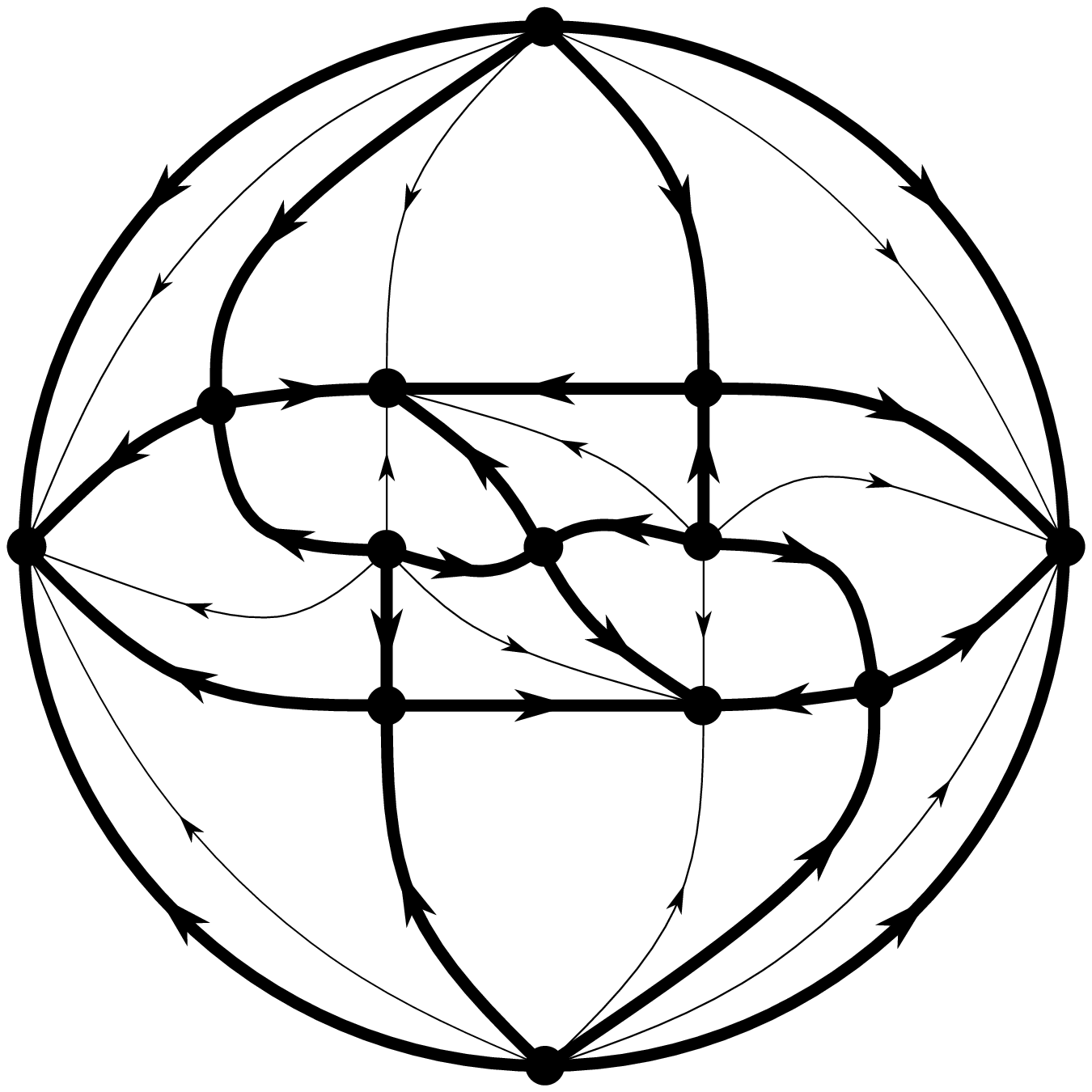} 
				\end{overpic}
				
				Case~$2.12c$.
			\end{center}
		\end{minipage}
	\end{center}
	$\;$
	\begin{center}
		\begin{minipage}{3.1cm}
			\begin{center}
				\begin{overpic}[height=3cm]{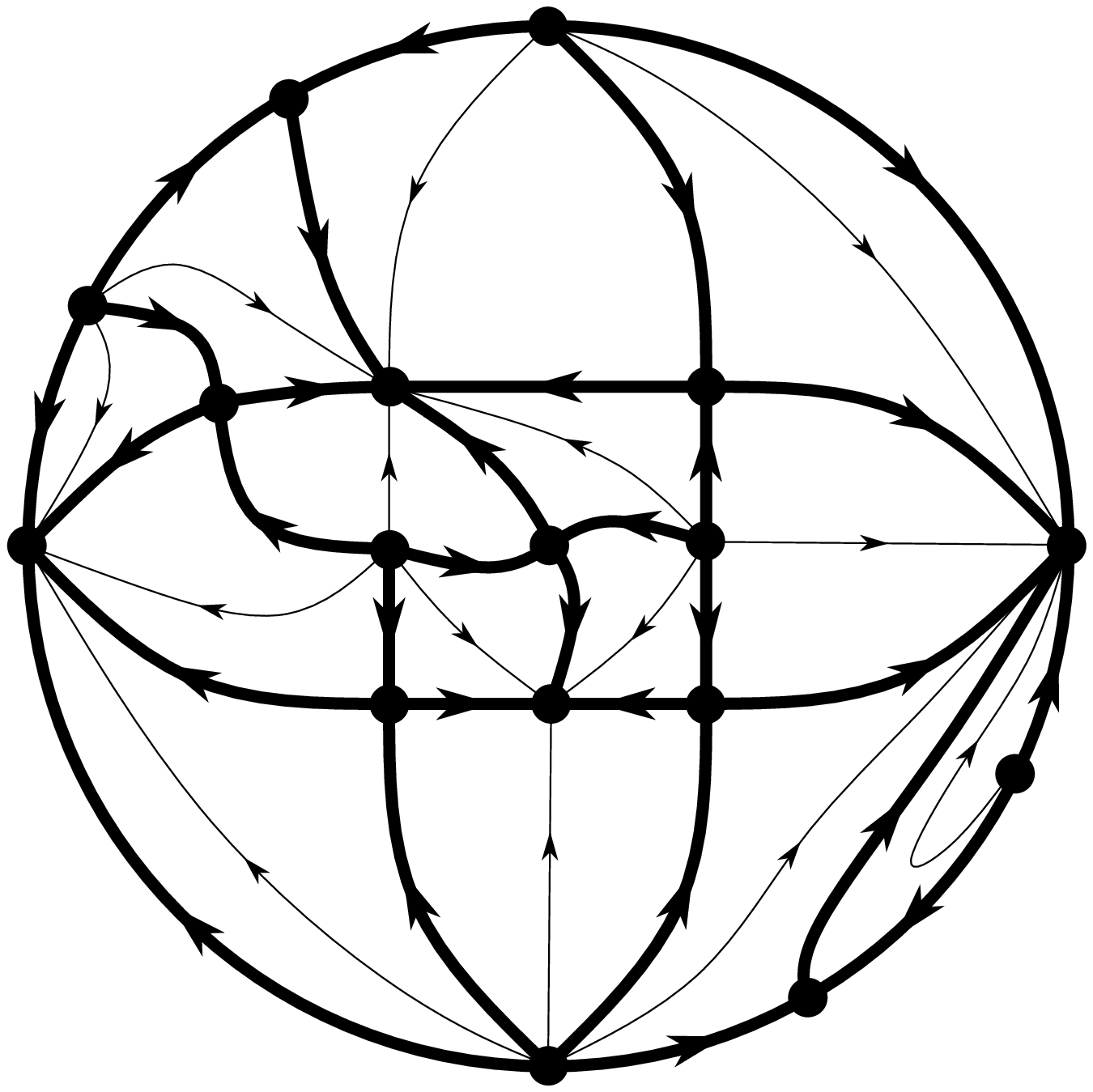} 
				\end{overpic}
				
				Case~$2.13a1$.
			\end{center}
		\end{minipage}
		\begin{minipage}{3.1cm}
			\begin{center}
				\begin{overpic}[height=3cm]{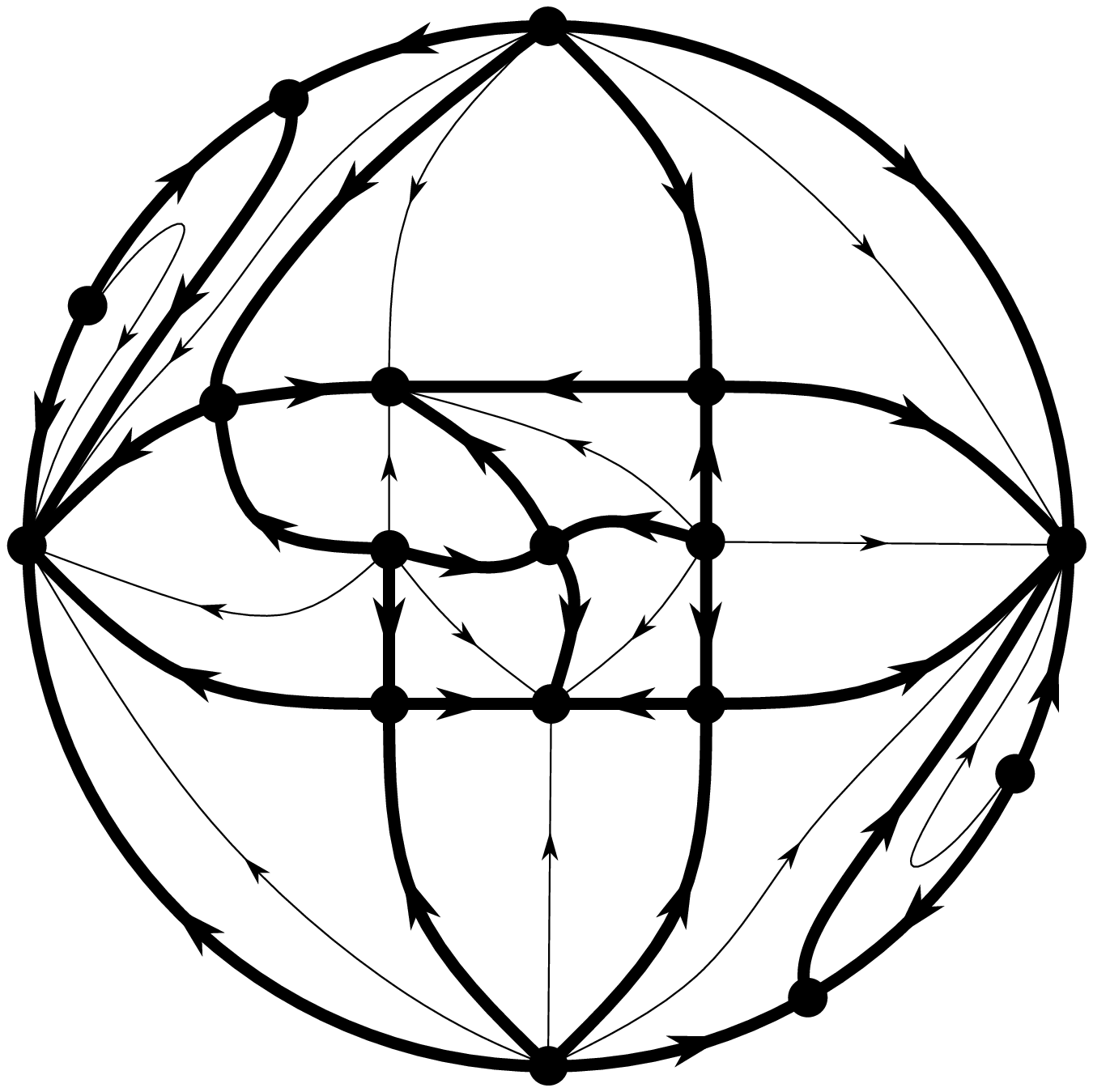} 
				\end{overpic}
				
				Case~$2.13a2$.
			\end{center}
		\end{minipage}
		\begin{minipage}{3.1cm}
			\begin{center}
				\begin{overpic}[height=3cm]{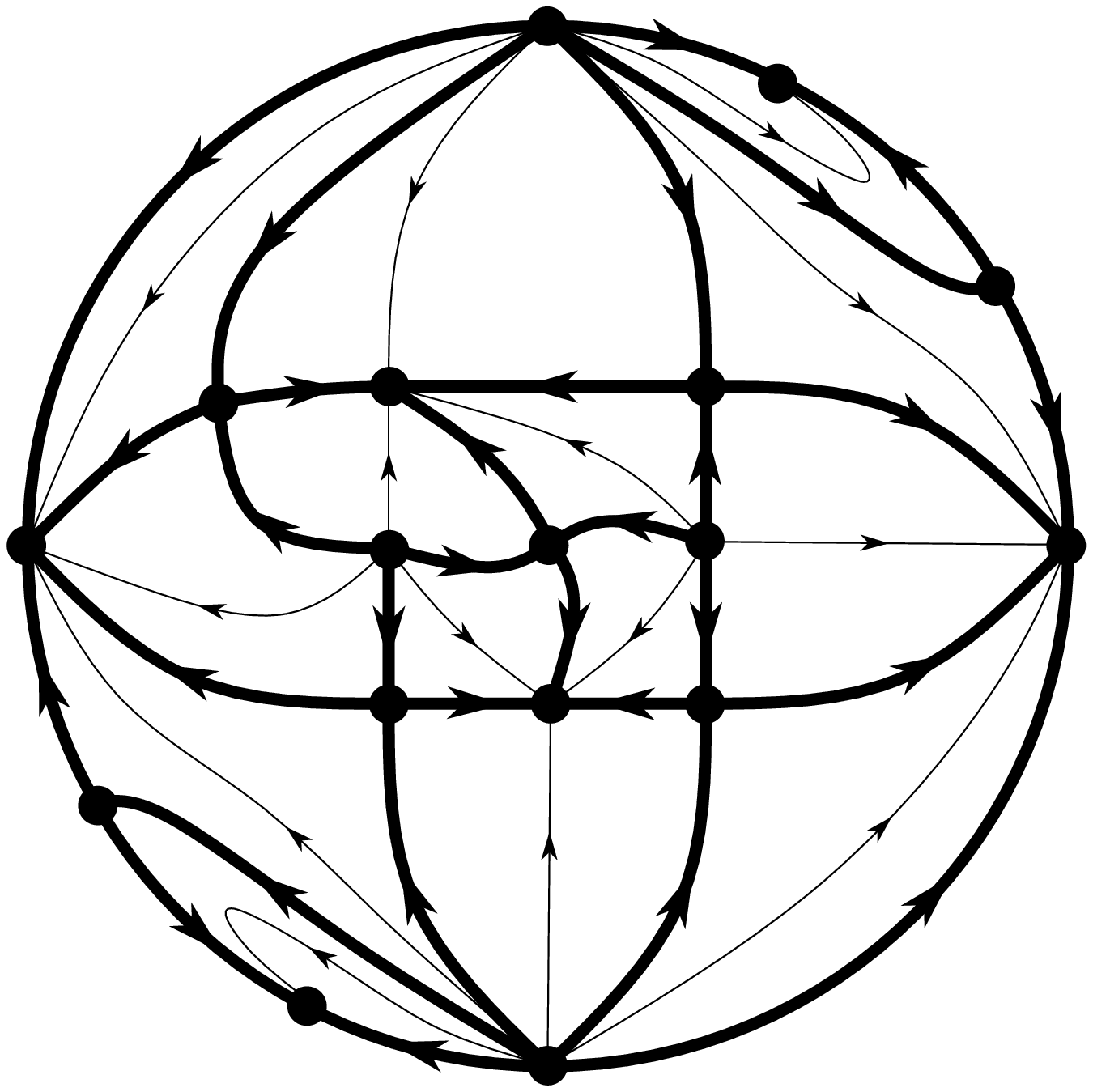} 
				\end{overpic}
				
				Case~$2.13b$.
			\end{center}
		\end{minipage}	
		\begin{minipage}{3.1cm}
			\begin{center}
				\begin{overpic}[height=3cm]{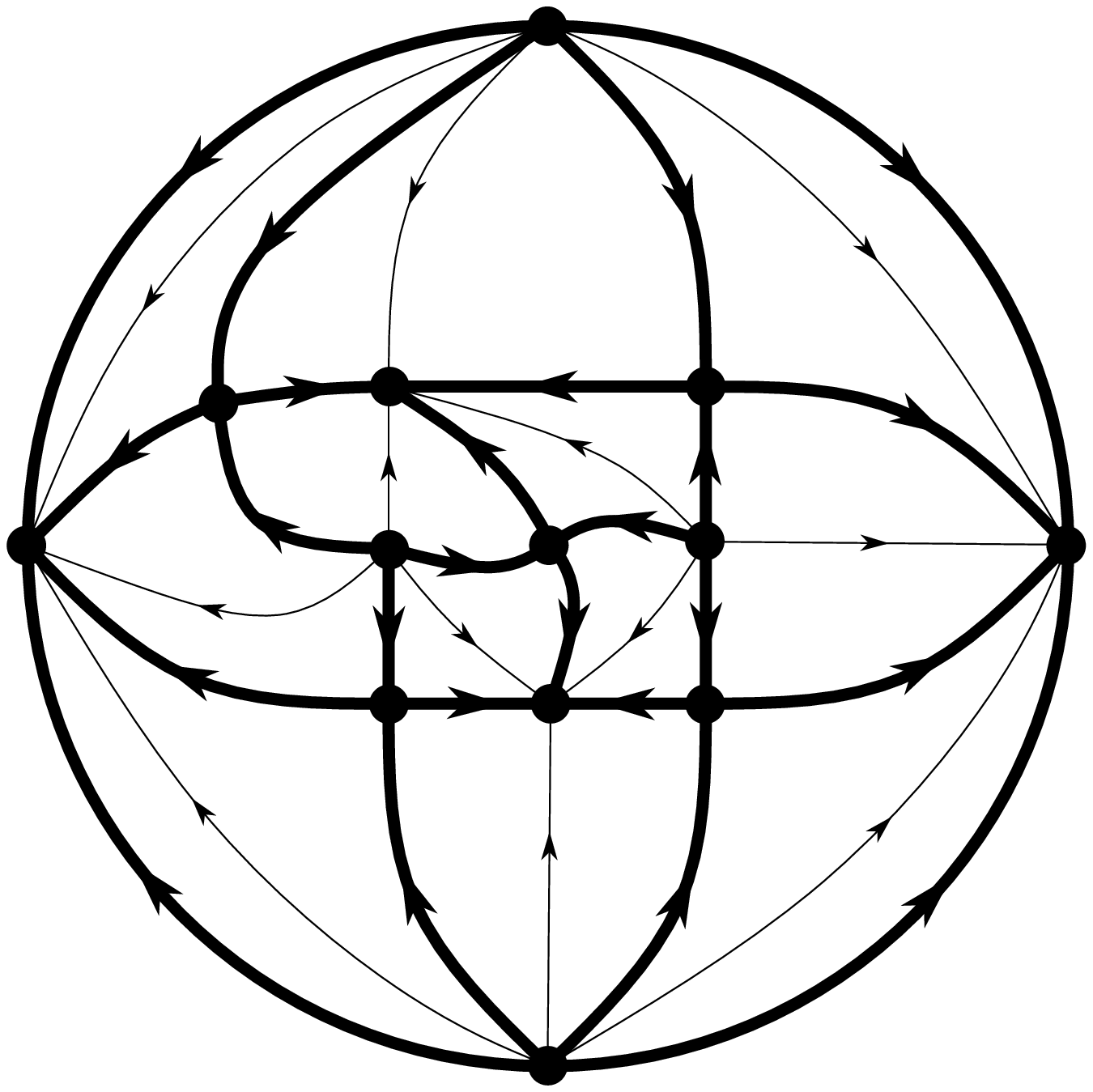} 
				\end{overpic}
				
				Case~$2.13c$.
			\end{center}
		\end{minipage}
	\end{center}
	$\;$
	\begin{center}
		\begin{minipage}{3.1cm}
			\begin{center}
				\begin{overpic}[height=3cm]{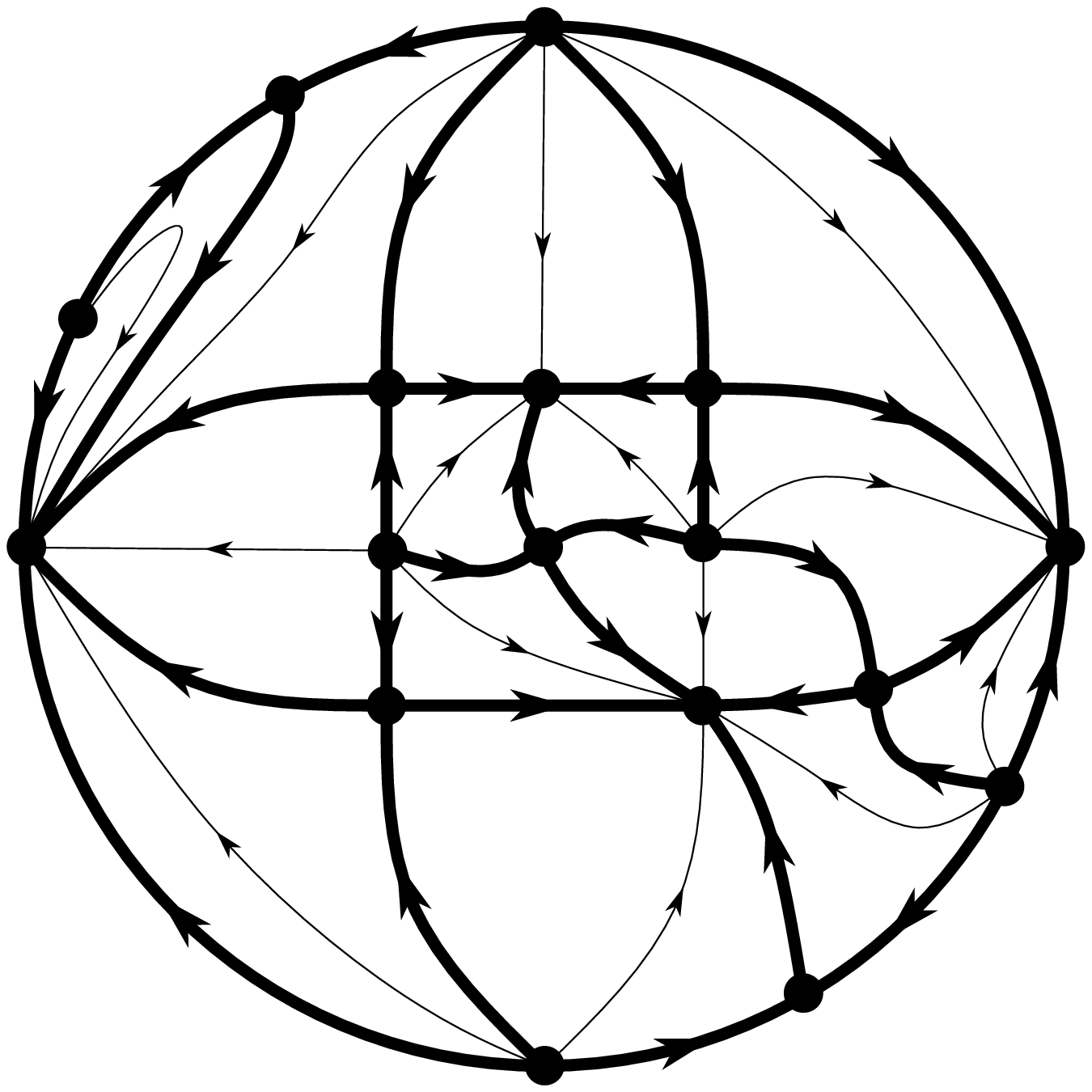} 
				\end{overpic}
				
				Case~$2.14a1$.
			\end{center}
		\end{minipage}
		\begin{minipage}{3.1cm}
			\begin{center}
				\begin{overpic}[height=3cm]{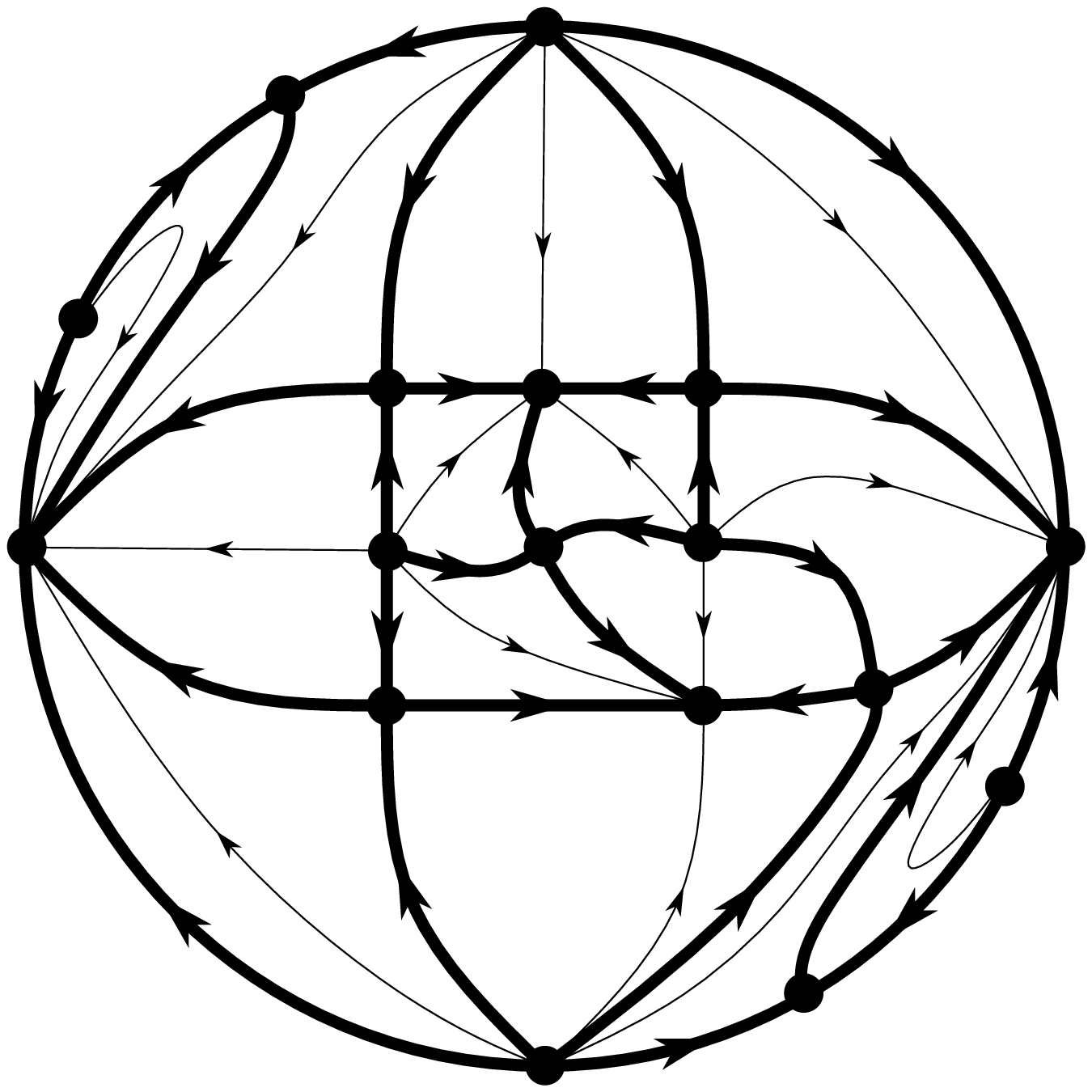} 
				\end{overpic}
				
				Case~$2.14a2$.
			\end{center}
		\end{minipage}
		\begin{minipage}{3.1cm}
			\begin{center}
				\begin{overpic}[height=3cm]{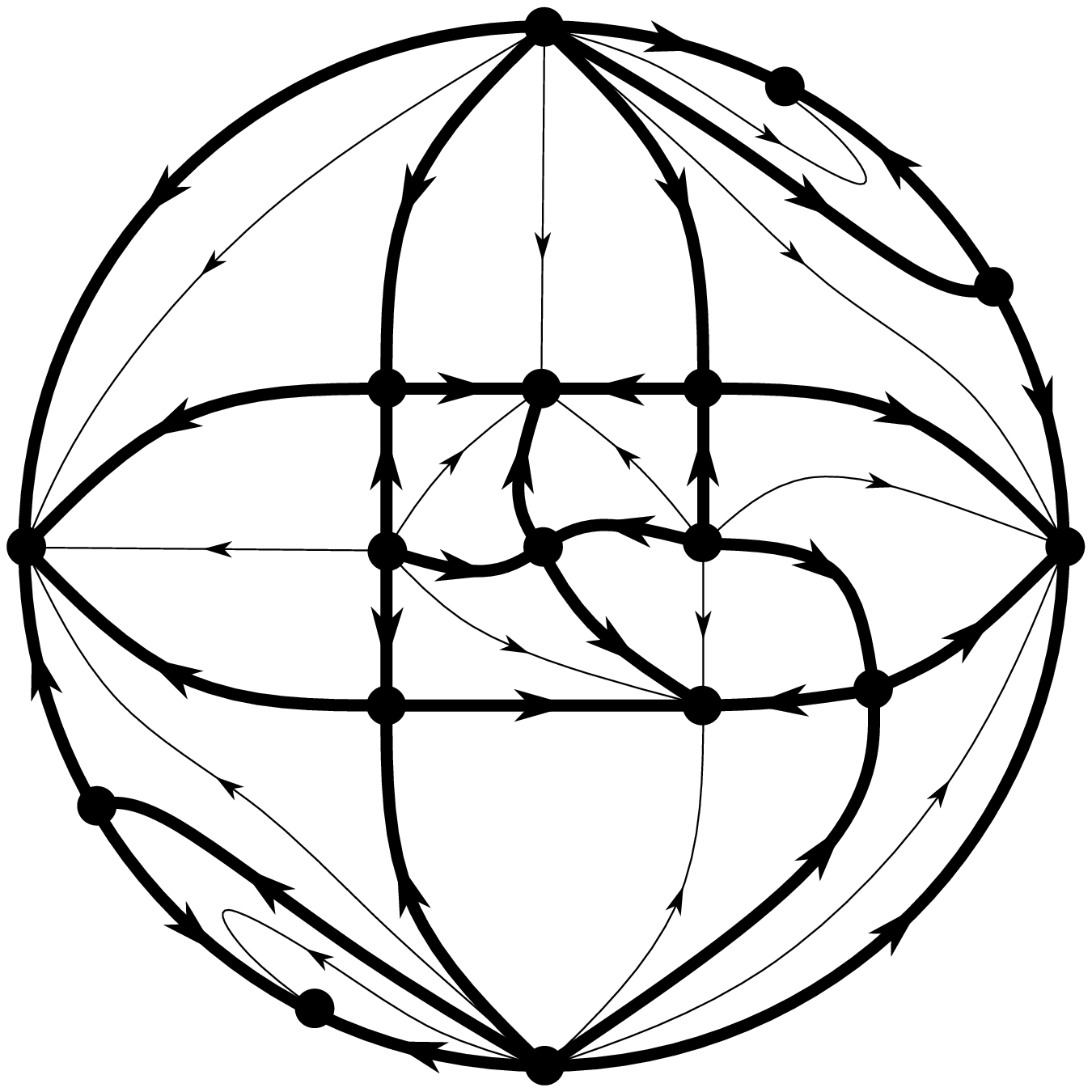} 
				\end{overpic}
				
				Case~$2.14b$.
			\end{center}
		\end{minipage}
		\begin{minipage}{3.1cm}
			\begin{center}
				\begin{overpic}[height=3cm]{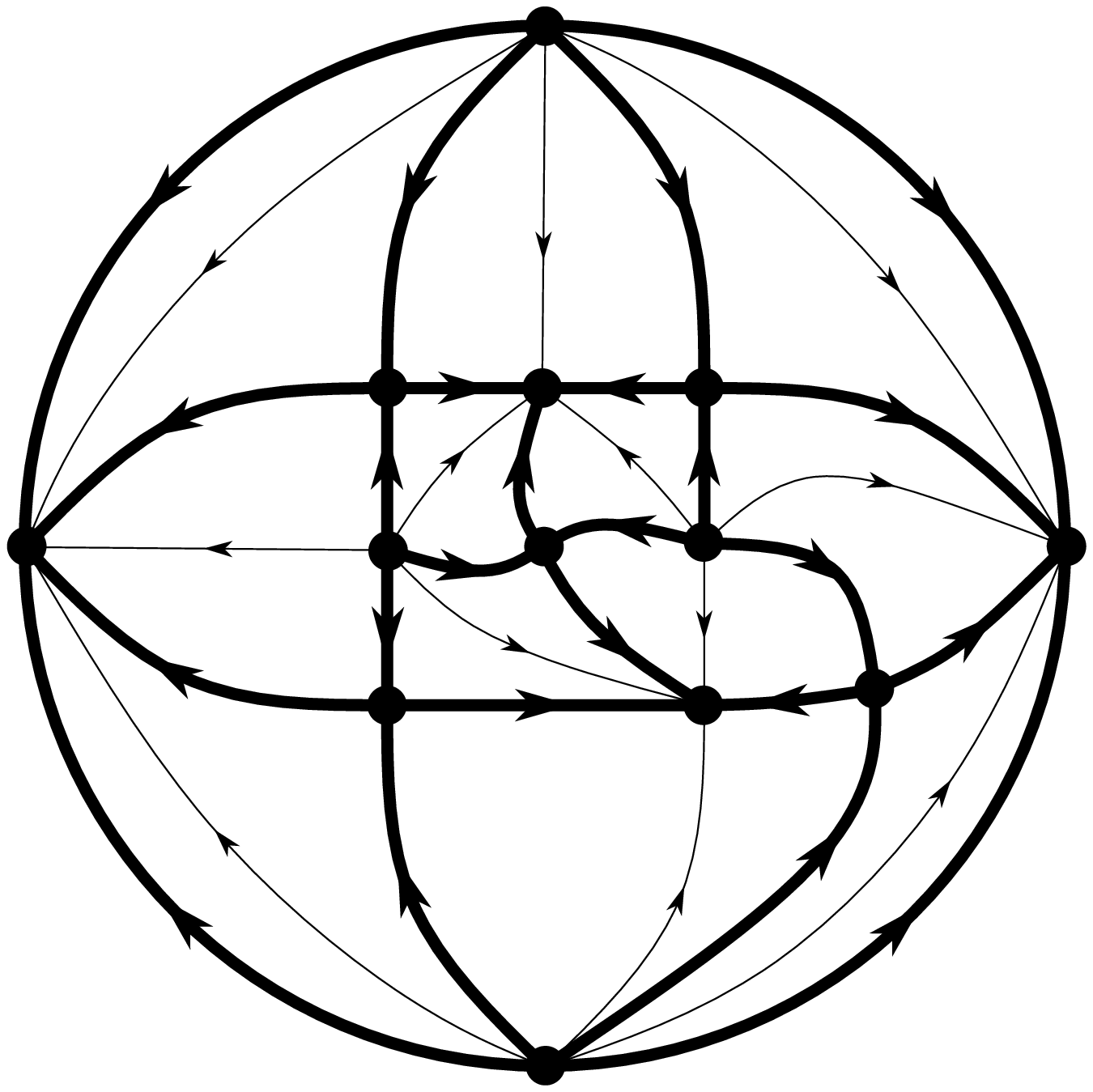} 
				\end{overpic}
				
				Case~$2.14c$.
			\end{center}
		\end{minipage}
	\end{center}
	$\;$
	\begin{center}
		\begin{minipage}{3.1cm}
			\begin{center}
				\begin{overpic}[height=3cm]{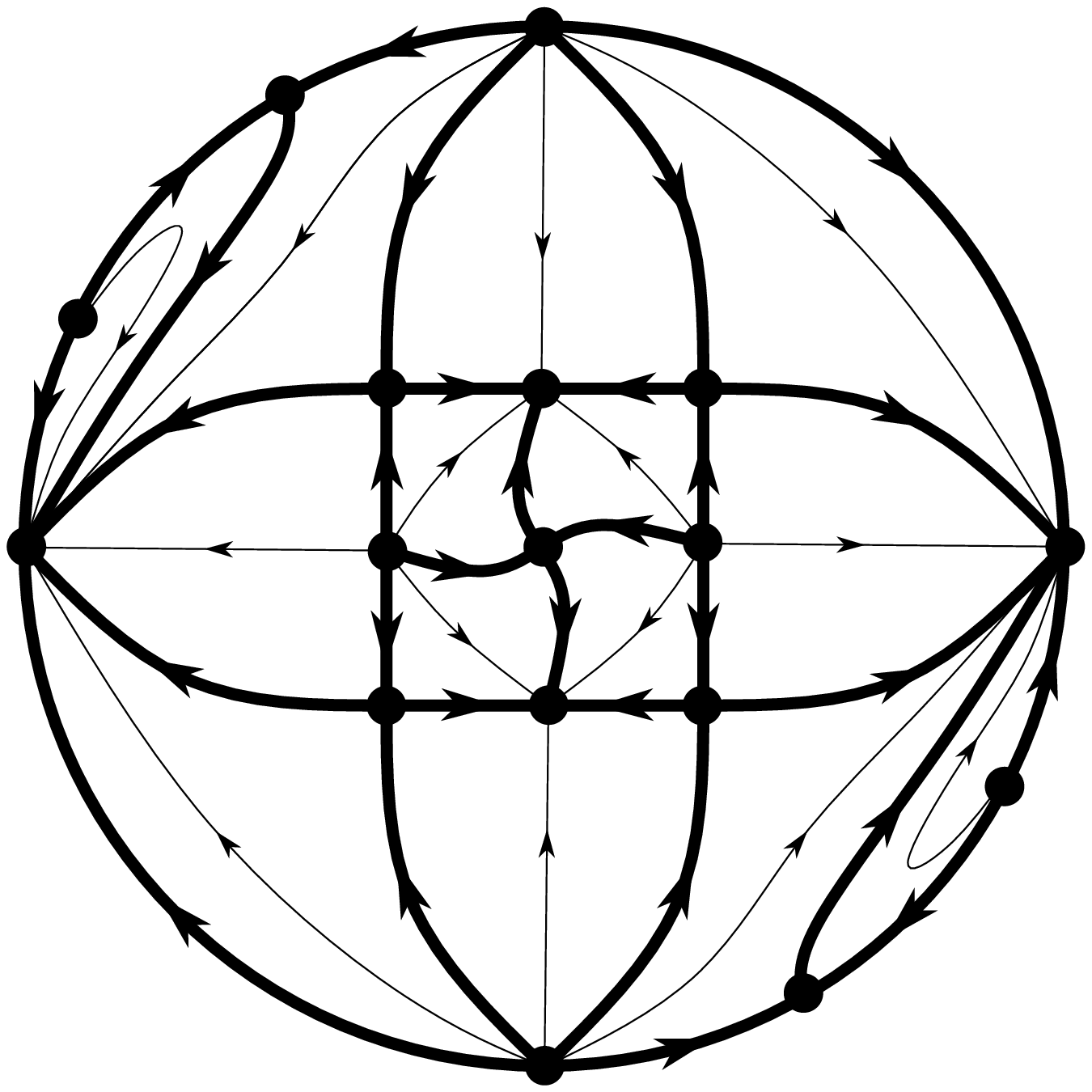} 
				\end{overpic}
				
				Case~$2.15a$.
			\end{center}
		\end{minipage}
		\begin{minipage}{3.1cm}
			\begin{center}
				\begin{overpic}[height=3cm]{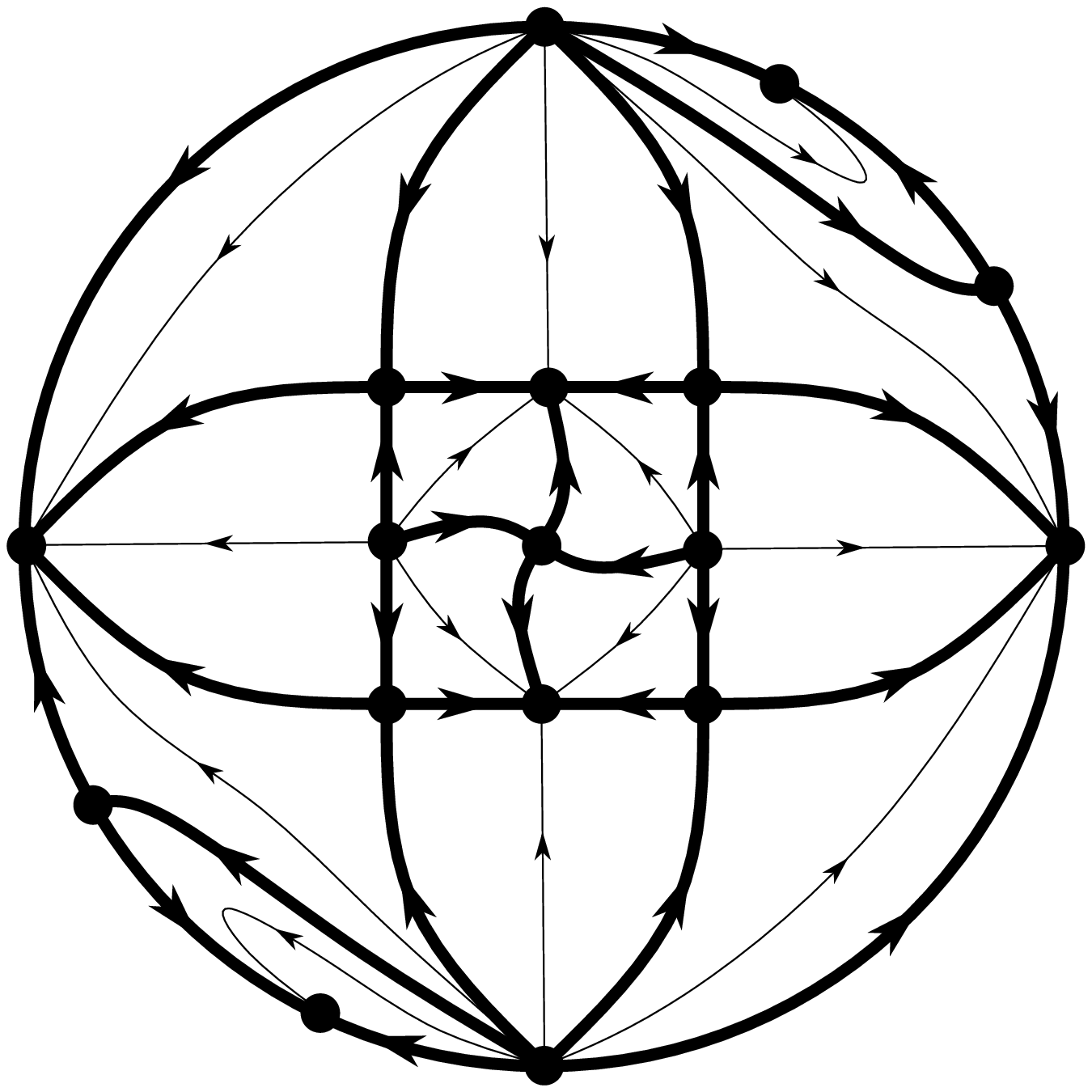} 
				\end{overpic}
				
				Case~$2.15b$.
			\end{center}
		\end{minipage}	
		\begin{minipage}{3.1cm}
			\begin{center}
				\begin{overpic}[height=3cm]{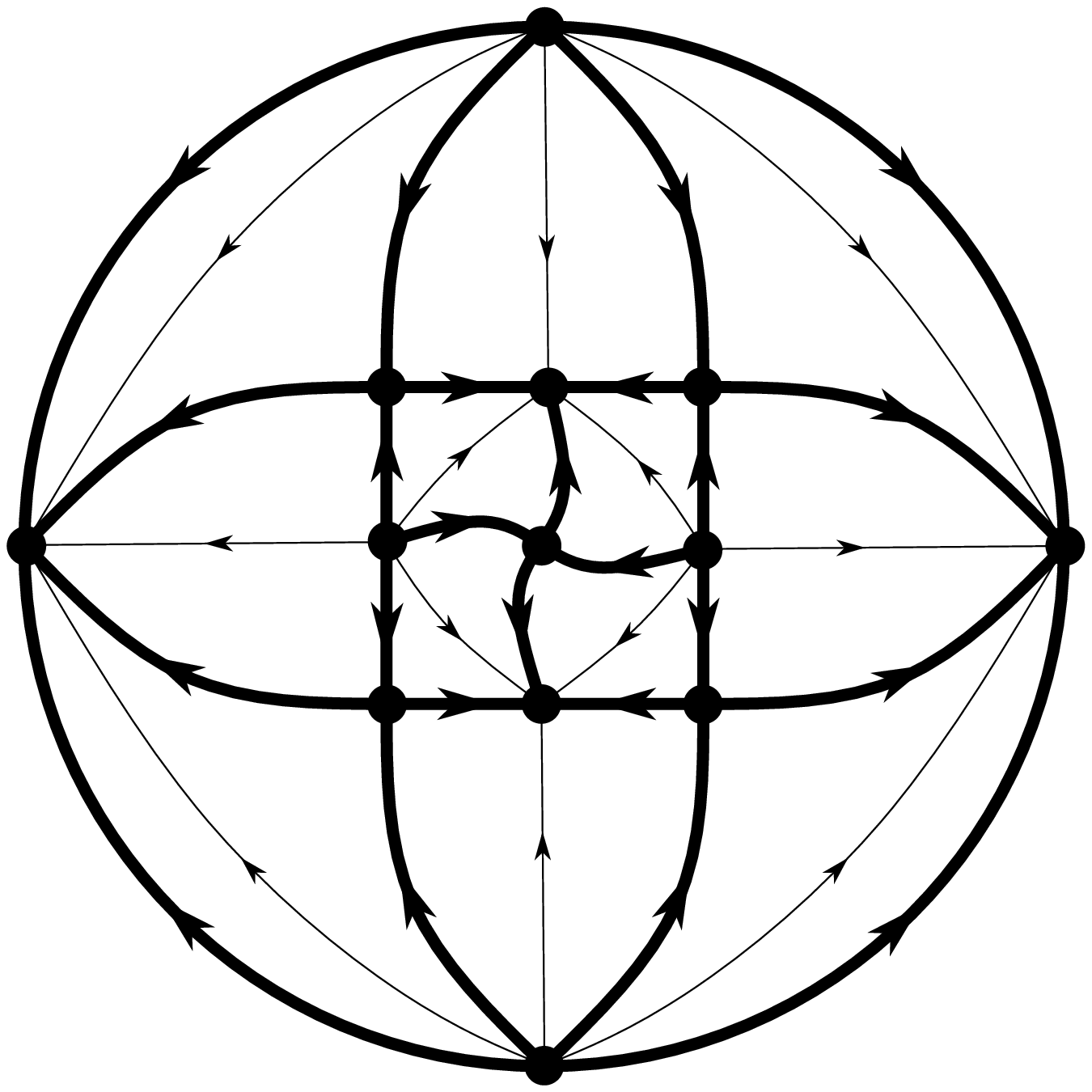} 
				\end{overpic}
				
				Case~$2.15c$.
			\end{center}
		\end{minipage}
	\end{center}
	\caption{Phase portraits from Cases~$2.10$ to $2.15$.}\label{Case1.2c}
\end{figure}

\clearpage

\section{Phase portraits under the hypothesis of Theorem~\ref{Theo13} (Family 3)}\label{Ap2}

\begin{table}[h]
	\caption{Table of the realizable cases under hypothesis of Theorem~\ref{Theo13} (Family $3$).}\label{Table5}
	\begin{tabular}{c c c c c}
		\hline 
		Cases & $q_1$ & $q_2$ & $q_3$ & $q_4$ \\
		\hline
		\rowcolor{mygray}
		$3.1$ & $q_1>p_2$ & $q_2>p_3$ & $q_3<p_4$ & $q_4<p_1$ \\
		$3.2$ & $q_1<p_2$ & $q_2>p_3$ & $q_3<p_4$ & $q_4<p_1$ \\
		\rowcolor{mygray}
		$3.3$ & $q_1>p_2$ & $q_2>p_3$ & $q_3>p_4$ & $q_4<p_1$ \\
		$3.4$ & $q_1<p_2$ & $q_2>p_3$ & $q_3>p_4$ & $q_4<p_1$ \\
		\rowcolor{mygray}
		$3.5$ & $q_1>p_2$ & $q_2<p_3$ & $q_3<p_4$ & $q_4<p_1$ \\
		$3.6$ & $q_1>p_2$ & $q_2<p_3$ & $q_3>p_4$ & $q_4<p_1$ \\
		\rowcolor{mygray}
		$3.7$ & $q_1<p_2$ & $q_2<p_3$ & $q_3>p_4$ & $q_4<p_1$ \\
		$3.8$ & $q_1>p_2$ & $q_2>p_3$ & $q_3<p_4$ & $q_4>p_1$ \\
		\rowcolor{mygray}
		$3.9$ & $q_1<p_2$ & $q_2>p_3$ & $q_3<p_4$ & $q_4>p_1$ \\
		$3.10$ & $q_1>p_2$ & $q_2>p_3$ & $q_3>p_4$ & $q_4>p_1$  \\
		\rowcolor{mygray}
		$3.11$ & $q_1<p_2$ & $q_2>p_3$ & $q_3>p_4$ & $q_4>p_1$ \\
		$3.12$ & $q_1>p_2$ & $q_2<p_3$ & $q_3<p_4$ & $q_4>p_1$ \\
		\rowcolor{mygray}
		$3.13$ & $q_1<p_2$ & $q_2<p_3$ & $q_3<p_4$ & $q_4>p_1$ \\
		$3.14$ & $q_1>p_2$ & $q_2<p_3$ & $q_3>p_4$ & $q_4>p_1$ \\
		\rowcolor{mygray}
		$3.15$ & $q_1<p_2$ & $q_2<p_3$ & $q_3>p_4$ & $q_4>p_1$  \\
		\hline
	\end{tabular}
\end{table}

\begin{figure}[h]
	\begin{center}
		\begin{minipage}{3.1cm}
			\begin{center}
				\begin{overpic}[height=3cm]{Case1.3.1.eps} 
				\end{overpic}
				
				Case~$3.1$.
			\end{center}
		\end{minipage}
		\begin{minipage}{3.1cm}
			\begin{center}
				\begin{overpic}[height=3cm]{Case1.3.2.eps} 
				\end{overpic}
				
				Case~$3.2$.
			\end{center}
		\end{minipage}
		\begin{minipage}{3.1cm}
			\begin{center}
				\begin{overpic}[height=3cm]{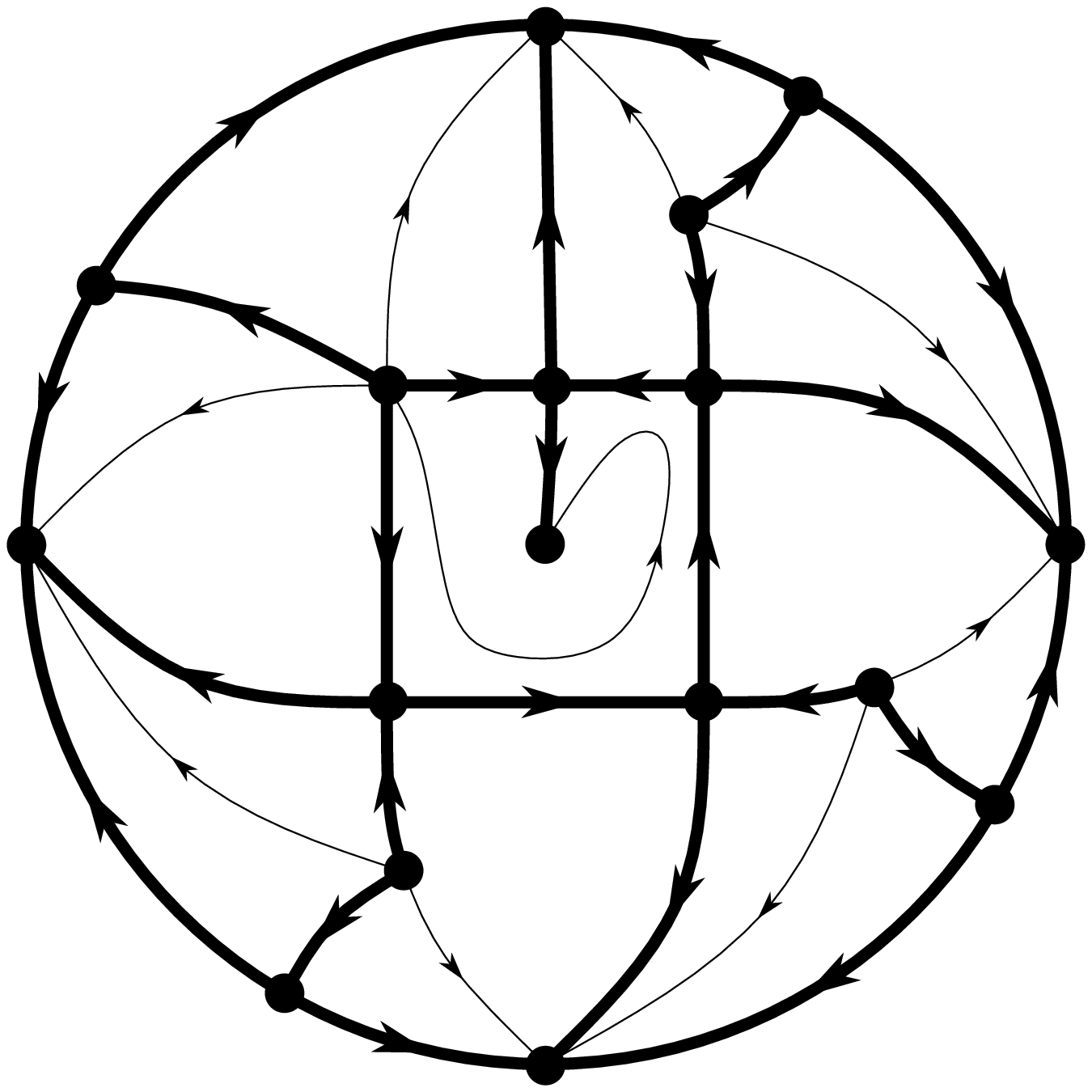} 
				\end{overpic}
				
				Case~$3.3$.
			\end{center}
		\end{minipage}	
		\begin{minipage}{3.1cm}
			\begin{center}
				\begin{overpic}[height=3cm]{Case1.3.4.eps} 
				\end{overpic}
				
				Case~$3.4$.
			\end{center}
		\end{minipage}
		\begin{minipage}{3.1cm}
			\begin{center}
				\begin{overpic}[height=3cm]{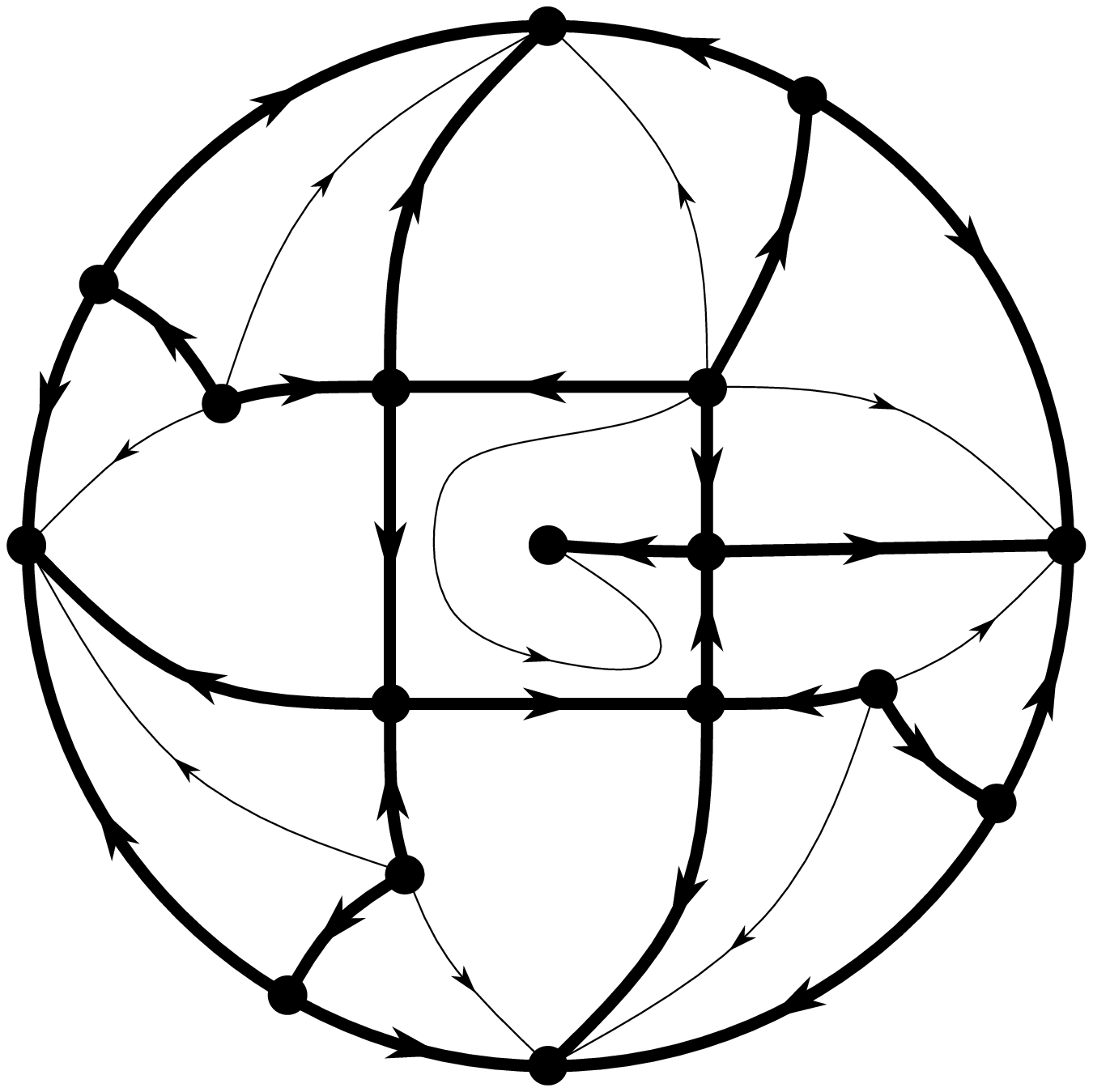} 
				\end{overpic}
				
				Case~$3.5$.
			\end{center}
		\end{minipage}
	\end{center}
	$\;$
	\begin{center}
		\begin{minipage}{3.1cm}
			\begin{center}
				\begin{overpic}[height=3cm]{Case1.3.6.eps} 
				\end{overpic}
				
				Case~$3.6$.
			\end{center}
		\end{minipage}
		\begin{minipage}{3.1cm}
			\begin{center}
				\begin{overpic}[height=3cm]{Case1.3.7.eps} 
				\end{overpic}
				
				Case~$3.7$.
			\end{center}
		\end{minipage}
		\begin{minipage}{3.1cm}
			\begin{center}
				\begin{overpic}[height=3cm]{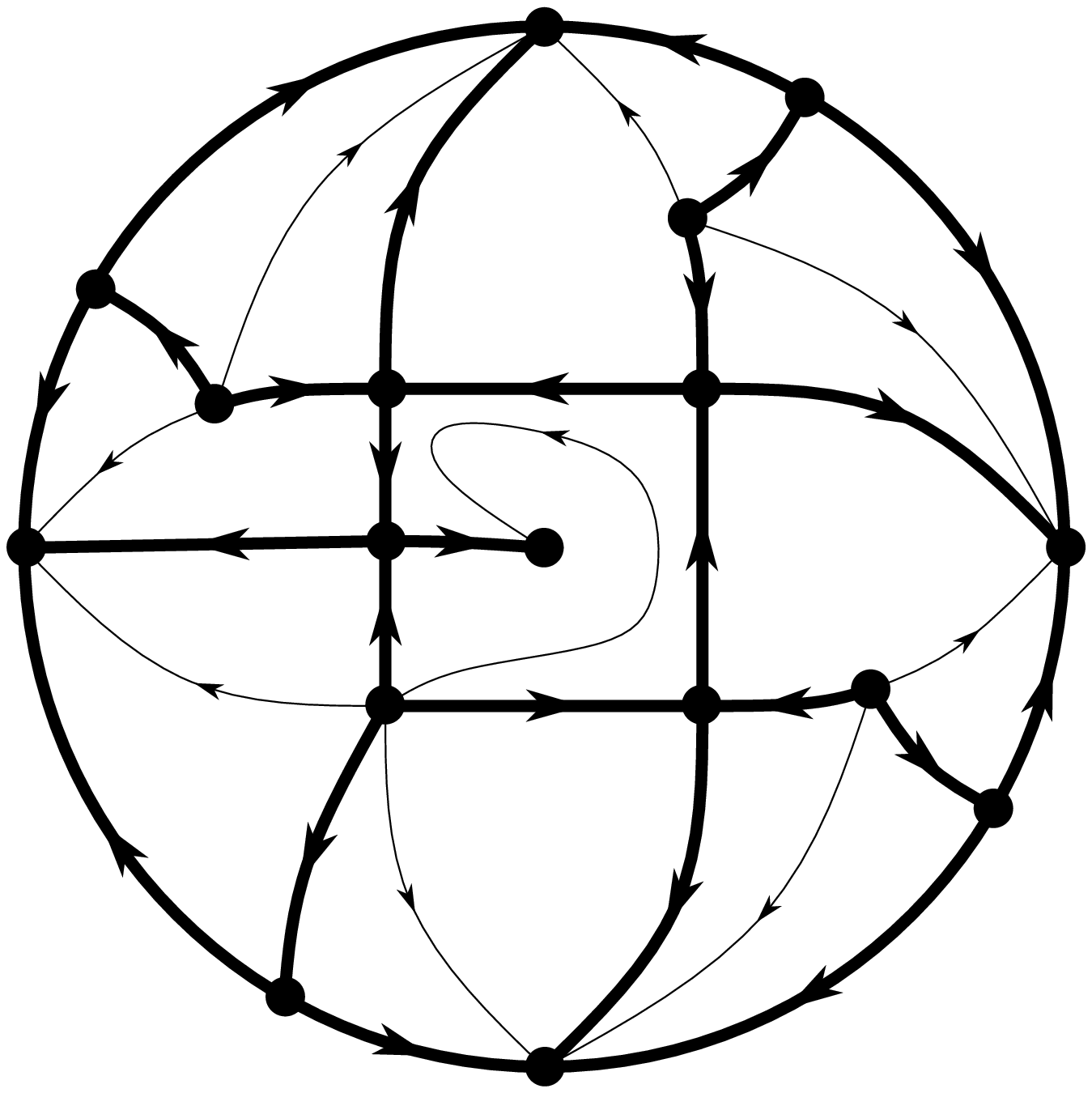} 
				\end{overpic}
				
				Case~$3.8$.
			\end{center}
		\end{minipage}
		\begin{minipage}{3.1cm}
			\begin{center}
				\begin{overpic}[height=3cm]{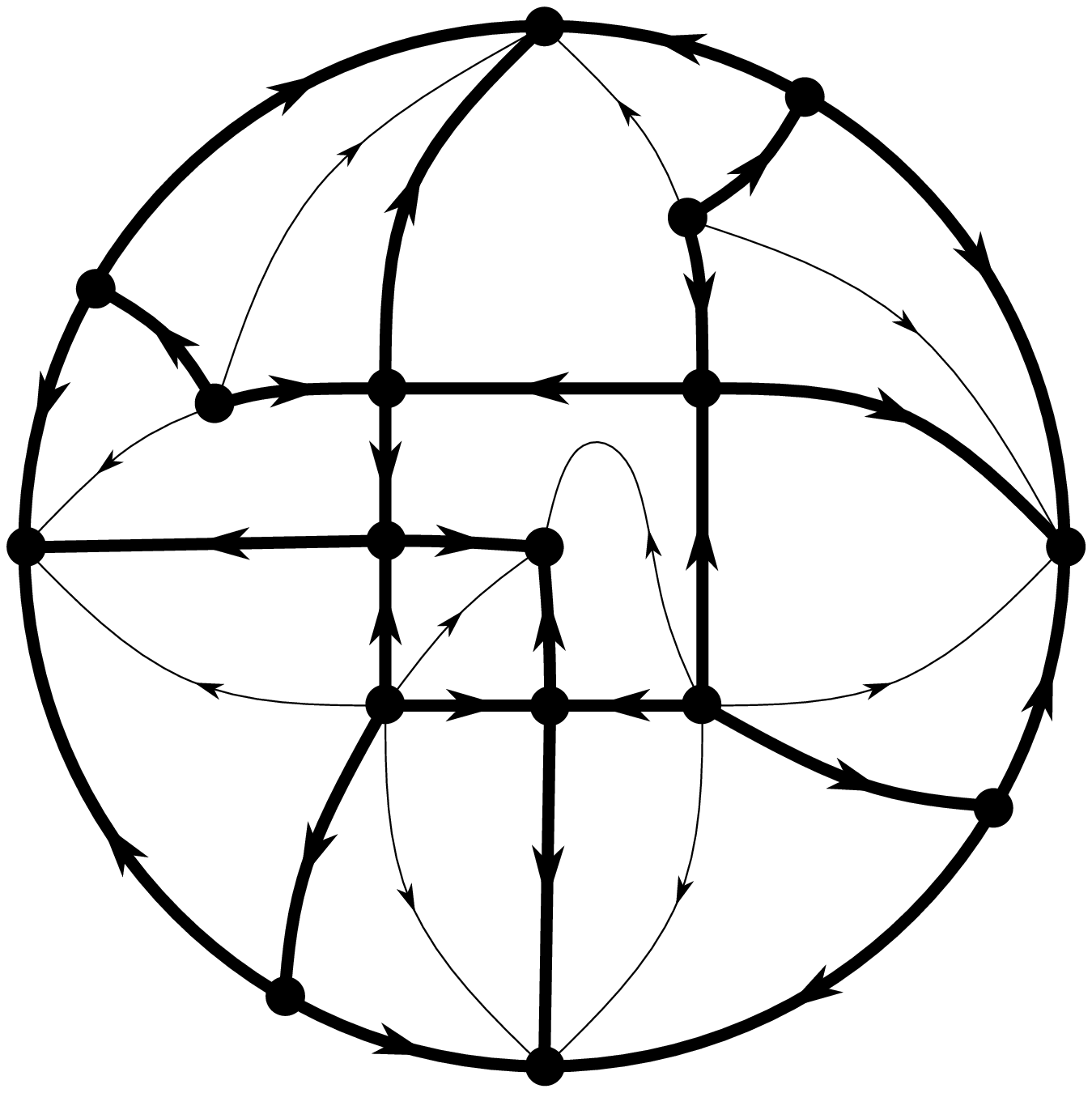} 
				\end{overpic}
				
				Case~$3.9$.
			\end{center}
		\end{minipage}
		\begin{minipage}{3.1cm}
			\begin{center}
				\begin{overpic}[height=3cm]{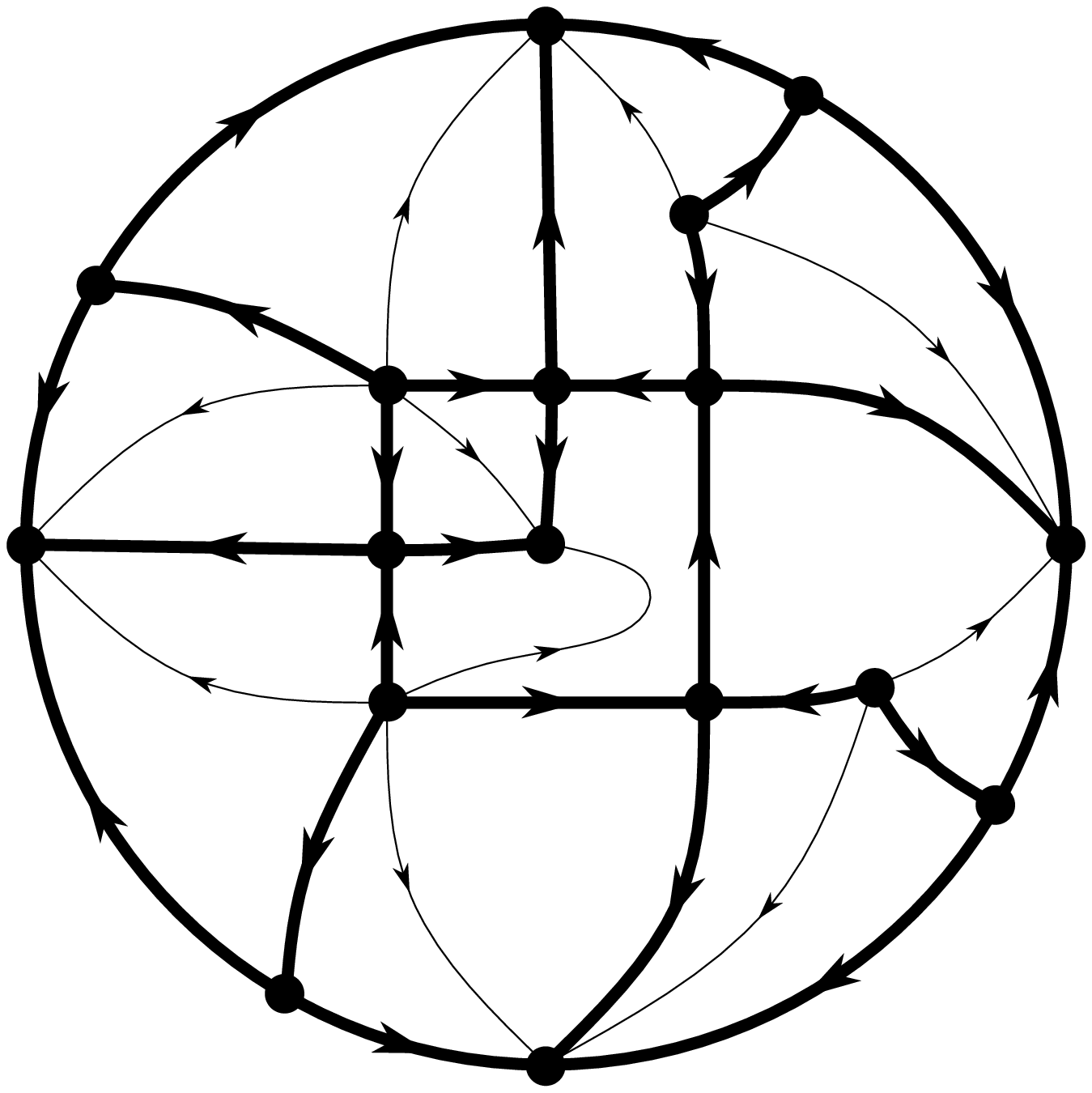} 
				\end{overpic}
				
				Case~$3.10$.
			\end{center}
		\end{minipage}
	\end{center}
	$\;$
	\begin{center}
		\begin{minipage}{3.1cm}
			\begin{center}
				\begin{overpic}[height=3cm]{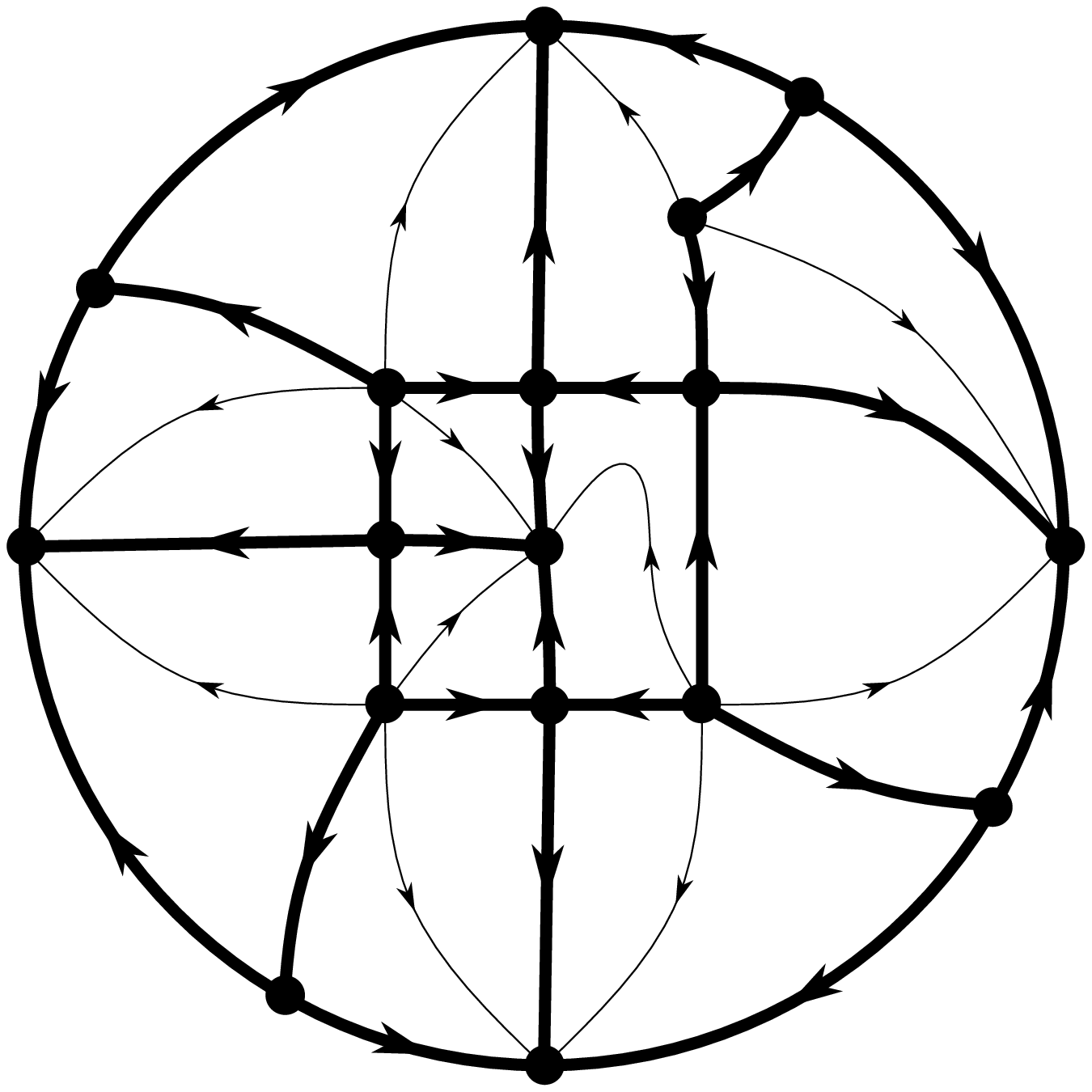} 
				\end{overpic}
				
				Case~$3.11$.
			\end{center}
		\end{minipage}
		\begin{minipage}{3.1cm}
			\begin{center}
				\begin{overpic}[height=3cm]{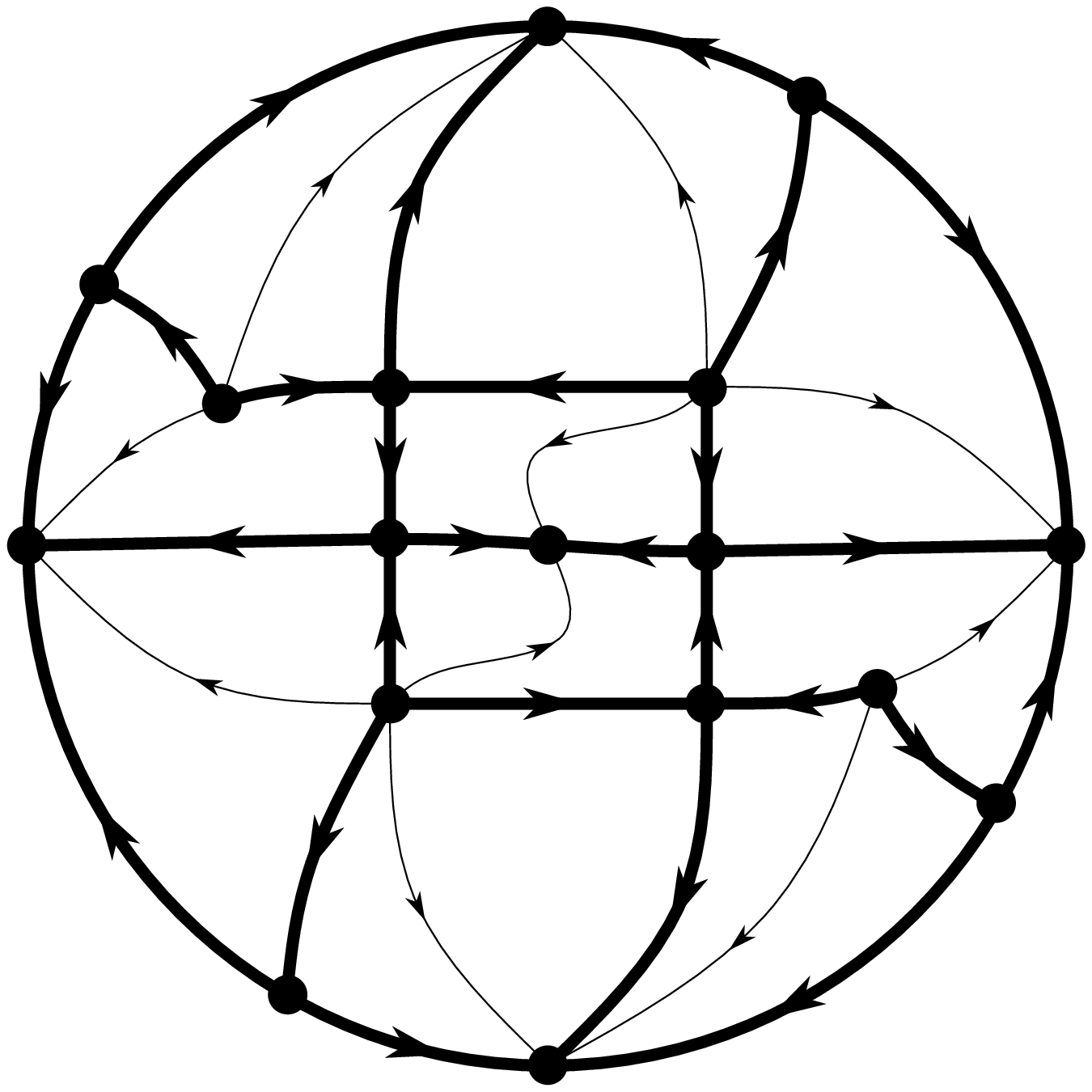} 
				\end{overpic}
				
				Case~$3.12$.
			\end{center}
		\end{minipage}
		\begin{minipage}{3.1cm}
			\begin{center}
				\begin{overpic}[height=3cm]{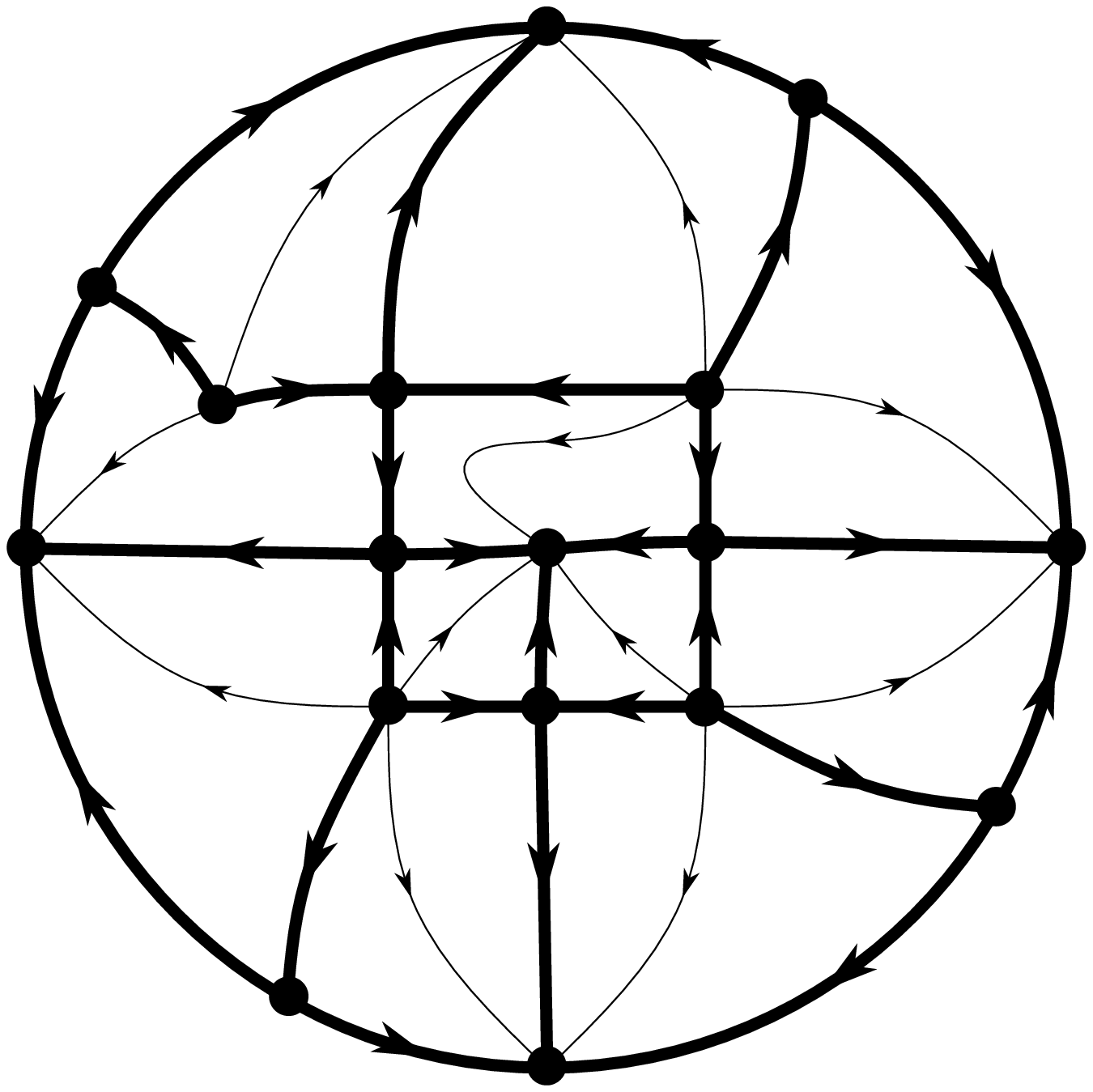} 
				\end{overpic}
				
				Case~$3.13$.
			\end{center}
		\end{minipage}
		\begin{minipage}{3.1cm}
			\begin{center}
				\begin{overpic}[height=3cm]{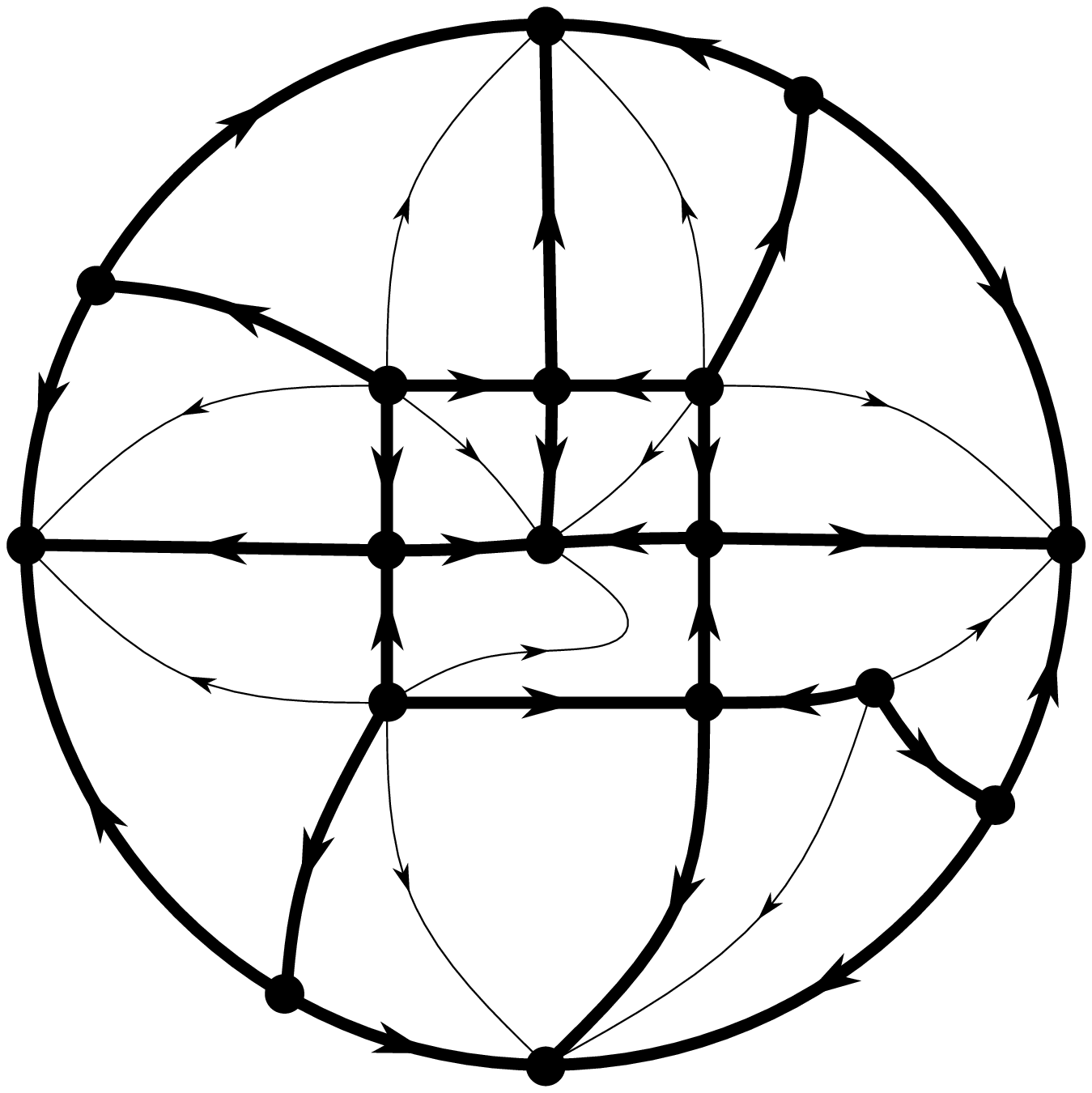} 
				\end{overpic}
				
				Case~$3.14$.
			\end{center}
		\end{minipage}
		\begin{minipage}{3.1cm}
			\begin{center}
				\begin{overpic}[height=3cm]{Case1.3.15.eps} 
				\end{overpic}
				
				Case~$3.15$.
			\end{center}
		\end{minipage}
	\end{center}
	\caption{Phase portraits of Family~$3$.}\label{Case1.3}
\end{figure}	

\section{Phase portraits under the hypothesis of Theorem~\ref{Theo14} (Family 4)}\label{Ap3}

To simplify the Table~\ref{Table6}, we denote,
	\[\det A=b_{01}-a_{01}b_{10}, \quad T=(\alpha-1)\alpha+b_{01}(\beta-1)\beta, \quad K=a_{01}b_{01}(\beta-1)\beta-b_{10}(\alpha-1)\alpha,\]
	\[\Delta=(b_{10}-a_{01})^2+4b_{01}, \quad \delta=\left\{\begin{array}{ll} 2\frac{|b_{01}|}{a_{01}}+(b_{10}-a_{01}), & \text{if } a_{01}\neq0, \vspace{0.1cm} \\ +\infty, & \text{if } a_{01}=0. \end{array}\right.\]

\begin{table}[h]
	\caption{Table of the realizable cases under hypothesis of Theorem~\ref{Theo14} (Family $4$).}\label{Table6}
	\begin{tabular}{c c c c c c c c c c}
		\hline 
		Cases & $q_1$ & $q_2$ & $q_3$ & $q_4$ & $\det A$ & $T$ & $K$ & $\Delta$ & $\delta$ \\
		\hline
		\rowcolor{mygray}
		$4.1a.i$ & $q_1>p_2$ & $q_2<p_2$ & $q_3<p_4$ & $q_4>p_4$ & $\det A>0$ & $T>0$ & $K>0$ & & \\
		$4.1a.ii$ & $q_1>p_2$ & $q_2<p_2$ & $q_3<p_4$ & $q_4>p_4$ & $\det A>0$ & $T>0$ & $K<0$ & & \\
		\rowcolor{mygray}
		$4.1b.i$ & $q_1>p_2$ & $q_2<p_2$ & $q_3<p_4$ & $q_4>p_4$ & $\det A>0$ & $T<0$ & $K>0$ & & \\
		$4.1b.ii$ & $q_1>p_2$ & $q_2<p_2$ & $q_3<p_4$ & $q_4>p_4$ & $\det A>0$ & $T<0$ & $K<0$ & & \\		
		\rowcolor{mygray}
		$4.2a$ & $q_1<p_2$ & $q_2<p_2$ & $q_3<p_4$ & $q_4>p_4$ & $\det A>0$ & $T>0$ & & & \\
		$4.2b$ & $q_1<p_2$ & $q_2<p_2$ & $q_3<p_4$ & $q_4>p_4$ & $\det A>0$ & $T<0$ & & & \\	
		\rowcolor{mygray}	
		$4.3a$ & $q_1>p_2$ & $q_2<p_2$ & $q_3>p_4$ & $q_4>p_4$ & $\det A>0$ & $T>0$ & & & \\
		$4.3b$ & $q_1>p_2$ & $q_2<p_2$ & $q_3>p_4$ & $q_4>p_4$ & $\det A>0$ & $T<0$ & & & \\	
		\rowcolor{mygray}	
		$4.4a.i$ & $q_1<p_2$ & $q_2<p_2$ & $q_3>p_4$ & $q_4>p_4$ & $\det A>0$ & $T>0$ &&& \\
		$4.4a.ii$ & $q_1<p_2$ & $q_2<p_2$ & $q_3>p_4$ & $q_4>p_4$ & $\det A>0$  & $T<0$ &&& \\
		\rowcolor{mygray}
		$4.4b.i$ & $q_1<p_2$ & $q_2<p_2$ & $q_3>p_4$ & $q_4>p_4$ &  $\det A<0$ & & & $\Delta>0$ & $\delta>0$ \\
		$4.4b.ii$ & $q_1<p_2$ & $q_2<p_2$ & $q_3>p_4$ & $q_4>p_4$ & $\det A<0$ & & & $\Delta>0$ &  $\delta<0$ \\
		\rowcolor{mygray}
		$4.4b.iii$ & $q_1<p_2$ & $q_2<p_2$ & $q_3>p_4$ & $q_4>p_4$ & $\det A<0$ & & & $\Delta<0$ & \\		
		$4.5a$ & $q_1>p_2$ & $q_2>p_2$ & $q_3<p_4$ & $q_4>p_4$ & $\det A>0$ & $T>0$ & & &  \\
		\rowcolor{mygray}
		$4.5b$ & $q_1>p_2$ & $q_2>p_2$ & $q_3<p_4$ & $q_4>p_4$ & $\det A>0$ & $T<0$ & & &  \\		
		$4.6a.i$ & $q_1>p_2$ & $q_2>p_2$ & $q_3>p_4$ & $q_4>p_4$ & $\det A>0$ & $T>0$ & & & \\
		\rowcolor{mygray}
		$4.6a.ii$ & $q_1>p_2$ & $q_2>p_2$ & $q_3>p_4$ & $q_4>p_4$ & $\det A>0$ & $T<0$ & & & \\
		$4.6b.i$ & $q_1>p_2$ & $q_2>p_2$ & $q_3>p_4$ & $q_4>p_4$ & $\det A<0$ & & & $\Delta>0$ & $\delta>0$ \\
		\rowcolor{mygray}
		$4.6b.ii$ & $q_1>p_2$ & $q_2>p_2$ & $q_3>p_4$ & $q_4>p_4$ & $\det A<0$ & & & $\Delta>0$ & $\delta<0$ \\
		$4.6b.iii$ & $q_1>p_2$ & $q_2>p_2$ & $q_3>p_4$ & $q_4>p_4$ & $\det A<0$ & & & $\Delta<0$ & \\	
		\rowcolor{mygray}	
		$4.7a$ & $q_1<p_2$ & $q_2>p_2$ & $q_3>p_4$ & $q_4>p_4$ & $\det A<0$ & & & $\Delta>0$ & $\delta>0$ \\
		$4.7b$ & $q_1<p_2$ & $q_2>p_2$ & $q_3>p_4$ & $q_4>p_4$ & $\det A<0$ & & & $\Delta>0$ & $\delta<0$ \\
		\rowcolor{mygray}
		$4.7c$ & $q_1<p_2$ & $q_2>p_2$ & $q_3>p_4$ & $q_4>p_4$ & $\det A<0$ & & & $\Delta<0$ & \\
		$4.8a$ & $q_1>p_2$ & $q_2<p_2$ & $q_3<p_4$ & $q_4<p_4$ & $\det A>0$ & $T>0$ & & & \\
		\rowcolor{mygray}
		$4.8b$ & $q_1>p_2$ & $q_2<p_2$ & $q_3<p_4$ & $q_4<p_4$ & $\det A>0$ & $T<0$ & & & \\		
		$4.9a.i$ & $q_1<p_2$ & $q_2<p_2$ & $q_3<p_4$ & $q_4<p_4$ & $\det A>0$ & $T>0$ & & & \\
		\rowcolor{mygray}
		$4.9a.ii$ & $q_1<p_2$ & $q_2<p_2$ & $q_3<p_4$ & $q_4<p_4$ & $\det A>0$ & $T<0$ & & & \\
		$4.9b.i$ & $q_1<p_2$ & $q_2<p_2$ & $q_3<p_4$ & $q_4<p_4$ & $\det A<0$ & & & $\Delta>0$ & $\delta>0$ \\
		\rowcolor{mygray}
		$4.9b.ii$ & $q_1<p_2$ & $q_2<p_2$ & $q_3<p_4$ & $q_4<p_4$ & $\det A<0$ & & & $\Delta>0$ & $\delta<0$ \\
		$4.9b.iii$ & $q_1<p_2$ & $q_2<p_2$ & $q_3<p_4$ & $q_4<p_4$ & $\det A<0$ & & & $\Delta>0$ & \\		
		\rowcolor{mygray}
		$4.10a$ & $q_1<p_2$ & $q_2<p_2$ & $q_3>p_4$ & $q_4<p_4$ & $\det A<0$ & & & $\Delta>0$ & $\delta>0$ \\
		$4.10b$ & $q_1<p_2$ & $q_2<p_2$ & $q_3>p_4$ & $q_4<p_4$ & $\det A<0$ & & & $\Delta>0$ & $\delta<0$ \\
		\rowcolor{mygray}
		$4.10c$ & $q_1<p_2$ & $q_2<p_2$ & $q_3>p_4$ & $q_4<p_4$ & $\det A<0$ & & & $\Delta<0$ & \\	
		$4.11a.i$ & $q_1>p_2$ & $q_2>p_2$ & $q_3<p_4$ & $q_4<p_4$ & $\det A>0$ & $T>0$ & & & \\
		\rowcolor{mygray}
		$4.11a.ii$ & $q_1>p_2$ & $q_2>p_2$ & $q_3<p_4$ & $q_4<p_4$ & $\det A>0$ & $T<0$ & & & \\
		$4.11b.i$ & $q_1>p_2$ & $q_2>p_2$ & $q_3<p_4$ & $q_4<p_4$ & $\det A<0$ & & & $\Delta>0$ & $\delta>0$ \\
		\rowcolor{mygray}
		$4.11b.ii$ & $q_1>p_2$ & $q_2>p_2$ & $q_3<p_4$ & $q_4<p_4$ & $\det A<0$ & & & $\Delta>0$ & $\delta<0$ \\
		$4.11b.iii$ & $q_1>p_2$ & $q_2>p_2$ & $q_3<p_4$ & $q_4<p_4$ & $\det A<0$ & & & $\Delta<0$ & \\	
		\rowcolor{mygray}	
		$4.12a$ & $q_1<p_2$ & $q_2>p_2$ & $q_3<p_4$ & $q_4<p_4$ & $\det A<0$  & & & $\Delta>0$ & $\delta>0$ \\
		$4.12b$ & $q_1<p_2$ & $q_2>p_2$ & $q_3<p_4$ & $q_4<p_4$ & $\det A<0$  & & & $\Delta>0$ & $\delta<0$ \\
		\rowcolor{mygray}
		$4.12c$ & $q_1<p_2$ & $q_2>p_2$ & $q_3<p_4$ & $q_4<p_4$ & $\det A<0$  & & & $\Delta<0$ & \\		
		$4.13a$ & $q_1>p_2$ & $q_2>p_2$ & $q_3>p_4$ & $q_4<p_4$ & $\det A<0$ & & & $\Delta>0$ & $\delta>0$ \\
		\rowcolor{mygray}
		$4.13b$ & $q_1>p_2$ & $q_2>p_2$ & $q_3>p_4$ & $q_4<p_4$ & $\det A<0$ & & & $\Delta>0$ & $\delta<0$ \\
		$4.13c$ & $q_1>p_2$ & $q_2>p_2$ & $q_3>p_4$ & $q_4<p_4$ & $\det A<0$ & & & $\Delta<0$ & \\		
		\rowcolor{mygray}
		$4.14a$ & $q_1<p_2$ & $q_2>p_2$ & $q_3>p_4$ & $q_4<p_4$ & $\det A<0$ & & & $\Delta>0$ & $\delta>0$ \\
		$4.14b$ & $q_1<p_2$ & $q_2>p_2$ & $q_3>p_4$ & $q_4<p_4$ & $\det A<0$ & & & $\Delta>0$ & $\delta<0$ \\
		\rowcolor{mygray}
		$4.14c$ & $q_1<p_2$ & $q_2>p_2$ & $q_3>p_4$ & $q_4<p_4$ & $\det A<0$ & & & $\Delta<0$ &  \\
		\hline
	\end{tabular}
\end{table}

\begin{figure}[h]
	\begin{center}
		\begin{minipage}{3.1cm}
			\begin{center}
				\begin{overpic}[height=3cm]{Case1.4.1a1.eps} 
				\end{overpic}
				
				Case~$4.1a.i$.
			\end{center}
		\end{minipage}
		\begin{minipage}{3.1cm}
			\begin{center}
				\begin{overpic}[height=3cm]{Case1.4.1a2.eps} 
				\end{overpic}
				
				Case~$4.1a.ii$.
			\end{center}
		\end{minipage}
		\begin{minipage}{3.1cm}
			\begin{center}
				\begin{overpic}[height=3cm]{Case1.4.1b1.eps} 
				\end{overpic}
				
				Case~$4.1b.i$.
			\end{center}
		\end{minipage}	
		\begin{minipage}{3.1cm}
			\begin{center}
				\begin{overpic}[height=3cm]{Case1.4.1b2.eps} 
				\end{overpic}
				
				Case~$4.1b.ii$.
			\end{center}
		\end{minipage}
		\begin{minipage}{3.1cm}
			\begin{center}
				\begin{overpic}[height=3cm]{Case1.4.2a.eps} 
				\end{overpic}
				
				Case~$4.2a$.
			\end{center}
		\end{minipage}
	\end{center}
	$\;$
	\begin{center}
		\begin{minipage}{3.1cm}
			\begin{center}
				\begin{overpic}[height=3cm]{Case1.4.2b.eps} 
				\end{overpic}
				
				Case~$4.2b$.
			\end{center}
		\end{minipage}
		\begin{minipage}{3.1cm}
			\begin{center}
				\begin{overpic}[height=3cm]{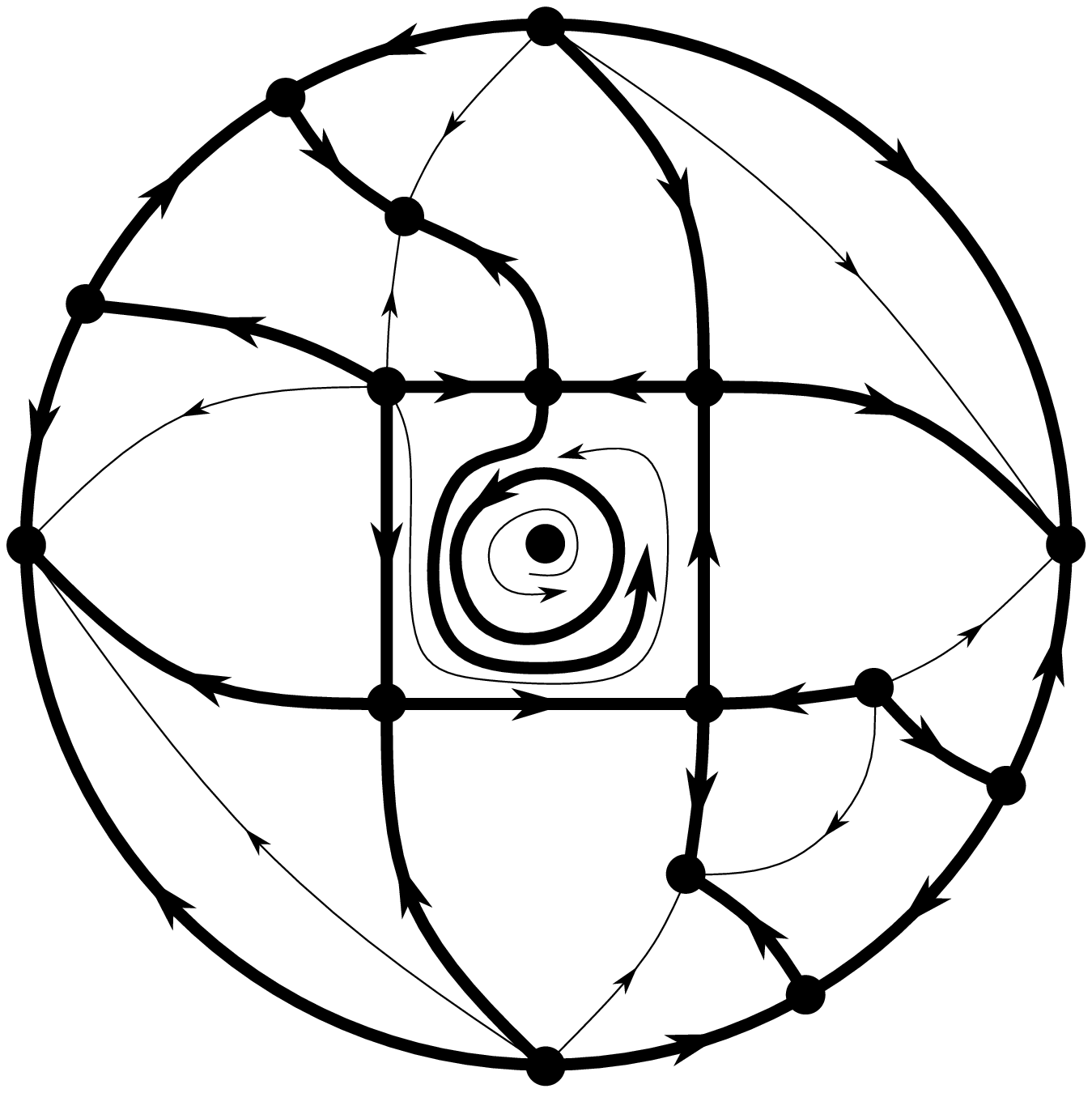} 
				\end{overpic}
				
				Case~$4.3a$.
			\end{center}
		\end{minipage}
		\begin{minipage}{3.1cm}
			\begin{center}
				\begin{overpic}[height=3cm]{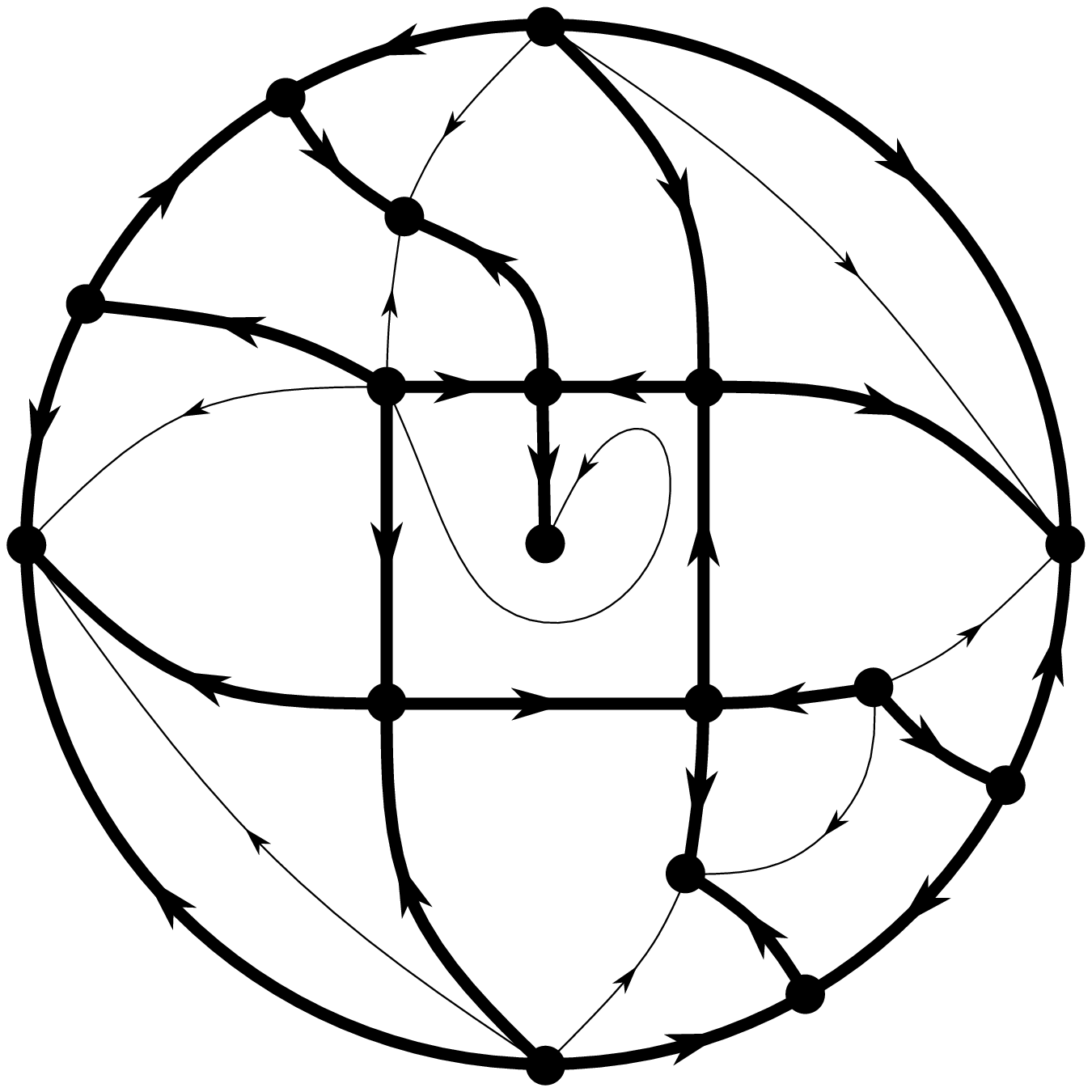} 
				\end{overpic}
				
				Case~$4.3b$.
			\end{center}
		\end{minipage}
		\begin{minipage}{3.1cm}
			\begin{center}
				\begin{overpic}[height=3cm]{Case1.4.4a1.eps} 
				\end{overpic}
				
				Case~$4.4a.i$.
			\end{center}
		\end{minipage}
		\begin{minipage}{3.1cm}
			\begin{center}
				\begin{overpic}[height=3cm]{Case1.4.4a2.eps} 
				\end{overpic}
				
				Case~$4.4a.ii$.
			\end{center}
		\end{minipage}
	\end{center}
	$\;$
	\begin{center}
		\begin{minipage}{3.1cm}
			\begin{center}
				\begin{overpic}[height=3cm]{Case1.4.4b1.1.eps} 
				\end{overpic}
				
				Case~$4.4b.i.1$.
			\end{center}
		\end{minipage}
		\begin{minipage}{3.1cm}
			\begin{center}
				\begin{overpic}[height=3cm]{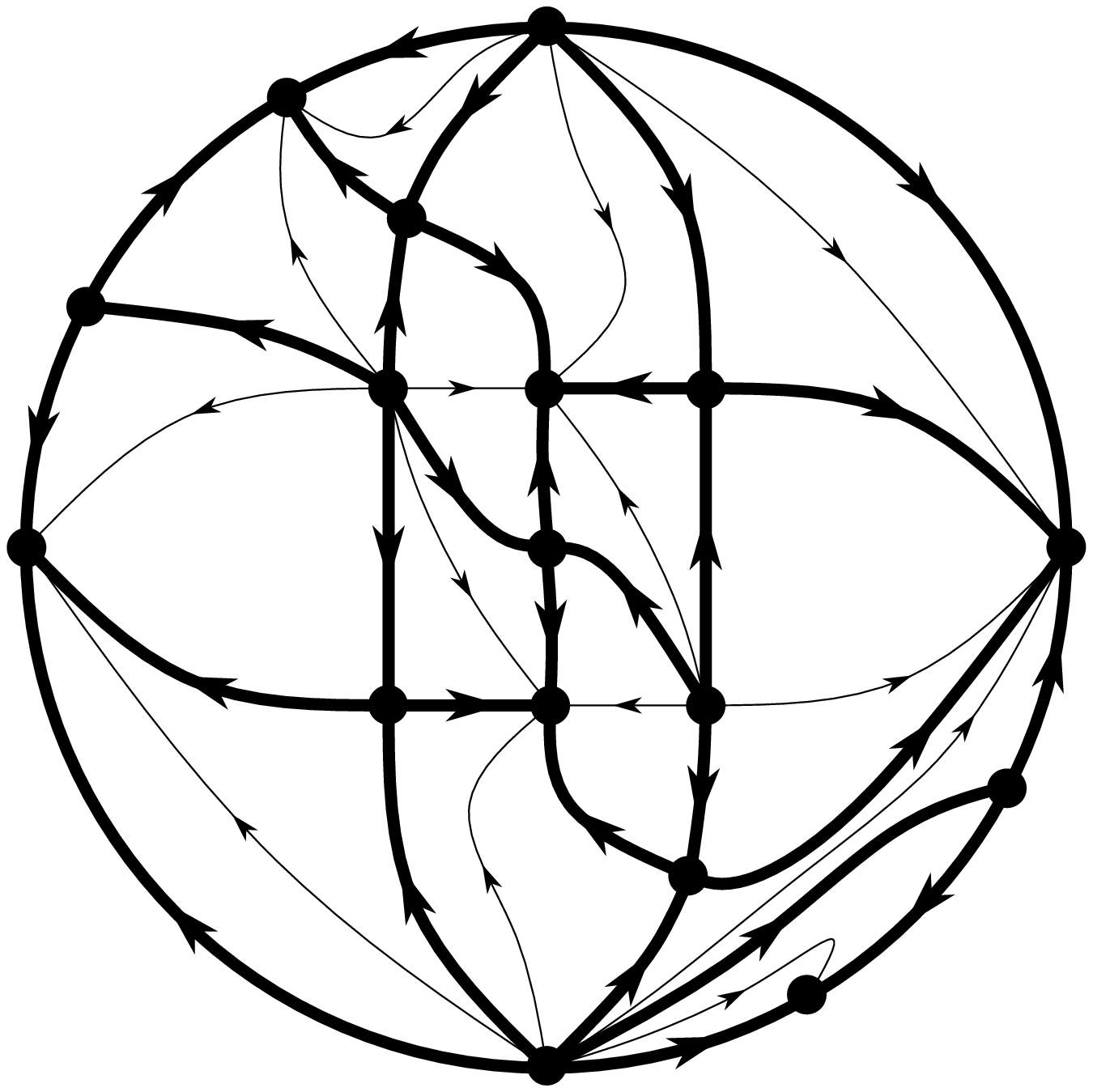} 
				\end{overpic}
				
				Case~$4.4b.i.2$.
			\end{center}
		\end{minipage}	
		\begin{minipage}{3.1cm}
			\begin{center}
				\begin{overpic}[height=3cm]{Case1.4.4b1.3.eps} 
				\end{overpic}
				
				Case~$4.4b.i.3$.
			\end{center}
		\end{minipage}
		\begin{minipage}{3.1cm}
			\begin{center}
				\begin{overpic}[height=3cm]{Case1.4.4b1.4.eps} 
				\end{overpic}
				
				Case~$4.4b.i.4$.
			\end{center}
		\end{minipage}
		\begin{minipage}{3.1cm}
			\begin{center}
				\begin{overpic}[height=3cm]{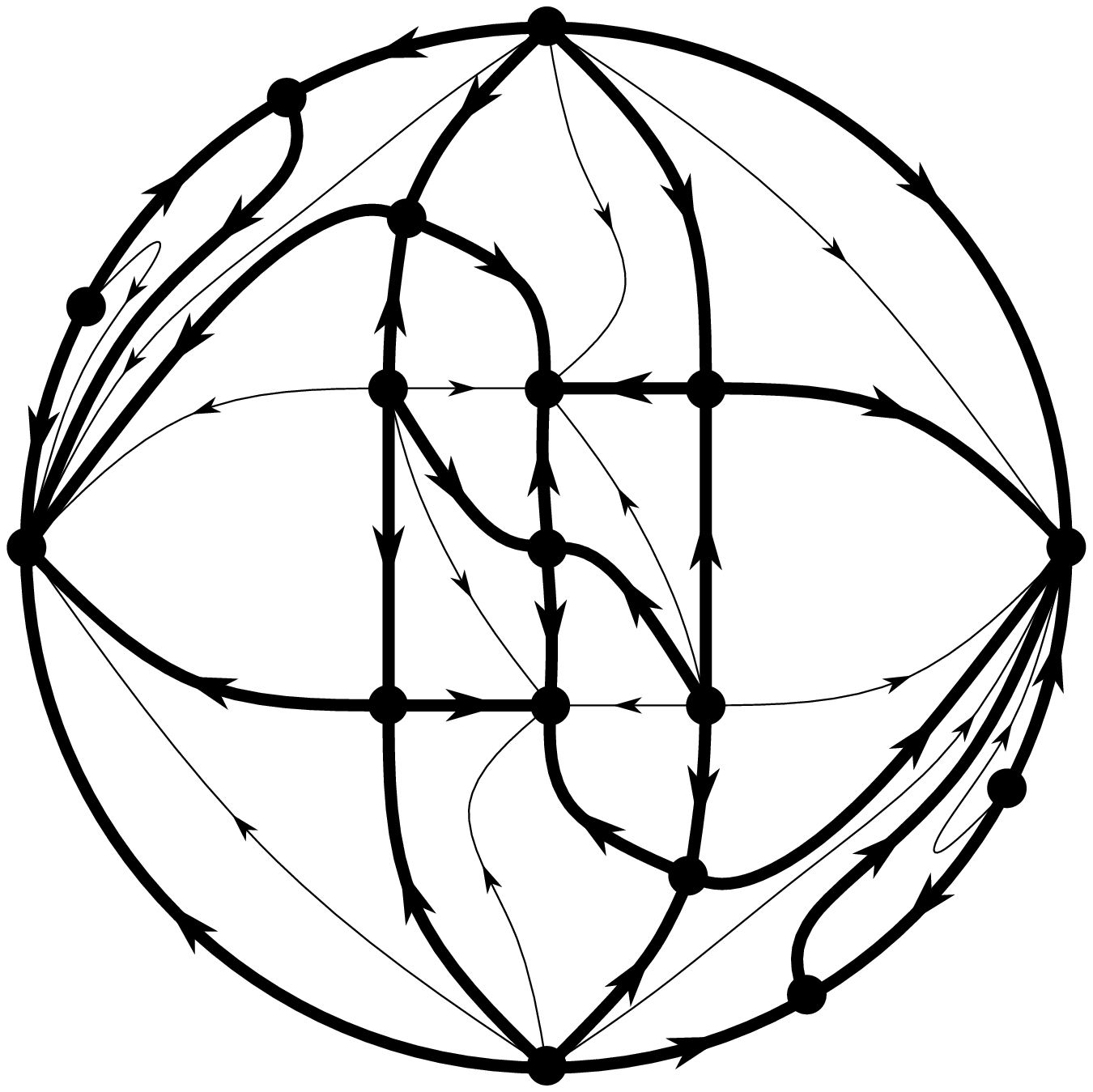} 
				\end{overpic}
				
				Case~$4.4b.ii$.
			\end{center}
		\end{minipage}
	\end{center}
	$\;$
	\begin{center}
		\begin{minipage}{3.1cm}
			\begin{center}
				\begin{overpic}[height=3cm]{Case1.4.4b3.eps} 
				\end{overpic}
				
				Case~$4.4b.iii$.
			\end{center}
		\end{minipage}
		\begin{minipage}{3.1cm}
			\begin{center}
				\begin{overpic}[height=3cm]{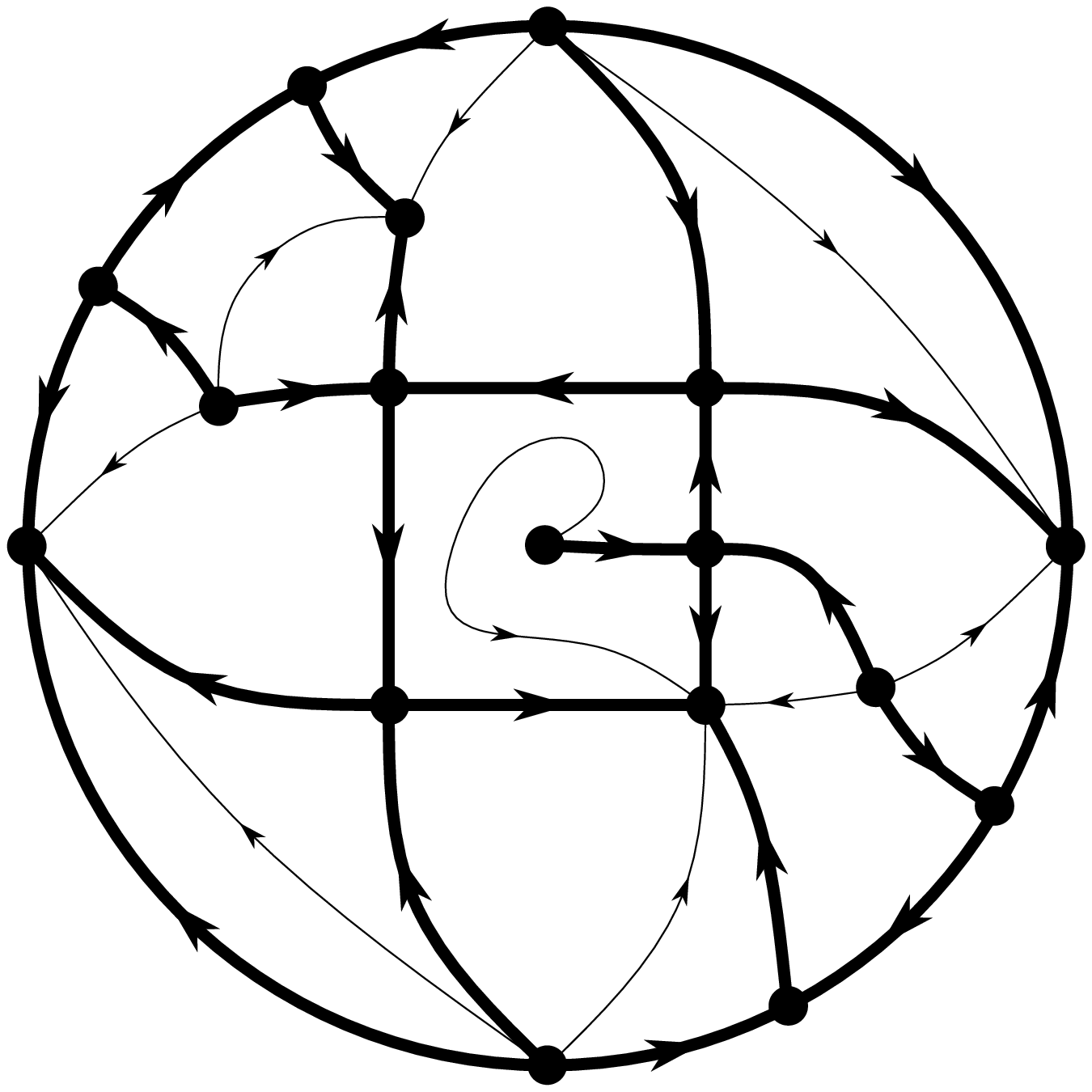} 
				\end{overpic}
				
				Case~$4.5a$.
			\end{center}
		\end{minipage}
		\begin{minipage}{3.1cm}
			\begin{center}
				\begin{overpic}[height=3cm]{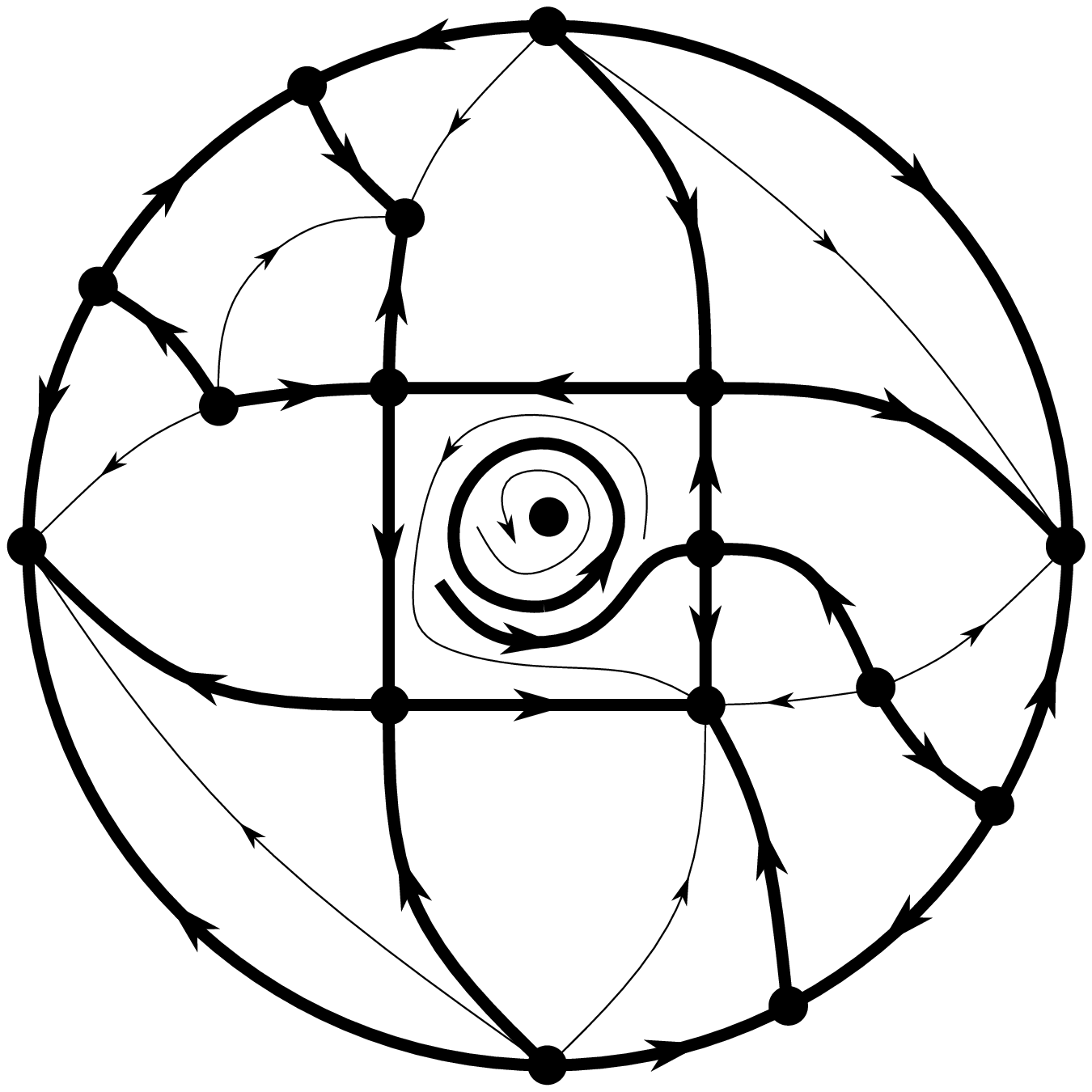} 
				\end{overpic}
				
				Case~$4.5b$.
			\end{center}
		\end{minipage}
		\begin{minipage}{3.1cm}
			\begin{center}
				\begin{overpic}[height=3cm]{Case1.4.6a1.1.eps} 
				\end{overpic}
				
				Case~$4.6a.i.1$.
			\end{center}
		\end{minipage}
		\begin{minipage}{3.1cm}
			\begin{center}
				\begin{overpic}[height=3cm]{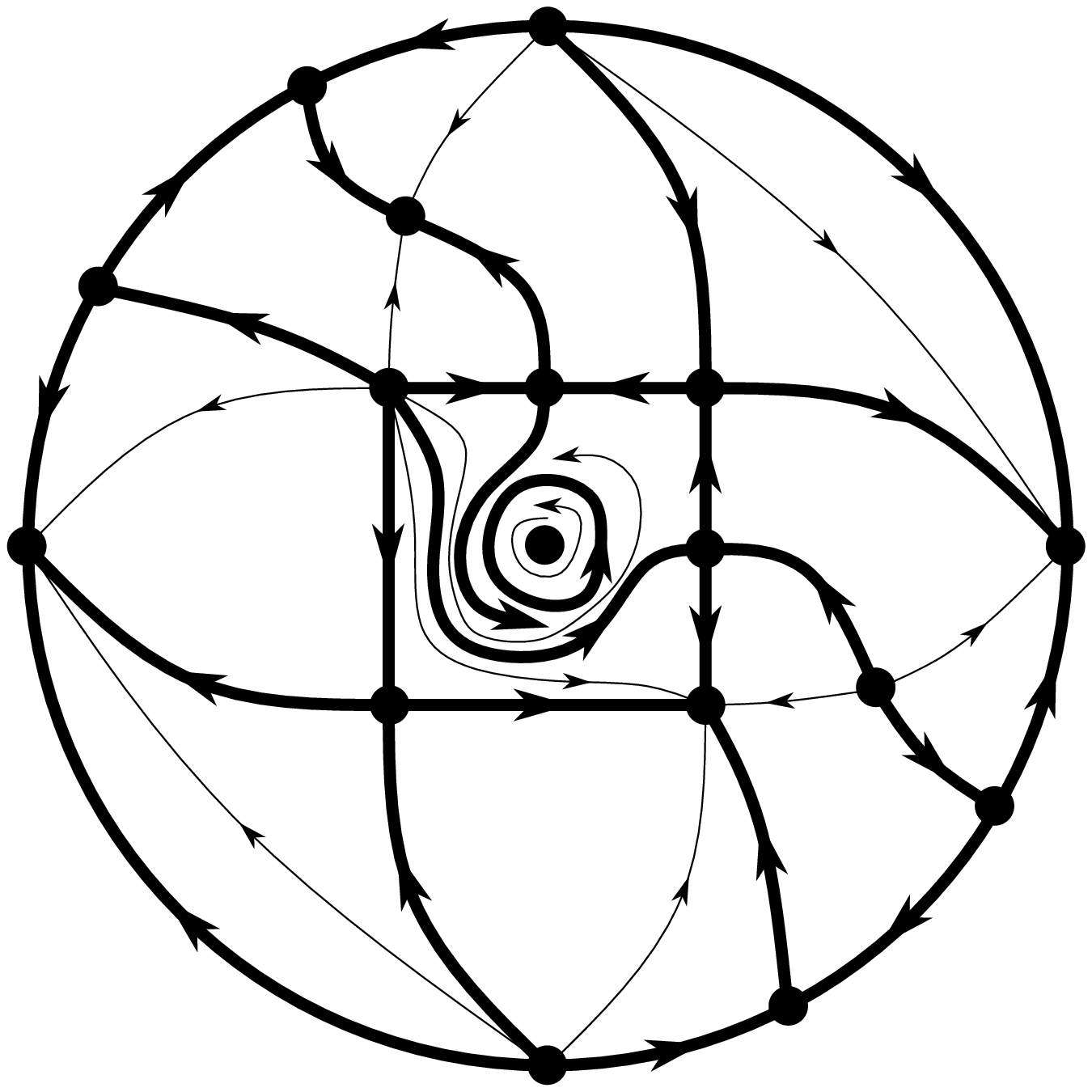} 
				\end{overpic}
				
				Case~$4.6a.i.2$.
			\end{center}
		\end{minipage}
	\end{center}
	$\;$
	\begin{center}
		\begin{minipage}{3.1cm}
			\begin{center}
				\begin{overpic}[height=3cm]{Case1.4.6a2.1.eps} 
				\end{overpic}
				
				Case~$4.6a.ii.1$.
			\end{center}
		\end{minipage}	
		\begin{minipage}{3.1cm}
			\begin{center}
				\begin{overpic}[height=3cm]{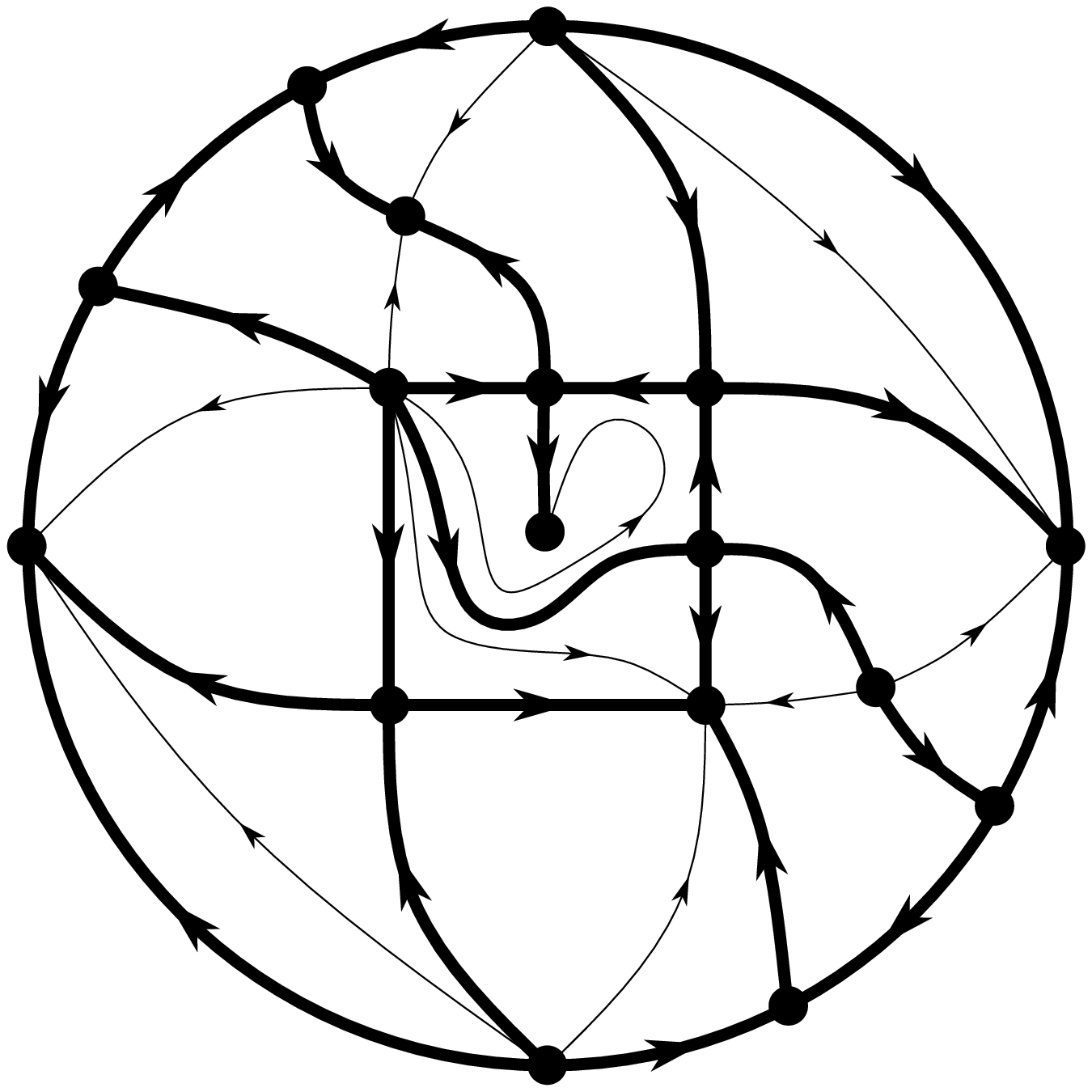} 
				\end{overpic}
				
				Case~$4.6a.ii.2$.
			\end{center}
		\end{minipage}
		\begin{minipage}{3.1cm}
			\begin{center}
				\begin{overpic}[height=3cm]{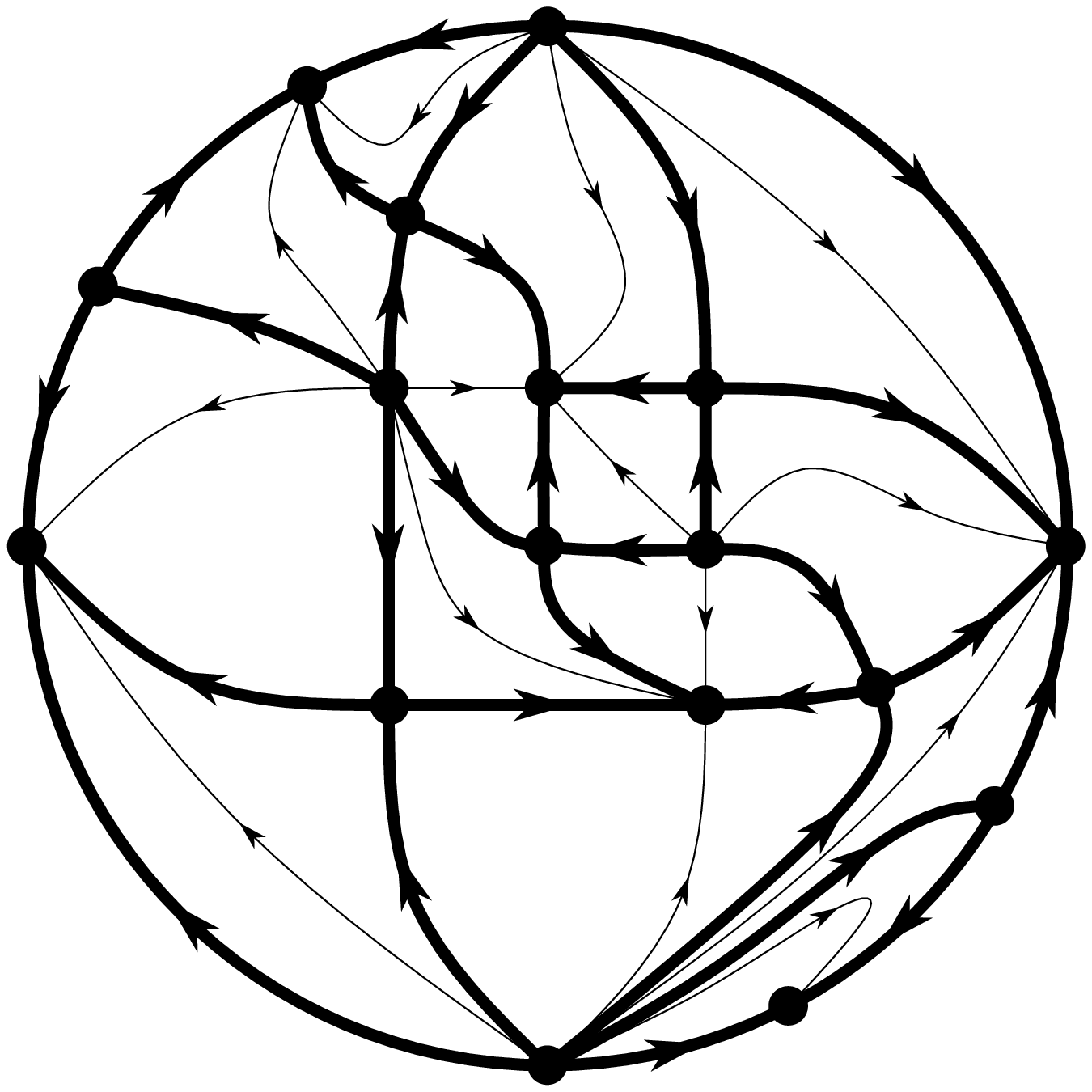} 
				\end{overpic}
				
				Case~$4.6b.i.1$.
			\end{center}
		\end{minipage}
		\begin{minipage}{3.1cm}
			\begin{center}
				\begin{overpic}[height=3cm]{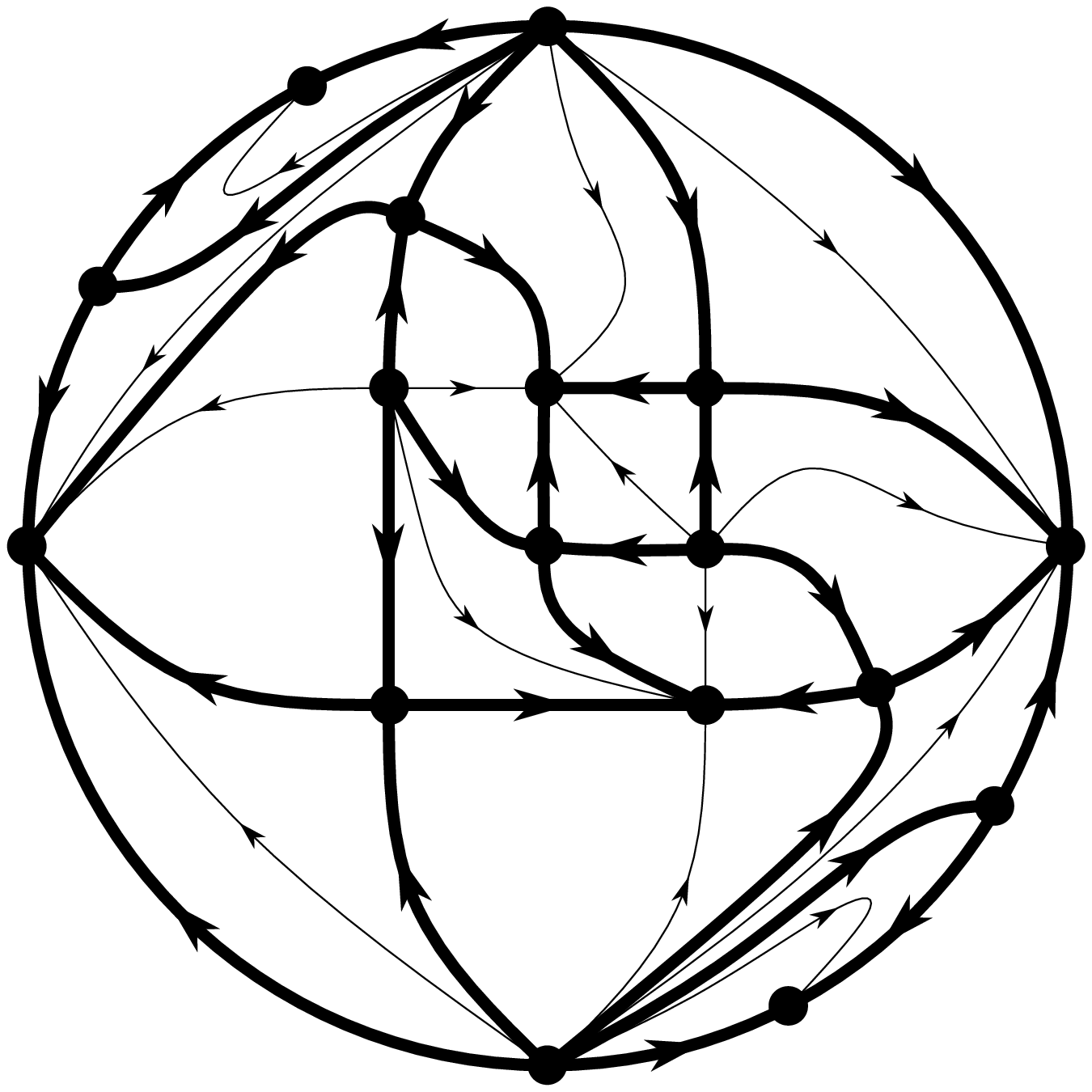} 
				\end{overpic}
				
				Case~$4.6b.i.2$.
			\end{center}
		\end{minipage}
		\begin{minipage}{3.1cm}
			\begin{center}
				\begin{overpic}[height=3cm]{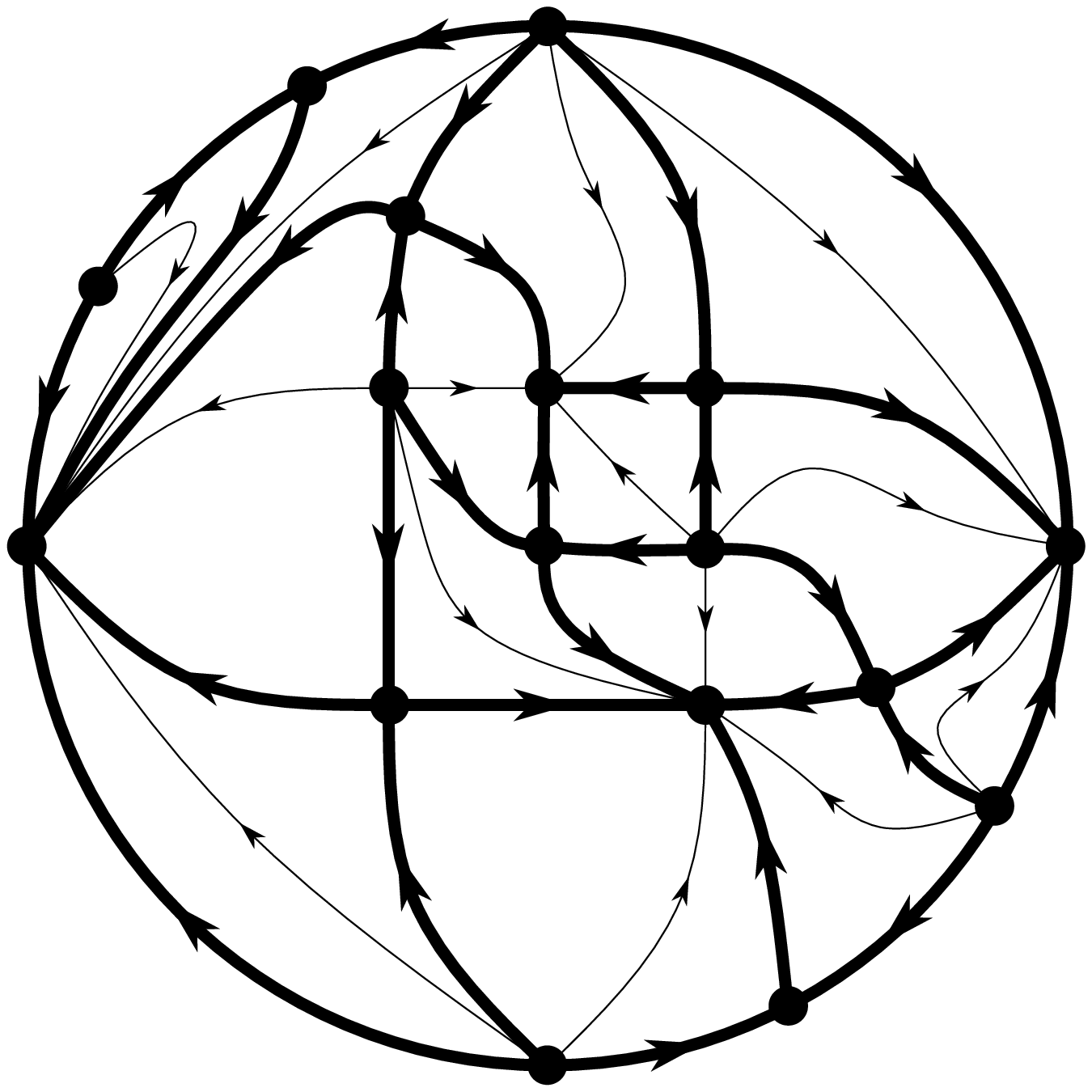} 
				\end{overpic}
				
				Case~$4.6b.ii.1$.
			\end{center}
		\end{minipage}
	\end{center}
	$\;$
	\begin{center}
		\begin{minipage}{3.1cm}
			\begin{center}
				\begin{overpic}[height=3cm]{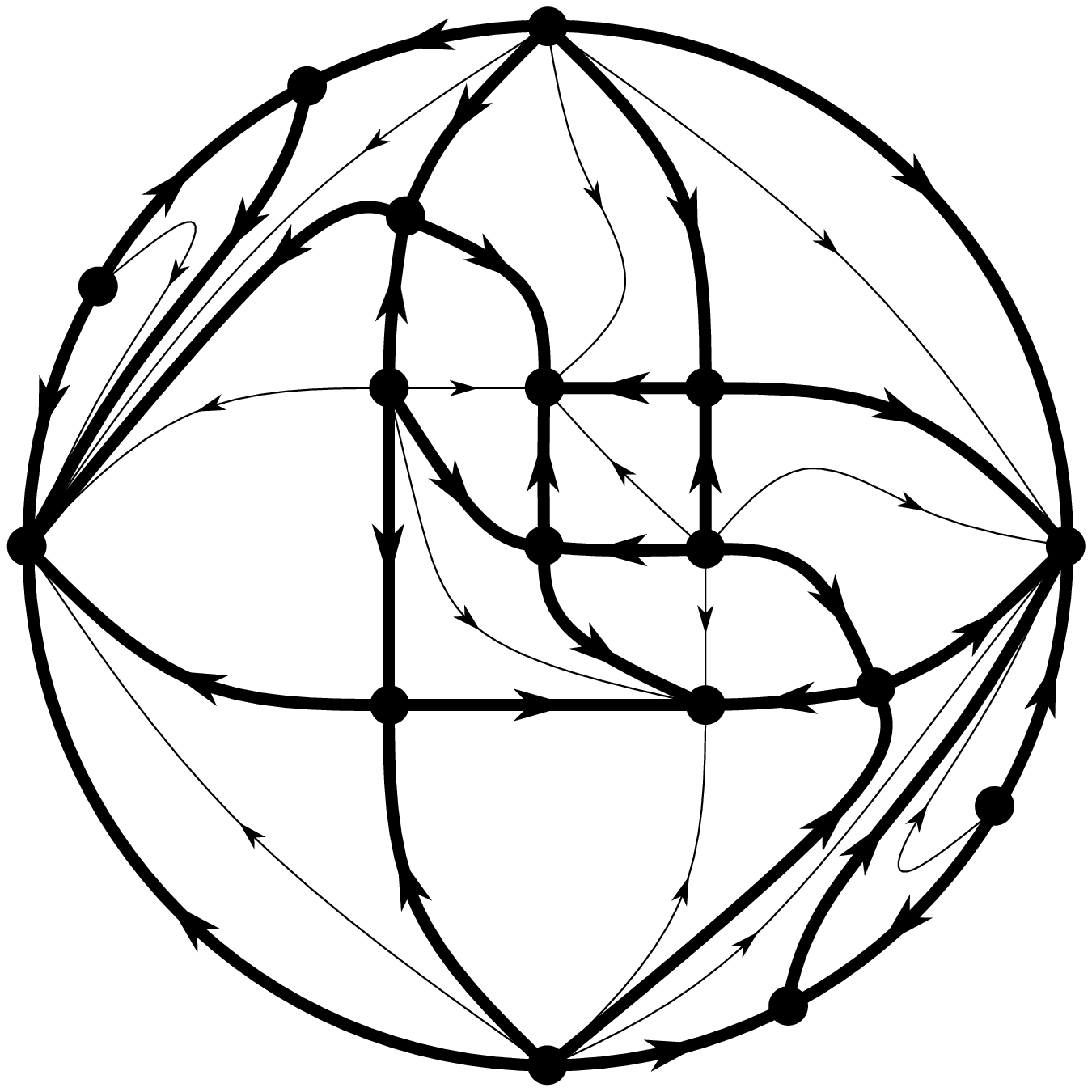} 
				\end{overpic}
				
				Case~$4.6b.ii.2$.
			\end{center}
		\end{minipage}
		\begin{minipage}{3.1cm}
			\begin{center}
				\begin{overpic}[height=3cm]{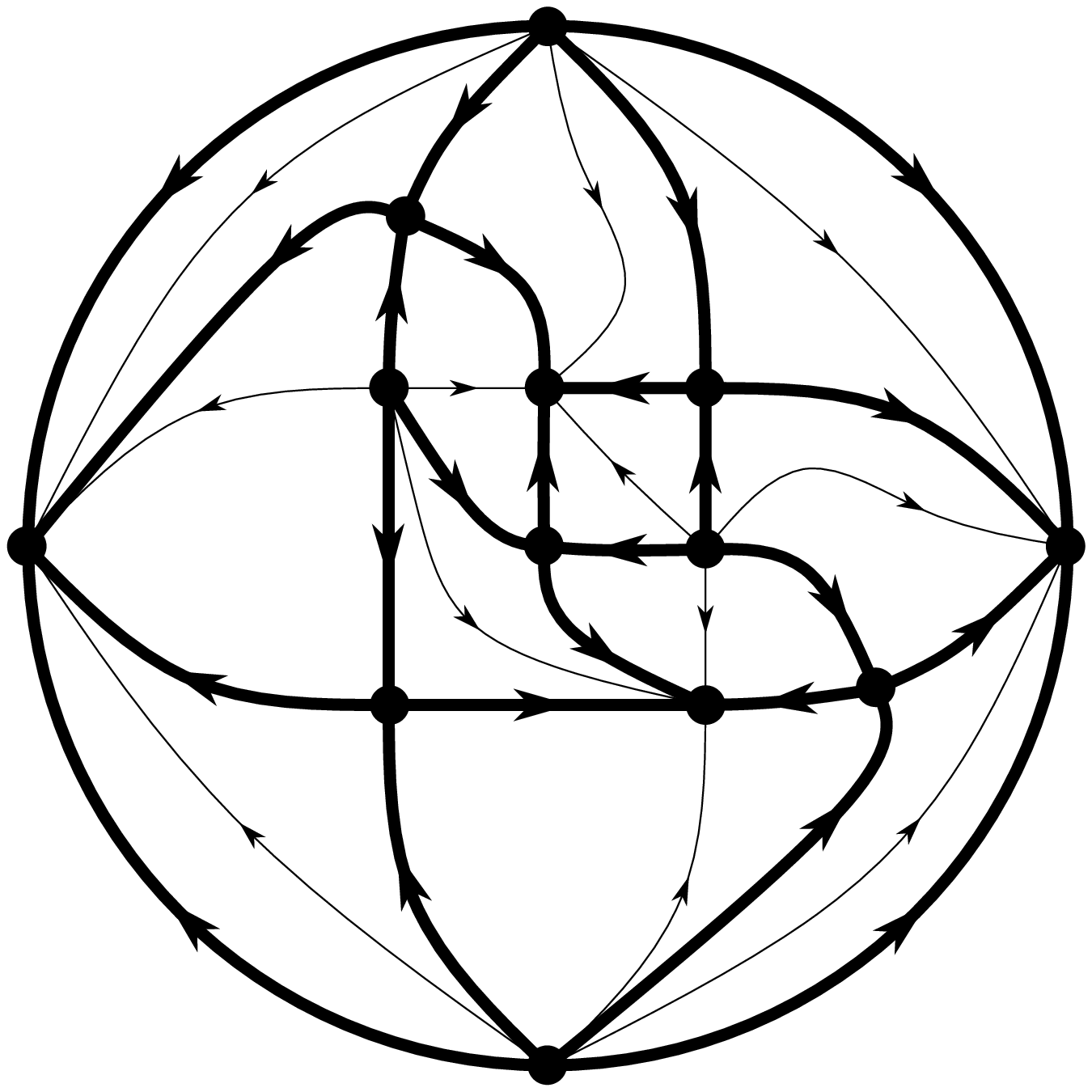} 
				\end{overpic}
				
				Case~$4.6b.iii$.
			\end{center}
		\end{minipage}
	\end{center}
	\caption{Phase portraits from Cases~$4.1$ to $4.6$.}\label{Case1.4a}
\end{figure}

\begin{figure}[h]
	\begin{center}
		\begin{minipage}{3.1cm}
			\begin{center}
				\begin{overpic}[height=3cm]{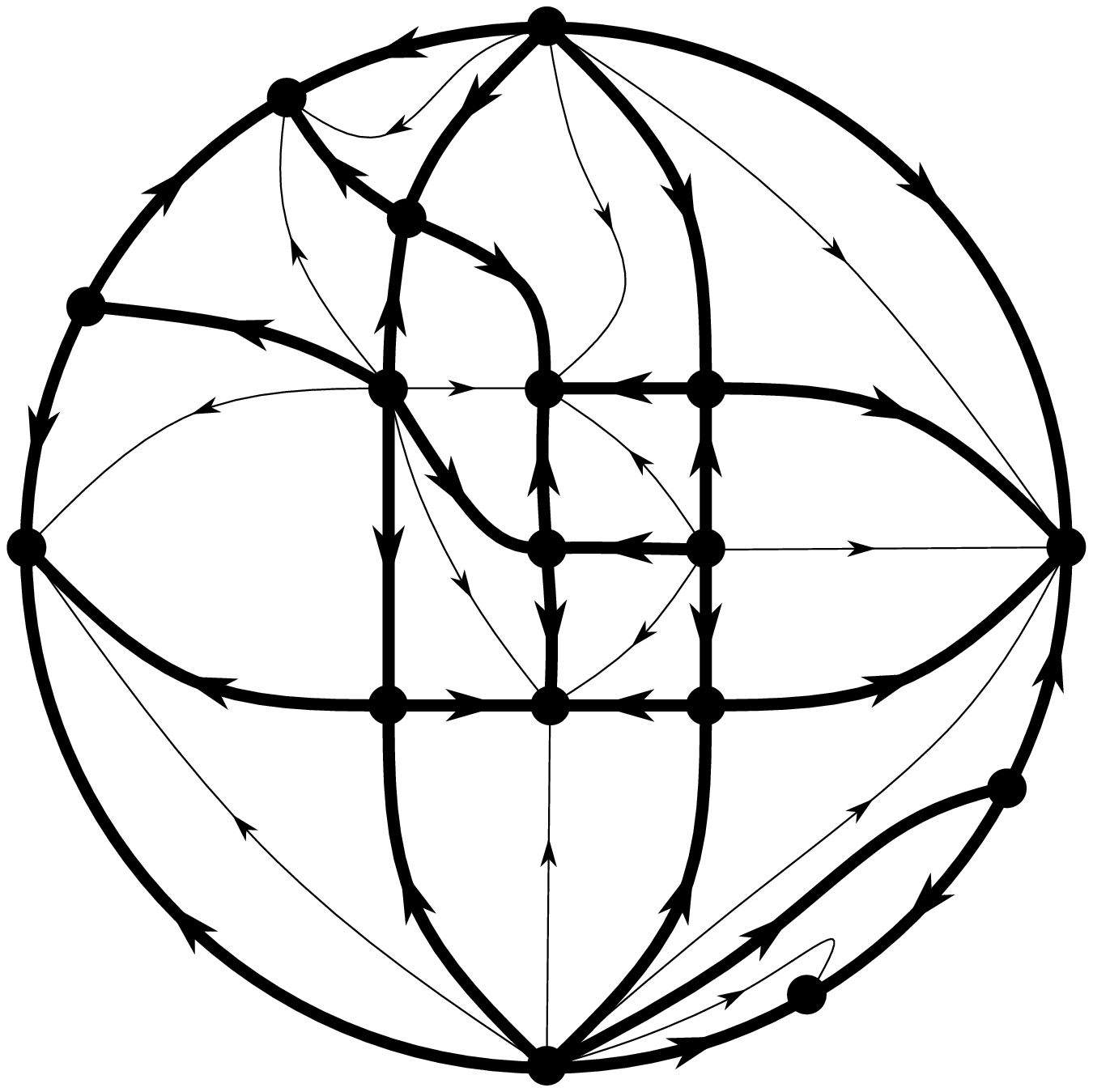} 
				\end{overpic}
				
				Case~$4.7a.1$.
			\end{center}
		\end{minipage}
		\begin{minipage}{3.1cm}
			\begin{center}
				\begin{overpic}[height=3cm]{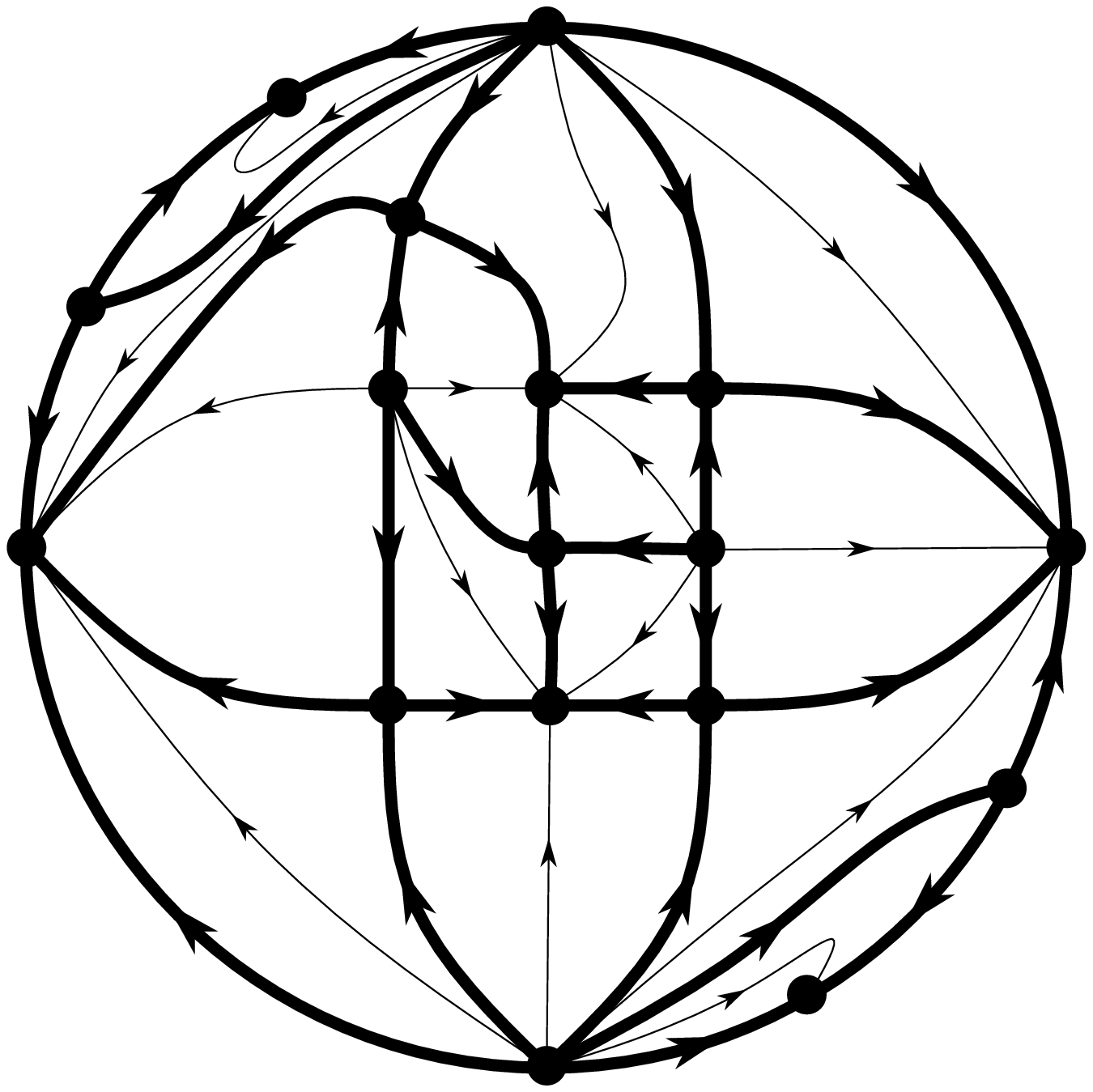} 
				\end{overpic}
				
				Case~$4.7a.2$.
			\end{center}
		\end{minipage}
		\begin{minipage}{3.1cm}
			\begin{center}
				\begin{overpic}[height=3cm]{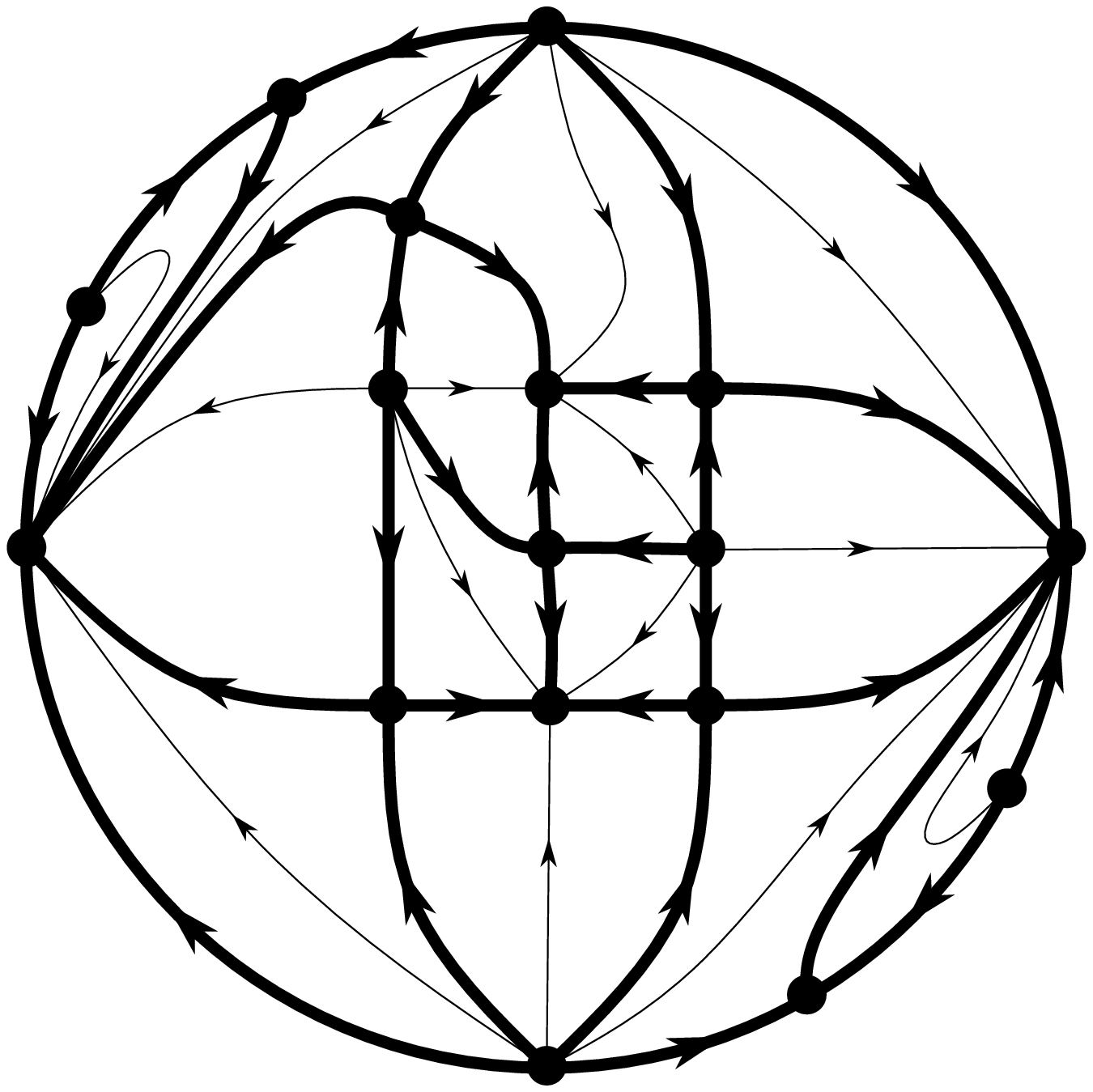} 
				\end{overpic}
				
				Case~$4.7b$.
			\end{center}
		\end{minipage}	
		\begin{minipage}{3.1cm}
			\begin{center}
				\begin{overpic}[height=3cm]{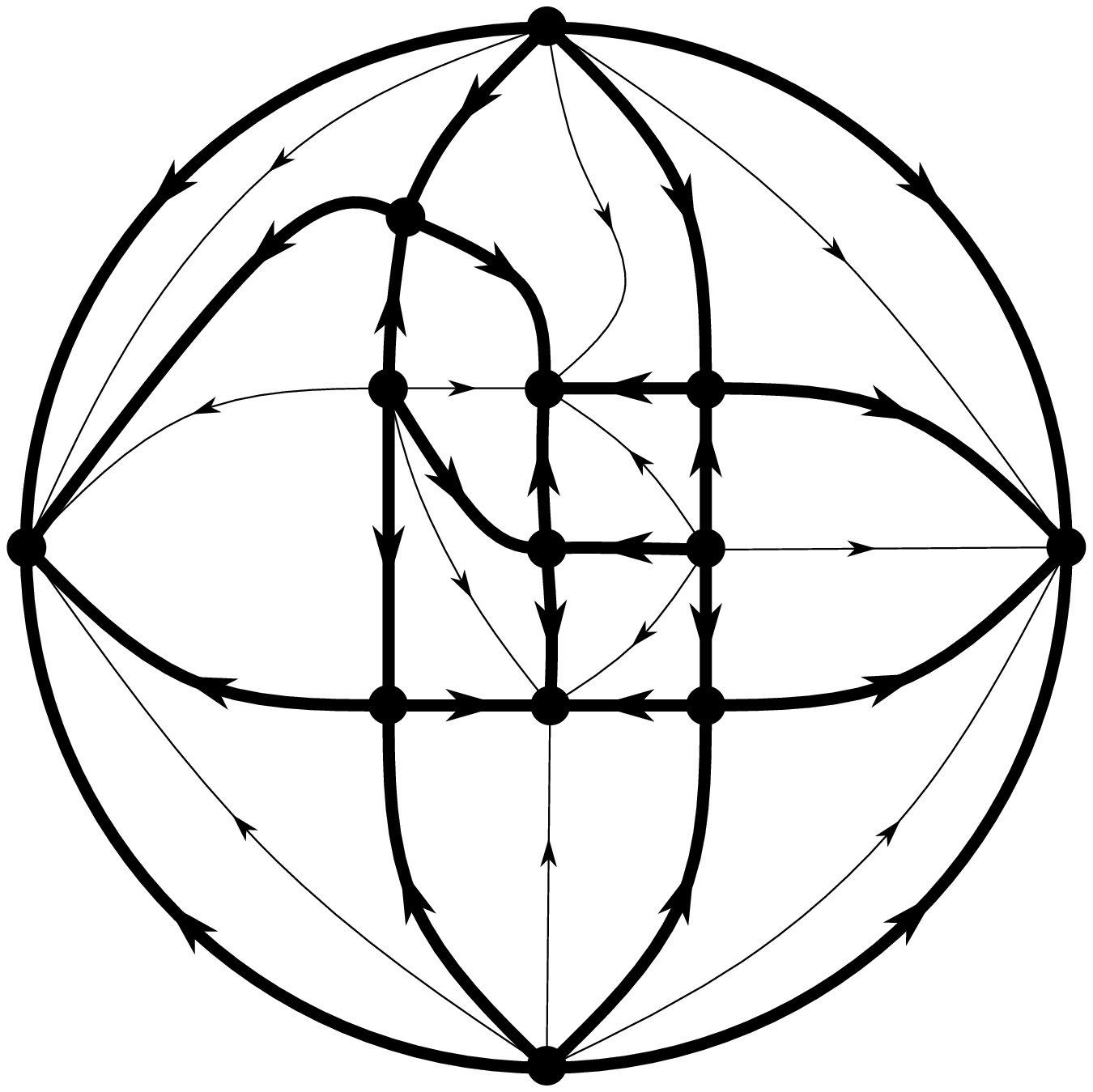} 
				\end{overpic}
				
				Case~$4.7c$.
			\end{center}
		\end{minipage}
		\begin{minipage}{3.1cm}
			\begin{center}
				\begin{overpic}[height=3cm]{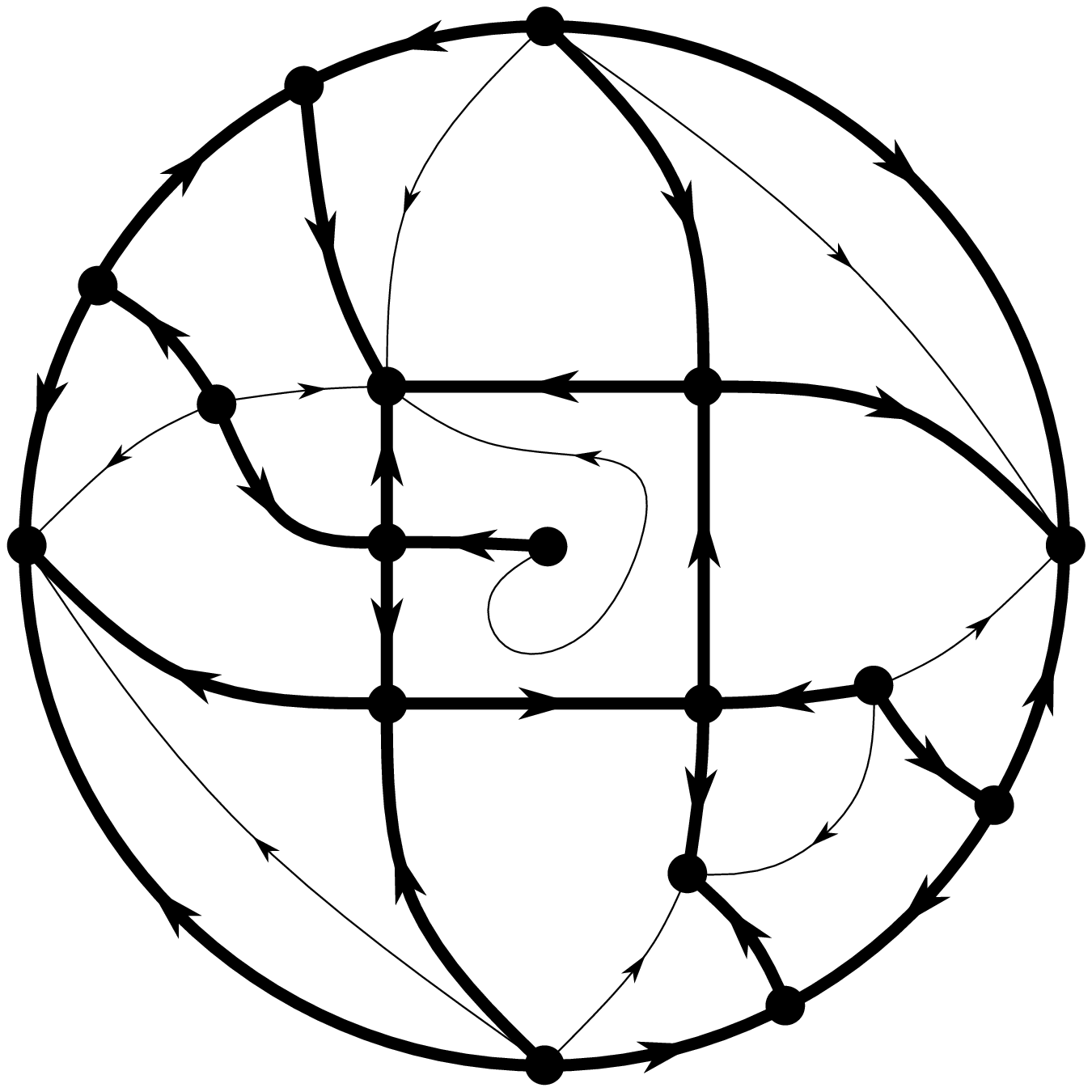} 
				\end{overpic}
				
				Case~$4.8a$.
			\end{center}
		\end{minipage}
	\end{center}
	$\;$
	\begin{center}
		\begin{minipage}{3.1cm}
			\begin{center}
				\begin{overpic}[height=3cm]{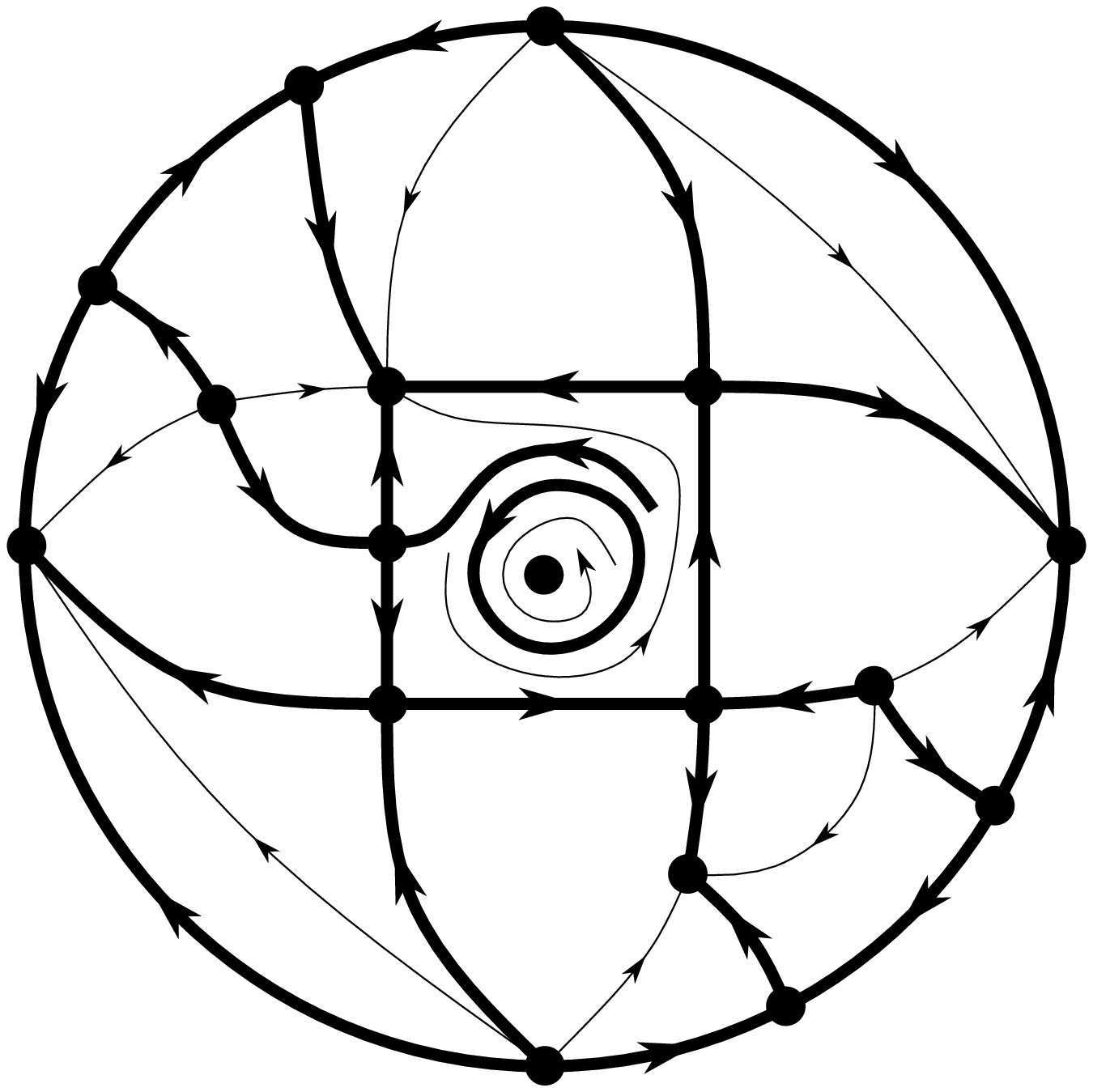} 
				\end{overpic}
				
				Case~$4.8b$.
			\end{center}
		\end{minipage}
		\begin{minipage}{3.1cm}
			\begin{center}
				\begin{overpic}[height=3cm]{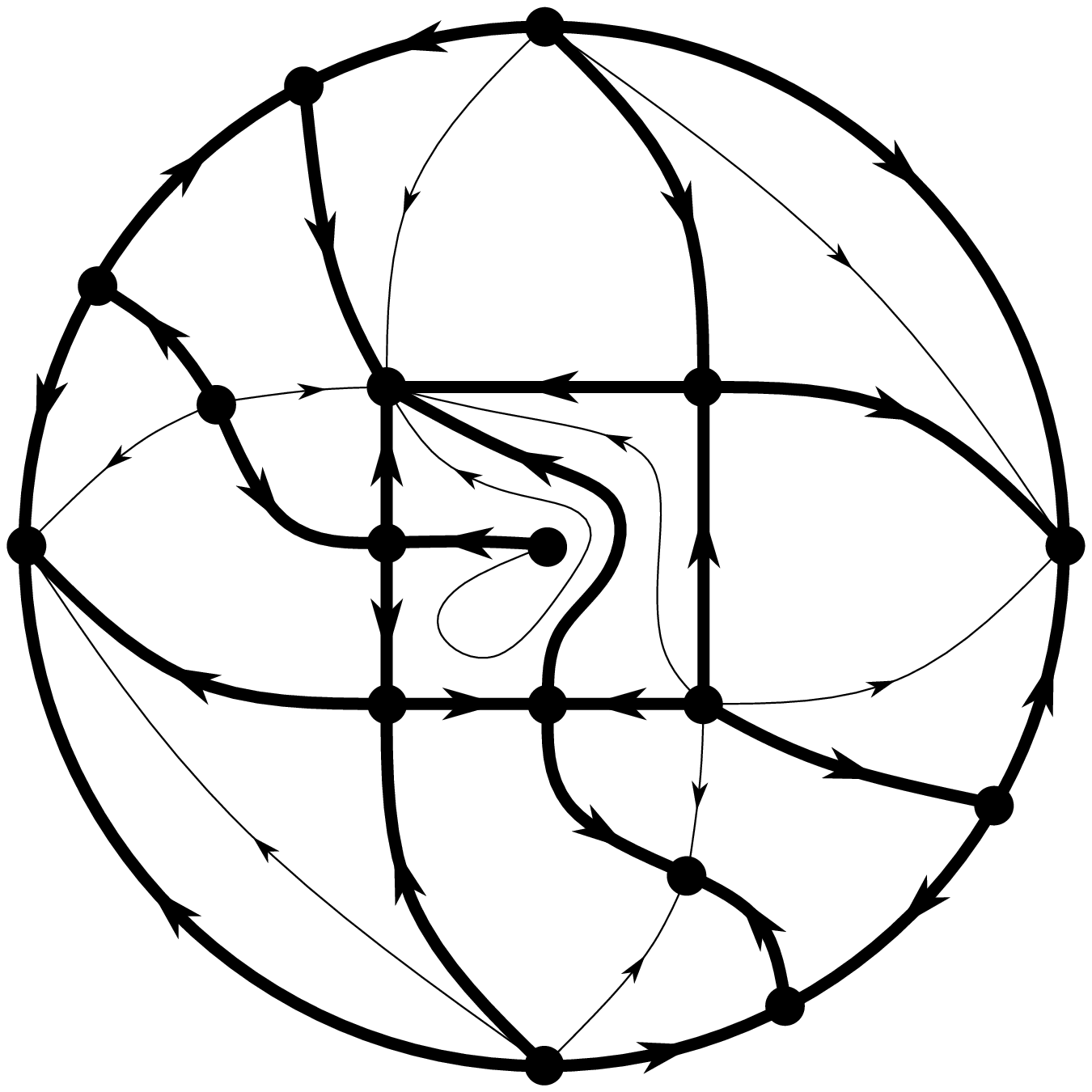} 
				\end{overpic}
				
				Case~$4.9a.i.1$.
			\end{center}
		\end{minipage}
		\begin{minipage}{3.1cm}
			\begin{center}
				\begin{overpic}[height=3cm]{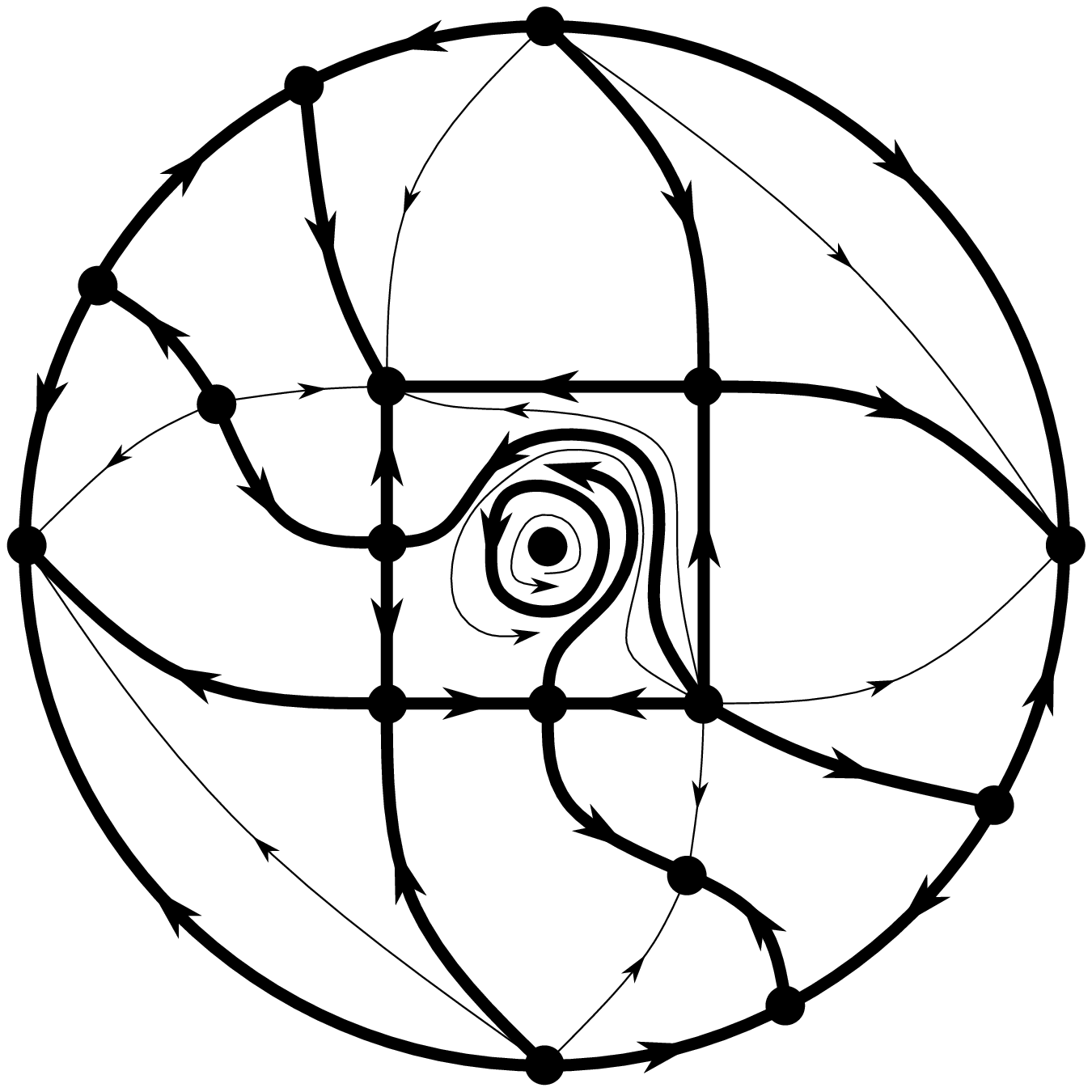} 
				\end{overpic}
				
				Case~$4.9a.i.2$.
			\end{center}
		\end{minipage}
		\begin{minipage}{3.1cm}
			\begin{center}
				\begin{overpic}[height=3cm]{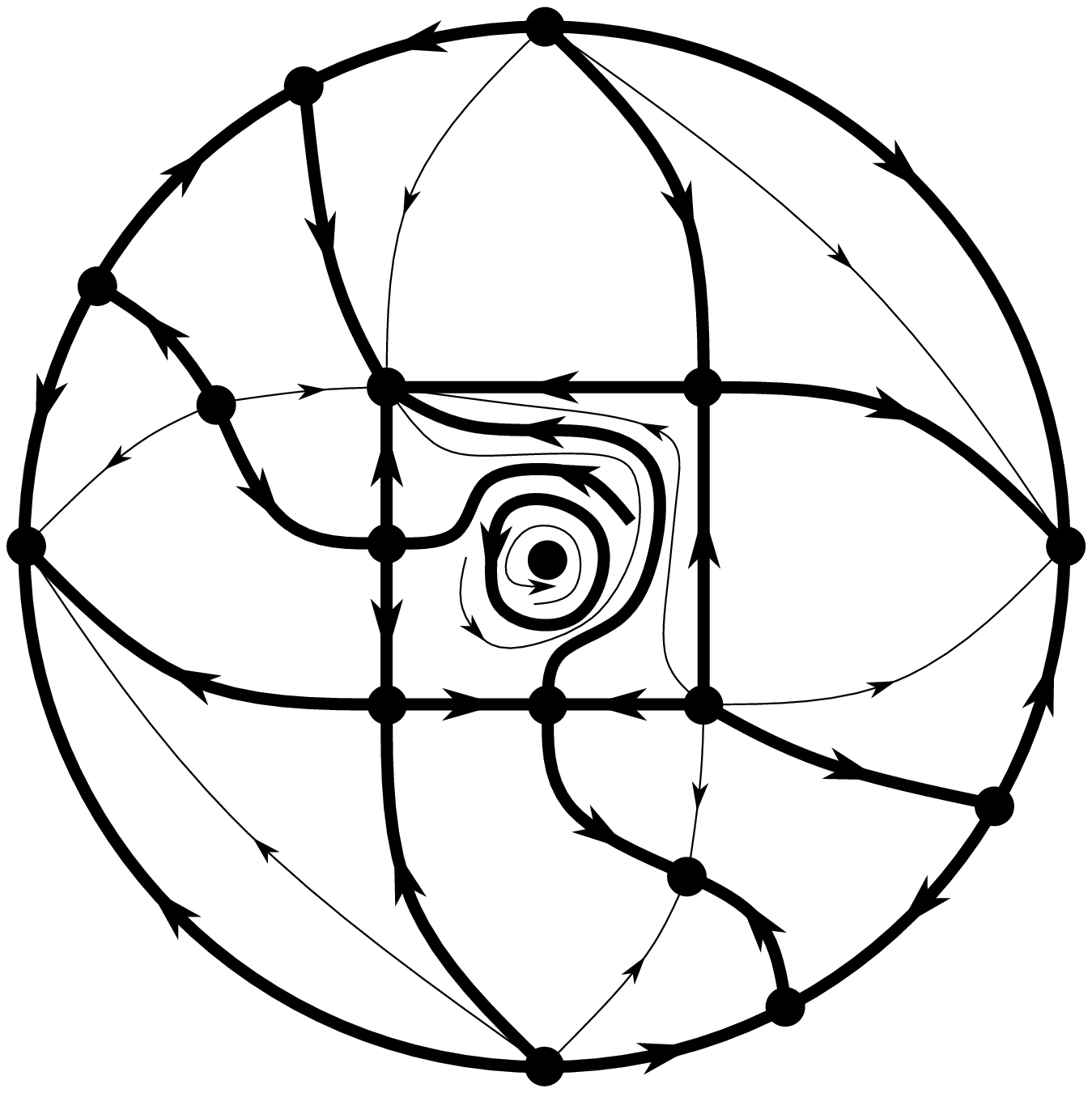} 
				\end{overpic}
				
				Case~$4.9a.ii.1$.
			\end{center}
		\end{minipage}	
		\begin{minipage}{3.1cm}
			\begin{center}
				\begin{overpic}[height=3cm]{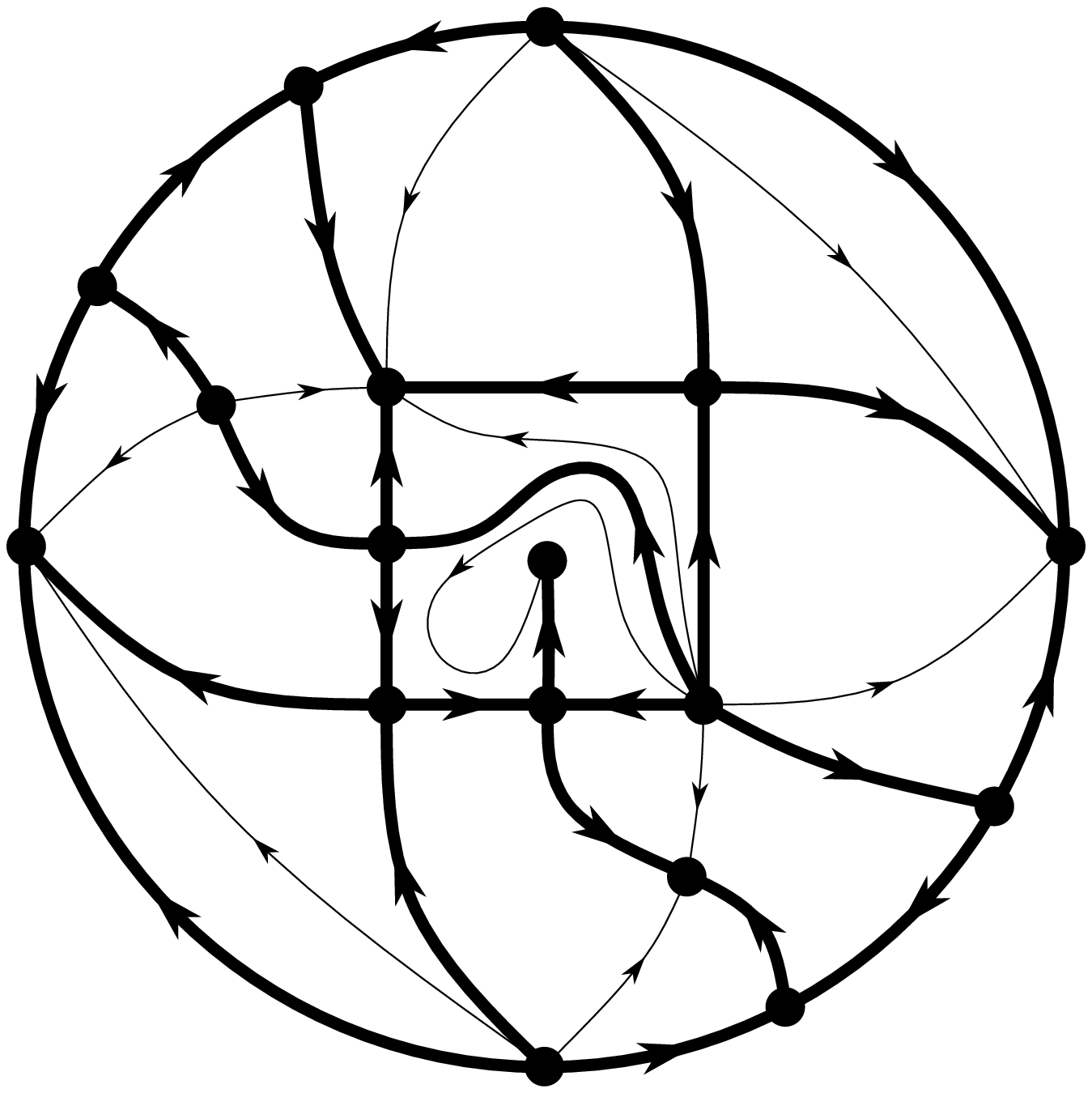} 
				\end{overpic}
				
				Case~$4.9a.ii.2$.
			\end{center}
		\end{minipage}
	\end{center}
	$\;$
	\begin{center}
		\begin{minipage}{3.1cm}
			\begin{center}
				\begin{overpic}[height=3cm]{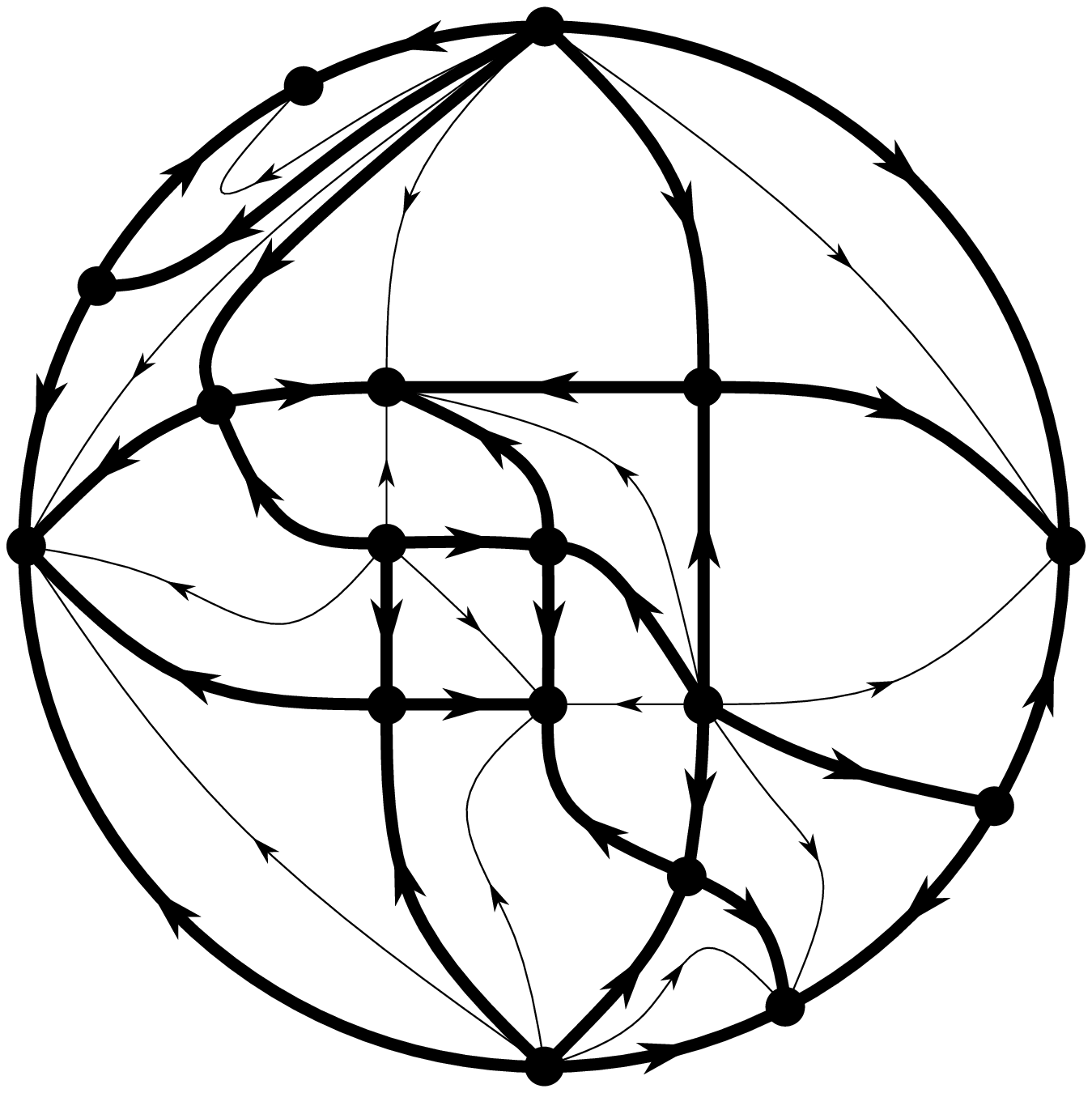} 
				\end{overpic}
				
				Case~$4.9b.i.1$.
			\end{center}
		\end{minipage}
		\begin{minipage}{3.1cm}
			\begin{center}
				\begin{overpic}[height=3cm]{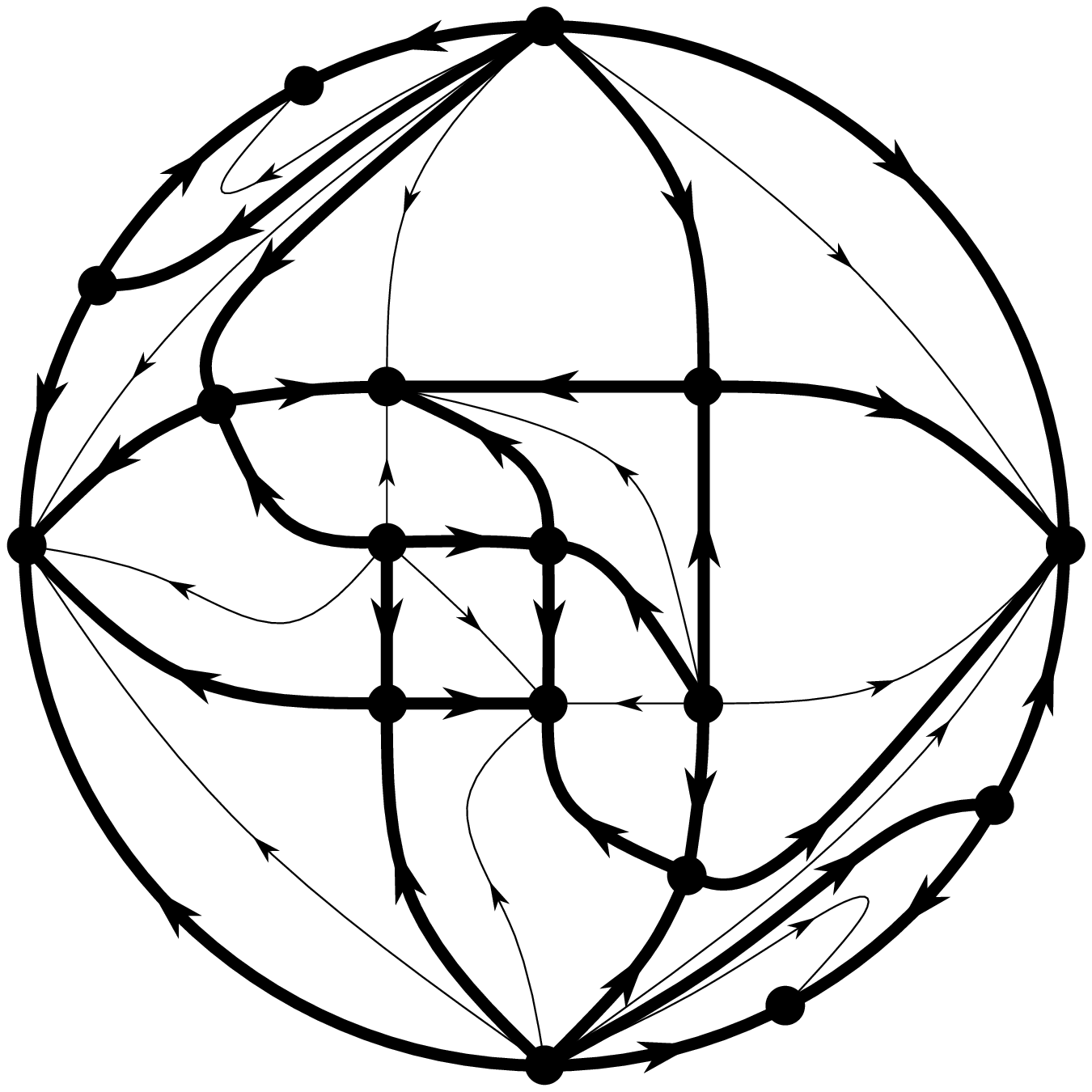} 
				\end{overpic}
				
				Case~$4.9b.i.2$.
			\end{center}
		\end{minipage}
		\begin{minipage}{3.1cm}
			\begin{center}
				\begin{overpic}[height=3cm]{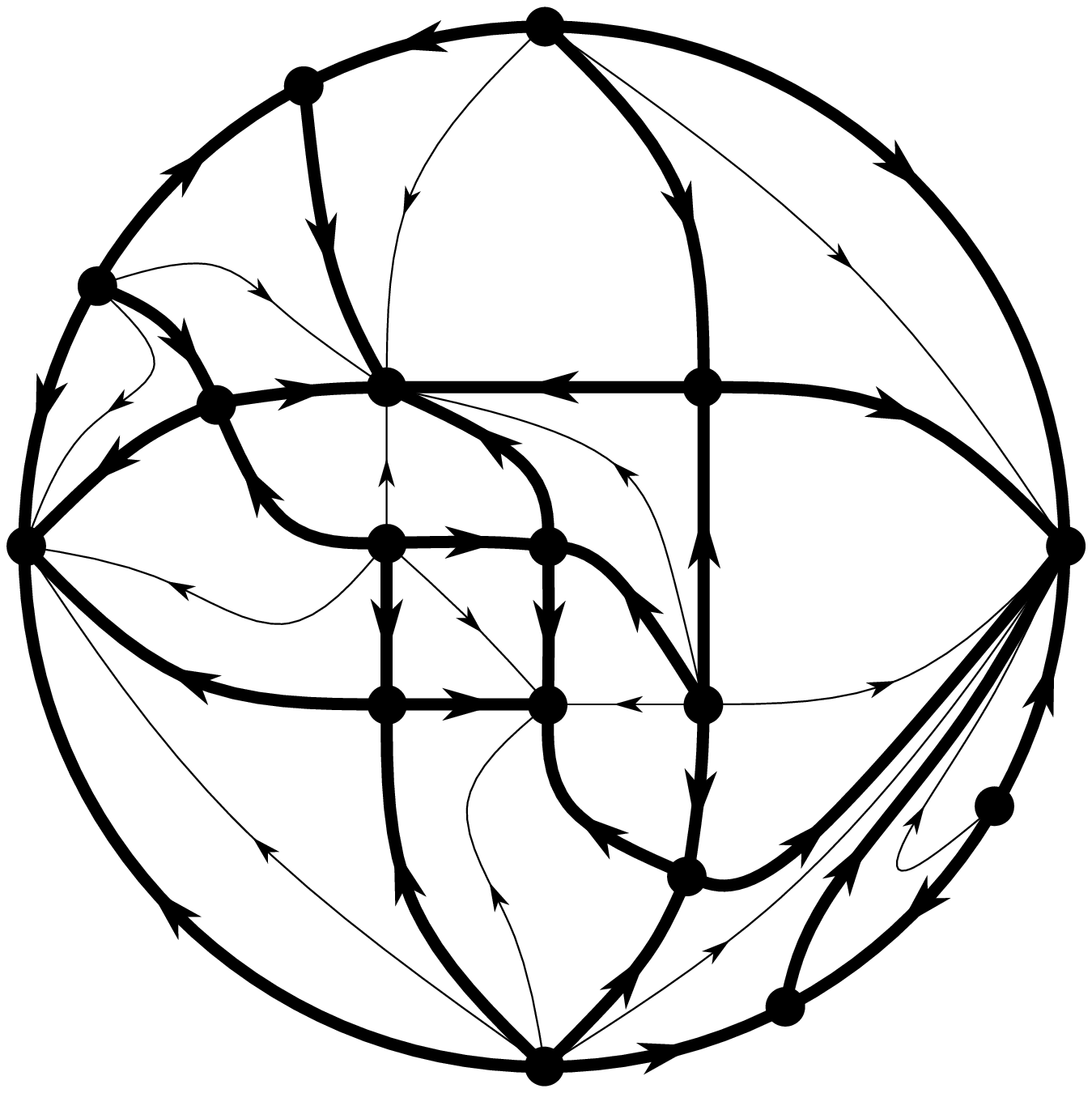} 
				\end{overpic}
				
				Case~$4.9b.ii.1$.
			\end{center}
		\end{minipage}
		\begin{minipage}{3.1cm}
			\begin{center}
				\begin{overpic}[height=3cm]{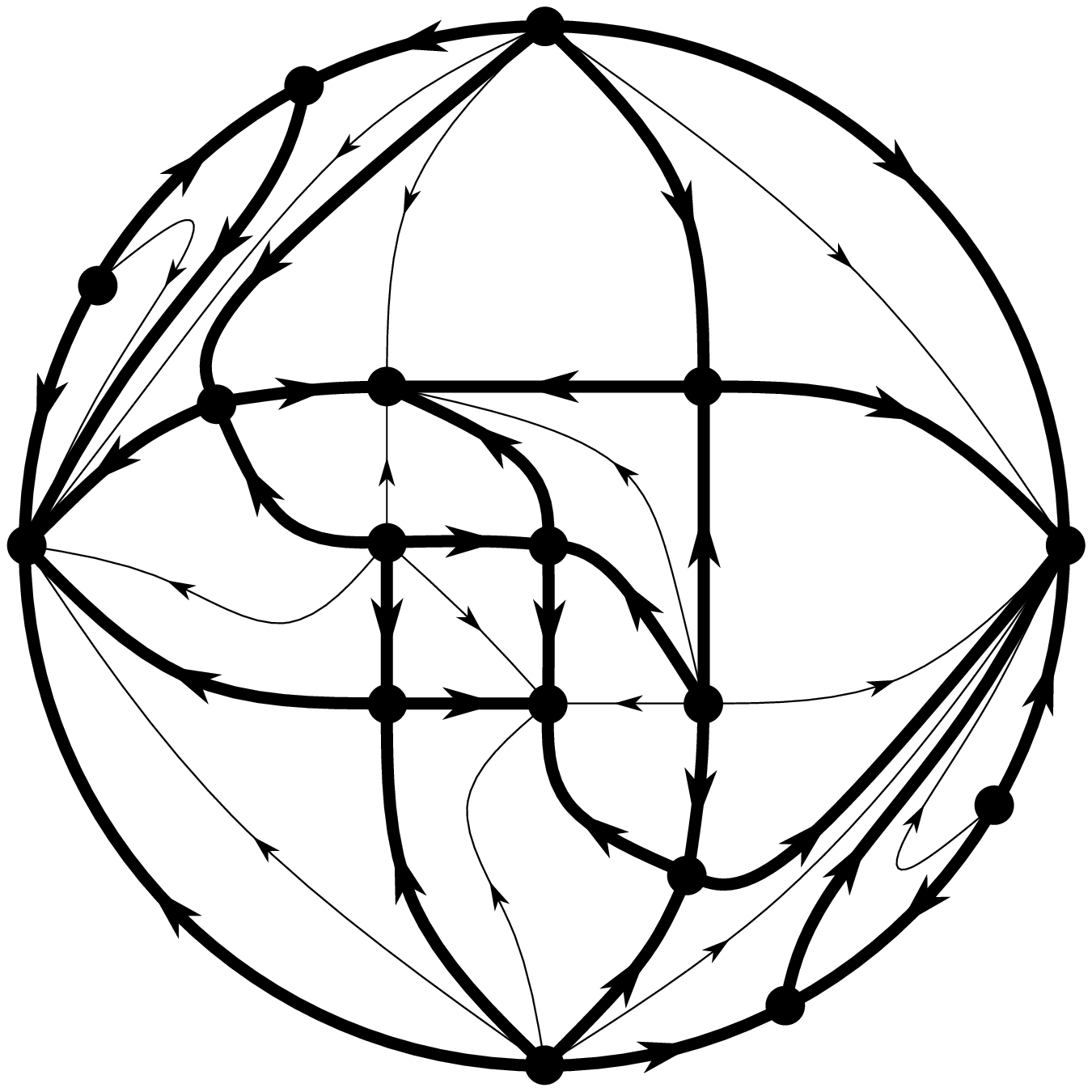} 
				\end{overpic}
				
				Case~$4.9b.ii.2$.
			\end{center}
		\end{minipage}
		\begin{minipage}{3.1cm}
			\begin{center}
				\begin{overpic}[height=3cm]{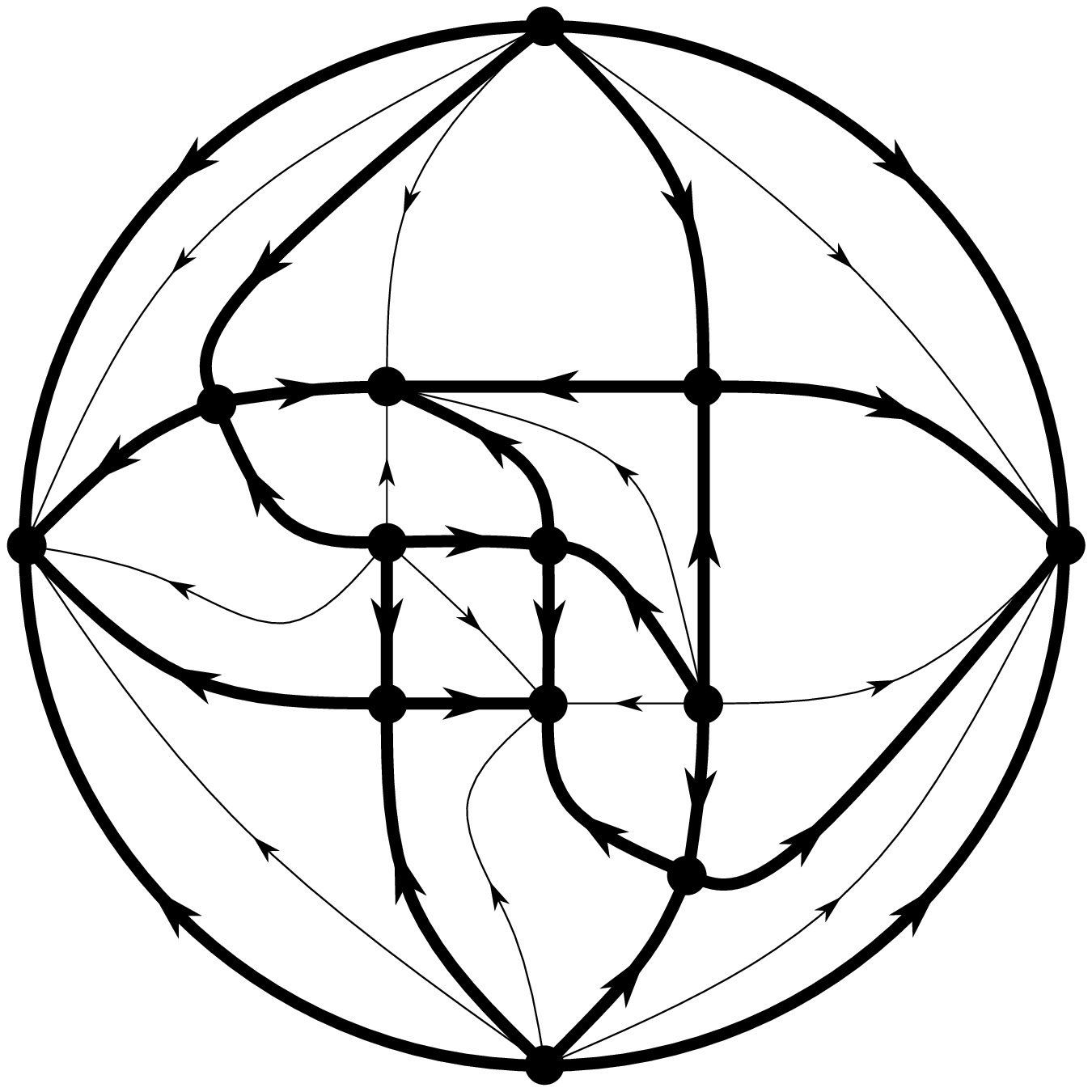} 
				\end{overpic}
				
				Case~$4.9b.iii$.
			\end{center}
		\end{minipage}
	\end{center}
	$\;$
	\begin{center}
		\begin{minipage}{3.1cm}
			\begin{center}
				\begin{overpic}[height=3cm]{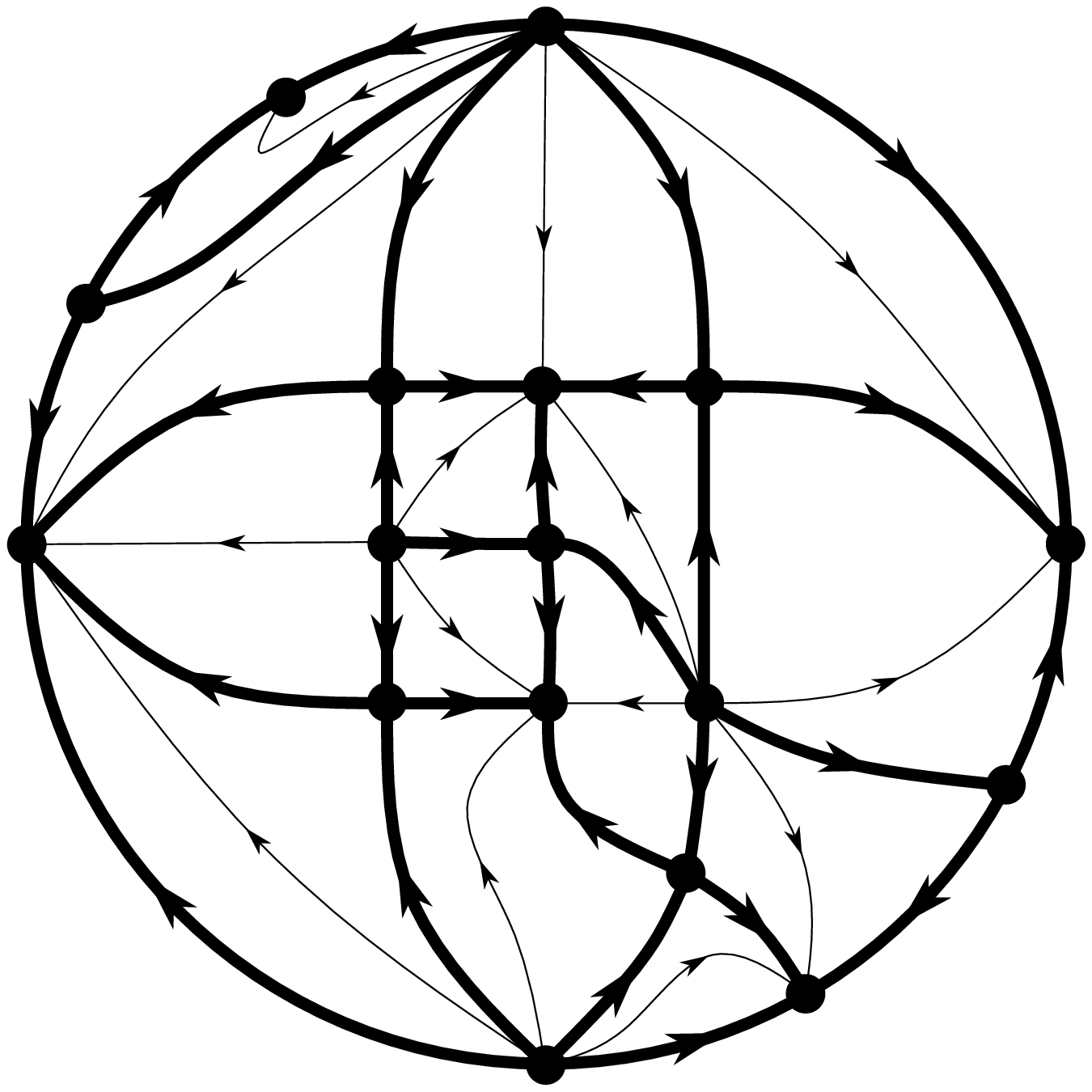} 
				\end{overpic}
				
				Case~$4.10a.1$.
			\end{center}
		\end{minipage}
		\begin{minipage}{3.1cm}
			\begin{center}
				\begin{overpic}[height=3cm]{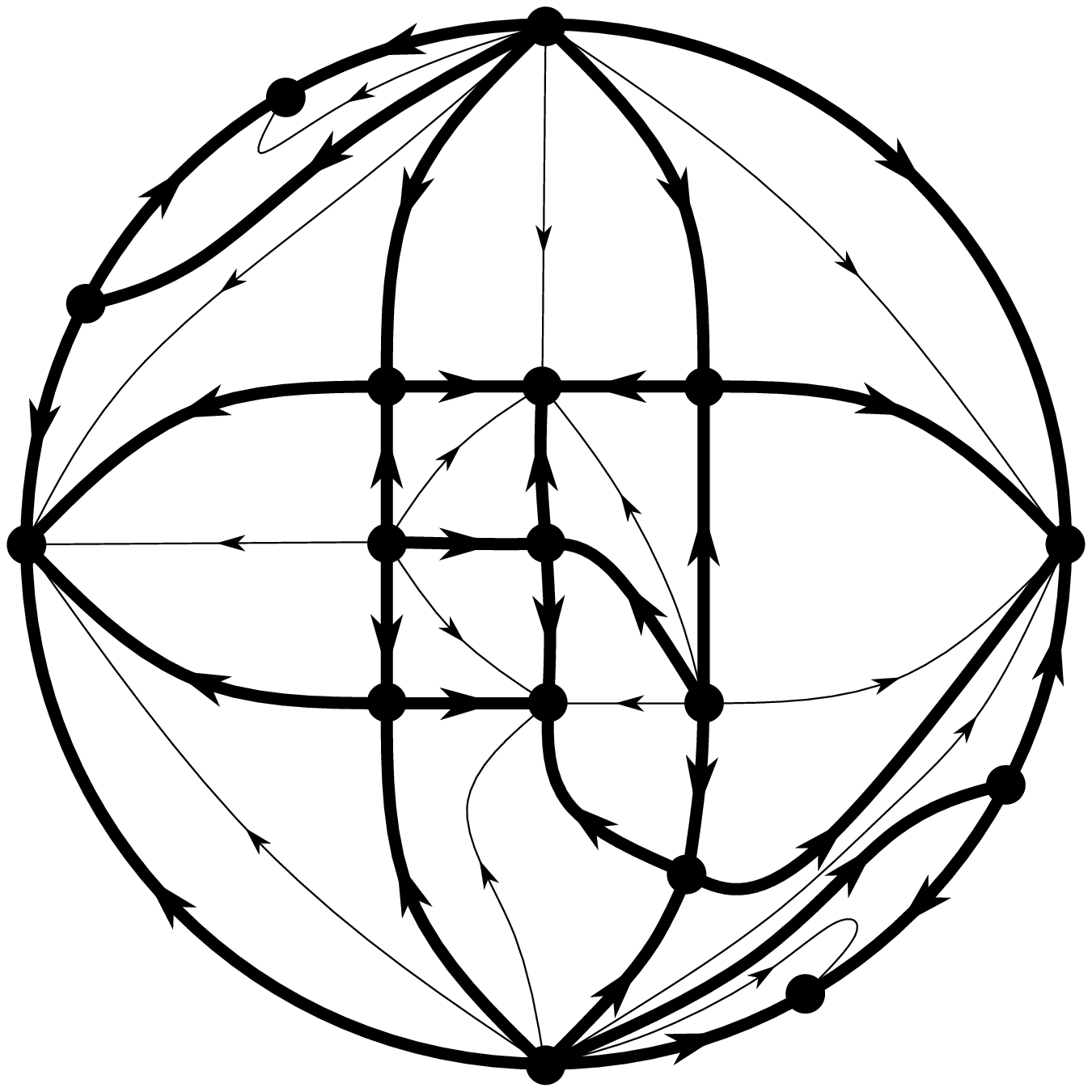} 
				\end{overpic}
				
				Case~$4.10a.2$.
			\end{center}
		\end{minipage}
		\begin{minipage}{3.1cm}
			\begin{center}
				\begin{overpic}[height=3cm]{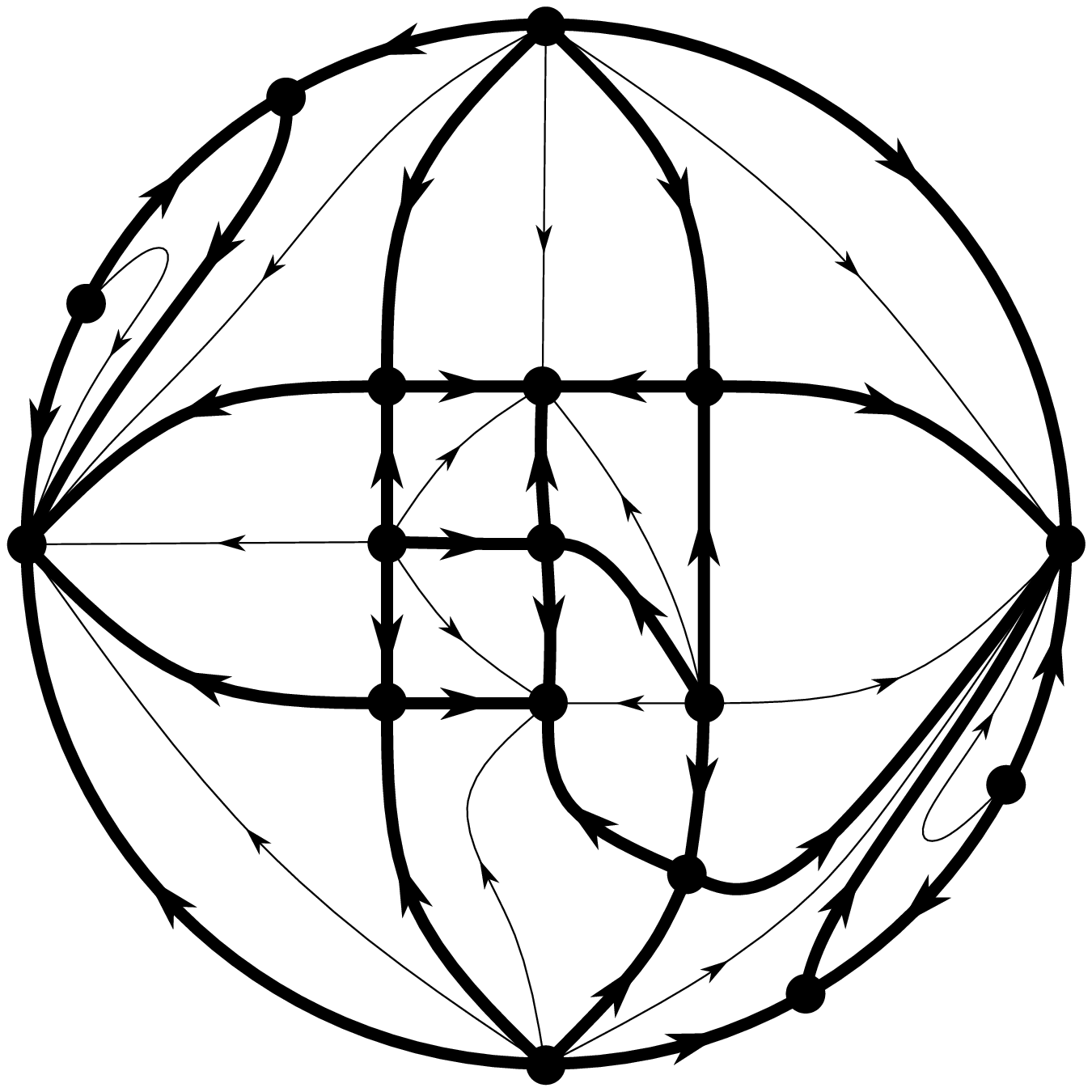} 
				\end{overpic}
				
				Case~$4.10b$.
			\end{center}
		\end{minipage}	
		\begin{minipage}{3.1cm}
			\begin{center}
				\begin{overpic}[height=3cm]{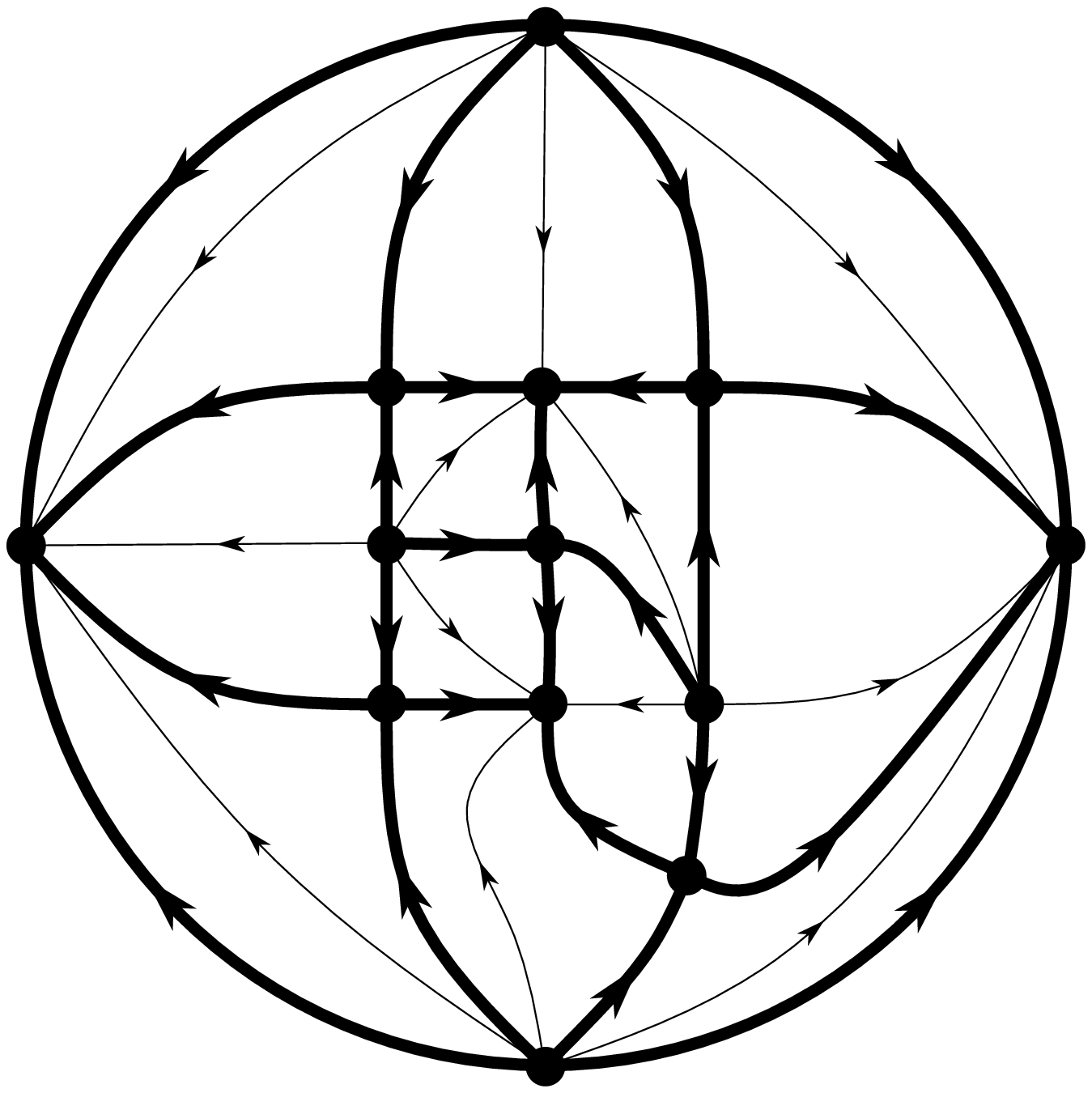} 
				\end{overpic}
				
				Case~$4.10c$.
			\end{center}
		\end{minipage}
		\begin{minipage}{3.1cm}
			\begin{center}
				\begin{overpic}[height=3cm]{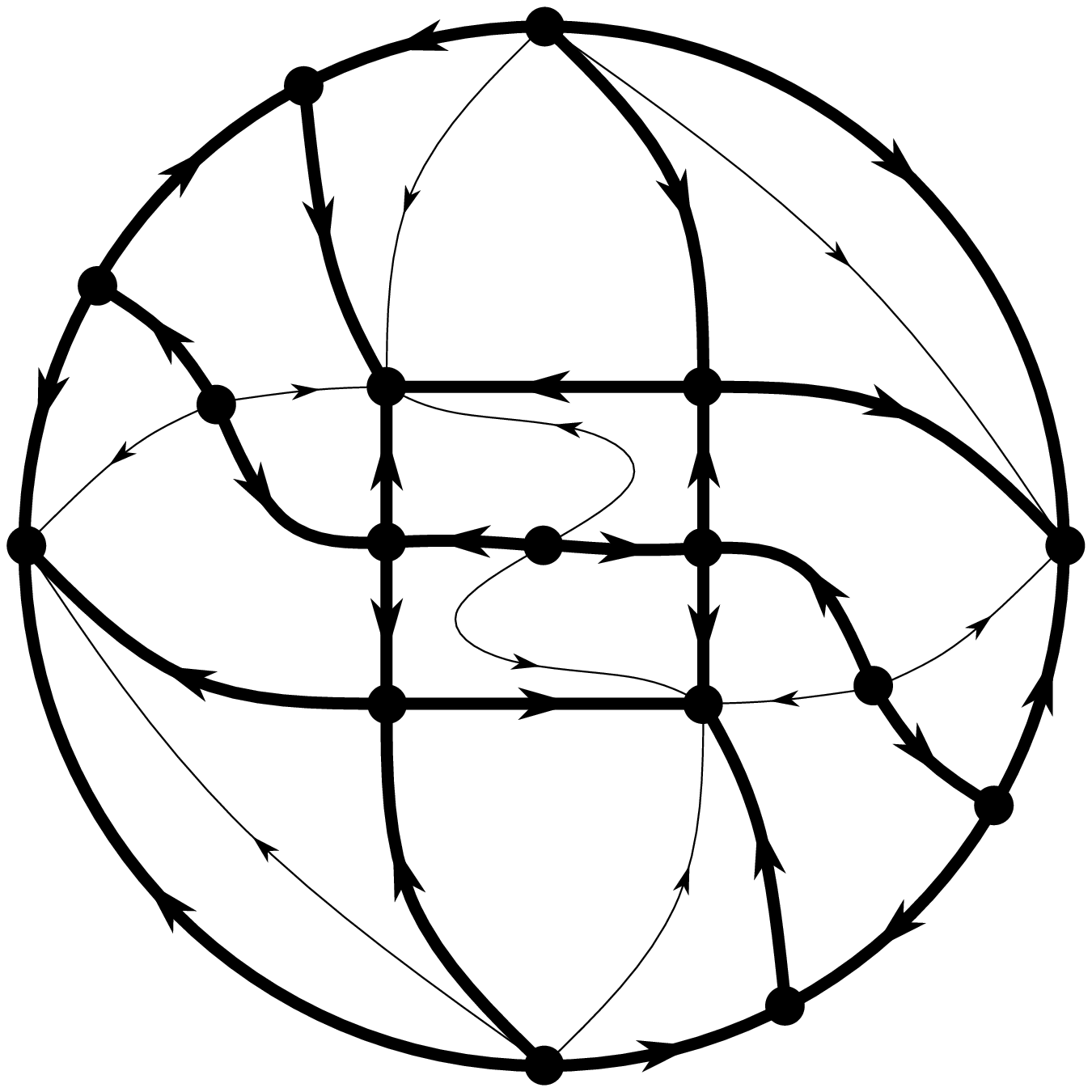} 
				\end{overpic}
				
				Case~$4.11a.i$.
			\end{center}
		\end{minipage}
	\end{center}
	$\;$
	\begin{center}
		\begin{minipage}{3.1cm}
			\begin{center}
				\begin{overpic}[height=3cm]{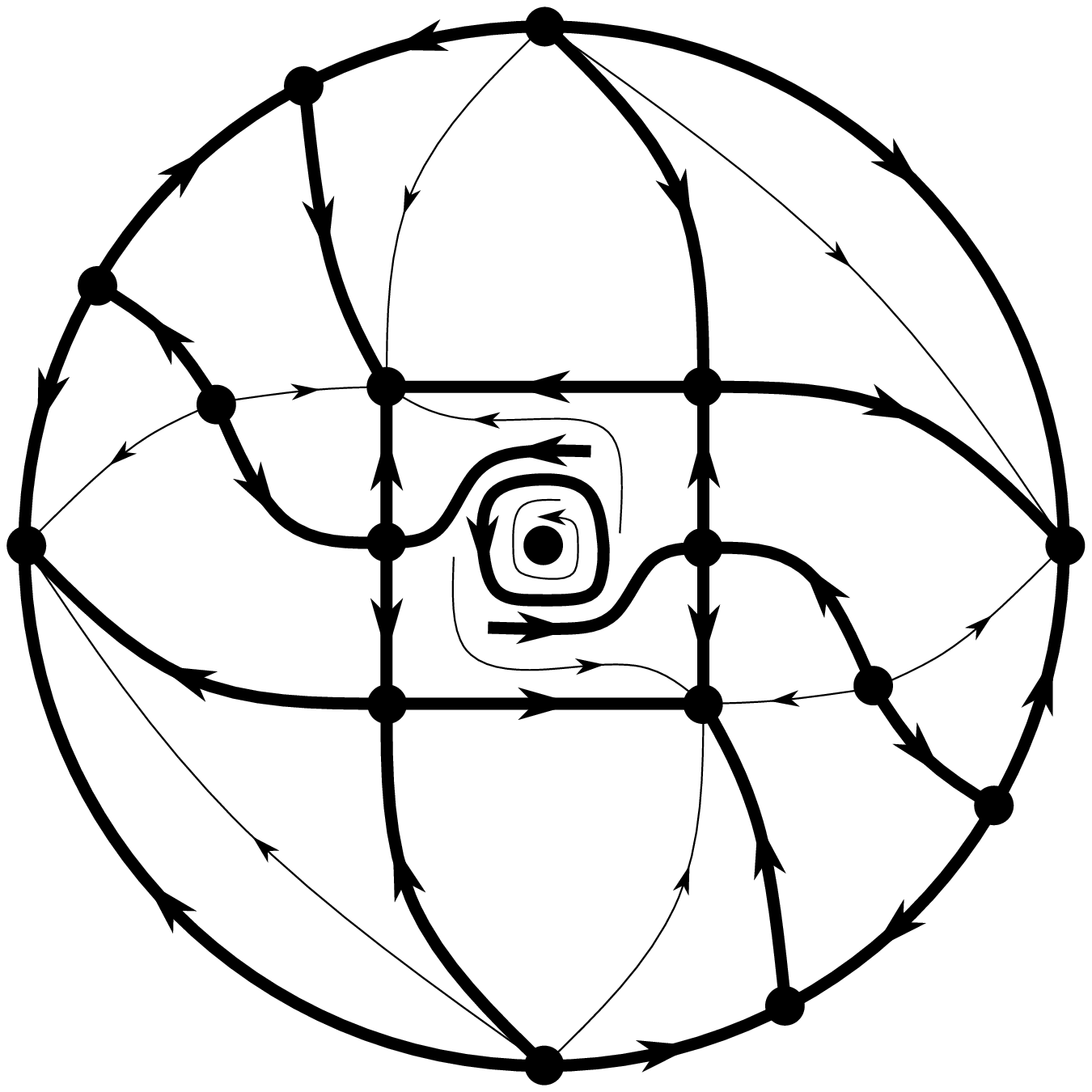} 
				\end{overpic}
				
				Case~$4.11a.ii$.
			\end{center}
		\end{minipage}
		\begin{minipage}{3.1cm}
			\begin{center}
				\begin{overpic}[height=3cm]{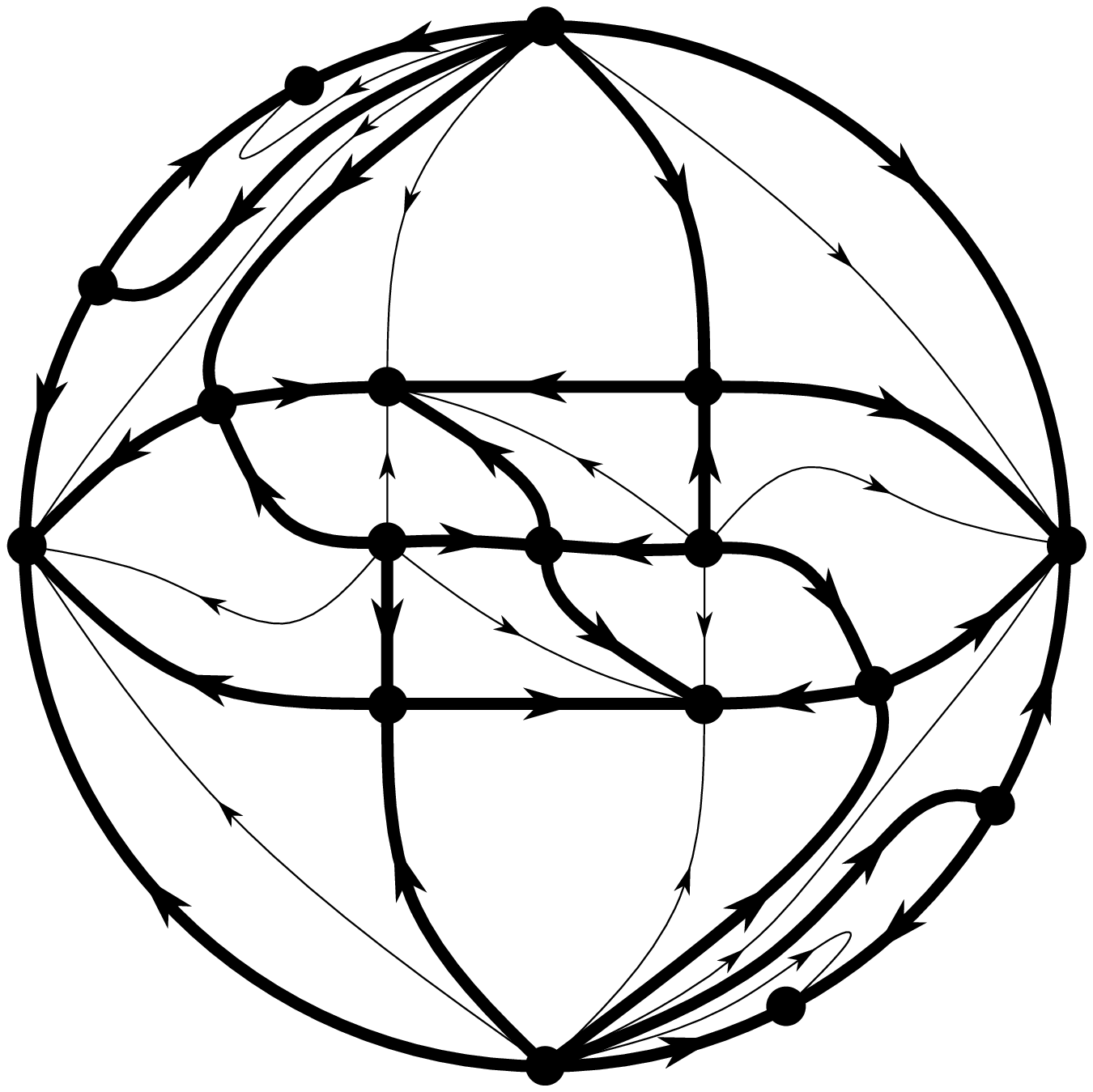} 
				\end{overpic}
				
				Case~$4.11b.i$.
			\end{center}
		\end{minipage}	
		\begin{minipage}{3.1cm}
			\begin{center}
				\begin{overpic}[height=3cm]{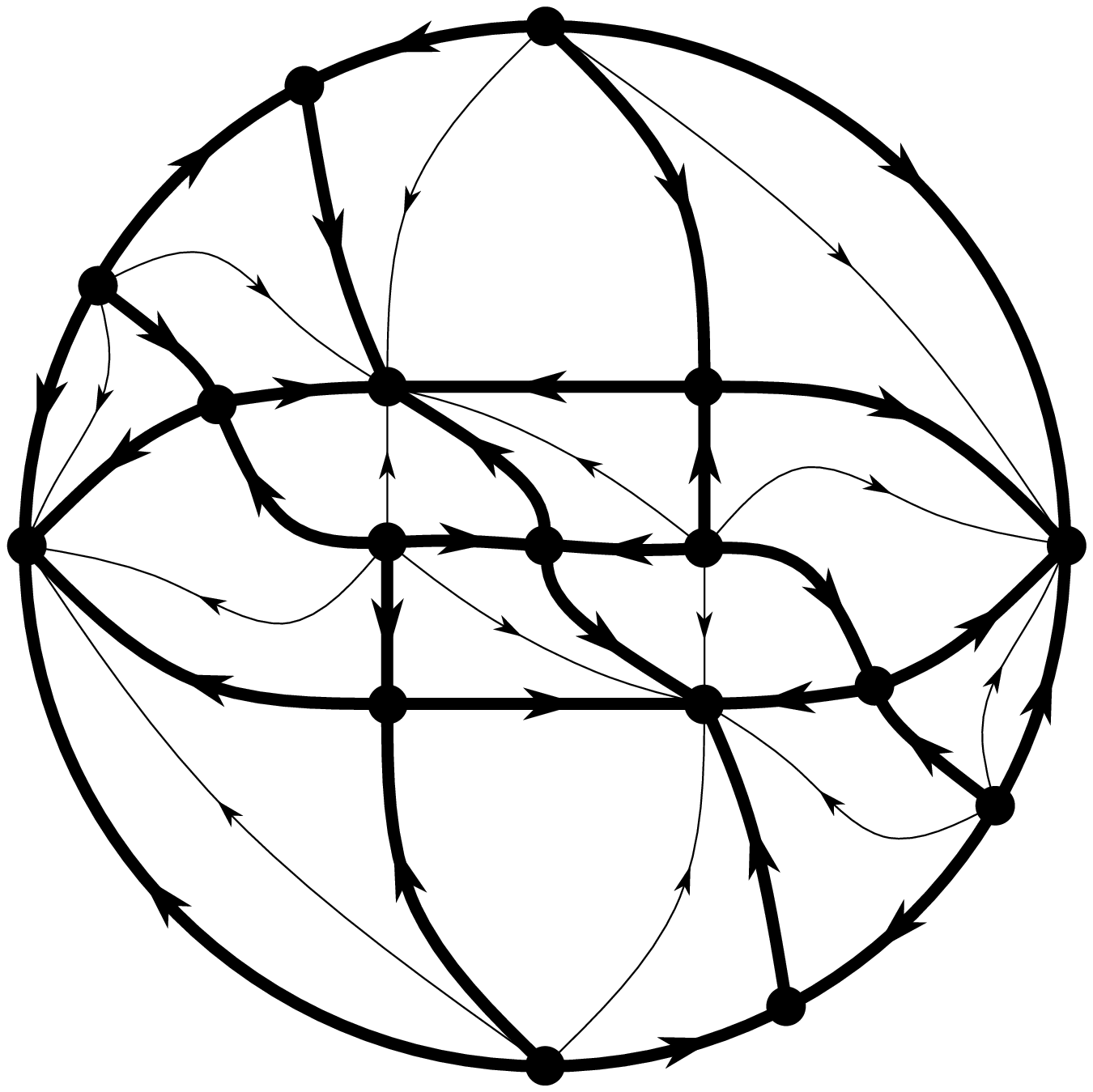} 
				\end{overpic}
				
				Case~$4.11b.ii.1$.
			\end{center}
		\end{minipage}
		\begin{minipage}{3.1cm}
			\begin{center}
				\begin{overpic}[height=3cm]{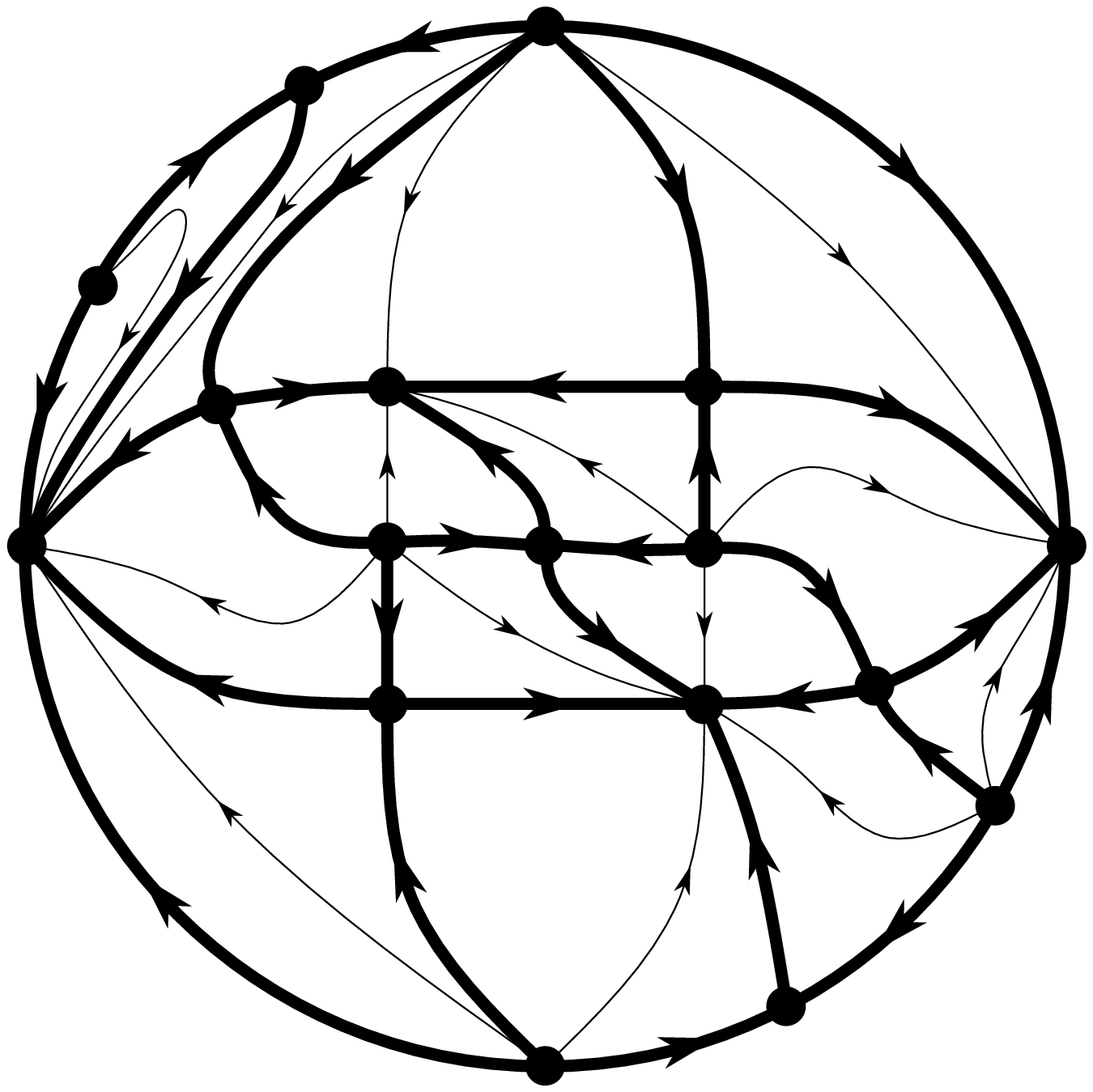} 
				\end{overpic}
				
				Case~$4.11b.ii.2$.
			\end{center}
		\end{minipage}
		\begin{minipage}{3.1cm}
			\begin{center}
				\begin{overpic}[height=3cm]{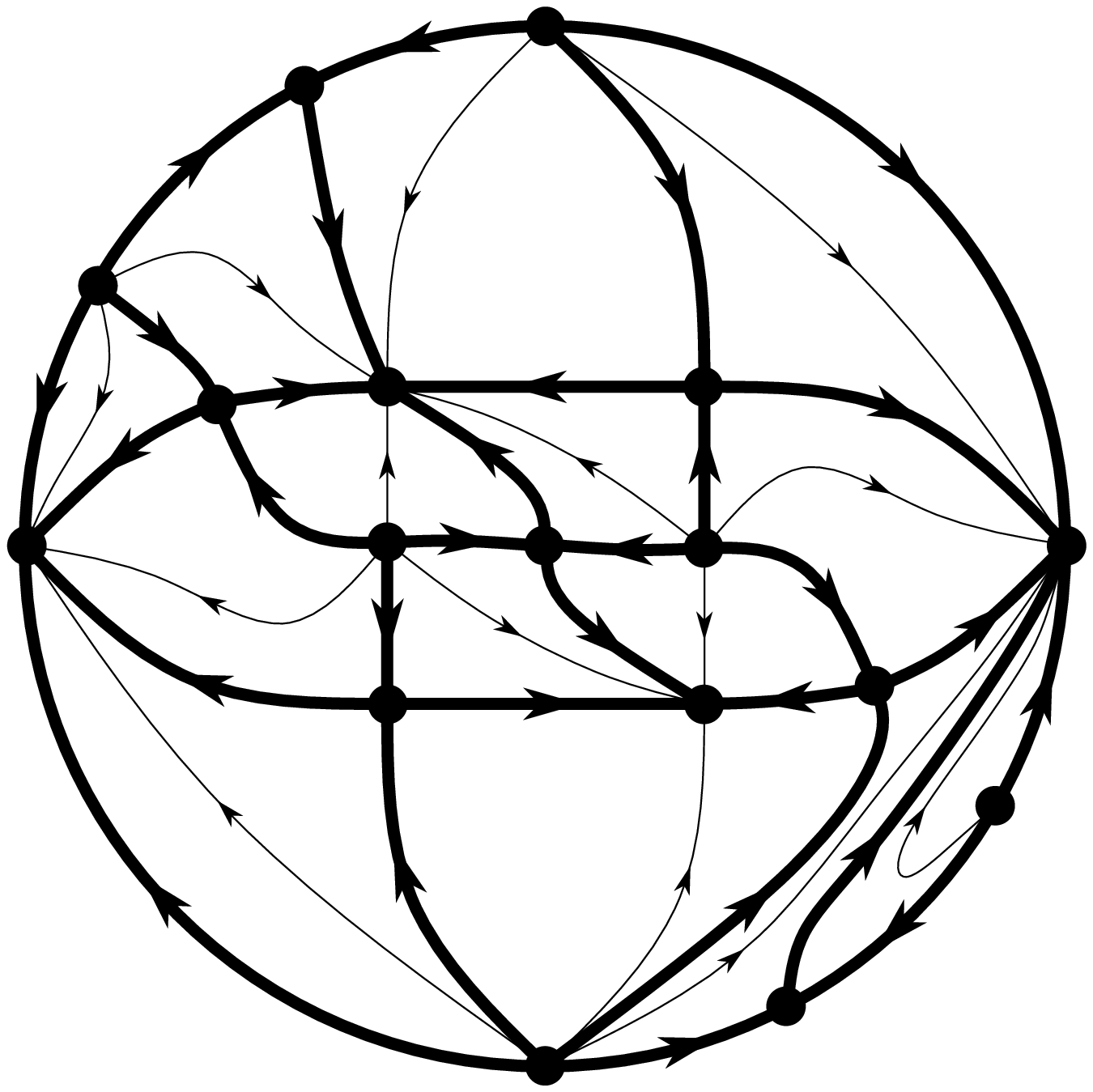} 
				\end{overpic}
				
				Case~$4.11b.ii.3$.
			\end{center}
		\end{minipage}
	\end{center}
	$\;$
	\begin{center}
		\begin{minipage}{3.1cm}
			\begin{center}
				\begin{overpic}[height=3cm]{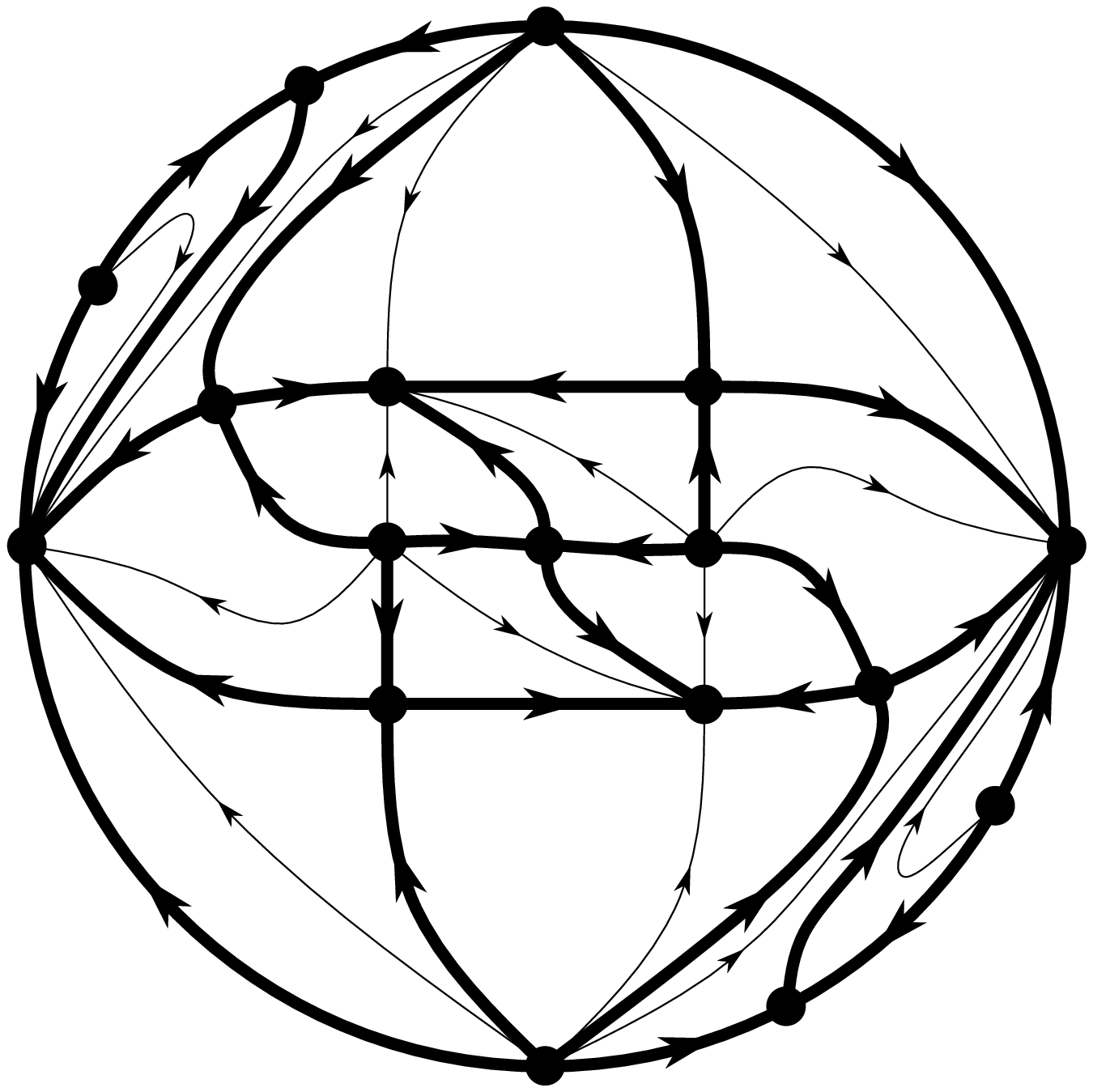} 
				\end{overpic}
				
				Case~$4.11b.ii.4$.
			\end{center}
		\end{minipage}
		\begin{minipage}{3.1cm}
			\begin{center}
				\begin{overpic}[height=3cm]{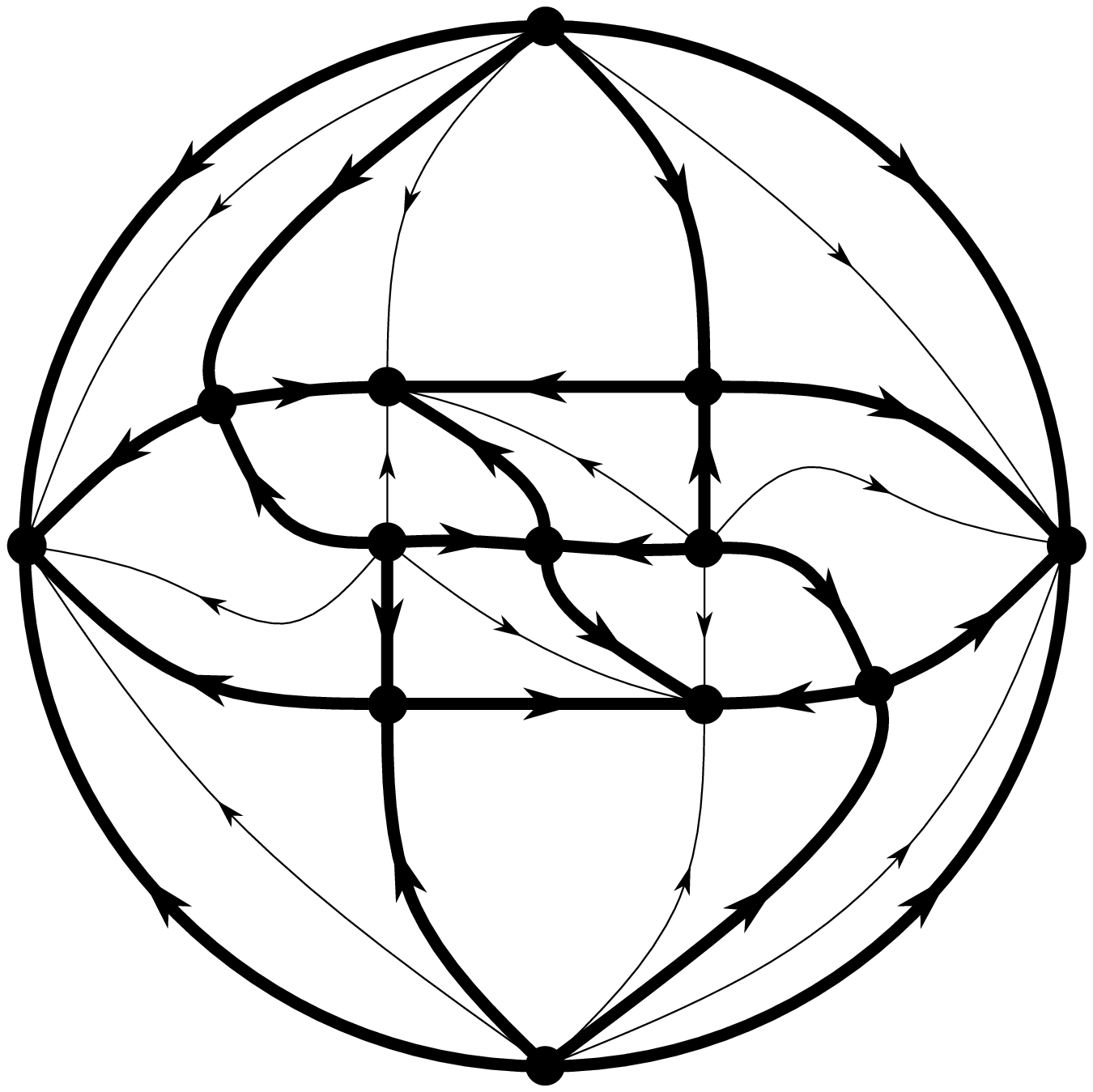} 
				\end{overpic}
				
				Case~$4.11b.iii$.
			\end{center}
		\end{minipage}
	\end{center}
	\caption{Phase portraits from Cases~$4.7$ to $4.11$.}\label{Case1.4b}
\end{figure}

\clearpage

\begin{figure}[h]
	\begin{center}
		\begin{minipage}{3.1cm}
			\begin{center}
				\begin{overpic}[height=3cm]{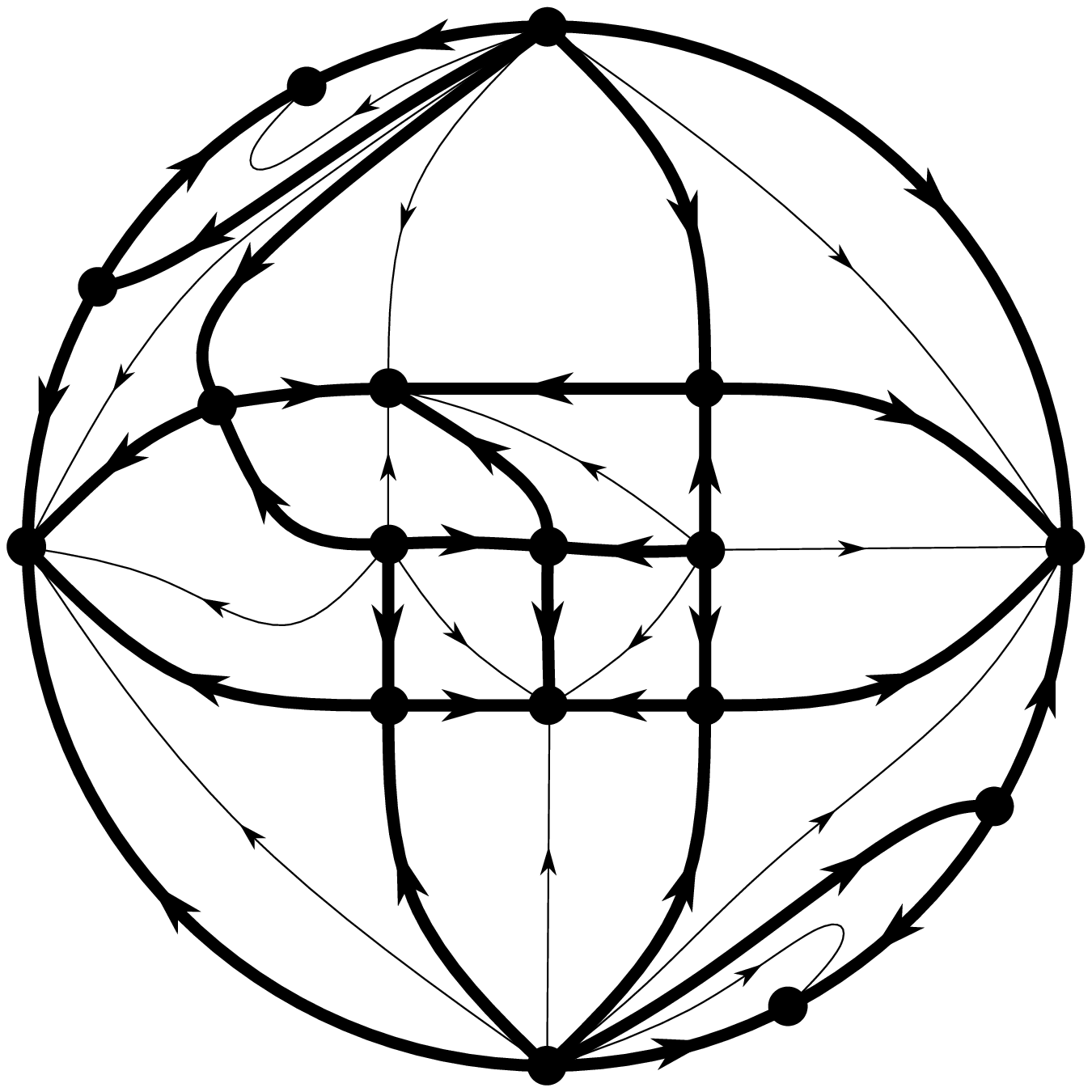} 
				\end{overpic}
				
				Case~$4.12a$.
			\end{center}
		\end{minipage}
		\begin{minipage}{3.1cm}
			\begin{center}
				\begin{overpic}[height=3cm]{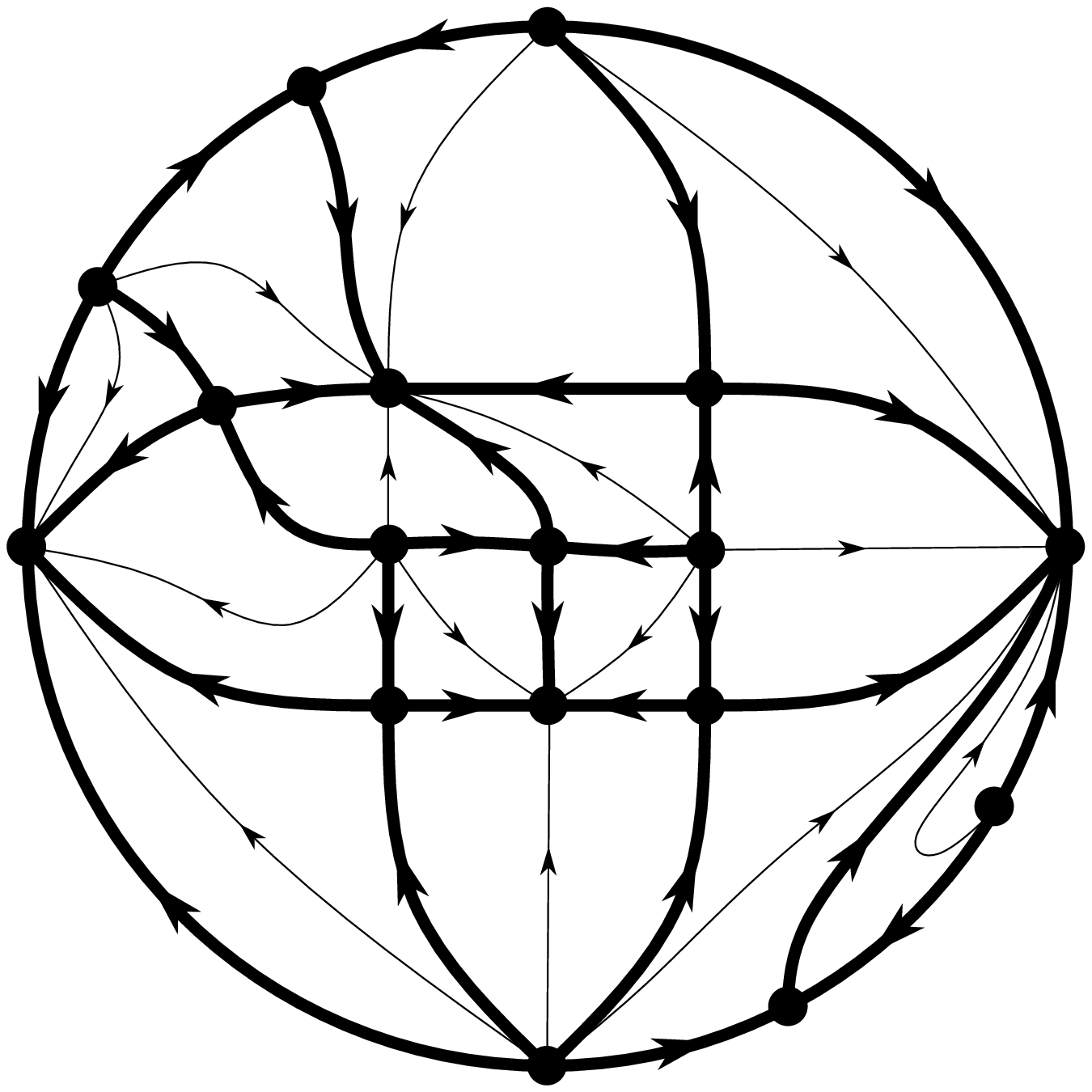} 
				\end{overpic}
				
				Case~$4.12b.1$.
			\end{center}
		\end{minipage}
		\begin{minipage}{3.1cm}
			\begin{center}
				\begin{overpic}[height=3cm]{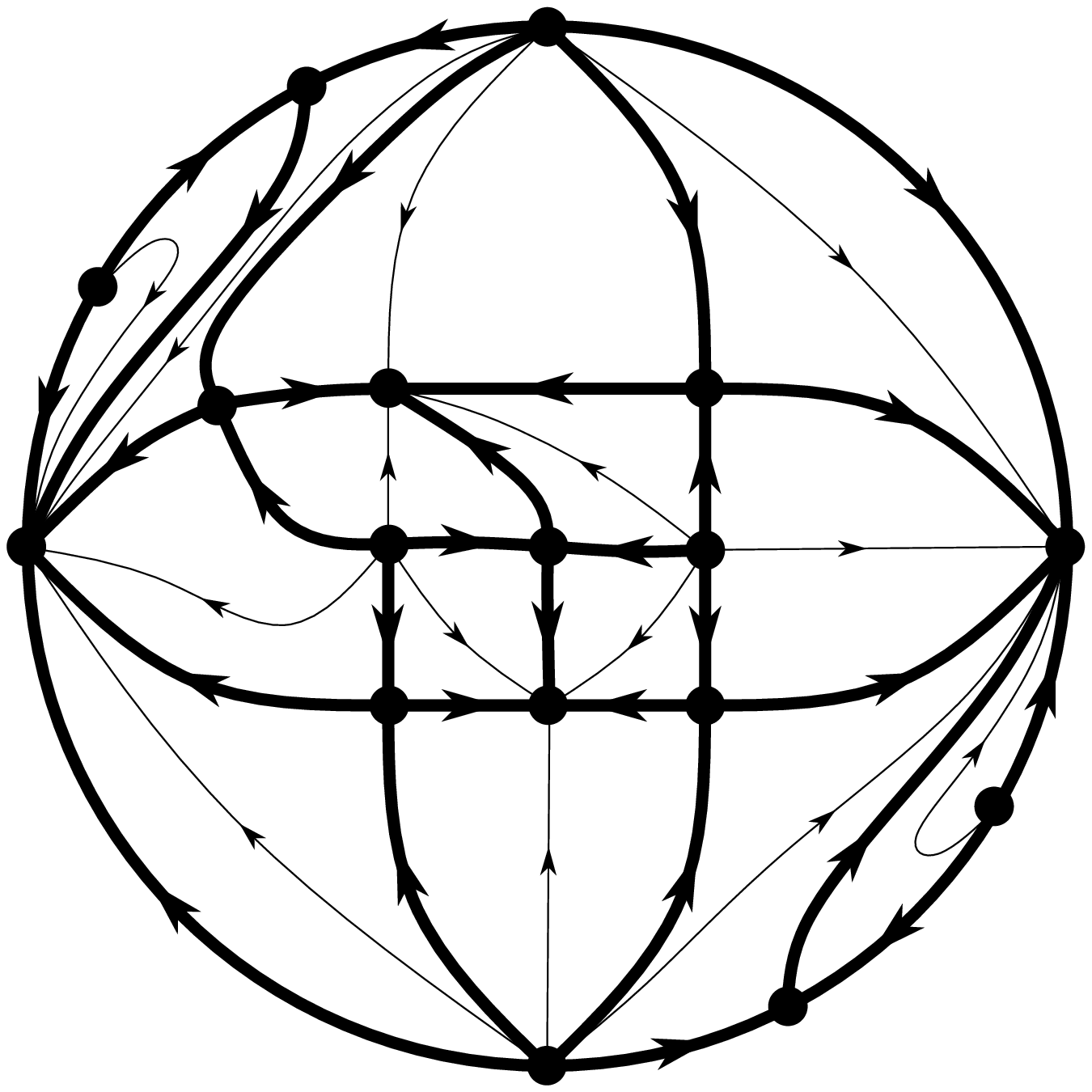} 
				\end{overpic}
				
				Case~$4.12b.2$.
			\end{center}
		\end{minipage}
		\begin{minipage}{3.1cm}
			\begin{center}
				\begin{overpic}[height=3cm]{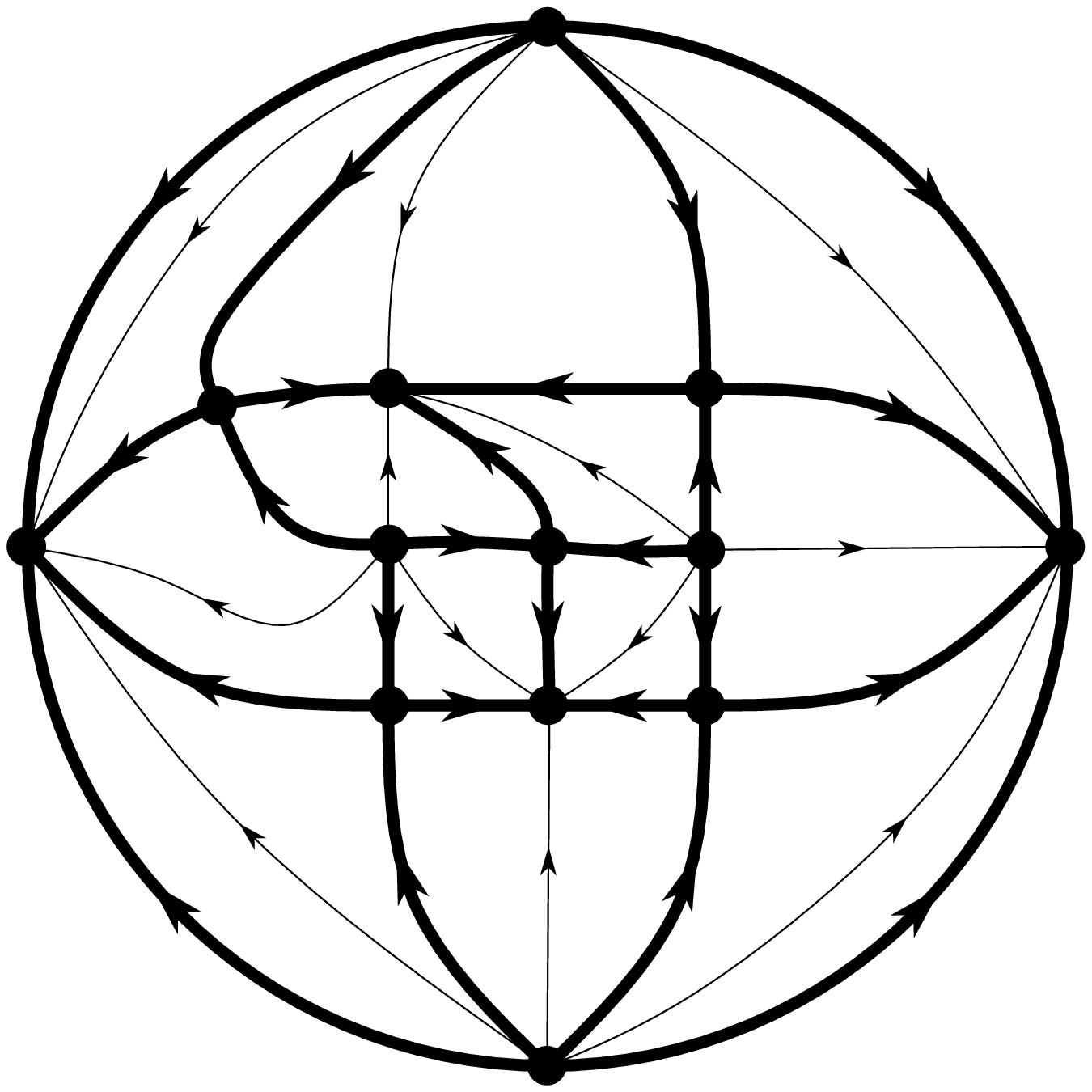} 
				\end{overpic}
				
				Case~$4.12c$.
			\end{center}
		\end{minipage}	
	\end{center}
	$\;$
	\begin{center} 
		\begin{minipage}{3.1cm}
			\begin{center}
				\begin{overpic}[height=3cm]{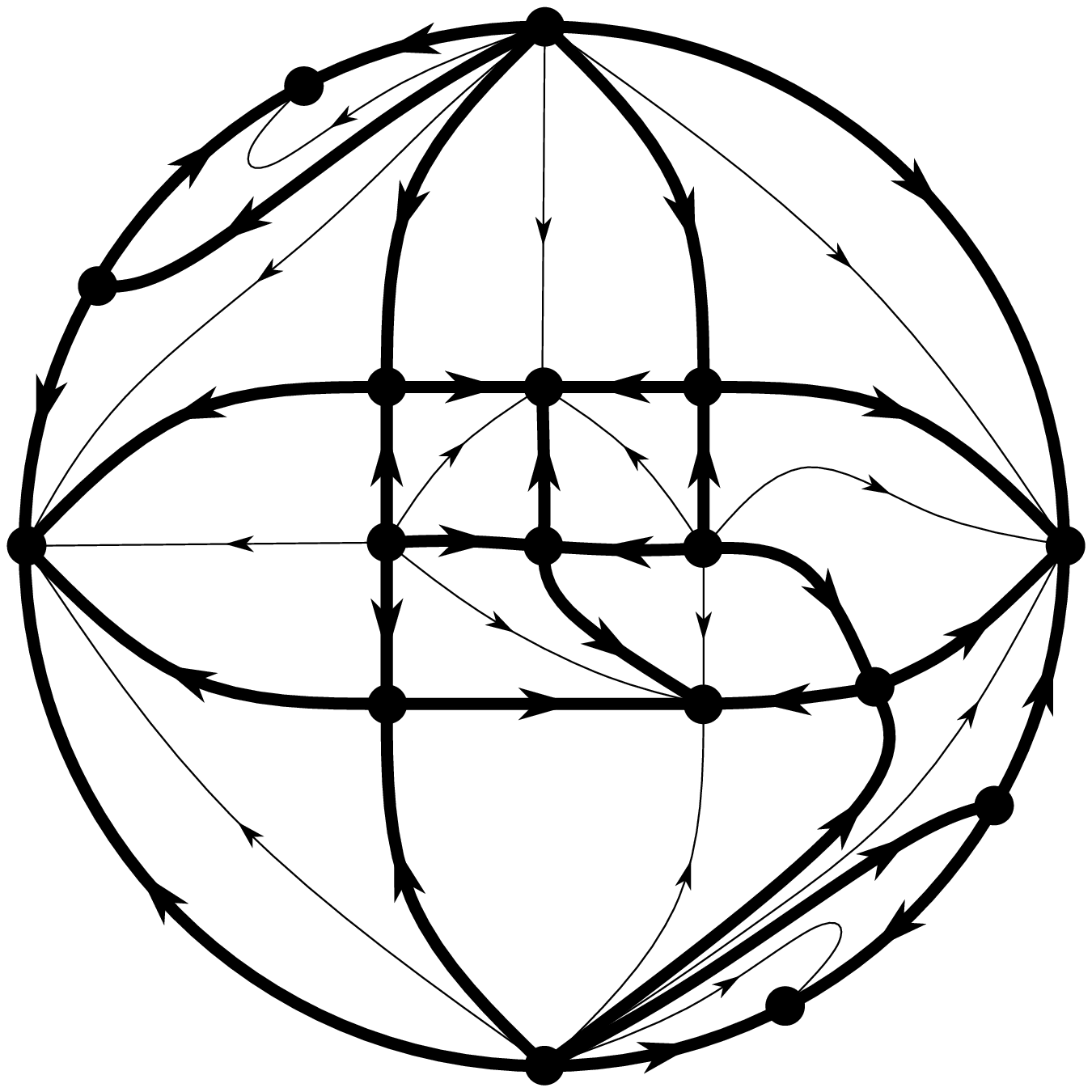} 
				\end{overpic}
				
				Case~$4.13a$.
			\end{center}
		\end{minipage}
		\begin{minipage}{3.1cm}
			\begin{center}
				\begin{overpic}[height=3cm]{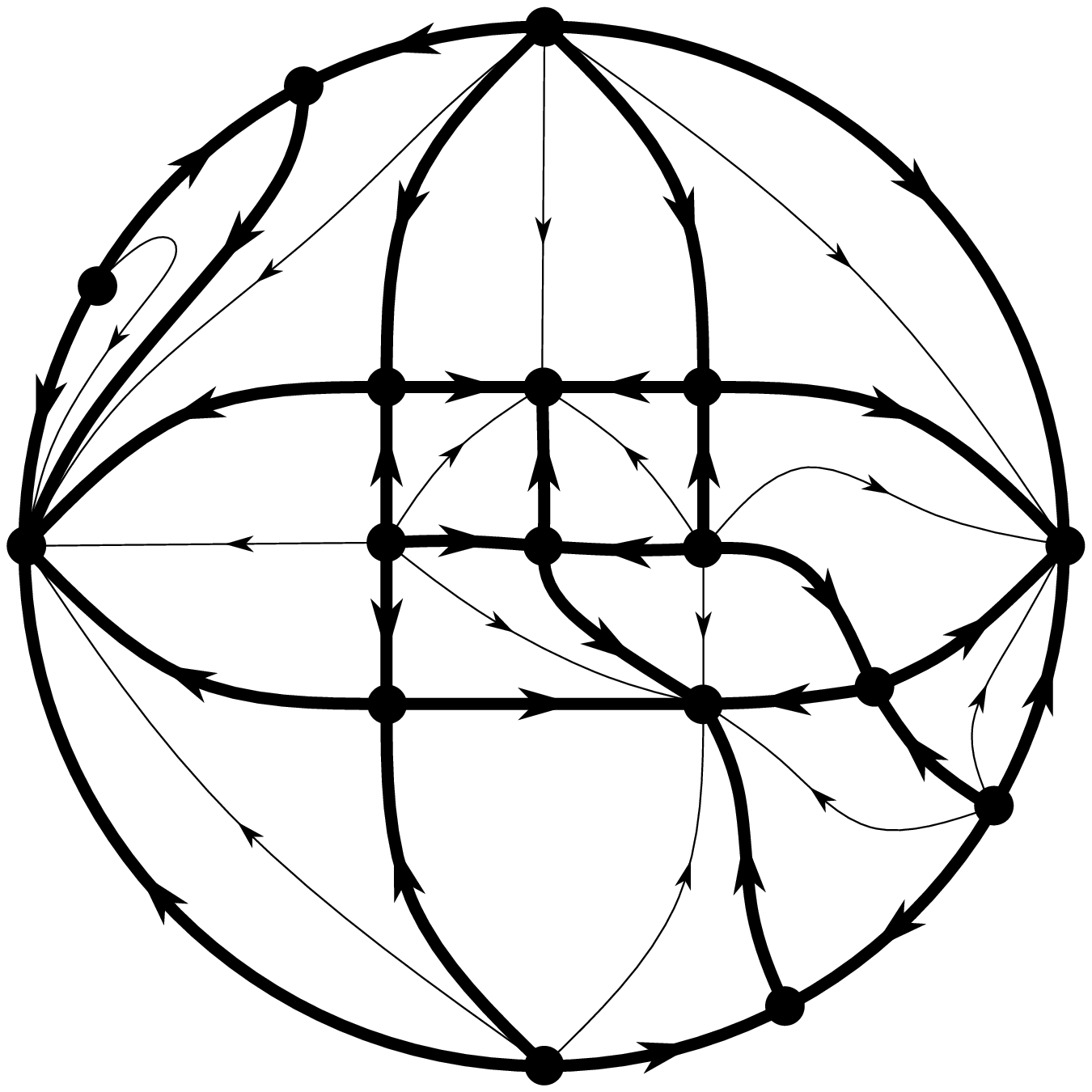} 
				\end{overpic}
				
				Case~$4.13b.1$.
			\end{center}
		\end{minipage}
		\begin{minipage}{3.1cm}
			\begin{center}
				\begin{overpic}[height=3cm]{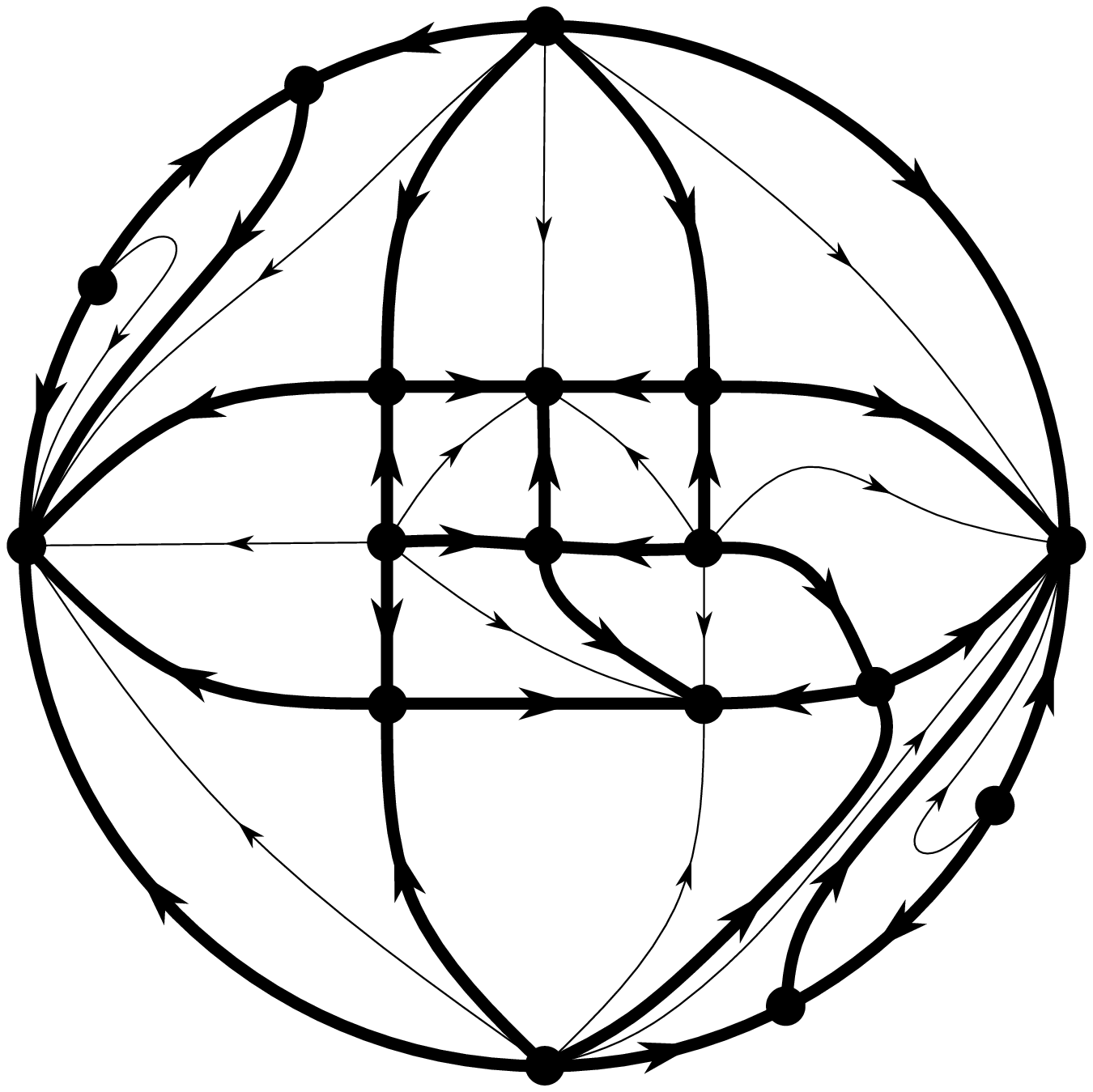} 
				\end{overpic}
				
				Case~$4.13b.2$.
			\end{center}
		\end{minipage}
		\begin{minipage}{3.1cm}
			\begin{center}
				\begin{overpic}[height=3cm]{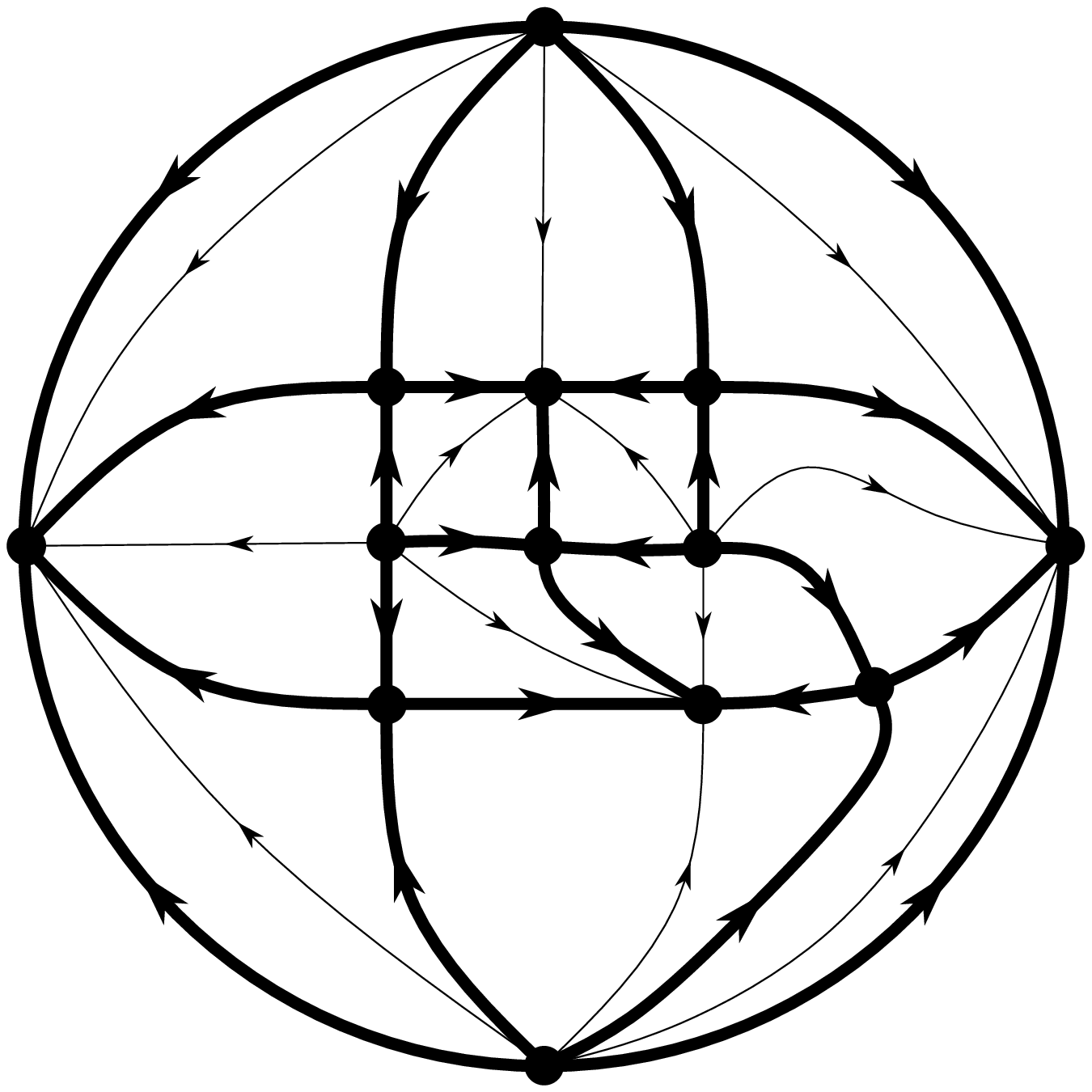} 
				\end{overpic}
				
				Case~$4.13c$.
			\end{center}
		\end{minipage}
	\end{center}
	$\;$
	\begin{center} 
		\begin{minipage}{3.1cm}
			\begin{center}
				\begin{overpic}[height=3cm]{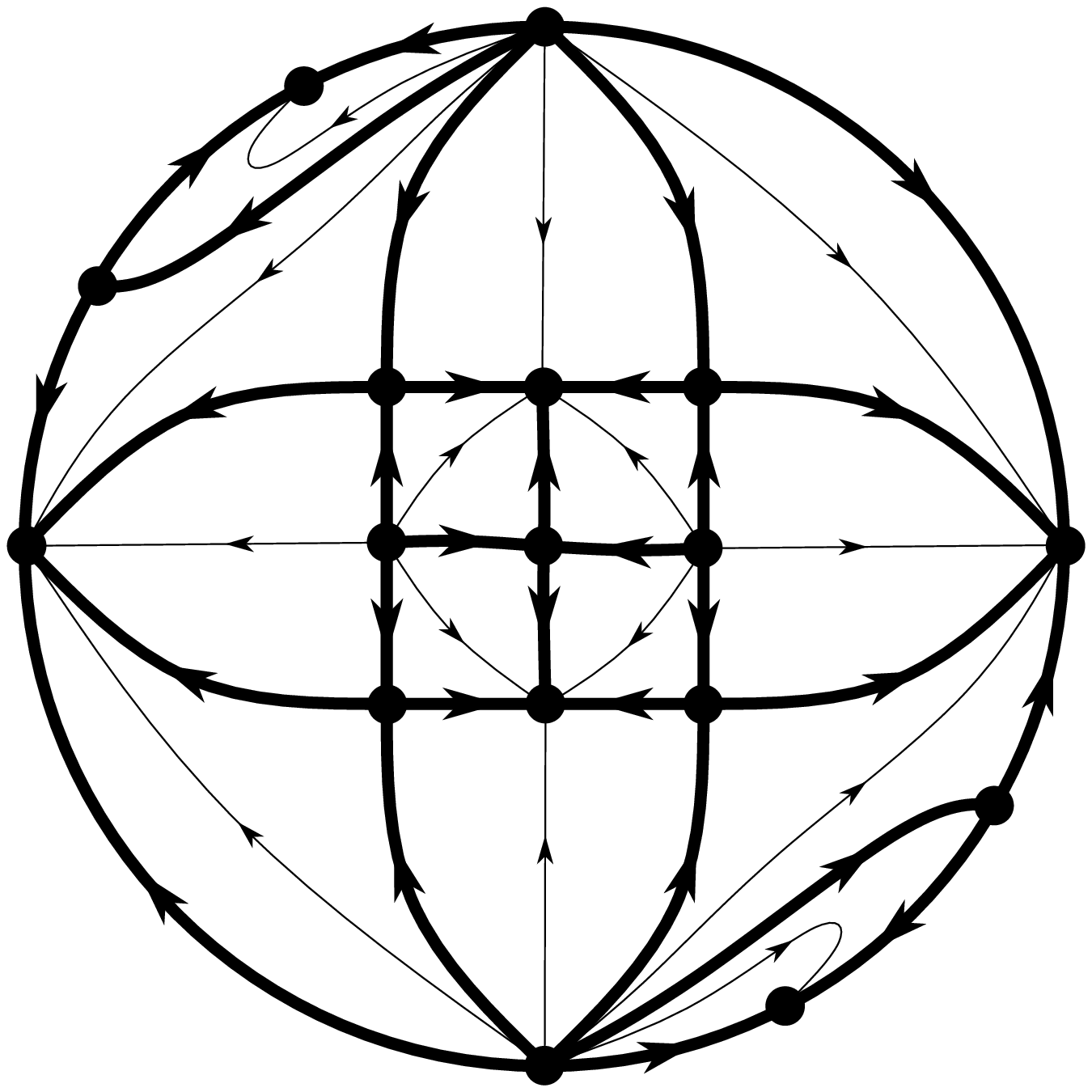} 
				\end{overpic}
				
				Case~$4.14a$.
			\end{center}
		\end{minipage}
		\begin{minipage}{3.1cm}
			\begin{center}
				\begin{overpic}[height=3cm]{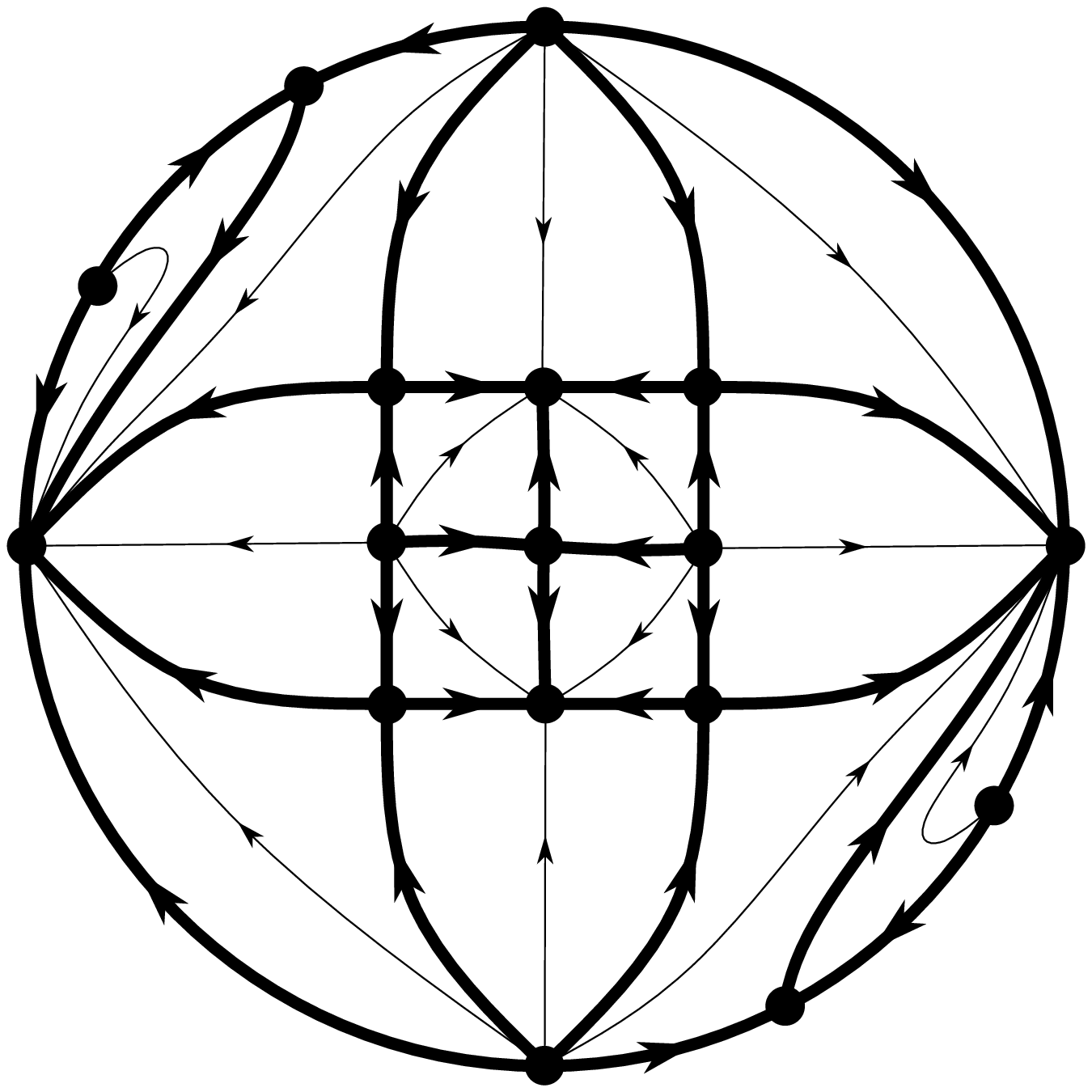} 
				\end{overpic}
				
				Case~$4.14b$.
			\end{center}
		\end{minipage}
		\begin{minipage}{3.1cm}
			\begin{center}
				\begin{overpic}[height=3cm]{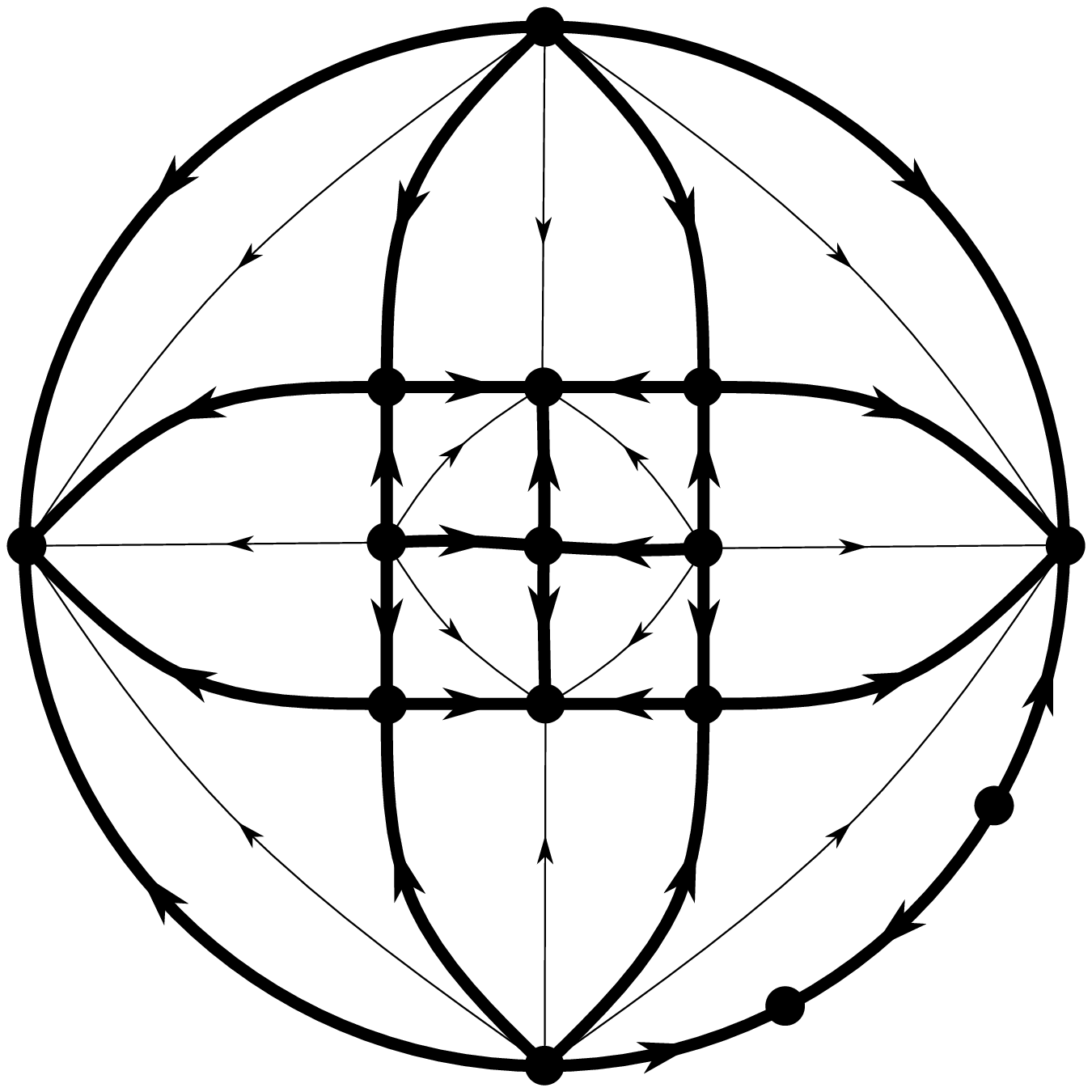} 
				\end{overpic}
				
				Case~$4.14c$.
			\end{center}
		\end{minipage}
	\end{center}
	\caption{Phase portraits from Cases~$4.12$ to $4.14$.}\label{Case1.4c}
\end{figure}

\end{document}